\documentclass{amsart}
\usepackage{amssymb,latexsym,amsmath,amscd,graphicx,graphics,epic,eepic,bm,color}
\usepackage{enumerate}
\usepackage[all,knot,poly]{xy}

\newcommand{\nat}{\mathbb{N}}
\newcommand{\zed}{\mathbb{Z}}

\newcommand{\C}{\mathbb{C}}

\newcommand{\Hom}{\mathrm{Hom}}
\newcommand{\im}{\mathrm{Im}}
\newcommand{\Sym}{\mathrm{Sym}}
\newcommand{\qb}[2]{\genfrac{[}{]}{0pt}{}{#1}{#2}}
\newcommand{\ve}{\varepsilon}
\newcommand{\wbar}{\overline}

\newcommand{\id}{\mathrm{id}}
\newcommand{\Id}{\mathrm{Id}}
\newcommand{\gdim}{\mathrm{gdim}}
\newcommand{\mf}{\mathrm{mf}}
\newcommand{\MF}{\mathrm{MF}}
\newcommand{\hmf}{\mathrm{hmf}}
\newcommand{\HMF}{\mathrm{HMF}}
\newcommand{\ch}{\mathsf{Ch}^{\mathsf{b}}}
\newcommand{\hch}{\mathsf{hCh}^{\mathsf{b}}}
\newcommand{\Tr}{\mathrm{Tr}}
\newcommand{\Kom}{\mathrm{Kom}}
\newcommand{\tc}{\mathrm{tc}}

\theoremstyle{plain}
\newtheorem{theorem}{Theorem}[section]
\newtheorem{lemma}[theorem]{Lemma}
\newtheorem{proposition}[theorem]{Proposition}
\newtheorem{corollary}[theorem]{Corollary}

\theoremstyle{definition}
\newtheorem{definition}[theorem]{Definition}

\newtheorem{acknowledgments}{Acknowledgments\ignorespaces}

\theoremstyle{remark}
\newtheorem{remark}[theorem]{Remark}

\numberwithin{equation}{section}

\begin{document}

\title{A colored $\mathfrak{sl}(N)$ homology for links in $S^3$}

\author{Hao Wu}

\address{Department of Mathematics, The George Washington University, Monroe Hall, Room 240, 2115 G Street, NW, Washington DC 20052}

\email{haowu@gwu.edu}

\subjclass[2000]{Primary 57M27}

\keywords{Reshetikhin-Turaev $\mathfrak{sl}(N)$ link invariant, Khovanov-Rozansky homology, matrix factorization, symmetric polynomial}

\begin{abstract}
Fix an integer $N\geq 2$. To each diagram of a link colored by $1,\dots,N$, we associate a chain complex of graded matrix factorizations. We prove that the homotopy type of this chain complex is invariant under Reidemeister moves. When every component of the link is colored by $1$, this chain complex is isomorphic to the chain complex defined by Khovanov and Rozansky in \cite{KR1}. The homology of this chain complex decategorifies to the Reshetikhin-Turaev $\mathfrak{sl}(N)$ polynomial of links colored by exterior powers of the defining representation.
\end{abstract}

\maketitle

\tableofcontents

\section{Introduction}\label{sec-intro}

\subsection{Background} In the early 1980s, Jones \cite{Jones} defined the Jones polynomial, which was generalized to the HOMFLYPT polynomial in \cite{HOMFLY,PT}. Later, Reshetikhin and Turaev \cite{Resh-Tur1} constructed a large family of polynomial invariants for framed links whose components are colored by finite dimensional representations of a complex semisimple Lie algebra. The HOMFLYPT polynomial is a special example of the Reshetikhin-Turaev invariants corresponding to the defining representation of $\mathfrak{sl}(N;\C)$. 

In general, the Reshetikhin-Turaev invariants for links are rather abstract. But, when the Lie algebra is $\mathfrak{sl}(N;\C)$ and every component of the link is colored by an exterior power of the defining representation, Murakami, Ohtsuki and Yamada \cite{MOY} gave a state sum formula for the corresponding $\mathfrak{sl}(N)$ Reshetikhin-Turaev invariant. Their construction comes with a set of graphical relations, which is known as the MOY calculus.

If every component of the link is colored by the defining representation, then the construction in \cite{MOY} recovers the uncolored $\mathfrak{sl}(N)$ HOMFLYPT polynomial. Modeling on this, Khovanov and Rozansky \cite{KR1} categorified the uncolored $\mathfrak{sl}(N)$ HOMFLYPT polynomial using matrix factorizations.  Their construction generalizes the Khovanov homology \cite{K1}.

\subsection{Some conventions}\label{subsec-conventions} Throughout this paper, $N$ is a fixed integer not less than $2$. 

All links and tangles in this paper are oriented and colored. That is, every component of the link or tangle is assigned an orientation and an element of $\{0,1,\dots,N\}$, which we call the color\footnote{In this paper, instead of saying that an object is colored by the $k$-fold exterior power of the defining representation of $\mathfrak{sl}(N;\C)$, we simply say that it is colored by $k$.} of this component. A link that is completely colored by $1$ is called uncolored.

Following the convention in \cite{KR1}, the degree of a polynomial in this paper is twice its usual degree.

\subsection{The colored $\mathfrak{sl}(N)$ link homology}\label{subsec-main-results} Our goal is to generalize Khovanov and Rozansky's construction in \cite {KR1} to categorify the Reshetikhin-Turaev $\mathfrak{sl}(N)$ polynomial of links colored by exterior powers of the defining representation. The following are our main results.

\begin{theorem}\label{main}
Let $D$ be a diagram of a tangle whose components are colored by elements of $\{0,1,\dots,N\}$, and $C(D)$ the chain complex defined in Definition \ref{complex-knotted-MOY-def}. Then:
\begin{enumerate}[(i)]
	\item $C(D)$ is a bounded chain complex over a homotopy category of graded matrix factorizations.
	\item $C(D)$ is $\zed_2\oplus\zed\oplus\zed$-graded, where the $\zed_2$-grading is the $\zed_2$-grading of the underlying matrix factorization, the first $\zed$-grading is the quantum grading of the underlying matrix factorization, and the second $\zed$-grading is the homological grading.
	\item The homotopy type of $C(D)$, with its $\zed_2\oplus\zed\oplus\zed$-grading, is invariant under Reidemeister moves. 
	\item If every component of $D$ is colored by $1$, then $C(D)$ is isomorphic to the chain complex defined by Khovanov and Rozansky in \cite{KR1}.
\end{enumerate}
\end{theorem}

Since the homotopy category $\hmf_{R,w}$ of graded matrix factorizations is not abelian, we can not directly define the homology of $C(D)$. But, as in \cite{KR1}, we can still construct a homology $H(D)$ from $C(D)$. Recall that each matrix factorization comes with a differential map $d_{mf}$. If $D$ is a link diagram, then the base ring $R$ is $\C$, and the potential $w=0$. So all the matrix factorizations in $C(D)$ are actually cyclic chain complexes. Taking the homology with respect to $d_{mf}$, we change $C(D)$ into a chain complex $(H(C(D),d_{mf}), d^\ast)$ of finite dimensional graded vector spaces, where $d^\ast$ is the differential map induced by the differential map $d$ of $C(D)$. We define 
\begin{equation}\label{def-homology-link}
H(D)= H(H(C(D),d_{mf}), d^\ast). 
\end{equation}
If $D$ is a diagram of a tangle with end points, then $R$ is a graded polynomial ring with homogeneous indeterminates of positive gradings, and $w$ is in the maximal homogeneous ideal $\mathfrak{I}$ of $R$ generated by all the indeterminates. So $(C(D)/\mathfrak{I}\cdot C(D), d_{mf})$ is a cyclic chain complex. Its homology $(H(C(D)/\mathfrak{I}\cdot C(D), d_{mf}),d^\ast)$ is a chain complex of finite dimensional graded vector spaces, where $d^\ast$ is the differential map induced by the differential map $d$ of $C(D)$. We define
\begin{equation}\label{def-homology-tangle}
H(D)= H(H(C(D)/\mathfrak{I}\cdot C(D),d_{mf}), d^\ast).
\end{equation}
In either case, $H(D)$ inherits the $\zed_2\oplus\zed\oplus\zed$-grading of $C(D)$. We call $H(D)$ the colored $\mathfrak{sl}(N)$ homology of $D$. The corollary below follows easily from Theorem \ref{main}.

\begin{corollary}\label{homology-inv-main}
Let $D$ be a diagram of a tangle whose components are colored by elements of $\{0,1,\dots,N\}$. Then $H(D)$ is a finite dimensional $\zed_2\oplus\zed\oplus\zed$-graded vector space over $\C$. Reidemeister moves of $D$ induce isomorphisms of $H(D)$ preserving its $\zed_2\oplus\zed\oplus\zed$-grading.
\end{corollary}

For a tangle $T$, denote by $H^{\ve,i,j}(T)$ the subspace of $H(T)$ of homogeneous elements of $\zed_2$-degree $\ve$, quantum degree $i$ and homological degree $j$. The Poincar\'e polynomial $\mathrm{P}_T (\tau, q, t)$ of $H(T)$ is defined to be 
\begin{equation}\label{def-poincare-polynomial}
\mathrm{P}_T (\tau, q, t) = \sum_{\ve,i,j} \tau^\ve q^i t^j \dim H^{\ve,i,j}(T) ~\in \C[\tau,q,t]/(\tau^2-1).
\end{equation}

Based on the construction by Murakami, Ohtsuki and Yamada \cite{MOY}, we give in Definition \ref{MOY-poly-def} a re-normalization $\mathrm{RT}_L(q)$ of the Reshetikhin-Turaev $\mathfrak{sl}(N)$ polynomial for links colored by non-negative integers. For a link $L$ colored by non-negative integers, the graded Euler characteristic of $H(L)$ is equal to $\mathrm{RT}_L(q)$. More precisely, we have the following theorem.

\begin{theorem}\label{euler-char-main}
Let $L$ be a link colored by non-negative integers. Then 
\[
\mathrm{P}_L (1, q, -1) = \mathrm{RT}_L(q).
\]

Moreover, define the total color $\tc(L)$ of $L$ to be the sum of the colors of the components of $L$. Then $H^{\ve,i,j}(L) =0$ if $\ve-\tc(L)= 1 \in \zed_2$ and therefore
\[
\mathrm{P}_L (\tau, q, t) = \tau^{\tc(L)} \sum_{i,j}  q^i t^j \dim H^{\tc(L),i,j}(L).
\]
\end{theorem}

\subsection{Deformations and applications} The construction of the colored $\mathfrak{sl}(N)$ link homology $H$ is based on matrix factorizations associated to MOY graphs with potentials induced by $X^{N+1}$. One can modify this construction by considering matrix factorizations with potentials induced by 
\[
f(X)=X^{N+1} + \sum_{k=1}^N B_{k} X^{N+1-k},
\]
where $B_{k}$ is a homogeneous indeterminate of degree $2k$. This gives an equivariant $\mathfrak{sl}(N)$ link homology $H_f$. $H_f$ is a finitely generated $\zed_2\oplus\zed\oplus\zed$-graded $\C[B_1,\dots,B_N]$-module. The construction of $H_f$ and the proof of its invariance are given in \cite{Wu-color-equi}, which generalizes the work of Krasner \cite{Krasner} in the uncolored case.

For any $b_1,\dots,b_N \in \C$, one can perform the above construction using matrix factorizations associated to MOY graphs with potentials induced by
\[
P(X)=X^{N+1} + \sum_{k=1}^N b_{k} X^{N+1-k},
\]
which gives a deformed $\mathfrak{sl}(N)$ link homology $H_P$. For any link $L$, $H_P(L)$ is a finitely dimensional $\zed_2\oplus\zed$-graded and $\zed$-filtered vector space over $\C$. The quotient map 
\[
\pi:\C[B_1,\dots,B_N] \rightarrow \C~(\cong \C[B_1,\dots,B_N]/(B_1-b_1,\dots,B_N-b_N))
\] 
given by $\pi(B_k)=b_k$ induces a functor $\varpi$ of between categories of matrix factorizations. Using this functor, one can easily show that the invariance of $H_f$ implies the invariance of $H_P$. As in the uncolored case, the filtration of $H_P$ induces a spectral sequence converging to $H_P$ with $E_1$-page isomorphic to the undeformed $\mathfrak{sl}(N)$ link homology $H$. Proofs of these results can be found in \cite{Wu-color-equi}.

When $P(X)$ is generic, that is, when $P'(X)$ has $N$ distinct root in $\C$, $H_P(L)$ admits a basis that generalizes the basis given by Lee \cite{Lee2} and Gornik \cite{Gornik}. See \cite{Wu-color-ras} for the construction. \cite{Wu-color-ras} also contains the definition of the colored $\mathfrak{sl}(N)$ Rasmussen invariants and the bounds for slice genus and self linking number given by these invariants.

The $\mathfrak{sl}(N)$ link homology $H$ itself also gives new bounds for the self linking number and the braid index. (See \cite{Wu-color-MFW}.) These bounds generalize the well known Morton-Franks-Williams inequality \cite{FW,Mo}.

\subsection{Other approaches to the colored $\mathfrak{sl}(N)$ link homology}  The Reshetikhin-Turaev $\mathfrak{sl}(N)$ polynomial of links colored by exterior powers of the defining representation has been categorified via several different approaches. Next we quickly review some recent results in this direction.

Using matrix factorizations, Yonezawa \cite{Yonezawa3} defined essentially the Poincar\'e polynomial $\mathrm{P}_T$ of the colored $\mathfrak{sl}(N)$ link homology $H$.

Stroppel \cite{Stroppel} gave a Lie-theoretic construction of the Khovanov homology, which is proved in \cite{Brundan-Stroppel} to be isomorphic to Khovanov's original construction. (See also \cite[Section 5]{Stroppel2}.) Mazorchuk and Stroppel \cite{Mazorchuk-Stroppel} described a Koszul dual construction for the $\mathfrak{sl}(N)$ link homology.

Mackaay, Stosic and Vaz \cite{Mackaay-Stosic-Vaz2} constructed a $\zed^{\oplus 3}$-graded HOMFLYPT homology for $1,2$-colored links, which generalizes Khovanov and Rozansky's construction in \cite{KR2}. Webster and Williamson \cite{Webster-Williamson} further generalized this homology to links colored by any non-negative integers using the equivariant cohomology of general linear groups and related spaces.

Cautis and Kamnitzer \cite{Cautis-Kamnitzer-1,Cautis-Kamnitzer-2} constructed a link homology using the derived category of coherent sheaves on certain flag-like varieties. Their homology is conjectured to be isomorphic to the $\mathfrak{sl}(N)$ Khovanov-Rozansky homology in \cite{KR1}. Using $\mathfrak{sl}(2)$ actions on certain categories of D-modules and coherent sheaves, they \cite{Cautis-talk} also categorified the $\mathfrak{sl}(N)$ polynomial for links in $S^3$ colored by exterior powers of the defining representation.

Using categorifications of the tensor products of integrable representations of Kac-Moody algebras and quantum groups, Webster \cite{Webster1,Webster2} categorified, for any simple complex Lie algebra $\mathfrak{g}$, the quantum $\mathfrak{g}$ invariant for links colored by any finite dimensional representations of $\mathfrak{g}$. All the aforementioned categorifications of the colored Reshetikhin-Turaev $\mathfrak{sl}(N)$ polynomial are expected to agree with Webster's categorification for $\mathfrak{g}=\mathfrak{sl}(N;\C)$.

\subsection{Outline of the proof} The present paper contains all the background knowledge needed to understand the construction of the colored $\mathfrak{sl}(N)$ link homology. Next we explain the structure of this paper and outline our proof. 

We review in Section \ref{sec-MOY-polynomial} the Murakami-Ohtsuki-Yamada construction of the Reshetikhin-Turaev $\mathfrak{sl}(N)$ polynomial for links colored by non-negative integers. In particular, we demonstrate that the $\mathfrak{sl}(N)$ MOY graph polynomial is uniquely determined by the MOY relations. 

Sections \ref{sec-mf} to \ref{sec-sym-poly} are reviews of algebraic structures used in our categorification. In Section \ref{sec-mf}, we recall the definition and properties of graded matrix factorizations. Then, in Section \ref{sec-mf-polynomial}, we take a closer look at graded matrix factorizations over polynomial rings. Section \ref{sec-sym-poly} is devoted to rings of symmetric polynomials, which serve as base rings in our construction.

Next, we define and study matrix factorizations associated to MOY graphs in Sections \ref{mf-MOY} to \ref{sec-MOY-V}. In particular, we prove direct sum decompositions (I-V), among which decompositions (I, II, IV, V) are essential in our construction of the colored $\mathfrak{sl}(N)$ homology.\footnote{Decomposition (III) is not explicitly used in the construction of the colored $\mathfrak{sl}(N)$ homology. The reader can skip this decomposition and its proof, that is, Subsections \ref{subsec-saddle-move}-\ref{subsec-2nd-composition} and Section \ref{sec-MOY-III}.} Decompositions (I-IV) are generalizations of the corresponding decompositions in \cite{KR1}. We prove these four decompositions by explicit constructions.\footnote{Using similar techniques, Yonezawa \cite{Yonezawa2} independently proved decompositions (I-III) and a special case of (IV).} Decomposition (V) is a further generalization of (IV) and is far more complex. We prove this decomposition by an induction based on decomposition (IV) using the Krull-Schmidt property of graded matrix factorizations. Note that these MOY decompositions decategorify to the corresponding MOY relations of the $\mathfrak{sl}(N)$ MOY graph polynomial. Thus, the graded dimension of our graph homology satisfies all the MOY relations and must be equal to the $\mathfrak{sl}(N)$ MOY graph polynomial.

The chain complex associated to a knotted MOY graph is defined in Section \ref{sec-chain-complex-def}. We resolve each knotted MOY graph into a collection of MOY graphs as in \cite{MOY} and then build a chain complex using the matrix factorizations associated to these MOY graphs. The homology of this chain complex is the $\mathfrak{sl}(N)$ homology of the knotted MOY graph. We observe that the graded Euler characteristic of the $\mathfrak{sl}(N)$ homology of a colored link is equal to its re-normalized Reshetikhin-Turaev $\mathfrak{sl}(N)$ polynomial. 

We prove in Sections \ref{sec-inv-fork} and \ref{sec-inv-reidemeister} that the homotopy type of the above chain complex is invariant under Reidemeister moves, which implies the invariance of the $\mathfrak{sl}(N)$ homology. In Section \ref{sec-inv-fork}, we prove the invariance of the homotopy type of our chain complex under fork sliding. With this in hand, we prove the colored invariance theorem in Section \ref{sec-inv-reidemeister} by reducing it to Khovanov and Rozansky's uncolored case \cite{KR1} using ``sliding bi-gons".

\begin{acknowledgments}
I would like to thank Mikhail Khovanov, Ben Webster and Yasuyoshi Yonezawa for interesting and helpful discussions. I am grateful to Yasuyoshi Yonezawa for sharing his lemma about graded Krull-Schmidt categories (see Lemma \ref{yonezawa-lemma} below) and to Mikhail Khovanov for suggesting an approach to understanding the Euler characteristic and the $\zed_2$-grading of the $\mathfrak{sl}(N)$ homology for colored links. (The proof of Theorem \ref{MOY-poly-skein-unique} below uses this approach.)

Most of the above mentioned discussions happened during Knots in Washington Conferences. I would like to thank the National Science Foundation and the George Washington University for supporting the Knots in Washington Conference Series.  
\end{acknowledgments}

\section{The MOY Calculus}\label{sec-MOY-polynomial}

\subsection{The HOMFLYPT polynomial} The HOMFLYPT polynomial defined in \cite{HOMFLY,PT} is an invariant for oriented links in $S^3$ in the form of a two variable polynomial $\mathsf{P}$. We normalize the HOMFLYPT polynomial using the following skein relations.
\[
\begin{cases}    
x\mathsf{P}(\setlength{\unitlength}{.5pt}
\begin{picture}(65,20)(50,0)
\put(100,-20){\vector(-1,1){40}}

\put(60,-20){\line(1,1){15}}

\put(85,5){\vector(1,1){15}}

\end{picture})-x^{-1}\mathsf{P}(\setlength{\unitlength}{.5pt}
\begin{picture}(65,20)(-110,0)
\put(-100,-20){\vector(1,1){40}}

\put(-60,-20){\line(-1,1){15}}

\put(-85,5){\vector(-1,1){15}}

\end{picture})= y\mathsf{P}(\setlength{\unitlength}{.5pt}
\begin{picture}(65,20)(50,0)
\put(100,-20){\vector(0,1){40}}

\put(60,-20){\vector(0,1){40}}

\end{picture}), &\\
   & \\
\mathsf{P}(\text{unknot})=\frac{x-x^{-1}}{y}. &
\end{cases}
\]

The specialization $\mathsf{P}_N=\mathsf{P}|_{x=q^N,~y=q-q^{-1}}$ is the $\mathfrak{sl}(N)$ HOMFLYPT polynomial and is determined by the skein relations
\[
\begin{cases}    
q^N\mathsf{P}_N(\setlength{\unitlength}{.5pt}
\begin{picture}(65,20)(50,0)
\put(100,-20){\vector(-1,1){40}}

\put(60,-20){\line(1,1){15}}

\put(85,5){\vector(1,1){15}}

\end{picture})-q^{-N}\mathsf{P}_N(\setlength{\unitlength}{.5pt}
\begin{picture}(65,20)(-110,0)
\put(-100,-20){\vector(1,1){40}}

\put(-60,-20){\line(-1,1){15}}

\put(-85,5){\vector(-1,1){15}}

\end{picture})= (q-q^{-1})\mathsf{P}_N(\setlength{\unitlength}{.5pt}
\begin{picture}(65,20)(50,0)
\put(100,-20){\vector(0,1){40}}

\put(60,-20){\vector(0,1){40}}

\end{picture}), &\\
   & \\
\mathsf{P}_N(\text{unknot})=\frac{q^N-q^{-N}}{q-q^{-1}}. &
\end{cases}
\]
$\mathsf{P}_N$ is a re-normalization of the Reshetikhin-Turaev polynomial of links colored by $1$, that is, the defining representation of $\mathfrak{sl}(N;\C)$. 

The general definition of the Reshetikhin-Turaev polynomials is rather abstract. In \cite{MOY}, Murakami, Ohtsuki and Yamada gave a combinatorial construction of the Reshetikhin-Turaev $\mathfrak{sl}(N)$ polynomial for links colored by non-negative integers, that is, exterior powers of the defining representation of $\mathfrak{sl}(N;\C)$. The construction of our colored $\mathfrak{sl}(N)$ link homology $H$ is modeled on their construction. 

In the remainder of this section, we review Murakami, Ohtsuki and Yamada's construction. Our notations and normalizations are slightly different from those used in \cite{MOY}.

\subsection{MOY graphs} 
\begin{definition}\label{MOY-graph-def}
An abstract MOY graph is an oriented graph with each edge colored by a non-negative integer such that, for every vertex $v$ with valence at least $2$, the sum of the colors of the edges entering $v$ is equal to the sum of the colors of the edges leaving $v$. We call this common sum the width of $v$.

A vertex of valence $1$ in an abstract MOY graph is called an end point. A vertex of valence greater than $1$ is called an internal vertex. An abstract MOY graph $\Gamma$ is said to be closed if it has no end points. We say that an abstract MOY graph is trivalent if all of its internal vertices have valence $3$.

A MOY graph is an embedding of an abstract MOY graph into $\mathbb{R}^2$ such that, through each internal vertex $v$, there is a straight line $L_v$ so that all the edges entering $v$ enter through one side of $L_v$ and all edges leaving $v$ leave through the other side of $L_v$.
\end{definition}

\begin{figure}[ht]

\setlength{\unitlength}{1pt}

\begin{picture}(360,80)(-180,-40)


\put(0,0){\vector(-1,1){15}}

\put(-15,15){\line(-1,1){15}}

\put(-23,25){\tiny{$i_1$}}

\put(0,0){\vector(-1,2){7.5}}

\put(-7.5,15){\line(-1,2){7.5}}

\put(-11,25){\tiny{$i_2$}}

\put(3,25){$\cdots$}

\put(0,0){\vector(1,1){15}}

\put(15,15){\line(1,1){15}}

\put(31,25){\tiny{$i_k$}}


\put(4,-2){$v$}

\multiput(-50,0)(5,0){19}{\line(1,0){3}}

\put(-70,0){$L_v$}

\put(45,0){\tiny{$i_1+i_2+\cdots +i_k = j_1+j_2+\cdots +j_l$}}


\put(-30,-30){\vector(1,1){15}}

\put(-15,-15){\line(1,1){15}}

\put(-26,-30){\tiny{$j_1$}}

\put(-15,-30){\vector(1,2){7.5}}

\put(-7.5,-15){\line(1,2){7.5}}

\put(-13,-30){\tiny{$j_2$}}

\put(3,-30){$\cdots$}

\put(30,-30){\vector(-1,1){15}}

\put(15,-15){\line(-1,1){15}}

\put(31,-30){\tiny{$j_l$}}

\end{picture}

\caption{An internal vertex of a MOY graph}\label{general-MOY-vertex-poly-figure}

\end{figure}
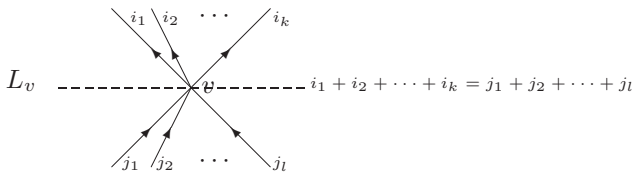

\subsection{The MOY graph polynomial} To each closed trivalent MOY graph, Murakami, Ohtsuki and Yamada \cite{MOY} associated a polynomial, which we call the MOY graph polynomial. They express the colored Reshetikhin-Turaev $\mathfrak{sl}(N)$ polynomial as a combination of MOY graph polynomials. We review the MOY graph polynomial in this subsection.

Define $\mathcal{N}=\{-N+1, -N+3,\cdots, N-3, N-1\}$ and $\mathcal{P(N)}$ to be the set of subsets of $\mathcal{N}$. For a finite set $A$, denote by $\#A$ the cardinality of $A$. Define a function $\pi:\mathcal{P(N)} \times \mathcal{P(N)} \rightarrow \zed_{\geq 0}$ by
\[
\pi (A_1, A_2) = \# \{(a_1,a_2) \in A_1 \times A_2 ~|~ a_1>a_2\} \text{ for } A_1,~A_2 \in \mathcal{P(N)}.
\]

\begin{figure}[ht]
$
\xymatrix{
\input{tri-vertex-s} & \text{or} & \input{tri-vertex-m}
} 
$
\caption{}\label{tri-vertex} 

\end{figure}

Let $\Gamma$ be a closed trivalent MOY graph, and $E(\Gamma)$ the set of edges of $\Gamma$. Denote by $\mathrm{c}:E(\Gamma) \rightarrow \nat$ the color function of $\Gamma$. That is, for every edge $e$ of $\Gamma$, $\mathrm{c}(e) \in \nat$ is the color of $e$. A state of $\Gamma$ is a function $\sigma: E(\Gamma) \rightarrow \mathcal{P(N)}$ such that
\begin{enumerate}[(i)]
	\item For every edge $e$ of $\Gamma$, $\#\sigma(e) = \mathrm{c}(e)$.
	\item For every vertex $v$ of $\Gamma$, as depicted in Figure \ref{tri-vertex}, we have $\sigma(e)=\sigma(e_1) \cup \sigma(e_2)$. (In particular, this implies that $\sigma(e_1) \cap \sigma(e_2)=\emptyset$.)
\end{enumerate}

For a state $\sigma$ of $\Gamma$ and a vertex $v$ of $\Gamma$ (as depicted in Figure \ref{tri-vertex}), the weight of $v$ with respect to $\sigma$ is defined to be 
\[
\mathrm{wt}(v;\sigma) = q^{\frac{\mathrm{c}(e_1)\mathrm{c}(e_2)}{2} - \pi(\sigma(e_1),\sigma(e_2))}.
\]

Given a state $\sigma$ of $\Gamma$, replace each edge $e$ of $\Gamma$ by $\mathrm{c}(e)$ parallel edges, assign to each of these new edges a different element of $\sigma(e)$ and, at every vertex, connect each pair of new edges assigned the same element of $\mathcal{N}$. This changes $\Gamma$ into a collection $\mathcal{C}$ of embedded oriented circles, each of which is assigned an element of $\mathcal{N}$. By abusing notation, we denote by $\sigma(C)$ the element of $\mathcal{N}$ assigned to $C\in \mathcal{C}$. Note that: 
\begin{itemize}
	\item There may be intersections between different circles in $\mathcal{C}$. But, each circle in $\mathcal{C}$ is embedded, that is, without self-intersections or self-tangency.
	\item There may be more than one way to do this. But if we view $\mathcal{C}$ as a virtue link and the intersection points between different elements of $\mathcal{C}$ as virtual crossings, then the above construction is unique up to purely virtual regular Reidemeister moves.
\end{itemize}
For each $C\in \mathcal{C}$, define the rotation number $\mathrm{rot}(C)$ the usual way. That is,
\begin{equation}\label{eq-def-rot-usual}
\mathrm{rot}(C) = 
\begin{cases}
1 & \text{if } C \text{ is counterclockwise,} \\
-1 & \text{if } C \text{ is clockwise.}
\end{cases}
\end{equation}
The rotation number $\mathrm{rot}(\sigma)$ of $\sigma$ is then defined to be
\[
\mathrm{rot}(\sigma) = \sum_{C\in \mathcal{C}} \sigma(C) \mathrm{rot}(C).
\]

The $\mathfrak{sl}(N)$ MOY polynomial of $\Gamma$ is defined to be
\begin{equation}\label{MOY-bracket-def}
\left\langle \Gamma \right\rangle_N := \sum_{\sigma} (\prod_v \mathrm{wt}(v;\sigma)) q^{\mathrm{rot}(\sigma)} \in \zed_{\geq 0}[q,q^{-1}],
\end{equation}
where $\sigma$ runs through all states of $\Gamma$ and $v$ runs through all vertices of $\Gamma$.

\subsection{The MOY calculus} Murakami, Ohtsuki and Yamada \cite{MOY} established a set of graphical relations for the $\mathfrak{sl}(N)$ MOY polynomial, which is known as the MOY calculus. The MOY calculus plays an important role in guiding us through the construction of the colored $\mathfrak{sl}(N)$ homology. 

Before stating the MOY calculus, we need to introduce our normalization of quantum integers.

\begin{definition}\label{def-quantum-integers}
Quantum integers are elements of $\zed[q,q^{-1}]$. In this paper, we use the normalization
\begin{eqnarray*}
[j] & := & \frac{q^j-q^{-j}}{q-q^{-1}}, \\ 
{[j]}! & := & [1] \cdot [2] \cdots [j], \\
\qb{j}{k} & := & \frac{[j]!}{[k]!\cdot [j-k]!}.
\end{eqnarray*}
\end{definition}

The following theorem is the MOY calculus.

\begin{theorem}\cite{MOY}\label{MOY-poly-skein}
The $\mathfrak{sl}(N)$ MOY graph polynomial $\left\langle \ast\right\rangle_N$ for close trivalent MOY graphs satisfies:
\begin{enumerate}
  \item $\left\langle \bigcirc_m \right\rangle_N = \qb{N}{m}$, where $\bigcirc_m$ is a circle colored by $m$.
  \item $\left\langle \setlength{\unitlength}{1pt}
\begin{picture}(50,50)(-80,20)

\put(-60,10){\vector(0,1){10}}

\put(-60,20){\vector(-1,1){20}}

\put(-60,20){\vector(1,1){10}}

\put(-50,30){\vector(-1,1){10}}

\put(-50,30){\vector(1,1){10}}

\put(-75,3){\tiny{$i+j+k$}}

\put(-55,21){\tiny{$j+k$}}

\put(-80,42){\tiny{$i$}}

\put(-60,42){\tiny{$j$}}

\put(-40,42){\tiny{$k$}}

\end{picture} \right\rangle_N = \left\langle \setlength{\unitlength}{1pt}
\begin{picture}(50,50)(40,20)

\put(60,10){\vector(0,1){10}}

\put(60,20){\vector(1,1){20}}

\put(60,20){\vector(-1,1){10}}

\put(50,30){\vector(1,1){10}}

\put(50,30){\vector(-1,1){10}}

\put(45,3){\tiny{$i+j+k$}}

\put(38,21){\tiny{$i+j$}}

\put(80,42){\tiny{$k$}}

\put(60,42){\tiny{$j$}}

\put(40,42){\tiny{$i$}}

\end{picture} \right\rangle_N$.
	\item $\left\langle \input{v-vector-m+n-bubble-slide}\right\rangle_N = \qb{m+n}{n} \cdot\left\langle \setlength{\unitlength}{.75pt}
\begin{picture}(55,80)(-20,40)
\put(0,0){\vector(0,1){80}}
\put(5,75){\tiny{$_{m+n}$}}
\end{picture}\right\rangle_N $.
	\item $\left\langle \setlength{\unitlength}{.75pt}
\begin{picture}(60,80)(-30,40)
\put(0,0){\vector(0,1){30}}
\put(0,30){\vector(0,1){20}}
\put(0,50){\vector(0,1){30}}

\put(-1,40){\line(1,0){2}}

\qbezier(0,30)(25,20)(25,30)
\qbezier(0,50)(25,60)(25,50)
\put(25,50){\vector(0,-1){20}}

\put(5,75){\tiny{$_{m}$}}
\put(5,5){\tiny{$_{m}$}}
\put(-30,38){\tiny{$_{m+n}$}}
\put(14,60){\tiny{$_{n}$}}
\end{picture}\right\rangle_N = \qb{N-m}{n} \cdot \left\langle \setlength{\unitlength}{.75pt}
\begin{picture}(40,80)(-20,40)
\put(0,0){\vector(0,1){80}}
\put(5,75){\tiny{$_{m}$}}
\end{picture}\right\rangle_N$.
	\item $\left\langle \input{decomp-III-1-slide}\right\rangle_N = \left\langle \setlength{\unitlength}{.75pt}
\begin{picture}(60,60)(-30,30)

\put(-20,0){\vector(0,1){60}}

\put(20,60){\vector(0,-1){60}}

\put(-25,30){\tiny{$_1$}}

\put(22,30){\tiny{$_m$}}
\end{picture}\right\rangle_N + [N-m-1] \cdot \left\langle \setlength{\unitlength}{.75pt}
\begin{picture}(60,60)(100,30)

\put(110,0){\vector(1,1){20}}

\put(130,20){\vector(1,-1){20}}

\put(130,40){\vector(0,-1){20}}

\put(130,40){\vector(-1,1){20}}

\put(150,60){\vector(-1,-1){20}}

\put(105,0){\tiny{$_1$}}

\put(105,55){\tiny{$_1$}}

\put(152,0){\tiny{$_m$}}

\put(152,55){\tiny{$_m$}}

\put(132,30){\tiny{$_{m-1}$}}

\end{picture}\right\rangle_N$.
	\item $\left\langle \input{decomp-IV-1-slide}\right\rangle_N = \qb{m-1}{n} \cdot \left\langle \setlength{\unitlength}{.75pt}
\begin{picture}(100,90)(-30,45)

\put(-20,0){\vector(0,1){45}}

\put(-20,45){\vector(0,1){45}}

\put(20,0){\vector(0,1){45}}

\put(20,45){\vector(0,1){45}}

\put(20,45){\vector(-1,0){40}}

\put(-27,20){\tiny{$_1$}}

\put(23,20){\tiny{$_{m+l-1}$}}

\put(-27,65){\tiny{$_l$}}

\put(23,65){\tiny{$_m$}}

\put(-5,38){\tiny{$_{l-1}$}}

\end{picture}\right\rangle_N  + \qb{m-1}{n-1} \cdot \left\langle \setlength{\unitlength}{.75pt}
\begin{picture}(80,90)(110,45)

\put(110,0){\vector(2,3){20}}

\put(150,0){\vector(-2,3){20}}

\put(130,30){\vector(0,1){30}}

\put(130,60){\vector(-2,3){20}}

\put(130,60){\vector(2,3){20}}

\put(117,20){\tiny{$_1$}}

\put(140,20){\tiny{$_{m+l-1}$}}

\put(117,65){\tiny{$_l$}}

\put(140,65){\tiny{$_m$}}

\put(133,42){\tiny{$_{m+l}$}}
\end{picture} \right\rangle_N$.\vspace{.5cm}	
	\item $\left\langle \input{decomp-V-1-slide}\right\rangle_N = \sum_{j=\max\{m-n,0\}}^m \qb{l}{k-j} \cdot \left\langle \input{decomp-V-2-slide} \right\rangle_N$. \vspace{.5cm}
\end{enumerate}
The above equations remain true if we reverse the orientation of the MOY graph or the orientation of $\mathbb{R}^2$.
\end{theorem}

The equations in Theorem \ref{MOY-poly-skein} actually uniquely determine the MOY graph polynomial. Some experts apparently knew this. But I did not find a written proof of this. So we include a proof here for the convenience of the reader.

\begin{theorem}\label{MOY-poly-skein-unique}
Equations in Theorem \ref{MOY-poly-skein} uniquely determine the $\mathfrak{sl}(N)$ MOY graph polynomial $\left\langle \ast\right\rangle_N$.
\end{theorem}

\begin{proof}
We prove this theorem by a double induction on the highest color of edges of $\Gamma$ and on the number of edges of $\Gamma$ with the highest color. 

Assume that $\left\langle \left\langle \ast \right\rangle\right\rangle_N$ also satisfies all the equations in Theorem \ref{MOY-poly-skein}. Kauffman and Vogel \cite{Kauffman-Vogel} proved that, for closed trivalent MOY graphs colored by $1,2$, the polynomial satisfying all the relations in Theorem \ref{MOY-poly-skein} is unique. That is, $\left\langle \left\langle \Gamma \right\rangle\right\rangle_N= \left\langle \Gamma \right\rangle_N$ if all edges of the MOY graph $\Gamma$ are colored by $1$ or $2$. 

Now assume that, for some $m\geq 2$, $\left\langle \left\langle \Gamma \right\rangle\right\rangle_N= \left\langle \Gamma \right\rangle_N$ if all edges of the MOY graph $\Gamma$ are colored by positive integers no greater than $m$. We use this to prove that $\left\langle \left\langle \Gamma \right\rangle\right\rangle_N= \left\langle \Gamma \right\rangle_N$ if all edges of $\Gamma$ are colored by positive integers no greater than $m+1$. To do this we induct on the number of edges colored by $m+1$ in $\Gamma$. 

Rephrasing our induction hypothesis, we can say that $\left\langle \left\langle \Gamma \right\rangle\right\rangle_N= \left\langle \Gamma \right\rangle_N$ if all edges of $\Gamma$ are colored by positive integers no greater than $m+1$ and exactly $0$ edges of $\Gamma$ is colored by $m+1$. Assume that, for some $k \geq 0$, $\left\langle \left\langle \Gamma \right\rangle\right\rangle_N= \left\langle \Gamma \right\rangle_N$ whenever all edges of $\Gamma$ are colored by positive integers no greater than $m+1$ and exactly $k$ edges of $\Gamma$ is colored by $m+1$.

Let $\Gamma$ be a MOY graph such that
\begin{itemize}
	\item all edge of $\Gamma$ are colored by positive integers no greater than $m+1$,
	\item exactly $k+1$ edges of $\Gamma$ is colored by $m+1$.
\end{itemize}
We claim that $\left\langle \left\langle \Gamma \right\rangle\right\rangle_N= \left\langle \Gamma \right\rangle_N$.

\textit{Case 1.} Assume that there is a circle $\bigcirc_{m+1}$ colored by $m+1$ in $\Gamma$. Let $\tilde{\Gamma}$ be $\Gamma$ with $\bigcirc_{m+1}$ removed. Then all edge of $\tilde{\Gamma}$ are colored by positive integers no greater than $m+1$, and exactly $k$ edges of $\tilde{\Gamma}$ is colored by $m+1$. So $\left\langle \left\langle \tilde{\Gamma} \right\rangle\right\rangle_N= \left\langle \tilde{\Gamma} \right\rangle_N$ and therefore $\left\langle \left\langle \Gamma \right\rangle\right\rangle_N= \qb{N}{m} \left\langle \left\langle \tilde{\Gamma} \right\rangle\right\rangle_N =\qb{N}{m} \left\langle \tilde{\Gamma} \right\rangle_N  = \left\langle \Gamma \right\rangle_N$.

\begin{figure}[ht]

$\input{high-color-m+1}$
\caption{}\label{high-color-nbhd} 

\end{figure}

\begin{figure}[ht]
$
\xymatrix{
\input{high-color-m+1-mod0} && \input{high-color-m+1-mod} && \input{high-color-m+1-mod2}
} 
$
\caption{}\label{high-color-nbhd-mod} 

\end{figure}

\textit{Case 2.} Assume there are no circles colored by $m+1$ in $\Gamma$. Then every edge in $\Gamma$ colored by $m+1$ is of the form in Figure \ref{high-color-nbhd}, where $1 \leq j,l \leq m$. Let $e$ be such an edges of $\Gamma$ as depicted in Figure \ref{high-color-nbhd}. We modify $\Gamma$ locally near $e$ as in Figure \ref{high-color-nbhd-mod}. This gives us new MOY graphs $\Gamma_0$, $\Gamma_1$ and $\Gamma_2$, which are identical to $\Gamma$ except in the neighborhoods shown in Figure \ref{high-color-nbhd-mod}. Note that each of $\Gamma_0$ and $\Gamma_2$ has exactly $k$ edges colored by $m+1$ and therefore $\left\langle \left\langle \Gamma_0 \right\rangle\right\rangle_N= \left\langle \Gamma_0 \right\rangle_N$ and $\left\langle \left\langle \Gamma_2 \right\rangle\right\rangle_N= \left\langle \Gamma_2 \right\rangle_N$. Using equations (2), (3) and (6) in Theorem \ref{MOY-poly-skein}, we get
\begin{eqnarray}
\label{MOY-gdim-rt-eq3-fake} \left\langle \left\langle \Gamma_1\right\rangle\right\rangle_N & = & [j]\cdot[l] \cdot \left\langle \left\langle \Gamma \right\rangle\right\rangle_N, \\
\label{MOY-gdim-rt-eq4-fake} \left\langle \left\langle \Gamma_0\right\rangle\right\rangle_N & = & \left\langle \left\langle \Gamma_1\right\rangle\right\rangle_N + [m-1] \cdot \left\langle \left\langle \Gamma_2\right\rangle\right\rangle_N, \\
\label{MOY-gdim-rt-eq3} \left\langle \Gamma_1\right\rangle_N & = & [j]\cdot[l] \cdot\left\langle \Gamma \right\rangle_N, \\
\label{MOY-gdim-rt-eq4} \left\langle \Gamma_0\right\rangle_N & = & \left\langle \Gamma_1\right\rangle_N + [m-1] \cdot \left\langle \Gamma_2\right\rangle_N.
\end{eqnarray}
It clearly follows that $\left\langle \left\langle \Gamma \right\rangle\right\rangle_N= \left\langle \Gamma \right\rangle_N$. This completes the double induction and proves Theorem \ref{MOY-poly-skein-unique}.
\end{proof}

\subsection{The colored Reshetikhin-Turaev $\mathfrak{sl}(N)$ polynomial}

\begin{definition}\cite{MOY}\label{MOY-poly-def}
For a link diagram $D$ colored by non-negative integers, define $\left\langle D \right\rangle_N$ by applying the following at every crossing of $D$.
\[
\left\langle \setlength{\unitlength}{1pt}
\begin{picture}(40,40)(-20,0)

\put(-20,-20){\vector(1,1){40}}

\put(20,-20){\line(-1,1){15}}

\put(-5,5){\vector(-1,1){15}}

\put(-11,15){\tiny{$_m$}}

\put(9,15){\tiny{$_n$}}

\end{picture} \right\rangle_N = \sum_{k=\max\{0,m-n\}}^{m} (-1)^{m-k} q^{k-m}\left\langle \input{square-m-n-k-left-poly}\right\rangle_N,
\]
\[
\left\langle \setlength{\unitlength}{1pt}
\begin{picture}(40,40)(-20,0)

\put(20,-20){\vector(-1,1){40}}

\put(-20,-20){\line(1,1){15}}

\put(5,5){\vector(1,1){15}}

\put(-11,15){\tiny{$_m$}}

\put(9,15){\tiny{$_n$}}

\end{picture} \right\rangle_N = \sum_{k=\max\{0,m-n\}}^{m} (-1)^{k-m} q^{m-k}\left\langle \input{square-m-n-k-left-poly}\right\rangle_N.
\]
\vspace{.5cm}

Also, for each crossing $c$ of $D$, define the shifting factor $\mathsf{s}(c)$ of $c$ by
\[
\mathsf{s}\left(\setlength{\unitlength}{1pt}
\begin{picture}(40,40)(-20,0)

\put(-20,-20){\vector(1,1){40}}

\put(20,-20){\line(-1,1){15}}

\put(-5,5){\vector(-1,1){15}}

\put(-11,15){\tiny{$_m$}}

\put(9,15){\tiny{$_n$}}

\end{picture}\right) = 
\begin{cases}
(-1)^{-m} q^{m(N+1-m)} & \text{if } m=n,\\
1 & \text{if } m \neq n,
\end{cases}
\]
\[
\mathsf{s}\left(\setlength{\unitlength}{1pt}
\begin{picture}(40,40)(-20,0)

\put(20,-20){\vector(-1,1){40}}

\put(-20,-20){\line(1,1){15}}

\put(5,5){\vector(1,1){15}}

\put(-11,15){\tiny{$_m$}}

\put(9,15){\tiny{$_n$}}

\end{picture}\right) = 
\begin{cases}
(-1)^m q^{-m(N+1-m)} & \text{if } m=n,\\
1 & \text{if } m \neq n.
\end{cases}
\]

The re-normalized Reshetikhin-Turaev $\mathfrak{sl}(N)$ polynomial $\mathrm{RT}_D(q)$ of $D$ is defined to be
\[
\mathrm{RT}_D(q) = \left\langle D \right\rangle_N \cdot \prod_c \mathsf{s}(c),
\]
where $c$ runs through all crossings of $D$.
\end{definition}

\begin{theorem}\cite{MOY}\label{MOY-poly-inv}
$\left\langle D \right\rangle_N$ is invariant under regular Reidemeister moves. $\mathrm{RT}_D(q)$ is invariant under all Reidemeister moves and is a re-normalization of the Reshetikhin-Turaev $\mathfrak{sl}(N)$ polynomial for links colored by non-negative integers.
\end{theorem}

\section{Graded Matrix Factorizations}\label{sec-mf}

In this section, we review the definition and properties of graded matrix factorizations over graded $\C$-algebras, most of which can be found in \cite{KR1,KR2,KR3,Ras2,Wu7}. Some of these properties are stated slightly more precisely here for the convenience of later applications. 

\subsection{$\zed$-pregraded and $\zed$-graded linear spaces} Let $V$ be a $\C$-linear space. A $\zed$-pregrading of $V$ is a collection $\{V^{(i)}|i\in\zed\}$ of $\C$-linear spaces, such that there exist injective $\C$-linear maps $\bigoplus_{i\in \zed} V^{(i)}\hookrightarrow V$ and $V\hookrightarrow \prod_{i\in \zed} V^{(i)}$ that make diagram \eqref{pregrade-def} commutative, where the horizontal map is the standard inclusion map from the direct sum to the direct product.
\begin{equation}\label{pregrade-def}
\xymatrix{
\bigoplus_{i\in \zed} V^{(i)} \ar@{^{(}->}[dr] \ar@{^{(}->}[rr]  & & \prod_{i\in \zed} V^{(i)} \\
& V \ar@{^{(}->}[ur] & \\
}
\end{equation}
From now on, we will identify $V^{(i)}$ with its image in $V$. An element $v$ of $V^{(i)}$ is called a homogeneous element of $V$ of degree $i$. In this case, we write $\deg v =i$.

A $\zed$-pregrading $\{V^{(i)}|i\in\zed\}$ of $V$ is called a $\zed$-grading if the $\C$-linear map $\bigoplus_{i\in \zed} V^{(i)}\hookrightarrow V$ is an isomorphism.

We say that a $\zed$-pregrading $\{V^{(i)}|i\in\zed\}$ of $V$ is bounded from below (resp. above) if there is an $m\in \zed$ such that $V^{(i)}=0$ whenever $i<m$ (resp. $i>m$.) We call a $\zed$-pregrading bounded if it is bounded from both below and above.

Let $V$ and $W$ be $\zed$-pregraded linear spaces with pregradings $\{V^{(i)}\}$ and $\{W^{(i)}\}$. A $\C$-linear map $f:V\rightarrow W$ is called homogeneous of degree $k$ if and only if $f(V^{(i)})\subset W^{(i+k)}$ for all $i\in\zed$.

Let $V$ be $\zed$-pregraded linear spaces with pregrading $\{V^{(i)}\}$. For any $j\in \zed$, define $V\{q^j\}$ to be $V$ with the pregrading shifted by $j$. That is, the pregrading $\{V\{q^j\}^{(i)}\}$ of $V\{q^j\}$ is defined by $V\{q^j\}^{(i)}=V^{(i-j)}$. More generally, for 
\[
F(q)=\sum_{j=k}^l a_j q^j ~\in \zed_{\geq0}[q,q^{-1}],
\]
we define the $\zed$-pregraded linear space $V\{F(q)\}$ to be
\[
V\{F(q)\} = \bigoplus_{j=k}^l (\underbrace{V\{q^j\}\oplus\cdots\oplus V\{q^j\}}_{a_j-\text{fold}})
\]
with the obvious pregrading $\{V\{F(q)\}^{(i)}\}$ given by
\[
V\{F(q)\}^{(i)} = \bigoplus_{j=k}^l (\underbrace{V^{(i-j)}\oplus\cdots\oplus V^{(i-j)}}_{a_j-\text{fold}})
\]

\subsection{Graded modules over a graded $\C$-algebra} In the rest of this section, $R$ will be a graded commutative unital $\C$-algebra, where ``graded" means that we fix a $\zed$-grading $\{R^{(i)}\}$ on the underlying $\C$-space of $R$ that satisfies $R^{(i)} \cdot R^{(j)} \subset R^{(i+j)}$. It is easy to see that $1\in R^{(0)}$. 

A $\zed$-grading of an $R$-module $M$ is a $\zed$-grading $\{M^{(i)}\}$ of its underlying $\C$-space satisfying $R^{(i)} \cdot M^{(j)} \subset M^{(i+j)}$. A graded $R$-module is an $R$-module with a fixed $\zed$-grading. For a graded $R$-module $M$ and $F(q)\in \zed_{\geq0}[q,q^{-1}]$, $M\{F(q)\}$ is defined as above.

\begin{lemma}\label{pregrade-hom}
Let $M_1$ and $M_2$ be graded $R$-modules. Then $\Hom_R(M_1,M_2)$ has a natural $\zed$-pregrading. If $M_1$ is finitely generated over $R$, then this pregrading is a grading.
\end{lemma}
\begin{proof}
Let $\{M_1^{(i)}\}$ and $\{M_2^{(i)}\}$ be the gradings of $M_1$ and $M_2$. Define 
\[
\Hom_R(M_1,M_2)^{(k)}=\{f\in\Hom_R(M_1,M_2) ~|~ f(M_1^{(i)}) \subset M_2^{(i+k)}\}.
\]
We claim that $\{\Hom_R(M_1,M_2)^{(k)}\}$ is a $\zed$-pregrading of $\Hom_R(M_1,M_2)$. To prove this, we only need to show that:
\begin{enumerate}[(i)]
	\item Any $f\in\Hom_R(M_1,M_2)$ can be uniquely expressed as a sum $\sum_{k=-\infty}^{\infty} f_k$, where $f_k$ is in $\Hom_R(M_1,M_2)^{(k)}$ and is called the homogeneous part of degree $k$ of $f$.
	\item For $f,g\in\Hom_R(M_1,M_2)$, $f=g$ only if all of their corresponding homogeneous part are equal.
\end{enumerate}
(ii) and the uniqueness part of (i) are simple and left to the reader. We only check the existence part of (i). 

For $l=1,2$, let $J_l^{(i)}:M_l^{(i)}\rightarrow M_l$ and $P_l^{(i)}:M_l\rightarrow M_l^{(i)}$ be the inclusion and projection in
\[
M_l = \bigoplus_{i\in\zed} M_l^{(i)}.
\] 
For $k\in\zed$, define a $\C$-linear map $f_k:M_1\rightarrow M_2$ by $f_k|_{M_1^{(i)}}= P_2^{(i+k)}\circ f \circ J_1^{(i)}$. Let $m\in M_1$. Then there exist $i_1 \leq i_2$ and $j_1 \leq j_2$ such that $m \in \bigoplus_{i=i_1}^{i_2} M_1^{(i)}$ and $f(m) \in \bigoplus_{j=j_1}^{j_2} M_2^{(j)}$. It is easy to see that $f_k(m)=0$ if $k>j_2-i_1$ or $k<j_1-i_2$. So the sum $\sum_{k=-\infty}^{\infty} f_k(m)$ is a finite sum for any $m\in M_1$. Thus the infinite sum $\sum_{k=-\infty}^{\infty} f_k$ is a well defined $\C$-linear map from $M_1$ to $M_2$. One can easily see that $f=\sum_{k=-\infty}^{\infty} f_k$ as $\C$-linear maps, and that $f_k$ is a homogeneous $\C$-linear map of degree $k$. It remains to check that $f_k$ is an $R$-module map for every $k$. Let $m\in M_1^{(i)}$ and $r\in R^{(j)}$. Then $rm\in M_1^{(i+j)}$ and 
\begin{eqnarray*}
f_k(rm) & = & P_2^{(i+j+k)}(f(rm)) = P_2^{(i+j+k)}(r f(m)) \\
& = & P_2^{(i+j+k)}(r \cdot \sum_{n=-\infty}^{\infty} f_n(m)) = rf_k(m).
\end{eqnarray*}
This implies that $f_k$ is an $R$-module map and, therefore, $f_k\in\Hom_R(M_1,M_2)^{(k)}$. Thus, $\{\Hom_R(M_1,M_2)^{(k)}\}$ is a $\zed$-pregrading of $\Hom_R(M_1,M_2)$.

Now assume that $M_1$ is generated by a finite subset $\{m_1,\dots,m_p\}$. For any $f\in\Hom_R(M_1,M_2)$, there exist $i_1 \leq i_2$, $j_1 \leq j_2$ such that $m_1,\dots,m_p \in \bigoplus_{i=i_1}^{i_2} M_1^{(i)}$ and $f(m_1),\dots,f(m_p) \in \bigoplus_{j=j_1}^{j_2} M_2^{(j)}$. It follows easily that $f_k=0$ if $k>j_2-i_1$ or $k<j_1-i_2$. So 
\[
f=\sum_{k=j_1-i_2}^{j_2-i_1} f_k \in \bigoplus_{k=-\infty}^{\infty} \Hom_R(M_1,M_2)^{(k)}.
\]
Thus, 
\[
\Hom_R(M_1,M_2) = \bigoplus_{k=-\infty}^{\infty} \Hom_R(M_1,M_2)^{(k)},
\]
which implies that $\{\Hom_R(M_1,M_2)^{(k)}\}$ is a $\zed$-grading of $\Hom_R(M_1,M_2)$.
\end{proof}

In the present paper, we are especially interested in free graded modules over $R$. Note that a free graded module need not have a basis consisting of homogeneous elements. Following \cite[Chapter 13]{Passman-book}, we introduce the following definition.

\begin{definition}\label{graded-free-def}
A graded module $M$ over $R$ is called graded-free if and only if it is a free module over $R$ with a homogeneous basis.
\end{definition}

All the modules involved in the construction of the $\mathfrak{sl}(N)$ homology are modules over polynomial rings. We will see in Section \ref{sec-mf-polynomial} that, if the grading of a free graded module over a polynomial ring is bounded below, then this module is graded-free.

\subsection{Graded matrix factorizations}\label{graded-matrix-factorizations} Recall that, $N$ $(\geq 2)$ is a fixed integer throughout the present paper. (It is the ``$N$" in ``$\mathfrak{sl}(N)$".) 

Let $R$ be a graded commutative unital $\C$-algebra, and $w$ a homogeneous element of $R$ of degree $2N+2$. A graded matrix factorization $M$ over $R$ with potential $w$ is a collection of two free graded $R$-modules $M_0$, $M_1$ and two homogeneous $R$-module maps $d_0:M_0\rightarrow M_1$, $d_1:M_1\rightarrow M_0$ of degree $N+1$, called differential maps, such that 
\[
d_1 \circ d_0=w\cdot\id_{M_0}, \hspace{1cm}  d_0 \circ d_1=w\cdot\id_{M_1}.
\]
We usually write $M$ as
\[
M_0 \xrightarrow{d_0} M_1 \xrightarrow{d_1} M_0.
\]
$M$ has two gradings: a $\zed_2$-grading that takes value $\ve$ on $M_\ve$ and a quantum grading, which is the $\zed$-grading of the underlying graded $R$-module. We denote by ``$\deg_{\zed_2}$" the degree from the $\zed_2$-grading and by ``$\deg$" the degree from the quantum grading.

Following \cite{KR1}, we denote by $M\left\langle 1\right\rangle$ the matrix factorization
\[
M_1 \xrightarrow{d_1} M_0 \xrightarrow{d_0} M_1,
\]
and write $M\left\langle j\right\rangle = M \underbrace{\left\langle 1\right\rangle\cdots\left\langle 1\right\rangle}_{j \text{ times }}$.

For graded matrix factorizations $M$ with potential $w$ and $M'$ with potential $w'$, the tensor product $M\otimes M'$ is the graded matrix factorization with 
\begin{eqnarray*}
(M\otimes M')_0 & = & (M_0\otimes M'_0)\oplus (M_1\otimes M'_1), \\
(M\otimes M')_1 & = & (M_1\otimes M'_0)\oplus (M_1\otimes M'_0),
\end{eqnarray*}
and the differential is given by the signed Leibniz rule. That is, for $m\in M_\ve$ and $m'\in M'$,
\[
d(m\otimes m')=(dm)\otimes m' + (-1)^\ve m \otimes (dm').
\]
The potential of $M\otimes M'$ is $w+w'$.

\begin{definition}\label{koszul-mf-def}
If $a_0,a_1\in R$ are homogeneous elements with $\deg a_0 +\deg a_1=2N+2$, then denote by $(a_0,a_1)_R$ the matrix factorization $R \xrightarrow{a_0} R\{q^{N+1-\deg{a_0}}\} \xrightarrow{a_1} R$, which has potential $a_0a_1$. More generally, if $a_{1,0},a_{1,1},\dots,a_{k,0},a_{k,1}\in R$ are homogeneous with $\deg a_{j,0} +\deg a_{j,1}=2N+2$, then denote by 
\[
\left(%
\begin{array}{cc}
  a_{1,0}, & a_{1,1} \\
  a_{2,0}, & a_{2,1} \\
  \dots & \dots \\
  a_{k,0}, & a_{k,1}
\end{array}%
\right)_R
\]
the tenser product 
\[
(a_{1,0},a_{1,1})_R \otimes_R (a_{2,0},a_{2,1})_R \otimes_R \cdots \otimes_R (a_{k,0},a_{k,1})_R,
\] 
which is a graded matrix factorization with potential $\sum_{j=1}^k a_{j,0} a_{j,1}$, and is call the Koszul matrix factorization associated to the above matrix. We drop``$R$" the notation when it is clear from the context. Note that the above Koszul matrix factorization is finitely generated over $R$.
\end{definition}

Since the Koszul matrix factorizations we use in this paper are more complex than those in \cite{KR1,KR2,Ras2,Wu7}, it is generally harder to compute them. So it is more important to keep good track of the signs. For this reason, we introduce the following notations.

\begin{definition}\label{ve-notation}
\begin{itemize}
	\item Let $I=\{0,1\}$. Define $\wbar{1}=0$ and $\wbar{0}=1$.
	\item For $\ve=(\ve_1,\dots,\ve_k) \in I^k$, define $|\ve|=\sum_{j=1}^{k} \ve_j$, and for $1\leq i \leq k$, define $|\ve|_i=\sum_{j=1}^{i-1} \ve_j$. Also define $\wbar{\ve}= (\wbar{\ve_1},\dots,\wbar{\ve_k})$ and $\ve'= (\ve_k,\ve_{k-1},\dots,\ve_1)$.
	\item In $(a_0,a_1)_R$, denote by $1_0$ the unit element of the copy of $R$ with $\zed_2$-grading $0$, and by $1_1$ the unit element of the copy of $R$ with $\zed_2$-grading $1$. Note that $\{1_0, 1_1\}$ is an $R$-basis for $(a_0,a_1)_R$.
	\item In 
\[
M = \left(%
\begin{array}{cc}
  a_{1,0}, & a_{1,1} \\
  \dots & \dots \\
  a_{k,0}, & a_{k,1}
\end{array}%
\right)_R,
\]
for any $\ve=(\ve_1,\dots,\ve_k) \in I^k$, define $1_{\ve}=1_{\ve_1}\otimes\cdots\otimes1_{\ve_1}$ in the tensor product
\[
(a_{1,0},a_{1,1})_R \otimes_R \cdots \otimes_R (a_{k,0},a_{k,1})_R.
\]
Note that $\{1_\ve~|~\ve\in I^k\}$ is an $R$-basis for $M$, and $1_\ve$ is a homogeneous element with $\zed_2$-degree $|\ve|$ and quantum degree $\sum_{j=1}^k \ve_j(N+1-\deg a_{j,0})$. In the above notations, the differential of $M$ is given by 
\begin{equation}\label{eq-diff-Koszul}
d (1_{\ve}) = \sum_{j=1}^k (-1)^{|\ve|_j}a_{j,\ve_j}\cdot 1_{(\ve_1,\dots,\ve_{j-1},\wbar{\ve_j},\ve_{j+1},\dots\ve_k)}.
\end{equation}
\end{itemize}
\end{definition}

\begin{remark}
In many cases, only the parity of $\ve_j$ matters and $I$ can be viewed as $\zed_2$. But, in some situations, we need more information and $I$ can not be identified with $\zed_2$.
\end{remark}

\subsection{Morphisms of graded matrix factorizations} Given two graded matrix factorizations $M$ with potential $w$ and $M'$ with potential $w'$ over $R$, consider the $R$-module $\Hom_R(M,M')$. It admits a $\zed_2$-grading that takes value 
\[
\left\{%
\begin{array}{l}
    0 \text{ on } \Hom^0_R(M,M')=\Hom_R(M_0,M'_0)\oplus\Hom_R(M_1,M'_1), \\ 
    1 \text{ on } \Hom^1_R(M,M')=\Hom_R(M_1,M'_0)\oplus\Hom_R(M_0,M'_1). 
\end{array}%
\right.
\]
By Lemma \ref{pregrade-hom}, it also admits a quantum pregrading induced by the quantum gradings of homogeneous elements. Moreover, $\Hom_R(M,M')$ has a differential map $d$ given by
\[
d(f)=d_{M'} \circ f -(-1)^\ve f \circ d_M \text{ for } f \in \Hom^\ve_R(M,M').
\]
Note that $d$ is homogeneous of degree $N+1$ and satisfies that 
\[
d^2=(w'-w) \cdot \id_{\Hom_R(M,M')}.
\]
Following \cite{KR1}, we write $M_\bullet = \Hom_R(M,R)$.

In general, $\Hom_R(M,M')$ is not a graded matrix factorization since $\Hom_R(M,M')$ is not necessarily a free $R$-module and its quantum pregrading is not necessarily a grading. But we have the following easy lemma. 

\begin{lemma}\label{hom-finite-gen}
Let $M$ and $M'$ be as above. Assume that $M$ is finitely generated over $R$. Then $\Hom_R(M,M')$ is a graded matrix factorization over $R$ of potential $w'-w$. In particular, $M_\bullet = \Hom_R(M,R)$ is a graded matrix factorization over $R$ of potential $-w$. 
\end{lemma}
\begin{proof}
Since $M$ is finitely generated, we know that $\Hom_R(M,M')$ is a free $R$-module and, by Lemma \ref{pregrade-hom}, the quantum pregrading is a grading.
\end{proof}

\begin{definition}\label{cong-and-sim}
Let $M$ and $M'$ be two graded matrix factorizations over $R$ with potential $w$. Then $\Hom_R(M,M')$, with the above differential map $d$, is a chain complex with a $\zed_2$ homological grading. 
\begin{enumerate}
\item We say that an $R$-module map $f:M\rightarrow M'$ is a morphism of matrix factorizations if $df=0$. Or, equivalently, for $f\in\Hom^\ve_R(M,M')$, $f$ is a morphism of matrix factorizations if $d_{M'}\circ f = (-1)^\ve f\circ d_M$. 
\item $f$ is called an isomorphism of matrix factorizations if it is a morphism of matrix factorizations and an isomorphism of the underlying $R$-modules. 
\item We say that $M,M'$ are isomorphic as graded matrix factorizations, or $M\cong M'$, if there is a homogeneous isomorphism $f:M\rightarrow M'$ that preserves the $\zed_2\oplus\zed$-grading.
\item Two morphisms $f,g:M\rightarrow M'$ of $\zed_2$-degree $\ve$ are homotopic if $f-g$ is a boundary element in $\Hom_R(M,M')$, that is, if $\exists ~h\in \Hom^{\ve+1}_R(M,M')$ such that $f-g = d(h) = d_{M'} \circ h -(-1)^{\ve+1} h \circ d_M$. 
\item We say that $M,M'$ are homotopic as graded matrix factorizations, or $M\simeq M'$, if there are homogeneous morphisms $f:M\rightarrow M'$ and $g:M'\rightarrow M$ preserving the $\zed_2\oplus\zed$-grading such that $g\circ f \simeq \id_M$ and $f\circ g \simeq \id_{M'}$. 
\end{enumerate}
\end{definition}

\begin{lemma}\label{rewrite-hom-finite-gen}
Let $M$ and $M'$ be two graded matrix factorizations over $R$ with potentials $w$ and $w'$. Assume that $M$ is finitely generated over $R$. Then the natural $R$-module isomorphism 
\[
M' \otimes M_\bullet = M' \otimes \Hom_R(M,R) \xrightarrow{\cong} \Hom_R(M,M')
\] 
is a homogeneous isomorphism preserving the $\zed_2\oplus\zed$-grading. 

In particular, $\Hom_R(M,M') \cong M' \otimes M_\bullet$ as graded matrix factorizations.
\end{lemma}
\begin{proof}
By Lemma \ref{hom-finite-gen}, $M' \otimes M_\bullet$ and $\Hom_R(M,M')$ are both graded matrix factorizations over $R$ with potential $w'-w$. The natural isomorphism $F$ between them is given by $F(m'\otimes f)(m)=f(m)\cdot m'$ for all $m'\in M'$, $f\in M_\bullet$ and $m\in M$. It is easy to check that $F$ preserves the $\zed_2\oplus\zed$-grading and commutes with the differential maps.
\end{proof}

The following lemma specifies the sign convention we use when tensoring two morphisms of matrix factorizations.

\begin{lemma}\label{morphism-sign}
Let $R$ be a graded commutative unital $\C$-algebra, and $M,~M',~\mathcal{M}, ~\mathcal{M}'$ graded matrix factorizations over $R$ such that $M,\mathcal{M}$ have the same potential and $M',\mathcal{M}'$ have the same potential. Assume that $f:M \rightarrow \mathcal{M}$ and $f':M'\rightarrow \mathcal{M}'$ are morphisms of matrix factorizations of $\zed_2$-degrees $j$ and $j'$. Define $F: M\otimes M' \rightarrow \mathcal{M}\otimes \mathcal{M}'$ by $F(m\otimes m') = (-1)^{i\cdot j'} f(m)\otimes f'(m')$ for $m \in M_i$ and $m'\in M'$. Then $F$ is a morphism of matrix factorizations of $\zed_2$-degree $j+j'$. In particular, if $f$ or $f'$ is homotopic to $0$, then so is $F$.

From now on, we will write $F = f \otimes f'$.
\end{lemma}
\begin{proof}
\begin{eqnarray*}
F\circ d(m \otimes m') & = & F( (dm)\otimes m' + (-1)^i m\otimes (dm')) \\
& = & (-1)^{(i+1)j'} f(dm)\otimes f'(m') + (-1)^{i+ij'} f(m) \otimes f'(dm'),
\end{eqnarray*}
\begin{eqnarray*}
d\circ F(m \otimes m') & = & (-1)^{ij'}d( f(m)\otimes f'(m')) \\
& = & (-1)^{ij'} ( d (f(m)) \otimes f'(m') +(-1)^{i+j} f(m) \otimes d (f'(m'))) \\
& = & (-1)^{ij'+j} f(dm)\otimes f'(m') + (-1)^{ij'+i+j+j'}f(m) \otimes f'(dm').
\end{eqnarray*}
So $F\circ d = (-1)^{j+j'} d\circ F$, that is, $F$ is a morphism of matrix factorizations of $\zed_2$-degree $j+j'$.

If $f$ is homotopic to $0$. Then there exits $h\in \Hom^{j+1}_R(M,\mathcal{M})$ such that 
\[
f= d(h) = d \circ h -(-1)^{j+1} h \circ d.
\]
Define $H \in \Hom^{j+j'+1}_R(M\otimes M',\mathcal{M}\otimes \mathcal{M}')$ by 
\[
H(m\otimes m') := (-1)^{ij'} h(m) \otimes f'(m'), \text{ for } m \in M_i \text{ and } m'\in M'.
\]
Then $d(H)= d\circ H -(-1)^{j+j'+1}H\circ d = F$. So $F$ is homotopic to $0$.

If $f'$ is homotopic to $0$. Then there exits $h'\in \Hom^{j'+1}_R(M',\mathcal{M}')$ such that 
\[
f'= d(h') = d \circ h' -(-1)^{j'+1} h' \circ d.
\] 
Define $H' \in \Hom^{j+j'+1}_R(M\otimes M',\mathcal{M}\otimes \mathcal{M}')$ by 
\[
H'(m\otimes m') := (-1)^{i(j'+1)} f(m) \otimes h'(m'), \text{ for } m \in M_i \text{ and } m'\in M'.
\]
Then $d(H')= d\circ H' -(-1)^{j+j'+1}H'\circ d = (-1)^{j}F$. So $F$ is homotopic to $0$.
\end{proof}

The following lemma is \cite[Proposition 2]{KR1}.

\begin{lemma}\cite[Proposition 2]{KR1}\label{entries-null-homotopic}
Let $R$ be a graded commutative unital $\C$-algebra, and $a_{1,0}, a_{1,1}, \dots, a_{k,0}, a_{k,1}$ homogeneous elements of $R$ with $\deg a_{j,0}+\deg a_{j,1}=2N+2$ $\forall~ j$. Let
\[
M = \left(%
\begin{array}{ll}
  a_{1,0}, & a_{1,1} \\
  a_{2,0}, & a_{2,1} \\
  \dots & \dots \\
  a_{k,0}, & a_{k,1}
\end{array}%
\right)_R
\]
If $x$ is an element of the ideal $(a_{1,0}, a_{1,1}, \dots, a_{k,0}, a_{k,1})$ of $R$, then multiplication by $x$ is a null-homotopic endomorphism of $M$.
\end{lemma}
\begin{proof}
(Following \cite{KR1}.) Multiplications by $a_{i,0}$ and $a_{i,1}$ are null-homotopic endomorphisms of $(a_{1,0},a_{1,1})$, and therefore, by Lemma \ref{morphism-sign}, are null-homotopic endomorphisms of $M$.
\end{proof}

Next we give precise definitions of several isomorphisms used in \cite{KR1}, which allow us to keep track of signs in later applications.

\begin{lemma}\label{bullet}
Let $R$ be a graded commutative unital $\C$-algebra, and $a_{1,0}, a_{1,1}, \dots, a_{k,0}, a_{k,1}$ homogeneous elements of $R$ with $\deg a_{j,0}+\deg a_{j,1}=2N+2$ $\forall~ j$. Let
\[
M = \left(%
\begin{array}{ll}
  a_{1,0}, & a_{1,1} \\
  a_{2,0}, & a_{2,1} \\
  \dots & \dots \\
  a_{k,0}, & a_{k,1}
\end{array}%
\right)_R
~\text{ and }~
M' = \left(%
\begin{array}{ll}
  -a_{k,1}, & a_{k,0} \\
  -a_{k-1,1}, & a_{k-1,0} \\
  \dots & \dots \\
  -a_{1,1}, & a_{1,0}
\end{array}%
\right)_R.
\]
Denote by $\{1^\ast_\ve~|~\ve\in I^k\}$ the basis of $M_\bullet$ dual to $\{1_\ve~|~\ve\in I^k\}$, that is, $1^\ast_\ve(1_\ve)=1$ and $1^\ast_\ve(1_\sigma)=0$ if $\sigma \neq \ve$. Recall that, by Definition \ref{ve-notation}, $\ve'=(\ve_k,\ve_{k-1},\dots,\ve_1)$ for $\ve=(\ve_1,\ve_2,\dots,\ve_k) \in I^k$. Then the $R$-homomorphism $F:M_\bullet \rightarrow M'$ given by $F(1^\ast_{\ve}) = 1_{\ve'}$ is an isomorphism of matrix factorizations that preserves the $\zed_2\oplus\zed$-grading. 

In particular, $M_\bullet \cong M'$ as graded matrix factorizations.
\end{lemma}
\begin{proof}
$F$ is clearly an isomorphism of $R$-modules. We only need to prove that $F$ is a morphism of matrix factorizations preserving the $\zed_2\oplus\zed$-grading. To simplify expressions, we use the notations $|\ve|$, $|\ve|_j$ and $\wbar{\ve_j}$ introduced in Definition \ref{ve-notation}. 

The element $1^\ast_\ve$ of $M_\bullet$ has $\zed_2$-grading $|\ve|$ and quantum grading $-\sum_{j=1}^k \ve_j(N+1-\deg a_{j,0})=\sum_{j=1}^k \ve_j(N+1-\deg a_{j,1}).$ And the element $1_{\ve'}$ of $M'$ has $\zed_2$-grading $|\ve'|=|\ve|$ and quantum grading $\sum_{j=1}^k \ve_j(N+1-\deg a_{j,1})$. So $F$ preserves the $\zed_2\oplus\zed$-grading. It remains to show that $F$ is a morphism of matrix factorizations. For $\ve=(\ve_1,\dots,\ve_k)\in I^k$, a straightforward calculation shows that 
\[
d(1^\ast_\ve)= \sum_{j=1}^{k} (-1)^{|\ve|-|\ve|_j+1}a_{j,\wbar{\ve_j}} \cdot 1^\ast_{(\ve_1,\dots,\ve_{j-1},\wbar{\ve_j},\ve_{j+1},\dots,\ve_k)},
\]
\[
d(1_{\ve'})= \sum_{j=1}^{k} (-1)^{|\ve|-|\ve|_j+1}a_{j,\wbar{\ve_j}} \cdot 1_{(\ve_k,\dots,\ve_{j+1},\wbar{\ve_j},\ve_{j-1},\dots,\ve_1)}.
\]
So $d_{M'}\circ F=F\circ d_{M_\bullet}$.
\end{proof}

\begin{lemma}\label{row-reverse-signs}
Let $R$ be a graded commutative unital $\C$-algebra, and $a_{1,0}, a_{1,1}, \dots, a_{k,0}, a_{k,1}$ homogeneous elements of $R$ with $\deg a_{j,0}+\deg a_{j,1}=2N+2$ $\forall~ j$. Let
\[
M = \left(%
\begin{array}{ll}
  a_{1,0}, & a_{1,1} \\
  a_{2,0}, & a_{2,1} \\
  \dots & \dots \\
  a_{k,0}, & a_{k,1}
\end{array}%
\right)_R
~\text{ and }~
M' = \left(%
\begin{array}{ll}
  a_{k,0}, & a_{k,1} \\
  a_{k-1,0}, & a_{k-1,1} \\
  \dots & \dots \\
  a_{1,0}, & a_{1,1}
\end{array}%
\right)_R.
\]
Define an $R$-homomorphism $F:M \rightarrow M'$ by $F(1_\ve)=(-1)^{\frac{|\ve|(|\ve|-1)}{2}} 1_{\ve'}$  $\forall ~\ve \in I^k$. (See Definition \ref{ve-notation} for the definitions of $|\ve|$ and $\ve'$.) Then $F$ is an isomorphism of matrix factorizations that preserves the $\zed_2\oplus\zed$-grading.

In particular, $M \cong M'$ as graded matrix factorizations.
\end{lemma}
\begin{proof}
$F$ is clearly an isomorphism of $R$-modules and preserves the $\zed_2\oplus\zed$-grading. It remains to show that $F$ is a morphism of matrix factorizations. When $k=1$, there is nothing to prove. When $k=2$, $F$ is given by the following diagram
\[
\begin{CD}
\left(%
\begin{array}{l}
  R\cdot 1_{(0,0)} \\
  R\cdot 1_{(1,1)}  
\end{array}%
\right)
@>{\left(%
\begin{array}{ll}
  a_{1,0} & -a_{2,1} \\
  a_{2,0} & a_{1,1}
\end{array}%
\right)}>>
\left(%
\begin{array}{l}
  R\cdot 1_{(1,0)} \\
  R\cdot 1_{(0,1)}  
\end{array}%
\right)
@>{\left(%
\begin{array}{ll}
  a_{1,1} & a_{2,1} \\
  -a_{2,0} & a_{1,0}
\end{array}%
\right)}>> 
\left(%
\begin{array}{l}
  R\cdot 1_{(0,0)} \\
  R\cdot 1_{(1,1)}  
\end{array}%
\right) \\
@VV{\left(%
\begin{array}{ll}
  1 & 0 \\
  0 & -1
\end{array}%
\right)}V 
@VV{\left(%
\begin{array}{ll}
  0 & 1 \\
  1 & 0
\end{array}%
\right)}V 
@VV{\left(%
\begin{array}{ll}
  1 & 0 \\
  0 & -1
\end{array}%
\right)}V \\
\left(%
\begin{array}{l}
  R\cdot 1_{(0,0)} \\
  R\cdot 1_{(1,1)}  
\end{array}%
\right)
@>{\left(%
\begin{array}{ll}
  a_{2,0} & -a_{1,1} \\
  a_{1,0} & a_{2,1}
\end{array}%
\right)}>>
\left(%
\begin{array}{l}
  R\cdot 1_{(1,0)} \\
  R\cdot 1_{(0,1)}  
\end{array}%
\right)
@>{\left(%
\begin{array}{ll}
  a_{2,1} & a_{1,1} \\
  -a_{1,0} & a_{2,0}
\end{array}%
\right)}>> 
\left(%
\begin{array}{l}
  R\cdot 1_{(0,0)} \\
  R\cdot 1_{(1,1)}  
\end{array}%
\right)
\end{CD}
\]
where the first row is $M$, the second row is $M'$, and $F$ is given by the vertical arrows. A direct computation shows that $F$ is a morphism. The general $k \geq2$ case follows from the $k=2$ case by a straightforward induction using Lemma \ref{morphism-sign}.
\end{proof}

\begin{lemma}\label{column-reverse-signs}
Let $R$ be a graded commutative unital $\C$-algebra, and $a_{1,0}, a_{1,1}, \dots, a_{k,0}, a_{k,1}$ homogeneous elements of $R$ with $\deg a_{j,0}+\deg a_{j,1}=2N+2$ $\forall~ j$. Let
\[
M = \left(%
\begin{array}{ll}
  a_{1,0}, & a_{1,1} \\
  a_{2,0}, & a_{2,1} \\
  \dots & \dots \\
  a_{k,0}, & a_{k,1}
\end{array}%
\right)_R
~\text{ and }~
M' =  \left(%
\begin{array}{ll}
  a_{1,1}, & a_{1,0} \\
  a_{2,1}, & a_{2,0} \\
  \dots & \dots \\
  a_{k,1}, & a_{k,0}
\end{array}%
\right)_R.
\]
For $\ve=(\ve_1,\dots,\ve_k)\in I^k$, write $s(\ve)=\sum_{j=1}^{k-1}(k-j)\ve_j$. Define an $R$-homomorphism $F:M \rightarrow M'$ by $F(1_\ve)=(-1)^{|\ve|+s(\ve)} 1_{\wbar{\ve}}$  $\forall ~\ve \in I^k$. (See Definition \ref{ve-notation} for the definitions of $|\ve|$ and $\wbar{\ve}$.) Then $F$ is an isomorphism of matrix factorizations of $\zed_2$-degree $k$ and quantum degree $\sum_{j=1}^k (N+1-\deg a_{j,1})$.
\end{lemma}
\begin{proof}
$F$ is clearly an isomorphism of $R$-modules. And the claims about its two gradings are easy to verify. One only needs to check that $F$ is a morphism of matrix factorization. This is easy when $k=1$. The general $k\geq 1$ case follows from the $k=1$ case by a straightforward induction using Lemma \ref{morphism-sign}.
\end{proof}

\subsection{Elementary operations on Koszul matrix factorizations} Khovanov and Rozansky \cite{KR1,KR2} and Rasmussen \cite{Ras2} introduced several elementary operations on Koszul matrix factorizations that give isomorphic or homotopic graded matrix factorizations. In this subsection, we recall these operations.

\begin{lemma}\cite{Ras2,Wu7}\label{general-twist}
Let $M$ be the graded matrix factorization
\[
M_0 \xrightarrow{d_0} M_1 \xrightarrow{d_1} M_0.
\]
over $R$ with potential $w$. Suppose that $H_i:M_i\rightarrow M_i$ are graded homomorphisms with $H_i^2=0$. Define $\tilde{d}_i:M_i\rightarrow M_{i+1}$ by 
\[
\tilde{d}_i= (\id_{M_{i+1}}-H_{i+1})\circ d_i \circ (\id_{M_{i}}+H_{i}),
\]
and $\widetilde{M}$ by 
\[
M_0 \xrightarrow{\tilde{d}_0} M_1 \xrightarrow{\tilde{d}_1} M_0.
\]
Then $\widetilde{M}$ is also a graded matrix factorization over $R$ with potential $w$. And $M\cong\widetilde{M}$.
\end{lemma}

\begin{corollary}\cite{Ras2}\label{twist}
Suppose $a_{1,0},a_{1,1},a_{2,0},a_{2,1},k$ are homogeneous elements in $R$ satisfying $\deg a_{j,0}+\deg a_{j,1}=2N+2$ and $\deg k = \deg a_{1,0} +\deg a_{2,0} -2N-2$. Then
\[
\left(%
\begin{array}{cc}
  a_{1,0} & a_{1,1} \\
  a_{2,0} & a_{2,1} 
\end{array}%
\right)_R
\cong
\left(%
\begin{array}{cc}
  a_{1,0}+ka_{2,1} & a_{1,1} \\
  a_{2,0}-ka_{1,1} & a_{2,1} 
\end{array}%
\right)_R.
\]
\end{corollary}

\begin{corollary}\cite{KR1,Ras2}\label{row-op} 
Suppose $a_{1,0},a_{1,1},a_{2,0},a_{2,1},c$ are homogeneous elements in $R$ satisfying $\deg a_{j,0}+\deg a_{j,1}=2N+2$ and $\deg c = \deg a_{1,0} -\deg a_{2,0}$. Then 
\[
\left(%
\begin{array}{cc}
  a_{1,0} & a_{1,1} \\
  a_{2,0} & a_{2,1} 
\end{array}%
\right)_R
\cong
\left(%
\begin{array}{cc}
  a_{1,0}+ca_{2,0} & a_{1,1} \\
  a_{2,0} & a_{2,1}-ca_{1,1} 
\end{array}%
\right)_R.
\]
\end{corollary}

The proofs of the above can be found in \cite{KR1,KR2,Ras2,Wu7} and are omitted.

\begin{definition}\label{regular-sequence}
Let $R$ be a commutative ring, and $a_1,\dots,a_k \in R$. The sequence $\{a_1,\dots,a_k\}$ is called $R$-regular if $a_1$ is not a zero divisor in $R$ and  $a_j$ is not a zero divisor in $R/(a_1,\dots,a_{j-1})$ for $j=2,\dots,k$.
\end{definition}

The next lemma is \cite[Theorem 2.1]{KR3} and a generalization of \cite[Lemma 3.10]{Ras2}.

\begin{lemma}\cite{KR3,Ras2}\label{freedom}
Let $R$ be a graded commutative unital $\C$-algebra. Suppose that $\{a_1,\dots,a_k\}$ is an $R$-regular sequence of homogeneous elements of $R$ with $\deg a_j \leq 2N+2$ $\forall ~j=1,\dots,k$. Assume that $f_1,\dots,f_k,g_1,\dots,g_k$ are homogeneous elements of $R$ such that $\deg f_j = \deg g_j = 2N+2 -\deg a_j$ and $\sum_{j=1}^k f_ja_j=\sum_{j=1}^kg_ja_j$. Then
\[
\left(%
\begin{array}{ll}
  f_1, & a_1 \\
  \dots & \dots \\
  f_k, & a_k
\end{array}%
\right)_R
\cong
\left(%
\begin{array}{ll}
  g_1, & a_1 \\
  \dots & \dots \\
  g_k, & a_k
\end{array}%
\right)_R.
\]
\end{lemma}
\begin{proof}
Induct on $k$. If $k=1$, then $a_1$ is not a zero divisor in $R$ and $(f_1-g_1)a_1=0$. So $f_1=g_1$ and $(f_1,a_1)_R=(g_1,a_1)_R$. Assume that the lemma is true for $k=m$. Consider the case $k=m+1$. $a_{m+1}$ is not a zero divisor in $R/(a_1\dots,a_m)$. But 
\[
(f_{m+1}-g_{m+1})a_{m+1}=\sum_{j=1}^m (g_j-f_j)a_j \in (a_1\dots,a_m).
\]
So $f_{m+1}-g_{m+1} \in (a_1\dots,a_m)$, that is, there exist $c_1,\dots,c_m \in R$ such that 
\[
f_{m+1}-g_{m+1} = \sum_{j=1}^m c_j a_j.
\]
Thus, by Corollary \ref{twist},
\[
\left(%
\begin{array}{ll}
  f_1, & a_1 \\
  \dots & \dots \\
  f_m, & a_m \\
  f_{m+1}, & a_{m+1}
\end{array}%
\right)_R
\cong
\left(%
\begin{array}{ll}
  f_1+c_1a_{m+1}, & a_1 \\
  \dots & \dots \\
  f_m+c_ma_{m+1}, & a_m \\
  g_{m+1}, & a_{m+1}
\end{array}%
\right)_R.
\]
It is easy to see that
\[
\sum_{j=1}^m (f_j+c_ja_{m+1})a_j = \sum_{j=1}^m g_ja_j.
\]
By induction hypothesis,
\[
\left(%
\begin{array}{ll}
  f_1+c_1a_{m+1}, & a_1 \\
  \dots & \dots \\
  f_m+c_ma_{m+1}, & a_m 
\end{array}%
\right)_R
\cong
\left(%
\begin{array}{ll}
  g_1, & a_1 \\
  \dots & \dots \\
  g_m, & a_m
\end{array}%
\right)_R.
\]
Therefore,
\[
\left(%
\begin{array}{ll}
  f_1, & a_1 \\
  \dots & \dots \\
  f_m, & a_m \\
  f_{m+1}, & a_{m+1}
\end{array}%
\right)_R
\cong
\left(%
\begin{array}{ll}
  f_1+c_1a_{m+1}, & a_1 \\
  \dots & \dots \\
  f_m+c_ma_{m+1}, & a_m \\
  g_{m+1}, & a_{m+1}
\end{array}%
\right)_R
\cong
\left(%
\begin{array}{ll}
  g_1, & a_1 \\
  \dots & \dots \\
  g_m, & a_m \\
  g_{m+1}, & a_{m+1}
\end{array}%
\right)_R.
\]
\end{proof}

Next we give six versions of \cite[Proposition 9]{KR1}, which give a method of simplifying matrix factorizations. Their proofs also give a method of finding cycles representing a given homology class in some chain complexes and finding morphisms of matrix factorizations representing a given homotopy class, which is important for our purpose. So we include their full proofs here.

\begin{proposition}[strong version]\label{b-contraction}
Let $R$ be a graded commutative unital $\C$-algebra, and $x$ a homogeneous indeterminate with $\deg x \leq 2N+2$. Let $P:R[x]\rightarrow R$ be the evaluation map at $x=0$, that is, $P(f(x))=f(0)$ $\forall ~ f(x) \in R[x]$.

Suppose that $a_1,\dots,a_k,b_1,\dots,b_k$ are homogeneous elements of $R[x]$ such that 
\begin{itemize}
	\item $\deg a_j +\deg b_j = 2N+2$ $\forall~j=1,\dots,k$,
	\item $\sum_{j=1}^k a_jb_j \in R$,
	\item $\exists~ i\in \{1,\dots,k\}$ such that $b_i=x$.
\end{itemize}
Then
\[
M=\left(%
\begin{array}{cc}
  a_1 & b_1 \\
  a_2 & b_2 \\
  \dots & \dots \\
  a_k & b_k
\end{array}%
\right)_{R[x]}
\text{ and }
M'=\left(%
\begin{array}{cc}
  P(a_1) & P(b_1) \\
  P(a_2) & P(b_2) \\
  \dots & \dots \\
  P(a_{i-1}) & P(b_{i-1}) \\
  P(a_{i+1}) & P(b_{i+1}) \\
  \dots & \dots \\
  P(a_k) & P(b_k)
\end{array}%
\right)_{R}
\]
are homotopic as graded matrix factorizations over $R$.
\end{proposition}
\begin{proof}
For $j\neq i$, Write $a'_j=P(a_j)\in R$ and $b'_j=P(b_j)\in R$. Then $\exists! ~c_j, k_j\in R[x]$ such that $a_j=a'_j+k_jx$ and $b_j=b'_j+c_jx$. By Corollaries \ref{twist} and \ref{row-op},
\[
M \cong M'' := 
\left(%
\begin{array}{cc}
  a'_1 & b'_1 \\
  \dots & \dots \\
  a'_{i-1} & b'_{i-1} \\
  a & x \\
  a'_{i+1} & b'_{i+1} \\
  \dots & \dots \\
  a'_k & b'_k
\end{array}%
\right)_{R[x]},
\]
where $a= a_i + \sum_{j\neq i}k_jb_j + \sum_{j\neq i} c_ja'_j$. Since $M,M''$ have the same potential, we know that $ax= \sum_{j=1}^k a_jb_j -\sum_{j \neq i}a'_jb'_j \in R$. So $a=0$. Thus, 
\[
M'' = 
\left(%
\begin{array}{cc}
  a'_1 & b'_1 \\
  \dots & \dots \\
  a'_{i-1} & b'_{i-1} \\
  0 & x \\
  a'_{i+1} & b'_{i+1} \\
  \dots & \dots \\
  a'_k & b'_k
\end{array}%
\right)_{R[x]}.
\]

Define $R$-module homomorphisms $F:M''\rightarrow M'$ and $G:M'\rightarrow M''$ by
\[
F(f(x)1_\ve) = \left\{%
\begin{array}{ll}
    f(0) 1_{(\ve_1,\dots,\ve_{i-1},\ve_{i+1},\dots,\ve_k)} & \text{if } \ve_i=0, \\
    0 & \text{if } \ve_i = 1, 
\end{array}%
\right.
\]
for $f(x)\in R[x]$ and $\ve=(\ve_1,\dots,\ve_k)\in I^k$ (see Definition \ref{ve-notation} for the definition of $1_\ve$) and
\[
G(r 1_{(\ve_1,\dots,\ve_{i-1},\ve_{i+1},\dots,\ve_k)}) = r 1_{(\ve_1,\dots,\ve_{i-1},0,\ve_{i+1},\dots,\ve_k)}
\] 
for $r\in R$ and $(\ve_1,\dots,\ve_{i-1},\ve_{i+1},\dots,\ve_k) \in I^{k-1}$. 

One can easily check that $F$ and $G$ are morphisms of matrix factorizations preserving the $\zed_2\oplus\zed$-grading and $F \circ G= \id_{M'}$. Note that $M''=\ker F \oplus \im G$ and
\begin{eqnarray*}
G\circ F|_{\ker F} & = & 0, \\
G\circ F|_{\im G} & = & \id_{\im G}.
\end{eqnarray*}
Define an $R$-module homomorphism $h:M''\rightarrow M''$ by 
\begin{eqnarray*}
h(1_{(\ve_1,\dots,\ve_{i-1},1,\ve_{i+1},\dots,\ve_k)}) & = & 0, \\
h((r+xf(x)) 1_{(\ve_1,\dots,\ve_{i-1},0,\ve_{i+1},\dots,\ve_k)}) & = & (-1)^{\sum_{j=1}^{i-1}\ve_j} f(x) 1_{(\ve_1,\dots,\ve_{i-1},1,\ve_{i+1},\dots,\ve_k)}
\end{eqnarray*}
for $r\in R$, $f(x) \in R[x]$ and $\ve_1,\dots,\ve_{i-1},\ve_{i+1},\dots,\ve_k \in I$. A straightforward computation shows that
\begin{eqnarray*}
(d\circ h + h \circ d)|_{\ker F} & = & \id_{\ker F}, \\
(d\circ h + h \circ d)|_{\im G} & = & 0.
\end{eqnarray*}
So $\id_{M''} - G\circ F = d\circ h + h \circ d$. Thus, we have $M'' \simeq M'$ and, therefore, $M \simeq M'$ as graded matrix factorizations over $R$.
\end{proof}

\begin{proposition}[weak version]\label{b-contraction-weak}
Let $R$ be a graded commutative unital $\C$-algebra, and $a_1,\dots,a_k,b_1,\dots,b_k$ homogeneous elements of $R$ such that $\deg a_j +\deg b_j = 2N+2$ and $\sum_{j=1}^k a_jb_j=0$. Then the matrix factorization 
\[
M=\left(%
\begin{array}{cc}
  a_1 & b_1 \\
  a_2 & b_2 \\
  \dots & \dots \\
  a_k & b_k
\end{array}%
\right)_R
\]
is a chain complex with a $\zed_2$ homological grading. Assume that, for a given $i\in \{1,\dots,k\}$, $b_i$ is not a zero divisor in $R$. Define $R'=R/(b_i)$, which inherits the grading of $R$. Let $P:R\rightarrow R'$ be the standard projection. Then 
\[
M'=\left(%
\begin{array}{cc}
  P(a_1) & P(b_1) \\
  P(a_2) & P(b_2) \\
  \dots & \dots \\
  P(a_{i-1}) & P(b_{i-1}) \\
  P(a_{i+1}) & P(b_{i+1}) \\
  \dots & \dots \\
  P(a_k) & P(b_k)
\end{array}%
\right)_{R'}
\]
is also a chain complex with a $\zed_2$ homological grading. And $H(M)\cong H(M')$ as $\zed_2\oplus\zed$-graded $R$-modules.
\end{proposition}
\begin{proof}
Define an $R$-module homomorphism $F:M\rightarrow M'$ by
\[
F(r1_\ve) = \left\{%
\begin{array}{ll}
    P(r) 1_{(\ve_1,\dots,\ve_{i-1},\ve_{i+1},\dots,\ve_k)} & \text{if } \ve_i=0, \\
    0 & \text{if } \ve_i = 1, 
\end{array}%
\right.
\]
for $r\in R$ and $\ve=(\ve_1,\dots,\ve_k)\in I^k$. (See Definition \ref{ve-notation} for the definition of $1_\ve$.) It is easy to check that $F$ is a surjective morphism of matrix factorizations preserving the $\zed_2\oplus\zed$-grading. The kernel of $F$ is the subcomplex
\[
\ker F = \bigoplus_{(\ve_1,\dots,\ve_{i-1},\ve_{i+1},\dots,\ve_k) \in I^{k-1}}(R \cdot 1_{(\ve_1,\dots,\ve_{i-1},1,\ve_{i+1},\dots,\ve_k)} \oplus b_iR \cdot 1_{(\ve_1,\dots,\ve_{i-1},0,\ve_{i+1},\dots,\ve_k)}).
\] 
Since $b_i$ is not a zero divisor, the division map $\varphi:b_iR\rightarrow R$ given by $\varphi(b_ir)=r$ is well defined. Define an $R$-module homomorphism $h:\ker F \rightarrow \ker F$ by 
\begin{eqnarray*}
h(1_{(\ve_1,\dots,\ve_{i-1},1,\ve_{i+1},\dots,\ve_k)}) & = & 0, \\
h(b_i 1_{(\ve_1,\dots,\ve_{i-1},0,\ve_{i+1},\dots,\ve_k)}) & = & (-1)^{\sum_{j=1}^{i-1}\ve_j} 1_{(\ve_1,\dots,\ve_{i-1},1,\ve_{i+1},\dots,\ve_k)}.
\end{eqnarray*}
Then 
\[
d|_{\ker F} \circ h+h \circ d|_{\ker F}=\id_{\ker F},
\] 
where $d$ is the differential map of $M$. In particular, this means that $H(\ker F)=0$. Then, using the long exact sequence induced by
\[
0\rightarrow \ker F \rightarrow M \xrightarrow{F} M' \rightarrow 0,
\]
it is easy to see that $F$ is a quasi-isomorphism.
\end{proof}

\begin{remark}\label{reverse-b-contraction}
The above proof of Proposition \ref{b-contraction-weak} also gives a method of finding cycles in $M$ whose image under $F$ is a given cycle in $M'$. Indeed, for every cycle $\alpha$ in $M'$, one can find an element $\beta \in M$ such that $F(\beta)=\alpha$. Then $F(d \beta)=d' F(\beta)=d' \alpha=0$, where $d'$ is the differential map of $M'$. So $d \beta \in \ker F$ and $d \beta = d h(d \beta)+hd(d \beta)=d h(d \beta)$. Thus, $\beta - h(d\beta)$ is a cycle in $M$. By definition, it clear that $F\circ h=0$. So $F(\beta - h(d\beta))=\alpha$. This observation is useful in finding cycles representing a given homology class and morphisms representing a given homotopy class. (In the situation in Proposition \ref{b-contraction}, one can also do the same by explicitly computing the morphism $M' \xrightarrow{\simeq} M'' \xrightarrow{\cong} M$, which is usually not any easier in practice.) This method also applies to the situation in corollaries \ref{a-contraction} and \ref{a-contraction-weak}, that is, contracting the matrix factorization using an entry in the left column. 
\end{remark}

Next we give the dual version of Proposition \ref{b-contraction-weak}.

\begin{corollary}[dual version]\label{b-contraction-dual}
Let $R$ be a graded commutative unital $\C$-algebra, and $\hat{R}$ a graded commutative unital sub-algebra of $R$ such that $R$ is a free $\hat{R}$-module. Suppose that $a_1,\dots,a_k,b_1,\dots,b_k$ are homogeneous elements of $R$ such that $\deg a_j +\deg b_j = 2N+2$ and $\sum_{j=1}^k a_jb_j = w \in \hat{R}$. Assume that, for a given $i\in \{1,\dots,k\}$, $b_i$ is not a zero divisor in $R$ and $R'=R/(b_i)$ is also a free $\hat{R}$-module. Define
\[
M=\left(%
\begin{array}{cc}
  a_1 & b_1 \\
  a_2 & b_2 \\
  \dots & \dots \\
  a_k & b_k
\end{array}%
\right)_R
\]
and
\[
M'=\left(%
\begin{array}{cc}
  P(a_1) & P(b_1) \\
  P(a_2) & P(b_2) \\
  \dots & \dots \\
  P(a_{i-1}) & P(b_{i-1}) \\
  P(a_{i+1}) & P(b_{i+1}) \\
  \dots & \dots \\
  P(a_k) & P(b_k)
\end{array}%
\right)_{R'},
\]
where $P:R\rightarrow R'$ is the standard projection. Then, for any matrix factorization $M''$ over $\hat{R}$ with potential $w$, there is a homogeneous quasi-isomorphism 
\[
\Hom_{\hat{R}}(M',M'') \rightarrow \Hom_{\hat{R}}(M,M'')
\]
preserving both the $\zed_2$-grading and the quantum pregrading.
\end{corollary}
\begin{proof}
Define an $R$-module homomorphism $F:M\rightarrow M'$ by
\[
F(r1_\ve) = \left\{%
\begin{array}{ll}
    P(r) 1_{(\ve_1,\dots,\ve_{i-1},\ve_{i+1},\dots,\ve_k)} & \text{if } \ve_i=0, \\
    0 & \text{if } \ve_i = 1, 
\end{array}%
\right.
\]
for $r\in R$ and $\ve=(\ve_1,\dots,\ve_k)\in I^k$. (See Definition \ref{ve-notation} for the definition of $1_\ve$.) Then $F$ is a surjective morphism of matrix factorizations preserving the $\zed_2\oplus\zed$-grading. So we have a short exact sequence
\[
0 \rightarrow \ker F \rightarrow M \xrightarrow{F} M' \rightarrow 0.
\]
Note that $\ker F$ and $M$ are free $R$-modules and $M'$ is a free $R'$-module. Thus, the above is a short exact sequence of free $\hat{R}$-modules. This implies that
\[
0 \rightarrow \Hom_{\hat{R}}(M',M'') \xrightarrow{F^{\sharp}} \Hom_{\hat{R}}(M,M'') \rightarrow \Hom_{\hat{R}}(\ker F,M'') \rightarrow 0
\]
is also exact. Similar to the proof of Proposition \ref{b-contraction-weak}, there exists an $R$-module map $h:\ker F \rightarrow \ker F$ of $\zed_2$-degree $1$ such that $\id_{\ker F} = d_M|_{\ker F}\circ h + h\circ d_M|_{\ker F}$. Define 
\[
H:\Hom_{\hat{R}}(\ker F,M'') \rightarrow \Hom_{\hat{R}}(\ker F,M'')
\] 
by $H(f)=(-1)^jf\circ h$ if $f$ has $\zed_2$-degree $j$. Then $H$ has $\zed_2$-degree $1$ and, for $f\in \Hom_{\hat{R}}(\ker F,M'')$ of $\zed_2$-degree $j$, 
\begin{eqnarray*}
& & (d\circ H + H \circ d) (f) \\
& = & d(H(f))+ H(d(f)) \\
& = & (-1)^j d(f\circ h) + (-1)^{j+1} (df) \circ h \\ 
& = & (-1)^j(d_{M''} \circ f \circ h -(-1)^{j+1}f \circ h \circ d_M|_{\ker F}) + (-1)^{j+1} (d_{M''} \circ f \circ h - (-1)^j f \circ d_M|_{\ker F} \circ h) \\
& = & f \circ (d_M|_{\ker F}\circ h + h\circ d_M|_{\ker F}) =f.
\end{eqnarray*}
This shows that $d\circ H + H \circ d = \id_{\Hom_{\hat{R}}(\ker F,M'')}$. Thus, $\Hom_{\HMF}(\ker F,M'')=0$ and, therefore, 
\[
F^{\sharp}: \Hom_{\hat{R}}(M',M'') \rightarrow \Hom_{\hat{R}}(M,M'')
\]
is a quasi-isomorphism preserving the $\zed_2\oplus\zed$-grading.
\end{proof}

\begin{remark}\label{reverse-b-contraction-dual}
Note that $F^{\sharp}: \Hom_{\hat{R}}(M',M'') \rightarrow \Hom_{\hat{R}}(M,M'')$ maps a morphism of matrix factorizations to a morphism of matrix factorizations. By successively using this map, we can sometimes find morphisms representing a given homotopy class. This method also applies to Corollary \ref{a-contraction-dual}.
\end{remark}

The following three corollaries describe how to contract a Koszul matrix factorization using an entry in the left column. Their proofs are very close to that of propositions \ref{b-contraction}, \ref{b-contraction-weak} and \ref{b-contraction-dual}, and are omitted.

\begin{corollary}[strong version]\label{a-contraction}
Let $R$ be a graded commutative unital $\C$-algebra, and $x$ a homogeneous indeterminate with $\deg x \leq 2N+2$. Let $P:R[x]\rightarrow R$ be the evaluation map at $x=0$, that is, $P(f(x))=f(0)$ $\forall ~ f(x) \in R[x]$.

Suppose that $a_1,\dots,a_k,b_1,\dots,b_k$ are homogeneous elements of $R[x]$ such that 
\begin{itemize}
	\item $\deg a_j +\deg b_j = 2N+2$ $\forall~j=1,\dots,k$,
	\item $\sum_{j=1}^k a_jb_j \in R$,
	\item $\exists~ i\in \{1,\dots,k\}$ such that $a_i=x$.
\end{itemize}
Then
\[
M=\left(%
\begin{array}{cc}
  a_1 & b_1 \\
  a_2 & b_2 \\
  \dots & \dots \\
  a_k & b_k
\end{array}%
\right)_{R[x]}
\text{ and }
M'=\left(%
\begin{array}{cc}
  P(a_1) & P(b_1) \\
  P(a_2) & P(b_2) \\
  \dots & \dots \\
  P(a_{i-1}) & P(b_{i-1}) \\
  P(a_{i+1}) & P(b_{i+1}) \\
  \dots & \dots \\
  P(a_k) & P(b_k)
\end{array}%
\right)_{R}\{q^{N+1-\deg x}\}\left\langle 1\right\rangle
\]
are homotopic as graded matrix factorizations over $R$.
\end{corollary}

\begin{corollary}[weak version]\label{a-contraction-weak}
Let $R$ be a graded commutative unital $\C$-algebra, and $a_1,\dots,a_k,b_1,\dots,b_k$ homogeneous elements of $R$ such that $\deg a_j +\deg b_j = 2N+2$ and $\sum_{j=1}^k a_jb_j=0$. Then the matrix factorization 
\[
M=\left(%
\begin{array}{cc}
  a_1 & b_1 \\
  a_2 & b_2 \\
  \dots & \dots \\
  a_k & b_k
\end{array}%
\right)_R
\]
is a chain complex with a $\zed_2$ homological grading. Assume that, for a given $i\in \{1,\dots,k\}$, $a_i$ is not a zero divisor in $R$. Define $R'=R/(a_i)$, which inherits the grading of $R$. Let $P:R\rightarrow R'$ be the standard projection. Then 
\[
M'=\left(%
\begin{array}{cc}
  P(a_1) & P(b_1) \\
  P(a_2) & P(b_2) \\
  \dots & \dots \\
  P(a_{i-1}) & P(b_{i-1}) \\
  P(a_{i+1}) & P(b_{i+1}) \\
  \dots & \dots \\
  P(a_k) & P(b_k)
\end{array}%
\right)_{R'}
\]
is also a chain complex with a $\zed_2$ homological grading. And $H(M) \cong H(M')\{q^{N+1-\deg a_i}\}\left\langle 1\right\rangle$ as $\zed_2\oplus\zed$-graded $R$-modules.
\end{corollary}

\begin{corollary}[dual version]\label{a-contraction-dual}
Let $R$ be a graded commutative unital $\C$-algebra, and $\hat{R}$ a graded commutative unital sub-algebra of $R$ such that $R$ is a free $\hat{R}$-module. Suppose that $a_1,\dots,a_k,b_1,\dots,b_k$ are homogeneous elements of $R$ such that $\deg a_j +\deg b_j = 2N+2$ and $\sum_{j=1}^k a_jb_j = w \in \hat{R}$. Assume that, for a given $i\in \{1,\dots,k\}$, $a_i$ is not a zero divisor in $R$ and $R'=R/(a_i)$ is also a free $\hat{R}$-module. Define
\[
M=\left(%
\begin{array}{cc}
  a_1 & b_1 \\
  a_2 & b_2 \\
  \dots & \dots \\
  a_k & b_k
\end{array}%
\right)_R
\]
and
\[
M'=\left(%
\begin{array}{cc}
  P(a_1) & P(b_1) \\
  P(a_2) & P(b_2) \\
  \dots & \dots \\
  P(a_{i-1}) & P(b_{i-1}) \\
  P(a_{i+1}) & P(b_{i+1}) \\
  \dots & \dots \\
  P(a_k) & P(b_k)
\end{array}%
\right)_{R'},
\]
where $P:R\rightarrow R'$ is the standard projection. Then, for any matrix factorization $M''$ over $\hat{R}$ with potential $w$, there is a homogeneous quasi-isomorphism 
\[
\Hom_{\hat{R}}(M',M'') \rightarrow \Hom_{\hat{R}}(M,M'')
\]
of $\zed_2$-degree $1$ and quantum degree $\deg a_i - N -1$.
\end{corollary}

\subsection{Categories of homotopically finite graded matrix factorizations} $R$ is again a graded commutative unital $\C$-algebra in this subsection. 

\begin{definition}\label{homotopically-finite-def}
Let $M$ be a graded matrix factorization over $R$ with potential $w$. We say that $M$ is homotopically finite if there exists a finitely generated graded matrix factorization $\mathcal{M}$ over $R$ with potential $w$ such that $M\simeq \mathcal{M}$. 
\end{definition}

\begin{definition}\label{hom-all-gradings-def}
Let $M$ and $M'$ be any two graded matrix factorizations over $R$ with potential $w$. Denote by $d$ the differential map of $\Hom_R(M,M')$.

$\Hom_{\MF}(M,M')$ is defined to be the submodule of $\Hom_R(M,M')$ consisting of morphisms of matrix factorizations from $M$ to $M'$. Or, equivalently, $\Hom_{\MF}(M,M') := \ker d$.

$\Hom_{\HMF}(M,M')$ is defined to be the $R$-module of homotopy classes of morphisms of matrix factorizations from $M$ to $M'$. Or, equivalently,
$\Hom_{\HMF}(M,M')$ is the homology of the chain complex $(\Hom_R(M,M'),d)$.
\end{definition}

It is clear that $\Hom_{\MF}(M,M')$ and $\Hom_{\HMF}(M,M')$ inherit the $\zed_2$-grading of $\Hom_R(M,M')$. Recall that $\Hom_R(M,M')$ has a natural quantum pregrading, and $d$ is homogeneous (with $\deg d =N+1$.) So $\Hom_{\MF}(M,M')$ and $\Hom_{\HMF}(M,M')$ also inherit the quantum pregrading from $\Hom_R(M,M')$.

\begin{definition}\label{hom-zero-grading-def}
Let $M$ and $M'$ be as in Definition \ref{hom-all-gradings-def}

$\Hom_{\mf}(M,M')$ is defined to be the $\C$-linear subspace of $\Hom_{\MF}(M,M')$ consisting of homogeneous morphisms with $\zed_2$-degree $0$ and quantum degree $0$.

$\Hom_{\hmf}(M,M')$ is defined to be the $\C$-linear subspace of $\Hom_{\HMF}(M,M')$ consisting of homogeneous elements with $\zed_2$-degree $0$ and quantum degree $0$.
\end{definition}

Now we introduce four categories of homotopically finite graded matrix factorizations relevant to our construction. We require the grading of the base ring to be bounded below. We will be mainly concerned with the homotopy categories $\HMF_{R,w}$ and $\hmf_{R,w}$.

\begin{definition}\label{categories-def}
Let $R$ a graded commutative unital $\C$-algebra, whose grading is bounded below. Let $w\in R$ be an homogeneous element of degree $2N+2$.  We define $\MF_{R,w}$, $\HMF_{R,w}$, $\mf_{R,w}$ and $\hmf_{R,w}$ by the following table. 

\begin{center}
\small{
\begin{tabular}{|c|c|c|}
\hline
Category & Objects & Morphisms \\
\hline
$\MF_{R,w}$ & all homotopically finite graded matrix factorizations over   & $\Hom_{\MF}$ \\
 &  $R$ of potential $w$ with quantum gradings bounded below &  \\
\hline
$\HMF_{R,w}$ & all homotopically finite graded matrix factorizations over    & $\Hom_{\HMF}$ \\
 & $R$ of potential $w$ with quantum gradings bounded below &  \\
\hline
$\mf_{R,w}$ & all homotopically finite graded matrix factorizations over    & $\Hom_{\mf}$ \\
 & $R$ of potential $w$ with quantum gradings bounded below &  \\
\hline
$\hmf_{R,w}$ & all homotopically finite graded matrix factorizations over    & $\Hom_{\hmf}$ \\
 & $R$ of potential $w$ with quantum gradings bounded below &  \\
\hline
\end{tabular}
}
\end{center}
\end{definition}

\begin{remark} 
\begin{enumerate}[(i)]
  \item The above categories are additive.
	\item The definitions of these categories here are slightly different from those in \cite{KR1}.
  \item The grading of a finitely generated graded matrix factorization over $R$ is bounded below. So finitely generated graded matrix factorizations are objects of the above categories.
  \item Comparing Definition \ref{categories-def} to Definition \ref{cong-and-sim}, one can see that, for any objects $M$ and $M'$ of the above categories, $M \cong M'$ means they are isomorphic as objects of $\mf_{R,w}$, and $M \simeq M'$ means they are isomorphic as objects of $\hmf_{R,w}$.
\end{enumerate}
\end{remark}

\begin{lemma}\label{hom-all-gradings-homo-finite}
Let $M$ and $M'$ be any two graded matrix factorizations over $R$ with potential $w$. Assume that $M$ is homotopically finite. Then the quantum pregrading on $\Hom_{\HMF}(M,M')$ is a grading.

In particular, if the grading of $R$ is bounded below and $M$ and $M'$ are objects of $\MF_{R,w}$, then $\Hom_{\HMF}(M,M')$ has a quantum grading.
\end{lemma}
\begin{proof}
Since $M$ is homotopically finite, there is a finitely generated graded matrix factorization $\mathcal{M}$ over $R$ with potential $w$ such that $M\simeq \mathcal{M}$. That is, there exit homogeneous morphisms $f:M\rightarrow \mathcal{M}$ and $g:\mathcal{M} \rightarrow M$ preserving both the $\zed_2$-grading and the quantum grading such that $g\circ f \simeq \id_M$ and $f \circ g \simeq \id_{\mathcal{M}}$. 

Denote by $d_M$, $d_{M'}$, $d$ the differential maps of $M$, $M'$ and $\Hom_R(M,M')$. Let $f^\sharp:\Hom_R(\mathcal{M},M') \rightarrow \Hom_R(M,M')$ and $g^\sharp:\Hom_R(M,M') \rightarrow \Hom_R(\mathcal{M},M')$ be the $R$-module maps induced by $f$ and $g$. One can easily check that $f^\sharp$ and $g^\sharp$ are chain maps. Since $g\circ f \simeq \id_M$, we know that there exist a homogeneous $R$-module map $h:M\rightarrow M$ of $\zed_2$-degree $1$ and quantum degree $-N-1$ such that 
\[
g\circ f - \id_M = d_M \circ h + h \circ d_M.
\] 
Define an $R$-module map $H:\Hom_R(M,M')\rightarrow \Hom_R(M,M')$ so that, for any $\alpha \in\Hom_R(M,M')$ with $\zed_2$-degree $\ve$, $H(\alpha)=(-1)^{\ve} \alpha \circ h$. Then, for such an $\alpha$, we have 
\begin{eqnarray*}
&   & (dH+Hd)(\alpha) \\
& = & (-1)^{\ve} d(\alpha \circ h) + (-1)^{\ve+1} (d\alpha)\circ h \\
& = & (-1)^{\ve} (d_{M'}\circ\alpha \circ h -(-1)^{\ve+1} \alpha \circ h \circ d_M) + (-1)^{\ve+1} (d_{M'}\circ\alpha  -(-1)^\ve \alpha \circ d_M) \circ h \\ 
& = & \alpha \circ (h \circ d_M + d_M \circ h) = \alpha \circ (g\circ f - \id_M) \\
& = & f^\sharp \circ g^\sharp (\alpha) - \alpha.
\end{eqnarray*}
This shows that $f^\sharp \circ g^\sharp \simeq \id_{\Hom_R(M,M')}$. Similarly, $g^\sharp \circ f^\sharp \simeq \id_{\Hom_R(\mathcal{M},M')}$. Thus, $\Hom_R(M,M') \simeq \Hom_R(\mathcal{M},M')$ and this homotopy equivalence preserves both the $\zed_2$-grading and the quantum pregrading. So $\Hom_{\HMF}(M,M') \cong \Hom_{\HMF}(\mathcal{M},M')$ and the isomorphism preserves both the $\zed_2$-grading and the quantum pregrading. But, by Lemma \ref{hom-finite-gen}, the quantum pregrading of $\Hom_R(\mathcal{M},M')$ is a grading. So the quantum pregrading of $\Hom_{\HMF}(M,M') \cong \Hom_{\HMF}(\mathcal{M},M')$ is also a grading.
\end{proof}

\subsection{Categories of chain complexes} Now we introduce our notations for categories of chain complexes.

\begin{definition}\label{categories-of-complexes}
Let $\mathcal{C}$ be an additive category. We denote by $\ch(\mathcal{C})$ the category of bounded chain complexes over $\mathcal{C}$. More precisely,
\begin{itemize}
	\item An object of $\ch(\mathcal{C})$ is a chain complex 
	\begin{equation}\label{chain-complex-form}
	\xymatrix{
	\cdots  \ar[r]^{d_{i-1}} & A_i \ar[r]^{d_{i}} & A_{i+1} \ar[r]^{d_{i+1}} & A_{i+2} \ar[r]^{d_{i+2}} & \cdots 
	}
	\end{equation}
	where $A_i$'s are objects of $\mathcal{C}$, $d_i$'s are morphisms of $\mathcal{C}$ such that $d_{i+1} \circ d_i =0$ for $i\in \zed$, and there exists integers $k\leq K$ such that $A_i =0$ if $i>K$ or $i<k$.
	\item A morphism $f$ of $\ch(\mathcal{C})$ is a commutative diagram
	\[
	\xymatrix{
	\cdots  \ar[r]^{d_{i-1}} & A_i \ar[r]^{d_{i}} \ar[d]_{f_{i}} & A_{i+1} \ar[r]^{d_{i+1}} \ar[d]_{f_{i+1}} & A_{i+2} \ar[r]^{d_{i+2}} \ar[d]_{f_{i+2}} & \cdots \\
		\cdots  \ar[r]^{d'_{i-1}} & A'_i \ar[r]^{d'_{i}} & A'_{i+1} \ar[r]^{d'_{i+1}} & A'_{i+2} \ar[r]^{d'_{i+2}}  & \cdots
	},
	\]
	where each row is an object of $\ch(\mathcal{C})$ and vertical arrows are morphisms of $\mathcal{C}$.
\end{itemize}
Chain homotopy in $\ch(\mathcal{C})$ is defined the usual way.

We denote by $\hch(\mathcal{C})$ the homotopy category of chain complexes over $\mathcal{C}$, or simply the homotopy category of $\mathcal{C}$. $\hch(\mathcal{C})$ is defined by 
\begin{itemize}
	\item An object of $\hch(\mathcal{C})$ is an object of $\ch(\mathcal{C})$.
	\item For any two objects $A$ and $B$ of $\hch(\mathcal{C})$, $\Hom_{\hch(\mathcal{C})}(A,B)$ is $\Hom_{\ch(\mathcal{C})}(A,B)$ modulo the subgroup of null homotopic morphisms.
\end{itemize}

An isomorphism in $\ch(\mathcal{C})$ is denoted by ``$\cong$". An isomorphism in $\hch(\mathcal{C})$ is commonly known as a homotopy equivalence and denoted by ``$\simeq$". 

Let $A$ be the object of $\ch(\mathcal{C})$ (and $\hch(\mathcal{C})$) given in \eqref{chain-complex-form}. Then $A$ admits an obvious bounded homological grading $\deg_h$ with $\deg_h A_i =i$.  Morphisms of $\ch(\mathcal{C})$ and $\hch(\mathcal{C})$ preserve this grading. Denote by $A\| k \|$ the object of $\ch(\mathcal{C})$ obtained by shifting the homological grading by $k$. That is, $A\| k \|$ is the same chain complex as $A$ except that $\deg_h A_i =i+k$ in $A\| k \|$.
\end{definition}

Let us try to understand how to compute  $\Hom_{\ch(\mathcal{C})}(A,B)$ and $\Hom_{\hch(\mathcal{C})}(A,B)$ for objects $A,~B$ of $\ch(\mathcal{C})$.

\begin{definition}\label{hom-kom-def}
Let $\mathcal{C}$ be an additive category, and $(A,d)$, $(B,d')$ objects of $\ch(\mathcal{C})$. Let $\Kom^0 (A,B)$ be the set of diagrams of the form
\[
	\xymatrix{
	\cdots  \ar[r]^{d_{i-1}} & A_i \ar[r]^{d_{i}} \ar[d]_{f_{i}} & A_{i+1} \ar[r]^{d_{i+1}} \ar[d]_{f_{i+1}} & A_{i+2} \ar[r]^{d_{i+2}} \ar[d]_{f_{i+2}} & \cdots \\
		\cdots  \ar[r]^{d'_{i-1}} & B_i \ar[r]^{d'_{i}} & B_{i+1} \ar[r]^{d'_{i+1}} & B_{i+2} \ar[r]^{d'_{i+2}}  & \cdots
	},
\]
where vertical arrows are morphisms of $\mathcal{C}$, and we do not require any commutativity. Note that $\Kom^0 (A,B)$ is an abelian group. 

For any $k\in\zed$, define $\Kom^k (A,B):= \Kom^0 (A\| k \|,B)$. Note that, if $f \in \Kom^k (A,B)$, then $D_k f:= f\circ d - (-1)^k d' \circ f$ is an element of $\Kom^{k+1} (A,B)$. Clearly, 
\[
(\Kom (A,B):= \bigoplus_{k\in\zed} \Kom^{k} (A,B), ~D:= \bigoplus_{k\in\zed} D_k)
\]
is a bounded chain complex of abelian groups with an obvious homological grading, in which $\Kom^k (A,B)$ has grading $k$.
\end{definition}

The following lemma is obvious from the definitions of $\Hom_{\ch(\mathcal{C})}(A,B)$ and $\Hom_{\hch(\mathcal{C})}(A,B)$.

\begin{lemma}\label{computing-hom-ch-hch}
Using notations from Definition \ref{hom-kom-def}, we have 
\begin{eqnarray*}
\Hom_{\ch(\mathcal{C})}(A,B) & = & \ker D_0, \\
\Hom_{\hch(\mathcal{C})}(A,B) & = & H^0(\Kom (A,B), D) =\ker D_0 / \im D_{-1}.
\end{eqnarray*}
\end{lemma}

\section{Graded Matrix Factorizations over a Polynomial Ring}\label{sec-mf-polynomial}

In this section, we review the algebraic properties of graded matrix factorizations over polynomial rings. Most of these properties can be found in \cite{KR1}.

In the rest of this section, we assume $R=\C[X_1,\dots,X_m]$ is a polynomial ring over $\C$, where $X_1,\dots,X_m$ are homogeneous indeterminates of positive integer degrees. There is a natural grading $\{R^{(i)}\}$ of $R$. It is clear that, for each $i$, $R^{(i)}$ is finite dimensional. In particular, $R^{(i)}=0$ if $i<0$ and $R^{(0)}=\C$. Also, $R$ has a unique maximal homogeneous ideal $\mathfrak{I} = (X_1,\dots,X_m)$.

\begin{definition}
For a homogeneous element $w\in \mathfrak{I}$ of degree $2N+2$, the Jacobian ideal of $w$ is defined to be $I_w=(\frac{\partial w}{\partial X_1},\dots, \frac{\partial w}{\partial X_m})$. We call $w$ non-degenerate if the Jacobian algebra $R_w:= R/I_w$ is finite dimensional over $\C$. Otherwise, we call $w$ degenerate.\footnote{In \cite{KR1}, the word ``potential" refers to a non-degenerate element of $\mathfrak{I}^2$. We adopt a more relaxed convention here. A potential can be any element of $\mathfrak{I}$, degenerate or non-degenerate.} 
\end{definition}

Note that, for any homogeneous element $w\in \mathfrak{I}$ of degree $2N+2$, Euler's formula gives that
\[
w = \frac{1}{2N+2} \sum_{i=1}^m (\deg X_i) \cdot X_i\frac{\partial w}{\partial X_i}.
\]
Thus, $w$ is in its Jacobian ideal.

\begin{lemma}\cite[Propositions 5]{KR1}\label{action-factor-thru-jacobian-ring}
Let $M$ and $M'$ be objects of $\HMF_{R,w}$. Then the action of $R$ on $\Hom_{\HMF}(M,M')$ factors through the Jocobian ring $R_w$.
\end{lemma}
\begin{proof} (Following \cite{KR1}.)
Choose a basis for $M$ and express the differential $d$ of $M$ as a matrix $D$. Differentiating $D^2=w\cdot \id$ by $X_i$, we get $\frac{\partial D}{\partial X_i} \circ D + D \circ \frac{\partial D}{\partial X_i} = \frac{\partial w}{\partial X_i} \cdot \id$. So multiplication by $\frac{\partial w}{\partial X_i}$ on $M$ is a morphism homotopic to $0$. Thus multiplication by $\frac{\partial w}{\partial X_i}$ on $\Hom_{\HMF}(M,M')$ is the zero map.
\end{proof}

\subsection{Homogeneous basis} In general, a free graded module over a graded ring is not necessarily graded-free, that is, need not have a basis consisting of homogeneous elements. (See Definition \ref{graded-free-def}.) However, if the base ring is $R$, and the grading on the free module is bounded below, then the module has a homogeneous basis. We prove this using argument in \cite[Chapter 13]{Passman-book}. First, we introduce the following definition from \cite[Chapter 13]{Passman-book}.

\begin{definition}\label{graded-proj-def}
Let $P$ be a graded $R$-module. We say that $P$ is graded-projective if, whenever we have a diagram
\[
\xymatrix{
& P \ar[d]^\beta &  \\
V \ar[r]^\alpha & W \ar[r] & 0 \\
}
\]
of graded $R$-modules with exact row, where $\alpha$ and $\beta$ are homogeneous $R$-module maps preserving the grading, there exists a homogeneous $R$-module map $\gamma: P \rightarrow V$ that preserves the grading and makes the following diagram commutative.
\[
\xymatrix{
& P \ar[d]^\beta \ar@{-->}[dl]_\gamma &  \\
V \ar[r]^\alpha & W \ar[r] & 0 \\
}
\]
\end{definition}

\begin{lemma}\cite{Passman-book}\label{homogeneous-basis-exists}
Let $M$ be a free graded $R$-module whose grading is bounded below. Then $M$ is graded-free over $R$, that is, $M$ admits a homogeneous basis over $R$.

In particular, for any homogeneous element $w\in \mathfrak{I}$ of degree $2N+2$, every object of $\MF_{R,w}$ admits a homogeneous basis.
\end{lemma}
\begin{proof}
Since $M$ is graded and free, it is graded and projective. By \cite[Lemma 13.3]{Passman-book}, $M$ is graded-projective. Recall that $R^{(0)}=\C$ and any $\C$-linear space has a basis. So, according to \cite[Page 130, Exercise 3]{Passman-book}, $M$ is graded-free. 

By definition, the quantum grading of every object of $\MF_{R,w}$ is bounded below. So the above argument applies to objects of $\MF_{R,w}$.
\end{proof}

\subsection{Homology of graded matrix factorizations over $R$}\label{homology-homotopy} Let $w\in \mathfrak{I}$ be a homogeneous element of degree $2N+2$, and $M$ a graded matrix factorization over $R$ with potential $w$. Note that $M/\mathfrak{I} M$ is a chain complex over $\C$, and it inherits the gradings of $M$.

\begin{definition}\label{homology-matrix-factorization-def}
$H_R(M)$ is defined to be the homology of $M/\mathfrak{I} M$. It inherits the gradings of $M$. If $R$ is clear from the context, we drop it from the notation. 

Denote by $H_R^{\ve,i}(M)$ the subspace of $H_R(M)$ consisting of homogeneous elements of $\zed_2$-degree $\ve$ and quantum degree $i$. If $\dim H_R^{\ve,i}(M) < \infty$ $\forall ~\ve,i$ , we define the graded dimension of $M$ to be
\[
\gdim_R(M) = \sum_{\ve,i} \tau^{\ve} q^i \dim_\C H_R^{\ve,i}(M) ~\in ~\zed[[q]][\tau]/(\tau^2-1).
\]
Again, if $R$ is clear from the context, we drop it from the notation.
\end{definition}

\begin{remark}
One needs to be careful when dropping $R$ from notations. For example, when $w=0$, $M$ is itself a chain complex. Denote by $H_\C(M)$ the usual homology of $M$. In general, $H_R(M)\ncong H_\C (M)$.  Carelessly dropping $R$ from notations in such situations may lead to confusion.
\end{remark}

Any homogeneous morphism of graded matrix factorizations induces a homogeneous homomorphism of the homology, and homotopic morphisms induce the same homomorphism of the homology. In particular, $f:M \xrightarrow{\simeq} M'$ being a homotopy equivalence implies that the induces map $f_\ast:H_R(M) \xrightarrow{\cong} H_R(M')$ is an isomorphism. Surprisingly, according to \cite[Proposition 8]{KR1}, the converse is also true. Next we review properties of the homology of matrix factorizations given in \cite{KR1}.

\begin{lemma}\label{free-module-quotient-tensor}
Let $M$ be a free graded $R$-module, whose grading is bounded below. Let $V=M/\mathfrak{I} M$. Then there is a homogeneous $R$-module map $F: V \otimes_\C R \rightarrow M$ preserving the grading. In particular, if $\{v_\beta | \beta \in \mathcal{B}\}$ is a homogeneous $\C$-basis for $V$, then $\{F(v_\beta \otimes 1) | \beta \in \mathcal{B}\}$ is a homogeneous $R$-basis for $M$.
\end{lemma}
\begin{proof}
By Lemma \ref{homogeneous-basis-exists}, $M$ has a homogeneous basis $\{e_\alpha|\alpha\in\mathcal{A}\}$. Then, as graded vector spaces, $V \cong \bigoplus_{\alpha\in\mathcal{A}} \C \cdot e_\alpha$. So, as graded $R$-modules,
\[
M \cong \bigoplus_{\alpha\in\mathcal{A}} R \cdot e_\alpha \cong V \otimes_\C R.
\]
This proves the existence of $F$. The second part of the lemma follows easily.
\end{proof}

The next proposition is a reformulation of \cite[Proposition 7]{KR1}. For the convenience of the reader, we give a detailed proof here.

\begin{proposition}\cite[Proposition 7]{KR1}\label{contractible-essential-decomp}
Let $M$ be a graded matrix factorization over $R$ with homogeneous potential $w\in\mathfrak{I}$ of degree $2N+2$. Assume the quantum grading of $M$ is bounded below. Then there exist graded matrix factorizations $M_c$ and $M_{es}$ over $R$ with potential $w$ such that
\begin{enumerate}[(i)]
	\item $M \cong M_c \oplus M_{es}$,
	\item $M_c \simeq 0$ and, therefore, $M\simeq M_{es}$,
	\item $M_{es} \cong H_R(M)\otimes_\C R$ as graded $R$-modules, and $H_R(M) \cong M_{es}/\mathfrak{I}M_{es}$ as graded $\C$-spaces.
\end{enumerate}
\end{proposition}

\begin{proof}
(Following \cite{KR1}.) Write $M$ as $M_0 \xrightarrow{d_0} M_1 \xrightarrow{d_1} M_0$. Then the chain complex $V:=M/\mathfrak{I}M$ is given by $V_0 \xrightarrow{\hat{d}_0} V_1 \xrightarrow{\hat{d}_1} V_0$, where $V_\ve=M_\ve/\mathfrak{I}M_\ve$ for $\ve=0,1$. By Lemma \ref{homogeneous-basis-exists}, $M_\ve$ has a homogeneous basis $\{e_\sigma|\sigma\in \mathcal{S}_\ve\}$, which induces a homogeneous $\C$-basis $\{\hat{e}_\sigma|\sigma\in\mathcal{S}_\ve\}$ for $V_\ve$. Under this homogeneous basis, the entries of matrices of $d_0$ and $d_1$ are homogeneous polynomials. And the matrices of $\hat{d}_0$ and $\hat{d}_1$ are obtained by letting $X_1=\cdots=X_m=0$ in the matrices of $d_0$ and $d_1$, which preserves scalar entries and kills entries with positive degrees.

We call $\{(\hat{u}_\rho,\hat{v}_\rho)|\rho\in\mathcal{P}\}$ a ``good" set if
\begin{itemize}
	\item $\{\hat{u}_\rho|\rho\in\mathcal{P}\}$ is set of linearly independent homogeneous elements in $V_0$,
	\item $\{\hat{v}_\rho|\rho\in\mathcal{P}\}$ is set of linearly independent homogeneous elements in $V_1$,
	\item $\hat{d}_0(\hat{u}_\rho) =\hat{v}_\rho$ and $\hat{d}_1(\hat{v}_\rho) =0$.
\end{itemize}
Using Zorn's Lemma, we find a maximal ``good" set $G=\{(\hat{u}_\alpha,\hat{v}_\alpha)|\alpha\in\mathcal{A}\}$. Using Zorn's Lemma again, we extend $\{\hat{u}_\alpha|\alpha\in\mathcal{A}\}$ into a homogeneous basis $\{\hat{u}_\alpha|\alpha\in\mathcal{A}\cup \mathcal{B}_0\}$ for $V_0$, and $\{\hat{v}_\alpha|\alpha\in\mathcal{A}\}$ into a homogeneous basis $\{\hat{v}_\alpha|\alpha\in\mathcal{A}\cup \mathcal{B}_1\}$ for $V_1$. For each $\beta \in \mathcal{B}_0$, we can write $\hat{d}_0 \hat{u}_\beta = \sum_{\alpha\in\mathcal{A}\cup \mathcal{B}_1} c_{\alpha\beta} \cdot \hat{v}_\alpha$, where $c_{\alpha\beta}\in \C$, and the right hand side is a finite sum. 

By Lemma \ref{free-module-quotient-tensor}, there is a homogeneous isomorphism $F_\ve:V_\ve \otimes_\C R \xrightarrow{\cong}M_\ve$ preserving the $\zed_2\oplus\zed$-grading. Let $u_\alpha = F_0(\hat{u}_\alpha \otimes 1)$ and  $v_\alpha = F_1(\hat{v}_\alpha \otimes 1)$. Then $\{u_\alpha|\alpha\in\mathcal{A}\cup \mathcal{B}_0\}$ and $\{v_\alpha|\alpha\in\mathcal{A}\cup \mathcal{B}_1\}$ are homogeneous $R$-bases for $M_0$ and $M_1$. Under this basis, we have that, for any $\alpha \in \mathcal{A}$,
\[
d u_\alpha = v_\alpha + \sum_{\beta \in \mathcal{A}\cup \mathcal{B}_1,~\beta \neq \alpha} f_{\beta \alpha} v_\beta,
\]
where $f_{\beta \alpha} \in \mathfrak{I}$ and the sum on the right hand side is a finite sum. That is, for each $\alpha$,
\begin{equation}\label{contractible-essential-decomp-finite-sum}
f_{\beta \alpha} = 0 \text{ for all but finitely many } \beta.
\end{equation}
Also, using that $\hat{d}_0(\hat{u}_\alpha) =\hat{v}_\alpha$ for $\alpha \in \mathcal{A}$, one can see that 
\begin{equation}\label{contractible-essential-decomp-vanish}
f_{\beta \alpha}=0 \text{ if } \beta \neq \alpha \text{ and } \deg v_\beta \geq \deg v_\alpha.
\end{equation}
For $\alpha\in \mathcal{A}$ and $k>0$, let 
\[
C_{\ast\alpha}^k =\{(\gamma_0,\dots,\gamma_k)\in \mathcal{A}^{k+1} | ~\gamma_k = \alpha, ~  \deg v_{\gamma_0} < \cdots <\deg v_{\gamma_k},~f_{\gamma_0\gamma_1}\cdots f_{\gamma_{k-1}\gamma_k} \neq0\}.
\] 
By \eqref{contractible-essential-decomp-finite-sum}, $C_{\ast\alpha}^k$ is a finite set. For each $\alpha$, $C_{\ast\alpha}^k = \emptyset$ for large $k$'s since the quantum grading of $M$ is bounded below. For $\alpha, \beta \in \mathcal{A}$ and $k>0$, let 
\[
C_{\beta\alpha}^k =\{(\gamma_0,\dots,\gamma_k)\in C_{\ast\alpha}^k | \gamma_0 = \beta\}.
\] 
Then $\cup_{\beta \in \mathcal{A}} C_{\beta\alpha}^k = C_{\ast\alpha}^k$. So each $C_{\beta\alpha}^k$ is finite. And, for each $k$, $C_{\beta\alpha}^k\neq \emptyset$ for only finitely many $\beta$. Also, by definition, it is easy to see that $C_{\beta\alpha}^k=\emptyset$ if $\deg v_\beta \geq \deg v_\alpha$. Moreover, for each $\alpha$, there is a $k_0>0$ such that $C_{\beta\alpha}^k=\emptyset$ for any $\beta$ whenever $k>k_0$.

Now define $t_{\beta\alpha} \in R$ by
\[
t_{\beta\alpha} = 
\begin{cases}
1 & \text{if } \beta=\alpha, \\
0 & \text{if } \beta \neq \alpha, ~\deg v_\beta \geq \deg v_\alpha, \\
\sum_{k\geq 1} (-1)^k \sum_{(\gamma_0,\dots,\gamma_k) \in C_{\beta\alpha}^k} f_{\gamma_0\gamma_1}\cdots f_{\gamma_{k-1}\gamma_k}   & \text{if } \deg v_\beta < \deg v_\alpha.
\end{cases}
\]
From the above discussion, we know that the sum on the right hand side is always finite. So $t_{\beta\alpha}$ is well defined. Furthermore, given an $\alpha \in \mathcal{A}$, $t_{\beta\alpha}=0$ for all but finitely many $\beta$. So, for $\alpha \in \mathcal{A}$, $u_\alpha':=\sum_{\beta\in \mathcal{A}} t_{\beta\alpha} u_\beta$ is well defined. And $\{u_\alpha'|\alpha\in\mathcal{A}\}\cup \{u_\beta|\beta\in\mathcal{B}_0\}$ is also a homogeneous $R$-basis for $M_0$. One can check that, for $\alpha \in \mathcal{A}$, 
\[
d u_\alpha' = v_\alpha + \sum_{\beta \in \mathcal{B}_1} f_{\beta \alpha}' v_\beta,
\]
where the right hand side is a finite sum, and $f_{\beta \alpha}' \in \mathfrak{I}$. Now let
\[
v_\alpha'=
\begin{cases}
v_\alpha + \sum_{\beta \in \mathcal{B}_1} f_{\beta \alpha}' v_\beta & \text{if } \alpha \in \mathcal{A}, \\
v_\alpha & \text{if } \alpha \in \mathcal{B}_1.
\end{cases}
\]
Then $\{v_\alpha'|\alpha\in\mathcal{A}\cup \mathcal{B}_1\}$ is a homogeneous $R$-basis for $M_1$. Under this basis, we have
\[
\begin{cases}
d u_\alpha' = v_\alpha' & \text{if } \alpha \in \mathcal{A}, \\
d u_\beta = \sum_{\alpha \in \mathcal{A}} g_{\alpha\beta} v_\alpha' + \sum_{\gamma \in \mathcal{B}_1} g_{\gamma\beta} v_\gamma' & \text{if } \beta \in \mathcal{B}_0,
\end{cases}  
\]
where the sums on the right hand side are finite. For $\beta \in \mathcal{B}_0$, we let 
\[
u_\beta'= u_\beta - \sum_{\alpha \in \mathcal{A}} g_{\alpha\beta} u_\alpha'.
\]
Then $\{u_\alpha'|\alpha\in\mathcal{A}\cup\mathcal{B}_0\}$ is again a homogeneous $R$-basis for $M_0$, and 
\[
\begin{cases}
d u_\alpha' = v_\alpha' & \text{if } \alpha \in \mathcal{A}, \\
d u_\beta' = \sum_{\gamma \in \mathcal{B}_1} g_{\gamma\beta}' v_\gamma' & \text{if } \beta \in \mathcal{B}_0,
\end{cases}  
\]
where the sum on the right hand side is finite. Using that $d^2 = w \id_M$, one can check that
\[
\begin{cases}
d v_\alpha' = w\cdot v_\alpha' & \text{if } \alpha \in \mathcal{A}, \\
d v_\beta' = \sum_{\gamma \in \mathcal{B}_0} g_{\gamma\beta}'' u_\gamma' & \text{if } \beta \in \mathcal{B}_1,
\end{cases}  
\]
where the sum on the right hand side is finite.

Define $M_{(1,w)}$ to be the submodule of $M$ spanned by $\{u_\alpha'|\alpha\in\mathcal{A}\} \cup \{v_\alpha'|\alpha\in\mathcal{A}\}$, and $M'$ the submodule of $M$ spanned by $\{u_\beta'|\beta\in\mathcal{B}_0\} \cup \{v_\beta'|\beta\in\mathcal{B}_1\}$. Then $M_{(1,w)}$ and $M'$ are both graded matrix factorizations and $M=M_{(1,w)} \oplus M'$. Note that
\begin{enumerate}[(a)]
	\item $M_{(1,w)}$ is a direct sum of components of the form $(1,w)_R \{q^k\}$,
	\item Under the standard projection $M\rightarrow M/\mathfrak{I}M$, we have, for $\alpha \in \mathcal{A}$, $u_\alpha' \mapsto \hat{u}_\alpha$ and $v_\alpha' \mapsto \hat{v}_\alpha$.
\end{enumerate}
In particular, (b) above means that $M'$ does not have direct sum components of the form $(1,w)_R \{q^k\}$. Otherwise, we can enlarge the ``good" set $G$, which contradicts the fact that $G$ is maximal. We then apply a similar argument to $M'$ and find a decomposition $M' = M_{(w,1)} \oplus M_{es}$ of graded matrix factorizations satisfying
\begin{itemize}
	\item $M_{(w,1)}$ is a direct sum of components of the form $(w,1)_R \{q^k\}$,
	\item $M_{es}$ has no direct sum component of the form $(1,w)_R \{q^k\}$ or $(w,1)_R \{q^k\}$.
\end{itemize}
Let $M_c = M_{(1,w)} \oplus M_{(w,1)}$. Then $M = M_c \oplus M_{es}$. Since $(1,w)_R \{q^k\}$ and $(w,1)_R \{q^k\}$ are both homotopic to $0$, $M_c \simeq 0$. So $M\simeq M_{es}$. It is clear that, under any homogeneous basis for $M_{es}$, all entries of the matrices representing the differential map of $M_{es}$ must be in $\mathfrak{I}$. Otherwise, a simple change of basis would show that $M_{es}$ has a component of the form $(1,w)_R \{q^k\}$ or $(w,1)_R \{q^k\}$. Therefore, $H_R(M) \cong H_R(M_{es}) \cong M_{es}/\mathfrak{I}M_{es}$. So, by Lemma \ref{free-module-quotient-tensor}, $M_{es} \cong H_R(M) \otimes_\C R$ as graded modules.
\end{proof}

The following corollaries are from \cite{KR1}.

\begin{corollary}\cite[Proposition 8]{KR1}\label{homotopy-equal-isomorphism-on-homology}
Let $M$ and $M'$ be graded matrix factorizations over $R$ with homogeneous potential $w\in\mathfrak{I}$ of degree $2N+2$. Assume the quantum gradings of $M$ and $M'$ are bounded below. Suppose that $f:M\rightarrow M'$ is a homogeneous morphism preserving the $\zed_2\oplus\zed$-grading. Then $f$ is a homotopy equivalence if and only if it induces an isomorphism of the homology $f_\ast: H_R(M) \rightarrow H_R(M')$.
\end{corollary}
\begin{proof} (Following \cite{KR1}.)
If $f$ is a homotopy equivalence, then $f_\ast$ is clearly an isomorphism. Let us now prove the converse. Assume $f_\ast$ is an isomorphism. Let $M = M_c \oplus M_{es}$ and $M' = M'_c \oplus M'_{es}$ be decompositions of $M$ and $M'$ given by Proposition \ref{contractible-essential-decomp}. So $f$ induces a morphism $f_{es}:M_{es} \rightarrow M'_{es}$. Note that $H_R(M) \cong M_{es}/\mathfrak{I}M_{es}$, $H_R(M') \cong M'_{es}/\mathfrak{I}M'_{es}$, $M_{es} = H_R(M) \otimes_\C R$ and $M'_{es} = H_R(M') \otimes_\C R$. So $f_{es}$ is an isomorphism since $f_\ast$ is an isomorphism. It follows that $f$ is a homotopy equivalence.
\end{proof}

\begin{corollary}\cite[Proposition 7]{KR1}\label{homology-detects-homotopy}
Let $M$ be a graded matrix factorization over $R$ with homogeneous potential $w\in\mathfrak{I}$ of degree $2N+2$. Assume the quantum grading of $M$ is bounded below. Then 
\begin{enumerate}[(i)]
	\item $M\simeq 0$ if and only if $H_R(M)=0$ or, equivalently, $\gdim_R (M)=0$;
	\item $M$ is homotopically finite if and only if $H_R(M)$ is finite dimensional over $\C$ or, equivalently, $\gdim_R (M)$ is a well defined element of $\zed[q,\tau]/(\tau^2-1)$.
\end{enumerate}
\end{corollary}
\begin{proof}
For (i), we have 
\[
M\simeq 0~\Rightarrow~H_R(M) =0~\Rightarrow~M_{es}\cong H_R(M) \otimes_\C R=0~\Rightarrow~M\simeq 0.
\]

Now consider (ii). If $M$ is homotopically finite, then there is a finitely generated graded matrix factorization $\mathcal{M}$ such that $M\simeq \mathcal{M}$. Note that $\mathcal{M}/\mathfrak{I}\mathcal{M}$ is finite dimensional over $\C$. This implies that $H_R(M)\cong H_R(\mathcal{M})$ is finite dimensional over $\C$. If $H_R(M)$ is finite dimensional over $\C$, then $M_{es}\cong H_R(M) \otimes_\C R$ is finitely generated over $R$. But $M\simeq M_{es}$. So $M$ is homotopically finite.
\end{proof}

\subsection{The Krull-Schmidt property} In this subsection, we review the Krull-Schmidt property of matrix factorizations and chain complexes of matrix factorizations. We follow the approach in \cite[Section 1]{Drozd} and \cite[Section 5]{KR1}.

\begin{definition}\cite{Drozd}\label{Krull-Schmidt-category-def}
An additive category $\mathcal{C}$ is called a $\C$-category if all morphism sets $\Hom_\mathcal{C}(A,B)$ are $\C$-linear spaces and the composition of morphisms is $\C$-bilinear. 

A $\C$-category $\mathcal{C}$ is called fully additive if every idempotent morphism of $\mathcal{C}$ splits, that is, defines a decomposition into a direct sum.

A $\C$-category $\mathcal{C}$ is called locally finite dimensional if, for every pair $A,B$ of objects of $\mathcal{C}$, $\Hom_\mathcal{C}(A,B)$ is finite dimensional over $\C$.

A $\C$-category $\mathcal{C}$ is called Krull-Schmidt if 
\begin{itemize}
	\item every object of $\mathcal{C}$ is isomorphic to a finite direct sum $A_1\oplus\cdots\oplus A_n$ of indecomposable objects of $\mathcal{C}$;
	\item and, if $A_1\oplus\cdots\oplus A_n \cong A'_1\oplus\cdots\oplus A'_l$, where $A_1,\dots\,A_n, A'_1,\dots,A'_l$ are indecomposable objects of $\mathcal{C}$, then $n=l$ and there is a permutation $\sigma$ of $\{1,\dots,n\}$ such that $A_i \cong A'_{\sigma(i)}$ for $i=1,\dots,n$.
\end{itemize}
\end{definition}

Note that, for any homogeneous potential $w\in\mathfrak{I}$ of degree $2N+2$, the categories $\MF_{R,w}$, $\HMF_{R,w}$, $\mf_{R,w}$ and $\hmf_{R,w}$ are all $\C$-categories. Moreover, if $\mathcal{C}$ is a $\C$-category, then $\ch(\mathcal{C})$ and $\hch(\mathcal{C})$ are also $\C$-categories. Then following lemma is from \cite[Section 1]{Drozd}.

\begin{lemma}\cite[Section 1]{Drozd}\label{falf-implies-KS}
If $\mathcal{C}$ is a fully additive and locally finite dimensional $\C$-category, then $\mathcal{C}$ is Krull-Schmidt.

Moreover, if $\mathcal{C}$ is a fully additive and locally finite dimensional $\C$-category, then $\ch(\mathcal{C})$ and $\hch(\mathcal{C})$ are both fully additive, locally finite dimensional and, therefore, Krull-Schmidt.
\end{lemma}

\begin{proof}[Sketch of proof](Following \cite{Drozd}.)
A $\C$-category $\mathcal{C}$ is called local if every object of $\mathcal{C}$ decomposes into a finite direct sum of objects with local endomorphism rings. One can check that $\mathcal{C}$ is local if it is fully additive and locally finite dimensional. By \cite[Theorem 3.6]{Bass-book}, local $\C$-categories are Krull-Schmidt. So fully additive locally finite dimensional $\C$-categories are Krull-Schmidt.

If $\mathcal{C}$ is a fully additive and locally finite dimensional $\C$-category, then $\ch(\mathcal{C})$ is also fully additive and locally finite dimensional. So $\ch(\mathcal{C})$ is local and, therefore, Krull-Schmidt. For every pair $(A,B)$ of objects of $\hch(\mathcal{C})$, $\Hom_{\hch(\mathcal{C})}(A,B)$ is a quotient space of $\Hom_{\ch(\mathcal{C})}(A,B)$. Thus, $\hch(\mathcal{C})$ is also locally finite dimensional. Since $\ch(\mathcal{C})$ is local, any object $A$ of $\hch(\mathcal{C})$ decomposes into 
\[
A \cong A_1 \oplus \cdots\oplus A_m,
\] 
where $\Hom_{\ch(\mathcal{C})}(A_i,A_i)$ is a local ring for each $i=1,\dots,m$. But $\Hom_{\hch(\mathcal{C})}(A_i,A_i)$ is a quotient ring of $\Hom_{\ch(\mathcal{C})}(A_i,A_i)$. So, for each $i$, $\Hom_{\hch(\mathcal{C})}(A_i,A_i)$ is either a local ring or $0$. In the latter case, $A_i$ is homotopic to $0$. This shows that $\hch(\mathcal{C})$ is local and, therefore, Krull-Schmidt. Since local $\C$-categories are fully additive, $\hch(\mathcal{C})$ is also fully additive.
\end{proof}

\begin{remark}
In \cite[Section 1]{Drozd}, the above lemma is actually proved for categories over any complete local Noetherian ring. (It is trivial to verify that $\C$ is a complete local Noetherian ring.)
\end{remark}

In the rest of this subsection, we assume that $w$ is a homogeneous element of $\mathfrak{I}$ with $\deg w =2N+2$. The next lemma is the lifting idempotent property from \cite[Section 5]{KR1}.

\begin{lemma}\cite[Section 5]{KR1}\label{idempotent-lifting}
Let $M$ be a finitely generated graded matrix factorization over $R$ with potential $w$. If a homogeneous morphism $f:M\rightarrow M$ of matrix factorizations preserves the $\zed_2\oplus\zed$-grading of $M$ and satisfies $f\circ f \simeq f$, then there is a homogeneous morphism $g:M\rightarrow M$ of matrix factorizations preserving the $\zed_2\oplus\zed$-grading of $M$ such that $g\simeq f$ and $g\circ g=g$.
\end{lemma}
\begin{proof} (Following \cite{KR1}.)
Let $P: M \rightarrow M_{es}$ and $J:M_{es} \rightarrow M$ be the projection and inclusion from the decomposition in Proposition \ref{contractible-essential-decomp}. Then $f$ induces a morphism $f_{es}=P\circ f\circ J: M_{es} \rightarrow M_{es}$, which satisfies $f_{es} \circ f_{es} \simeq f_{es}$.

Let 
\[
\alpha:\Hom_{\mf}(M_{es},M_{es}) \rightarrow \Hom_{\hmf}(M_{es},M_{es})
\] 
be the natural projection taking each morphism to its homotopy class, and 
\[
\beta:\Hom_{\mf}(M_{es},M_{es}) \rightarrow \Hom_\C(H_R(M),H_R(M))
\] 
the map taking each morphism to the induced map on the homology.  Then $\ker \alpha$ and $\ker \beta$ are ideals of the ring $\Hom_{\mf}(M_{es},M_{es})$, and $\ker \alpha \subset \ker \beta$. 

Choose a homogeneous basis $\{e_1,\dots,e_n\}$ for $M_{es}$. For any $h \in \ker \beta$, let $H$ be its matrix under this basis. By Proposition \ref{contractible-essential-decomp}, $H_R(M) \cong M_{es}/\mathfrak{I}M_{es}$. Since $\beta (h)=0$, we know that entries of $H$ are elements of $\mathfrak{I}$. This implies that, if $h \in (\ker \beta)^k$, then entries of $H$ are elements of $\mathfrak{I}^k$. But the matrix of a homogeneous morphism preserving the quantum grading can not have entries of arbitrarily large degrees. Thus, $(\ker \beta)^k =0$ for $k\gg0$ and, therefore, $(\ker \alpha)^k =0$ for $k\gg0$. This shows that $\ker \alpha$ is a nilpotent ideal of $\Hom_{\mf}(M_{es},M_{es})$. By \cite[Theorem 1.7.3]{Benson-book}, nilpotent ideals have the lifting idempotents property. Thus, there is a homogeneous morphism $g_{es}:M_{es}\rightarrow M_{es}$ of matrix factorizations preserving the $\zed_2\oplus\zed$-grading of $M_{es}$ that satisfies $g_{es}\simeq f_{es}$ and $g_{es}\circ g_{es}=g_{es}$.

Now define a morphism $g:M\rightarrow M$ by $g=J\circ g_{es} \circ P$. It is easy to check that $g$ preserves the $\zed_2\oplus\zed$-grading of $M$ and satisfies $g\simeq f$ and $g\circ g=g$.
\end{proof}

\begin{lemma}\cite[Proposition 24]{KR1}\label{fully-additive-hmf}
$\hmf_{R,w}$ is fully additive.
\end{lemma}
\begin{proof} (Following \cite{KR1}.)
Let $M$ be an object of $\hmf_{R,w}$, and $f:M\rightarrow M$ a homogeneous morphism of matrix factorizations preserving the $\zed_2\oplus\zed$-grading of $M$ and satisfying $f\circ f \simeq f$. By the definition of $\hmf_{R,w}$, $M$ is homotopically finite. So, by Proposition \ref{contractible-essential-decomp} and Corollary \ref{homology-detects-homotopy}, $M_{es}$ is finitely generated over $R$. Note that $f$ induces a morphism $f_{es}:M_{es}\rightarrow M_{es}$ such that $f_{es}\circ f_{es} \simeq f_{es}$. By the lifting idempotent property (Lemma \ref{idempotent-lifting}), there is a morphism $g:M_{es} \rightarrow M_{es}$ preserving the $\zed_2\oplus\zed$-grading of $M_{es}$ such that $g\simeq f_{es}$ and $g\circ g =g$. Now $g$ induces a decomposition of graded $R$-modules $M_{es} = gM_{es} \oplus (\id-g)M_{es}$. In particular, $gM_{es}$ and $(\id-g)M_{es}$ are both projective modules over $R$. Recall that $R=\C[X_1,\dots,X_m]$ is a polynomial ring. The well known Quillen-Suslin Theorem tells us that any projective $R$-module is a free $R$-module. So $gM_{es}$ and $(\id-g)M_{es}$ are finitely generated graded free $R$-modules. Since $g$ is a morphism of matrix factorizations, the differential map on $M_{es}$ preserves $gM_{es}$ and $(\id-g)M_{es}$, which makes them objects of $\hmf_{R,w}$ and the above decomposition a decomposition of graded matrix factorizations. Altogether, we have $M \simeq M_{es} = gM_{es} \oplus (\id-g)M_{es}$ as graded matrix factorizations.
\end{proof}

In \cite{KR1}, Khovanov and Rozansky proved that $\HMF_{R,w}$ is locally finite dimensional under the assumption that $w$ is non-degenerate. Since we only need $\hmf_{R,w}$ to be locally finite dimensional, the assumption of non-degeneracy is not necessary. The following is a modified version of \cite[Proposition 6]{KR1}.

\begin{lemma}\cite[Propositions 6]{KR1}\label{locally-finite-hmf}
$\hmf_{R,w}$ is locally finite dimensional.
\end{lemma}
\begin{proof}
Let $M$ and $M'$ be objects of $\hmf_{R,w}$. Then there exists finitely generated graded matrix factorizations $\mathcal{M}$ and $\mathcal{M}'$ over $R$ of potential $w$ such that $M\simeq \mathcal{M}$ and $M'\simeq \mathcal{M}'$. So $\Hom_{\HMF}(M,M') \cong \Hom_{\HMF}(\mathcal{M},\mathcal{M}')$. Recall that $R$ is a polynomial ring and, therefore, Noetherian. Thus, $\Hom_{\HMF}(\mathcal{M},\mathcal{M}')$ is finitely generated over $R$ since $\Hom_R(\mathcal{M},\mathcal{M}')$ is finitely generated over $R$. Let $v_1,\dots v_k$ be a finite set of homogeneous generators of $\Hom_{\HMF}(M,M')$ over $R$ and $a=\min_{i=1,\dots,k} \deg v_i$. Then $\Hom_{\hmf}(M,M')$ is a quotient space of a subspace of the finite dimensional space
\[
(\bigoplus_{i=1}^k \C \cdot v_i) \otimes_\C (\bigoplus_{j=0}^{-a} R^{(j)}),
\]
where $R^{(j)}$ is the $\C$-subspace of $R$ of homogeneous elements of degree $j$. Therefore, $\Hom_{\hmf}(M,M')$ is finite dimensional over $\C$.
\end{proof}

The Krull-Schmidt property follows easily from Lemmas \ref{falf-implies-KS}, \ref{fully-additive-hmf} and \ref{locally-finite-hmf}.

\begin{proposition}\cite[Proposition 25]{KR1}\label{hmf-is-KS}
Assume that $w$ is a homogeneous element of $\mathfrak{I}$ with $\deg w =2N+2$. Then $\hmf_{R,w}$, $\ch(\hmf_{R,w})$ and $\hch(\hmf_{R,w})$ are all Krull-Schmidt.
\end{proposition}

\subsection{Yonezawa's lemma} Yonezawa \cite{Yonezawa-notes} introduced a lemma about isomorphisms in a graded Krull-Schmidt category that is very useful in the proof of the invariance of the colored $\mathfrak{sl}(N)$ homology. Next we review this lemma and show that it applies to $\hmf_{R,w}$ and $\hch(\hmf_{R,w})$. Our statement of Yonezawa's lemma is slightly different from the original version in \cite{Yonezawa-notes}.

First, we recall a simple property of Krull-Schmidt categories.

\begin{lemma}\label{KS-oplus-cancel}
Let $\mathcal{C}$ be a Krull-Schmidt category, and $A,B,C$ objects of $\mathcal{C}$. If $A \oplus C \cong B \oplus C$, then $A \cong B$.
\end{lemma}
\begin{proof}
Decompose both sides of $A \oplus C \cong B \oplus C$ into direct sums of indecomposable objects and compare the components of these direct sums.
\end{proof}

\begin{definition}\label{strongly-non-periodic-def}
Let $\mathcal{C}$ be an additive category, and $F:\mathcal{C} \rightarrow \mathcal{C}$ an autofunctor with inverse functor $F^{-1}$. We say that $F$ is strongly non-periodic if, for any object $A$ of $\mathcal{C}$ and $k \in \zed$, $F^k(A) \cong A$ implies that either $A\cong 0$ or $k=0$.

Denote by $\zed_{\geq 0}[F,F^{-1}]$ the ring of formal Laurent polynomials of $F$ whose coefficients are non-negative integers. Each $G=\sum_{i=k}^l b_i F^i \in \zed_{\geq 0}[F,F^{-1}]$ admits a natural interpretation as an endofunctor on $\mathcal{C}$, that is, for any object $A$ of $\mathcal{C}$, 
\[
G(A) = \bigoplus_{i=k}^l (\underbrace{F^i(A)\oplus\cdots\oplus F^i(A)}_{b_i \text{ fold}}).
\] 
\end{definition}

The following is Yonezawa's lemma.

\begin{lemma}\cite{Yonezawa-notes}\label{yonezawa-lemma}
Let $\mathcal{C}$ be a Krull-Schmidt category, and $F:\mathcal{C} \rightarrow \mathcal{C}$ a strongly non-periodic autofunctor. Suppose that $A$, $B$ are objects of $\mathcal{C}$, and there exists a $G \in \zed_{\geq 0}[F,F^{-1}]$ such that $G\neq 0$ and $G(A) \cong G(B)$. Then $A \cong B$.
\end{lemma}
\begin{proof}
For any objects $C$ and $C'$, we say that they are in the same orbit if $C \cong F^k(C')$ for some $k\in \zed$. If $C$ and $C'$ are in the same orbit, and $C\ncong0$, then we can define a relative degree so that $\deg (C,C')=k$ if $C \cong F^k(C')$. This relative degree is well defined since $F$ is strongly non-periodic.

Clearly, $F$ preserves direct sum decompositions, maps isomorphic objects to isomorphic objects and maps indecomposable objects to indecomposable objects. 

For any object $C$ of $\mathcal{C}$, if $C\cong C_1\oplus \cdots \oplus C_l$, where $C_1,\dots,C_l$ are indecomposable objects of $\mathcal{C}$, then we call $l$ the length of $C$ and denote it by $L(C)$. Since $\mathcal{C}$ is Krull-Schmidt, $L(C)$ is well defined. Clearly, $L(C)=L(F(C))=L(F^k(C))$. More generally, for any $X\in\zed_{\geq 0}[F,F^{-1}]$, let $X(1)=X|_{F=1} \in \zed_{\geq 0}$. Then, $L(X(C))=X(1)L(C)$ for any object $C$. 

If $X\neq 0$, define the degree $\deg X$ of $X$ to be the maximal $k$ so that the coefficient of $F^k$ in $X$ is non-zero.

We prove the lemma by inducting on the length of $A$. If $L(A)=0$, then $A\cong 0$ and $L(B)G(1)=L(A)G(1)=0$. Since $G(1)>0$, this implies that $L(B)=0$ and, therefore, $B\cong0$. So $A\cong B$. Assume that the lemma is true if $L(A)=l-1$. Now suppose $L(A)=l$. Decompose $G(A) \cong G(B)$ into indecomposable objects and find all the orbits of indecomposable objects that appear in this decomposition. This gives us
\[
G(A) \cong G(B) \cong G_1(C_1) \oplus \cdots \oplus G_k(C_k),
\]
where $G_1,\dots,G_k$ are non-zero elements of $\zed_{\geq 0}[F,F^{-1}]$, and $C_1,\dots,C_k$ are indecomposable objects in disjoint orbits. Thus,
\begin{eqnarray*}
A & \cong & f_1(C_1) \oplus \cdots \oplus f_k(C_k), \\
B & \cong & g_1(C_1) \oplus \cdots \oplus g_k(C_k),
\end{eqnarray*}
where $f_1,\dots,f_k,g_1,\dots,g_k$ are non-zero elements of $\zed_{\geq 0}[F,F^{-1}]$. Compare $\deg f_1$ and $\deg g_1$. Using the strong non-periodicity of $F$ and the uniqueness of the decomposition into indecomposable objects, it is easy to conclude that $\deg f_1 + \deg G = \deg G_1 = \deg g_1 +\deg G$. So $\deg f_1 = \deg g_1 \triangleq d$. Define $\hat{f}_1 := f_1 - F^d,~\hat{g}_1 := g_1 - F^d \in \zed_{\geq 0}[F,F^{-1}]$. Let
\begin{eqnarray*}
\hat{A} & = & \hat{f}_1(C_1) \oplus f_2(C_2) \oplus \cdots \oplus f_k(C_k), \\
\hat{B} & = & \hat{g}_1(C_1) \oplus g_2(C_2) \oplus \cdots \oplus g_k(C_k).
\end{eqnarray*}
Then 
\[
G(\hat{A}) \oplus (F^d\cdot G)(C_1) \cong G(A) \cong G(B) \cong G(\hat{B}) \oplus (F^d\cdot G)(C_1).
\]
By Lemma \ref{KS-oplus-cancel}, we have that $G(\hat{A}) \cong G(\hat{B})$. But $L(\hat{A}) = l-1$. So, by induction hypothesis, $\hat{A} \cong \hat{B}$. Thus, $A \cong \hat{A} \oplus F^d(C_1) \cong \hat{B} \oplus F^d(C_1) \cong B$.
\end{proof}

Note that the quantum grading shift functor $\{q\}$ on $\hmf_{R,w}$ induces a quantum grading shift functor on $\hch(\hmf_{R,w})$, which we again denote by $\{q\}$. The following is an easy consequence of Lemma \ref{yonezawa-lemma} and is very useful later in our construction.

\begin{proposition}\label{yonezawa-lemma-hmf}
Assume that $w$ is a homogeneous element of $\mathfrak{I}$ with $\deg w =2N+2$. The functor $\{q\}$ is strongly non-periodic on both $\hmf_{R,w}$ and $\hch(\hmf_{R,w})$. Therefore, for any non-zero element $f(q) \in \zed_{\geq 0}[q,q^{-1}]$,
\begin{itemize}
	\item if $M$ and $M'$ are objects of $\hmf_{R,w}$, and $M\{f(q)\} \simeq M'\{f(q)\}$, then $M \simeq M'$;
	\item if $C$ and $C'$ are objects of $\hch(\hmf_{R,w})$, and $C\{f(q)\} \simeq C'\{f(q)\}$, then $C \simeq C'$.
\end{itemize}
\end{proposition}
\begin{proof}
We only need to show that $\{q\}$ is strongly non-periodic on both $\hmf_{R,w}$ and $\hch(\hmf_{R,w})$. The second half of the proposition follows from this and Proposition \ref{hmf-is-KS} and Lemma \ref{yonezawa-lemma}.

Let $M$ be any object of $\hmf_{R,w}$. Assume that $M\simeq M\{q^k\}$ for some $k\neq 0$. Without loss of generality, assume $k>0$. Since $M$ is homotopically finite, there exists a finitely generated object $\mathcal{M}$ of $\hmf_{R,w}$ such that $M \simeq \mathcal{M}$. So $\mathcal{M} \simeq \mathcal{M}\{q^k\}$, and, therefore $\mathcal{M} \simeq \mathcal{M}\{q^{ak}\}$ for any $a\in \zed_{>0}$. Let $\{e_1,\dots,e_n\}$ be a homogeneous basis for $\mathcal{M}$. Set $u=\max_{1 \leq i \leq n} \deg e_i$ and $l=\min_{1 \leq i \leq n} \deg e_i$. Note that $l$ is the lowest grading for any non-vanishing homogeneous elements of $\mathcal{M}$. Choose an $a\in \zed_{>0}$ such that $ak > u-l$. Then $\mathcal{M} \simeq \mathcal{M}\{q^{ak}\}$ implies that there are homogeneous morphisms $f:\mathcal{M} \rightarrow \mathcal{M}$ of degree $-ak$ and $g:\mathcal{M} \rightarrow \mathcal{M}$ of degree $ak$ such that $f \circ g \simeq g \circ f \simeq \id_{\mathcal{M}}$. Note that $\deg f(e_i) \leq -ak+u <l$ $\forall~i=1,\dots, n$, which implies that $f(e_i)=0$ $\forall~i=1,\dots, n$. So $f=0$ and, therefore, $\id_{\mathcal{M}}\simeq0$. Thus, $M \simeq \mathcal{M} \simeq 0$. This shows that $\{q\}$ is strongly non-periodic on $\hmf_{R,w}$.

Note that any object of $\hch(\hmf_{R,w})$ is isomorphic to an object whose underlying $R$-module is finitely generated, and any morphism of $\hch(\hmf_{R,w})$ can be realized as a finite collection of homogeneous morphisms of graded matrix factorizations. So the above argument works for $\hch(\hmf_{R,w})$ too. Thus, $\{q\}$ is also strongly non-periodic on $\hch(\hmf_{R,w})$.
\end{proof}

\section{Symmetric Polynomials}\label{sec-sym-poly}

In this section, we review properties of symmetric polynomials used in this paper. Most of these materials can be found in, for example, \cite{Fulton-notes,Fulton-Harris,Lascoux-book,Lascoux-notes,Macdonald-book,Zhou-notes}.

\subsection{Notations and basic examples}\label{basic-sym-poly} In this paper, an alphabet means a finite collection of homogeneous indeterminates of degree $2$. For an alphabet $\mathbb{X}=\{x_1,\dots,x_m\}$, we denote by $\C[\mathbb{X}]$ the polynomial ring $\C[x_1,\dots,x_m]$ and by $\Sym(\mathbb{X})$ the ring of symmetric polynomials over $\C$ in $\mathbb{X}=\{x_1,\dots,x_m\}$. Note that the grading on $\C[\mathbb{X}]$ (and $\Sym(\mathbb{X})$) is given by $\deg x_j =2$. For $k=1,2,\dots,m$, we denote by $X_k$ the $k$-th elementary symmetric polynomial in $\mathbb{X}$. That is,
\[
X_k := \sum_{1\leq i_1<i_2<\cdots<i_k\leq m} x_{i_1}x_{i_1}\cdots x_{i_k}.
\]
$X_k$ is a homogeneous symmetric polynomial of degree $2k$. It is well known that $X_1,\cdots,X_m$ are independent and $\Sym(\mathbb{X})=\C[X_1,\dots,X_m]$. For convenience, we define 
\[
X_0=1 \text{ and } X_k=0  \text{ if } k<0 \text{ or } k>m. 
\]
There are two more relevant families of basic symmetric polynomials. The power sum symmetric polynomials $\{p_k(\mathbb{X})~|~k\in\zed\}$ given by
\[
p_k(\mathbb{X})=
\left\{%
\begin{array}{ll}
    \sum_{i=1}^{m} x_i^k & \text{if } k\geq0, \\
    0 & \text{if } k<0, 
\end{array}%
\right.
\]
and the complete symmetric polynomials $\{h_k(\mathbb{X})~|~k\in\zed\}$ given by
\[
h_k(\mathbb{X})=
\left\{%
\begin{array}{ll}
    \sum_{1\leq i_1\leq i_2 \leq \cdots \leq i_k\leq m} x_{i_1}x_{i_1}\cdots x_{i_k} & \text{if } k>0, \\
    1 & \text{if } k=0, \\
    0 & \text{if } k<0. 
\end{array}%
\right.
\]

Consider the generating functions of $\{X_k\}$, $\{p_k(\mathbb{X})\}$ and $\{h_k(\mathbb{X})\}$, that is, the power series
\begin{eqnarray*}
E(t) & = & \sum_{k=0}^{m} (-1)^k X_k t^k = \prod_{i=1}^{m} (1-x_it), \\
P(t) & = & \sum_{k=0}^{\infty} p_{k+1}(\mathbb{X}) t^k = \sum_{i=1}^{m} \frac{x_i}{1-x_it}, \\
H(t) & = & \sum_{k=0}^{\infty} h_k(\mathbb{X}) t^k = \prod_{i=1}^{m} (1-x_it)^{-1}.
\end{eqnarray*}
It is easy to see that $E(t)\cdot H(t)=1$, $E'(t)\cdot H(t) = -P(t)$ and $E(t)\cdot P(t)=-E'(t)$. Hence,

\begin{eqnarray}
\label{complete-recursion} \sum_{k=0}^l (-1)^k X_k h_{l-k}(\mathbb{X}) & = & 
\left\{%
\begin{array}{ll}
    0 & \text{if } l>0, \\ 
    1 & \text{if } l=0, 
\end{array}%
\right.  \\
\label{complete-power} \sum_{k=1}^l (-1)^{k-1} kX_k h_{l-k}(\mathbb{X}) & = & p_{l}(\mathbb{X}), \\
\label{newton} \sum_{k=0}^{l-1} (-1)^k X_k p_{l-k}(\mathbb{X}) & = & (-1)^{l+1} lX_l,
\end{eqnarray}
where \eqref{newton} is known as Newton's Identity.

Since $\Sym(\mathbb{X})=\C[X_1,\dots,X_m]$, $p_k(\mathbb{X})$ and $h_k(\mathbb{X})$ can be uniquely expressed as polynomials in $X_1,\cdots,X_m$. In fact, we know that 
\begin{equation}\label{power}
p_k(\mathbb{X}) = p_{m,k}(X_1,\dots,X_m) = \left|%
\begin{array}{cccccc}
    X_1 & X_2 & X_3 & \cdots & X_{k-1} & kX_k \\
    1 & X_1 & X_2 & \cdots & X_{k-2} & (k-1)X_{k-1} \\
    0 & 1 & X_1 & \cdots & X_{k-3} & (k-2)X_{k-2} \\
    \cdots & \cdots & \cdots & \cdots & \cdots & \cdots \\
    0 & 0 & 0 & \cdots & X_1 & 2X_2 \\
    0 & 0 & 0 & \cdots & 1 & X_1
\end{array}%
\right|,
\end{equation}
and
\begin{equation}\label{complete}
h_k(\mathbb{X}) = h_{m,k}(X_1,\dots,X_m) = \left|%
\begin{array}{cccccc}
    X_1 & X_2 & X_3 & \cdots & X_{k-1} & X_k \\
    1 & X_1 & X_2 & \cdots & X_{k-2} & X_{k-1} \\
    0 & 1 & X_1 & \cdots & X_{k-3} & X_{k-2} \\
    \cdots & \cdots & \cdots & \cdots & \cdots & \cdots \\
    0 & 0 & 0 & \cdots & X_1 & X_2 \\
    0 & 0 & 0 & \cdots & 1 & X_1
\end{array}%
\right|.
\end{equation}
Equations \eqref{power} and \eqref{complete} can be proved inductively using equations \eqref{complete-recursion} and \eqref{newton}.

\begin{lemma}\label{power-derive}
\[
\frac{\partial}{\partial X_j} p_{m,l}(X_1,\dots,X_m) = (-1)^{j+1} l h_{m,l-j}(X_1,\dots,X_m).
\]
\end{lemma}
\begin{proof}
Induct on $l$. If $l<j$, then both sides of the above equation are $0$, and, therefore, the lemma is true. If $l=j$, by Newton's Identity \eqref{newton}, we have 
\[
p_{m,j}+ \sum_{k=1}^{j-1} (-1)^k X_k p_{m,j-k} = (-1)^{j+1} jX_j.
\]
Derive this equation by $X_j$, we get
\[
\frac{\partial}{\partial X_j} p_{m,j} = (-1)^{j+1} j.
\]
So the lemma is true when $l \leq j$.

Assume that $\exists ~n \geq j$ such that the lemma is true $\forall ~l\leq n$. Consider $l=n+1$. Using Newton's Identity \eqref{newton} again, we get
\[
p_{m,n+1} + \sum_{k=1}^{n} (-1)^k X_k p_{m,n+1-k} = (-1)^n (n+1)X_{n+1}.
\] 
Deriving this equation by $X_j$, we get 
\[
\frac{\partial}{\partial X_j} p_{m,n+1} + (-1)^j p_{m,n+1-j} + \sum_{k=1}^{n} (-1)^k X_k \frac{\partial}{\partial X_j}p_{m,n+1-k} =0.
\]
So, by induction hypothesis,
\begin{eqnarray*}
\frac{\partial}{\partial X_j} p_{m,n+1} & = & (-1)^{j+1} p_{m,n+1-j} + \sum_{k=1}^{n+1-j} (-1)^{k+j}(n+1-k) X_k h_{m,n+1-k-j} \\ 
(\text{by \eqref{complete-power}}) & = & \sum_{k=1}^{n+1-j}(-1)^{k+j} (n+1) X_k h_{m,n+1-k-j} \\
(\text{by \eqref{complete-recursion}}) & = & (-1)^{j+1} (n+1) h_{m,n+1-j}.
\end{eqnarray*}
\end{proof}

\subsection{Partitions and linear bases for the space of symmetric polynomials}\label{subsec-partition-schur} A partition $\lambda$ is a finite non-increasing sequence of non-negative integers $(\lambda_1\geq\dots\geq\lambda_m)$. Two partitions are considered the same if one can be changed into the other by adding or removing $0$'s at the end. For a partition $\lambda=(\lambda_1\geq\dots\geq\lambda_m)$, write $|\lambda|=\sum_{j=1}^{m} \lambda_j$ and
$l(\lambda)=\#\{j~|~\lambda_j>0\}$. There is a natural ordering of partitions. For two partitions $\lambda=(\lambda_1\geq\dots\geq\lambda_m)$ and $\mu=(\mu_1\geq\dots\geq\mu_n)$, we say that $\lambda>\mu$ if the first non-vanishing $\lambda_j-\mu_j$ is positive.

It is well known that 
\begin{equation}\label{compute-quantum-binary}
\qb{m+n}{n} =  q^{-mn}\sum_{\lambda:~l(\lambda)\leq m, ~\lambda_1 \leq n} q^{2|\lambda|},
\end{equation}
where $\lambda$ runs through partitions satisfying the given conditions.

The Ferrers diagram of a partition $\lambda=(\lambda_1\geq\dots\geq\lambda_m)$ has $\lambda_i$ boxes in the $i$-th row from the top with rows of boxes lined up on the left. Reflecting this Ferrers digram across the northwest-southeast diagonal, we get the Ferrers diagram of another partition $\lambda'=(\lambda'_1\geq\dots\geq\lambda'_k)$, which is called the conjugate of $\lambda$. Clearly, $\lambda'_i=\#\{j~|~\lambda_j\geq i\}$ and $(\lambda')'=\lambda$.

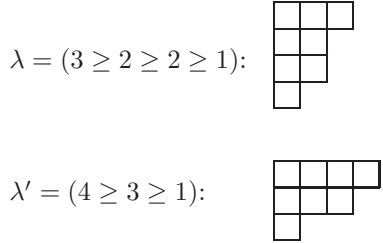
\begin{figure}[ht]

\setlength{\unitlength}{1pt}

\begin{picture}(360,90)(-180,-50)


\put(-80,15){$\lambda=(3\geq2\geq2\geq1)$:}

\put(-80,-35){$\lambda'=(4\geq3\geq1)$:}


\put(20,40){\line(1,0){30}}

\put(20,30){\line(1,0){30}}

\put(20,20){\line(1,0){20}}

\put(20,10){\line(1,0){20}}

\put(20,0){\line(1,0){10}}

\put(20,0){\line(0,1){40}}

\put(30,0){\line(0,1){40}}

\put(40,10){\line(0,1){30}}

\put(50,30){\line(0,1){10}}


\put(20,-20){\line(1,0){40}}

\put(20,-30){\line(1,0){40}}

\put(20,-40){\line(1,0){30}}

\put(20,-50){\line(1,0){10}}

\put(20,-50){\line(0,1){30}}

\put(30,-50){\line(0,1){30}}

\put(40,-40){\line(0,1){20}}

\put(50,-40){\line(0,1){20}}

\put(60,-30){\line(0,1){10}}

\end{picture}

\caption{Ferrers diagrams of a partition and its conjugate}
\end{figure}

We are interested in partitions because they are used to index linear bases for the space of symmetric polynomials. We are particularly interested in two of such bases -- complete symmetric polynomials and Schur polynomials.

Given an alphabet $\mathbb{X}=\{x_1,\dots,x_m\}$ of $m$ indeterminates and a partition $\lambda=(\lambda_1\geq\dots\geq\lambda_m)$ of length $l(\lambda)\leq m$, define 
\[
h_{\lambda}(\mathbb{X}) = h_{\lambda_1}(\mathbb{X}) \cdot h_{\lambda_2}(\mathbb{X}) \cdots h_{\lambda_m}(\mathbb{X}),
\]
where $h_{\lambda_j}(\mathbb{X})$ is defined as in the previous subsection. $h_{\lambda}(\mathbb{X})$ is called the complete symmetric polynomial in $\mathbb{X}$ associated to $\lambda$. This generalizes the definition of complete symmetric polynomials given in the previous subsection. It is known that the set $\{h_{\lambda}(\mathbb{X})~|~l(\lambda)\leq m\}$ is a $\C$-linear basis for $\Sym(\mathbb{X})$. In particular, $\{h_{\lambda}(\mathbb{X})~|~l(\lambda)\leq m,~|\lambda|=d\}$ is a $\C$-linear basis for the subspace of $\Sym(\mathbb{X})$ of homogeneous symmetric polynomials of degree $2d$. (Recall that our degree is twice the usual degree.)

For the alphabet $\mathbb{X}=\{x_1,\dots,x_m\}$ and a partition $\lambda=(\lambda_1\geq\dots\geq\lambda_m)$ of length $l(\lambda)\leq m$, the Schur polynomial in $\mathbb{X}$ associated to $\lambda$ is
\[
S_\lambda(\mathbb{X}) = \frac{\left|%
\begin{array}{lllll}
x_1^{\lambda_1+m-1} & x_1^{\lambda_2+m-2} & \cdots & x_1^{\lambda_{m-1}+1} & x_1^{\lambda_m} \\
x_2^{\lambda_1+m-1} & x_2^{\lambda_2+m-2} & \cdots & x_2^{\lambda_{m-1}+1} & x_2^{\lambda_m} \\
\cdots & \cdots & \cdots & \cdots & \cdots \\
x_{m-1}^{\lambda_1+m-1} & x_{m-1}^{\lambda_2+m-2} & \cdots & x_{m-1}^{\lambda_{m-1}+1} & x_{m-1}^{\lambda_m} \\
x_m^{\lambda_1+m-1} & x_m^{\lambda_2+m-2} & \cdots & x_m^{\lambda_{m-1}+1} & x_m^{\lambda_m}
\end{array}%
\right|}
{\left|%
\begin{array}{lllll}
x_1^{m-1} & x_1^{m-2} & \cdots & x_1 & 1 \\
x_2^{m-1} & x_2^{m-2} & \cdots & x_2 & 1 \\
\cdots & \cdots & \cdots & \cdots & \cdots \\
x_{m-1}^{m-1} & x_{m-1}^{m-2} & \cdots & x_{m-1} & 1 \\
x_m^{m-1} & x_m^{m-2} & \cdots & x_m & 1
\end{array}%
\right|}.
\]
Note that the denominator here is the Vandermonde polynomial, which equals $\prod_{i<j}(x_i-x_j)$. $S_\lambda(\mathbb{X})$ can also be also computed using the following formulas:
\begin{equation}\label{schur-complete}
S_\lambda(\mathbb{X}) = \det (h_{\lambda_i -i +j}(\mathbb{X})) = \left|%
\begin{array}{llll}
h_{\lambda_1}(\mathbb{X}) & h_{\lambda_1+1}(\mathbb{X}) & \dots & h_{\lambda_1+m-1}(\mathbb{X}) \\
h_{\lambda_2-1}(\mathbb{X}) & h_{\lambda_2}(\mathbb{X}) & \dots & h_{\lambda_2+m-2}(\mathbb{X}) \\
\dots & \dots & \dots & \dots \\
h_{\lambda_m-m+1}(\mathbb{X}) & h_{\lambda_m-m+2}(\mathbb{X}) & \dots & h_{\lambda_m}(\mathbb{X}) 
\end{array}%
\right|,
\end{equation}
and
\begin{equation}\label{schur-elementary}
S_\lambda(\mathbb{X}) = \det (X_{\lambda'_i -i +j}) = \left|%
\begin{array}{cccc}
X_{\lambda'_1} & X_{\lambda'_1+1} & \dots & X_{\lambda'_1+k-1} \\
X_{\lambda'_2-1} & X_{\lambda'_2} & \dots & X_{\lambda'_2+k-2} \\
\dots & \dots & \dots & \dots \\
X_{\lambda'_k-k+1} & X_{\lambda'_k-k+2} & \dots & X_{\lambda'_k} 
\end{array}%
\right|,
\end{equation}
where $\lambda'=(\lambda'_1\geq\dots\geq\lambda'_k)$ is the conjugate of $\lambda$. In particular, for $j\geq 0$, 
\begin{eqnarray*}
h_j(\mathbb{X}) & = & S_{(j)}(\mathbb{X}), \\
X_j & = & S_{(\underbrace{1\geq1\geq\cdots\geq1}_{j \text{ parts}})}(\mathbb{X}).
\end{eqnarray*}
The set $\{S_\lambda(\mathbb{X})~|~l(\lambda)\leq m,~|\lambda|=d\}$ is also a basis for the $\C$-space of homogeneous symmetric polynomials in $\mathbb{X}$ of degree $2d$. (Again, recall that our degree is twice the usual degree.) 

The above two bases for the space of symmetric polynomials are related by 
\begin{equation}\label{kostka-equation}
h_{\lambda}(\mathbb{X}) = \sum_{\mu} K_{\mu \lambda} S_{\mu}(\mathbb{X}),
\end{equation}
where $K_{\mu \lambda}$ is the Kostka number defined by 
\begin{itemize}
	\item $K_{\mu \lambda}=0$ if $|\mu|\neq |\lambda|$;
	\item For partitions $\mu=(\mu_1\geq \cdots \geq \mu_m)$ and $\lambda=(\lambda_1\geq\dots\geq\lambda_m)$ with $|\mu|=|\lambda|$, $K_{\mu \lambda}$ is number of ways to fill boxes of the Ferrers diagram of $\mu$ with $\lambda_1$ $1$'s, $\lambda_2$ $2$'s,..., $\lambda_m$ $m$'s, such that the numbers in each row are nondecreasing from left to right, and the numbers in each column are strictly increasing from top to bottom.
\end{itemize}

\begin{lemma}\label{special-kostka}
$K_{\lambda \lambda}=1$ and $K_{\mu \lambda}=0$ if $\lambda>\mu$, that is, if the first non-vanishing $\lambda_j-\mu_j$ is positive.
\end{lemma}

For an alphabet $\mathbb{X}=\{x_1,\dots,x_m\}$, there is also a notion of Schur polynomial in $-\mathbb{X}$, which will be useful in the next subsection. First, for any $j\in \zed$, define 
\[
h_j(-\mathbb{X})=(-1)^jX_j. 
\]
More generally, for any partition $\lambda=(\lambda_1\geq\cdots\geq\lambda_n)$ with $\lambda_1\leq m$, 
\begin{equation}\label{negative-schur-complete}
S_\lambda(-\mathbb{X}) = \det (h_{\lambda_i -i +j}(-\mathbb{X})) = \left|%
\begin{array}{llll}
h_{\lambda_1}(-\mathbb{X}) & h_{\lambda_1+1}(-\mathbb{X}) & \dots & h_{\lambda_1+n-1}(-\mathbb{X}) \\
h_{\lambda_2-1}(-\mathbb{X}) & h_{\lambda_2}(-\mathbb{X}) & \dots & h_{\lambda_2+n-2}(-\mathbb{X}) \\
\dots & \dots & \dots & \dots \\
h_{\lambda_n-n+1}(-\mathbb{X}) & h_{\lambda_n-n+2}(-\mathbb{X}) & \dots & h_{\lambda_n}(-\mathbb{X}) 
\end{array}%
\right|.
\end{equation}
If we write the Schur polynomials in $\mathbb{X}$ as $S_{\lambda}(\mathbb{X})=S_{\lambda}(x_1,\dots,x_m)$, then, by comparing \eqref{negative-schur-complete} to \eqref{schur-elementary}, one can see that the Schur polynomials in $-\mathbb{X}$ is given by 
\begin{equation}\label{positive-negative-schur}
S_\lambda(-\mathbb{X}) = S_{\lambda'}(-x_1,\dots,-x_m),
\end{equation}
where $\lambda'$ is the conjugate of $\lambda$.

See, for example, \cite[Appendix A]{Fulton-Harris} and \cite{Lascoux-notes} for more on partitions and symmetric polynomials.

\subsection{Partially symmetric polynomials} Let $\mathbb{X}=\{x_1,\dots,x_m\}$ and $\mathbb{Y}=\{y_1,\dots,y_n\}$ be two disjoint alphabets. Then $\mathbb{X}\cup\mathbb{Y}$ is also an alphabet. Denote by $\Sym(\mathbb{X}|\mathbb{Y})$ the ring of polynomials in $\mathbb{X}\cup\mathbb{Y}$ over $\C$ that are symmetric in $\mathbb{X}$ and symmetric in $\mathbb{Y}$. Then $\Sym(\mathbb{X}\cup\mathbb{Y})$, the ring of symmetric polynomials over $\C$ in $\mathbb{X}\cup\mathbb{Y}$, is a subring of $\Sym(\mathbb{X}|\mathbb{Y})$. In other words, $\Sym(\mathbb{X}|\mathbb{Y})$ is a $\Sym(\mathbb{X}\cup\mathbb{Y})$-module. The following theorem explains the structure of this module. (See \cite[pages 16-19]{Lascoux-notes} for a detailed discussion.)

\begin{theorem}\cite[Proposition \emph{Gr}5]{Lascoux-notes}\label{part-symm-str}
Let $\mathbb{X}=\{x_1,\dots,x_m\}$ and $\mathbb{Y}=\{y_1,\dots,y_n\}$ be two disjoint alphabets. Then $\Sym(\mathbb{X}|\mathbb{Y})$ is a graded-free $\Sym(\mathbb{X}\cup\mathbb{Y})$-module. 

Denote by $\Lambda_{m,n}$ the set of partitions $\Lambda_{m,n}=\{\lambda~|~l(\lambda)\leq m, ~\lambda_1\leq n\}$. Then 
\[\{S_\lambda(\mathbb{X})~|~ \lambda \in \Lambda_{m,n}\} \text{ and } \{S_\lambda(-\mathbb{Y})~|~ \lambda \in \Lambda_{m,n}\}
\] 
are two homogeneous bases for the $\Sym(\mathbb{X}\cup\mathbb{Y})$-module $\Sym(\mathbb{X}|\mathbb{Y})$.

Moreover, there is a unique $\Sym(\mathbb{X}\cup\mathbb{Y})$-module homomorphism 
\[
\zeta:\Sym(\mathbb{X}|\mathbb{Y}) \rightarrow \Sym(\mathbb{X}\cup\mathbb{Y}),
\] 
called the Sylvester operator, such that, for $\lambda,\mu \in \Lambda_{m,n}$,
\[
\zeta(S_\lambda(\mathbb{X}) \cdot S_\mu(-\mathbb{Y})) = \left\{%
\begin{array}{ll}
    1 & \text{if } \lambda_j + \mu_{m+1-j} =n ~\forall j=1,\dots,m, \\
    0 & \text{otherwise.}  \\
\end{array}%
\right.
\]
\end{theorem}

Comparing Theorem \ref{part-symm-str} to equation \eqref{compute-quantum-binary}, we get the following corollary.

\begin{corollary}\label{part-symm-grade}
Let $\mathbb{X}=\{x_1,\dots,x_m\}$ and $\mathbb{Y}=\{y_1,\dots,y_n\}$ be two disjoint alphabets. Then, as graded $\Sym(\mathbb{X}\cup\mathbb{Y})$-modules, 
\[
\Sym(\mathbb{X}|\mathbb{Y}) \cong \Sym(\mathbb{X}\cup\mathbb{Y})\{\qb{m+n}{n}\cdot q^{mn}\}.
\]
\end{corollary}

More generally, given a collection $\{\mathbb{X}_1,\dots,\mathbb{X}_l\}$ of pairwise disjoint alphabets, we denote by $\Sym(\mathbb{X}_1|\cdots|\mathbb{X}_l)$ the ring of polynomials in $\mathbb{X}_1\cup\cdots\cup\mathbb{X}_l$ over $\C$ that are symmetric in each $\mathbb{X}_i$, which is naturally a graded-free $\Sym(\mathbb{X}_1\cup\cdots\cup\mathbb{X}_l)$-module. Moreover,
\[
\Sym(\mathbb{X}_1|\cdots|\mathbb{X}_l) \cong \Sym(\mathbb{X}_1) \otimes_{\C} \cdots \otimes_{\C} \Sym(\mathbb{X}_l).
\]

\subsection{The cohomology ring of a complex Grassmannian} Denote by $G_{m,N}$ the complex $(m,N)$-Grassmannian, that is, the manifold of all complex $m$-dimensional subspaces of $\C^N$. The cohomology ring of $G_{m,N}$ is isomorphic to a quotient ring of a ring of symmetric polynomials. See for example \cite[Lecture 6]{Fulton-notes} for more.

\begin{theorem}\label{grassmannian}
Let $\mathbb{X}$ be an alphabet of $m$ independent indeterminates. Then 
$H^\ast(G_{m,N};\C) \cong \Sym(\mathbb{X})/(h_{N+1-m}(\mathbb{X}),h_{N+2-m}(\mathbb{X}),\dots,h_{N}(\mathbb{X}))$ as graded $\C$-algebras. As a graded $\C$-linear space, $H^\ast(G_{m,N};\C)$ has a homogeneous basis 
\[
\{S_\lambda(\mathbb{X})~|~ \lambda=(\lambda_1\geq\dots\geq\lambda_m), ~l(\lambda)\leq m, ~\lambda_1\leq N-m\}.
\]

Under the above basis, the Poincar\'{e} duality of $H^\ast(G_{m,N};\C)$ is given by a $\C$-linear trace map 
\[
\Tr:\Sym(\mathbb{X})/(h_{N+1-m}(\mathbb{X}),h_{N+2-m}(\mathbb{X}),\dots,h_{N}(\mathbb{X})) \rightarrow \C
\] 
satisfying
\[
\Tr(S_\lambda(\mathbb{X}) \cdot S_\mu(\mathbb{X})) = \left\{%
\begin{array}{ll}
    1 & \text{if } \lambda_j + \mu_{m+1-j} =N-m ~\forall j=1,\dots,m, \\
    0 & \text{otherwise.}  \\
\end{array}%
\right.
\]
\end{theorem}

Comparing Theorem \ref{grassmannian} to equation \eqref{compute-quantum-binary}, we get the following corollary.

\begin{corollary}\label{grassmannian-grade}
As graded $\C$-linear spaces,
\[
H^\ast(G_{m,N};\C) \cong \C \{\qb{N}{m}\cdot q^{m(N-m)}\},
\]
where $\C$ on the right hand side has grading $0$.
\end{corollary}

\section{Matrix Factorizations Associated to MOY Graphs}\label{mf-MOY}

\subsection{Markings of MOY graphs}

\begin{definition}\label{MOY-marking-def}
A marking of a MOY graph $\Gamma$ consists of the following:
\begin{enumerate}
	\item A finite collection of marked points on $\Gamma$ such that
	\begin{itemize}
	\item every edge of $\Gamma$ has at least one marked point;
	\item all the end points (vertices of valence $1$) are marked;
	\item none of the internal vertices (vertices of valence at least $2$) are marked.
  \end{itemize}
  \item An assignment of pairwise disjoint alphabets to the marked points such that the alphabet associated to a marked point on an edge of color $m$ has $m$ independent indeterminates. (Recall that an alphabet is a finite collection of homogeneous indeterminates of degree $2$.)
\end{enumerate}
\end{definition}

\begin{figure}[ht]

\setlength{\unitlength}{1pt}

\begin{picture}(360,80)(-180,-40)


\put(0,0){\vector(-1,1){15}}

\put(-15,15){\line(-1,1){15}}

\put(-23,25){\tiny{$i_1$}}

\put(-33,32){\small{$\mathbb{X}_1$}}

\put(0,0){\vector(-1,2){7.5}}

\put(-7.5,15){\line(-1,2){7.5}}

\put(-11,25){\tiny{$i_2$}}

\put(-18,32){\small{$\mathbb{X}_2$}}

\put(3,25){$\cdots$}

\put(0,0){\vector(1,1){15}}

\put(15,15){\line(1,1){15}}

\put(31,25){\tiny{$i_k$}}

\put(27,32){\small{$\mathbb{X}_k$}}


\put(4,-2){$v$}

\multiput(-50,0)(5,0){19}{\line(1,0){3}}

\put(-70,0){$L_v$}

\put(45,0){\tiny{$i_1+i_2+\cdots +i_k = j_1+j_2+\cdots +j_l$}}


\put(-30,-30){\vector(1,1){15}}

\put(-15,-15){\line(1,1){15}}

\put(-26,-30){\tiny{$j_1$}}

\put(-33,-40){\small{$\mathbb{Y}_1$}}

\put(-15,-30){\vector(1,2){7.5}}

\put(-7.5,-15){\line(1,2){7.5}}

\put(-13,-30){\tiny{$j_2$}}

\put(-18,-40){\small{$\mathbb{Y}_2$}}

\put(3,-30){$\cdots$}

\put(30,-30){\vector(-1,1){15}}

\put(15,-15){\line(-1,1){15}}

\put(31,-30){\tiny{$j_l$}}

\put(27,-40){\small{$\mathbb{Y}_l$}}

\end{picture}

\caption{}\label{general-MOY-vertex}

\end{figure}

\subsection{The matrix factorization associated to a MOY graph} Recall that $N$ is a fixed positive integer. (It is the ``$N$" in ``$\mathfrak{sl}(N)$".) For a MOY graph $\Gamma$ with a marking, cut it at its marked points. This gives a collection of marked MOY graphs, each of which is a star-shaped neighborhood of a vertex in $G$ and is marked only at its endpoints. (If an edge of $\Gamma$ has two or more marked points, then some of these pieces may be oriented arcs from one marked point to another. In this case, we consider such an arc as a neighborhood of an additional vertex of valence $2$ in the middle of that arc.)

Let $v$ be a vertex of $\Gamma$ with coloring and marking around it given as in Figure \ref{general-MOY-vertex}. Set $m=i_1+i_2+\cdots +i_k = j_1+j_2+\cdots +j_l$ (the width of $v$.) Define 
\[
R=\Sym(\mathbb{X}_1|\dots|\mathbb{X}_k|\mathbb{Y}_1|\dots|\mathbb{Y}_l).
\] 
Write $\mathbb{X}=\mathbb{X}_1\cup\cdots\cup \mathbb{X}_k$ and $\mathbb{Y}=\mathbb{Y}_1\cup\cdots\cup \mathbb{Y}_l$. Denote by $X_j$ the $j$-th elementary symmetric polynomial in $\mathbb{X}$ and by $Y_j$ the $j$-th elementary symmetric polynomial in $\mathbb{Y}$. For $j=1,\dots,m$, define
\begin{equation}\label{eq-def-U-j}
U_j = \frac{p_{m,N+1}(Y_1,\dots,Y_{j-1},X_j,\dots,X_m) - p_{m,N+1}(Y_1,\dots,Y_j,X_{j+1},\dots,X_m)}{X_j-Y_j},
\end{equation}
where $p_{m,N+1}$ is the polynomial given by equation \eqref{power} in Subsection \ref{basic-sym-poly}. The matrix factorization associated to the vertex $v$ is
\[
C(v)=\left(%
\begin{array}{cc}
  U_1 & X_1-Y_1 \\
  U_2 & X_2-Y_2 \\
  \dots & \dots \\
  U_m & X_m-Y_m
\end{array}%
\right)_R
\{q^{-\sum_{1\leq s<t \leq k} i_si_t}\},
\]
whose potential is $\sum_{j=1}^m (X_j-Y_j)U_j = p_{N+1}(\mathbb{X})-p_{N+1}(\mathbb{Y})$, where $p_{N+1}(\mathbb{X})$ and $p_{N+1}(\mathbb{Y})$ are the $(N+1)$-th power sum symmetric polynomials in $\mathbb{X}$ and $\mathbb{Y}$. (See Subsection \ref{basic-sym-poly} for the definition.)

\begin{remark}\label{MOY-freedom}
Since 
\[
\Sym(\mathbb{X}|\mathbb{Y})=\C[X_1,\dots,X_m,Y_1,\dots,Y_m]=\C[X_1-Y_1,\dots,X_m-Y_m,Y_1,\dots,Y_m],
\] 
it is clear that $\{X_1-Y_1,\dots,X_m-Y_m\}$ is $\Sym(\mathbb{X}|\mathbb{Y})$-regular. (See Definition \ref{regular-sequence}.) By Theorem \ref{part-symm-str}, $R$ is a free $\Sym(\mathbb{X}|\mathbb{Y})$-module. So $\{X_1-Y_1,\dots,X_m-Y_m\}$ is also $R$-regular. Thus, by Lemma \ref{freedom}, the isomorphism type of $C(v)$ does not depend on the particular choice of $U_1,\dots,U_m$ as long as they are homogeneous with the right degrees and the potential of $C(v)$ remains $\sum_{j=1}^m (X_j-Y_j)U_j = p_{N+1}(\mathbb{X})-p_{N+1}(\mathbb{Y})$. From now on, we will only specify our choice for $U_1,\dots,U_m$ when it is actually used in the computation. Otherwise, we will simply denote them by $\ast$'s. 
\end{remark}

\begin{definition}\label{MOY-mf-def}
\[
C(\Gamma) := \bigotimes_{v} C(v),
\]
where $v$ runs through all the interior vertices of $\Gamma$ (including those additional $2$-valent vertices.) Here, the tensor product is done over the common end points. More precisely, for two sub-MOY graphs $\Gamma_1$ and $\Gamma_2$ of $\Gamma$ intersecting only at (some of) their open end points, let $\mathbb{W}_1,\dots,\mathbb{W}_n$ be the alphabets associated to these common end points. Then, in the above tensor product, $C(\Gamma_1)\otimes C(\Gamma_2)$ is the tensor product $C(\Gamma_1)\otimes_{\Sym(\mathbb{W}_1|\dots|\mathbb{W}_n)} C(\Gamma_2)$.

$C(\Gamma)$ has a $\zed_2$-grading and a quantum grading.

If $\Gamma$ is closed, that is, has no end points, then $C(\Gamma)$ is considered a matrix factorization over $\C$. 

Assume $\Gamma$ has end points. Let $\mathbb{E}_1,\dots,\mathbb{E}_n$ be the alphabets assigned to all end points of $\Gamma$, among which $\mathbb{E}_1,\dots,\mathbb{E}_k$ are assigned to exits and $\mathbb{E}_{k+1},\dots,\mathbb{E}_n$ are assigned to entrances. Then the potential of $C(\Gamma)$ is  
\[
w= \sum_{i=1}^k p_{N+1}(\mathbb{E}_i) - \sum_{j=k+1}^n p_{N+1}(\mathbb{E}_j).
\]
Let $R_\partial=\Sym(\mathbb{E}_1|\cdots|\mathbb{E}_n)$. Although the alphabets assigned to all marked points on $\Gamma$ are used in its construction, $C(\Gamma)$ is viewed as a matrix factorization over $R_\partial$. Note that, in this case, $w$ is a non-degenerate element of $R_\partial$.

We allow the MOY graph to be empty. In this case, we define 
\[
C(\emptyset)=\C\rightarrow 0 \rightarrow \C,
\]
where the $\zed_2$-grading and the quantum grading of $\C$ are both $0$.
\end{definition}

\begin{lemma}\label{marking-independence}
If $\Gamma$ is a MOY graph, then the homotopy type of $C(\Gamma)$ does not depend on the choice of the marking.
\end{lemma}
\begin{proof}
We only need to show that adding or removing an extra marked point corresponds to a homotopy of matrix factorizations preserving the $\zed_2\oplus\zed$-grading. This follows easily from Proposition \ref{b-contraction}.
\end{proof}

\begin{definition}\label{homology-MOY-def}
Let $\Gamma$ be a MOY graph with a marking. 
\begin{enumerate}[(i)]
	\item If $\Gamma$ is closed, that is, has no open end points, then $C(\Gamma)$ is a chain complex. Denote by $H(\Gamma)$ the homology of $C(\Gamma)$. Note that $H(\Gamma)$ inherits the $\zed_2\oplus\zed$-grading of $C(\Gamma)$.
	\item If $\Gamma$ has end points, let $\mathbb{E}_1,\dots,\mathbb{E}_n$ be the alphabets assigned to all end points of $\Gamma$, and $R_\partial=\Sym(\mathbb{E}_1|\cdots|\mathbb{E}_n)$. Denote by $E_{i,j}$ the $j$-th elementary symmetric polynomial in $\mathbb{E}_i$ and by $\mathfrak{I}$ the maximal homogeneous ideal of $R_\partial$ generated by $\{E_{i,j}\}$. Then $H(\Gamma)$ is defined to be $H_{R_\partial}(C(\Gamma))$, that is, the homology of the chain complex $C(\Gamma)/\mathfrak{I}\cdot C(\Gamma)$. Clearly, $H(\Gamma)$ inherits the $\zed_2\oplus\zed$-grading of $C(\Gamma)$.
\end{enumerate}
Note that (i) is a special case of (ii).
\end{definition}

\begin{lemma}\label{width-cap}
If $\Gamma$ is a MOY graph with a vertex of width greater than $N$, then $C(\Gamma)\simeq 0$.
\end{lemma}
\begin{proof}
Suppose the vertex $v$ of $\Gamma$ has width $m>N$. Then, by Newton's Identity \eqref{newton}, it is easy to check that, in \eqref{eq-def-U-j}, 
$U_{N+1}= (-1)^{N} (N+1)$. By Lemma \ref{entries-null-homotopic}, we have $\id_{C(v)} \simeq 0$. This implies that $C(v) \simeq 0$ and, therefore, $C(\Gamma) \simeq 0$.
\end{proof}

Since rectangular partitions come up frequently in this paper, we introduce the following notations.

\begin{definition}\label{rectangular-partition-def}
Denote by  $\lambda_{m,n}$ the partition
\[
\lambda_{m,n} := (\underbrace{n \geq \cdots \geq n}_{m \text{ parts}}),
\] 
and by $\Lambda_{m,n}$ the set of partitions
\[
\Lambda_{m,n} :=  \{\mu=(\mu_1\geq\cdots\geq\mu_m) ~|~ \mu_1 \leq n\}.
\]
\end{definition}

The following is a generalization of \cite[Proposition 2.4]{Gornik}.

\begin{lemma}\label{schur-null-homotopic}
Let $\Gamma$ be a MOY graph, and $\mathbb{X}=\{x_1,\dots,x_m\}$ an alphabet associated to a marked point on an edge of $\Gamma$ of color $m$.  
Suppose that $\mu$ is a partition with $\mu>\lambda_{m,N-m}$, that is, $\mu_1-(N-m)>0$. Then multiplication by $S_{\mu}(\mathbb{X})$ is a null-homotopic endomorphism of $C(\Gamma)$.
\end{lemma}
\begin{proof}
Cut $\Gamma$ at all the marked points into local pieces, and let $\Gamma'$ be a local piece containing the point marked by $\mathbb{X}$ (as an end point.) Let $\mathbb{W}_1,\dots,\mathbb{W}_l$ be the alphabets marking the other end points of $\Gamma'$. Then $C(\Gamma')$ is of the form 
\[
C(\Gamma') = \left(%
\begin{array}{ll}
  a_{1,0}, & a_{1,1} \\
  a_{2,0}, & a_{2,1} \\
  \dots & \dots \\
  a_{k,0}, & a_{k,1}
\end{array}%
\right)_{\Sym(\mathbb{X}|\mathbb{W}_1|\cdots|\mathbb{W}_l)},
\]
and has potential 
\[
\pm p_{N+1}(\mathbb{X}) + \sum_{i=1}^l \pm p_{N+1}(\mathbb{W}_i)=\sum_{j=1}^k a_{j,0}a_{j,1}.
\]
Let $X_j$ be the $j$-th elementary symmetric polynomial in $\mathbb{X}$. Derive the above equation by $X_j$. By Lemma \ref{power-derive}, we get
\[
\pm (N+1) h_{N+1-j}(\mathbb{X}) = \sum_{j=1}^k (\frac{\partial a_{j,0}}{\partial X_j}\cdot a_{j,1}+a_{j,0} \cdot \frac{\partial a_{j,1}}{\partial X_j}).
\]
So $h_{N}(\mathbb{X}),h_{N-1}(\mathbb{X}),\dots,h_{N-m+1}(\mathbb{X})$ are in the ideal $(a_{1,0}, a_{1,1},\dots,a_{k,0}, a_{k,1})$ of $\Sym(\mathbb{X}|\mathbb{W}_1|\cdots|\mathbb{W}_l)$. By Lemma \ref{entries-null-homotopic}, multiplications by these polynomials are null-homotopic endomorphisms of $C(\Gamma')$ and, by Lemma \ref{morphism-sign}, of $C(\Gamma)$. By equation \eqref{schur-complete} and recursive relation \eqref{complete-recursion}, if $\mu>\lambda_{m,N-m}$, then $S_{\mu}(\mathbb{X})$ is in the ideal $(h_{N}(\mathbb{X}),h_{N-1}(\mathbb{X}),\dots,h_{N-m+1}(\mathbb{X}))$. So the multiplication by $S_{\mu}(\mathbb{X})$ is null-homotopic.
\end{proof}

\begin{lemma}\label{MOY-object-of-hmf} 
Let $\Gamma$ be a MOY graph, and $\mathbb{E}_1,\dots,\mathbb{E}_n$ the alphabets assigned to all end points of $\Gamma$, among which $\mathbb{E}_1,\dots,\mathbb{E}_k$ are assigned to exits and $\mathbb{E}_{k+1},\dots,\mathbb{E}_n$ are assigned to entrances. (Here we allow $n=0$, that is, $\Gamma$ to be closed.) Write $R_\partial=\Sym(\mathbb{E}_1|\cdots|\mathbb{E}_n)$ and $w= \sum_{i=1}^k p_{N+1}(\mathbb{E}_i) - \sum_{j=k+1}^n p_{N+1}(\mathbb{E}_j)$. Then $C(\Gamma)$ is an object of $\hmf_{R_\partial,w}$.
\end{lemma}
\begin{proof}
Let $\mathbb{W}_1,\dots,\mathbb{W}_m$ be the alphabets assigned to interior marked points of $\Gamma$. Then $C(\Gamma)$ is a finitely generated Koszul matrix factorization over 
\[
\widetilde{R}=\Sym(\mathbb{W}_1|\cdots|\mathbb{W}_m|\mathbb{E}_1|\cdots|\mathbb{E}_n).
\]
This implies that the quantum grading of $C(\Gamma)$ is bounded below. So, to show that $C(\Gamma)$ is an object of $\hmf_{R_\partial,w}$, it remains to prove that $C(\Gamma)$ is homotopically finite. By Corollary \ref{homology-detects-homotopy}, we only need to demonstrate that $H(\Gamma)$ is finite dimensional. 

Let $\mathfrak{I}$ be the maximal homogeneous ideal of $R_\partial=\Sym(\mathbb{E}_1|\cdots|\mathbb{E}_n)$. Then $C(\Gamma)/\mathfrak{I}C(\Gamma)$ is a chain complex of finitely generated modules over $R'=\Sym(\mathbb{W}_1|\cdots|\mathbb{W}_m)$. Note that $R'$ is a polynomial ring and, therefore, a Noetherian ring. So the homology of $C(\Gamma)/\mathfrak{I}C(\Gamma)$, that is, $H(\Gamma)$, is also finitely generated over $R'$. But Lemma \ref{schur-null-homotopic} implies that the action of $R'$ on $H(\Gamma)$ factors through a finite dimensional quotient ring of $R'$. So $H(\Gamma)$ is finite dimensional over $\C$.
\end{proof}

\begin{figure}[ht]

\begin{picture}(360,90)(-180,-50)


\put(-130,0){\vector(-3,2){22.5}}

\put(-152.5,15){\line(-3,2){22.5}}

\put(-178,25){\tiny{$i_1$}}

\put(-178,32){\small{$\mathbb{X}_1$}}

\put(-152,15){$\cdots$}

\put(-150,25){\tiny{$i_s$}}

\put(-148,32){\small{$\mathbb{X}_s$}}

\put(-115,25){\tiny{$i_{s+1}$}}

\put(-117,32){\small{$\mathbb{X}_{s+1}$}}

\put(-130,0){\vector(0,1){7.5}}

\put(-130,7.5){\line(0,1){7.5}}

\put(-130,15){\vector(-1,1){7.5}}

\put(-137.5,22.5){\line(-1,1){7.5}}

\put(-130,15){\vector(1,1){7.5}}

\put(-122.5,22.5){\line(1,1){7.5}}

\put(-131,10){\line(1,0){2}}

\put(-129,6){\small{$\mathbb{A}$}}

\put(-119,15){$\cdots$}

\put(-130,0){\vector(3,2){22.5}}

\put(-107.5,15){\line(3,2){22.5}}

\put(-84,25){\tiny{$i_k$}}

\put(-87,32){\small{$\mathbb{X}_k$}}


\put(-123,-2){$v$}


\put(-152.5,-15){\line(3,2){22.5}}

\put(-175,-30){\vector(3,2){22.5}}

\put(-178,-27){\tiny{$j_1$}}

\put(-178,-38){\small{$\mathbb{Y}_1$}}

\put(-152,-20){$\cdots$}

\put(-137.5,-15){\line(1,2){7.5}}

\put(-145,-30){\vector(1,2){7.5}}

\put(-150,-27){\tiny{$j_t$}}

\put(-148,-38){\small{$\mathbb{Y}_t$}}

\put(-122.5,-15){\line(-1,2){7.5}}

\put(-115,-30){\vector(-1,2){7.5}}

\put(-115,-27){\tiny{$j_{t+1}$}}

\put(-117,-38){\small{$\mathbb{Y}_{t+1}$}}

\put(-119,-20){$\cdots$}

\put(-107.5,-15){\line(-3,2){22.5}}

\put(-85,-30){\vector(-3,2){22.5}}

\put(-84,-27){\tiny{$j_l$}}

\put(-87,-38){\small{$\mathbb{Y}_l$}}

\put(-130,-50){$\Gamma_1$}


\put(0,0){\vector(-3,2){22.5}}

\put(-22.5,15){\line(-3,2){22.5}}

\put(-48,25){\tiny{$i_1$}}

\put(-48,32){\small{$\mathbb{X}_1$}}

\put(-22,15){$\cdots$}

\put(0,0){\vector(-1,2){7.5}}

\put(-7.5,15){\line(-1,2){7.5}}

\put(-20,25){\tiny{$i_s$}}

\put(-18,32){\small{$\mathbb{X}_s$}}

\put(0,0){\vector(1,2){7.5}}

\put(7.5,15){\line(1,2){7.5}}

\put(15,25){\tiny{$i_{s+1}$}}

\put(13,32){\small{$\mathbb{X}_{s+1}$}}

\put(11,15){$\cdots$}

\put(0,0){\vector(3,2){22.5}}

\put(22.5,15){\line(3,2){22.5}}

\put(46,25){\tiny{$i_k$}}

\put(43,32){\small{$\mathbb{X}_k$}}


\put(7,-2){$v$}


\put(-22.5,-15){\line(3,2){22.5}}

\put(-45,-30){\vector(3,2){22.5}}

\put(-48,-27){\tiny{$j_1$}}

\put(-48,-38){\small{$\mathbb{Y}_1$}}

\put(-22,-20){$\cdots$}

\put(-7.5,-15){\line(1,2){7.5}}

\put(-15,-30){\vector(1,2){7.5}}

\put(-20,-27){\tiny{$j_t$}}

\put(-18,-38){\small{$\mathbb{Y}_t$}}

\put(7.5,-15){\line(-1,2){7.5}}

\put(15,-30){\vector(-1,2){7.5}}

\put(15,-27){\tiny{$j_{t+1}$}}

\put(13,-38){\small{$\mathbb{Y}_{t+1}$}}

\put(11,-20){$\cdots$}

\put(22.5,-15){\line(-3,2){22.5}}

\put(45,-30){\vector(-3,2){22.5}}

\put(46,-27){\tiny{$j_l$}}

\put(43,-38){\small{$\mathbb{Y}_l$}}

\put(0,-50){$\Gamma$}


\put(130,0){\vector(-3,2){22.5}}

\put(107.5,15){\line(-3,2){22.5}}

\put(82,25){\tiny{$i_1$}}

\put(82,32){\small{$\mathbb{X}_1$}}

\put(108,15){$\cdots$}

\put(130,0){\vector(-1,2){7.5}}

\put(122.5,15){\line(-1,2){7.5}}

\put(110,25){\tiny{$i_s$}}

\put(112,32){\small{$\mathbb{X}_s$}}

\put(130,0){\vector(1,2){7.5}}

\put(137.5,15){\line(1,2){7.5}}

\put(145,25){\tiny{$i_{s+1}$}}

\put(143,32){\small{$\mathbb{X}_{s+1}$}}

\put(141,15){$\cdots$}

\put(130,0){\vector(3,2){22.5}}

\put(152.5,15){\line(3,2){22.5}}

\put(176,25){\tiny{$i_k$}}

\put(173,32){\small{$\mathbb{X}_k$}}


\put(137,-2){$v$}


\put(107.5,-15){\line(3,2){22.5}}

\put(85,-30){\vector(3,2){22.5}}

\put(82,-27){\tiny{$j_1$}}

\put(82,-38){\small{$\mathbb{Y}_1$}}

\put(108,-20){$\cdots$}

\put(115,-30){\vector(1,1){7.5}}

\put(122.5,-22.5){\line(1,1){7.5}}

\put(110,-27){\tiny{$j_t$}}

\put(112,-38){\small{$\mathbb{Y}_t$}}

\put(145,-30){\vector(-1,1){7.5}}

\put(137.5,-22.5){\line(-1,1){7.5}}

\put(145,-27){\tiny{$j_{t+1}$}}

\put(143,-38){\small{$\mathbb{Y}_{t+1}$}}

\put(130,-7.5){\line(0,1){7.5}}

\put(130,-15){\vector(0,1){7.5}}

\put(129,-5){\line(1,0){2}}

\put(132,-12){\small{$\mathbb{B}$}}

\put(141,-20){$\cdots$}

\put(152.5,-15){\line(-3,2){22.5}}

\put(175,-30){\vector(-3,2){22.5}}

\put(176,-27){\tiny{$j_l$}}

\put(173,-38){\small{$\mathbb{Y}_l$}}

\put(130,-50){$\Gamma_2$}
\end{picture}

\caption{}\label{edge-contraction-figure}

\end{figure}

\begin{lemma}\label{edge-contraction}
Let $\Gamma$, $\Gamma_1$ and $\Gamma_2$ be the MOY graphs shown in Figure \ref{edge-contraction-figure}. Then $C(\Gamma_1) \simeq C(\Gamma_2) \simeq C(\Gamma)$.
\end{lemma}
\begin{proof}
We only prove that $C(\Gamma_1) \simeq C(\Gamma)$. The proof of $C(\Gamma_2) \simeq C(\Gamma)$ is similar. Set $m=i_1+i_2+\cdots +i_k = j_1+j_2+\cdots +j_l$. Let $R=\Sym(\mathbb{X}_1|\dots|\mathbb{X}_k|\mathbb{Y}_1|\dots|\mathbb{Y}_l)$, and $\widetilde{R}= \Sym(\mathbb{X}_1|\dots|\mathbb{X}_k|\mathbb{Y}_1|\dots|\mathbb{Y}_l|\mathbb{A})$. Set $\mathbb{X}=\mathbb{X}_1\cup\cdots\cup \mathbb{X}_k$ and $\mathbb{Y}=\mathbb{Y}_1\cup\cdots\cup \mathbb{Y}_l$. Denote by $X_j$ the $j$-th elementary symmetric polynomial in $\mathbb{X}$, by $Y_j$ the $j$-th elementary symmetric polynomial in $\mathbb{Y}$, and by $A_j$ the $j$-th elementary symmetric polynomial in $\mathbb{A}$. Moreover, denote by $X'_j$ the $j$-th elementary symmetric polynomial in $\mathbb{X}_1\cup\cdots\cup\mathbb{X}_{s-1}\cup\mathbb{X}_{s+2}\cup\cdots\cup \mathbb{X}_k$, and, for $i=s,s+1$, $X_{i,j}$ the $j$-th elementary symmetric polynomial in $\mathbb{X}_i$. Then 
\[
X_j= \sum_{p+q+r=j}X'_pX_{s,q}X_{s+1,r},
\]
the $j$-th elementary symmetric polynomial in $\mathbb{X}_s\cup\mathbb{X}_{s+1}$ is
\[
\sum_{p+q=j}X_{s,p}X_{s+1,q},
\]
and the $j$-th elementary symmetric polynomial in $\mathbb{X}_1\cup\cdots\cup\mathbb{X}_{s-1}\cup\mathbb{X}_{s+2}\cup\cdots\cup \mathbb{X}_k\cup \mathbb{A}$ is
\[
\sum_{p+q=j}X'_pA_q.
\]
Note that 
\[
\widetilde{R} = R[A_1- X_{s,1} - X_{s+1,1},\dots,A_j - \sum_{p+q=j}X_{s,p}X_{s+1,q},\dots,A_{i_s+i_{s+1}} - X_{s,i_{s}}X_{s+1,i_{s+1}}].
\] 
So, by Proposition \ref{b-contraction}, 
\begin{eqnarray*}
C(\Gamma_1) & \cong & 
\left(%
\begin{array}{ll}
  \ast & X'_1+A_1-Y_1 \\
  \dots & \dots \\
  \ast & \sum_{p+q=j}X'_pA_q-Y_j \\
  \dots & \dots \\
  \ast & X'_{m-i_s-i_{s+1}}A_{i_s+i_{s+1}}-Y_m \\
  \ast & A_1- X_{s,1} - X_{s+1,1} \\
  \dots & \dots \\
  \ast & A_j - \sum_{p+q=j}X_{s,p}X_{s+1,q} \\
  \dots & \dots \\
  \ast & A_{i_s+i_{s+1}} - X_{s,i_{s}}X_{s+1,i_{s+1}} 
\end{array}%
\right)_{\widetilde{R}}
\{q^{-\sum_{1\leq t_1<t_2 \leq k} i_{t_1}i_{t_2}}\} \\
 & \simeq &
\left(%
\begin{array}{ll}
  \ast & X_1-Y_1 \\
  \dots & \dots \\
  \ast & X_m-Y_m 
\end{array}%
\right)_{R}
\{q^{-\sum_{1\leq t_1<t_2 \leq k} i_{t_1}i_{t_2}}\} \\
& \cong & 
C(\Gamma).
\end{eqnarray*}
\end{proof}

Lemma \ref{edge-contraction} implies that the matrix factorization associated to any MOY graph is homotopic to that associated to a trivalent MOY graph. So, theoretically, we do not lose any information by considering only the trivalent MOY graphs. But, in some cases, it is more convenient to use vertices of higher valence.

\begin{figure}[ht]

\setlength{\unitlength}{1pt}

\begin{picture}(360,100)(-180,-50)


\put(-100,25){$\Gamma_1$:}

\put(-60,10){\vector(0,1){10}}

\put(-60,20){\vector(-1,1){20}}

\put(-60,20){\vector(1,1){10}}

\put(-50,30){\vector(-1,1){10}}

\put(-50,30){\vector(1,1){10}}

\put(-75,3){\tiny{$i+j+k$}}

\put(-55,21){\tiny{$j+k$}}

\put(-80,42){\tiny{$i$}}

\put(-60,42){\tiny{$j$}}

\put(-40,42){\tiny{$k$}}


\put(20,25){$\Gamma'_1$:}

\put(60,10){\vector(0,1){10}}

\put(60,20){\vector(1,1){20}}

\put(60,20){\vector(-1,1){10}}

\put(50,30){\vector(1,1){10}}

\put(50,30){\vector(-1,1){10}}

\put(45,3){\tiny{$i+j+k$}}

\put(38,21){\tiny{$i+j$}}

\put(80,42){\tiny{$k$}}

\put(60,42){\tiny{$j$}}

\put(40,42){\tiny{$i$}}


\put(-100,-25){$\Gamma_2$:}

\put(-60,-30){\vector(0,-1){10}}

\put(-80,-10){\vector(1,-1){20}}

\put(-50,-20){\vector(-1,-1){10}}

\put(-60,-10){\vector(1,-1){10}}

\put(-40,-10){\vector(-1,-1){10}}

\put(-75,-47){\tiny{$i+j+k$}}

\put(-55,-29){\tiny{$j+k$}}

\put(-80,-8){\tiny{$i$}}

\put(-60,-8){\tiny{$j$}}

\put(-40,-8){\tiny{$k$}}


\put(20,-25){$\Gamma'_2$:}

\put(60,-30){\vector(0,-1){10}}

\put(80,-10){\vector(-1,-1){20}}

\put(50,-20){\vector(1,-1){10}}

\put(60,-10){\vector(-1,-1){10}}

\put(40,-10){\vector(1,-1){10}}

\put(45,-47){\tiny{$i+j+k$}}

\put(38,-29){\tiny{$i+j$}}

\put(80,-8){\tiny{$k$}}

\put(60,-8){\tiny{$j$}}

\put(40,-8){\tiny{$i$}}

\end{picture}

\caption{}\label{contract-expand-figure}

\end{figure}

\begin{corollary}\label{contract-expand}
Suppose that $\Gamma_1$, $\Gamma'_1$, $\Gamma_2$ and $\Gamma'_2$ are the MOY graphs shown in Figure \ref{contract-expand-figure}. Then $C(\Gamma_1) \simeq C(\Gamma'_1)$ and $C(\Gamma_2) \simeq C(\Gamma'_2)$.
\end{corollary}
\begin{proof}
This is a special case of Lemma \ref{edge-contraction}.
\end{proof}

\subsection{Direct sum decomposition (II)} We now generalize direct sum decomposition (II) in \cite{KR1}.

\begin{figure}[ht]

\setlength{\unitlength}{1pt}

\begin{picture}(360,70)(-180,-10)


\put(-60,0){\vector(0,1){15}}

\qbezier(-60,15)(-70,15)(-70,25)

\put(-70,25){\vector(0,1){10}}

\qbezier(-70,35)(-70,45)(-60,45)

\put(-71,26){\line(1,0){2}}

\qbezier(-60,15)(-50,15)(-50,25)

\put(-50,25){\vector(0,1){10}}

\qbezier(-50,35)(-50,45)(-60,45)

\put(-51,26){\line(1,0){2}}

\put(-60,45){\vector(0,1){15}}

\put(-65,55){\tiny{$n$}}

\put(-58,54){\small{$\mathbb{Y}$}}

\put(-65,0){\tiny{$n$}}

\put(-58,0){\small{$\mathbb{X}$}}

\put(-77,30){\tiny{$m$}}

\put(-77,22){\small{$\mathbb{A}$}}

\put(-47,30){\tiny{$n-m$}}

\put(-48,22){\small{$\mathbb{B}$}}

\put(-63,-10){$\Gamma$}


\put(60,0){\vector(0,1){60}}

\put(62,54){\small{$\mathbb{Y}$}}

\put(55,30){\tiny{$n$}}

\put(62,0){\small{$\mathbb{X}$}}

\put(57,-10){$\Gamma_1$}

\end{picture}

\caption{}\label{decomp-II-figure}

\end{figure}

\begin{theorem}[Direction Sum Decomposition II]\label{decomp-II}
Suppose that $\Gamma$ and $\Gamma_1$ are the MOY graphs shown in Figure \ref{decomp-II-figure}, where $n\geq m \geq 0$. Then 
\[
C(\Gamma) \simeq C(\Gamma_1)\{\qb{n}{m}\}.
\]
\end{theorem}
\begin{proof}
Denote by $X_j$ be $j$-th elementary symmetric polynomial in $\mathbb{X}$, and use similar notations for the other alphabets. Let $\mathbb{W}=\mathbb{A}\cup\mathbb{B}$. Then the $j$-th elementary symmetric polynomial in $\mathbb{W}$ is
\[
W_j=\sum_{p+q=j}A_pB_q.
\] 
By Theorem \ref{part-symm-str} and Corollary \ref{part-symm-grade}, 
\[
\Sym(\mathbb{X}|\mathbb{Y}|\mathbb{A}|\mathbb{B}) = \Sym(\mathbb{X}|\mathbb{Y}|\mathbb{W})\{q^{m(n-m)}\qb{n}{m}\}.
\]
So
\begin{eqnarray*}
C(\Gamma) & \cong & 
\left(%
\begin{array}{ll}
  \ast & Y_1 - W_1 \\
  \dots & \dots \\
  \ast & Y_n - W_n \\
  \ast & W_1 - X_1 \\
  \dots & \dots \\
  \ast & W_n - X_m 
\end{array}%
\right)_{\Sym(\mathbb{X}|\mathbb{Y}|\mathbb{A}|\mathbb{B})}
\{q^{-m(n-m)}\}, \\
& \cong & 
\left(%
\begin{array}{ll}
  \ast & Y_1 - W_1 \\
  \dots & \dots \\
  \ast & Y_n - W_n \\
  \ast & W_1 - X_1 \\
  \dots & \dots \\
  \ast & W_n - X_m 
\end{array}%
\right)_{\Sym(\mathbb{X}|\mathbb{Y}|\mathbb{W})}
\{\qb{n}{m}\}, \\
& \simeq & 
\left(%
\begin{array}{ll}
  \ast & Y_1 - X_1 \\
  \dots & \dots \\
  \ast & Y_n - X_n
\end{array}%
\right)_{\Sym(\mathbb{X}|\mathbb{Y})}
\{\qb{n}{m}\}, \\
& \cong & C(\Gamma_1)\{\qb{n}{m}\}.
\end{eqnarray*}
where the homotopy is given by Proposition \ref{b-contraction}. 
\end{proof}

\subsection{Direct sum decomposition (I)} We now generalize direct sum decomposition (I) in \cite{KR1}. We start with a special case. 

\begin{figure}[ht]

\setlength{\unitlength}{1pt}

\begin{picture}(360,70)(-180,-10)


\put(-120,0){\vector(0,1){15}}

\put(-120,15){\vector(0,1){15}}

\put(-120,30){\line(0,1){15}}

\put(-120,45){\vector(0,1){15}}

\qbezier(-120,15)(-100,5)(-100,25)

\qbezier(-120,45)(-100,55)(-100,35)

\put(-100,35){\vector(0,-1){10}}

\put(-130,55){\tiny{$m$}}

\put(-130,0){\tiny{$m$}}

\put(-130,30){\tiny{$N$}}

\put(-95,30){\tiny{$N-m$}}

\put(-123,-10){$\Gamma$}


\put(0,0){\vector(0,1){60}}

\put(5,54){\small{$\mathbb{Y}$}}

\put(-10,30){\tiny{$m$}}

\put(5,0){\small{$\mathbb{X}$}}

\put(-3,-10){$\Gamma_1$}


\put(110,0){\vector(1,3){10}}

\put(120,30){\vector(-1,3){10}}

\put(120,30){\line(1,3){5}}

\qbezier(125,45)(130,60)(130,45)

\put(120,30){\line(1,-3){5}}

\qbezier(125,15)(130,0)(130,15)

\put(130,45){\vector(0,-1){15}}

\put(130,30){\line(0,-1){15}}

\put(100,0){\tiny{$m$}}

\put(100,55){\tiny{$m$}}

\put(115,54){\small{$\mathbb{Y}$}}

\put(115,0){\small{$\mathbb{X}$}}

\put(129,28){\line(1,0){2}}

\put(135,25){\small{$\mathbb{W}$}}

\put(135,35){\tiny{$N-m$}}

\put(117,-10){$\Gamma'$}

\end{picture}

\caption{}\label{decomp-I-special-figure}

\end{figure}

\begin{lemma}\label{decomp-I-special}
Suppose that $\Gamma$ and $\Gamma_1$ are the MOY graphs shown in Figure \ref{decomp-I-special-figure}. Then 
$C(\Gamma) \simeq C(\Gamma_1)\left\langle N-m \right\rangle$.
\end{lemma}
\begin{proof}
By Lemma \ref{edge-contraction}, we have $C(\Gamma) \simeq C(\Gamma')$. So we only need to show that $C(\Gamma') \simeq C(\Gamma_1)\left\langle N-m \right\rangle$. We put markings on $\Gamma'$ and $\Gamma_1$ as in Figure \ref{decomp-I-special-figure}. Denote by $X_j$ the $j$-th elementary symmetric polynomial in $\mathbb{X}$, and use similar notations for the other alphabets. Write $\mathbb{A}=\mathbb{Y}\cup \mathbb{W}$ and $\mathbb{B}=\mathbb{X}\cup \mathbb{W}$. Then the $j$-th elementary symmetric polynomials in $\mathbb{A}$ and $\mathbb{B}$ are
\begin{eqnarray*}
A_j & = & \sum_{k+l=j} Y_kW_l, \\
B_j & = & \sum_{k+l=j} X_kW_l.
\end{eqnarray*}
Define
\[
U_j = \frac{p_{N,N+1}(B_1,\dots,B_{j-1},A_j,\dots,A_m) - p_{N,N+1}(B_1,\dots,B_j,A_{j+1},\dots,A_m)}{A_j-B_j}.
\]
Then
\[
C(\Gamma') =
\left(%
\begin{array}{cc}
  U_1 & A_1-B_1 \\
  U_2 & A_2-B_2 \\
  \dots & \dots \\
  U_N & A_N-B_N
\end{array}%
\right)_{\Sym(\mathbb{X}|\mathbb{Y}|\mathbb{W})}
\{q^{-m(N-m)}\}.
\]
Using the relation $A_j-B_j=\sum_{k+l=j} (Y_k-X_k)W_l$ and, specially, $A_1-B_1 =Y_1-X_1$, we can inductive change the entries in the right column into $Y_1-X_1, Y_2-X_2,\dots, Y_m-X_m,0\dots,0$ by the row operation given in Corollary \ref{row-op}. Note that these row operations do not change $U_{m+1},\dots,U_N$ in the left column. Thus,
\[
C(\Gamma') \cong
\left(%
\begin{array}{cc}
  \ast & Y_1-X_1 \\
  \dots & \dots \\
  \ast & Y_m-X_m \\
  U_{m+1} & 0 \\
  \dots & \dots \\
  U_N & 0
\end{array}%
\right)_{\Sym(\mathbb{X}|\mathbb{Y}|\mathbb{W})}
\{q^{-m(N-m)}\}.
\]
Using Newton's Identity \eqref{newton}, one can verify that
\[
p_{N,N+1}(A_1,\dots,A_N) = f_j + A_{N+1-j}(c_jA_j+g_j),
\]
where $f_j$ is a polynomial in $A_1,\dots,A_{N-j},A_{N+2-j},\dots,A_N$, and $g_j$ is a polynomial in $A_1,\dots,A_{j-1}$, and 
\[
c_j =\left\{%
\begin{array}{ll}
  (-1)^{N+1}\frac{N+1}{2}, & \text{ if } N+1-j=j, \\
  (-1)^{N+1}(N+1), & \text{ if } N+1-j \neq j.
\end{array}%
\right.
\]
Therefore, 
\begin{eqnarray*}
&  & U_{N+1-j}  \\
& & \\
& = & 
\left\{%
\begin{array}{ll}
  (-1)^{N+1}(N+1)B_j + \alpha_j(B_1,\dots,B_{j-1}), & \text{ if } N+1-j>j, \\
  (-1)^{N+1}\frac{N+1}{2}(A_j+B_j) + \beta_j(B_1,\dots,B_{j-1}), & \text{ if } N+1-j=j, \\
  (-1)^{N+1}(N+1)A_j + \gamma_j(B_1,\dots,B_{N+1-j},A_{N+1-j},\dots,A_{j-1}), & \text{ if } N+1-j<j,
\end{array}%
\right.
\end{eqnarray*}
where $\alpha_j,\beta_j,\gamma_j$ are polynomials in the given indeterminates.

So, for $j=1,\dots,N-m$, $U_{N+1-j}$ can be expressed as a polynomial
\[
U_{N+1-j} = (-1)^{N+1}(N+1)W_j + u_j(X_1,\dots,X_m,Y_1,\dots,Y_m,W_1,\dots,W_{j-1}).
\]
This implies that $U_{N},\dots,U_{m+1}$ are independent indeterminates over $\Sym(\mathbb{X}|\mathbb{Y})$, and $\Sym(\mathbb{X}|\mathbb{Y}|\mathbb{W}) = \Sym(\mathbb{X}|\mathbb{Y})[U_{N},\dots,U_{m+1}]$. Hence, by Corollary \ref{a-contraction},
\begin{eqnarray*}
& & \left(%
\begin{array}{cc}
  \ast & Y_1-X_1 \\
  \dots & \dots \\
  \ast & Y_m-X_m \\
  U_{m+1} & 0 \\
  \dots & \dots \\
  U_N & 0
\end{array}%
\right)_{\Sym(\mathbb{X}|\mathbb{Y}|\mathbb{W})}
\{q^{-m(N-m)}\} \\
& \simeq &
\left(%
\begin{array}{cc}
  \ast & Y_1-X_1 \\
  \dots & \dots \\
  \ast & Y_m-X_m 
\end{array}%
\right)_{\Sym(\mathbb{X}|\mathbb{Y})}
\{q^{-m(N-m)+\sum_{j=m+1}^N(N+1-\deg U_j)}\}\left\langle N-m\right\rangle \\
& \cong & C(\Gamma_1)\left\langle N-m\right\rangle.
\end{eqnarray*}
Thus, $C(\Gamma) \simeq C(\Gamma') \simeq C(\Gamma_1)\left\langle N-m \right\rangle$.
\end{proof}

The general case follows easily from Lemma \ref{decomp-I-special}.

\begin{figure}[ht]

\setlength{\unitlength}{1pt}

\begin{picture}(360,160)(-180,-100)


\put(-120,0){\vector(0,1){15}}

\put(-120,15){\vector(0,1){15}}

\put(-120,30){\line(0,1){15}}

\put(-120,45){\vector(0,1){15}}

\qbezier(-120,15)(-100,5)(-100,25)

\qbezier(-120,45)(-100,55)(-100,35)

\put(-100,35){\vector(0,-1){10}}

\put(-130,55){\tiny{$m$}}

\put(-130,0){\tiny{$m$}}

\put(-145,30){\tiny{$m+n$}}

\put(-95,30){\tiny{$n$}}

\put(-123,-10){$\Gamma$}


\put(0,0){\vector(0,1){60}}

\put(-10,30){\tiny{$m$}}

\put(-3,-10){$\Gamma_1$}


\put(120,0){\vector(0,1){5}}

\put(120,5){\vector(0,1){25}}

\put(120,30){\line(0,1){15}}

\put(120,45){\vector(0,1){15}}

\qbezier(120,5)(170,0)(170,25)

\qbezier(120,55)(170,60)(170,35)

\put(170,35){\vector(0,-1){10}}

\qbezier(120,15)(130,5)(130,25)

\qbezier(120,45)(130,55)(130,35)

\put(130,35){\vector(0,-1){10}}

\put(110,55){\tiny{$m$}}

\put(110,0){\tiny{$m$}}

\put(95,45){\tiny{$m+n$}}

\put(110,30){\tiny{$N$}}

\put(95,10){\tiny{$m+n$}}

\put(130,30){\tiny{$N-m-n$}}

\put(175,30){\tiny{$n$}}

\put(117,-10){$\Gamma_2$}


\put(120,-90){\vector(0,1){5}}

\put(120,-85){\vector(0,1){25}}

\put(120,-60){\line(0,1){15}}

\put(120,-45){\vector(0,1){15}}

\qbezier(150,-80)(170,-70)(170,-65)

\qbezier(150,-40)(170,-50)(170,-55)

\put(170,-55){\vector(0,-1){10}}

\qbezier(150,-80)(130,-70)(130,-65)

\qbezier(150,-40)(130,-50)(130,-55)

\put(130,-55){\vector(0,-1){10}}

\put(120,-40){\vector(1,0){30}}

\put(150,-80){\vector(-1,0){30}}

\put(110,-35){\tiny{$m$}}

\put(110,-90){\tiny{$m$}}

\put(110,-60){\tiny{$N$}}

\put(130,-60){\tiny{$N-m-n$}}

\put(125,-38){\tiny{$N-m$}}

\put(125,-87){\tiny{$N-m$}}

\put(175,-60){\tiny{$n$}}

\put(117,-100){$\Gamma_3$}


\put(0,-90){\vector(0,1){15}}

\put(0,-75){\vector(0,1){15}}

\put(0,-60){\line(0,1){15}}

\put(0,-45){\vector(0,1){15}}

\qbezier(0,-75)(20,-85)(20,-65)

\qbezier(0,-45)(20,-35)(20,-55)

\put(20,-55){\vector(0,-1){10}}

\put(-10,-35){\tiny{$m$}}

\put(-10,-60){\tiny{$N$}}

\put(-10,-90){\tiny{$m$}}

\put(25,-60){\tiny{$N-m$}}

\put(-3,-100){$\Gamma_4$}

\end{picture}

\caption{}\label{decomp-I-figure}

\end{figure}

\begin{theorem}[Direction Sum Decomposition I]\label{decomp-I}
Suppose that $\Gamma$ and $\Gamma_1$ are the MOY graphs shown in Figure \ref{decomp-I-figure}. Then 
\[
C(\Gamma) \simeq C(\Gamma_1)\{ \qb{N-m}{n}\}\left\langle n \right\rangle.
\]
\end{theorem}
\begin{proof}
Consider the MOY graphs in Figure \ref{decomp-I-figure}. By Lemma \ref{decomp-I-special}, $C(\Gamma)\simeq C(\Gamma_2)\left\langle N-m-n\right\rangle$. By Corollary \ref{contract-expand}, $C(\Gamma_2)\simeq C(\Gamma_3)$. By Theorem \ref{decomp-II}, $C(\Gamma_3)\simeq C(\Gamma_4)\qb{N-m}{n}$. And by Lemma \ref{decomp-I-special} again, $C(\Gamma_4)\simeq C(\Gamma_1)\left\langle N-m\right\rangle$. Putting everything together, we get $C(\Gamma) \simeq C(\Gamma_1)\{ \qb{N-m}{n}\}\left\langle n \right\rangle$.
\end{proof}

\section{Circles}\label{sec-circles}

In this section, we study matrix factorizations associated to circles. The results will be useful in Section \ref{sec-morph}.

\subsection{Homotopy type} The following describes the homotopy type of the matrix factorization associated to a colored circle and follows easily from Direction Sum Decompositions I and II (Theorems \ref{decomp-I} and \ref{decomp-II}.)

\begin{figure}[ht]

\setlength{\unitlength}{1pt}

\begin{picture}(360,70)(-220,-10)


\qbezier(-150,60)(-170,60)(-170,45)

\qbezier(-150,60)(-130,60)(-130,45)

\put(-170,15){\vector(0,1){30}}

\put(-130,15){\line(0,1){30}}

\qbezier(-150,0)(-170,0)(-170,15)

\qbezier(-150,0)(-130,0)(-130,15)

\put(-153,-10){$\Gamma$}

\put(-165,30){\tiny{$m$}}
 

\qbezier(-80,60)(-100,60)(-100,45)

\qbezier(-80,60)(-60,60)(-60,45)

\put(-100,15){\vector(0,1){15}}

\put(-100,30){\line(0,1){15}}

\put(-60,45){\vector(0,-1){30}}

\qbezier(-80,0)(-100,0)(-100,15)

\qbezier(-80,0)(-60,0)(-60,15)

\qbezier(-100,45)(-90,60)(-90,45)

\qbezier(-100,15)(-90,0)(-90,15)

\put(-90,45){\vector(0,-1){30}}

\put(-83,-10){$\Gamma_1$}

\put(-110,30){\tiny{$N$}}

\put(-90,30){\tiny{$N-m$}}

\put(-55,30){\tiny{$m$}}


\qbezier(-10,60)(-30,60)(-30,45)

\qbezier(-10,60)(10,60)(10,45)

\put(-30,15){\vector(0,1){15}}

\put(-30,30){\line(0,1){15}}

\qbezier(-10,0)(-30,0)(-30,15)

\qbezier(-10,0)(10,0)(10,15)

\qbezier(10,45)(-5,40)(-5,35)

\qbezier(10,45)(25,40)(25,35)

\put(-5,35){\vector(0,-1){10}}

\put(25,35){\vector(0,-1){10}}

\qbezier(10,15)(-5,20)(-5,25)

\qbezier(10,15)(25,20)(25,25)

\put(-40,30){\tiny{$N$}}

\put(-3,30){\tiny{$N-m$}}

\put(27,30){\tiny{$m$}}

\put(-13,-10){$\Gamma_2$} 


\qbezier(80,60)(60,60)(60,45)

\qbezier(80,60)(100,60)(100,45)

\put(60,15){\vector(0,1){30}}

\put(100,15){\line(0,1){30}}

\qbezier(80,0)(60,0)(60,15)

\qbezier(80,0)(100,0)(100,15)

\put(77,-10){$\Gamma_3$}

\put(65,30){\tiny{$N$}}

\put(99,30){\line(1,0){2}}

\put(105,28){\small{$\mathbb{X}$}}

\end{picture}

\caption{}\label{circle-dimension-figure}

\end{figure}

\begin{corollary}\label{circle-dimension}
If $\Gamma$ is a circle colored by $m$, then $C(\Gamma) \simeq C(\emptyset)\{\qb{N}{m}\}\left\langle m \right\rangle$, where $C(\emptyset)$ is the matrix factorization $\C\rightarrow 0 \rightarrow \C$. As a consequence, $H(\Gamma) \cong C(\emptyset)\{\qb{N}{m}\}\left\langle m \right\rangle$.
\end{corollary}
\begin{proof}
Consider $\Gamma_3$ in Figure \ref{circle-dimension-figure} first, which is the special case when $m=N$. Note that, by Lemma \ref{power-derive},
\[
C(\Gamma_3) \cong \left(%
\begin{array}{cc}
  \frac{\partial p_{N+1}(\mathbb{X})}{\partial X_1} & 0 \\
  \dots & \dots \\
  \frac{\partial p_{N+1}(\mathbb{X})}{\partial X_k} & 0 \\
  \dots & \dots \\
  \frac{\partial p_{N+1}(\mathbb{X})}{\partial X_N} & 0 \\
\end{array}%
\right)_{\Sym(\mathbb{X})}
=
\left(%
\begin{array}{cc}
  (N+1)h_N(\mathbb{X}) & 0 \\
  \dots & \dots \\
  (-1)^{k+1}(N+1)h_{N+1-k}(\mathbb{X}) & 0 \\
  \dots & \dots \\
  (-1)^{N+1}(N+1)h_1(\mathbb{X}) & 0 \\
\end{array}%
\right)_{\Sym(\mathbb{X})}
\]
where $X_k$ is the $k$-th elementary symmetric polynomial in $\mathbb{X}$. But
\[
\Sym(\mathbb{X}) = \C[h_1(\mathbb{X}),\dots,h_N(\mathbb{X})].
\]
So, by applying Corollary \ref{a-contraction} repeatedly, we get $C(\Gamma_3) \simeq C(\emptyset)\left\langle N \right\rangle$. 

For the general case, using Theorem \ref{decomp-II} and Lemma \ref{decomp-I-special}, we have 
\[
C(\Gamma) \simeq C(\Gamma_1)\left\langle N-m\right\rangle = C(\Gamma_2) \left\langle N-m\right\rangle \simeq C(\Gamma_3) \{\qb{N}{m} \}\left\langle N-m\right\rangle.
\]
So $C(\Gamma) \simeq C(\emptyset) \{\qb{N}{m}\}\left\langle m \right\rangle$.
\end{proof}

\subsection{Module structure of the homology} Now Let $\bigcirc_m$ be a circle colored by $m$ and marked by a single point as shown in Figure \ref{circle-module-figure}. Then $H(\bigcirc_m)$ is a graded $\Sym(\mathbb{X})$-module. In this subsection, we show that, after a grading shift, this module is isomorphic to the cohomology of the complex $(m,N)$-Grassmannian $G_{m,N}$. In particular, as a $\Sym(\mathbb{X})$-module, $H(\bigcirc_m)$ is generated by a single generator.

\begin{figure}[ht]

\setlength{\unitlength}{1pt}

\begin{picture}(360,60)(-180,0)


\qbezier(0,60)(-20,60)(-20,45)

\qbezier(0,60)(20,60)(20,45)

\put(-20,15){\vector(0,1){30}}

\put(20,15){\line(0,1){30}}

\qbezier(0,0)(-20,0)(-20,15)

\qbezier(0,0)(20,0)(20,15)

\put(19,30){\line(1,0){2}}

\put(-15,30){\tiny{$m$}}

\put(25,25){\small{$\mathbb{X}$}}
 
\end{picture}

\caption{}\label{circle-module-figure}

\end{figure}

We need the following fact about symmetric polynomials to carry out our proof.

\begin{proposition}\label{complete-non-zero-divisor}
Let $\mathbb{X}=\{x_1,\dots,x_m\}$ be an alphabet with $m$ independent indeterminates. If $n \geq m$, then the sequence $\{h_n(\mathbb{X}), h_{n-1}(\mathbb{X}), \dots, h_{n+1-m}(\mathbb{X})\}$ is $\Sym(\mathbb{X})$-regular. (See Definition \ref{regular-sequence}.)
\end{proposition}
\begin{proof}
For $n,j\geq 1$, define a ideal $\mathcal{I}_{n,j}$ of $\Sym(\mathbb{X})$ by $\mathcal{I}_{n,1}=\{0\}$ and $\mathcal{I}_{n,j}=(h_n(\mathbb{X}),h_{n-1}(\mathbb{X}),\dots,h_{n+2-j}(\mathbb{X}))$ for $j\geq 2$. For $1\leq j \leq m \leq n$, let $P_{m,n,j}$ and $Q_{m,n,j}$ be the following statements:
\begin{itemize}
	\item $P_{m,n,j}$: ``$h_{n+1-j}(\mathbb{X})$ is not a zero divisor in $\Sym(\mathbb{X})/\mathcal{I}_{n,j}$."
	\item $Q_{m,n,j}$: ``$X_m=x_1\cdots x_m$ is not a zero divisor in $\Sym(\mathbb{X})/\mathcal{I}_{n,j}$."
\end{itemize}
We prove these two statements by induction for all $m,n,j$ satisfying $1\leq j \leq m \leq n$. Note that, by Definition \ref{regular-sequence}, $\{h_n(\mathbb{X}), h_{n-1}(\mathbb{X}), \dots, h_{n+1-m}(\mathbb{X})\}$ is $\Sym(\mathbb{X})$-regular if and only if $P_{m,n,j}$ is true for $1 \leq j \leq m$.

If $m=1$, then $1\leq j \leq m$ forces $j=1$. Since $\mathcal{I}_{n,1}=\{0\}$, $P_{1,n,1}$ and $Q_{1,n,1}$ are trivially true for all $n\geq 1$. Assume that, for some $m\geq 2$, $P_{m-1,n,j}$ and $Q_{m-1,n,j}$ are true for all $n, j$ with $1\leq j \leq m-1 \leq n$. Consider $P_{m,n,j}$ and $Q_{m,n,j}$ for $n,j$ satisfying $1\leq j \leq m \leq n$.

\textbf{(i)} First, we prove $Q_{m,n,j}$ for all $n, j$ with $1\leq j \leq m \leq n$ by induction on $j$. When $j=1$, $\mathcal{I}_{n,j}=\mathcal{I}_{n,1}=\{0\}$. So $Q_{m,n,1}$ is trivially true. Assume that $Q_{m,n,j-1}$ is true for some $j\geq 2$. Assume $g,g_n,\dots,g_{n+2-j} \in \Sym(\mathbb{X})$ satisfy that 
\begin{equation}\label{eq-complete-non-zero-divisor-1}
gX_m = \sum_{k=n+2-j}^n g_k h_k(\mathbb{X}).
\end{equation}
Note that $g,g_n,\dots,g_{n+2-j}$ are polynomials in $X_1,\dots,X_m$. We shall write 
\[
g=g(X_1,\dots,X_m),~g_n=g(X_1,\dots,X_m),\dots,~g_{n+2-j}=g(X_1,\dots,X_m).
\]

Denote by $X'_j$ the $j$-th elementary symmetric polynomial in $\mathbb{X}'=\{x_1,\dots,x_{m-1}\}$. Then $X_j|_{x_m=0}=X'_j$ and $h_j(\mathbb{X})|_{x_m=0}=h_j(\mathbb{X}')$. Plug $x_m=0$ into \eqref{eq-complete-non-zero-divisor-1}. We get
\[
\sum_{k=n+2-j}^n g_k(X'_1,\dots,X'_{m-1},0) h_k(\mathbb{X}')=0.
\]
Specially, 
\[
g_{n+2-j}(X'_1,\dots,X'_{m-1},0) h_{n+2-j}(\mathbb{X}') \in (h_n(\mathbb{X}'),h_{n-1}(\mathbb{X}'),\dots,h_{n+3-j}(\mathbb{X}')) \subset \Sym(\mathbb{X}').
\] 
But Statement $P_{m-1,n,j-1}$ is true. So 
\[
g_{n+2-j}(X'_1,\dots,X'_{m-1},0) \in (h_n(\mathbb{X}'),h_{n-1}(\mathbb{X}'),\dots,h_{n+3-j}(\mathbb{X}')).
\] 
That is, 
\begin{eqnarray*}
g_{n+2-j}(X'_1,\dots,X'_{m-1},0) & = & \sum_{k=n+3-j}^n \alpha_{k}(X'_1,\dots,X'_{m-1})h_k(\mathbb{X}') \\
& = & \sum_{k=n+3-j}^n \alpha_{k}(X'_1,\dots,X'_{m-1}) h_{m,k}(X'_1,\dots,X'_{m-1},0).
\end{eqnarray*}
Note that $X'_1,\dots,X'_{m-1}$ are independent indeterminates over $\C$. So the above equation remains true when we replace $X'_1,\dots,X'_{m-1}$ by any other variables. In particular,
\[
g_{n+2-j}(X_1,\dots,X_{m-1},0) = \sum_{k=n+3-j}^n \alpha_{k}(X_1,\dots,X_{m-1}) h_{m,k}(X_1,\dots,X_{m-1},0),
\]
which implies that there exits $\alpha\in\Sym(\mathbb{X})$ such that
\begin{eqnarray*}
&  & g_{n+2-j}(X_1,\dots,X_{m-1},X_m)  \\ 
& = & \alpha X_m + \sum_{k=n+3-j}^n \alpha_{k}(X_1,\dots,X_{m-1}) h_{m,k}(X_1,\dots,X_{m-1},X_m) \\
& = & \alpha X_m + \sum_{k=n+3-j}^n \alpha_{k}(X_1,\dots,X_{m-1}) h_k(\mathbb{X}).
\end{eqnarray*}
Plug this into \eqref{eq-complete-non-zero-divisor-1}. We get
\[
(g-\alpha h_{n+2-j}(\mathbb{X}))X_m = \sum_{k=n+3-j}^n (g_k+\alpha_k(X_1,\dots,X_{m-1}) h_{n+2-j}(\mathbb{X}))h_k(\mathbb{X}).
\]
But $Q_{m,n,j-1}$ is true. So $g-\alpha h_{n+2-j}(\mathbb{X}) \in \mathcal{I}_{n,j-1}$ and, therefore, $g \in \mathcal{I}_{n,j}$. This proves $Q_{m,n,j}$. Thus, $Q_{m,n,j}$ is true for all $n,j$ satisfying $1\leq j \leq m \leq n$.

\textbf{(ii)} Now we prove $P_{m,n,j}$ for all $n, j$ with $1\leq j \leq m \leq n$.

\emph{Case A.} $1\leq j \leq m-1$. Assume that $h_{n+1-j}(\mathbb{X})$ is a zero divisor in $\Sym(\mathbb{X})/\mathcal{I}_{n,j}$. Define 
\[
\Lambda = \{g \in \Sym(\mathbb{X})~|~g \text{ is homogeneous, } g \notin \mathcal{I}_{n,j},~ gh_{n+1-j}(\mathbb{X}) \in \mathcal{I}_{n,j}\}.
\]
Then $\Lambda \neq \emptyset$. Write $2\nu = \min_{g\in \Lambda} \deg g$. (Recall that we use the degree convention $\deg x_j=2$.) Let $g$ be such that $g\in \Lambda$ and $\deg g =2\nu$. Then there exist $g_n, g_{n-1},\dots,g_{n+2-j} \in \Sym(\mathbb{X})$ such that $\deg g_k = 2(\nu+n+1-j-k)$ and 
\begin{equation}\label{eq-complete-non-zero-divisor-2}
g h_{n+1-j}(\mathbb{X}) = \sum_{k=n+2-j}^n g_k h_k(\mathbb{X}).
\end{equation} 
Note that $g,g_n,\dots,g_{n+2-j}$ are polynomials in $X_1,\dots,X_m$. We shall write 
\[
g=g(X_1,\dots,X_m),~g_n=g(X_1,\dots,X_m),\dots,~g_{n+2-j}=g(X_1,\dots,X_m).
\]
In particular,
\begin{equation}\label{eq-g-expression}
g=g(X_1,\dots,X_m) = \sum_{l=0}^{\left\lfloor \frac{\nu}{m} \right\rfloor} f_l(X_1,\dots,X_{m-1}) X_m^l,
\end{equation}
where $f_l(X_1,\dots,X_{m-1})\in \Sym(\mathbb{X})$ is homogeneous of degree $2(\nu-lm)$.

Plugging $x_m=0$ into \eqref{eq-complete-non-zero-divisor-2}, we get
\[
f_0(X'_1,\dots,X'_{m-1})h_{n+1-j}(\mathbb{X}') = \sum_{k=n+2-j}^n g_k(X'_1,\dots,X'_{m-1},0) h_k(\mathbb{X}'),
\]
where $\mathbb{X}'=\{x_1,\dots,x_{m-1}\}$ and $X'_j$ is the $j$-th elementary symmetric polynomial in $\mathbb{X}'$. But $P_{m-1,n,j}$ is true since $1\leq j \leq m-1 <n$. So 
\[
f_0(X'_1,\dots,X'_{m-1}) \in (h_n(\mathbb{X}'),h_{n-1}(\mathbb{X}'),\dots,h_{n+2-j}(\mathbb{X}')) \subset \Sym(\mathbb{X}').
\]
Thus, 
\begin{eqnarray*}
f_0(X'_1,\dots,X'_{m-1}) & = & \sum_{k=n+2-j}^n \alpha_{k}(X'_1,\dots,X'_{m-1})h_k(\mathbb{X}') \\
& = & \sum_{k=n+2-j}^n \alpha_{k}(X'_1,\dots,X'_{m-1}) h_{m,k}(X'_1,\dots,X'_{m-1},0),
\end{eqnarray*}
where $\alpha_k(X'_1,\dots,X'_{m-1}) \in \Sym(\mathbb{X}')$ is homogeneous of degree $2(\nu-k)$.
But $X'_1,\dots,X'_{m-1}$ are independent indeterminates over $\C$. So the above equation remains true when we replace $X'_1,\dots,X'_{m-1}$ by any other variables. In particular,
\begin{eqnarray}
\label{eq-f-0-sum}&  & f_0(X_1,\dots,X_{m-1})  \\ 
\nonumber & = & \sum_{k=n+2-j}^n \alpha_{k}(X_1,\dots,X_{m-1}) h_{m,k}(X_1,\dots,X_{m-1},0) \\
\nonumber & = & \alpha X_m + \sum_{k=n+2-j}^n \alpha_{k}(X_1,\dots,X_{m-1}) h_k(\mathbb{X}),
\end{eqnarray}
where $\alpha\in\Sym(\mathbb{X})$ is homogeneous of degree $2(\nu-m)$. Plug this in to \eqref{eq-complete-non-zero-divisor-2}. We get
\begin{eqnarray*}
& & X_m (\alpha + \sum_{l=1}^{\left\lfloor \frac{\nu}{m} \right\rfloor} f_l(X_1,\dots,X_{m-1}) X_m^{l-1}) h_{n+1-j}(\mathbb{X}) \\ 
& = & \sum_{k=n+2-j}^n (g_k -\alpha_{k}(X_1,\dots,X_{m-1}) h_{n+1-j}(\mathbb{X})) h_k(\mathbb{X})  \hspace{2pc} \in \mathcal{I}_{n,j}.
\end{eqnarray*}
By $Q_{m,n,j}$, we have $(\alpha + \sum_{l=1}^{\left\lfloor \frac{\nu}{m} \right\rfloor} f_l(X_1,\dots,X_{m-1}) X_m^{l-1}) h_{n+1-j}(\mathbb{X}) \in  \mathcal{I}_{n,j}$. But $\alpha + \sum_{l=1}^{\left\lfloor \frac{\nu}{m} \right\rfloor} f_l(X_1,\dots,X_{m-1}) X_m^{l-1}$ is homogeneous of degree $2(\nu-m)<2\nu$. By the definition of $\nu$, this implies that $\alpha + \sum_{l=1}^{\left\lfloor \frac{\nu}{m} \right\rfloor} f_l(X_1,\dots,X_{m-1}) X_m^{l-1} \in \mathcal{I}_{n,j}$. Then, by \eqref{eq-g-expression} and \eqref{eq-f-0-sum},
\[
g = X_m(\alpha + \sum_{l=1}^{\left\lfloor \frac{\nu}{m} \right\rfloor} f_l(X_1,\dots,X_{m-1}) X_m^{l-1}) +\sum_{k=n+2-j}^n \alpha_{k}(X_1,\dots,X_{m-1}) h_k(\mathbb{X})  \in \mathcal{I}_{n,j}.
\]
This is a contradiction. So $P_{m,n,j}$ is true for all $n,j$ such that $1\leq j \leq m-1$, $m \leq n$.

\emph{Case B.} $j=m$. We induct on $n$. Note that $h_m(\mathbb{X}),h_{m-1}(\mathbb{X}),\dots,h_1(\mathbb{X})$ are independent over $\C$, and $\Sym(\mathbb{X})=\C[h_m(\mathbb{X}),h_{m-1}(\mathbb{X}),\dots,h_1(\mathbb{X})]$. When $n=m$, $h_{n+1-m}(\mathbb{X})=h_1(\mathbb{X})$ and $\Sym(\mathbb{X})/\mathcal{I}_{m,m} \cong \C[h_1(\mathbb{X})]$. So $P_{m,m,m}$ is true. Assume that $P_{m,n-1,m}$ is true for some $n>m$. Suppose that $g_n,\dots,g_{n+1-m}\in \Sym(\mathbb{X})$ satisfy 
\begin{equation}\label{eq-complete-non-zero-divisor-3}
\sum_{k=n+1-m}^n g_k h_k(\mathbb{X})=0.
\end{equation}
By equation \eqref{complete-recursion}, we have 
\[
h_n(\mathbb{X}) = \sum_{k=n-m}^{n-1}(-1)^{n-k+1}X_{n-k}h_k(\mathbb{X}).
\]
Plugging this into \eqref{eq-complete-non-zero-divisor-3}, we get 
\begin{equation}\label{eq-complete-non-zero-divisor-4}
(-1)^{m+1}X_m g_n h_{n-m}(\mathbb{X}) + \sum_{k=n+1-m}^{n-1} (g_k+(-1)^{n-k+1}X_{n-k}g_n) h_k(\mathbb{X}) = 0
\end{equation}
So $X_m g_n h_{n-m}(\mathbb{X}) \in \mathcal{I}_{n-1,m}$. Since $P_{m,n-1,m}$ and $Q_{m,n-1,m}$ are both true, this implies that $g_n \in \mathcal{I}_{n-1,m}$. Hence, there exist $\alpha_{n-1},\dots,\alpha_{n+1-m} \in \Sym(\mathbb{X})$ such that
\begin{equation}\label{eq-complete-non-zero-divisor-5}
g_n = \sum_{k=n+1-m}^{n-1}\alpha_k h_k(\mathbb{X}).
\end{equation}
Plugging this into \eqref{eq-complete-non-zero-divisor-4}, we get
\[
\sum_{k=n+1-m}^{n-1}(g_k+(-1)^{n-k+1}X_{n-k}g_n+(-1)^{m+1}\alpha_k X_m h_{n-m}(\mathbb{X})) h_k(\mathbb{X})=0.
\]
By $P_{m,n-1,m-1}$, this implies
\[
g_{n+1-m}+(-1)^mX_{m-1}g_n+(-1)^{m+1}\alpha_{n+1-m} X_m h_{n-m}(\mathbb{X}) \in \mathcal{I}_{n-1,m-1}.
\]
Comparing this with \eqref{eq-complete-non-zero-divisor-5}, we get
\[
g_{n+1-m}+ \alpha_{n+1-m}((-1)^mX_{m-1} h_{n+1-m}(\mathbb{X})+(-1)^{m+1} X_m h_{n-m}(\mathbb{X})) \in \mathcal{I}_{n-1,m-1}.
\]
Therefore, 
\[
g_{n+1-m}+ \alpha_{n+1-m} h_n(\mathbb{X}) = g_{n+1-m}+ \alpha_{n+1-m}\sum_{k=n-m}^{n-1}(-1)^{n-k+1}X_{n-k}h_k(\mathbb{X}) \in \mathcal{I}_{n-1,m-1}.
\]
Thus, $g_{n+1-m} \in \mathcal{I}_{n,m}$. This proves $P_{m,n,m}$. So $P_{m,n,m}$ is true for all $n\geq m$.

Combining \emph{Case A} and \emph{Case B}, we know that $P_{m,n,j}$ is true for all $n,j$ such that $1\leq j \leq m \leq n$. 

\textbf{(i)} and \textbf{(ii)} show that $P_{m,n,j}$ and $Q_{m,n,j}$ are true for all $n,j$ satisfying $1\leq j \leq m \leq n$. This completes the induction.
\end{proof}

\begin{proposition}\label{circle-module}
Let $\bigcirc_m$ be a circle colored by $m~(\leq N)$ and marked by a single alphabet $\mathbb{X}$ of $m$ indeterminates. Then, as $\zed_2 \oplus \zed$-graded $\Sym(\mathbb{X})$-modules, 
\[
H(\bigcirc_m) \cong \Sym(\mathbb{X})/(h_N(\mathbb{X}),h_{N-1}(\mathbb{X}),\dots,h_{N+1-m}(\mathbb{X})) \{q^{-m(N-m)}\} \left\langle m \right\rangle,
\] 
where $\Sym(\mathbb{X})/(h_N(\mathbb{X}),h_{N-1}(\mathbb{X}),\dots,h_{N+1-m}(\mathbb{X}))$ has $\zed_2$-grading $0$.

In particular, as graded modules over $\Sym(\mathbb{X})$, $H(\bigcirc_m) \cong H^\ast(G_{m,N})\{q^{-m(N-m)}\}$, where $G_{m,N}$ is the complex $(m,N)$-Grassmannian.
\end{proposition}
\begin{proof}
By definition,
\[
C(\bigcirc_m) =
\left(%
\begin{array}{cc}
  U_1 & 0 \\
  \dots & \dots \\
  U_m & 0 
\end{array}%
\right)_{\Sym(\mathbb{X})},
\]
where $U_j=\frac{\partial}{\partial X_j}p_{m,N+1}(X_1,\dots,X_m)$. By Lemma \ref{power-derive}, we know 
\[
U_j=(-1)^{j+1} (N+1) h_{m,N+1-j}(X_1,\dots,X_m).
\]
Then, by Proposition \ref{complete-non-zero-divisor}, $U_j$ is not a zero divisor in $\Sym(\mathbb{X})/(U_1,\dots,U_{j-1})$. Thus, we can apply Corollary \ref{a-contraction-weak} successively to the rows of $C(\bigcirc_m)$ from top to bottom and conclude that 
\[
H(\bigcirc_m) \cong \Sym(\mathbb{X})/(h_N(\mathbb{X}),h_{N-1}(\mathbb{X}),\dots,h_{N+1-m}(\mathbb{X})) \{q^{-m(N-m)}\} \left\langle m \right\rangle.
\] 
The last statement in the proposition follows from Theorem \ref{grassmannian}.
\end{proof}

\begin{definition}\label{def-circ-generator}
From the above proposition, we know that $H(\bigcirc_m)$ is generated, as a $\Sym(\mathbb{X})$-module, by the homology class corresponding to 
\[
1 \in \Sym(\mathbb{X})/(h_N(\mathbb{X}),h_{N-1}(\mathbb{X}),\dots,h_{N+1-m}(\mathbb{X})).
\]
We call this homology class the generating class and denote it by $\mathfrak{G}$. 
\end{definition}

\subsection{Cycles representing the generating class} To understand the action of a morphism of matrix factorizations on the homology of a colored circle, we need to understand its action on the generating class $\mathfrak{G}$. In order to do that, we sometimes need to represent $\mathfrak{G}$ by cycles in a matrix factorization associated to that circle. In particular, we will find such cycles in matrix factorizations associated to a colored circle with one or two marked points. To describe these cycles, we invoke the ``$1_\ve$" notation introduced in Definition \ref{ve-notation}.

\begin{lemma}\label{circle-rep-one-mark}
Let $\bigcirc_m$ be a circle colored by $m~(\leq N)$ and marked by a single alphabet $\mathbb{X}$ of $m$ indeterminates. (See Figure \ref{circle-module-figure}.) Write $U_j=\frac{\partial}{\partial X_j}p_{m,N+1}(X_1,\dots,X_m)$. Then, in 
\[
C(\bigcirc_m) =
\left(%
\begin{array}{cc}
  U_1 & 0 \\
  \dots & \dots \\
  U_m & 0 
\end{array}%
\right)_{\Sym(\mathbb{X})},
\]
the element $1_{(1,1,\dots,1)}$ is a cycle representing (a non-zero scalar multiple of) the generating class $\mathfrak{G} \in H(\bigcirc_m)$.
\end{lemma}
\begin{proof}
Write 
\[
M_j=
\left(%
\begin{array}{cc}
  U_j & 0 \\
  \dots & \dots \\
  U_m & 0 
\end{array}%
\right)_{\Sym(\mathbb{X})/(U_1,\dots,U_{j-1})}.
\]
Then, the homology of $\Gamma$ is computed by 
\begin{eqnarray*}
H(\bigcirc_m) & = & H(M_1) \\
 & \cong & H(M_2)\{q^{N+1-\deg U_1}\}\left\langle 1\right\rangle \\
 & \cong & \dots \\
 & \cong & H(M_m)\{q^{(m-1)(N+1)- \sum_{j=1}^{m-1}\deg U_j}\} \left\langle m-1\right\rangle \\
 & \cong & \Sym(\mathbb{X})/(h_N(\mathbb{X}),h_{N-1}(\mathbb{X}),\dots,h_{N+1-m}(\mathbb{X})) \{q^{-m(N-m)}\} \left\langle m \right\rangle.
\end{eqnarray*}
It is easy to see that $1_1 \in M_m$ represents $\mathfrak{G}$. Next, we use the method described in Remark \ref{reverse-b-contraction} to inductively construct a cycle in $C(\bigcirc_m)$ representing the generating class. Assume, for some $j$, $1_{(1,1,\dots,1)}\in M_j$ is a cycle representing $\mathfrak{G}$. Note that $1_{(1,1,\dots,1)}\in M_{j-1}$ is mapped to $1_{(1,1,\dots,1)}\in M_j$ by the quasi-isomorphism $M_{j-1}\rightarrow M_j\{q^{N+1-\deg U_{j-1}}\}\left\langle 1\right\rangle$.\footnote{Please see the proof of Proposition \ref{b-contraction-weak} for the definition of this quasi-isomorphism. Note that the setup there is slightly different. In the proof of Proposition \ref{b-contraction-weak}, $b_i$ is in the right column. But here, $U_{j-1}$ is in the left column.} But every entry in the right column of $M_{j-1}$ is $0$. So $d(1_{(1,1,\dots,1)})=0$, and therefore $1_{(1,1,\dots,1)}$ is a cycle representing $\mathfrak{G}$. This shows that $1_{(1,1,\dots,1)} \in M_1 = C(\bigcirc_m)$ is a cycle representing the generating class $\mathfrak{G} \in H(\bigcirc_m)$.
\end{proof}

\begin{figure}[ht]

\setlength{\unitlength}{1pt}

\begin{picture}(360,70)(-180,-10)


\qbezier(0,60)(-20,60)(-20,45)

\qbezier(0,60)(20,60)(20,45)

\put(-20,15){\vector(0,1){30}}

\put(20,45){\vector(0,-1){30}}

\qbezier(0,0)(-20,0)(-20,15)

\qbezier(0,0)(20,0)(20,15)

\put(19,30){\line(1,0){2}}

\put(-21,30){\line(1,0){2}}

\put(-3,-10){$\Gamma$}

\put(-15,30){\tiny{$m$}}

\put(25,25){\small{$\mathbb{X}$}}

\put(-30,25){\small{$\mathbb{Y}$}}
 
\end{picture}

\caption{}\label{circle-rep-two-marks-figure}

\end{figure}

\begin{lemma}\label{circle-rep-two-marks}
Let $\bigcirc_m$ be a circle colored by $m~(\leq N)$ and marked by two alphabets $\mathbb{X}$, $\mathbb{Y}$. (See Figure \ref{circle-rep-two-marks-figure}.) Use the definition 
\[
C(\bigcirc_m) =
\left(%
\begin{array}{cc}
  U_1 & X_1-Y_1 \\
  \dots & \dots \\
  U_m & X_m-Y_m \\
  U_1 & Y_1-X_1 \\
  \dots & \dots \\
  U_m & Y_m-X_m \\
\end{array}%
\right)_{\Sym(\mathbb{X}|\mathbb{Y})},
\]
where $X_j$ and $Y_j$ are the $j$-th elementary symmetric polynomials in $\mathbb{X}$ and in $\mathbb{Y}$, and $U_j\in \Sym(\mathbb{X}|\mathbb{Y})$ is homogeneous of degree $2(N+1-j)$ and satisfies 
\[
\sum_{j=1}^m (X_j-Y_j)U_j = p_{N+1}(\mathbb{X}) - p_{N+1}(\mathbb{Y}).
\]
Then the element 
\[
\sum_{\ve=(\ve_1,\dots\ve_m)\in I^m} (-1)^{\frac{|\ve|(|\ve|-1)}{2}+(m+1)|\ve|+\sum_{j=1}^{m-1}(m-j)\ve_j} 1_{\ve}\otimes 1_{\wbar{\ve}} \in C(\bigcirc_m)
\] 
is a cycle representing (a non-zero scalar multiple of) the generating class $\mathfrak{G} \in H(\bigcirc_m)$.
\end{lemma}
\begin{proof}
Although this lemma can be proved by the method used in the previous lemma, the computation is much more complex. So here we use a different approach by considering morphisms of matrix factorizations. From Proposition \ref{circle-dimension}, we have $H(\bigcirc_m) \cong C(\emptyset)\{\qb{N}{m}\}\left\langle m \right\rangle$. So the subspace of $H(\bigcirc_m)$ of elements of quantum degree $-m(N-m)$ is $1$-dimensional over $\C$ and is spanned by the generating class $\mathfrak{G}$. So, to prove the lemma, we only need to show that the above element of $C(\bigcirc_m)$ is a homogeneous cycle of quantum degree $-m(N-m)$ representing a non-zero homology class.

\begin{figure}[ht]

\setlength{\unitlength}{1pt}

\begin{picture}(360,40)(-180,20)

\qbezier(0,60)(-20,60)(-20,45)

\qbezier(0,60)(20,60)(20,45)

\put(-20,30){\vector(0,1){15}}

\put(20,45){\vector(0,-1){15}}

\put(-3,20){$\Gamma_1$}

\put(-3,55){\tiny{$m$}}

\put(25,25){\small{$\mathbb{X}$}}

\put(-30,25){\small{$\mathbb{Y}$}}
 
\end{picture}

\caption{}\label{circle-rep-two-marks-figure-2}

\end{figure}

Let $\Gamma_1$ be the oriented arc shown in Figure \ref{circle-rep-two-marks-figure-2}. Then, by Lemmas \ref{bullet}, \ref{row-reverse-signs}, and \ref{column-reverse-signs}, 
\[
\Hom_{\Sym(\mathbb{X}|\mathbb{Y})}(C(\Gamma_1),C(\Gamma_1)) \cong C(\Gamma_1) \otimes_{\Sym(\mathbb{X}|\mathbb{Y})} C(\Gamma_1)_\bullet \cong C(\bigcirc_m)\{q^{m(N-m)}\}\left\langle m \right\rangle.
\]
Consider the identity map $\id:C(\Gamma_1)\rightarrow C(\Gamma_1)$. It is a morphism of matrix factorizations and, therefore, a cycle in $\Hom_{\Sym(\mathbb{X}|\mathbb{Y})}(C(\Gamma_1),C(\Gamma_1))$. Assume $\id$ is homotopic to $0$, that is, there exists $h \in \Hom_{\Sym(\mathbb{X}|\mathbb{Y})}(C(\Gamma_1),C(\Gamma_1))$ of $\zed_2$-degree $1$ such that $\id=d\circ h +h\circ d$. Then, for any cycle $f \in \Hom_{\Sym(\mathbb{X}|\mathbb{Y})}(C(\Gamma_1),C(\Gamma_1))$ of $\zed_2$-degree $i$, we have 
\[
f = f \circ \id = f \circ (d\circ h +h\circ d) = (-1)^i (d \circ (f\circ h) - (-1)^{i+1} (f\circ h) \circ d),
\]
which is a boundary element in $\Hom_{\Sym(\mathbb{X}|\mathbb{Y})}(C(\Gamma_1),C(\Gamma_1))$. This implies that the homology of $\Hom_{\Sym(\mathbb{X}|\mathbb{Y})}(C(\Gamma_1),C(\Gamma_1))$ is $0$, which is a contradiction since $H(\bigcirc_m)\neq 0$. Thus, $\id$ is a cycle representing a non-zero homology class. Under the above isomorphism, $\id$ is mapped to a homogeneous cycle in $C(\bigcirc_m)$ of quantum degree $-m(N-m)$ representing a non-zero homology class. Thus, the image of $\id$ is a cycle representing a non-zero scalar multiple of the generating class $\mathfrak{G}$.

Next, we check that the image of $\id$ is in fact the cycle given in the lemma. Under the homogeneous isomorphism 
\[
\Hom_{\Sym(\mathbb{X}|\mathbb{Y})}(C(\Gamma_1),C(\Gamma_1)) \xrightarrow{\cong} C(\Gamma_1) \otimes_{\Sym(\mathbb{X}|\mathbb{Y})} C(\Gamma_1)_\bullet
\] 
preserving the $\zed_2\oplus\zed$-grading, we have
\[
\id \mapsto \sum_{\ve\in I^m} 1_{\ve}\otimes 1_{\ve}^\ast \in C(\Gamma_1) \otimes_{\Sym(\mathbb{X}|\mathbb{Y})} C(\Gamma_1)_\bullet.
\]
By Lemma \ref{bullet}, under the homogeneous isomorphism 
\[
C(\Gamma_1) \otimes_{\Sym(\mathbb{X}|\mathbb{Y})} C(\Gamma_1)_\bullet \xrightarrow{\cong} M_1 :=
\left(%
\begin{array}{cc}
  U_1 & X_1-Y_1 \\
  \dots & \dots \\
  U_m & X_m-Y_m \\
  Y_m-X_m & U_m \\
  \dots & \dots \\
  Y_1-X_1 & U_1\\
\end{array}%
\right)_{\Sym(\mathbb{X}|\mathbb{Y})}
\]
preserving the $\zed_2\oplus\zed$-grading, we have
\[
\sum_{\ve\in I^m} 1_{\ve}\otimes 1_{\ve}^\ast \in C(\Gamma_1) \mapsto \sum_{\ve=(\ve_1,\dots\ve_m)\in I^m}  1_{\ve}\otimes 1_{(\ve_m,\dots\ve_1)} \in M_1.
\]
By Lemma \ref{row-reverse-signs}, under the homogeneous isomorphism
\[
M_1 \xrightarrow{\cong} M_2 := \left(%
\begin{array}{cc}
  U_1 & X_1-Y_1 \\
  \dots & \dots \\
  U_m & X_m-Y_m \\
  Y_1-X_1 & U_1 \\
  \dots & \dots \\
  Y_m-X_m & U_m\\
\end{array}%
\right)_{\Sym(\mathbb{X}|\mathbb{Y})}
\]
preserving the $\zed_2\oplus\zed$-grading, we have
\[
\sum_{\ve=(\ve_1,\dots\ve_m)\in I^m}  1_{\ve}\otimes 1_{(\ve_m,\dots\ve_1)} \mapsto \sum_{\ve\in I^m} (-1)^{\frac{|\ve|(|\ve|-1)}{2}} 1_{\ve}\otimes 1_{\ve} \in M_2.
\]
And, by Lemmas \ref{morphism-sign} and  \ref{column-reverse-signs}, under the homogeneous isomorphism 
\[
M_2 \rightarrow  
C(\bigcirc_m) =
\left(%
\begin{array}{cc}
  U_1 & X_1-Y_1 \\
  \dots & \dots \\
  U_m & X_m-Y_m \\
  U_1 & Y_1-X_1 \\
  \dots & \dots \\
  U_m & Y_m-X_m \\
\end{array}%
\right)_{\Sym(\mathbb{X}|\mathbb{Y})}
\]
of $\zed_2$-degree $m$ and quantum degree $-m(N-m)$, we have
\[
\sum_{\ve\in I^m} (-1)^{\frac{|\ve|(|\ve|-1)}{2}} 1_{\ve}\otimes 1_{\ve} \mapsto \sum_{\ve=(\ve_1,\dots\ve_m)\in I^m} (-1)^{\frac{|\ve|(|\ve|-1)}{2}+(m+1)|\ve|+\sum_{j=1}^{m-1}(m-j)\ve_j} 1_{\ve}\otimes 1_{\wbar{\ve}} \in C(\bigcirc_m).
\]
Thus,
\[
\sum_{\ve=(\ve_1,\dots\ve_m)\in I^m} (-1)^{\frac{|\ve|(|\ve|-1)}{2}+(m+1)|\ve|+\sum_{j=1}^{m-1}(m-j)\ve_j} 1_{\ve}\otimes 1_{\wbar{\ve}} \in C(\bigcirc_m)
\]
is the image of $\id \in \Hom_{\Sym(\mathbb{X}|\mathbb{Y})}(C(\Gamma_1),C(\Gamma_1))$ under the isomorphism 
\[
\Hom_{\Sym(\mathbb{X}|\mathbb{Y})}(C(\Gamma_1),C(\Gamma_1)) \xrightarrow{\cong} C(\bigcirc_m).
\]
\end{proof}

\section{Morphisms Induced by Local Changes of MOY Graphs}\label{sec-morph}

In this section, we establish several morphisms of matrix factorizations induced by certain local changes of MOY graphs, some of which have implicitly appeared in Sections \ref{mf-MOY} and \ref{sec-circles}. These morphisms are building blocks of more complex morphisms in Direct Sum Decompositions (III-V) and in chain complexes of colored link diagrams. 

\subsection{Terminology} Most morphisms defined in the rest of this paper are defined only up to homotopy and scaling by a non-zero scalar. To simplify our exposition, we introduce the following notations.

\begin{definition}
Suppose that $V$ is a linear space over $\C$ and $u,v\in V$. We write $u \propto v$ if $\exists ~c\in \C\setminus \{0\}$ such that $u = c \cdot v$.

Suppose that $W$ is a chain complex over a $\C$-algebra and $u,v$ are cycles in $W$, we write $u \approx v$ if $\exists ~c\in \C\setminus \{0\}$ such that $u$ is homologous to $c \cdot v$. In particular, if $M,M'$ are matrix factorizations of the same potential over a graded commutative unital $\C$-algebra and $f,g:M\rightarrow M'$ are morphisms of matrix factorizations, we write $f \approx g$ if $\exists ~c\in \C\setminus \{0\}$ such that $f \simeq c \cdot g$.
\end{definition}

Let $\Gamma_1,\Gamma_2$ be two MOY graphs with a one-to-one correspondence $F$ between their end points such that
\begin{itemize}
	\item every exit corresponds to an exit, and every entrance corresponds to an entrance,
	\item edges adjacent to corresponding end points have the same color.
\end{itemize}
Mark $\Gamma_1,\Gamma_2$ so that every pair of corresponding end points are assigned the same alphabet. Assume $\mathbb{X}_1,\mathbb{X}_2,\dots,\mathbb{X}_n$ are the alphabets assigned to the end points of $\Gamma_1$ and $\Gamma_2$.

\begin{definition}
\[
\Hom_F(C(\Gamma_1),C(\Gamma_2)) := \Hom_{\Sym(\mathbb{X}_1|\mathbb{X}_2|\cdots|\mathbb{X}_n)}(C(\Gamma_1),C(\Gamma_2)),
\]
which is a $\zed_2$-graded chain complex, where the $\zed_2$-grading is induced by the $\zed_2$-gradings of $C(\Gamma_1), ~C(\Gamma_2)$. The quantum gradings of $C(\Gamma_1), ~C(\Gamma_2)$ induce a quantum pregrading on $\Hom_F(C(\Gamma_1),C(\Gamma_2))$.

Denote by $\Hom_{\HMF,F}(C(\Gamma_1),C(\Gamma_2))$ the homology of $\Hom_F(C(\Gamma_1),C(\Gamma_2))$, that is, the $\Sym(\mathbb{X}_1|\mathbb{X}_2|\cdots|\mathbb{X}_n)$-module of homotopy classes of morphisms from $C(\Gamma_1)$ to $C(\Gamma_2)$. It inherits the $\zed_2$-grading from $\Hom_F(C(\Gamma_1),C(\Gamma_2))$. And the quantum pregrading of $\Hom_F(C(\Gamma_1),C(\Gamma_2))$ induces a quantum grading on $\Hom_{\HMF,F}(C(\Gamma_1),C(\Gamma_2))$. (See Lemmas \ref{hom-all-gradings-homo-finite} and \ref{MOY-object-of-hmf}.)

We drop $F$ from the above notations if it is clear from the context.
\end{definition}

\begin{lemma}\label{hom-hmf-compute}
$\Hom_{\HMF,F}(C(\Gamma_1),C(\Gamma_2))$ does not depend on the choice of markings.
\end{lemma}
\begin{proof}
This lemma follows easily from Proposition \ref{b-contraction} and Corollary \ref{b-contraction-dual}.
\end{proof}

\subsection{Bouquet move} First we recall the homotopy equivalence induced by the bouquet moves in Figure \ref{bouquet-move-figure}. From Corollary \ref{contract-expand}, we know bouquet moves induce homotopy equivalence. In this subsection, we show that, up to homotopy and scaling, a bouquet move induces a unique homotopy equivalence.

\begin{figure}[ht]

\setlength{\unitlength}{1pt}

\begin{picture}(360,100)(-180,-50)


\put(-100,25){$\Gamma_1$:}

\put(-60,10){\vector(0,1){10}}

\put(-60,20){\vector(-1,1){20}}

\put(-60,20){\vector(1,1){10}}

\put(-50,30){\vector(-1,1){10}}

\put(-50,30){\vector(1,1){10}}

\put(-75,3){\tiny{$i+j+k$}}

\put(-55,21){\tiny{$j+k$}}

\put(-80,42){\tiny{$i$}}

\put(-60,42){\tiny{$j$}}

\put(-40,42){\tiny{$k$}}


\put(-15,25){$\longleftrightarrow$}


\put(20,25){$\Gamma'_1$:}

\put(60,10){\vector(0,1){10}}

\put(60,20){\vector(1,1){20}}

\put(60,20){\vector(-1,1){10}}

\put(50,30){\vector(1,1){10}}

\put(50,30){\vector(-1,1){10}}

\put(45,3){\tiny{$i+j+k$}}

\put(38,21){\tiny{$i+j$}}

\put(80,42){\tiny{$k$}}

\put(60,42){\tiny{$j$}}

\put(40,42){\tiny{$i$}}


\put(-100,-25){$\Gamma_2$:}

\put(-60,-30){\vector(0,-1){10}}

\put(-80,-10){\vector(1,-1){20}}

\put(-50,-20){\vector(-1,-1){10}}

\put(-60,-10){\vector(1,-1){10}}

\put(-40,-10){\vector(-1,-1){10}}

\put(-75,-47){\tiny{$i+j+k$}}

\put(-55,-29){\tiny{$j+k$}}

\put(-80,-8){\tiny{$i$}}

\put(-60,-8){\tiny{$j$}}

\put(-40,-8){\tiny{$k$}}


\put(-15,-25){$\longleftrightarrow$}


\put(20,-25){$\Gamma'_2$:}

\put(60,-30){\vector(0,-1){10}}

\put(80,-10){\vector(-1,-1){20}}

\put(50,-20){\vector(1,-1){10}}

\put(60,-10){\vector(-1,-1){10}}

\put(40,-10){\vector(1,-1){10}}

\put(45,-47){\tiny{$i+j+k$}}

\put(38,-29){\tiny{$i+j$}}

\put(80,-8){\tiny{$k$}}

\put(60,-8){\tiny{$j$}}

\put(40,-8){\tiny{$i$}}

\end{picture}

\caption{}\label{bouquet-move-figure}

\end{figure}

\begin{lemma}\label{bouquet-move-lemma}
Suppose that $\Gamma_1$, $\Gamma'_1$, $\Gamma_2$ and $\Gamma'_2$ are the MOY graphs shown in Figure \ref{bouquet-move-figure}. Then, as $\zed_2\oplus\zed$-graded vector spaces over $\C$, 
\begin{eqnarray*}
& & \Hom_\HMF (C(\Gamma_1), C(\Gamma'_1)) \cong \Hom_\HMF (C(\Gamma_2),C(\Gamma'_2)) \\
& \cong &  C(\emptyset)\{\qb{N}{i+j+k}\qb{i+j+k}{k}\qb{i+j}{j}q^{(i+j+k)(N-i-j-k)+ij+jk+ki}\}.
\end{eqnarray*}
In particular, the subspaces of the above spaces of homogeneous elements of quantum degree $0$ are $1$-dimensional.
\end{lemma}

\begin{figure}[ht]

\setlength{\unitlength}{1pt}

\begin{picture}(360,75)(-180,-15)

\qbezier(-40,50)(-40,60)(0,60)

\qbezier(40,50)(40,60)(0,60)

\qbezier(-40,10)(-40,0)(0,0)

\qbezier(40,10)(40,0)(0,0)

\put(40,50){\vector(0,-1){40}}

\qbezier(-40,10)(0,15)(0,30)

\qbezier(-40,50)(0,45)(0,30)

\put(0,30){\vector(0,1){0}}

\qbezier(-40,10)(-60,10)(-60,20)

\qbezier(-40,50)(-60,50)(-60,40)

\put(-60,20){\vector(0,1){0}}

\qbezier(-60,20)(-80,20)(-80,30)

\qbezier(-60,20)(-40,20)(-40,30)

\qbezier(-60,40)(-80,40)(-80,30)

\qbezier(-60,40)(-40,40)(-40,30)

\put(-40,50){\vector(1,0){0}}

\put(-80,30){\vector(0,1){0}}

\put(-40,30){\vector(0,1){0}}

\put(45,30){\tiny{$i+j+k$}}

\put(-75,10){\tiny{$i+j$}}

\put(-75,45){\tiny{$i+j$}}

\put(5,30){\tiny{$k$}}

\put(-35,30){\tiny{$j$}}

\put(-90,30){\tiny{$i$}}

\put(-5,-15){$\Gamma$}

\end{picture}

\caption{}\label{bouquet-move-figure-2}

\end{figure}

\begin{proof}
We only compute $\Hom_\HMF (C(\Gamma_1), C(\Gamma'_1))$. The computation of $\Hom_\HMF (C(\Gamma_2),C(\Gamma'_2))$ is similar. By Corollary \ref{contract-expand} and Lemmas \ref{bullet}, \ref{row-reverse-signs}, \ref{column-reverse-signs}, one can see that 
\begin{eqnarray*}
\Hom_\HMF (C(\Gamma_1), C(\Gamma'_1)) & \cong & \Hom_\HMF (C(\Gamma'_1), C(\Gamma'_1)) \\
& \cong & H(\Gamma)\left\langle i+j+k \right\rangle\{q^{(i+j+k)(N-i-j-k)+ij+jk+ki}\},
\end{eqnarray*}
where $\Gamma$ is the MOY graph in Figure \ref{bouquet-move-figure-2}. Using Decomposition (II) (Theorems \ref{decomp-II}) and Corollary \ref{circle-dimension}, we have that
\[
H(\Gamma) \cong C(\emptyset)\left\langle i+j+k \right\rangle\{\qb{N}{i+j+k}\qb{i+j+k}{k}\qb{i+j}{j}\}.
\]
The lemma follows from these isomorphisms.
\end{proof}

\begin{remark}\label{bouquet-move-remark}
From Corollary \ref{contract-expand} and Lemma \ref{bouquet-move-lemma}, one can see that, up to homotopy and scaling, a bouquet move induces a unique homotopy equivalence. In the rest of this paper, we usually denote such a homotopy equivalence by $h$.
\end{remark}

\subsection{Circle creation and annihilation}

\begin{lemma}\label{circle-empty-hmf}
Let $\bigcirc_m$ be a circle colored by $m$. Then, as $\zed_2\oplus\zed$-graded vector spaces over $\C$,
\[
\Hom_{HMF}(C(\bigcirc_m),C(\emptyset)) \cong \Hom_{HMF}(C(\emptyset),C(\bigcirc_m)) \cong C(\emptyset)\{\qb{N}{m}\}\left\langle m \right\rangle,
\]
where $C(\emptyset)$ is the matrix factorization $\C\rightarrow 0 \rightarrow \C$.

In particular, the subspaces of $\Hom_{HMF}(C(\emptyset),C(\bigcirc_m))$ and $\Hom_{HMF}(C(\bigcirc_m),C(\emptyset))$ of elements of quantum degree $-m(N-m)$ are $1$-dimensional. 
\end{lemma}
\begin{proof}
The natural isomorphism $\Hom_\C(C(\emptyset),C(\bigcirc_m))\cong C(\bigcirc_m)$ is an isomorphism of matrix factorizations preserving the $\zed_2\oplus\zed$-grading. So, by Corollary \ref{circle-dimension}, 
\[
\Hom_{HMF}(C(\emptyset),C(\bigcirc_m)) \cong H(\bigcirc_m) \cong C(\emptyset)\{\qb{N}{m}\}\left\langle m \right\rangle.
\]
Using Corollary \ref{circle-dimension} again, we have
\[
\Hom_{HMF}(C(\bigcirc_m),C(\emptyset)) \cong \Hom_\C(C(\emptyset)\{\qb{N}{m}\}\left\langle m \right\rangle, C(\emptyset)) \cong C(\emptyset)\{\qb{N}{m}\}\left\langle m \right\rangle.
\]
\end{proof}

Lemma \ref{circle-empty-hmf} leads to the following definitions, which generalize the corresponding definitions in \cite{KR1}.

\begin{definition}
Let $\bigcirc_m$ be a circle colored by $m$. Associate to the circle creation a homogeneous morphism 
\[
\iota: C(\emptyset)(\cong \C) \rightarrow C(\bigcirc_m)
\] 
of quantum degree $-m(N-m)$ not homotopic to $0$.

Associate to the circle annihilation a homogeneous morphism 
\[
\epsilon:C(\bigcirc_m) \rightarrow C(\emptyset) (\cong \C)
\] 
of quantum degree $-m(N-m)$ not homotopic to $0$.

By Lemma \ref{circle-empty-hmf}, $\iota$ and $\epsilon$ are unique up to homotopy and scaling. Both of them have $\zed_2$-degree $m$.
\end{definition}

Using the natural isomorphism $\Hom_\C(C(\emptyset),C(\bigcirc_m))\cong C(\bigcirc_m)$, one can see that 
\begin{equation}\label{iota-def-computation}
\iota(1) \approx \mathfrak{G},
\end{equation}
where $\mathfrak{G}$ is the generating class of $H(\bigcirc_m)$.

Mark $\bigcirc_m$ by a single alphabet $\mathbb{X}$. By Lemma \ref{circle-rep-one-mark}, the element $1_{(1,1,\dots,1)}$ of $C(\bigcirc_m)$ is a cycle representing (a non-zero scalar multiple of) the generating class $\mathfrak{G} \in H(\bigcirc_m)$. From the proof of Proposition \ref{circle-module}, we know that there is a $\Sym(\mathbb{X})$-linear quasi-isomorphism
\[
P: C(\bigcirc_m) \rightarrow \Sym(\mathbb{X})/(h_N(\mathbb{X}),h_{N-1}(\mathbb{X}),\dots,h_{N+1-m}(\mathbb{X})) \{q^{-m(N-m)}\} \left\langle m \right\rangle
\]
satisfying $P(1_{(1,1,\dots,1)})=1$. By Corollary \ref{a-contraction-dual} and Remark \ref{reverse-b-contraction-dual}, $P$ induces a quasi-isomorphism 
\begin{eqnarray*}
 && \Hom_\C(\Sym(\mathbb{X})/(h_N(\mathbb{X}),h_{N-1}(\mathbb{X}),\dots,h_{N+1-m}(\mathbb{X})) \{q^{-m(N-m)}\} \left\langle m \right\rangle,\C) \\ 
& \xrightarrow{P^\sharp} & \Hom_\C(C(\bigcirc_m),C(\emptyset)).
\end{eqnarray*}
Recall that, by Theorem \ref{grassmannian}, there is a $\C$-linear trace map 
\[
\Tr:\Sym(\mathbb{X})/(h_{N+1-m}(\mathbb{X}),h_{N+2-m}(\mathbb{X}),\dots,h_{N}(\mathbb{X})) \rightarrow \C
\] 
satisfying
\[
\Tr(S_\lambda(\mathbb{X}) \cdot S_\mu(\mathbb{X})) = \left\{%
\begin{array}{ll}
    1 & \text{if } \lambda_j + \mu_{m+1-j} =N-m ~\forall j=1,\dots,m, \\
    0 & \text{otherwise,}  \\
\end{array}%
\right.
\]
where $\lambda,\mu\in \Lambda_{m,N-m}=\{(\lambda_1\geq\cdots\geq\lambda_m)~|~\lambda_1\leq N-m\}$ and $S_{\lambda}(\mathbb{X})$ is the Schur polynomial in $\mathbb{X}$ associated to the partition $\lambda$. Note that $P^\sharp(\Tr) = \Tr \circ P:C(\bigcirc_m) \rightarrow C(\emptyset)$ is homogeneous of $\zed_2$-grading $m$ and quantum grading $-m(N-m)$, and 
\begin{eqnarray}
\label{Tr-circ-P-property} && P^\sharp(\Tr)(S_\lambda(\mathbb{X}) \cdot S_\mu(\mathbb{X})\cdot 1_{(1,1,\dots,1)}) \\
\nonumber & = & \Tr(S_\lambda(\mathbb{X}) \cdot S_\mu(\mathbb{X})\cdot P(1_{(1,1,\dots,1)})) = \Tr(S_\lambda(\mathbb{X}) \cdot S_\mu(\mathbb{X})\cdot 1) \\
\nonumber & =&
\begin{cases}    
1 & \text{if } \lambda_j + \mu_{m+1-j} =N-m ~\forall j=1,\dots,m, \\
0 & \text{otherwise.} 
\end{cases}    
\end{eqnarray}
This implies that $P^\sharp(\Tr)$ induces a non-zero homomorphism on the homology. So $P^\sharp(\Tr)$ is homotopically non-trivial. Therefore,
\begin{equation}\label{epsilon-def-computation}
\epsilon \approx P^\sharp(\Tr) = \Tr \circ P.
\end{equation}

\begin{corollary}\label{iota-epsilon-composition}
Denote by $\mathfrak{m}(S_{\lambda}(\mathbb{X}))$ the morphism $C(\bigcirc_m)\rightarrow C(\bigcirc_m)$ induced by the multiplication by $S_{\lambda}(\mathbb{X})$. Then, for any $\lambda, \mu \in \Lambda_{m,N-m}$,
\[
\epsilon \circ \mathfrak{m}(S_{\lambda}(\mathbb{X})) \circ \mathfrak{m}(S_{\mu}(\mathbb{X})) \circ \iota \approx
\begin{cases}
\id_{C(\emptyset)} & \text{if } \lambda_j + \mu_{m+1-j} =N-m ~\forall j=1,\dots,m, \\
0 & \text{otherwise.} 
\end{cases}
\]
\end{corollary}
\begin{proof}
This corollary follows easily from \eqref{iota-def-computation}, \eqref{Tr-circ-P-property} and \eqref{epsilon-def-computation}.
\end{proof}

\subsection{Edge splitting and merging}\label{subsec-edge-split-morph} Let $\Gamma_0$ and $\Gamma_1$ be the MOY graphs in Figure \ref{edge-splitting}. We call the change $\Gamma_0\leadsto\Gamma_1$ an edge splitting and the change $\Gamma_1\leadsto\Gamma_0$ an edge merging. In this subsection, we define the morphisms $\phi$ and $\overline{\phi}$ associated to edge splitting and merging.

\begin{figure}[ht]

\setlength{\unitlength}{1pt}

\begin{picture}(360,75)(-180,-90)


\put(-67,-45){\tiny{$m+n$}}

\put(-70,-75){\vector(0,1){50}}

\put(-71,-50){\line(1,0){2}}

\put(-67,-30){\small{$\mathbb{X}$}}

\put(-95,-53){\small{$\mathbb{A}\cup\mathbb{B}$}}

\put(-67,-75){\small{$\mathbb{Y}$}}

\put(-75,-90){$\Gamma_0$}


\put(-25,-50){\vector(1,0){50}}

\put(25,-60){\vector(-1,0){50}}

\put(-5,-47){\small{$\phi$}}

\put(-5,-70){\small{$\overline{\phi}$}}


\put(70,-75){\vector(0,1){10}}

\put(70,-35){\vector(0,1){10}}

\qbezier(70,-65)(60,-65)(60,-55)

\qbezier(70,-35)(60,-35)(60,-45)

\put(60,-55){\vector(0,1){10}}

\put(59,-55){\line(1,0){2}}

\qbezier(70,-65)(80,-65)(80,-55)

\qbezier(70,-35)(80,-35)(80,-45)

\put(80,-55){\vector(0,1){10}}

\put(79,-55){\line(1,0){2}}

\put(73,-30){\tiny{$m+n$}}

\put(73,-70){\tiny{$m+n$}}

\put(83,-50){\tiny{$n$}}

\put(51,-50){\tiny{$m$}}

\put(60,-30){\small{$\mathbb{X}$}}

\put(60,-75){\small{$\mathbb{Y}$}}

\put(50,-58){\small{$\mathbb{A}$}}

\put(83,-58){\small{$\mathbb{B}$}}

\put(65,-90){$\Gamma_1$}

\end{picture}

\caption{}\label{edge-splitting}

\end{figure}

\begin{lemma}\label{edge-splitting-lemma}
Let $\Gamma_0$ and $\Gamma_1$ be the MOY graphs in Figure \ref{edge-splitting}. Then, as $\zed_2\oplus\zed$-graded vector spaces over $\C$, 
\[
\Hom_{HMF}(C(\Gamma_0),C(\Gamma_1)) \cong \Hom_{HMF}(C(\Gamma_1),C(\Gamma_0)) \cong C(\emptyset) \{q^{(N-m-n)(m+n)}\qb{N}{m+n} \qb{m+n}{m}\}.
\]
In particular, the lowest quantum gradings of the above spaces are $-mn$, and the subspaces of these spaces of homogeneous elements of quantum grading $-mn$ are $1$-dimensional.
\end{lemma}
\begin{proof}
By Theorem \ref{decomp-II}, $C(\Gamma_1) \simeq C(\Gamma_0)\{\qb{m+n}{m}\}$. So 
\[
\Hom(C(\Gamma_0),C(\Gamma_1)) \simeq \Hom(C(\Gamma_0),C(\Gamma_0 ))\{\qb{m+n}{m}\} \simeq \Hom(C(\Gamma_1),C(\Gamma_0)).
\]
Denote by $\bigcirc_{m+n}$ the circle colored by $m+n$. Then, from the proof of Lemma \ref{circle-rep-two-marks}, we have 
\begin{eqnarray*}
\Hom(C(\Gamma_0),C(\Gamma_0 )) & \cong & C(\bigcirc_{m+n})\{q^{(N-m-n)(m+n)}\} \left\langle m+n \right\rangle \\
& \simeq & C(\emptyset) \{q^{(N-m-n)(m+n)}\qb{N}{m+n}\}.
\end{eqnarray*}
This proves the lemma.
\end{proof}

\begin{definition}\label{morphism-edge-splitting-merging-def}
Let $\Gamma_0$ and $\Gamma_1$ be the MOY graphs in Figure \ref{edge-splitting}. Associate to the edge splitting a homogeneous morphism 
\[
\phi: C(\Gamma_0) \rightarrow C(\Gamma_1)
\] 
of quantum degree $-mn$ not homotopic to $0$.

Associate to the edge merging a homogeneous morphism 
\[
\overline{\phi}:C(\Gamma_1) \rightarrow C(\Gamma_0)
\] 
of quantum degree $-mn$ not homotopic to $0$.

By Lemma \ref{edge-splitting-lemma}, the morphisms $\phi$ and $\overline{\phi}$ are well defined up to scaling and homotopy, and both of them have $\zed_2$-grading $0$. 
\end{definition}

It is not hard to find explicit forms of these morphisms. In fact, $\phi$ is the composition
\[
C(\Gamma_0) \xrightarrow{\id} C(\Gamma_0)\{q^{-mn}\} \hookrightarrow C(\Gamma_0) \{\qb{m+n}{m}\} \xrightarrow{\simeq} C(\Gamma_1),
\]
and $\overline{\phi}$ is the composition  
\[
C(\Gamma_1) \xrightarrow{\simeq} C(\Gamma_0) \{\qb{m+n}{m}\} \twoheadrightarrow C(\Gamma_0)\{q^{mn}\} \xrightarrow{\id} C(\Gamma_0),
\]
where $\hookrightarrow$ and $\twoheadrightarrow$ are the natural inclusion and projection maps. 

More precisely, from the proof of Theorem \ref{decomp-II}, we know that
\[
C(\Gamma_1) \simeq C(\Gamma_0) \otimes_{\Sym(\mathbb{A}\cup\mathbb{B})} (\Sym(\mathbb{A}|\mathbb{B}))\{q^{-mn}\}.
\]
The natural inclusion map $\Sym(\mathbb{A}\cup\mathbb{B}) \hookrightarrow \Sym(\mathbb{A}|\mathbb{B})$, which is $\Sym(\mathbb{A}\cup\mathbb{B})$-linear and has grading $0$, induces a homogeneous morphism 
\[
C(\Gamma_0) \xrightarrow{\phi'} C(\Gamma_1) ~(\simeq C(\Gamma_0) \otimes_{\Sym(\mathbb{A}\cup\mathbb{B})} (\Sym(\mathbb{A}|\mathbb{B}))\{q^{-mn}\}) 
\]
of $\zed_2$-degree $0$ and quantum degree $-mn$ given by $\phi'(r)= r \otimes 1$. 

From Theorem \ref{part-symm-str}, there is a unique $\Sym(\mathbb{A}\cup\mathbb{B})$-linear homogeneous projection $\zeta:\Sym(\mathbb{A}|\mathbb{B}) \rightarrow \Sym(\mathbb{A}\cup\mathbb{B})$ of degree $-2mn$, called the Sylvester operator, satisfying, for $\lambda,\mu\in \Lambda_{m,n}=\{(\lambda_1\geq\cdots\geq\lambda_m) ~|~\lambda_1\leq n\}$,
\[
\zeta(S_{\lambda}(\mathbb{A})\cdot S_{\mu}(-\mathbb{B})) = \left\{%
\begin{array}{ll}
    1 & \text{if } \lambda_j + \mu_{m+1-j} = n ~\forall j=1,\dots,m, \\ 
    0 & \text{otherwise,}
\end{array}%
\right. 
\]
The Sylvester operator $\zeta$ induces a homogeneous morphism 
\[
(C(\Gamma_0) \otimes_{\Sym(\mathbb{A}\cup\mathbb{B})} (\Sym(\mathbb{A}|\mathbb{B}))\{q^{-mn}\}\simeq)~ C(\Gamma_1) \xrightarrow{\overline{\phi'}} C(\Gamma_0)
\]
of $\zed_2$-degree $0$ and quantum degree $-mn$ given by 
\[
\overline{\phi'} (r \otimes (S_{\lambda}(\mathbb{A})\cdot S_{\mu}(-\mathbb{B}))) = \left\{%
\begin{array}{ll}
    r & \text{if } \lambda_j + \mu_{m+1-j} = n ~\forall j=1,\dots,m, \\ 
    0 & \text{otherwise,}
\end{array}%
\right. 
\]
where $\lambda,\mu\in \Lambda_{m,n}$.

Clearly, $\overline{\phi'}(S_{\lambda_{m,n}}(\mathbb{A}) \cdot \phi'(r)) =r$ $\forall~ r \in C(\Gamma_0)$. So $\phi'$ and $\overline{\phi'}$ are not homotopic to $0$. Thus, $\phi \approx \phi'$ and $\overline{\phi} \approx \overline{\phi'}$. In particular, we have the following lemma.

\begin{lemma}\label{phibar-compose-phi}
Let $\Gamma_0$ and $\Gamma_1$ be the MOY graphs in Figure \ref{edge-splitting}. Then
\[
\overline{\phi} \circ \mathfrak{m}(S_{\lambda}(\mathbb{A})\cdot S_{\mu}(-\mathbb{B})) \circ \phi \approx \left\{%
\begin{array}{ll}
    \id_{C(\Gamma_0)} & \text{if } \lambda_j + \mu_{m+1-j} = n ~\forall j=1,\dots,m, \\ 
    0 & \text{otherwise,}
\end{array}%
\right. 
\]
where $\lambda,\mu\in \Lambda_{m,n}$ and $\mathfrak{m}(S_{\lambda}(\mathbb{A})\cdot S_{\mu}(-\mathbb{B}))$ is the morphism induced by the multiplication of $S_{\lambda}(\mathbb{A})\cdot S_{\mu}(-\mathbb{B})$.
\end{lemma}

\subsection{Adjoint Koszul matrix factorizations} Let $\Gamma_0$ and $\Gamma_1$ be the MOY graphs in Figure \ref{KR-chi-maps-fig}. Khovanov and Rozansky \cite{KR1} defined morphisms $C(\Gamma_0) \xrightarrow{\chi^0} C(\Gamma_1)$ and $C(\Gamma_1) \xrightarrow{\chi^1} C(\Gamma_0)$, which play an important role in the construction of their link homology. We generalize these $\chi$-morphisms in two subsections. First we construct morphisms between adjoint Koszul matrix factorizations in this subsection. Then, in next subsection, we apply this construction to matrix factorizations of MOY graphs to define the general $\chi$-morphisms.

\begin{figure}[ht]

\setlength{\unitlength}{1pt}

\begin{picture}(360,75)(-180,-15)


\put(-65,0){\tiny{$1$}}

\put(-18,0){\tiny{$1$}}

\put(-65,55){\tiny{$1$}}

\put(-18,55){\tiny{$1$}}

\put(-38,30){\tiny{$2$}}

\put(-60,0){\vector(1,1){20}}

\put(-20,0){\vector(-1,1){20}}

\put(-40,40){\vector(-1,1){20}}

\put(-40,40){\vector(1,1){20}}

\put(-40,20){\vector(0,1){20}}

\put(-43,-15){$\Gamma_1$}


\put(-15,40){\vector(1,0){30}}

\put(15,20){\vector(-1,0){30}}

\put(-5,44){\small{$\chi^1$}}

\put(-5,10){\small{$\chi^0$}}


\put(25,0){\vector(0,1){60}}

\put(55,0){\vector(0,1){60}}

\put(62,0){\tiny{$1$}}

\put(20,0){\tiny{$1$}}

\put(37,-15){$\Gamma_0$}

\end{picture}

\caption{}\label{KR-chi-maps-fig}

\end{figure}

\begin{definition}\label{def-adjoint-mf}
Let $R$ be a graded commutative unital $\C$-algebra. Suppose that, for $i,j=1,\dots,n$, $a_j$, $b_i$ and $t_{ij}$ are homogeneous elements of $R$ satisfying $\deg a_j +\deg b_i + \deg T_{ij}=2N+2$. Let
\[
A= \left(%
\begin{array}{c}
  a_1 \\
  a_2 \\
  \dots \\
  a_n
\end{array}%
\right),
\hspace{.5cm}
B = \left(%
\begin{array}{c}
  b_1 \\
  b_2 \\
  \dots \\
  b_n
\end{array}%
\right),
\hspace{.5cm}
T = \left(%
\begin{array}{cccc}
  T_{11} & T_{12} & \dots & T_{1n} \\
  T_{21} & T_{22} & \dots & T_{2n} \\
  \dots & \dots & \dots & \dots \\
  T_{n1} & T_{n2} & \dots & T_{nn}
\end{array}%
\right).
\]
Then $M:=(A,T^tB)_R$ and $M':=(TA,B)_R$ are both graded Koszul matrix factorizations over $R$ with potential $w = \sum_{i,j=1}^n a_j b_i T_{ij}$. Here, $T^t$ is the transposition of $T$. We call $M$ and $M'$ adjoint Koszul matrix factorizations and $T$ the relation matrix. 
\end{definition}

Our objective is to construct a pair of morphisms between $M$ and $M'$ with certain properties. The following is the main result of this subsection.

\begin{proposition}\label{general-jumping-factor}
Let $M$ and $M'$ be as in Definition \ref{def-adjoint-mf}. Then there exist morphisms $F:M\rightarrow M'$ and $G: M' \rightarrow M$ satisfying:
\begin{enumerate}[(i)]
	\item $\deg_{\zed_2} F = \deg_{\zed_2} G =0$, $\deg F =0$ and 
	\[
	\deg G = \deg \det(T)=2n(N+1)-\sum_{k=1}^n (\deg a_k + \deg b_k),
	\]
	\item $G \circ F = \det(T) \cdot \id_M$ and $F\circ G = \det(T) \cdot \id_{M'}$.
\end{enumerate}
\end{proposition}

As a special case of Proposition \ref{general-jumping-factor}, we have the following corollary, which was established in \cite[Subsection 2.1]{KR2}.

\begin{corollary}\cite{KR2}\label{jumping-factor}
Let $a,b,t$ be homogeneous elements of $R$ with $\deg a + \deg b + \deg t =2N+2$. Then there exist homogeneous morphisms 
\begin{eqnarray*}
f: (a,tb)_R \rightarrow (ta,b)_R, && \\
g: (ta,b)_R \rightarrow (a,tb)_R, &&
\end{eqnarray*}
such that 
\begin{enumerate}[(i)]
  \item $\deg_{\zed_2} f = \deg_{\zed_2} g =0$, $\deg f =0$ and $\deg g = \deg t$.
	\item $g \circ f = t \cdot \id_{(a,tb)_R}$ and $f \circ g = t \cdot \id_{(ta,b)_R}$.
\end{enumerate}
\end{corollary}

Before proving Proposition \ref{general-jumping-factor}, we recall an alternative construction of Koszul matrix factorizations given in \cite[Section 2]{KR1}. 

Let $R^n= \underbrace{R\oplus \cdots \oplus R}_{n-\text{fold}}$, and $e_i=(0,\dots,0,\underbrace{1}_{i-\text{th}},0,\dots,0)^t$. The $\{e_1,\dots,e_n\}$ is an $R$-basis for $R^n$. Define $T: R^n \rightarrow R^n$ by $T(e_j)= \sum_{i=1}^n T_{ij}e_i$. Let $(R^n)^\ast$ be the dual of $R^n$ over $R$, $\{e_1^\ast,\dots,e_n^\ast\}$ the basis of $(R^n)^\ast$ dual to $\{e_1,\dots,e_n\}$, and $T^\ast: (R^n)^\ast \rightarrow (R^n)^\ast$ the dual map of $T$. Then $T^\ast (e_i^\ast)= \sum_{j=1}^n T_{ij}e_j^\ast$.

Set 
\begin{eqnarray*}
\alpha & = & \sum_{i=1}^n a_i e_i =(e_1,\dots,e_n)A \in R^n, \\
\beta & = & \sum_{i=1}^n b_i e_i^\ast = (e_1^\ast,\dots,e_n^\ast) B \in (R^n)^\ast.
\end{eqnarray*}
Then $T\alpha = (e_1,\dots,e_n)TA$ and $T^\ast\beta = (e_1^\ast,\dots,e_n^\ast) T^tB$.

From \cite[Section 2]{KR1}, we know that $M=(A,T^tB)_R$ is the matrix factorization 
\[
\bigwedge_{even} R^n \xrightarrow{\wedge \alpha + \neg T^\ast\beta} \bigwedge_{odd} R^n \xrightarrow{\wedge \alpha + \neg T^\ast\beta} \bigwedge_{even} R^n,
\]
in which, for any $i_1<\cdots<i_k$, $e_{i_1}\wedge\cdots\wedge e_{i_k}$ is homogeneous with $\zed_2$-grading $k$ and quantum grading $k(N+1)-\sum_{l=1}^k \deg a_{i_l}$. 

Similarly, $M'=(TA,B)_R$ is the matrix factorization 
\[
\bigwedge_{even} R^n \xrightarrow{\wedge T\alpha + \neg \beta} \bigwedge_{odd} R^n \xrightarrow{\wedge T\alpha + \neg \beta} \bigwedge_{even} R^n,
\]
in which, for any $i_1<\cdots<i_k$, $e_{i_1}\wedge\cdots\wedge e_{i_k}$ is homogeneous with $\zed_2$-grading $k$ and quantum grading $-k(N+1)+\sum_{l=1}^k \deg b_{i_l}$.

Note that $T$ induces an $R$-algebra endomorphism $T: \bigwedge R^n \rightarrow \bigwedge R^n$ by 
\[
T(e_{i_1}\wedge\cdots\wedge e_{i_k}):= Te_{i_1}\wedge\cdots\wedge Te_{i_k}.
\]

Define $R$-module map $\mathcal{D}:R^n \oplus (R^n)^\ast \rightarrow R^n \oplus (R^n)^\ast$ by $\mathcal{D}(e_i)=e_i^\ast$ and $\mathcal{D}(e_i^\ast)=e_i$. Then $\mathcal{D}^2=\id$. We define $T^t:R^n\rightarrow R^n$ by $T^t=\mathcal{D} \circ T^\ast \circ\mathcal{D}$. Then the matrix of $T^t$ under the basis $\{e_1,\dots,e_n\}$ is the transposition of $T$. $T^t$ induces an $R$-algebra endomorphism $T^t: \bigwedge R^n \rightarrow \bigwedge R^n$ by 
\[
T^t(e_{i_1}\wedge\cdots\wedge e_{i_k}):= T^t e_{i_1}\wedge\cdots\wedge T^t e_{i_k}.
\]

Next we introduce the Hodge $\star$-operator. $\star:\bigwedge R^n \rightarrow \bigwedge R^n$ is an $R$-module map defined so that, for any $i_1<\cdots<i_k$, $\star(e_{i_1}\wedge\cdots\wedge e_{i_k})=e_{j_1}\wedge\cdots\wedge e_{j_{n-k}}$, where $(e_{i_1},\dots, e_{i_k},e_{j_1},\dots,e_{j_{n-k}})$ is an even permutation of $(e_1,\dots,e_n)$. 

To simplify the exposition, we use the following notations in the rest of this subsection. 
\begin{itemize}
	\item $\mathcal{I}_k:=\{I=(i_1,\dots,i_k)|1\leq i_1<\cdots<i_k \leq n\}$.
	\item For any $I=(i_1,\dots,i_k) \in \mathcal{I}_k$, $\bar{I}$ is the unique element $\bar{I}=(j_1,\dots,j_{n-k})\in \mathcal{I}_{n-k}$ such that $\{i_1,\dots,i_k,j_1,\dots,j_{n-k}\}=\{1,\dots,n\}$.
	\item $(I,\bar{I})$ is the parity of the permutation $(i_1,\dots,i_k,j_1,\dots,j_{n-k})$ of $(1,\dots,n)$.
	\item $e_I := e_{i_1}\wedge\cdots\wedge e_{i_k}$. Note that $\star e_I = (-1)^{(I,\bar{I})}e_{\bar{I}}$.
	\item For $I=(i_1,\dots,i_k),~L=(l_1,\dots,l_k) \in \mathcal{I}_k$, we denote by $T_{LI}$ the matrix 
\[
T_{LI} = \left(%
\begin{array}{cccc}
  T_{l_1 i_1} & T_{l_1 i_2} & \dots & T_{l_1 i_k} \\
  T_{l_2 i_1} & T_{l_2 i_2} & \dots & T_{l_2 i_k} \\
  \dots & \dots & \dots & \dots \\
  T_{l_k i_1} & T_{l_k i_2} & \dots & T_{l_k i_k}
\end{array}%
\right).
\]
\end{itemize}

\begin{lemma}\label{T-star-composition}
For any $I=(i_1,\dots,i_k) \in \mathcal{I}_k$,
\[
\star T^t \star T (e_I) = T \star T^t\star (e_I) = (-1)^{k(n-k)}\det(T) \cdot e_I.
\]
\end{lemma}
\begin{proof}
We first prove
\begin{equation}\label{T-star-composition-1}
\star T^t \star T (e_I) = (-1)^{k(n-k)}\det(T) \cdot e_I.
\end{equation}
Note that
\begin{eqnarray*}
T(e_I) & = & Te_{i_1}\wedge\cdots\wedge Te_{i_k} \\
& = & (\sum_{j_1=1}^n T_{j_1 i_1}e_{j_1})\wedge \cdots \wedge (\sum_{j_k=1}^n T_{j_k i_k}e_{j_k}) \\
& = & \sum_{J \in \mathcal{I}_k} \det (T_{JI}) \cdot e_J, \\
& & \\
\star e_J & = & (-1)^{(J,\bar{J})} e_{\bar{J}}, \\
& & \\
T^t (e_{\bar{J}}) & = & \sum_{L\in \mathcal{I}_{k}} \det (T^t_{\bar{L}\bar{J}}) \cdot e_{\bar{L}} = \sum_{L\in \mathcal{I}_{k}} \det (T_{\bar{J}\bar{L}}) \cdot e_{\bar{L}}, \\
& & \\
\star e_{\bar{L}} & = & (-1)^{(\bar{L},L)} e_L.
\end{eqnarray*}
Also, if we write $J=(j_1,\dots,j_k)$ and $L=(l_1,\dots,l_k)$, then
\begin{eqnarray*}
(J,\bar{J}) & = & \sum_{m=1}^k (j_m-m)= \sum_{m=1}^k j_m -\frac{k(k+1)}{2}, \\
(\bar{L},L) & = & \sum_{m=1}^k (n-k+m-l_m) = k(n-k) + \frac{k(k+1)}{2} - \sum_{m=1}^k l_m.
\end{eqnarray*}
Using the above equations and the Laplace Formula, we get
\begin{eqnarray*}
& & \star T^t \star T (e_I) \\
& = & (-1)^{k(n-k)} \sum_{L\in \mathcal{I}_{k}} \sum_{J \in \mathcal{I}_k} (-1)^{\sum_{m=1}^k j_m- \sum_{m=1}^k l_m} \det (T_{\bar{J}\bar{L}}) \cdot \det (T_{JI}) \cdot e_L \\
& = & (-1)^{k(n-k)} \det (T) \cdot e_I.
\end{eqnarray*}
Thus, \eqref{T-star-composition-1} is true. In particular, if $T=\id$, then $T^t=\id$ and \eqref{T-star-composition-1} implies that
\begin{equation}\label{T-star-composition-2}
\star \star (e_I) = (-1)^{k(n-k)} \cdot e_I.
\end{equation}

Replacing $T$ by $T^t$ in \eqref{T-star-composition-1}, we get
\begin{equation}\label{T-star-composition-3}
\star T \star T^t (e_I) = (-1)^{k(n-k)}\det(T^t) \cdot e_I = (-1)^{k(n-k)}\det(T) \cdot e_I.
\end{equation}
Note that \eqref{T-star-composition-1}, \eqref{T-star-composition-2} and \eqref{T-star-composition-3} are true for all $k$ and all $I\in\mathcal{I}_k$. So we have that
\begin{eqnarray*}
T \star T^t\star (e_I) & = & (-1)^{k(n-k)} \star\star T \star T^t\star (e_I) \\
& = & (-1)^{k(n-k)} \star(\star T \star T^t(\star e_I)) \\
& = & (-1)^{k(n-k)} \cdot (-1)^{k(n-k)} \cdot \det(T)\cdot \star\star e_I \\
& = & (-1)^{k(n-k)} \det(T)\cdot e_I.
\end{eqnarray*}
\end{proof}

\begin{lemma}\label{T-alpha-beta-commute}
\begin{eqnarray}
\label{alpha-T-commute} T\circ (\wedge \alpha) & = & (\wedge T\alpha) \circ T, \\
\label{beta-T-commute} T \circ (\neg T^\ast \beta) & = & (\neg \beta) \circ T.
\end{eqnarray}
\end{lemma}
\begin{proof}
For any $I=(i_1,\dots,i_k) \in \mathcal{I}_k$, 
\begin{eqnarray*}
T\circ (\wedge \alpha) (e_{i_1}\wedge\cdots\wedge e_{i_k}) & = & T (e_{i_1}\wedge\cdots\wedge e_{i_k}\wedge \alpha) \\
& = & T (e_{i_1}\wedge\cdots\wedge e_{i_k})\wedge T\alpha \\
& = & (\wedge T\alpha) \circ T (e_{i_1}\wedge\cdots\wedge e_{i_k}).
\end{eqnarray*}
So \eqref{alpha-T-commute} is true. 

Similarly,
\begin{eqnarray*}
&  & T \circ (\neg T^\ast \beta) (e_{i_1}\wedge\cdots\wedge e_{i_k}) \\
& = & T(\sum_{m=1}^k (-1)^{m-1} \beta(Te_{i_m}) \cdot e_{i_1}\wedge\cdots \widehat{e_{i_m}}\cdots\wedge e_{i_k}) \\
& = & \sum_{m=1}^k (-1)^{m-1} \beta(Te_{i_m}) \cdot T(e_{i_1})\wedge\cdots \widehat{T(e_{i_m})}\cdots\wedge T(e_{i_k}) \\
& = & (\neg \beta) (T(e_{i_1})\wedge\cdots\wedge T(e_{i_k})) \\
& = & (\neg \beta) \circ T (e_{i_1}\wedge\cdots\wedge e_{i_k}).
\end{eqnarray*}
So \eqref{beta-T-commute} is true.
\end{proof}

\begin{lemma}\label{star-alpha-beta-commute}
\begin{eqnarray}
\label{alpha-star-commute} \star\circ (\wedge \alpha) & = & (\neg \mathcal{D}\alpha) \circ \star, \\
\label{beta-star-commute} \star\circ (\neg \beta) & = & (-1)^{n-1} (\wedge \mathcal{D}\beta) \circ \star.
\end{eqnarray}
\end{lemma}
\begin{proof}
For any $I=(i_1,\dots,i_k) \in \mathcal{I}_k$, let $\bar{I} =( j_1,\dots,j_{n-k})$. Then
\begin{eqnarray*}
& &   \star\circ (\wedge \alpha) (e_{i_1}\wedge\cdots\wedge e_{i_k}) \\
& = & \star (e_{i_1}\wedge\cdots\wedge e_{i_k}\wedge \alpha) \\
& = & \star(\sum_{m=1}^{n-k} a_{j_m} \cdot e_{i_1}\wedge\cdots\wedge e_{i_k}\wedge e_{j_m}) \\
& = & \sum_{m=1}^{n-k} a_{j_m} \cdot \star(e_{i_1}\wedge\cdots\wedge e_{i_k}\wedge e_{j_m}) \\
& = & \sum_{m=1}^{n-k} a_{j_m} \cdot (-1)^{(I,\bar{I})+m-1} \cdot e_{j_1}\wedge\cdots \widehat{e_{j_m}}\cdots\wedge e_{j_{n-k}} \\
& = & (\neg \mathcal{D}\alpha) \circ \star (e_{i_1}\wedge\cdots\wedge e_{i_k}).
\end{eqnarray*}
So \eqref{alpha-star-commute} is true.

Similarly,
\begin{eqnarray*}
& &   \star\circ (\neg \beta) (e_{i_1}\wedge\cdots\wedge e_{i_k}) \\
& = & \star (\sum_{m=1}^k (-1)^{m-1} b_{i_m} \cdot e_{i_1}\wedge\cdots \widehat{e_{i_m}}\cdots\wedge e_{i_k}) \\
& = & \sum_{m=1}^k (-1)^{m-1} b_{i_m} \cdot \star (e_{i_1}\wedge\cdots \widehat{e_{i_m}}\cdots\wedge e_{i_k}) \\
& = & \sum_{m=1}^k (-1)^{m-1} b_{i_m} \cdot (-1)^{(I,\bar{I})+n-m} \cdot e_{\bar{I}} \wedge e_{i_m} \\
& = & (-1)^{n-1} \star (e_I) \wedge \mathcal{D}\beta \\
& = & (-1)^{n-1} (\wedge \mathcal{D}\beta) \circ \star (e_{i_1}\wedge\cdots\wedge e_{i_k}).
\end{eqnarray*}
So \eqref{beta-star-commute} is true.
\end{proof}

\begin{lemma}\label{D-alpha-T-commute}
\begin{eqnarray}
\label{T-D-commute} \mathcal{D}\circ T^t \circ \mathcal{D} & = & T^\ast, \\
\label{alpha-T-D-commute} \mathcal{D}\circ T (\alpha) & = & (T^t)^\ast \circ \mathcal{D}(\alpha).
\end{eqnarray}
\end{lemma}
\begin{proof}
Recall that $T^t$ is defined by $T^t=\mathcal{D}\circ T^\ast \circ \mathcal{D}$ and that $\mathcal{D}^2=\id$. \eqref{T-D-commute} follows immediately. Replace $T$ by $T^t$ in \eqref{T-D-commute}, we get $\mathcal{D}\circ T \circ \mathcal{D} = (T^t)^\ast$. Plugging $\mathcal{D}(\alpha)$ into this equation, we get \eqref{alpha-T-D-commute}.
\end{proof}

\begin{lemma}\label{star-T-alpha-beta-commute}
\begin{eqnarray}
\label{alpha-T-star-commute} (\star T^t \star) \circ (\wedge T\alpha) & = & (-1)^{n-1} (\wedge\alpha) \circ (\star T^t \star), \\
\label{beta-T-star-commute} (\star T^t \star) \circ (\neg \beta) & = & (-1)^{n-1} (\neg T^\ast \beta) \circ (\star T^t \star).
\end{eqnarray}
\end{lemma}
\begin{proof}
Note that Lemmas \ref{T-alpha-beta-commute} through \ref{D-alpha-T-commute} are true for any $\alpha \in R^n$, $\beta \in (R^n)^\ast$ and $T \in \Hom_R(R^n,R^n)$. So
\[
\left.%
\begin{array}{rclcl}
  (\star T^t \star) \circ (\wedge T\alpha) & = & (\star T^t)  \circ (\neg \mathcal{D} T\alpha) \circ \star & \hspace{1cm} & \text{(by \eqref{alpha-star-commute})} \\
 & = & (\star T^t)  \circ (\neg (T^t)^\ast\mathcal{D} \alpha) \circ \star & \hspace{1cm} & \text{(by \eqref{alpha-T-D-commute})} \\
 & = & \star  \circ (\neg \mathcal{D} \alpha) \circ  (T^t\star) & \hspace{1cm} & \text{(by \eqref{beta-T-commute})} \\
 & = & (-1)^{n-1}(\wedge \mathcal{D}^2 \alpha) \circ  (\star T^t\star) & \hspace{1cm} & \text{(by \eqref{beta-star-commute})} \\
 & = & (-1)^{n-1}(\wedge \alpha) \circ  (\star T^t\star) & \hspace{1cm} & \text{since } \mathcal{D}^2=\id. \\
\end{array}%
\right.
\]
This proves \eqref{alpha-T-star-commute}.

Similarly, we have
\[
\left.%
\begin{array}{rclcl}
  (\star T^t \star) \circ (\neg \beta) & = & (-1)^{n-1} (\star T^t)  \circ (\wedge \mathcal{D} \beta) \circ \star & \hspace{1cm} & \text{(by \eqref{beta-star-commute})} \\
 & = & (-1)^{n-1} \star   \circ (\wedge T^t\mathcal{D} \beta) \circ (T^t\star) & \hspace{1cm} & \text{(by \eqref{alpha-T-commute})} \\
 & = & (-1)^{n-1} (\neg \mathcal{D} T^t\mathcal{D} \beta) \circ (\star T^t\star) & \hspace{1cm} & \text{(by \eqref{alpha-star-commute})} \\
 & = & (-1)^{n-1} (\neg T^\ast \beta) \circ (\star T^t\star) & \hspace{1cm} & \text{(by \eqref{T-D-commute})} \\
\end{array}%
\right.
\]
This proves \eqref{beta-T-star-commute}.
\end{proof}

Now we are ready to prove Proposition \ref{general-jumping-factor}.

\begin{proof}[Proof of Proposition \ref{general-jumping-factor}]
Define $F:M\rightarrow M'$ by $F=T:\bigwedge R^n \rightarrow \bigwedge R^n$. Also, define $G:M' \rightarrow M$ by $G(e_I) = (-1)^{k(n-k)} \star F^t \star (e_I)$ $\forall~I=(i_1,\dots,i_k) \in \mathcal{I}_k$. Then Lemmas \ref{T-alpha-beta-commute} and \ref{star-T-alpha-beta-commute} imply that $F$ and $G$ are morphisms of matrix factorizations. Lemma \ref{T-star-composition} implies that $G \circ F = \det(T) \cdot \id_M$ and $F\circ G = \det(T) \cdot \id_{M'}$. It is easy to see that $\deg_{\zed_2} F = \deg_{\zed_2} G =0$. It remains to show that $F$ and $G$ are homogeneous with the correct quantum gradings.

For $I=(i_1,\dots,i_k) \in \mathcal{I}_k$, let  
\begin{eqnarray*}
S(I) & = & \sum_{m=1}^k i_m, \\
S_a(I) & = & \sum_{m=1}^k \deg a_{i_m}, \\
S_b(I) & = & \sum_{m=1}^k \deg b_{i_m}.
\end{eqnarray*}
Recall that, $e_I$ is a homogeneous element of both $M$ and $M'$. As an element of $M$, the quantum grading of $e_I$ is $\deg_M e_I = k(N+1)-S_a(I)$. And, as an element of $M'$, its quantum grading is $\deg_{M'}e_I = S_b(I)-k(N+1)$. It is easy to check that, for $I,J \in \mathcal{I}_k$, $\deg T_{JI}$ is homogeneous with $\deg T_{JI} = 2k(N+1) - S_a(I)-S_b(J)$. So 
\begin{eqnarray*}
\deg_{M'} \det(T_{JI}) e_J & = & 2k(N+1) - S_a(I)-S_b(J) + S_b(J)-k(N+1) \\
& = & k(N+1) - S_a(I) = \deg_M e_I.
\end{eqnarray*}
But 
\[
F(e_I) = T(e_I) = \sum_{J\in \mathcal{I}_k} \det(T_{JI}) e_J. 
\]
This shows that $F$ is homogeneous with quantum degree $0$. 

Similarly, 
\begin{eqnarray*}
G(e_I) & = & (-1)^{k(n-k)} \star T^t \star (e_I) = \sum_{J\in \mathcal{I}_k} (-1)^{S(I)+S(J)} \det(T^t_{\bar{J}\bar{I}}) e_J \\
& = & \sum_{J\in \mathcal{I}_k} (-1)^{S(I)+S(J)} \det(T_{\bar{I}\bar{J}}) e_J.
\end{eqnarray*}
Note that each term $\det(T_{\bar{I}\bar{J}}) e_J$ is homogeneous in $M$ with quantum degree
\begin{eqnarray*}
&   & \deg_M \det(T_{\bar{I}\bar{J}}) e_J \\
& = & 2(n-k)(N+1) - S_a(\bar{J})-S_b(\bar{I}) + k(N+1)-S_a(J) \\
& = & (2n-k)(N+1) - S_a(1,\dots,n) - (S_b(1,\dots,n) - S_b(I)) \\
& = & (2n(N+1) - S_a(1,\dots,n) - S_b(1,\dots,n)) + (S_b(I) -k(N+1)) \\
& = & \deg \det(T) + \deg_{M'} e_I.
\end{eqnarray*}
This shows that $G$ is homogeneous with quantum degree $\deg \det(T)$.
\end{proof}

\begin{remark}
First, note that Lemma \ref{row-reverse-signs} and Corollary \ref{row-op} are both special cases of Proposition \ref{general-jumping-factor}. Second, recall that Rasmussen \cite{Ras2} explained that the $\zed_2$-grading of a Koszul matrix factorization can be lifted to a $\zed$-grading. $F$ and $G$ in Proposition \ref{general-jumping-factor} preserve this $\zed$-grading.
\end{remark}

\subsection{General $\chi$-morphisms} 

\begin{figure}[ht]

\setlength{\unitlength}{1pt}

\begin{picture}(360,75)(-180,-15)


\put(-120,0){\vector(1,1){20}}

\put(-100,20){\vector(1,-1){20}}

\put(-100,40){\vector(0,-1){20}}

\put(-100,40){\vector(1,1){20}}

\put(-120,60){\vector(1,-1){20}}

\put(-101,30){\line(1,0){2}}

\put(-132,45){\tiny{$_{m+n-l}$}}

\put(-115,15){\tiny{$_l$}}

\put(-90,45){\tiny{$_m$}}

\put(-90,15){\tiny{$_n$}}

\put(-95,28){\tiny{$_{n-l}$}}

\put(-130,55){\small{$\mathbb{A}$}}

\put(-75,0){\small{$\mathbb{Y}$}}

\put(-113,27){\small{$\mathbb{D}$}}

\put(-130,0){\small{$\mathbb{B}$}}

\put(-75,55){\small{$\mathbb{X}$}}

\put(-102,-15){$\Gamma_0$}


\put(-30,35){\vector(1,0){60}}

\put(30,25){\vector(-1,0){60}}

\put(-3,40){\small{$\chi^0$}}

\put(-3,15){\small{$\chi^1$}}


\put(60,10){\vector(1,1){20}}

\put(60,50){\vector(1,-1){20}}

\put(80,30){\vector(1,0){20}}

\put(100,30){\vector(1,1){20}}

\put(100,30){\vector(1,-1){20}}

\put(68,45){\tiny{$_{m+n-l}$}}

\put(70,15){\tiny{$_{l}$}}

\put(108,45){\tiny{$_m$}}

\put(106,15){\tiny{$_n$}}

\put(81,32){\tiny{$_{m+n}$}}

\put(50,45){\small{$\mathbb{A}$}}

\put(122,10){\small{$\mathbb{Y}$}}

\put(50,10){\small{$\mathbb{B}$}}

\put(122,45){\small{$\mathbb{X}$}}

\put(88,-15){$\Gamma_1$}

\end{picture}

\caption{}\label{general-general-chi-maps-figure}

\end{figure}

The following proposition is the main result of this subsection.

\begin{proposition}\label{general-general-chi-maps}
Let $\Gamma_0$ and $\Gamma_1$ be the MOY graphs in Figure \ref{general-general-chi-maps-figure}, where $1\leq l\leq n <m+n \leq N$. Then there exist homogeneous morphisms $\chi^0:C(\Gamma_0) \rightarrow C(\Gamma_1)$ and $\chi^1:C(\Gamma_1) \rightarrow C(\Gamma_0)$ satisfying
\begin{enumerate}[(i)]
	\item both $\chi^0$ and $\chi^1$ have $\zed_2$-degree $0$ and quantum degree $ml$.
	\item \begin{eqnarray*}
	\chi^1 \circ \chi^0 & \simeq & (\sum_{\lambda\in\Lambda_{l,m}} (-1)^{|\lambda|} S_{\lambda'}(\mathbb{X}) S_{\lambda^c}(\mathbb{B})) \cdot \id_{C(\Gamma_0)}, \\
	\chi^0 \circ \chi^1 & \simeq & (\sum_{\lambda\in\Lambda_{l,m}} (-1)^{|\lambda|} S_{\lambda'}(\mathbb{X}) S_{\lambda^c}(\mathbb{B})) \cdot \id_{C(\Gamma_1)}, \\
	\end{eqnarray*}
	where $\Lambda_{l,m}= \{\mu=(\mu_1\geq\cdots\geq\mu_l) ~|~ \mu_1 \leq m\}$, $\lambda'\in \Lambda_{m,l}$ is the conjugate of $\lambda$, and $\lambda^c$ is the complement of $\lambda$ in $\Lambda_{l,m}$, that is, if $\lambda=(\lambda_1\geq\cdots\geq\lambda_l)\in \Lambda_{l,m}$, then $\lambda^c = (m-\lambda_l\geq\cdots\geq m-\lambda_1)$.
\end{enumerate}
\end{proposition}

Before proving Proposition \ref{general-general-chi-maps}, we first simplify $C(\Gamma_0)$ and $C(\Gamma_1)$ to show that they are homotopic to a pair of adjoint Koszul matrix factorizations. 

Let $R=\Sym(\mathbb{X}|\mathbb{Y}|\mathbb{A}|\mathbb{B})$. Denote by $X_i$ the $i$-th elementary symmetric polynomial in $\mathbb{X}$ and so on. Recall that 
\[
C(\Gamma_0) = \left(%
\begin{array}{cc}
  \ast & X_1+D_1-A_1 \\
  \dots & \dots \\
  \ast & \sum_{i=0}^{n-l} X_{k-i}D_{i}-A_k \\
  \dots & \dots \\
  \ast & X_mD_{n-l} - A_{m+n-l} \\
  \ast & Y_1-D_1-B_1 \\
  \dots & \dots \\
  \ast & Y_k-\sum_{i=0}^{n-l} B_{k-i}D_{i} \\
  \dots & \dots \\
  \ast & Y_n-B_lD_{n-l} \\
\end{array}%
\right)_{\Sym(\mathbb{X}|\mathbb{Y}|\mathbb{A}|\mathbb{B}|\mathbb{D})}\{q^{-m(n-l)}\}.
\]
We exclude $\mathbb{D}$ from the base ring by applying Proposition \ref{b-contraction} to the rows
\[
\left(%
\begin{array}{cc}
  \ast & Y_1-D_1-B_1 \\
  \dots & \dots \\
  \ast & Y_{n-l}-\sum_{i=0}^{n-l} B_{n-l-i}D_i \\
\end{array}%
\right).
\]
This gives us 
\begin{equation}\label{gen-chi-Gamma-0-1}
C(\Gamma_0) \simeq \left(%
\begin{array}{cc}
  \ast & X_1+D_1-A_1 \\
  \dots & \dots \\
  \ast & \sum_{i=0}^{n-l} X_{k-i}D_{i}-A_k \\
  \dots & \dots \\
  \ast & X_mD_{n-l} - A_{m+n-l} \\
  \ast & Y_{n-l+1}-\sum_{i=0}^{n-l} B_{n-l+1-i}D_i \\
  \dots & \dots \\
  \ast & Y_{n-l+k}-\sum_{i=0}^{n-l} B_{n-l+k-i}D_i \\
  \dots & \dots \\
  \ast & Y_n-B_lD_{n-l} \\
\end{array}%
\right)_{R}\{q^{-m(n-l)}\},
\end{equation}
where 
\begin{equation}\label{eq-D-k-sum}
D_k= 
\begin{cases}
\sum_{i=0}^k (-1)^i h_i(\mathbb{B})Y_{k-i} & \text{if } k=0,1,\dots n-l, \\
0 & \text{otherwise}. \\
\end{cases}
\end{equation}
Since the above sum will appear repeatedly in this subsection, we set
\begin{equation}\label{eq-def-T-k}
T_k= 
\begin{cases}
\sum_{i=0}^k (-1)^i h_i(\mathbb{B})Y_{k-i} & \text{if } k\geq 0, \\
0 & \text{if } k<0. \\
\end{cases}
\end{equation}

Now consider $Y_{n-l+k}-\sum_{i=0}^{n-l} B_{n-l+k-i}D_i$. For $k=1$, using equation \eqref{complete-recursion},  we get
\begin{eqnarray*}
Y_{n-l+1}-\sum_{i=0}^{n-l} B_{n-l+1-i}D_i
& = & Y_{n-l+1}-\sum_{i=0}^{n-l} B_{n-l+1-i} \sum_{j=0}^i (-1)^{i-j} h_{i-j}(\mathbb{B}) Y_j \\
& = & Y_{n-l+1}-\sum_{j=0}^{n-l}Y_j \sum_{i=j}^{n-l}(-1)^{i-j} h_{i-j}(\mathbb{B})B_{n-l+1-i} \\
& = & Y_{n-l+1}-\sum_{j=0}^{n-l}Y_j \sum_{i=0}^{n-l-j}(-1)^{i} h_{i}(\mathbb{B})B_{n-l+1-j-i} \\
& = & Y_{n-l+1}+\sum_{j=0}^{n-l} (-1)^{n-l+1-j} Y_j h_{n-l+1-j}(\mathbb{B}) \\
& = & T_{n-l+1}.
\end{eqnarray*}
If $k>1$, then
\begin{eqnarray*}
& & Y_{n-l+k}-\sum_{i=0}^{n-l} B_{n-l+k-i}D_i \\
& = & Y_{n-l+k}-\sum_{i=0}^{n-l} B_{n-l+k-i}\sum_{j=0}^i (-1)^{i-j} h_{i-j}(\mathbb{B}) Y_j \\
& = & Y_{n-l+k}-\sum_{j=0}^{n-l}Y_j \sum_{i=j}^{n-l}(-1)^{i-j} h_{i-j}(\mathbb{B})B_{n-l+k-i} \\
& = & Y_{n-l+k}-\sum_{j=0}^{n-l}Y_j \sum_{i=0}^{n-l-j}(-1)^{i} h_{i}(\mathbb{B})B_{n-l+k-j-i} \\
& = & Y_{n-l+k}+\sum_{j=0}^{n-l}Y_j \sum_{i=n-l-j+1}^{n-l-j+k}(-1)^{i} h_{i}(\mathbb{B})B_{n-l+k-j-i} \\
& = & Y_{n-l+k}+\sum_{j=0}^{n-l}Y_j \sum_{i=0}^{k-1}(-1)^{n-l+k-j-i} h_{n-l+k-j-i}(\mathbb{B})B_{i} \\
& = & Y_{n-l+k}+\sum_{i=0}^{k-1}B_{i}\sum_{j=0}^{n-l}Y_j (-1)^{n-l+k-j-i} h_{n-l+k-j-i}(\mathbb{B}) \\
& = & Y_{n-l+k}+\sum_{i=0}^{k-1}B_{i}(T_{n-l+k-i}-\sum_{j=n-l+1}^{n-l+k-i}Y_j (-1)^{n-l+k-j-i} h_{n-l+k-j-i}(\mathbb{B})) \\
& = & Y_{n-l+k}+\sum_{i=0}^{k-1}B_{i}T_{n-l+k-i}-\sum_{i=0}^{k-1}B_{i}\sum_{j=n-l+1}^{n-l+k-i}Y_j (-1)^{n-l+k-j-i} h_{n-l+k-j-i}(\mathbb{B}). \\
\end{eqnarray*}
But, by equation \eqref{complete-recursion},
\begin{eqnarray*}
& & \sum_{i=0}^{k-1}B_{i}\sum_{j=n-l+1}^{n-l+k-i}Y_j (-1)^{n-l+k-j-i} h_{n-l+k-j-i}(\mathbb{B}) \\
& = & \sum_{i=0}^{k-1}B_{i}\sum_{j=0}^{k-1-i}Y_{j+n-l+1} (-1)^{k-1-j-i} h_{k-1-j-i}(\mathbb{B}) \\
& = & \sum_{j=0}^{k-1}Y_{j+n-l+1}\sum_{i=0}^{k-1-j} (-1)^{k-1-j-i} h_{k-1-j-i}(\mathbb{B})B_{i} \\
& = & Y_{k+n-l}.
\end{eqnarray*}
So
\[
Y_{n-l+k}-\sum_{i=0}^{n-l} B_{n-l+k-i}D_i
= Y_{n-l+k}+\sum_{i=0}^{k-1}B_{i}T_{n-l+k-i}-Y_{k+n-l}
= \sum_{i=0}^{k-1}B_{i}T_{n-l+k-i}
\]
and, therefore,
\[
Y_{n-l+k}-\sum_{i=0}^{n-l} B_{n-l+k-i}D_i - \sum_{i=1}^{k-1}B_{i}T_{n-l+k-i} = T_{n-l+k}.
\]
Thus, we can apply Corollary \ref{row-op} successively to the right hand side of \eqref{gen-chi-Gamma-0-1} to get
\begin{equation}\label{gen-chi-Gamma-0-2}
C(\Gamma_0) \simeq \left(%
\begin{array}{cc}
  \ast & X_1+D_1-A_1 \\
  \dots & \dots \\
  \ast & \sum_{i=0}^{n-l} X_{k-i}D_{i}-A_k \\
  \dots & \dots \\
  \ast & X_mD_{n-l} - A_{m+n-l} \\
  \ast & T_{n-l+1} \\
  \dots & \dots \\
  \ast & T_n \\
\end{array}%
\right)_{R}\{q^{-m(n-l)}\}.
\end{equation}

\begin{lemma}\label{gen-chi-T-recursive}
If $k>n$, then $T_k = - \sum_{j=1}^l B_j T_{k-j}$.
\end{lemma}
\begin{proof}
For $k>n$,
\begin{eqnarray*}
T_k & = & \sum_{i=0}^k (-1)^i h_i(\mathbb{B}) Y_{k-i}  = \sum_{i=0}^n (-1)^{k-i} h_{k-i}(\mathbb{B}) Y_i \\
_{(\text{by } \eqref{complete-recursion}, \text{ note that } k>n)}~ & = & -\sum_{i=0}^n (-1)^{k-i} Y_i \sum_{j=1}^l (-1)^j B_j h_{k-i-j}(\mathbb{B}) \\
& = & - \sum_{j=1}^l B_j \sum_{i=0}^n (-1)^{k-i-j} Y_i h_{k-i-j}(\mathbb{B}) \\
& = & - \sum_{j=1}^l B_j T_{k-j}.
\end{eqnarray*}
\end{proof}

\begin{lemma}\label{gen-chi-Gamma-0-lemma}
For any $k\geq 0$, define $W_k = \sum_{i=0}^k T_i X_{k-i}$. Then
\begin{equation}\label{gen-chi-Gamma-0-3}
C(\Gamma_0) \simeq \left(%
\begin{array}{cc}
  \ast & W_1-A_1 \\
  \dots & \dots \\
  \ast & W_k-A_k \\
  \dots & \dots \\
  \ast & W_{m+n-l} - A_{m+n-l} \\
  \ast & T_{n-l+1} \\
  \dots & \dots \\
  \ast & T_n \\
\end{array}%
\right)_{R}\{q^{-m(n-l)}\}.
\end{equation}
\end{lemma}
\begin{proof}
Consider $\sum_{i=0}^{n-l} X_{k-i}D_{i}-A_k$. If $k\leq n-l$, then, by equations \eqref{eq-D-k-sum} and \eqref{eq-def-T-k},
\[
\sum_{i=0}^{n-l} X_{k-i}D_{i}-A_k = \sum_{i=0}^{k} X_{k-i}T_{i}-A_k = W_k-A_k.
\]
So the row $(\ast,~\sum_{i=0}^{n-l} X_{k-i}D_{i}-A_k)$ in \eqref{gen-chi-Gamma-0-2} is already $(\ast,~W_k-A_k)$.

If $k>n-l$, then, by equations \eqref{eq-D-k-sum} and \eqref{eq-def-T-k}, 
\[
\sum_{i=0}^{n-l} X_{k-i}D_{i}-A_k = \sum_{i=0}^{n-l} X_{k-i}T_{i}-A_k = W_k-A_k - \sum_{i=n-l+1}^{k} X_{k-i}T_{i}.
\]
By Lemma \ref{gen-chi-T-recursive}, if $i\geq n-l+1$, then $T_i$ can be expressed as a combination of $T_{n-l+1},\dots,T_n$. So we can apply Corollary \ref{row-op} to the row $(\ast,~\sum_{i=0}^{n-l} X_{k-i}D_{i}-A_k)$ and the bottom $l$ rows in \eqref{gen-chi-Gamma-0-2} to change the former into $(\ast,~W_k-A_k)$.
\end{proof}

Now consider $\Gamma_1$ in Figure \ref{general-general-chi-maps-figure}. Recall that
\begin{equation}\label{gen-chi-Gamma-1-1}
C(\Gamma_1) \simeq \left(%
\begin{array}{cc}
  \ast & X_1+Y_1-A_1-B_1 \\
  \dots & \dots \\
  \ast & \sum_{i=0}^k X_i Y_{k-i} - \sum_{i=0}^k A_i B_{k-i} \\
  \dots & \dots \\
  \ast & X_m Y_n - A_{m+n-l}B_l \\
\end{array}%
\right)_{R}\{q^{-mn}\}.
\end{equation}

\begin{lemma}\label{gen-chi-Gamma-1-lemma}
\begin{equation}\label{gen-chi-Gamma-1-2}
C(\Gamma_1) \simeq \left(%
\begin{array}{cc}
  \ast & W_1-A_1 \\
  \dots & \dots \\
  \ast & W_k-A_k \\
  \dots & \dots \\
  \ast & W_{m+n-l}-A_{m+n-l} \\
  \ast & W_{m+n-l+1} \\
  \dots & \dots \\
  \ast & W_{m+n} \\
\end{array}%
\right)_{R}\{q^{-mn}\},
\end{equation}
where, as in Lemma \ref{gen-chi-Gamma-0-lemma}, $W_k = \sum_{i=0}^k T_i X_{k-i}$.
\end{lemma}
\begin{proof}
We have
\begin{eqnarray*}
\sum_{j=0}^k(-1)^{k-j} h_{k-j}(\mathbb{B})\sum_{i=0}^j X_iY_{j-i} & = & \sum_{i=0}^k X_i\sum_{j=i}^k (-1)^{k-j} h_{k-j}(\mathbb{B})Y_{j-i} \\
& = &  \sum_{i=0}^k X_i\sum_{j=0}^{k-i} (-1)^{k-i-j} h_{k-i-j}(\mathbb{B})Y_j \\
& = &  \sum_{i=0}^k X_i T_{k-i} = W_k
\end{eqnarray*}
and, by \eqref{complete-recursion},
\begin{eqnarray*}
\sum_{j=0}^k(-1)^{k-j} h_{k-j}(\mathbb{B})\sum_{i=0}^j A_iB_{j-i} & = & \sum_{i=0}^k A_i\sum_{j=i}^k (-1)^{k-j} h_{k-j}(\mathbb{B})B_{j-i} \\
& = &  \sum_{i=0}^k A_i\sum_{j=0}^{k-i} (-1)^{k-i-j} h_{k-i-j}(\mathbb{B})B_j \\
& = &  A_k.
\end{eqnarray*}
Thus, 
\begin{eqnarray*}
& & \sum_{j=1}^k(-1)^{k-j} h_{k-j}(\mathbb{B})(\sum_{i=0}^j X_iY_{j-i}-\sum_{i=0}^j A_iB_{j-i}) \\
& = & \sum_{j=0}^k(-1)^{k-j} h_{k-j}(\mathbb{B})(\sum_{i=0}^j X_iY_{j-i}-\sum_{i=0}^j A_iB_{j-i}) \\
& = & W_k -A_k.
\end{eqnarray*}
This implies that we can get \eqref{gen-chi-Gamma-1-2} by successively applying Corollary \ref{row-op} to the right hand side of \eqref{gen-chi-Gamma-1-1}.
\end{proof}

By the definition of $W_k$ in Lemma \ref{gen-chi-Gamma-0-lemma}, we know that
\[
\left(%
\begin{array}{l}
W_{m+n} \\
W_{m+n-1} \\
\dots \\
W_{m+n-l+1} \\
\end{array}%
\right)
=
\Omega 
\left(%
\begin{array}{l}
T_{m+n} \\
T_{m+n-1} \\
\dots \\
T_{n-l+1} \\
\end{array}%
\right),
\]
where 
\begin{equation}\label{eq-def-matrix-Omega}
\Omega = (X_{j-i})_{l\times (m+l)},
\end{equation}
that is, $\Omega$ is the $l\times (m+l)$ matrix whose $(i,j)$-th entry is $X_{j-i}$. By Lemma \ref{gen-chi-T-recursive}, we have, for $k\geq1$,
\[
\left(%
\begin{array}{l}
T_{n+k} \\
T_{n+k-1} \\
\dots \\
T_{n-l+1} \\
\end{array}%
\right)
=
\Theta_k
\left(%
\begin{array}{l}
T_{n+k-1} \\
T_{n+k-2} \\
\dots \\
T_{n-l+1} \\
\end{array}%
\right),
\]
where $\Theta_k$ is the $(k+l)\times (k+l-1)$ matrix whose first row is $(-B_1,-B_2,\dots,-B_l,0,\dots,0)$ and whose next $(k+l-1)$ rows form the $(k+l-1)\times (k+l-1)$ identity matrix $\Id_{k+l-1}$. Define 
\begin{equation}\label{eq-def-matrix-Theta}
\Theta = \Theta_m \Theta_{m-1}\cdots\Theta_2\Theta_1. 
\end{equation}
Then 
\[
\left(%
\begin{array}{l}
W_{m+n} \\
W_{m+n-1} \\
\dots \\
W_{m+n-l+1} \\
\end{array}%
\right)
=
\Omega \Theta
\left(%
\begin{array}{l}
T_{n} \\
T_{n-1} \\
\dots \\
T_{n-l+1} \\
\end{array}%
\right),
\]
where $\Omega \Theta$ is clearly an $l\times l$ matrix. So 
\begin{equation}\label{gen-chi-relation-matrix}
\left(%
\begin{array}{l}
W_1-A_1 \\
\dots \\
W_{m+n-l}-A_{m+n-l} \\
W_{m+n} \\
\dots \\
W_{m+n-l+1} \\
\end{array}%
\right)
=
\left(%
\begin{array}{ll}
\Id_{m+n-l} & 0 \\
0 & \Omega \Theta\\
\end{array}%
\right)
\left(%
\begin{array}{l}
W_1-A_1 \\
\dots \\
W_{m+n-l}-A_{m+n-l} \\
T_{n} \\
\dots \\
T_{n-l+1} \\
\end{array}%
\right).
\end{equation}

\begin{lemma}\label{gen-chi-adjoint}
$C(\Gamma_0)\{q^{m(n-l)}\}$ and $C(\Gamma_1)\{q^{mn}\}$ are homotopic to a pair of adjoint Koszul matrix factorizations with the relation matrix 
\[
\left(%
\begin{array}{ll}
\Id_{m+n-l} & 0 \\
0 & \Omega \Theta\\
\end{array}%
\right)^t = 
\left(%
\begin{array}{ll}
\Id_{m+n-l} & 0 \\
0 & \Theta^t \Omega^t\\
\end{array}%
\right).
\]
\end{lemma}
\begin{proof}
This lemma follows from \eqref{gen-chi-relation-matrix}, Lemmas \ref{gen-chi-Gamma-0-lemma}, \ref{gen-chi-Gamma-1-lemma} and Remark \ref{MOY-freedom}.
\end{proof}

The morphisms $\chi^0$ and $\chi^1$ in Proposition \ref{general-general-chi-maps} are constructed by applying Proposition \ref{general-jumping-factor} to this pair of adjoint matrix factorizations in Lemma \ref{gen-chi-adjoint}. To prove that $\chi^0$ and $\chi^1$ satisfy all the requirements in Proposition \ref{general-general-chi-maps}, all we need to do is to compute 
\[
\det\left(%
\begin{array}{ll}
\Id_{m+n-l} & 0 \\
0 & \Omega \Theta\\
\end{array}%
\right)^t
=\det(\Omega \Theta).
\]
For this purpose, we introduce some properties of Schur polynomials associated to hook partitions. 

For $i,j\geq 0$, let $L_{i,j}=(i+1\geq\underbrace{1\geq\cdots\geq1}_{j~``1"s})$, the $(i,j)$-hook partition. Note that $L_{i,j}'=L_{j,i}$. So, by equations \eqref{schur-complete} and \eqref{schur-elementary},
\begin{eqnarray}
\label{schur-hook-1} \hspace {1cm} S_{L_{i,j}}(\mathbb{B}) & = & \left|%
\begin{array}{llllll}
h_{i+1}(\mathbb{B}) & h_{i+2}(\mathbb{B}) & h_{i+3}(\mathbb{B}) & \cdots & h_{i+j}(\mathbb{B}) & h_{i+j+1}(\mathbb{B}) \\
1 & h_{1}(\mathbb{B}) & h_{2}(\mathbb{B}) & \cdots & h_{j-1}(\mathbb{B}) & h_{j}(\mathbb{B}) \\
0 & 1 & h_{1}(\mathbb{B}) & \cdots & h_{j-2}(\mathbb{B}) & h_{j-1}(\mathbb{B}) \\
0& 0 & 1 & \cdots & h_{j-3}(\mathbb{B}) & h_{j-2}(\mathbb{B}) \\
\cdots & \cdots & \cdots & \cdots & \cdots & \cdots\\
0 & 0 & 0 & \cdots & 1 & h_{1}(\mathbb{B}) \\
\end{array}%
\right|, \\
\label{schur-hook-2} S_{L_{i,j}}(\mathbb{B}) & = & \left|%
\begin{array}{llllll}
B_{j+1} & B_{j+2} & B_{j+3} & \cdots & B_{j+i} & B_{j+i+1} \\
1 & B_{1}& B_{2} & \cdots & B_{i-1} & B_{i} \\
0 & 1 & B_{1} & \cdots & B_{i-3} & B_{i-1} \\
0 & 0 & 1 & \cdots & B_{i-3} & B_{i-2} \\
\cdots & \cdots & \cdots & \cdots & \cdots &\cdots \\
0& 0 & 0 & \cdots & 1 & B_{1} \\
\end{array}%
\right|.
\end{eqnarray}
Using \eqref{schur-hook-1} and \eqref{schur-hook-2}, we can extend the definition of $S_{L_{i,j}}(\mathbb{B})$ to allow one of $i,j$ to be negative. This gives:
\begin{enumerate}[(i)]
	\item if $j\geq 0$, then
	\[
	S_{L_{i,j}}(\mathbb{B}) = \begin{cases}
	S_{L_{i,j}}(\mathbb{B}) \text{ as in \eqref{schur-hook-1}} & \text{if } i\geq 0, \\
	(-1)^j & \text{if } i=-j-1, \\
	0 & \text{if } i<0 \text{ and } i\neq-j-1, \\
	\end{cases}
	\]
	\item if $i\geq 0$, then
	\[
	S_{L_{i,j}}(\mathbb{B}) = \begin{cases}
	S_{L_{i,j}}(\mathbb{B}) \text{ as in \eqref{schur-hook-2}} & \text{if } j\geq 0, \\
	(-1)^i & \text{if } j=-i-1, \\
	0 & \text{if } j<0 \text{ and } j\neq-i-1. \\
	\end{cases}
	\]
\end{enumerate}

\begin{lemma}\label{schur-hook-recursion}
Define $\tau_{i,j}= (-1)^{i+1}S_{L_{i,j}}(\mathbb{B})$. Then, for $i,j\geq 0$, 
\begin{eqnarray}
\label{schur-hook-recursion-1} B_{i+j+1} & = & -\sum_{k=0}^i B_k \tau_{i-k,j}, \\
\label{schur-hook-recursion-2} (-1)^{i+j+1} h_{i+j+1}(\mathbb{B}) & = & \sum_{k=0}^j (-1)^k h_{k}(\mathbb{B}) \tau_{i,j-k}.
\end{eqnarray}
\end{lemma}
\begin{proof}
We prove \eqref{schur-hook-recursion-1} first. Note that, for any $i\geq 0$, 
\[
h_{i+1}(\mathbb{B}) = \sum_{k=1}^\infty (-1)^{k+1} B_k h_{i+1-k}(\mathbb{B}), 
\]
where the right hand side is in fact a finite sum. Applying this equation to every entry in the first row of \eqref{schur-hook-1}, we get
\begin{eqnarray*}
S_{L_{i,j}}(\mathbb{B}) & = & \sum_{k=1}^\infty (-1)^{k+1} B_k S_{L_{i-k,j}}(\mathbb{B}) \\
_{(\text{by (i) above})}~ & = & (\sum_{k=1}^i (-1)^{k+1} B_k S_{L_{i-k,j}}(\mathbb{B})) + (-1)^iB_{i+j+1}.
\end{eqnarray*}
Equation \eqref{schur-hook-recursion-1} follows from this and the definition of $\tau_{i,j}$.

Now we prove \eqref{schur-hook-recursion-2}. For any $j\geq 0$, 
\[
B_{j+1} = \sum_{k=1}^\infty (-1)^{k+1} h_{k}(\mathbb{B}) B_{j+1-k}, 
\]
where the right hand side is again a finite sum. Apply this equation to every entry in the first row of \eqref{schur-hook-2}, we get
\begin{eqnarray*}
S_{L_{i,j}}(\mathbb{B}) & = & \sum_{k=1}^\infty (-1)^{k+1} h_k(\mathbb{B}) S_{L_{i,j-k}}(\mathbb{B}) \\
_{(\text{by (ii) above})}~ & = & (\sum_{k=1}^j (-1)^{k+1} h_k(\mathbb{B}) S_{L_{i,j-k}}(\mathbb{B})) + (-1)^j h_{i+j+1}(\mathbb{B}).
\end{eqnarray*}
Equation \eqref{schur-hook-recursion-2} follows from this and the definition of $\tau_{i,j}$.
\end{proof}

\begin{lemma}\label{gen-chi-computing-Theta}
Let $\Theta$ be the matrix defined in \eqref{eq-def-matrix-Theta}. Then
\[
\Theta = (\tau_{m-i,j-1})_{(l+m)\times l} = \left(%
\begin{array}{llll}
\tau_{m-1,0} & \tau_{m-1,1} & \cdots &\tau_{m-1,l-1} \\
\tau_{m-2,0} & \tau_{m-2,1} & \cdots &\tau_{m-2,l-1} \\
\cdots & \cdots & \cdots & \cdots \\
\tau_{0,0} & \tau_{0,1} & \cdots &\tau_{0,l-1} \\
1& 0 & \cdots &0 \\
0 & 1 &\cdots &0 \\
\cdots & \cdots & \cdots & \cdots \\
0 & 0 &\cdots & 1 \\
\end{array}%
\right).
\]
\end{lemma}
\begin{proof}
Note that $\tau_{0,j-1}=-B_j$. Recall that
\[
\Theta_1 =  \left(%
\begin{array}{llll}
-B_1 & -B_2 & \cdots & -B_l \\
1& 0 & \cdots &0 \\
0 & 1 &\cdots &0 \\
\cdots & \cdots & \cdots & \cdots \\
0 & 0 &\cdots & 1 \\
\end{array}%
\right)
=  \left(%
\begin{array}{llll}
\tau_{0,0} & \tau_{0,1} & \cdots &\tau_{0,l-1} \\
1& 0 & \cdots &0 \\
0 & 1 &\cdots &0 \\
\cdots & \cdots & \cdots & \cdots \\
0 & 0 &\cdots & 1 \\
\end{array}%
\right)
\]
and
\[
\Theta_k = \left(%
\begin{array}{ll}
\Theta_1 & 0 \\
0 & \Id_{k-1}
\end{array}%
\right).
\]
Using equation \eqref{schur-hook-recursion-1} in Lemma \ref{schur-hook-recursion}, it is easy to prove by induction that
\[
\Theta_k\Theta_{k-1}\cdots\Theta_1 = (\tau_{k-i,j-1})_{(l+k)\times l} = \left(%
\begin{array}{llll}
\tau_{k-1,0} & \tau_{k-1,1} & \cdots &\tau_{k-1,l-1} \\
\tau_{k-2,0} & \tau_{k-2,1} & \cdots &\tau_{k-2,l-1} \\
\cdots & \cdots & \cdots & \cdots \\
\tau_{0,0} & \tau_{0,1} & \cdots &\tau_{0,l-1} \\
1& 0 & \cdots &0 \\
0 & 1 &\cdots &0 \\
\cdots & \cdots & \cdots & \cdots \\
0 & 0 &\cdots & 1 \\
\end{array}%
\right).
\]
But $\Theta=\Theta_m\Theta_{m-1}\cdots\Theta_1$. So the lemma is the $k=m$ case of the above equation.
\end{proof}

Define 
\begin{itemize}
	\item $u_{i,j}= (-1)^{j-i} h_{j-i}(\mathbb{B})$ for $1 \leq i,j \leq l$,
	\item $v_{i,j} = (-1)^{j+m-i} h_{j+m-i}(\mathbb{B})$ for $1\leq i \leq l+m$ and $1\leq j\leq l$.
\end{itemize}
Let
\begin{eqnarray*}
U& = &(u_{i,j})_{l \times l} = \left(%
\begin{array}{lllll}
1 & -h_1(\mathbb{B}) & \cdots & (-1)^{l-2} h_{l-2}(\mathbb{B}) & (-1)^{l-1} h_{l-1}(\mathbb{B}) \\
0 & 1  & \cdots & (-1)^{l-3} h_{l-3}(\mathbb{B}) &(-1)^{l-2} h_{l-2}(\mathbb{B}) \\
\cdots & \cdots & \cdots & \cdots &\cdots \\
0 & 0 & \cdots & 1 &-h_1(\mathbb{B}) \\
0 & 0 & \cdots & 0 & 1 \\
\end{array}%
\right), \\
V & =& (v_{i,j})_{(l+m) \times l} \\
& = &   \left(%
\begin{array}{lllll}
(-1)^m h_m(\mathbb{B}) & (-1)^{m+1}h_{m+1}(\mathbb{B}) & \cdots &  (-1)^{m+l-2} h_{m+l-2}(\mathbb{B}) & (-1)^{m+l-1} h_{m+l-1}(\mathbb{B}) \\
\cdots & \cdots & \cdots & \cdots &\cdots \\
-h_1(\mathbb{B}) & h_{2}(\mathbb{B}) & \cdots &  (-1)^{l-1} h_{l-1}(\mathbb{B}) & (-1)^{l} h_{l}(\mathbb{B}) \\
1 & -h_1(\mathbb{B}) & \cdots & (-1)^{l-2} h_{l-2}(\mathbb{B}) & (-1)^{l-1} h_{l-1}(\mathbb{B}) \\
0 & 1  & \cdots & (-1)^{l-3} h_{l-3}(\mathbb{B}) &(-1)^{l-2} h_{l-2}(\mathbb{B}) \\
\cdots & \cdots & \cdots & \cdots &\cdots \\
0 & 0 & \cdots & 1 &-h_1(\mathbb{B}) \\
0 & 0 & \cdots & 0 & 1 \\
\end{array}%
\right), \\
\end{eqnarray*}

\begin{lemma}\label{gen-chi-simplify-Theta-to-V}
Let $\Theta$ be the matrix defined in \eqref{eq-def-matrix-Theta}. Then $\Theta U = V$.
\end{lemma}
\begin{proof}
This follows from Lemma \ref{gen-chi-computing-Theta} and equation \eqref{schur-hook-recursion-2} in Lemma \ref{schur-hook-recursion}.
\end{proof}

\begin{lemma}\label{gen-chi-computing-det-Omega-Theta}
Let $\Omega$ and $\Theta$ be the matrices defined in \eqref{eq-def-matrix-Omega} and \eqref{eq-def-matrix-Theta}. Then
\[
\det\left(%
\begin{array}{ll}
\Id_{m+n-l} & 0 \\
0 & \Omega \Theta\\
\end{array}%
\right)^t
=\det(\Omega \Theta) = (-1)^{ml} \sum_{\lambda\in\Lambda_{l,m}} (-1)^{|\lambda|} S_{\lambda'}(\mathbb{X}) S_{\lambda^c}(\mathbb{B}).
\]
\end{lemma}
\begin{proof}
Note that $\det U=1$. So, by Lemma \ref{gen-chi-simplify-Theta-to-V}, $\det (\Omega \Theta) = \det (\Omega \Theta U) = \det (\Omega V)$. Let 
\[
\mathcal{I}:=\{I=(i_1,\dots,i_l)|1\leq i_1<\cdots<i_l \leq l+m\}.
\]
For any $I=(i_1,\dots,i_l) \in \mathcal{I}$, define
\begin{itemize}
	\item $\Omega_I$ to be the $l\times l$ minor matrix of $\Omega$ consisting of the $i_1,\dots,i_l$-th columns of $\Omega$,
	\item $V_I$ to be the $l\times l$ minor matrix of $V$ consisting of the $i_1,\dots,i_l$-th rows of $V$.
\end{itemize}
Then, by the Binet-Cauchy Theorem, 
\begin{equation}\label{gen-chi-Binet-Cauchy}
\det (\Omega \Theta) = \det (\Omega V) = \sum_{I \in \mathcal{I}} \det \Omega_I \cdot \det V_I.
\end{equation}

Recall that $\Lambda_{l,m}= \{\mu=(\mu_1\geq\cdots\geq\mu_l) ~|~ \mu_1 \leq m\}$. Note that there is a one-to-one correspondence $\jmath:\mathcal{I} \rightarrow \Lambda_{l,m}$ given by 
\[
\jmath(I)=(i_l-l\geq i_{l-1}-l+1 \geq \cdots \geq i_1-1)
\] 
for any $I=(i_1,\dots,i_l) \in \mathcal{I}$. The inverse of $\jmath$ is given by
\[
\jmath^{-1} (\lambda) = (\lambda_l+1,\lambda_{l-1}+2,\dots,\lambda_1+l)
\]
for any $\lambda=(\lambda_1\geq\cdots\geq \lambda_l)\in \Lambda_{l,m}$.

For any $I=(i_1,\dots,i_l) \in \mathcal{I}$, 
\begin{eqnarray*}
\det \Omega_I & = & \left|%
\begin{array}{lllll}
X_{i_1-1} & X_{i_2-1} & \dots & X_{i_{l-1}-1} & X_{i_l-1} \\ 
X_{i_1-2} & X_{i_2-2} & \dots & X_{i_{l-1}-2} & X_{i_l-2} \\ 
\cdots & \cdots & \cdots & \cdots &\cdots \\
X_{i_1-l+1} & X_{i_2-l+1} & \dots & X_{i_{l-1}-l+1} & X_{i_l-l+1} \\ 
X_{i_1-l} & X_{i_2-l} & \dots & X_{i_{l-1}-l} & X_{i_l-l} \\ 
\end{array}%
\right| \\
& = &
\left|%
\begin{array}{lllll}
X_{i_l-l} & X_{i_l-l+1} & \dots & X_{i_{l}-2} & X_{i_l-1} \\ 
X_{i_{l-1}-l} & X_{i_{l-1}-l+1} & \dots & X_{i_{l-1}-2} & X_{i_{l-1}-1} \\ 
\cdots & \cdots & \cdots & \cdots &\cdots \\
X_{i_2-l} & X_{i_2-l+1} & \dots & X_{i_{2}-2} & X_{i_2-1} \\ 
X_{i_1-l} & X_{i_1-l+1} & \dots & X_{i_{1}-2} & X_{i_1-1} \\ 
\end{array}%
\right| \\
& = & S_{\jmath(I)'}(\mathbb{X}).
\end{eqnarray*}
To shorten the equation, we let $h_i = h_i(\mathbb{B})$ and $\hbar_{i}=(-1)^i h_i(\mathbb{B})$. Then, for any $I=(i_1,\dots,i_l) \in \mathcal{I}$,
\begin{eqnarray*}
&& \det V_I \\
& = & \left|%
\begin{array}{lllll}
\hbar_{m+1-i_1} & \hbar_{m+2-i_1} & \dots & \hbar_{m+l-1-i_1} & \hbar_{m+l-i_1} \\
\hbar_{m+1-i_2} & \hbar_{m+2-i_2} & \dots & \hbar_{m+l-1-i_2} & \hbar_{m+l-i_2} \\
\cdots & \cdots & \cdots & \cdots &\cdots \\
\hbar_{m+1-i_{l-1}} & \hbar_{m+2-i_{l-1}} & \dots & \hbar_{m+l-1-i_{l-1}} & \hbar_{m+l-i_{l-1}} \\
\hbar_{m+1-i_l} & \hbar_{m+2-i_l} & \dots & \hbar_{m+l-1-i_l} & \hbar_{m+l-i_l} \\
\end{array}%
\right| \\
& = & (-1)^{ml+\frac{l(l+1)}{2} -\sum_{k=1}^l i_k}
\left|%
\begin{array}{lllll}
h_{m+1-i_1} & h_{m+2-i_1} & \dots & h_{m+l-1-i_1} & h_{m+l-i_1} \\
h_{m+1-i_2} & h_{m+2-i_2} & \dots & h_{m+l-1-i_2} & h_{m+l-i_2} \\
\cdots & \cdots & \cdots & \cdots &\cdots \\
h_{m+1-i_{l-1}} & h_{m+2-i_{l-1}} & \dots & h_{m+l-1-i_{l-1}} & h_{m+l-i_{l-1}} \\
h_{m+1-i_l} & h_{m+2-i_l} & \dots & h_{m+l-1-i_l} & h_{m+l-i_l} \\
\end{array}%
\right| \\
& = & (-1)^{ml+|\jmath(I)|} S_{\jmath(I)^c}(\mathbb{B}).
\end{eqnarray*}
So \eqref{gen-chi-Binet-Cauchy} gives that
\begin{eqnarray*}
\det (\Omega \Theta) & = & (-1)^{ml}\sum_{I \in \mathcal{I}}  (-1)^{|\jmath(I)|}  S_{\jmath(I)'}(\mathbb{X}) S_{\jmath(I)^c}(\mathbb{B}) \\
& = & (-1)^{ml} \sum_{\lambda\in\Lambda_{l,m}} (-1)^{|\lambda|} S_{\lambda'}(\mathbb{X}) S_{\lambda^c}(\mathbb{B}). \\
\end{eqnarray*}
\end{proof}

\begin{proof}[Proof of Proposition \ref{general-general-chi-maps}]
Proposition \ref{general-general-chi-maps} now follows easily from Proposition \ref{general-jumping-factor}, Lemma \ref{gen-chi-adjoint} and Lemma \ref{gen-chi-computing-det-Omega-Theta}.
\end{proof}

\begin{proposition}\label{general-general-chi-maps-HMF}
Let $\Gamma_0$, $\Gamma_1$, $\chi^0$ and $\chi^1$ be as in Proposition \ref{general-general-chi-maps}. Then $\chi^0$ and $\chi^1$ are homotopically non-trivial. Moreover, up to homotopy and scaling, $\chi^0$ (resp. $\chi^1$) is the unique homotopically non-trivial homogeneous morphism of quantum degree $ml$ from $C(\Gamma_0)$ to $C(\Gamma_1)$ (resp. from $C(\Gamma_1)$ to $C(\Gamma_0)$.)
\end{proposition}

\begin{figure}[ht]

\setlength{\unitlength}{1pt}

\begin{picture}(360,75)(-180,-15)


\put(-130,10){\vector(1,1){10}}

\put(-120,20){\line(1,-1){10}}

\put(-120,40){\vector(0,-1){20}}

\put(-120,40){\line(1,1){10}}

\put(-130,50){\vector(1,-1){10}}

\put(-152,45){\tiny{$_{m+n-l}$}}

\put(-135,15){\tiny{$_l$}}

\put(-115,28){\tiny{$_{n-l}$}}

\qbezier(-110,10)(-100,0)(-90,10)

\qbezier(-110,50)(-100,60)(-90,50)

\qbezier(-130,10)(-140,0)(-100,0)

\qbezier(-30,10)(-20,0)(-100,0)

\qbezier(-130,50)(-140,60)(-100,60)

\qbezier(-30,50)(-30,60)(-100,60)

\put(-90,-15){$\Gamma$}

\put(-90,10){\vector(1,1){20}}

\put(-90,50){\vector(1,-1){20}}

\put(-70,30){\vector(1,0){20}}

\put(-50,30){\line(1,1){20}}

\put(-50,30){\line(1,-1){20}}

\put(-82,45){\tiny{$_{m}$}}

\put(-80,15){\tiny{$_{n}$}}

\put(-69,32){\tiny{$_{m+n}$}}


\put(30,40){\vector(3,1){30}}

\put(30,40){\line(1,1){10}}

\put(20,50){\vector(1,-1){10}}

\put(120,45){\tiny{$_{m+n-l}$}}

\put(67,45){\tiny{$_{m+n-l}$}}

\put(33,50){\tiny{$_{m}$}}

\put(120,15){\tiny{$_l$}}

\put(40,40){\tiny{$_{n-l}$}}

\qbezier(40,50)(50,60)(60,50)

\qbezier(60,10)(50,0)(90,0)

\qbezier(120,10)(130,0)(90,0)

\qbezier(20,50)(10,60)(50,60)

\qbezier(120,50)(130,60)(50,60)

\put(60,-15){$\Gamma'$}

\put(60,10){\vector(1,1){20}}

\put(60,50){\vector(1,-1){20}}

\put(80,30){\vector(1,0){20}}

\put(100,30){\line(1,1){20}}

\put(100,30){\line(1,-1){20}}

\put(81,32){\tiny{$_{m+n}$}}

\end{picture}

\caption{}\label{general-general-chi-maps-HMF-figure}

\end{figure}

\begin{proof}
Using Lemmas \ref{gen-chi-Gamma-0-lemma} and \ref{gen-chi-Gamma-1-lemma}, one can check that, as graded vector spaces, 
\begin{eqnarray*}
\Hom_\HMF(C(\Gamma_1),C(\Gamma_0)) & \cong & H(\Gamma) \left\langle m+n\right\rangle  \{q^{(m+n)(N-m-n)+mn+ml+nl-l^2}\}, \\
\Hom_\HMF(C(\Gamma_0),C(\Gamma_1)) & \cong & H(\overline{\Gamma}) \left\langle m+n\right\rangle \{q^{(m+n)(N-m-n)+mn+ml+nl-l^2}\}, \\
\end{eqnarray*}
where $\Gamma$ is the MOY graph in Figure \ref{general-general-chi-maps-HMF-figure}, and $\overline{\Gamma}$ is $\Gamma$ with orientation reversed. By Corollary \ref{contract-expand}, we have $H(\Gamma) \cong H(\Gamma')$. Then, by Decomposition (II) (Theorem \ref{decomp-II}) and Corollary \ref{circle-dimension}, we have
\[
H(\Gamma) \cong H(\Gamma') \cong C(\emptyset) \left\langle m+n \right\rangle \{\qb{m+n-l}{m} \qb{m+n}{l}\qb{N}{m+n}\}.
\]
Similarly,
\[
H(\overline{\Gamma}) \cong C(\emptyset) \left\langle m+n \right\rangle \{\qb{m+n-l}{m} \qb{m+n}{l}\qb{N}{m+n}\}.
\]
Thus, as graded vector spaces,
\begin{eqnarray*}
& & \Hom_\HMF(C(\Gamma_1),C(\Gamma_0)) \\
& \cong & \Hom_\HMF(C(\Gamma_0),C(\Gamma_1)) \\
& \cong & C(\emptyset) \{\qb{m+n-l}{m} \qb{m+n}{l}\qb{N}{m+n} q^{(m+n)(N-m-n)+mn+ml+nl-l^2}\}. \\
\end{eqnarray*}
In particular, the lowest non-vanishing quantum grading of the above spaces is $ml$, and the subspaces of these spaces of homogeneous elements of quantum degree $ml$ are $1$-dimensional. So, to prove the proposition, we only need to show that $\chi^0$ and $\chi^1$ are homotopically non-trivial. To prove this, we use the diagram in Figure \ref{general-general-chi-maps-non-vanishing}.

\begin{figure}[ht]
$
\xymatrix{
\input{v-vector-m+n} \ar@<10ex>[rr]^{\phi_1\otimes\phi_2}  & &  \ar@<-8ex>[ll]^{\overline{\phi}_1\otimes\overline{\phi}_2} \input{v-vector-two-bubbles}  \ar@<10ex>[rr]^{\chi^1} & & \input{v-vector-theta-shape} \ar@<-8ex>[ll]^{\chi^0} \\
}
$
\caption{}\label{general-general-chi-maps-non-vanishing}

\end{figure}

Consider the morphisms in Figure \ref{general-general-chi-maps-non-vanishing}, where $\phi_1$ and $\overline{\phi}_1$ (resp. $\phi_2$ and $\overline{\phi}_2$) are induced by the edge splitting and merging of the upper (resp. lower ) bubble, and $\chi^0$ and $\chi^1$ are the morphisms from Proposition \ref{general-general-chi-maps}. Let us compute the composition 
\[
(\overline{\phi}_1\otimes\overline{\phi}_2) \circ \mathfrak{m}(S_{\lambda_{m,n}}(-\mathbb{Y}) \cdot S_{\lambda_{l,n-l}}(-\mathbb{A})) \circ \chi^0 \circ \chi^1 \circ (\phi_1\otimes\phi_2),
\]
where $\mathfrak{m}(S_{\lambda_{m,n}}(-\mathbb{Y}) \cdot S_{\lambda_{l,n-l}}(-\mathbb{A}))$ is the morphism induced by the multiplication by $S_{\lambda_{m,n}}(-\mathbb{Y}) \cdot S_{\lambda_{l,n-l}}(-\mathbb{A})$. By Proposition \ref{general-general-chi-maps}, we have
\begin{eqnarray*}
&& (\overline{\phi}_1\otimes\overline{\phi}_2) \circ \mathfrak{m}(S_{\lambda_{m,n}}(-\mathbb{Y}) \cdot S_{\lambda_{l,n-l}}(-\mathbb{A})) \circ \chi^0 \circ \chi^1 \circ (\phi_1\otimes\phi_2) \\
& \simeq & (\overline{\phi}_1\otimes\overline{\phi}_2) \circ \mathfrak{m}(S_{\lambda_{m,n}}(-\mathbb{Y}) \cdot S_{\lambda_{l,n-l}}(-\mathbb{A}) \cdot (\sum_{\lambda\in\Lambda_{l,m}} (-1)^{|\lambda|} S_{\lambda'}(\mathbb{X}) S_{\lambda^c}(\mathbb{B})))  \circ (\phi_1\otimes\phi_2) \\
& = & \sum_{\lambda\in\Lambda_{l,m}} (-1)^{|\lambda|} (\overline{\phi}_1 \circ \mathfrak{m}(S_{\lambda_{m,n}}(-\mathbb{Y}) \cdot S_{\lambda'}(\mathbb{X})) \circ \phi_1) \otimes (\overline{\phi}_2 \circ \mathfrak{m}(S_{\lambda_{l,n-l}}(-\mathbb{A}) \cdot S_{\lambda^c}(\mathbb{B})) \circ \phi_2).
\end{eqnarray*}
But, by Lemma \ref{phibar-compose-phi}, we have that, for $\lambda\in\Lambda_{l,m}$,
\begin{eqnarray*}
\overline{\phi}_1 \circ \mathfrak{m}(S_{\lambda_{m,n}}(-\mathbb{Y}) \cdot S_{\lambda'}(\mathbb{X})) \circ \phi_1 & \approx & \begin{cases}
\id & \text{if } \lambda=(0\geq\cdots\geq0), \\
0 & \text{if } \lambda\neq(0\geq\cdots\geq0),
\end{cases} \\
\overline{\phi}_2 \circ \mathfrak{m}(S_{\lambda_{l,n-l}}(-\mathbb{A}) \cdot S_{\lambda^c}(\mathbb{B})) \circ \phi_2 & \approx & \begin{cases}
\id & \text{if } \lambda=(0\geq\cdots\geq0), \\
0 & \text{if } \lambda\neq(0\geq\cdots\geq0).
\end{cases} \\
\end{eqnarray*}
So, 
\[
(\overline{\phi}_1\otimes\overline{\phi}_2) \circ \mathfrak{m}(S_{\lambda_{m,n}}(-\mathbb{Y}) \cdot S_{\lambda_{l,n-l}}(-\mathbb{A})) \circ \chi^0 \circ \chi^1 \circ (\phi_1\otimes\phi_2) \approx \id,
\]
which implies that $\chi^0$ and $\chi^1$ are not homotopic to $0$.
\end{proof}

\subsection{Adding and removing a loop}\label{subsec-loop-morph} Using the $\chi$-morphisms and the morphisms associated to circle creation and annihilation, one can construction morphisms associated to adding and removing loops.

\begin{figure}[ht]

\setlength{\unitlength}{1pt}

\begin{picture}(360,75)(-180,-90)


\put(-57,-45){\tiny{$m$}}

\put(-60,-75){\vector(0,1){50}}

\put(-70,-30){\small{$\mathbb{X}$}}

\put(-65,-90){$\Gamma_0$}


\put(-25,-50){\vector(1,0){50}}

\put(25,-60){\vector(-1,0){50}}

\put(-5,-47){\small{$\psi$}}

\put(-5,-70){\small{$\overline{\psi}$}}


\put(60,-75){\vector(0,1){10}}

\put(60,-65){\vector(0,1){20}}

\put(60,-50){\vector(0,1){25}}

\put(59,-52){\line(1,0){2}}

\qbezier(60,-60)(80,-75)(80,-55)

\qbezier(60,-40)(80,-25)(80,-45)

\put(80,-45){\vector(0,-1){10}}

\put(79,-47){\line(1,0){2}}

\put(63,-30){\tiny{$m$}}

\put(63,-70){\tiny{$m$}}

\put(83,-55){\tiny{$n$}}

\put(38,-45){\tiny{$m+n$}}

\put(50,-30){\small{$\mathbb{X}$}}

\put(83,-50){\small{$\mathbb{B}$}}

\put(50,-55){\small{$\mathbb{Y}$}}

\put(55,-90){$\Gamma_1$}

\end{picture}

\caption{}\label{loop-addition}

\end{figure}

\begin{lemma}\label{loop-addition-lemma}
Let $\Gamma_0$ and $\Gamma_1$ be the MOY graphs in Figure \ref{loop-addition}. Then, as $\zed_2\oplus\zed$-graded vector spaces over $\C$,
\[
\Hom_{HMF}(C(\Gamma_0),C(\Gamma_1)) \cong \Hom_{HMF}(C(\Gamma_1),C(\Gamma_0)) \cong C(\emptyset) \{\qb{N}{m}\qb{N-m}{n} q^{m(N-m)}\} \left\langle n \right\rangle.
\]
In particular, the subspaces of these spaces of homogeneous elements of quantum degree $-n(N-n)+mn$ are $1$-dimensional.
\end{lemma}
\begin{proof}
By Theorem \ref{decomp-I}, we have that $C(\Gamma_1) \simeq C(\Gamma_0) \{\qb{N-m}{n}\}\left\langle n \right\rangle$. So 
\[
\Hom_\HMF(C(\Gamma_0),C(\Gamma_1)) \cong \Hom_\HMF(C(\Gamma_0),C(\Gamma_0))\{\qb{N-m}{n}\}\left\langle n \right\rangle \cong \Hom_\HMF(C(\Gamma_1),C(\Gamma_0)).
\]
Denote by $\bigcirc_m$ a circle colored by $m$. Then, from the proof of Lemma \ref{circle-rep-two-marks}, we have  
\[
\Hom_\HMF(C(\Gamma_0),C(\Gamma_0)) \cong H(\bigcirc_m) \{q^{m(N-m)}\} \left\langle m \right\rangle \cong C(\emptyset)\{\qb{N}{m} q^{m(N-m)}\}.
\]
This proves the lemma.
\end{proof}

\begin{definition}
Let $\Gamma_0$ and $\Gamma_1$ be the MOY graphs in Figure \ref{loop-addition}. Associate to the loop addition a homogeneous morphism 
\[
\psi: C(\Gamma_0)\rightarrow C(\Gamma_1)
\]
of quantum degree $-n(N-n)+mn$ not homotopic to $0$.

Associate to the loop removal a homogeneous morphism 
\[
\overline{\psi}: C(\Gamma_1)\rightarrow C(\Gamma_0)
\]
of quantum degree $-n(N-n)+mn$ not homotopic to $0$.

By Lemma \ref{loop-addition-lemma}, $\psi$ and $\overline{\psi}$ are well defined up to homotopy and scaling. Both of them have $\zed_2$-grading $n$.
\end{definition}

\begin{figure}[ht]
$
\xymatrix{
\input{v-vector-m} \ar@<14ex>[rr]^{\iota}  & &  \ar@<-12ex>[ll]^{\epsilon} \input{v-vector-m-circle-n}  \ar@<14ex>[rr]^{\chi^0} & & \input{v-vector-m-loop-n} \ar@<-12ex>[ll]^{\chi^1} \\
}
$
\caption{}\label{loop-addition-explicit}

\end{figure}

The above definitions of $\psi$ and $\overline{\psi}$ here are implicit. Next we give explicit constructions of $\psi$ and $\overline{\psi}$. Consider the diagram in Figure \ref{loop-addition-explicit}, where $\chi^0$, $\chi^1$ are the morphisms given by proposition \ref{general-general-chi-maps}, $\iota$, $\epsilon$ are the morphisms induced by the apparent circle creation and annihilation. Then $\chi^0 \circ \iota:C(\Gamma_0) \rightarrow C(\Gamma_1)$ and $\epsilon \circ\chi^1:C(\Gamma_1) \rightarrow C(\Gamma_0)$ are both homogeneous morphisms of $\zed_2$-degree $n$ and quantum degree $-n(N-n)+mn$. 

\begin{proposition}\label{decomposing-psi-bar}
$\psi \approx \chi^0 \circ \iota$, $\overline{\psi} \approx \epsilon \circ \chi^1$. Moreover, we have
\begin{eqnarray}
\label{psibar-comp-psi-1} \overline{\psi} \circ \mathfrak{m}(S_{\mu}(\mathbb{B})) \circ \psi & \approx & \left\{%
\begin{array}{ll}
    \id_{C(\Gamma_0)} & \text{if } \mu=\lambda_{n,N-m-n}, \\ 
    0 & \text{if } |\mu| < n(N-m-n),
\end{array}%
\right. \\
\label{psibar-comp-psi} \overline{\psi} \circ \mathfrak{m}(S_{\mu}(\mathbb{Y})) \circ \psi & \approx & \left\{%
\begin{array}{ll}
    \id_{C(\Gamma_0)} & \text{if } \mu=\lambda_{n,N-m-n}, \\ 
    0 & \text{if } |\mu| < n(N-m-n),
\end{array}%
\right.
\end{eqnarray}
where $\mathfrak{m}(\ast)$ is the morphism given by the multiplication by $\ast$, and $|\mu|=\sum_{j=1}^n \mu_j$ for $\mu =(\mu_1\geq\cdots\geq\mu_n)$.
\end{proposition}
\begin{proof}
To prove $\psi \approx \chi^0 \circ \iota$ and $\overline{\psi} \approx \epsilon \circ \chi^1$, we only need to show that $\chi^0 \circ \iota$ and $\epsilon \circ \chi^1$ are not homotopic to $0$. We prove this by showing that
\begin{equation}\label{chi-iota-epsilon-composition}
\epsilon \circ \chi^1 \circ \mathfrak{m}(S_{\mu}(\mathbb{B})) \circ \chi^0 \circ \iota  \approx  \left\{%
\begin{array}{ll}
    \id_{C(\Gamma_0)} & \text{if } \mu=\lambda_{n,N-m-n}, \\ 
    0 & \text{if } |\mu| < n(N-m-n),
\end{array}%
\right.
\end{equation}
which also implies \eqref{psibar-comp-psi-1}.

Note that the lowest non-vanishing quantum grading of $\Hom_\HMF (C(\Gamma_0),C(\Gamma_0))$ is $0$ and, if $|\mu| < n(N-m-n)$, then the quantum degree of $\epsilon \circ \chi^1 \circ \mathfrak{m}(S_{\mu}(\mathbb{B})) \circ \chi^0 \circ \iota$ is negative. This implies that $\epsilon \circ \chi^1 \circ \mathfrak{m}(S_{\mu}(\mathbb{B})) \circ \chi^0 \circ \iota  \simeq 0$ if $|\mu| < n(N-m-n)$. Now consider the case $\mu=\lambda_{n,N-m-n}$. By Proposition \ref{general-general-chi-maps}, We have 
\begin{eqnarray*}
&& \epsilon \circ \chi^1 \circ \mathfrak{m}(S_{\lambda_{n,N-m-n}}(\mathbb{B})) \circ \chi^0 \circ \iota \\
& = & \epsilon \circ \mathfrak{m}(S_{\lambda_{n,N-m-n}}(\mathbb{B})) \circ \chi^1  \circ \chi^0 \circ \iota \\ 
& = & \epsilon \circ \mathfrak{m}(S_{\lambda_{n,N-m-n}}(\mathbb{B})\cdot\sum_{\lambda\in\Lambda_{n,m}} (-1)^{|\lambda|} S_{\lambda'}(\mathbb{X}) S_{\lambda^c}(\mathbb{B})) \circ \iota \\
& = & \sum_{\lambda\in\Lambda_{n,m}} (-1)^{|\lambda|} S_{\lambda'}(\mathbb{X}) \cdot \epsilon \circ \mathfrak{m}(S_{\lambda_{n,N-m-n}}(\mathbb{B})\cdot S_{\lambda^c}(\mathbb{B})) \circ \iota
\end{eqnarray*}
where $\Lambda_{n,m}=\{\mu=(\mu_1\geq\cdots\geq\mu_n) ~|~ \mu_1 \leq m\}$, $\lambda'\in \Lambda_{m,n}$ is the conjugate of $\lambda$, and $\lambda^c$ is the complement of $\lambda$ in $\Lambda_{n,m}$. That is, if $\lambda=(\lambda_1\geq\cdots\geq\lambda_n)\in \Lambda_{n,m}$, then $\lambda^c = (m-\lambda_n\geq\cdots\geq m-\lambda_1)$. By Corollary \ref{iota-epsilon-composition}, we have, for $\lambda\in\Lambda_{n,m}$,
\[
\epsilon \circ \mathfrak{m}(S_{\lambda_{n,N-m-n}}(\mathbb{B})\cdot S_{\lambda^c}(\mathbb{B})) \circ \iota \approx \begin{cases}
\id_{C(\Gamma_0)} & \text{if } \lambda = (0\geq\cdots\geq0), \\
0 & \text{if } \lambda \neq (0\geq\cdots\geq0).
\end{cases}
\]
This completes the proof for \eqref{chi-iota-epsilon-composition}. Thus, we have proved $\psi \approx \chi^0 \circ \iota$, $\overline{\psi} \approx \epsilon \circ \chi^1$ and \eqref{psibar-comp-psi-1}. 

It remains to prove \eqref{psibar-comp-psi}. Note that, as endomorphisms of $C(\Gamma_1)$,
\[
\mathfrak{m}(S_{\mu}(\mathbb{Y})) \simeq \mathfrak{m}(S_{\mu}(\mathbb{B}\cup\mathbb{X})) = \mathfrak{m}(S_{\mu}(\mathbb{B})+F_\mu(\mathbb{B},\mathbb{X})),
\]
where $F_\mu(\mathbb{B},\mathbb{X}) \in \Sym(\mathbb{B}|\mathbb{X})$ and its total degree in $\mathbb{B}$ is strictly less than $2|\mu|$. Then, by \eqref{psibar-comp-psi-1}, we have that, for any partition $\mu$ with $|\mu| \leq n(N-m-n)$,
\[
\overline{\psi} \circ \mathfrak{m}(S_{\mu}(\mathbb{Y})) \circ \psi \simeq \overline{\psi} \circ \mathfrak{m}(S_{\mu}(\mathbb{B}\cup\mathbb{X})) \circ \psi \simeq \overline{\psi} \circ \mathfrak{m}(S_{\mu}(\mathbb{B}) + F_\mu(\mathbb{B},\mathbb{X})) \circ \psi \simeq \overline{\psi} \circ \mathfrak{m}(S_{\mu}(\mathbb{B})) \circ \psi.
\]
So \eqref{psibar-comp-psi} follows from \eqref{psibar-comp-psi-1}.
\end{proof}

\subsection{Saddle move}\label{subsec-saddle-move} Next we define the morphism $\eta$ induced by a saddle move. Unlike the morphisms in the previous subsections, we will not give an explicit formula for $\eta$. Instead, we prove two composition lemmas for $\eta$, which are all we need to know about $\eta$ in this paper.

\begin{figure}[ht]

\setlength{\unitlength}{1pt}

\begin{picture}(360,80)(-180,-15)


\put(-5,35){$\eta$}

\put(-25,30){\vector(1,0){50}}


\qbezier(-140,10)(-120,30)(-140,50)

\put(-140,50){\vector(-1,1){0}}

\qbezier(-100,10)(-120,30)(-100,50)

\put(-100,10){\vector(1,-1){0}}

\multiput(-130,30)(4.5,0){5}{\line(1,0){2}}

\put(-150,50){\small{$\mathbb{X}$}}

\put(-140,30){\tiny{$m$}}

\put(-150,10){\small{$\mathbb{A}$}}

\put(-105,30){\tiny{$m$}}

\put(-95,50){\small{$\mathbb{Y}$}}

\put(-95,10){\small{$\mathbb{B}$}}

\put(-125,-15){$\Gamma_0$}


\qbezier(100,10)(120,30)(140,10)

\put(140,10){\vector(1,-1){0}}

\qbezier(100,50)(120,30)(140,50)

\put(100,50){\vector(-1,1){0}}

\put(90,50){\small{$\mathbb{X}$}}

\put(120,45){\tiny{$m$}}

\put(90,10){\small{$\mathbb{A}$}}

\put(145,50){\small{$\mathbb{Y}$}}

\put(120,13){\tiny{$m$}}

\put(145,10){\small{$\mathbb{B}$}}

\put(115,-15){$\Gamma_1$}

\end{picture}

\caption{}\label{saddle-move-figure}

\end{figure}

\begin{lemma}\label{saddle-hmf}
Let $\Gamma_0$ and $\Gamma_1$ be the MOY graphs in Figure \ref{saddle-move-figure}. Then, as $\zed_2\oplus\zed$-graded vector spaces over $\C$,
\[
\Hom_{HMF}(C(\Gamma_0),C(\Gamma_1)) \cong C(\emptyset) \{\qb{N}{m} q^{2m(N-m)}\} \left\langle m \right\rangle.
\]
In particular, the subspace of $\Hom_{HMF}(C(\Gamma_0),C(\Gamma_1))$ of homogeneous elements of quantum degree $m(N-m)$ is $1$-dimensional.
\end{lemma}
\begin{proof}
Let $\bigcirc_m$ be a circle colored by $m$ with $4$ marked points. By lemmas \ref{bullet}, \ref{row-reverse-signs} and \ref{column-reverse-signs}, one can see that $\Hom(C(\Gamma_0),C(\Gamma_1))\cong C(\bigcirc_m) \{q^{2m(N-m)}\}$. The lemma follows from this and Corollary \ref{circle-dimension}.
\end{proof}

\begin{definition}
Let $\Gamma_0$ and $\Gamma_1$ be the MOY graphs in Figure \ref{saddle-move-figure}. Associate to the saddle move $\Gamma_0 \leadsto \Gamma_1$ a homogeneous morphism
\[
\eta : C(\Gamma_0) \rightarrow C(\Gamma_1)
\]
of quantum degree $m(N-m)$ that is not homotopic to $0$. By Lemma \ref{saddle-hmf}, $\eta$ is well defined up to homotopy and scaling, and $\deg_{\zed_2} \eta=m$. 
\end{definition}

\subsection{The first composition formula}\label{subsec-1st-composition} In this subsection, we prove that the composition in Figure \ref{creation+saddle+figure} gives, up to homotopy and scaling, the identity map of the matrix factorization. Topologically, this means that a pair of canceling $0$- and $1$-handles induce the identity morphism.

\begin{figure}[ht]

\setlength{\unitlength}{1pt}

\begin{picture}(360,70)(-180,-10)
\put(-110,30){\tiny{$m$}}

\put(-100,0){\vector(0,1){60}}

\put(-65,35){$\iota$}

\put(-80,30){\vector(1,0){40}}

\put(-103,-10){$\Gamma$}


\put(-20,30){\tiny{$m$}}

\put(-10,0){\vector(0,1){60}}

\put(12,48){\tiny{$m$}}

\put(15,30){\oval(20,30)}

\put(25,35){\vector(0,1){0}}

\multiput(-10,30)(5,0){3}{\line(1,0){3}}

\put(0,-10){$\Gamma_1$}

\put(45,35){$\eta$}

\put(30,30){\vector(1,0){40}}


\put(90,55){\vector(0,1){5}}

\qbezier(90,55)(90,45)(100,45)

\qbezier(100,45)(110,45)(110,30)

\qbezier(110,30)(110,15)(100,15)

\qbezier(100,15)(90,15)(90,5)

\put(90,0){\line(0,1){5}}

\put(115,30){\tiny{$m$}}

\put(90,-10){$\Gamma$}

\end{picture}

\caption{}\label{creation+saddle+figure}

\end{figure}

\begin{lemma}\label{1st-comp-lemma}
Let $\Gamma_0$ and $\Gamma_1$ be the MOY graphs in Figure \ref{saddle-move-figure}. Denote by $X_j$ the $j$-th elementary symmetric polynomial in $\mathbb{X}$ and so on. Then under the identification
\[
\Hom (C(\Gamma_0),C(\Gamma_1)) \cong C(\Gamma_1)\otimes_{\Sym(\mathbb{X}|\mathbb{Y}|\mathbb{A}|\mathbb{B})} C(\Gamma_0)_\bullet \cong 
\left(%
\begin{array}{cc}
  \ast & X_1-Y_1 \\
  \dots & \dots \\
  \ast & X_m-Y_m \\
  \ast & B_1-A_1 \\
  \dots & \dots \\
  \ast & B_m-A_m \\
  A_1-X_1 & \ast \\
  \dots & \dots \\
  A_m - X_m & \ast \\
  Y_1 - B_1 & \ast \\
  \dots & \dots \\
  Y_m - B_m & \ast
\end{array}%
\right)_{\Sym(\mathbb{X}|\mathbb{Y}|\mathbb{A}|\mathbb{B})},
\]
we have 
\[
\eta \approx \rho + (\sum_{\ve=(\ve_1,\dots\ve_m)\in I^m} (-1)^{\frac{|\ve|(|\ve|-1)}{2}+(m+1)|\ve|+\sum_{j=1}^{m-1}(m-j)\ve_j} 1_{\ve}\otimes 1_{\wbar{\ve}}) \otimes 1_{(\underbrace{1,1,\dots,1}_{2m})},
\]
where $I=\{0,1\}$ and $\rho$ is of the form 
\[
\rho = \sum_{\ve_1,\ve_2\in I^m, ~\ve_3\in I^{2m}, ~\ve_3\neq (1,1,\dots,1)} f_{(\ve_1,\ve_2,\ve_3)} 1_{\ve_1} \otimes 1_{\ve_2} \otimes1_{\ve_3}.
\]
\end{lemma}
\begin{proof}
Write $R_0=\Sym(\mathbb{X}|\mathbb{Y}|\mathbb{A}|\mathbb{B})$, and 
\[
R_k =
\left\{%
\begin{array}{ll}
    R/(A_1-X_1,\dots,A_k-X_k) & \text{if } 1\leq k \leq m, \\ 
    R/(A_1-X_1,\dots,A_m-X_m,Y_1-B_1,\dots,Y_{k-m}-B_{k-m}) & \text{if } m+1 \leq k \leq 2m.
\end{array}%
\right.
\]
Define
\[
M_k = \begin{cases}
\left(%
\begin{array}{cc}
  \ast & X_1-Y_1 \\
  \dots & \dots \\
  \ast & X_m-Y_m \\
  \ast & B_1-A_1 \\
  \dots & \dots \\
  \ast & B_m-A_m \\
  A_{k+1}-X_{k+1} & \ast \\
  \dots & \dots \\
  A_m - X_m & \ast \\
  Y_1 - B_1 & \ast \\
  \dots & \dots \\
  Y_m - B_m & \ast
\end{array}%
\right)_{R_k}  & \text{ if } 0 \leq k \leq m-1, \\
\left(%
\begin{array}{cc}
  \ast & X_1-Y_1 \\
  \dots & \dots \\
  \ast & X_m-Y_m \\
  \ast & B_1-A_1 \\
  \dots & \dots \\
  \ast & B_m-A_m \\
  Y_{k-m+1} - B_{k-m+1} & \ast \\
  \dots & \dots \\
  Y_m - B_m & \ast
\end{array}%
\right)_{R_k}  & \text{ if } m \leq k \leq 2m-, \\
\left(%
\begin{array}{cc}
  \ast & X_1-Y_1 \\
  \dots & \dots \\
  \ast & X_m-Y_m \\
  \ast & B_1-A_1 \\
  \dots & \dots \\
  \ast & B_m-A_m 
\end{array}%
\right)_{R_{2m}}
\cong C(\Gamma) & \text{ if } k=2m,
\end{cases}
\]
where $\Gamma$ is a circle colored by $m$ with two marked points shown in Figure \ref{circle-rep-two-marks-figure-2}. Then $\Hom_{HFM}(C(\Gamma_0),C(\Gamma_1))$ can be computed by the following homotopy
\begin{eqnarray*}
\Hom (C(\Gamma_0),C(\Gamma_1)) & \cong & M_0 \simeq \cdots \simeq M_k\{q^{n_k}\}\left\langle k \right\rangle \simeq \cdots \simeq M_{2m}\{q^{n_{2m}}\}\left\langle 2m \right\rangle \\
& \cong & C(\Gamma) \{q^{2m(N-m)}\} \simeq C(\emptyset) \{\qb{N}{m} q^{2m(N-m)}\} \left\langle m \right\rangle,
\end{eqnarray*}
where $n_k$ can be inductively determined using Corollary \ref{a-contraction-weak}. In particular, $n_{2m} = 2m(N-m)$. Let $\eta_k \in M_k$ be the image of $\eta$ under the above homotopy. Then $\eta_k$ is a cycle and represents, up to scaling, the unique homology class in $H(M_k)$ of quantum degree $m(N-m)- n_k$.

By Lemma \ref{circle-rep-two-marks}, 
\[
\eta_{2m} \approx \sum_{\ve\in I^m} (-1)^{\frac{|\ve|(|\ve|-1)}{2}+(m+1)|\ve|+s(\ve)} 1_{\ve}\otimes 1_{\wbar{\ve}} \in M_{2m},
\]
where $s(\ve) = \sum_{j=1}^{m-1}(m-j)\ve_j$ for $\ve=(\ve_1,\dots\ve_m)\in I^m$. Assume that 
\[
\eta_k \approx \rho_k + (\sum_{\ve\in I^m} (-1)^{\frac{|\ve|(|\ve|-1)}{2}+(m+1)|\ve|+s(\ve)} 1_{\ve}\otimes 1_{\wbar{\ve}}) \otimes 1_{(\underbrace{1,1,\dots,1}_{2m-k})} \in M_k,
\]
where $\rho_k$ is of the form
\[
\rho_k = \sum_{\ve_1,\ve_2\in I^m, ~\ve_3\in I^{2m-k}, ~\ve_3\neq (1,1,\dots,1)} f_{k,(\ve_1,\ve_2,\ve_3)} 1_{\ve_1} \otimes 1_{\ve_2} \otimes1_{\ve_3}.
\]
Note that 
\[
\widetilde{\eta_k} \approx \widetilde{\rho_k} + (\sum_{\ve\in I^m} (-1)^{\frac{|\ve|(|\ve|-1)}{2}+(m+1)|\ve|+s(\ve)} 1_{\ve}\otimes 1_{\wbar{\ve}}) \otimes 1_{(\underbrace{1,1,\dots,1}_{2m-k+1})} 
\]
is a chain in $M_{k-1}$ mapped to $\eta_k$ under the homotopy 
\[
M_{k-1}\{q^{n_{k-1}}\}\left\langle k-1 \right\rangle \xrightarrow{\simeq} M_k\{q^{n_{k}}\}\left\langle k \right\rangle,
\] 
where
\[
\widetilde{\rho_k} = \sum_{\ve_1,\ve_2\in I^m, ~\ve_3\in I^{2m-k}, ~\ve_3\neq (1,1,\dots,1)} f_{k,(\ve_1,\ve_2,\ve_3)} 1_{\ve_1} \otimes 1_{\ve_2} \otimes1_{(1,\ve_3)}.
\]
Then, by Corollary \ref{a-contraction-weak} and Remark \ref{reverse-b-contraction}, we have that
\begin{eqnarray*}
\eta_{k-1} & \approx & \widetilde{\eta_k} - h \circ d(\widetilde{\eta_k}) \\
& = & \widetilde{\rho_k}- h \circ d(\widetilde{\eta_k}) + (\sum_{\ve\in I^m} (-1)^{\frac{|\ve|(|\ve|-1)}{2}+(m+1)|\ve|+s(\ve)} 1_{\ve}\otimes 1_{\wbar{\ve}}) \otimes 1_{(\underbrace{1,1,\dots,1}_{2m-k+1})}.
\end{eqnarray*}
See the proof of Proposition \ref{b-contraction-weak} for the definition of $h$ and note the slightly different setup here.\footnote{We are eliminating a row here using its left entry rather than the right entry as in Proposition \ref{b-contraction-weak}.} By the definition of $h$, (again, note the difference in the setup,) one can check that $h \circ d(\widetilde{\eta_k})$ is of the form
\[
h \circ d(\widetilde{\eta_k}) = \sum_{\ve_1,\ve_2\in I^m, ~\ve_3\in I^{2m-k}} g_{k,(\ve_1,\ve_2,\ve_3)} 1_{\ve_1} \otimes 1_{\ve_2} \otimes1_{(0,\ve_3)}.
\]
Therefore, $\rho_{k-1}:=\widetilde{\rho_k}- h \circ d(\widetilde{\eta_k})$ is of the form
\[
\rho_{k-1} = \sum_{\ve_1,\ve_2\in I^m, ~\ve_3\in I^{2m-k+1}, ~\ve_3\neq (1,1,\dots,1)} f_{k-1,(\ve_1,\ve_2,\ve_3)} 1_{\ve_1} \otimes 1_{\ve_2} \otimes 1_{\ve_3}.
\]
Thus, we have inductively constructed a $\rho=\rho_0 \in M_0$ of the form 
\[
\rho = \sum_{\ve_1,\ve_2\in I^m, ~\ve_3\in I^{2m}, ~\ve_3\neq (1,1,\dots,1)} f_{(\ve_1,\ve_2,\ve_3)} 1_{\ve_1} \otimes 1_{\ve_2} \otimes1_{\ve_3}
\]
such that 
\[
\eta \approx \rho + (\sum_{\ve\in I^m} (-1)^{\frac{|\ve|(|\ve|-1)}{2}+(m+1)|\ve|+s(\ve)} 1_{\ve}\otimes 1_{\wbar{\ve}}) \otimes 1_{(\underbrace{1,1,\dots,1}_{2m})}.
\]
\end{proof}

\begin{proposition}\label{creation+compose+saddle}
Let $\Gamma$ and $\Gamma_1$ be the MOY graphs in Figure \ref{creation+saddle+figure}, $\iota:C(\Gamma)\rightarrow C(\Gamma_1)$ the morphism associated to the circle creation and $\eta:C(\Gamma_1)\rightarrow C(\Gamma)$ the morphism associated to the saddle move. Then $\eta\circ \iota \approx \id_{C(\Gamma)}$.
\end{proposition}
\begin{proof}
From the proof of Lemma \ref{circle-rep-two-marks}, we know that 
\[
\Hom_{HMF}(C(\Gamma),C(\Gamma)) \cong C(\emptyset) \{\qb{N}{m}q^{m(N-m)}\}.
\]
In particular, the subspace of $\Hom_{HMF}(C(\Gamma),C(\Gamma))$ of elements of quantum degree $0$ is $1$-dimensional and spanned by $\id_{C(\Gamma)}$. Note that the quantum degree of $\eta\circ \iota$ is $0$. So, to prove that $\eta\circ \iota \approx \id_{C(\Gamma)}$, we only need to show that $\eta\circ \iota$ is not homotopic to $0$. We do so by identifying the two ends of $\Gamma$ and showing that $\eta_\ast \circ \iota_\ast\neq 0$.

\begin{figure}[ht]

\setlength{\unitlength}{1pt}

\begin{picture}(360,70)(-180,-10)
\put(-122,48){\tiny{$m$}}

\put(-120,30){\oval(20,30)}

\put(-110,35){\vector(0,1){0}}

\put(-131,30){\line(1,0){2}}

\put(-140,30){\small{$\mathbb{X}$}}

\put(-75,35){$\iota$}

\put(-90,30){\vector(1,0){40}}

\put(-123,-10){$\widetilde{\Gamma}$}


\put(-23,48){\tiny{$m$}}

\put(-20,30){\oval(20,30)}

\put(-10,35){\vector(0,1){0}}

\put(12,48){\tiny{$m$}}

\put(-31,30){\line(1,0){2}}

\put(-40,30){\small{$\mathbb{X}$}}

\put(15,30){\oval(20,30)}

\put(5,25){\vector(0,-1){0}}

\put(24,30){\line(1,0){2}}

\put(30,30){\small{$\mathbb{Y}$}}

\multiput(-10,30)(5,0){3}{\line(1,0){3}}

\put(0,-10){$\widetilde{\Gamma}_1$}

\put(65,35){$\eta$}

\put(50,30){\vector(1,0){40}}


\put(109,30){\line(1,0){2}}

\put(100,30){\small{$\mathbb{X}$}}

\put(117,48){\tiny{$m$}}

\put(120,30){\oval(20,30)}

\put(130,35){\vector(0,1){0}}

\put(129,30){\line(1,0){2}}

\put(135,30){\small{$\mathbb{Y}$}}

\put(115,-10){$\widetilde{\Gamma}$}

\end{picture}

\caption{}\label{creation+saddle+closeup+figure}

\end{figure}

Identify the two end points in each of the MOY graphs in Figure \ref{creation+saddle+figure} and put markings on them as in Figure \ref{creation+saddle+closeup+figure}. Denote by $\widetilde{\Gamma}$ and $\widetilde{\Gamma}_1$ the resulting MOY graphs. Denote by $\mathfrak{G}$ the generating class of $H(\widetilde{\Gamma})$ and by $\mathfrak{G}_{\mathbb{X}},\mathfrak{G}_{\mathbb{Y}}$ the generating classes of the homology of the two circles in $H(\widetilde{\Gamma}_1)$. Then $\iota_\ast (\mathfrak{G}) \propto \mathfrak{G}_{\mathbb{X}}\otimes\mathfrak{G}_{\mathbb{Y}}$. 

By Lemmas \ref{bullet} and \ref{1st-comp-lemma}, under the identification $\Hom_{\Sym(\mathbb{X}|\mathbb{Y})} (C(\widetilde{\Gamma}_1),C(\widetilde{\Gamma})) \cong C(\widetilde{\Gamma})\otimes_{\Sym(\mathbb{X}|\mathbb{Y})} \Hom_{\Sym(\mathbb{X}|\mathbb{Y})}(C(\widetilde{\Gamma}_1),\Sym(\mathbb{X}|\mathbb{Y}))$, we have
\[
\eta \approx \rho + (\sum_{\ve=(\ve_1,\dots\ve_m)\in I^m} (-1)^{\frac{|\ve|(|\ve|-1)}{2}+(m+1)|\ve|+\sum_{j=1}^{m-1}(m-j)\ve_j} 1_{\ve}\otimes 1_{\wbar{\ve}}) \otimes 1_{(\underbrace{1,1,\dots,1}_{2m})}^\ast,
\] 
where $\rho$ is of the form 
\[
\rho = \sum_{\ve_1,\ve_2\in I^m, ~\ve_3\in I^{2m}, ~\ve_3\neq (1,1,\dots,1)} f_{(\ve_1,\ve_2,\ve_3)} 1_{\ve_1} \otimes 1_{\ve_2} \otimes 1_{\ve_3}^\ast.
\]
Note that:
\begin{itemize}
	\item By Lemma \ref{circle-rep-one-mark}, $1_{(\underbrace{1,1,\dots,1}_{2m})}$ is a cycle in $C(\widetilde{\Gamma}_1)$ representing $\mathfrak{G}_{\mathbb{X}}\otimes\mathfrak{G}_{\mathbb{Y}}$.
	\item By Lemma \ref{circle-rep-two-marks}, $\sum_{\ve=(\ve_1,\dots\ve_m)\in I^m} (-1)^{\frac{|\ve|(|\ve|-1)}{2}+(m+1)|\ve|+\sum_{j=1}^{m-1}(m-j)\ve_j} 1_{\ve}\otimes 1_{\wbar{\ve}}$ is a cycle in $C(\widetilde{\Gamma})$ representing $\mathfrak{G}$.
	\item $\rho(1_{(\underbrace{1,1,\dots,1}_{2m})})=0$.
\end{itemize}
Putting these together, we get $\eta_\ast (\mathfrak{G}_{\mathbb{X}}\otimes\mathfrak{G}_{\mathbb{Y}}) \propto \mathfrak{G}$. Thus, $\eta_\ast \circ \iota_\ast (\mathfrak{G}) \propto \mathfrak{G}$. This shows that $\eta \circ \iota$ is not homotopic to $0$ and, therefore, $\eta\circ \iota \approx \id_{C(\Gamma)}$.
\end{proof}

\begin{remark}
From the proof of Proposition \ref{creation+compose+saddle}, we can see that $\eta$ gives $H(\widetilde{\Gamma})$ a ring structure and $H(\widetilde{\Gamma})\{q^{m(N-m)}\} \cong  H^\ast(G_{m,N};\C)$ as $\zed$-graded $\C$-algebras, where $G_{m,N}$ is the complex $(m,N)$-Grassmannian.
\end{remark}

\subsection{The second composition formula}\label{subsec-2nd-composition} In this subsection, we show that the composition in Figure \ref{saddle+annihilation+figure} also gives, up to homotopy and scaling, the identity map. Topologically, this means that a pair of canceling $1$- and $2$-handles induce the identity morphism. The key to the proof is a good choice of entries in the left columns of the matrix factorizations involved. Our choice is given in the following lemma.

\begin{figure}[ht]

\setlength{\unitlength}{1pt}

\begin{picture}(360,70)(-180,-10)


\put(-110,55){\vector(0,1){5}}

\qbezier(-110,55)(-110,45)(-100,45)

\qbezier(-100,45)(-90,45)(-90,30)

\qbezier(-90,30)(-90,15)(-100,15)

\qbezier(-100,15)(-110,15)(-110,5)

\put(-110,0){\line(0,1){5}}

\multiput(-100,15)(0,5){6}{\line(0,1){3}}

\put(-85,30){\tiny{$m$}}

\put(-55,35){$\eta$}

\put(-70,30){\vector(1,0){40}}

\put(-103,-10){$\Gamma$}


\put(-20,30){\tiny{$m$}}

\put(-10,0){\vector(0,1){60}}

\put(2,48){\tiny{$m$}}

\put(5,30){\oval(20,30)}

\put(45,35){$\varepsilon$}

\put(30,30){\vector(1,0){40}}

\put(-5,-10){$\Gamma_1$}


\put(90,30){\tiny{$m$}}

\put(100,0){\vector(0,1){60}}

\put(97,-10){$\Gamma$}

\end{picture}

\caption{}\label{saddle+annihilation+figure}

\end{figure}

\begin{lemma}\label{fancy-U}
Let $\mathbb{X},\mathbb{Y}$ be disjoint alphabets, each having $m$ ($\leq N$) indeterminates. For $j=1,\dots,m$, define 
\begin{eqnarray*}
U_j(\mathbb{X},\mathbb{Y}) & = & (-1)^{j-1}p_{N+1-j}(\mathbb{Y}) + \sum_{k=1}^{m} (-1)^{k+j}j X_k h_{N+1-k-j}(\mathbb{Y}) \\
&& + \sum_{k=1}^{m}\sum_{l=1}^{m}(-1)^{k+l} l X_k X_l \xi_{N+1-k-l,j}(\mathbb{X},\mathbb{Y}),
\end{eqnarray*}
where $X_j$ and $Y_j$ are the $j$-th elementary symmetric polynomials in $\mathbb{X}$ and $\mathbb{Y}$, and 
\[
\xi_{n,j}(\mathbb{X},\mathbb{Y}) = \frac{h_{m,n}(Y_1,\dots,Y_{j-1},X_j,\dots,X_m)- h_{m,n}(Y_1,\dots,Y_j,X_{j+1},\dots,X_m)}{X_j-Y_j}.
\]
Then $U_j(\mathbb{X},\mathbb{Y})$ is homogeneous of degree $2(N+1-j)$ and 
\[
\sum_{j=1}^{m}(X_j-Y_j)U_j(\mathbb{X},\mathbb{Y}) = p_{N+1}(\mathbb{X}) - p_{N+1}(\mathbb{Y}).
\]
\end{lemma}
\begin{proof}
The claims about the homogeneity and degree of $U_j(\mathbb{X},\mathbb{Y})$ are easy to verify and left to the reader. We only prove the last equation. Since $N\geq m$, by Newton's Identity \eqref{newton}, we have that
\begin{eqnarray*}
& & p_{N+1}(\mathbb{X}) - p_{N+1}(\mathbb{Y}) \\
& = & \sum_{k=1}^m (-1)^{k-1}(X_k p_{N+1-k}(\mathbb{X}) - Y_k p_{N+1-k}(\mathbb{Y})) \\
& = & \sum_{k=1}^m (-1)^{k-1}(X_k - Y_k) p_{N+1-k}(\mathbb{Y}) + \sum_{k=1}^m (-1)^{k-1}X_k (p_{N+1-k}(\mathbb{X}) - p_{N+1-k}(\mathbb{Y})).
\end{eqnarray*}
By \eqref{complete-power},
\begin{eqnarray*}
& & p_{N+1-k}(\mathbb{X}) - p_{N+1-k}(\mathbb{Y}) \\
& = & \sum_{l=1}^m (-1)^{l-1} l (X_lh_{N+1-k-l}(\mathbb{X}) - Y_lh_{N+1-k-l}(\mathbb{Y})) \\
& = & \sum_{l=1}^m (-1)^{l-1} l (X_l - Y_l) h_{N+1-k-l}(\mathbb{Y}) + \sum_{l=1}^m (-1)^{l-1} l X_l(h_{N+1-k-l}(\mathbb{X}) - h_{N+1-k-l}(\mathbb{Y})) \\
& = & \sum_{l=1}^m (-1)^{l-1} l (X_l - Y_l) h_{N+1-k-l}(\mathbb{Y}) + \sum_{l=1}^m (-1)^{l-1} l X_l \sum_{j=1}^m \xi_{N+1-k-l,j}(\mathbb{X},\mathbb{Y}) (X_j-Y_j).
\end{eqnarray*}
Substituting this back into the first equation, we get
\begin{eqnarray*}
& & p_{N+1}(\mathbb{X}) - p_{N+1}(\mathbb{Y}) \\
& = & \sum_{k=1}^m (-1)^{k-1}(X_k - Y_k) p_{N+1-k}(\mathbb{Y}) + \sum_{k=1}^m (-1)^{k-1}X_k \sum_{l=1}^m (-1)^{l-1} l (X_l - Y_l) h_{N+1-k-l}(\mathbb{Y}) \\
& & + \sum_{k=1}^m (-1)^{k-1}X_k \sum_{l=1}^m (-1)^{l-1} l X_l \sum_{j=1}^m \xi_{N+1-k-l,j}(\mathbb{X},\mathbb{Y})(X_j-Y_j) \\
& = & \sum_{j=1}^{m}(X_j-Y_j)U_j(\mathbb{X},\mathbb{Y}).
\end{eqnarray*}
\end{proof}

In the rest of this subsection, we use heavily the notations introduced in Definition \ref{ve-notation}. The next lemma is a special case of Remark \ref{reverse-b-contraction}.

\begin{lemma}\label{reverse-b-contraction-special}
Let $R$ be a graded commutative unital $\C$-algebra, and $X$ an homogeneous indeterminate over $R$. Assume that $f_{1,0}(X),f_{1,1}(X),\dots,f_{k,0}(X),f_{k,1}(X)$ are homogeneous elements in $R[X]$ such that 
\begin{eqnarray*}
\deg f_{j,0}(X) + \deg f_{j,1}(X) & = & 2N+2, \\
\sum_{j=1}^k f_{j,0}(X) f_{j,1}(X) & = & 0.
\end{eqnarray*}
Suppose that $f_{1,1}(X)=X-A$, where $A\in R$ is a homogeneous element of degree $\deg A= \deg X$.
Define
\[
M = \left(%
\begin{array}{cc}
  f_{1,0}(X) & f_{1,1}(X) \\
  f_{2,0}(X) & f_{2,1}(X) \\
  \dots & \dots \\
  f_{k,0}(X) & f_{k,1}(X) 
\end{array}%
\right)_{R[X]} 
\text{ and }
M' = \left(%
\begin{array}{cc}
  f_{2,0}(A) & f_{2,1}(A) \\
  f_{3,0}(A) & f_{3,1}(A) \\
  \dots & \dots \\
  f_{k,0}(A) & f_{k,1}(A) 
\end{array}%
\right)_{R}.
\] 
Then $M$ and $M'$ are homotopic graded chain complexes over $R$. 

Let $F:M \rightarrow M'$ be the quasi-isomorphism from the proof of Proposition \ref{b-contraction-weak}. If 
\[
\alpha = \sum_{\ve \in I^{k-1}} a_\ve 1_\ve 
\] 
is a cycle in $M'$, where $a_\ve \in R$, then 
\[
\widetilde{\alpha} = \sum_{\ve \in I^{k-1}} a_\ve 1_{(0,\ve)} - \sum_{\ve =(\ve_2,\dots,\ve_k)\in I^{k-1}} a_\ve (\sum_{j=2}^k (-1)^{|(0,\ve)|_j} g_{j,\ve_j}(X) 1_{(1,\ve_2,\dots,\ve_{j-1},\wbar{\ve_j},\ve_{j+1},\dots,\ve_k)})
\]
is a cycle in $M$ and $F(\widetilde{\alpha}) =\alpha$, where $|(0,\ve)|_j = \sum_{l=2}^{j-1}\ve_l$ and $g_{j,\ve_j}(X) = \frac{f_{j,\ve_j}(X) - f_{j,\ve_j}(A)}{X-A}$.
\end{lemma}
\begin{proof}
Let $\beta = \sum_{\ve \in I^{k-1}} a_\ve 1_{(0,\ve)}\in M$. Then $F(\beta)=\alpha$. By Remark \ref{reverse-b-contraction}, we know that $d(\beta) \in \ker F$, $\beta - h \circ d(\beta)$ is a cycle in $M$ and $F(\beta - h \circ d(\beta))=\alpha$, where $h:\ker F \rightarrow \ker F$ is defined in the proof of Proposition \ref{b-contraction-weak}. But
\begin{eqnarray*}
h \circ d(\beta) & = & h \circ d (\sum_{\ve \in I^{k-1}} a_\ve 1_{(0,\ve)}) \\
& = & h(\sum_{\ve=(\ve_2,\dots,\ve_k) \in I^{k-1}} a_\ve (f_{1,0}(X)1_{(1,\ve)}+ \sum_{j=2}^k (-1)^{|(0,\ve)|_j} f_{j,\ve_j}(X)1_{(0,\ve_2,\dots,\ve_{j-1},\wbar{\ve_j},\ve_{j+1},\dots,\ve_k)})).
\end{eqnarray*}
By the definition of $h$, we know that $h(1_{(1,\ve)})=0$. Moreover, since $\alpha$ is a cycle in $M'$, we have
\[
0=d\alpha = \sum_{\ve=(\ve_2,\dots,\ve_k) \in I^{k-1}} a_\ve (\sum_{j=2}^k (-1)^{|(0,\ve)|_j} f_{j,\ve_j}(A)1_{(\ve_2,\dots,\ve_{j-1},\wbar{\ve_j},\ve_{j+1},\dots,\ve_k)}).
\] 
So, in $M$, we have
\[
0= \sum_{\ve=(\ve_2,\dots,\ve_k) \in I^{k-1}} a_\ve (\sum_{j=2}^k (-1)^{|(0,\ve)|_j} f_{j,\ve_j}(A)1_{(0,\ve_2,\dots,\ve_{j-1},\wbar{\ve_j},\ve_{j+1},\dots,\ve_k)}).
\]
Thus,
\begin{eqnarray*}
h \circ d(\beta) & = & h(\sum_{\ve=(\ve_2,\dots,\ve_k) \in I^{k-1}} a_\ve (\sum_{j=2}^k (-1)^{|(0,\ve)|_j} (f_{j,\ve_j}(X)-f_{j,\ve_j}(A))1_{(0,\ve_2,\dots,\ve_{j-1},\wbar{\ve_j},\ve_{j+1},\dots,\ve_k)})) \\
& = & \sum_{\ve =(\ve_2,\dots,\ve_k)\in I^{k-1}} a_\ve (\sum_{j=2}^k (-1)^{|(0,\ve)|_j} g_{j,\ve_j}(X) 1_{(1,\ve_2,\dots,\ve_{j-1},\wbar{\ve_j},\ve_{j+1},\dots,\ve_k)}),
\end{eqnarray*}
where the last equation comes from the definition of $h$. This shows that $\beta - h \circ d(\beta) = \widetilde{\alpha}$ and proves the lemma.
\end{proof}

\begin{figure}[ht]

\setlength{\unitlength}{1pt}

\begin{picture}(360,80)(-180,-15)


\put(-5,35){$\eta$}

\put(-25,30){\vector(1,0){50}}


\qbezier(-140,10)(-120,30)(-100,10)

\put(-140,50){\vector(-1,1){0}}

\qbezier(-140,50)(-120,30)(-100,50)

\put(-100,10){\vector(1,-1){0}}

\multiput(-120,20)(0,4.5){5}{\line(0,1){2}}

\put(-150,50){\small{$\mathbb{X}$}}

\put(-150,10){\small{$\mathbb{A}$}}

\put(-95,50){\small{$\mathbb{Y}$}}

\put(-95,10){\small{$\mathbb{B}$}}

\put(-120,45){\tiny{$m$}}

\put(-120,13){\tiny{$m$}}

\put(-125,-15){$\Gamma_0$}


\qbezier(100,10)(120,30)(100,50)

\put(140,10){\vector(1,-1){0}}

\qbezier(140,10)(120,30)(140,50)

\put(100,50){\vector(-1,1){0}}

\put(90,50){\small{$\mathbb{X}$}}

\put(90,10){\small{$\mathbb{A}$}}

\put(145,50){\small{$\mathbb{Y}$}}

\put(145,10){\small{$\mathbb{B}$}}

\put(135,30){\tiny{$m$}}

\put(100,30){\tiny{$m$}}

\put(115,-15){$\Gamma_1$}

\end{picture}

\caption{}\label{saddle-move-figure2}

\end{figure}

Let $\Gamma_0$ and $\Gamma_1$ be the MOY graphs in Figure \ref{saddle-move-figure2}. And $\eta:C(\Gamma_0)\rightarrow C(\Gamma_1)$ the morphism induced by the saddle move. We have
\[
\Hom(C(\Gamma_0),C(\Gamma_1)) \cong C(\Gamma) \otimes_{\Sym(\mathbb{X}|\mathbb{Y}|\mathbb{A}|\mathbb{B})} C(\Gamma_0)_\bullet \cong 
\left(%
\begin{array}{cc}
  V_1(\mathbb{A},\mathbb{X}) & X_1-A_1 \\
  \dots & \dots \\
  V_m(\mathbb{A},\mathbb{X}) & X_m-A_m \\
  V_1(\mathbb{B},\mathbb{Y}) & B_1-Y_1 \\
  \dots & \dots \\
  V_m(\mathbb{B},\mathbb{Y}) & B_m-Y_m \\
  A_1 - B_1 & U_1(\mathbb{A},\mathbb{B}) \\
  \dots & \dots \\
  A_m - B_m & U_m(\mathbb{A},\mathbb{B}) \\
  Y_1-X_1 & U_1(\mathbb{X},\mathbb{Y}) \\
  \dots & \dots \\
  Y_m - X_m & U_m(\mathbb{X},\mathbb{Y}) 
\end{array}%
\right)_{\Sym(\mathbb{X}|\mathbb{Y}|\mathbb{A}|\mathbb{B})},
\]
where $X_j$ is the $j$-th elementary symmetric polynomial in $\mathbb{X}$, $U_j$ is given by Lemma \ref{fancy-U}, and
\[
V_j(\mathbb{X},\mathbb{Y}) := \frac{p_{m,N+1}(Y_1,\dots,Y_{j-1},X_j,\dots,X_m)-p_{m,N+1}(Y_1,\dots,Y_j,X_{j+1},\dots,X_m)}{X_j-Y_j}.
\]
By definition, it is easy to see that
\begin{eqnarray}
\label{pd-triangle-1} \frac{\partial}{\partial X_k} V_j(\mathbb{X},\mathbb{Y}) & = &  0 \text { if } j>k, \\
\label{pd-triangle-2} \frac{\partial}{\partial Y_k} V_j(\mathbb{X},\mathbb{Y}) & = &  0 \text { if } j<k.
\end{eqnarray}

Set $R_0=\Sym(\mathbb{X}|\mathbb{Y}|\mathbb{A}|\mathbb{B}) = \C[X_1,\dots,X_m,Y_1,\dots,Y_m,A_1,\dots,A_m,B_1,\dots,B_m]$, and, for $1\leq k \leq m$,
\begin{eqnarray*}
R_k & = & R_0 / (X_1-A_1,\dots,X_k-A_k) \\
& \cong & \C[X_1,\dots,X_m,Y_1,\dots,Y_m,A_{k+1},\dots,A_m,B_1,\dots,B_m], \\
R_{m+k} & = & R_0 / (X_1-A_1,\dots,X_m-A_m,B_1-Y_1,\dots,B_k-Y_k) \\
& \cong & \C[X_1,\dots,X_m,Y_1,\dots,Y_m,B_{k+1},\dots,B_m].
\end{eqnarray*}
Define
\begin{eqnarray*}
M_k & = & \left(%
\begin{array}{cc}
  V_{k+1}(\mathbb{A},\mathbb{X}) & X_{k+1}-A_{k+1} \\
  \dots & \dots \\
  V_m(\mathbb{A},\mathbb{X}) & X_m-A_m \\
  V_1(\mathbb{B},\mathbb{Y}) & B_1-Y_1 \\
  \dots & \dots \\
  V_m(\mathbb{B},\mathbb{Y}) & B_m-Y_m \\
  A_1 - B_1 & U_1(\mathbb{A},\mathbb{B}) \\
  \dots & \dots \\
  A_m - B_m & U_m(\mathbb{A},\mathbb{B}) \\
  Y_1-X_1 & U_1(\mathbb{X},\mathbb{Y}) \\
  \dots & \dots \\
  Y_m - X_m & U_m(\mathbb{X},\mathbb{Y}) 
\end{array}%
\right)_{R_k} \text{ for } k=0,1,\dots,m-1, \\
M_{m+k} & = &  \left(%
\begin{array}{cc}
  V_{k+1}(\mathbb{B},\mathbb{Y}) & B_{k+1}-Y_{k+1} \\
  \dots & \dots \\
  V_m(\mathbb{B},\mathbb{Y}) & B_m-Y_m \\
  A_1 - B_1 & U_1(\mathbb{A},\mathbb{B}) \\
  \dots & \dots \\
  A_m - B_m & U_m(\mathbb{A},\mathbb{B}) \\
  Y_1-X_1 & U_1(\mathbb{X},\mathbb{Y}) \\
  \dots & \dots \\
  Y_m - X_m & U_m(\mathbb{X},\mathbb{Y}) 
\end{array}%
\right)_{R_{m+k}} \text{ for } k=0,1,\dots,m-1, \\
M_{2m} & = & 
\left(%
\begin{array}{cc}
  X_1 - Y_1 & U_1(\mathbb{X},\mathbb{Y}) \\
  \dots & \dots \\
  X_m - Y_m & U_m(\mathbb{X},\mathbb{Y}) \\
  Y_1-X_1 & U_1(\mathbb{X},\mathbb{Y}) \\
  \dots & \dots \\
  Y_m - X_m & U_m(\mathbb{X},\mathbb{Y})
\end{array}%
\right)_{\Sym(\mathbb{X}|\mathbb{Y})}.
\end{eqnarray*}
By Proposition \ref{b-contraction}, $M_0 \simeq M_1 \simeq \cdots \simeq M_{2m}$. Let $\eta_k$ be the image of $\eta$ in $M_{k}$. Then, use method in the proof of Lemma \ref{circle-rep-two-marks}, one can check that
\[
\eta_{2m} \approx \sum_{\ve\in I^m} (-1)^{\frac{|\wbar{\ve}|(|\wbar{\ve}|-1)}{2}+|\wbar{\ve}| +s(\wbar{\ve})} 1_{\ve}\otimes 1_{\wbar{\ve}},
\]
where $s(\ve):=\sum_{j=1}^{m-1}(m-j)\ve_j$ for $\ve=(\ve_1,\dots,\ve_m)\in I^m$. 

Next, we apply Lemma \ref{reverse-b-contraction-special} to find a cycle representing $\eta$ in $M_0$. 

Write
\begin{eqnarray*}
\theta_{j,0}(X_1,\dots,X_m,Y_1,\dots,Y_m) & = & X_j-Y_j, \\
\theta_{j,1}(X_1,\dots,X_m,Y_1,\dots,Y_m) & = & U_j(\mathbb{X},\mathbb{Y}).
\end{eqnarray*}
And define, for $k=1,\dots,m$, $\ve\in \zed_2$,
\begin{eqnarray*}
\Theta_{j,\ve}^k  & = & \frac{\theta_{j,\ve}(X_1,\dots,X_{k-1},A_{k}\dots,A_m,B_1,\dots,B_m) - \theta_{j,\ve}(X_1,\dots,X_{k},A_{k+1}\dots,A_m,B_1,\dots,B_m)}{X_k-A_k}, \\
\Theta_{j,\ve}^{m+k}  & = & \frac{\theta_{j,\ve}(X_1,\dots,X_m,Y_1\dots,Y_{k-1},B_{k},\dots,B_m) - \theta_{j,\ve}(X_1,\dots,X_m,Y_1\dots,Y_k,B_{k+1},\dots,B_m)}{B_k-Y_k}. \\
\end{eqnarray*}
It is easy to see that, for $1\leq k,j \leq m$, 
\begin{equation}\label{Theta-0}
\Theta_{j,0}^k=\Theta_{j,0}^{m+k} =
\left\{%
\begin{array}{ll}
    -1 & \text{if } j=k, \\ 
    0 & \text{if } j\neq k.
\end{array}%
\right.
\end{equation}

In the following computation, we shall call an element of $M_0$ an ``irrelevant term" if it is of the form $c\cdot 1_{\ve_1} \otimes 1_{\ve_2} \otimes 1_{\ve_3}$ where $c \in R_0$, $\ve_1 \in I^{2m}$ and $\ve_2,\ve_3 \in I^{m}$ such that either $\ve_1 \neq (1,1,\dots,1)$ or $\ve_2 \neq \wbar{\ve_3}$.

Define $\mathcal{F}$ to be the set of functions from $\{1,2,\dots,2m\}$ to $\{1,2,\dots,m\}$ and 
\begin{eqnarray*}
\mathcal{F}_{even} & = & \{f\in \mathcal{F} ~|~ \# f^{-1}(j) \text{ is even for } j=1,2,\dots,m\}, \\
\mathcal{F}_2 & = & \{f\in \mathcal{F} ~|~ \# f^{-1}(j)=2 \text{ for } j=1,2,\dots,m\}.
\end{eqnarray*}
For $f\in \mathcal{F}$, $k=1,2,\dots,2m$, define 
\begin{eqnarray*}
\nu_{f,k} & = &  \# \{k'~|~ k<k'\leq 2m,~f(k')<f(k)\}, \\
\nu_f & = & \sum_{k=1}^{2m} \nu_{f,k}, \\
\mu_{f,k} & = & \# \{k'~|~ k<k'\leq 2m,~f(k')=f(k)\}.
\end{eqnarray*}
For $f\in \mathcal{F}$, $\ve=(\ve_1,\dots,\ve_m) \in I^m$, define $\varphi_f(\ve) = (e_1,\dots,e_m)\in I^m$, where $e_j \in I$ satisfies
\[
e_j \equiv \ve_j + \# \{k ~|~ 1\leq k \leq 2m,~f(k)=j\} \mod 2.
\]

Applying Lemma \ref{reverse-b-contraction-special} repeatedly, we get that
\begin{eqnarray*}
& & \eta_0 \\
& \approx & \sum_{\ve\in I^m}(-1)^{\frac{|\wbar{\ve}|(|\wbar{\ve}|-1)}{2}+|\wbar{\ve}| +s(\wbar{\ve})+2m} \sum_{f \in \mathcal{F}} (\prod_{k=1}^{2m} (-1)^{|\ve|_{f(k)}+\nu_{f,k}} \Theta_{f(k),\ve_{f(k)}+\mu_{f,k}}^k) 1_{(1,\dots,1)}\otimes 1_{\varphi_f(\ve)}\otimes 1_{\wbar{\ve}} \\
& & + \text{ irrelevant terms},
\end{eqnarray*}
where $\ve_j$ is the $j$-the entry in $\ve$. Note that, if $f \notin \mathcal{F}_{even}$, then $\varphi_f(\ve) \neq \ve$ and the corresponding term in the above sum is also irrelevant. So we can simplify the above and get
\begin{eqnarray*}
\eta_0 & \approx & \sum_{\ve\in I^m}(-1)^{\frac{|\wbar{\ve}|(|\wbar{\ve}|-1)}{2}+|\wbar{\ve}| +s(\wbar{\ve})} \sum_{f \in \mathcal{F}_{even}} (-1)^{\nu_{f}} (\prod_{k=1}^{2m} \Theta_{f(k),\ve_{f(k)}+\mu_{f,k}}^k) 1_{(1,\dots,1)}\otimes 1_{\ve}\otimes 1_{\wbar{\ve}} \\
& & + \text{ irrelevant terms}.
\end{eqnarray*}

\begin{figure}[ht]

\setlength{\unitlength}{1pt}

\begin{picture}(360,70)(-180,-10)
\put(-122,48){\tiny{$m$}}

\put(-120,30){\oval(20,30)}

\multiput(-120,15)(0,5){6}{\line(0,1){3}}

\put(-110,35){\vector(0,1){0}}

\put(-131,30){\line(1,0){2}}

\put(-140,30){\small{$\mathbb{X}$}}

\put(-111,30){\line(1,0){2}}

\put(-105,30){\small{$\mathbb{Y}$}}

\put(-75,35){$\eta$}

\put(-90,30){\vector(1,0){40}}

\put(-123,-10){$\widetilde{\Gamma}$}


\put(-23,48){\tiny{$m$}}

\put(-20,30){\oval(20,30)}

\put(-10,35){\vector(0,1){0}}

\put(12,48){\tiny{$m$}}

\put(-31,30){\line(1,0){2}}

\put(-40,30){\small{$\mathbb{X}$}}

\put(15,30){\oval(20,30)}

\put(5,25){\vector(0,-1){0}}

\put(24,30){\line(1,0){2}}

\put(30,30){\small{$\mathbb{Y}$}}

\put(0,-10){$\widetilde{\Gamma}_1$}

\put(65,35){$\epsilon$}

\put(50,30){\vector(1,0){40}}


\put(109,30){\line(1,0){2}}

\put(100,30){\small{$\mathbb{X}$}}

\put(117,48){\tiny{$m$}}

\put(120,30){\oval(20,30)}

\put(130,35){\vector(0,1){0}}

\put(115,-10){$\widetilde{\Gamma}$}

\end{picture}

\caption{}\label{saddle+annihilation+closeup+figure}

\end{figure}

In Figure \ref{saddle-move-figure2}, identify the two end points of $\Gamma_0$ marked by $\mathbb{X}$ and $\mathbb{A}$, and identify the two end points of $\Gamma_0$ marked by $\mathbb{Y}$ and $\mathbb{B}$. This changes $\Gamma_0$ into $\widetilde{\Gamma}$ in Figure \ref{saddle+annihilation+closeup+figure}. Similarly, by identifying the two end points of $\Gamma_1$ in Figure \ref{saddle-move-figure2} marked by $\mathbb{X}$ and $\mathbb{A}$ and identifying the two end points of $\Gamma_1$ marked by $\mathbb{Y}$ and $\mathbb{B}$, we change $\Gamma_1$ into $\widetilde{\Gamma}_1$ in Figure \ref{saddle+annihilation+closeup+figure}. Let $\mathfrak{G}$ be the generating class of $H(\widetilde{\Gamma})$, and $\mathfrak{G}_{\mathbb{X}}$, $\mathfrak{G}_{\mathbb{Y}}$ the generating classes of the homology of the two circles in $\widetilde{\Gamma}_1$. By Lemma \ref{circle-rep-two-marks}, $\mathfrak{G}$ is represented in 
\[
C(\widetilde{\Gamma}) =
\left(%
\begin{array}{cc}
  U_1(\mathbb{X},\mathbb{Y}) & Y_1-X_1 \\
  \dots & \dots \\
  U_m(\mathbb{X},\mathbb{Y}) & Y_m-X_m \\
  U_1(\mathbb{X},\mathbb{Y}) & X_1-Y_1 \\
  \dots & \dots \\
  U_m(\mathbb{X},\mathbb{Y}) & X_m-Y_m 
\end{array}%
\right)_{\Sym(\mathbb{X}|\mathbb{Y})}
\]
by the cycle
\[
G = \sum_{\ve\in I^m} (-1)^{\frac{|\ve|(|\ve|-1)}{2}+(m+1)|\ve|+s(\ve)} 1_{\ve}\otimes 1_{\wbar{\ve}}.
\] 

Define $\widetilde{\Theta}_{j,0}^k,~\widetilde{\Theta}_{j,0}^{m+k},~\widetilde{\Theta}_{j,1}^k,~\widetilde{\Theta}_{j,1}^{m+k}$ by substituting $A_1=X_1,\dots, ~A_m=X_m, ~B_1=Y_1,\dots,~B_m=Y_m$ into $\Theta_{j,0}^k,~\Theta_{j,0}^{m+k},~\Theta_{j,1}^k,~\Theta_{j,1}^{m+k}$. Then, for $1\leq k,j \leq m$,
\begin{eqnarray}
\label{Theta-pd-0}\widetilde{\Theta}_{j,0}^k & = & \widetilde{\Theta}_{j,0}^{m+k} =
\left\{%
\begin{array}{ll}
    -1 & \text{if } j=k, \\ 
    0 & \text{if } j\neq k.
\end{array}%
\right. \\
\label{Theta-pd-1} \widetilde{\Theta}_{j,1}^k & := & \Theta_{j,1}^k|_{A_1=X_1,\dots,A_m=X_m,B_1=Y_1,\dots,B_m=Y_m} = -\frac{\partial}{\partial X_k} U_j(\mathbb{X},\mathbb{Y}), \\
\label{Theta-pd-2} \widetilde{\Theta}_{j,1}^{m+k} & := & \Theta_{j,1}^{m+k}|_{A_1=X_1,\dots,A_m=X_m,B_1=Y_1,\dots,B_m=Y_m} = \frac{\partial}{\partial Y_k} U_j(\mathbb{X},\mathbb{Y}).
\end{eqnarray}

Using the formula for $\eta_0$ and lemmas \ref{bullet}, \ref{row-reverse-signs}, we get that $\eta(G)$ is represented in 
\[
C(\widetilde{\Gamma}_1) = 
\left(%
\begin{array}{cc}
  V_1(\mathbb{X},\mathbb{X}) & 0 \\
  \dots & \dots \\
  V_m(\mathbb{X},\mathbb{X}) & 0 \\
  V_1(\mathbb{Y},\mathbb{Y}) & 0 \\
  \dots & \dots \\
  V_m(\mathbb{Y},\mathbb{Y}) & 0 \\
\end{array}%
\right)_{\Sym(\mathbb{X}|\mathbb{Y})}
\]
by the cycle
\begin{eqnarray*}
&& \eta(G) \\
& \approx & \sum_{\ve\in I^m}(-1)^{\frac{|\wbar{\ve}|(|\wbar{\ve}|-1)}{2}+|\wbar{\ve}| +s(\wbar{\ve}) + \frac{|\ve|(|\ve|-1)}{2}+(m+1)|\ve|+s(\ve)+\frac{m(m-1)}{2}}\sum_{f \in \mathcal{F}_{even}} (-1)^{\nu_{f}} (\prod_{k=1}^{2m} \widetilde{\Theta}_{f(k),\ve_{f(k)}+\mu_{f,k}}^k) 1_{(1,\dots,1)} \\
& & + \text{ irrelevant terms},
\end{eqnarray*}
where the ``irrelevant terms" are terms not of the form $c \cdot 1_{(1,\dots,1)}$. By definition, it is easy to see that 
\begin{eqnarray*}
|\ve| + |\wbar{\ve}| & = & m, \\
s(\ve) + s(\wbar{\ve}) & = & \sum_{j=1}^{m-1}(m-j) =\frac{m(m-1)}{2}.
\end{eqnarray*}
Then one can check that
\begin{eqnarray*}
& & \frac{|\wbar{\ve}|(|\wbar{\ve}|-1)}{2}+|\wbar{\ve}| +s(\wbar{\ve}) + \frac{|\ve|(|\ve|-1)}{2}+(m+1)|\ve|+s(\ve)+\frac{m(m-1)}{2} \\
& \equiv & |\ve|^2 + \frac{m(m+1)}{2} \equiv |\ve| + \frac{m(m+1)}{2} \mod 2.
\end{eqnarray*}
Therefore,
\begin{eqnarray*}
\eta(G) & \approx & (-1)^{\frac{m(m+1)}{2}} \sum_{\ve\in I^m}\sum_{f \in \mathcal{F}_{even}} (-1)^{|\ve|+\nu_{f}} (\prod_{k=1}^{2m} \widetilde{\Theta}_{f(k),\ve_{f(k)}+\mu_{f,k}}^k) 1_{(1,\dots,1)} \\
& & + \text{ irrelevant terms}.
\end{eqnarray*}
This shows that 
\[
\eta_\ast(\mathfrak{G})  \propto  (-1)^{\frac{m(m+1)}{2}} \sum_{\ve\in I^m}\sum_{f \in \mathcal{F}_{even}} (-1)^{|\ve|+\nu_{f}} (\prod_{k=1}^{2m} \widetilde{\Theta}_{f(k),\ve_{f(k)}+\mu_{f,k}}^k) \cdot (\mathfrak{G}_{\mathbb{X}} \otimes \mathfrak{G}_{\mathbb{Y}}).
\]
Hence,
\[
\epsilon_\ast \circ \eta_\ast (\mathfrak{G}) \propto (-1)^{\frac{m(m+1)}{2}} \epsilon_\ast(\sum_{\ve\in I^m}\sum_{f \in \mathcal{F}_{even}} (-1)^{|\ve|+\nu_{f}} (\prod_{k=1}^{2m} \widetilde{\Theta}_{f(k),\ve_{f(k)}+\mu_{f,k}}^k) \cdot \mathfrak{G}_{\mathbb{Y}}) \cdot \mathfrak{G},
\]
where $\epsilon:C(\widetilde{\Gamma}_1)\rightarrow C(\widetilde{\Gamma})$ is the morphism associated to the annihilation of the circle marked by $\mathbb{Y}$. 

Since $\eta$ is homogeneous of degree $m(N-m)$, the polynomial 
\[
\widetilde{\Xi} = \sum_{\ve\in I^m}\sum_{f \in \mathcal{F}_{even}} (-1)^{|\ve|+\nu_{f}} (\prod_{k=1}^{2m} \widetilde{\Theta}_{f(k),\ve_{f(k)}+\mu_{f,k}}^k)
\]
is homogeneous of degree $2m(N-m)$. Let $\widetilde{\Xi}^+$ be the part of $\widetilde{\Xi}$ with positive total degree in $\mathbb{X}$. Then the total degree of $\widetilde{\Xi}^+$ in $\mathbb{Y}$ is less than $2m(N-m)$. By Corollary \ref{iota-epsilon-composition}, we know that $\epsilon_\ast(\widetilde{\Xi}^+ \cdot \mathfrak{G}_{\mathbb{Y}}) =0$. So 
\[
\epsilon_\ast(\widetilde{\Xi} \cdot \mathfrak{G}_{\mathbb{Y}}) = \epsilon_\ast((\widetilde{\Xi}-\widetilde{\Xi}^+) \cdot \mathfrak{G}_{\mathbb{Y}}) = \epsilon_\ast((\widetilde{\Xi}|_{X_1=X_2=\cdots=X_m=0}) \cdot \mathfrak{G}_{\mathbb{Y}}).
\]

Next, consider $\hat{\Xi}:=\widetilde{\Xi}|_{X_1=X_2=\cdots=X_m=0}$. Let $\hat{\Theta}_{j,\ve}^k = \widetilde{\Theta}_{j,\ve}^k|_{X_1=X_2=\cdots=X_m=0}$. Then
\[
\hat{\Xi} =  \sum_{\ve\in I^m}\sum_{f \in \mathcal{F}_{even}} (-1)^{|\ve|+\nu_{f}} (\prod_{k=1}^{2m} \hat{\Theta}_{f(k),\ve_{f(k)}+\mu_{f,k}}^k).
\]
Moreover, by equations \eqref{Theta-pd-0}, \eqref{Theta-pd-1}, \eqref{Theta-pd-2}, the definition of $U_j$ in Lemma \ref{fancy-U} and Lemma \ref{power-derive}, we have, for $1 \leq k,j \leq m$,
\begin{eqnarray*}
\hat{\Theta}_{j,0}^k & = & \hat{\Theta}_{j,0}^{m+k} =
\left\{%
\begin{array}{ll}
    -1 & \text{if } j=k, \\ 
    0 & \text{if } j\neq k,
\end{array}%
\right. \\
 \hat{\Theta}_{j,1}^k & = & -\frac{\partial}{\partial X_k} U_j(\mathbb{X},\mathbb{Y})|_{X_1=X_2=\cdots=X_m=0} = (-1)^{k+j+1}j \cdot h_{N+1-k-j}(\mathbb{Y}), \\
\hat{\Theta}_{j,1}^{m+k} & = & \frac{\partial}{\partial Y_k} U_j(\mathbb{X},\mathbb{Y})|_{X_1=X_2=\cdots=X_m=0} = (-1)^{k+j}(N+1-j) \cdot h_{N+1-k-j}(\mathbb{Y}).
\end{eqnarray*}
Now split $\hat{\Xi}$ into $\hat{\Xi} = \hat{\Xi}_1 + \hat{\Xi}_2$, where
\begin{eqnarray*}
\hat{\Xi}_1 & = & \sum_{\ve\in I^m}\sum_{f \in \mathcal{F}_2} (-1)^{|\ve|+\nu_{f}} (\prod_{k=1}^{2m} \hat{\Theta}_{f(k),\ve_{f(k)}+\mu_{f,k}}^k), \\
\hat{\Xi}_2 & = & \sum_{\ve\in I^m}\sum_{f \in \mathcal{F}_{even}\setminus \mathcal{F}_2} (-1)^{|\ve|+\nu_{f}} (\prod_{k=1}^{2m} \hat{\Theta}_{f(k),\ve_{f(k)}+\mu_{f,k}}^k).
\end{eqnarray*}

We compute $\hat{\Xi}_1$ first. For every pair of $f \in \mathcal{F}_2$ and $\ve=(\ve_1,\dots,\ve_m) \in I^m$, there is a bijection 
\[
f_\ve:\{1,2,\dots,2m\} \rightarrow \{1,2,\dots,m\} \times \zed_2
\] 
given by $f_\ve(k) = (f(k),\ve_{f(k)}+\mu_{f,k})$. Note that $(f,\ve)\mapsto f_\ve$ is a bijection from $\mathcal{F}_2 \times I^m$ to the set of bijections $\{1,2,\dots,2m\} \rightarrow \{1,2,\dots,m\} \times \zed_2$. Define an order on $\{1,2,\dots,m\} \times \zed_2$ by 
\[
(1,1)<(1,0)<(2,1)<(2,0)<\cdots<(m,1)<(m,0).
\]
Then, for $(f,\ve) \in \mathcal{F}_2\times I^m$, 
\[
|\ve|+\nu_{f} = \# \{ (k,k')~|~1\leq k < k' \leq 2m,~f_\ve(k)>f_\ve(k')\}.
\]
Thus,
\begin{eqnarray*}
\hat{\Xi}_1 & = & \sum_{\ve\in I^m}\sum_{f \in \mathcal{F}_2} (-1)^{|\ve|+\nu_{f}} (\prod_{k=1}^{2m} \hat{\Theta}_{f(k),\ve_{f(k)}+\mu_{f,k}}^k) \\
& = & \left|%
\begin{array}{lllllll}
  \hat{\Theta}_{1,1}^1 & \hat{\Theta}_{1,0}^1 & \hat{\Theta}_{2,1}^1 & \hat{\Theta}_{2,0}^1 & \dots & \hat{\Theta}_{m,1}^1 & \hat{\Theta}_{m,0}^1  \\
  \hat{\Theta}_{1,1}^2 & \hat{\Theta}_{1,0}^2 & \hat{\Theta}_{2,1}^2 & \hat{\Theta}_{2,0}^2 & \dots & \hat{\Theta}_{m,1}^2 & \hat{\Theta}_{m,0}^2  \\
  \dots & \dots & \dots & \dots & \dots & \dots & \dots  \\
  \hat{\Theta}_{1,1}^{2m-1} & \hat{\Theta}_{1,0}^{2m-1} & \hat{\Theta}_{2,1}^{2m-1} & \hat{\Theta}_{2,0}^{2m-1} & \dots & \hat{\Theta}_{m,1}^{2m-1} & \hat{\Theta}_{m,0}^{2m-1}  \\
  \hat{\Theta}_{1,1}^{2m} & \hat{\Theta}_{1,0}^{2m} & \hat{\Theta}_{2,1}^{2m} & \hat{\Theta}_{2,0}^{2m} & \dots & \hat{\Theta}_{m,1}^{2m} & \hat{\Theta}_{m,0}^{2m}  \\
\end{array}%
\right| \\
& = & (-1)^{\frac{m(m+1)}{2}} \left|%
\begin{array}{llllll}
  \hat{\Theta}_{1,0}^1 & \dots & \hat{\Theta}_{m,0}^1 & \hat{\Theta}_{1,1}^1 & \dots & \hat{\Theta}_{m,1}^1 \\
  \dots & \dots & \dots & \dots & \dots & \dots \\
  \hat{\Theta}_{1,0}^m & \dots & \hat{\Theta}_{m,0}^m & \hat{\Theta}_{1,1}^m & \dots & \hat{\Theta}_{m,1}^m \\
  \hat{\Theta}_{1,0}^{m+1} & \dots & \hat{\Theta}_{m,0}^{m+1} & \hat{\Theta}_{1,1}^{m+1} & \dots & \hat{\Theta}_{m,1}^{m+1} \\
  \dots & \dots & \dots & \dots & \dots & \dots \\
  \hat{\Theta}_{1,0}^{2m} & \dots & \hat{\Theta}_{m,0}^{2m} & \hat{\Theta}_{1,1}^{2m} & \dots & \hat{\Theta}_{m,1}^{2m} \\
\end{array}%
\right| 
\end{eqnarray*}
Note that both $m\times m$ blocks on the left are $-\mathbf{I}$, where $\mathbf{I}$ is the $m\times m$ unit matrix. So
\begin{eqnarray*}
\hat{\Xi}_1 & = &
(-1)^{\frac{m(m+1)}{2}}\cdot(-1)^m \left|%
\begin{array}{lll}
  \hat{\Theta}_{1,1}^{m+1}-\hat{\Theta}_{1,1}^{1} & \dots & \hat{\Theta}_{m,1}^{m+1}-\hat{\Theta}_{m,1}^{1} \\
  \dots & \dots & \dots \\
  \hat{\Theta}_{1,1}^{2m}-\hat{\Theta}_{1,1}^{m} & \dots & \hat{\Theta}_{m,1}^{2m}-\hat{\Theta}_{m,1}^{m} \\
\end{array}%
\right|  \\
& = & (-1)^{\frac{m(m-1)}{2}} \det ((-1)^{k+j}(N+1)h_{N+1-k-j}(\mathbb{Y}))_{1\leq k,j \leq m} \\
& = & (-1)^{\frac{m(m-1)}{2}} (N+1)^m \det (h_{N+1-k-j}(\mathbb{Y}))_{1\leq k,j \leq m} \\
& = & (N+1)^m \det(h_{N-m-k+j}(\mathbb{Y}))_{1\leq k,j \leq m} \\
& = & (N+1)^m S_{\lambda_{m,N-m}}(\mathbb{Y}),
\end{eqnarray*}
where $\lambda_{m,N-m}= (\underbrace{N-m \geq \cdots \geq N-m}_{m \text{ parts}})$.

The sum $\hat{\Xi}_2$ is harder to understand. But, to determine $\epsilon_\ast(\hat{\Xi}_2 \cdot \mathfrak{G}_{\mathbb{Y}})$, we only need to know the coefficient of $S_{\lambda_{m,N-m}}(\mathbb{Y})$ in the decomposition of $\hat{\Xi}_2$ into Schur polynomials, which is not very hard to do. First, we consider the decomposition of $\hat{\Xi}_2$ into complete symmetric polynomials. Since $\hat{\Xi}_2$ is homogeneous of degree $2m(N-m)$, we have 
\[
\hat{\Xi}_2 = \sum_{|\lambda| = m(N-m), ~l(\lambda)\leq m} c_{\lambda} \cdot h_{\lambda}(\mathbb{Y}), 
\]
where $c_{\lambda} \in \C$. Note that $\hat{\Xi}_2$ is defined by
\[
\hat{\Xi}_2 = \sum_{\ve\in I^m}\sum_{f \in \mathcal{F}_{even}\setminus \mathcal{F}_2} (-1)^{|\ve|+\nu_{f}} (\prod_{k=1}^{2m} \hat{\Theta}_{f(k),\ve_{f(k)}+\mu_{f,k}}^k),
\]
in which every term is a scalar multiple of a complete symmetric polynomial associated to a partition of length $\leq m$. If the term corresponding to $\ve\in I^m$ and $f \in \mathcal{F}_{even}\setminus \mathcal{F}_2$ makes a non-zero contribution to $c_{\lambda_{m,N-m}}$, then we know that, for every $k=1,\dots,m$, 
\[
f(k) = \left\{%
\begin{array}{ll}
    k & \text{if } \ve_{f(k)}+\mu_{f,k}=0, \\ 
    m+1-k & \text{if } \ve_{f(k)}+\mu_{f,k}=1,
\end{array}%
\right.
\] 
and
\[
f(m+k) = \left\{%
\begin{array}{ll}
    k & \text{if } \ve_{f(m+k)}+\mu_{f,m+k}=0, \\ 
    m+1-k & \text{if } \ve_{f(m+k)}+\mu_{f,m+k}=1.
\end{array}%
\right.
\] 
In particular,
\[
f \in \mathcal{F}^{\diamond}:= \{g \in \mathcal{F}_{even}\setminus \mathcal{F}_2 ~|~ g(k),g(m+k) \in \{k,m+1-k\},~\forall~k=1,\dots,m\}.
\]
Now, for an $f\in \mathcal{F}^{\diamond}$, we have $f\in \mathcal{F}_{even}\setminus \mathcal{F}_2$. So there is a $j \in \{1,\dots,m\}$ such that $\# f^{-1}(j)$ is an even number greater than $2$. From the above definition of $\mathcal{F}^{\diamond}$, we can see that $f^{-1}(j) \subset \{j, m+j, m+1-j, 2m+1-j\}$. Thus, $\# f^{-1}(j)=4$ and $f^{-1}(j) = \{j, m+j, m+1-j, 2m+1-j\}$, which implies that $f^{-1}(m+1-j)=\emptyset$. Let $\ve,\sigma \in I^m$ be such that $\ve_{m+1-j} \neq \sigma_{m+1-j}$ and $\ve_l = \sigma_l$ if $l\neq m+1-j$. Then 
\[
(-1)^{|\ve|+\nu_{f}} (\prod_{k=1}^{2m} \hat{\Theta}_{f(k),\ve_{f(k)}+\mu_{f,k}}^k) = -(-1)^{|\sigma|+\nu_{f}} (\prod_{k=1}^{2m} \hat{\Theta}_{f(k),\sigma_{f(k)}+\mu_{f,k}}^k).
\]
This implies that, for every $f \in \mathcal{F}^{\diamond}$,
\[
\sum_{\ve\in I^m} (-1)^{|\ve|+\nu_{f}} (\prod_{k=1}^{2m} \hat{\Theta}_{f(k),\ve_{f(k)}+\mu_{f,k}}^k) =0.
\]
Therefore, $c_{\lambda_{m,N-m}}=0$. By Lemma \ref{special-kostka}, one can see that the coefficient of $S_{\lambda_{m,N-m}}(\mathbb{Y})$ is also $0$ in the decomposition of $\hat{\Xi}_2$ into Schur polynomials. So $\epsilon_\ast(\hat{\Xi}_2 \cdot \mathfrak{G}_{\mathbb{Y}})=0$.

Altogether, we have shown that 
\begin{eqnarray*}
\epsilon_\ast \circ \eta_\ast (\mathfrak{G}) & \propto & (-1)^{\frac{m(m+1)}{2}} \epsilon_\ast(\hat{\Xi} \cdot \mathfrak{G}_{\mathbb{Y}}) \cdot \mathfrak{G} \\
& = & (-1)^{\frac{m(m+1)}{2}} \epsilon_\ast(\hat{\Xi}_1 \cdot \mathfrak{G}_{\mathbb{Y}}) \cdot \mathfrak{G} + (-1)^{\frac{m(m+1)}{2}} \epsilon_\ast(\hat{\Xi}_2 \cdot \mathfrak{G}_{\mathbb{Y}}) \cdot \mathfrak{G} \\
& = & (-1)^{\frac{m(m+1)}{2}} (N+1)^m \epsilon_\ast(S_{\lambda_{m,N-m}}(\mathbb{Y})\cdot\mathfrak{G}_{\mathbb{Y}}) \cdot \mathfrak{G} \\
& \propto & (-1)^{\frac{m(m+1)}{2}} (N+1)^m \mathfrak{G} \neq 0,
\end{eqnarray*}
which proves the following lemma.

\begin{lemma}\label{saddle+compose+annihilation+closeup}
Let $\widetilde{\Gamma}$ and $\widetilde{\Gamma}_1$ be the MOY graphs in Figure \ref{saddle+annihilation+closeup+figure}, $\eta:C(\widetilde{\Gamma})\rightarrow C(\widetilde{\Gamma}_1)$ the morphism associated to the saddle move and $\epsilon:C(\widetilde{\Gamma}_1)\rightarrow C(\widetilde{\Gamma})$ the morphism associated to the annihilation of the circle marked by $\mathbb{Y}$. Then $\epsilon_\ast \circ \eta_\ast (\mathfrak{G}) \propto \mathfrak{G}$, where $\mathfrak{G}$ is the generating class of $H(\widetilde{\Gamma})$. In particular, $\epsilon_\ast \circ \eta_\ast \neq 0$.
\end{lemma}

Now, using an argument similar to the proof of Proposition \ref{creation+compose+saddle}, we can easily prove the following main conclusion of this subsection.

\begin{proposition}\label{saddle+compose+annihilation}
Let $\Gamma$ and $\Gamma_1$ be the MOY graphs in Figure \ref{saddle+annihilation+figure}, $\eta:C(\Gamma)\rightarrow C(\Gamma_1)$ the morphism associated to the saddle move and $\epsilon:C(\Gamma_1)\rightarrow C(\Gamma)$ the morphism associated to circle annihilation. Then $\epsilon \circ \eta \approx \id_{C(\Gamma)}$.
\end{proposition}
\begin{proof}
We know that the subspace of $\Hom_{HMF}(C(\Gamma),C(\Gamma))$ of elements of quantum degree $0$ is $1$-dimensional and spanned by $\id_{C(\Gamma)}$. Note that the quantum degree of $\epsilon \circ \eta$ is $0$. So, to prove that $\epsilon \circ \eta \approx \id_{C(\Gamma)}$, we only need the fact that $\epsilon \circ \eta$ is not homotopic to $0$, which follows from Lemma \ref{saddle+compose+annihilation+closeup}.
\end{proof}

\section{Direct Sum Decomposition (III)}\label{sec-MOY-III}

In this section, we prove Theorem \ref{decomp-III}, which categorifies \cite[Lemma 5.2]{MOY} and generalizes direct sum decomposition (III) in \cite{KR1}.

\begin{figure}[ht]

\setlength{\unitlength}{1pt}

\begin{picture}(360,75)(-180,-15)

\put(-170,0){\vector(1,1){20}}

\put(-150,20){\vector(1,0){20}}

\put(-130,20){\vector(0,1){20}}

\put(-130,20){\vector(1,-1){20}}

\put(-130,40){\vector(-1,0){20}}

\put(-150,40){\vector(0,-1){20}}

\put(-150,40){\vector(-1,1){20}}

\put(-110,60){\vector(-1,-1){20}}

\put(-175,0){\tiny{$1$}}

\put(-175,55){\tiny{$1$}}

\put(-127,28){\tiny{$1$}}

\put(-108,0){\tiny{$m$}}

\put(-108,55){\tiny{$m$}}

\put(-157,28){\tiny{$m$}}

\put(-148.5,43){\tiny{$m+1$}}

\put(-148.5,13){\tiny{$m+1$}}

\put(-142,-15){$\Gamma$}


\put(-20,0){\vector(0,1){60}}

\put(20,60){\vector(0,-1){60}}

\put(-25,28){\tiny{$1$}}

\put(22,28){\tiny{$m$}}

\put(-2,-15){$\Gamma_0$}


\put(110,0){\vector(1,1){20}}

\put(130,20){\vector(1,-1){20}}

\put(130,40){\vector(0,-1){20}}

\put(130,40){\vector(-1,1){20}}

\put(150,60){\vector(-1,-1){20}}

\put(105,0){\tiny{$1$}}

\put(105,55){\tiny{$1$}}

\put(152,0){\tiny{$m$}}

\put(152,55){\tiny{$m$}}

\put(132,28){\tiny{$m-1$}}

\put(128,-15){$\Gamma_1$}

\end{picture}

\caption{}\label{decomposition-III-figure}

\end{figure}

\begin{theorem}\label{decomp-III}
Let $\Gamma$, $\Gamma_0$ and $\Gamma_1$ be the MOY graphs in Figure \ref{decomposition-III-figure}, where $m\leq N-1$. Then 
\[
C(\Gamma) \simeq C(\Gamma_0) \oplus C(\Gamma_1)\{[N-m-1]\} \left\langle 1 \right\rangle.
\]
\end{theorem}

\begin{remark}
Theorem \ref{decomp-III} is not directly used in the proof of the invariance of the colored $\mathfrak{sl}(N)$ homology.
\end{remark}

\subsection{Relating $\Gamma$ and $\Gamma_0$} In this subsection, we generalize the method in \cite[Subsection 3.3]{Wu7} to construct morphisms between $C(\Gamma)$ and $C(\Gamma_0)$. In fact, the result we get is slightly more general than what is needed to prove Theorem \ref{decomp-III}.

\begin{figure}[ht]

\setlength{\unitlength}{1pt}

\begin{picture}(360,75)(-180,-15)

\put(0,15){\vector(0,1){30}}

\qbezier(0,45)(50,60)(50,45)

\qbezier(0,45)(-50,60)(-50,45)

\qbezier(0,15)(50,0)(50,15)

\qbezier(0,15)(-50,0)(-50,15)

\put(50,45){\vector(0,-1){30}}

\put(-50,45){\vector(0,-1){30}}

\put(-1,30){\line(1,0){2}}

\put(-51,30){\line(1,0){2}}

\put(49,30){\line(1,0){2}}

\put(-60,47){\tiny{$m$}}

\put(55,47){\tiny{$n$}}

\put(5,30){\tiny{$m+n$}}

\put(-58,25){\small{$\mathbb{A}$}}

\put(53,25){\small{$\mathbb{B}$}}

\put(-8,25){\small{$\mathbb{X}$}}

\put(-2,-15){$\Gamma$}

\end{picture}

\caption{}\label{butterfly}

\end{figure}

\begin{lemma}\label{butterfly-structure}
Let $\Gamma$ be the MOY graph in Figure \ref{butterfly}. Then, as graded matrix factorizations over $\Sym(\mathbb{A}\cup\mathbb{B})$,
\[
C(\Gamma) \simeq C(\emptyset) \otimes_{\C} (\Sym(\mathbb{A}|\mathbb{B})/(h_N(\mathbb{A}\cup\mathbb{B}),\dots,h_{N-m-n+1}(\mathbb{A}\cup\mathbb{B}))) \{q^{-(m+n)(N-m-n)}\}\left\langle m+n \right\rangle,
\]
and, as graded $\C$-linear spaces,
\[
\Hom_{HMF}(C(\emptyset),C(\Gamma)) \cong \Hom_{HMF}(C(\Gamma),C(\emptyset)) \cong H(\Gamma) \cong C(\emptyset)\{\qb{N}{m+n} \qb{m+n}{n}\} \left\langle m+n \right\rangle.
\]
In particular, the subspaces of these spaces of homogeneous elements of quantum degree $-(m+n)(N-m-n)-mn$ are $1$-dimensional.
\end{lemma}
\begin{proof}
The homotopy equivalence follows from Proposition \ref{circle-module} and the proof of Theorem \ref{decomp-II}. The rest of the lemma follows from this homotopy equivalence and Theorems \ref{part-symm-str}, \ref{grassmannian}.
\end{proof}

Denote by $\bigcirc_{m+n}$ a circle colored by $m+n$. Then there are morphisms $C(\bigcirc_{m+n}) \xrightarrow{\phi} C(\Gamma)$ and $C(\Gamma) \xrightarrow{\overline{\phi}} C(\bigcirc_{m+n})$ induced by the edge splitting and merging. Denote by $\iota$ and $\epsilon$ the morphisms associated to creating and annihilating $\bigcirc_{m+n}$. Then $C(\emptyset) \xrightarrow{\widetilde{\iota}:=\phi\circ \iota} C(\Gamma)$ and $C(\Gamma) \xrightarrow{\widetilde{\epsilon}:=\epsilon\circ \overline{\phi}} C(\emptyset)$ are homogeneous morphisms of quantum degree $-(m+n)(N-m-n)-mn$ and $\zed_2$-degree $m+n$.

\begin{lemma}\label{butterfly-morphisms}
$\widetilde{\iota}$ and $\widetilde{\epsilon}$ are not homotopic to $0$. Therefore, they span the $1$-dimensional subspaces of $\Hom_{HMF}(C(\emptyset),C(\Gamma))$ and $\Hom_{HMF}(C(\Gamma),C(\emptyset))$ of homogeneous elements of quantum degree $-(m+n)(N-m-n)-mn$.
\end{lemma}
\begin{proof}
By Corollary \ref{iota-epsilon-composition} and Lemma \ref{phibar-compose-phi}, we have 
\begin{eqnarray*}
&& \widetilde{\epsilon} \circ \mathfrak{m}(S_{\lambda_{m,n}}(\mathbb{A}) \cdot S_{\lambda_{m+n,N-m-n}}(\mathbb{X})) \circ \widetilde{\iota} \\
& = & \epsilon\circ \overline{\phi} \circ \mathfrak{m}(S_{\lambda_{m,n}}(\mathbb{A}) \cdot S_{\lambda_{m+n,N-m-n}}(\mathbb{X})) \circ \phi\circ \iota \\
& \approx & \epsilon\circ \mathfrak{m}(S_{\lambda_{m+n,N-m-n}}(\mathbb{X})) \circ \overline{\phi} \circ \mathfrak{m}(S_{\lambda_{m,n}}(\mathbb{A})) \circ \phi\circ \iota \\
& \approx & \epsilon\circ \mathfrak{m}(S_{\lambda_{m+n,N-m-n}}(\mathbb{X})) \circ \iota \approx \id.
\end{eqnarray*}
This shows that $\widetilde{\iota}$ and $\widetilde{\epsilon}$ are not homotopic to $0$. The rest of the lemma follows from this and Lemma \ref{butterfly-structure}.
\end{proof}

\begin{figure}[ht]

\setlength{\unitlength}{1pt}

\begin{picture}(360,75)(-180,-15)


\put(15,50){\vector(-1,0){30}}

\put(-15,50){\vector(0,-1){40}}

\put(-15,10){\vector(1,0){30}}

\put(15,10){\vector(0,1){40}}

\qbezier(-15,50)(-30,60)(-30,50)

\qbezier(-15,10)(-30,0)(-30,10)

\put(-30,50){\vector(0,-1){40}}

\qbezier(15,50)(30,60)(30,50)

\qbezier(15,10)(30,0)(30,10)

\put(30,10){\vector(0,1){40}}

\put(-13,28){\tiny{$n$}}

\put(24,28){\tiny{$n$}}

\put(8,28){\tiny{$m$}}

\put(-28,28){\tiny{$m$}}

\put(-10,52){\tiny{$m+n$}}

\put(-10,12){\tiny{$m+n$}}

\put(-2,-15){$\Gamma_2$}

\end{picture}

\caption{}\label{ears}

\end{figure}

\begin{lemma}\label{lemma-def-hat-iota}
Denote by $\Gamma_2$ the MOY graph in Figure \ref{ears} and by $\bigcirc_{m+n}$ a circle colored by $m+n$. As $\C$-linear spaces, 
\[
\Hom_{HMF}(C(\emptyset),C(\Gamma_2)) \cong \Hom_{HMF}(C(\Gamma_2),C(\emptyset)) \cong C(\emptyset)\{\qb{N}{m+n}\cdot\qb{m+n}{n}^2\}\left\langle m+n\right\rangle.
\] 

In particular, the subspaces of these spaces consisting of homogeneous elements of quantum degree $-(m+n)(N-m-n)-2mn$ are $1$-dimensional, and are spanned by the compositions
\[
C(\emptyset) \xrightarrow{\iota} C(\bigcirc_{m+n}) \xrightarrow{\phi_1\otimes \phi_2} C(\Gamma_2)
\hspace{.5cm} \text{ and } \hspace{.5cm}
C(\Gamma_2) \xrightarrow{\overline{\phi}_1\otimes \overline{\phi}_2} C(\bigcirc_{m+n}) \xrightarrow{\epsilon} C(\emptyset),
\]
where $\iota$ and $\epsilon$ are morphisms associated to creating and annihilating $\bigcirc_{m+n}$, and $\phi_1, ~\phi_2$ (resp. $\overline{\phi}_1,~ \overline{\phi}_2$) are morphisms associated to the two apparent edge splitting (resp. merging.)

In the rest of this section, we denote by $\widehat{\bm{\iota}}$ the composition $C(\emptyset) \xrightarrow{(\phi_1\otimes \phi_2)\circ\iota} C(\Gamma_2)$, and by $\widehat{\bm{\epsilon}}$ the composition $C(\Gamma_2) \xrightarrow{\epsilon\circ(\overline{\phi}_1\otimes \overline{\phi}_2)} C(\emptyset)$.
\end{lemma}

\begin{proof}
By Theorem \ref{decomp-II} and Proposition \ref{circle-dimension}, we have 
\[
C(\Gamma_2) \simeq C(\emptyset)\{\qb{N}{m+n}\cdot\qb{m+n}{n}^2\}\left\langle m+n\right\rangle.
\]
The structures of $\Hom_{HMF}(C(\emptyset),C(\Gamma_2))$ and $\Hom_{HMF}(C(\Gamma_2),C(\emptyset))$ follow from this. It then follows that the subspaces of these spaces consisting of homogeneous elements of quantum degree $-(m+n)(N-m-n)-2mn$ are $1$-dimensional. 

It is easy to check that $\widehat{\bm{\iota}}$ and $\widehat{\bm{\epsilon}}$ are both homogeneous with quantum degree $-(m+n)(N-m-n)-2mn$ and $\zed_2$-degree $m+n$. Similar to the proof of Lemma \ref{butterfly-morphisms}, one can use Corollary \ref{iota-epsilon-composition} and Lemma \ref{phibar-compose-phi} to show that $\widehat{\bm{\iota}}$ and $\widehat{\bm{\epsilon}}$ are not homotopic to $0$. This completes the proof.
\end{proof}

\begin{figure}[ht]

\setlength{\unitlength}{1pt}

\begin{picture}(360,75)(-170,-15)


\put(-170,23){\Huge{$\emptyset$}}

\put(-150,30){\vector(1,0){30}}

\put(-138,32){\small{$\widehat{\bm{\iota}}$}}


\put(-65,50){\vector(-1,0){30}}

\put(-95,50){\vector(0,-1){40}}

\put(-95,10){\vector(1,0){30}}

\put(-65,10){\vector(0,1){40}}

\qbezier(-95,50)(-110,60)(-110,50)

\qbezier(-95,10)(-110,0)(-110,10)

\put(-110,50){\vector(0,-1){40}}

\qbezier(-65,50)(-50,60)(-50,50)

\qbezier(-65,10)(-50,0)(-50,10)

\put(-50,10){\vector(0,1){40}}

\multiput(-100,7)(0,3){16}{\line(0,1){1}}

\multiput(-60,7)(0,3){16}{\line(0,1){1}}

\put(-59,28){\tiny{$\ddag$}}

\put(-104,28){\tiny{$\dag$}}

\put(-93,28){\tiny{$n$}}

\put(-48,50){\tiny{$n$}}

\put(-72,28){\tiny{$m$}}

\put(-117,50){\tiny{$m$}}

\put(-90,52){\tiny{$m+n$}}

\put(-90,12){\tiny{$m+n$}}

\put(-82,-15){$\Gamma_2$}

\put(-40,30){\vector(1,0){40}}

\put(-33,34){\small{$\eta_\dag \otimes \eta_\ddag$}}


\put(15,30){\oval(10,50)}

\put(20,40){\vector(0,1){0}}

\put(75,50){\vector(-1,0){30}}

\put(45,50){\vector(0,-1){40}}

\put(45,10){\vector(1,0){30}}

\put(75,10){\vector(0,1){40}}

\qbezier(45,50)(30,60)(30,50)

\qbezier(45,10)(30,0)(30,10)

\put(30,50){\vector(0,-1){40}}

\qbezier(75,50)(90,60)(90,50)

\qbezier(75,10)(90,0)(90,10)

\put(90,10){\vector(0,1){40}}

\put(105,30){\oval(10,50)}

\put(100,20){\vector(0,-1){0}}

\put(47,28){\tiny{$n$}}

\put(85,30){\tiny{$n$}}

\put(96,50){\tiny{$n$}}

\put(68,28){\tiny{$m$}}

\put(21,50){\tiny{$m$}}

\put(32,30){\tiny{$m$}}

\put(50,52){\tiny{$m+n$}}

\put(50,12){\tiny{$m+n$}}

\put(30,-15){$\Gamma_2\sqcup \bigcirc_m \sqcup \bigcirc_n$}

\put(120,30){\vector(1,0){30}}

\put(133,34){\small{$\widehat{\bm{\epsilon}}$}}


\put(165,30){\oval(10,50)}

\put(170,40){\vector(0,1){0}}

\put(180,30){\oval(10,50)}

\put(175,20){\vector(0,-1){0}}

\put(161,40){\tiny{$m$}}

\put(181,40){\tiny{$n$}}

\put(150,-15){$\bigcirc_m \sqcup \bigcirc_n$}

\end{picture}

\caption{}\label{comp-bigiota-etas-bigepsilon-figure}

\end{figure}

\begin{lemma}\label{comp-bigiota-etas-bigepsilon}
Denote by $\bigcirc_m \sqcup \bigcirc_n$ the disjoint union of two circles colored by $m$ and $n$. Define the morphism $f:C(\emptyset)\rightarrow C(\bigcirc_m \sqcup \bigcirc_n)$ to be the composition in Figure \ref{comp-bigiota-etas-bigepsilon-figure}, that is, $f = \widehat{\bm{\epsilon}} \circ (\eta_\dag \otimes \eta_\ddag) \circ \widehat{\bm{\iota}}$. Then $f \approx \iota_m \otimes \iota_n$, where $\iota_m, \iota_n$ are the morphisms associated to creating the two circles in $\bigcirc_m \sqcup \bigcirc_n$.
\end{lemma}

\begin{proof}
It is easy to check that 
\[
\Hom_{HMF}(C(\emptyset),C(\bigcirc_m \sqcup \bigcirc_n)) \cong H(\bigcirc_m \sqcup \bigcirc_n) \cong C(\emptyset) \{\qb{N}{m}\cdot\qb{N}{n}\}\left\langle m+n \right\rangle.
\]
In particular, the subspace of $\Hom_{HMF}(C(\emptyset),C(\bigcirc_m \sqcup \bigcirc_n))$ of homogeneous elements of quantum degree $-m(N-m)-n(N-n)$ is $1$-dimensional and spanned by $\iota_m\otimes\iota_n$. One can see that $f$ is homogeneous of quantum degree $-m(N-m)-n(N-n)$. So, to prove the lemma, we only need to check that $f$ is not null homotopic. We do this by showing that $f_\ast(1)\neq 0$.

Note that $f = \widehat{\bm{\epsilon}} \circ (\eta_\dag \otimes \eta_\ddag) \circ \widehat{\bm{\iota}} = (\widehat{\bm{\epsilon}} \circ \eta_\ddag) \circ (\eta_\dag\circ \widehat{\bm{\iota}})$. 

\begin{figure}[ht]

\setlength{\unitlength}{1pt}

\begin{picture}(360,60)(-170,0)


\put(-170,23){\Huge{$\emptyset$}}

\put(-155,30){\vector(1,0){20}}

\put(-147,32){\small{$\iota_m$}}


\put(-120,30){\oval(20,50)}

\put(-110,40){\vector(0,1){0}}

\put(-108,50){\tiny{$m$}}

\put(-105,30){\vector(1,0){20}}

\put(-97,32){\small{$\psi$}}


\qbezier(-65,50)(-80,60)(-80,50)

\qbezier(-65,10)(-80,0)(-80,10)

\put(-80,50){\vector(0,-1){40}}

\qbezier(-65,50)(-50,60)(-50,50)

\qbezier(-65,10)(-50,0)(-50,10)

\put(-65,50){\vector(0,-1){40}}

\put(-50,10){\vector(0,1){40}}

\put(-63,28){\tiny{$n$}}

\put(-78,28){\tiny{$m$}}

\put(-49,50){\tiny{$m+n$}}

\put(-40,30){\vector(1,0){20}}

\put(-32,32){\small{$\phi$}}


\put(35,50){\vector(-1,0){30}}

\put(5,50){\vector(0,-1){40}}

\put(5,10){\vector(1,0){30}}

\put(35,10){\vector(0,1){40}}

\qbezier(5,50)(-10,60)(-10,50)

\qbezier(5,10)(-10,0)(-10,10)

\put(-10,50){\vector(0,-1){40}}

\qbezier(35,50)(50,60)(50,50)

\qbezier(35,10)(50,0)(50,10)

\put(50,10){\vector(0,1){40}}

\multiput(0,7)(0,3){16}{\line(0,1){1}}

\put(-4,28){\tiny{$\dag$}}

\put(7,28){\tiny{$n$}}

\put(52,50){\tiny{$n$}}

\put(28,28){\tiny{$m$}}

\put(-17,50){\tiny{$m$}}

\put(10,52){\tiny{$m+n$}}

\put(10,12){\tiny{$m+n$}}

\put(60,30){\vector(1,0){20}}

\put(67,34){\small{$\eta_\dag$}}


\put(95,30){\oval(10,50)}

\put(100,40){\vector(0,1){0}}

\put(155,50){\vector(-1,0){30}}

\put(125,50){\vector(0,-1){40}}

\put(125,10){\vector(1,0){30}}

\put(155,10){\vector(0,1){40}}

\qbezier(125,50)(110,60)(110,50)

\qbezier(125,10)(110,0)(110,10)

\put(110,50){\vector(0,-1){40}}

\qbezier(155,50)(170,60)(170,50)

\qbezier(155,10)(170,0)(170,10)

\put(170,10){\vector(0,1){40}}

\put(127,28){\tiny{$n$}}

\put(165,30){\tiny{$n$}}

\put(148,28){\tiny{$m$}}

\put(101,50){\tiny{$m$}}

\put(112,30){\tiny{$m$}}

\put(130,52){\tiny{$m+n$}}

\put(130,12){\tiny{$m+n$}}

\end{picture}

\caption{}\label{comp-bigiota-etas-bigepsilon-figure-1}

\end{figure}

\begin{figure}[ht]

\setlength{\unitlength}{1pt}

\begin{picture}(360,60)(-170,0)


\put(-170,23){\Huge{$\emptyset$}}

\put(-155,30){\vector(1,0){20}}

\put(-147,32){\small{$\iota_m$}}


\put(-120,30){\oval(20,50)}

\put(-110,40){\vector(0,1){0}}

\multiput(-120,5)(0,3){17}{\line(0,1){1}}

\put(-124,28){\tiny{$\dag$}}

\put(-108,50){\tiny{$m$}}

\put(-105,30){\vector(1,0){20}}

\put(-97,32){\small{$\eta_\dag$}}


\put(-70,30){\oval(10,50)}

\put(-65,40){\vector(0,1){0}}

\put(-76,20){\line(1,0){2}}

\put(-55,30){\oval(10,50)}

\put(-50,40){\vector(0,1){0}}

\put(-51,20){\line(1,0){2}}

\put(-83,50){\tiny{$m$}}

\put(-48,50){\tiny{$m$}}

\put(-83,16){\small{$\mathbb{X}$}}

\put(-48,16){\small{$\mathbb{Y}$}}

\put(-40,30){\vector(1,0){20}}

\put(-32,32){\small{$\psi$}}


\put(-5,30){\oval(10,50)}

\put(0,40){\vector(0,1){0}}

\put(-18,50){\tiny{$m$}}

\qbezier(25,50)(10,60)(10,50)

\qbezier(25,10)(10,0)(10,10)

\put(10,50){\vector(0,-1){40}}

\qbezier(25,50)(40,60)(40,50)

\qbezier(25,10)(40,0)(40,10)

\put(25,50){\vector(0,-1){40}}

\put(40,10){\vector(0,1){40}}

\put(27,28){\tiny{$n$}}

\put(12,28){\tiny{$m$}}

\put(41,50){\tiny{$m+n$}}

\put(-11,20){\line(1,0){2}}

\put(9,20){\line(1,0){2}}

\put(-18,16){\small{$\mathbb{X}$}}

\put(11,16){\small{$\mathbb{Y}$}}

\put(50,30){\vector(1,0){30}}

\put(63,34){\small{$\phi$}}


\put(95,30){\oval(10,50)}

\put(100,40){\vector(0,1){0}}

\put(155,50){\vector(-1,0){30}}

\put(125,50){\vector(0,-1){40}}

\put(125,10){\vector(1,0){30}}

\put(155,10){\vector(0,1){40}}

\qbezier(125,50)(110,60)(110,50)

\qbezier(125,10)(110,0)(110,10)

\put(110,50){\vector(0,-1){40}}

\qbezier(155,50)(170,60)(170,50)

\qbezier(155,10)(170,0)(170,10)

\put(170,10){\vector(0,1){40}}

\put(127,28){\tiny{$n$}}

\put(165,30){\tiny{$n$}}

\put(148,28){\tiny{$m$}}

\put(101,50){\tiny{$m$}}

\put(112,30){\tiny{$m$}}

\put(130,52){\tiny{$m+n$}}

\put(130,12){\tiny{$m+n$}}

\put(89,20){\line(1,0){2}}

\put(109,20){\line(1,0){2}}

\put(82,16){\small{$\mathbb{X}$}}

\put(111,16){\small{$\mathbb{Y}$}}

\end{picture}

\caption{}\label{comp-bigiota-etas-bigepsilon-figure-2}

\end{figure}

We consider $\eta_\dag\circ \widehat{\bm{\iota}}$ first. By Proposition \ref{decomposing-psi-bar}, one can see $\widehat{\bm{\iota}}\approx\phi\circ\psi\circ\iota_m$, where the morphisms on the right hand side are given in Figure \ref{comp-bigiota-etas-bigepsilon-figure-1}. So $\eta_\dag\circ \widehat{\bm{\iota}}$ is given by the composition in Figure \ref{comp-bigiota-etas-bigepsilon-figure-1}. If we choose marked points appropriately, then $\phi\circ\psi$ and $\eta_\dag$ act on different factors in the tensor product of matrix factorizations. So they commute. Thus $\eta_\dag\circ \widehat{\bm{\iota}} = (\phi\circ\psi)\circ(\eta_\dag\circ\iota_m)$, where the composition on the right hand side is given in Figure \ref{comp-bigiota-etas-bigepsilon-figure-2}. Denote by $\iota_{\mathbb{X}}$ and $\epsilon_{\mathbb{X}}$ (resp. $\iota_{\mathbb{Y}}$ and $\epsilon_{\mathbb{Y}}$) the morphisms associated to creating and annihilating the circle marked by $\mathbb{X}$ (resp. $\mathbb{Y}$.) Then $(\iota_{\mathbb{X}})_\ast(1)$ and $(\iota_{\mathbb{Y}})_\ast(1)$ are the generating classes of the homology of the circles marked by $\mathbb{X}$ and $\mathbb{Y}$. By Proposition \ref{saddle+compose+annihilation}, we have $\epsilon_{\mathbb{Y}} \circ \eta_\dag \circ \iota_m \approx \iota_{\mathbb{X}}$. So $(\epsilon_{\mathbb{Y}} \circ \eta_\dag \circ \iota_m)_\ast (1) \propto (\iota_{\mathbb{X}})_\ast(1)$. By Theorem \ref{grassmannian}, Proposition \ref{circle-module} and Corollary \ref{iota-epsilon-composition}, this implies that 
\[
(\eta_\dag \circ \iota_m)_\ast (1) \propto (S_{\lambda_{m,N-m}}(\mathbb{Y}) + H) \cdot (\iota_{\mathbb{X}})_\ast(1) \otimes (\iota_{\mathbb{Y}})_\ast(1),
\]
where $H$ is an element in $\Sym(\mathbb{X}|\mathbb{Y})$ whose total degree in $\mathbb{Y}$ is less than $2m(N-m)$. By Proposition \ref{decomposing-psi-bar} and the definition of $\widehat{\bm{\iota}}$, we have that 
\[
(\phi\circ\psi)_\ast((\iota_{\mathbb{Y}})_\ast(1)) \propto \widehat{\bm{\iota}}_\ast (1).
\]
Thus,
\[
(\eta_\dag \circ \widehat{\bm{\iota}})_\ast (1) \propto (S_{\lambda_{m,N-m}}(\mathbb{Y}) + H) \cdot (\iota_{\mathbb{X}})_\ast(1) \otimes \widehat{\bm{\iota}}_\ast(1).
\]

\begin{figure}[ht]

\setlength{\unitlength}{1pt}

\begin{picture}(360,60)(-170,0)


\put(-125,50){\vector(-1,0){30}}

\put(-155,50){\vector(0,-1){40}}

\put(-155,10){\vector(1,0){30}}

\put(-125,10){\vector(0,1){40}}

\qbezier(-155,50)(-170,60)(-170,50)

\qbezier(-155,10)(-170,0)(-170,10)

\put(-170,50){\vector(0,-1){40}}

\qbezier(-125,50)(-110,60)(-110,50)

\qbezier(-125,10)(-110,0)(-110,10)

\put(-110,10){\vector(0,1){40}}

\multiput(-120,7)(0,3){16}{\line(0,1){1}}

\put(-119,28){\tiny{$\ddag$}}

\put(-153,28){\tiny{$n$}}

\put(-108,28){\tiny{$n$}}

\put(-132,28){\tiny{$m$}}

\put(-168,28){\tiny{$m$}}

\put(-150,52){\tiny{$m+n$}}

\put(-150,12){\tiny{$m+n$}}

\put(-171,20){\line(1,0){2}}

\put(-169,16){\small{$\mathbb{Y}$}}

\put(-100,30){\vector(1,0){20}}

\put(-95,34){\small{$\eta_\ddag$}}


\put(-30,50){\vector(-1,0){30}}

\put(-60,50){\vector(0,-1){40}}

\put(-60,10){\vector(1,0){30}}

\put(-30,10){\vector(0,1){40}}

\qbezier(-60,50)(-75,60)(-75,50)

\qbezier(-60,10)(-75,0)(-75,10)

\put(-75,50){\vector(0,-1){40}}

\qbezier(-30,50)(-15,60)(-15,50)

\qbezier(-30,10)(-15,0)(-15,10)

\put(-15,10){\vector(0,1){40}}

\put(-5,30){\oval(10,50)}

\put(-10,20){\vector(0,-1){0}}

\put(-58,28){\tiny{$n$}}

\put(-20,28){\tiny{$n$}}

\put(-9,28){\tiny{$n$}}

\put(-37,28){\tiny{$m$}}

\put(-73,30){\tiny{$m$}}

\put(-55,52){\tiny{$m+n$}}

\put(-55,12){\tiny{$m+n$}}

\put(-76,20){\line(1,0){2}}

\put(-74,16){\small{$\mathbb{Y}$}}

\put(5,30){\vector(1,0){20}}

\put(13,34){\small{$\overline{\phi}$}}


\qbezier(45,50)(30,60)(30,50)

\qbezier(45,10)(30,0)(30,10)

\put(30,50){\vector(0,-1){40}}

\qbezier(45,50)(60,60)(60,50)

\qbezier(45,10)(60,0)(60,10)

\put(45,10){\vector(0,1){40}}

\put(60,10){\vector(0,1){40}}

\put(56,28){\tiny{$n$}}

\put(66,28){\tiny{$n$}}

\put(46,28){\tiny{$m$}}

\put(10,50){\tiny{$m+n$}}

\put(29,20){\line(1,0){2}}

\put(31,16){\small{$\mathbb{W}$}}

\put(59,20){\line(1,0){2}}

\put(53,16){\small{$\mathbb{A}$}}

\put(74,20){\line(1,0){2}}

\put(77,16){\small{$\mathbb{B}$}}

\put(70,30){\oval(10,50)}

\put(65,20){\vector(0,-1){0}}

\put(80,30){\vector(1,0){20}}

\put(88,34){\small{$\overline{\psi}$}}


\put(110,30){\oval(10,50)}

\put(115,40){\vector(0,1){0}}

\put(125,30){\oval(10,50)}

\put(120,20){\vector(0,-1){0}}

\put(104,20){\line(1,0){2}}

\put(99,16){\small{$\mathbb{A}$}}

\put(129,20){\line(1,0){2}}

\put(132,16){\small{$\mathbb{B}$}}

\put(106,40){\tiny{$n$}}

\put(121,40){\tiny{$n$}}

\put(135,30){\vector(1,0){20}}

\put(141,34){\small{$\epsilon_{\mathbb{A}}$}}


\put(165,30){\oval(10,50)}

\put(160,20){\vector(0,-1){0}}

\put(169,20){\line(1,0){2}}

\put(172,16){\small{$\mathbb{B}$}}

\put(172,40){\tiny{$n$}}

\end{picture}

\caption{}\label{comp-bigiota-etas-bigepsilon-figure-3}

\end{figure}

\begin{figure}[ht]

\setlength{\unitlength}{1pt}

\begin{picture}(360,75)(-170,-15)


\put(-125,50){\vector(-1,0){30}}

\put(-155,50){\vector(0,-1){40}}

\put(-155,10){\vector(1,0){30}}

\put(-125,10){\vector(0,1){40}}

\qbezier(-155,50)(-170,60)(-170,50)

\qbezier(-155,10)(-170,0)(-170,10)

\put(-170,50){\vector(0,-1){40}}

\qbezier(-125,50)(-110,60)(-110,50)

\qbezier(-125,10)(-110,0)(-110,10)

\put(-110,10){\vector(0,1){40}}

\put(-153,28){\tiny{$n$}}

\put(-116,28){\tiny{$n$}}

\put(-132,28){\tiny{$m$}}

\put(-168,28){\tiny{$m$}}

\put(-150,52){\tiny{$m+n$}}

\put(-150,12){\tiny{$m+n$}}

\put(-171,20){\line(1,0){2}}

\put(-169,16){\small{$\mathbb{Y}$}}

\put(-105,30){\vector(1,0){20}}

\put(-98,34){\small{$\overline{\phi}$}}


\qbezier(-65,50)(-80,60)(-80,50)

\qbezier(-65,10)(-80,0)(-80,10)

\put(-80,50){\vector(0,-1){40}}

\qbezier(-65,50)(-50,60)(-50,50)

\qbezier(-65,10)(-50,0)(-50,10)

\put(-65,10){\vector(0,1){40}}

\put(-50,10){\vector(0,1){40}}

\put(-64,28){\tiny{$n$}}

\put(-54,28){\tiny{$n$}}

\put(-64,28){\tiny{$m$}}

\put(-100,50){\tiny{$m+n$}}

\put(-81,20){\line(1,0){2}}

\put(-79,16){\small{$\mathbb{W}$}}

\put(-45,30){\vector(1,0){20}}

\put(-37,34){\small{$\overline{\psi}$}}

\put(-70,-15){$\Gamma_3$}


\put(-5,30){\oval(20,50)}

\put(5,40){\vector(0,1){0}}

\multiput(-5,5)(0,3){17}{\line(0,1){1}}

\put(7,50){\tiny{$n$}}

\put(-3,28){\tiny{$\ddag$}}

\put(-16,20){\line(1,0){2}}

\put(-24,16){\small{$\mathbb{A}$}}

\put(4,20){\line(1,0){2}}

\put(7,16){\small{$\mathbb{B}$}}

\put(15,30){\vector(1,0){20}}

\put(21,34){\small{$\eta_\ddag$}}


\put(50,30){\oval(10,50)}

\put(55,40){\vector(0,1){0}}

\put(65,30){\oval(10,50)}

\put(60,20){\vector(0,-1){0}}

\put(44,20){\line(1,0){2}}

\put(39,16){\small{$\mathbb{A}$}}

\put(69,20){\line(1,0){2}}

\put(72,16){\small{$\mathbb{B}$}}

\put(46,40){\tiny{$n$}}

\put(61,40){\tiny{$n$}}

\put(80,30){\vector(1,0){20}}

\put(86,34){\small{$\epsilon_{\mathbb{A}}$}}


\put(115,30){\oval(10,50)}

\put(110,20){\vector(0,-1){0}}

\put(119,20){\line(1,0){2}}

\put(122,16){\small{$\mathbb{B}$}}

\put(122,40){\tiny{$n$}}

\end{picture}

\caption{}\label{comp-bigiota-etas-bigepsilon-figure-4}

\end{figure}

Next we consider $\widehat{\bm{\epsilon}} \circ \eta_\ddag$. Since the circle marked by $\mathbb{X}$ is not affected by these morphisms, we temporarily drop that circle from our figures. By Proposition \ref{decomposing-psi-bar}, $\widehat{\bm{\epsilon}} \approx \epsilon_{\mathbb{A}} \circ \overline{\psi}\circ \overline{\phi}$, where the morphisms on the right hand side are given in Figure \ref{comp-bigiota-etas-bigepsilon-figure-3}. So $\widehat{\bm{\epsilon}} \circ \eta_\ddag$ is given by the composition in Figure \ref{comp-bigiota-etas-bigepsilon-figure-3}. If we choose marked points appropriately, then $\overline{\psi}\circ \overline{\phi}$ and $\eta_\ddag$ act on different factors in the tensor product of matrix factorizations. So they commute. Therefore, $\widehat{\bm{\epsilon}} \circ \eta_\ddag$ is also given by the composition in Figure \ref{comp-bigiota-etas-bigepsilon-figure-4}. By Proposition \ref{saddle+compose+annihilation}, $\epsilon_{\mathbb{A}} \circ \eta_\ddag \approx \id$. So $\widehat{\bm{\epsilon}} \circ \eta_\ddag \approx \overline{\psi}\circ \overline{\phi}$, where $\overline{\psi}$ and $\overline{\phi}$ are given in Figure \ref{comp-bigiota-etas-bigepsilon-figure-4}. Denote by $\widetilde{\iota}$ the morphism given in Lemma \ref{butterfly-morphisms} associated to creating $\Gamma_3$ in Figure \ref{comp-bigiota-etas-bigepsilon-figure-4}. Then, by \cite[Proposition Gr3]{Lascoux-notes} and the explicit description of $\overline{\phi}$ in Subsection \ref{subsec-edge-split-morph}, we have
\[
\overline{\phi}((S_{\lambda_{m,N-m}}(\mathbb{Y}) + H) \cdot (\iota_{\mathbb{X}})_\ast(1) \otimes \widehat{\bm{\iota}}_\ast(1)) \propto (S_{\lambda_{m,N-m-n}}(\mathbb{W}) + h) \cdot (\iota_{\mathbb{X}})_\ast(1) \otimes \widetilde{\iota}_\ast(1),
\]
where $h=\zeta(H)$ is an element of $\Sym(\mathbb{X}|\mathbb{W})$ with total degree in $\mathbb{W}$ less than $2m(N-m-n)$, and $\zeta$ is the Sylvester operator given in Theorem \ref{part-symm-str}. Let $\psi_n: C(\bigcirc_n) \rightarrow C(\Gamma_3)$ be the morphism associated to the loop addition. (Note that this morphism is not the $\psi$ in Figures \ref{comp-bigiota-etas-bigepsilon-figure-1}, \ref{comp-bigiota-etas-bigepsilon-figure-2}.) Then, by Proposition \ref{decomposing-psi-bar}, one can see that $\widetilde{\iota} \approx \psi_n\circ\iota_n$.

Altogether, we have
\begin{eqnarray*}
f_\ast(1) & \propto & (\widehat{\bm{\epsilon}} \circ \eta_\ddag)_\ast \circ (\eta_\dag\circ \widehat{\bm{\iota}})_\ast (1) \\
& \propto & (\widehat{\bm{\epsilon}} \circ \eta_\ddag)_\ast((S_{\lambda_{m,N-m}}(\mathbb{Y}) + H) \cdot (\iota_{\mathbb{X}})_\ast(1)\otimes \widehat{\bm{\iota}}_\ast(1)) \\
& \propto & (\overline{\psi}\circ \overline{\phi})_\ast((S_{\lambda_{m,N-m}}(\mathbb{Y}) + H) \cdot (\iota_{\mathbb{X}})_\ast(1)\otimes \widehat{\bm{\iota}}_\ast(1)) \\
& \propto & \overline{\psi}_\ast ((S_{\lambda_{m,N-m-n}}(\mathbb{W}) + h) \cdot (\iota_{\mathbb{X}})_\ast(1) \otimes \widetilde{\iota}_\ast(1)) \\
& \propto & (\iota_{\mathbb{X}})_\ast(1) \otimes (\overline{\psi}_\ast \circ \mathfrak{m}(S_{\lambda_{m,N-m-n}}(\mathbb{W}) + h) \circ (\psi_n)_\ast \circ (\iota_n)_\ast(1)) \\
& \propto & (\iota_{\mathbb{X}})_\ast(1) \otimes (\iota_n)_\ast(1),
\end{eqnarray*}
where the last step follows from equation \eqref{psibar-comp-psi} in Proposition \ref{decomposing-psi-bar}. It is clear that the circle marked by $\mathbb{X}$ is the ``$\bigcirc_m$" in $\bigcirc_m\sqcup\bigcirc_n$. So the above computation shows that $f_\ast(1) \propto (\iota_m)_\ast(1) \otimes (\iota_n)_\ast(1) \neq 0$. This proves the lemma.
\end{proof}

\begin{figure}[ht]

\setlength{\unitlength}{1pt}

\begin{picture}(360,75)(-180,-15)

\put(-100,0){\vector(1,1){20}}

\put(-80,20){\vector(1,0){20}}

\put(-60,20){\vector(0,1){20}}

\put(-60,20){\vector(1,-1){20}}

\put(-60,40){\vector(-1,0){20}}

\put(-80,40){\vector(0,-1){20}}

\put(-80,40){\vector(-1,1){20}}

\put(-40,60){\vector(-1,-1){20}}

\put(-105,0){\tiny{$n$}}

\put(-105,55){\tiny{$n$}}

\put(-57,28){\tiny{$n$}}

\put(-38,0){\tiny{$m$}}

\put(-38,55){\tiny{$m$}}

\put(-87,28){\tiny{$m$}}

\put(-78.5,43){\tiny{$m+n$}}

\put(-78.5,13){\tiny{$m+n$}}

\put(-72,-15){$\Gamma$}


\put(50,0){\vector(0,1){60}}

\put(90,60){\vector(0,-1){60}}

\put(45,28){\tiny{$n$}}

\put(92,28){\tiny{$m$}}

\put(68,-15){$\Gamma_0$}

\end{picture}

\caption{}\label{Gamma-Gamma-0-figure}

\end{figure}

\begin{definition}\label{maps-Gamma-Gamma-0}
Let $\Gamma$ and $\Gamma_0$ be the MOY graphs in Figure \ref{Gamma-Gamma-0-figure}. (They are slightly more general than those in Theorem \ref{decomp-III}.) Define the morphism 
\[
F:C(\Gamma_0)\rightarrow C(\Gamma)
\] 
to be the composition in Figure \ref{Gamma-Gamma-0-figure-1}, and the morphism 
\[
G:C(\Gamma)\rightarrow C(\Gamma_0)
\]
to be the composition in Figure \ref{Gamma-Gamma-0-figure-2}, where $\widehat{\bm{\iota}}$, $\widehat{\bm{\epsilon}}$ are defined in Lemma \ref{lemma-def-hat-iota}, and $\eta_{\Box}$, $\eta_\triangle$, $\eta_\dag$, $\eta_\ddag$ are the morphisms associated to the corresponding saddle moves. 
\end{definition}

\begin{figure}[ht]

\setlength{\unitlength}{1pt}

\begin{picture}(360,75)(-180,-15)


\put(-170,0){\vector(0,1){60}}

\put(-130,60){\vector(0,-1){60}}

\put(-175,28){\tiny{$n$}}

\put(-128,28){\tiny{$m$}}

\put(-152,-15){$\Gamma_0$}

\put(-100,30){\vector(1,0){30}}

\put(-88,34){\small{$\widehat{\bm{\iota}}$}}


\put(-40,0){\vector(0,1){60}}

\multiput(-39,30)(3,0){5}{\line(1,0){1}}

\put(-10,20){\vector(1,0){20}}

\put(10,20){\vector(0,1){20}}

\put(10,40){\vector(-1,0){20}}

\put(-10,40){\vector(0,-1){20}}

\qbezier(-10,40)(-25,60)(-25,40)

\qbezier(-10,20)(-25,0)(-25,20)

\put(-25,40){\vector(0,-1){20}}

\qbezier(10,40)(25,60)(25,40)

\qbezier(10,20)(25,0)(25,20)

\put(25,20){\vector(0,1){20}}

\multiput(26,30)(3,0){5}{\line(1,0){1}}

\put(40,60){\vector(0,-1){60}}

\put(-45,28){\tiny{$n$}}

\put(-23,28){\tiny{$n$}}

\put(5,28){\tiny{$n$}}

\put(42,28){\tiny{$m$}}

\put(18,28){\tiny{$m$}}

\put(-8,28){\tiny{$m$}}

\put(-9,42){\tiny{$m+n$}}

\put(-9,14){\tiny{$m+n$}}

\put(-35,32){\tiny{$\Box$}}

\put(30,32){\tiny{$\Diamond$}}

\put(60,30){\vector(1,0){30}}

\put(60,34){\small{$\eta_{\Box} \otimes \eta_{\Diamond}$}}


\put(110,0){\vector(1,1){20}}

\put(130,20){\vector(1,0){20}}

\put(150,20){\vector(0,1){20}}

\put(150,20){\vector(1,-1){20}}

\put(150,40){\vector(-1,0){20}}

\put(130,40){\vector(0,-1){20}}

\put(130,40){\vector(-1,1){20}}

\put(170,60){\vector(-1,-1){20}}

\put(105,0){\tiny{$n$}}

\put(105,55){\tiny{$n$}}

\put(153,28){\tiny{$n$}}

\put(172,0){\tiny{$m$}}

\put(157,55){\tiny{$m$}}

\put(123,28){\tiny{$m$}}

\put(131,43){\tiny{$m+n$}}

\put(131,13){\tiny{$m+n$}}

\put(138,-15){$\Gamma$}

\end{picture}

\caption{Definition of $F$}\label{Gamma-Gamma-0-figure-1}

\end{figure}
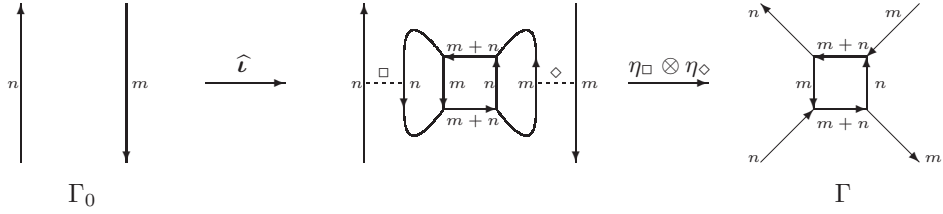

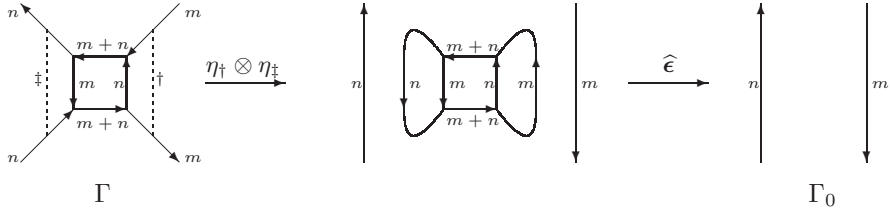
\begin{figure}[ht]

\setlength{\unitlength}{1pt}

\begin{picture}(360,75)(-180,-15)


\put(-170,0){\vector(1,1){20}}

\put(-150,20){\vector(1,0){20}}

\put(-130,20){\vector(0,1){20}}

\put(-130,20){\vector(1,-1){20}}

\put(-130,40){\vector(-1,0){20}}

\put(-150,40){\vector(0,-1){20}}

\put(-150,40){\vector(-1,1){20}}

\put(-110,60){\vector(-1,-1){20}}

\multiput(-160,10)(0,3){14}{\line(0,1){1}}

\multiput(-120,10)(0,3){14}{\line(0,1){1}}

\put(-165,30){\tiny{$\ddag$}}

\put(-119,30){\tiny{$\dag$}}

\put(-175,0){\tiny{$n$}}

\put(-175,55){\tiny{$n$}}

\put(-135,28){\tiny{$n$}}

\put(-108,0){\tiny{$m$}}

\put(-108,55){\tiny{$m$}}

\put(-148,28){\tiny{$m$}}

\put(-149,43){\tiny{$m+n$}}

\put(-149,13){\tiny{$m+n$}}

\put(-142,-15){$\Gamma$}

\put(-100,30){\vector(1,0){30}}

\put(-100,34){\small{$\eta_\dag \otimes \eta_\ddag$}}


\put(-40,0){\vector(0,1){60}}

\put(-10,20){\vector(1,0){20}}

\put(10,20){\vector(0,1){20}}

\put(10,40){\vector(-1,0){20}}

\put(-10,40){\vector(0,-1){20}}

\qbezier(-10,40)(-25,60)(-25,40)

\qbezier(-10,20)(-25,0)(-25,20)

\put(-25,40){\vector(0,-1){20}}

\qbezier(10,40)(25,60)(25,40)

\qbezier(10,20)(25,0)(25,20)

\put(25,20){\vector(0,1){20}}

\put(40,60){\vector(0,-1){60}}

\put(-45,28){\tiny{$n$}}

\put(-23,28){\tiny{$n$}}

\put(5,28){\tiny{$n$}}

\put(42,28){\tiny{$m$}}

\put(18,28){\tiny{$m$}}

\put(-8,28){\tiny{$m$}}

\put(-9,42){\tiny{$m+n$}}

\put(-9,14){\tiny{$m+n$}}

\put(60,30){\vector(1,0){30}}

\put(73,34){\small{$\widehat{\bm{\epsilon}}$}}


\put(110,0){\vector(0,1){60}}

\put(150,60){\vector(0,-1){60}}

\put(105,28){\tiny{$n$}}

\put(152,28){\tiny{$m$}}

\put(128,-15){$\Gamma_0$}

\end{picture}

\caption{Definition of $G$}\label{Gamma-Gamma-0-figure-2}

\end{figure}

\begin{proposition}\label{Gamma-Gamma-0-proposition}
Let $F$ and $G$ be the morphisms given in Definition \ref{maps-Gamma-Gamma-0}. Then $F$ and $G$ are both homogeneous morphisms of quantum degree $0$ and $\zed_2$-degree $0$. Moreover, $G \circ F \approx \id_{C(\Gamma_0)}$.
\end{proposition}

\begin{proof}
Recall that $\widehat{\bm{\iota}}$, $\widehat{\bm{\epsilon}}$ are homogeneous morphisms of quantum degree $-(m+n)(N-m-n)-2mn$ and $\zed_2$-degree $m+n$, and $\eta_{\Box} \otimes \eta_{\Diamond}$, $\eta_\dag \otimes \eta_\ddag$ are homogeneous morphisms of quantum degree $m(N-m)+n(N-n)$ and $\zed_2$-degree $m+n$. So $F$ and $G$ are homogeneous morphisms of quantum degree $0$ and $\zed_2$-degree $0$.

Next we consider the composition $G \circ F$. With appropriate markings of MOY graphs, $\eta_{\Box} \otimes \eta_{\Diamond}$ and $G$ act on different factors of a tensor product. So they commute. Hence, 
\[
G \circ F = (\eta_{\Box} \otimes \eta_{\Diamond}) \circ G \circ \widehat{\bm{\iota}} = (\eta_{\Box} \otimes \eta_{\Diamond}) \circ (\widehat{\bm{\epsilon}} \circ (\eta_\dag \otimes \eta_\ddag) \circ \widehat{\bm{\iota}}),
\]
where the right hand side is the composition in Figure \ref{Gamma-Gamma-0-figure-3}. By Lemma \ref{comp-bigiota-etas-bigepsilon}, 
\[
\widehat{\bm{\epsilon}} \circ (\eta_\dag \otimes \eta_\ddag) \circ \widehat{\bm{\iota}}\approx \iota_n \otimes \iota_m, 
\]
where $\iota_m$ and $\iota_n$ are the morphisms associated to creating $\bigcirc_m$ and $\bigcirc_n$. So, by Proposition \ref{creation+compose+saddle}, 
\[
G \circ F \approx (\eta_{\Box} \otimes \eta_{\Diamond}) \circ (\iota_n \otimes \iota_m) = (\eta_{\Box} \circ \iota_n) \otimes (\eta_{\Diamond} \circ \iota_m) \approx \id_{C(\Gamma_0)}.
\] 
\end{proof}

\begin{figure}[ht]

\setlength{\unitlength}{1pt}

\begin{picture}(360,75)(-180,-15)


\put(-170,0){\vector(0,1){60}}

\put(-130,60){\vector(0,-1){60}}

\put(-175,28){\tiny{$n$}}

\put(-128,28){\tiny{$m$}}

\put(-152,-15){$\Gamma_0$}

\put(-115,30){\vector(1,0){60}}

\put(-115,34){\small{$\widehat{\bm{\epsilon}} \circ (\eta_\dag \otimes \eta_\ddag) \circ \widehat{\bm{\iota}}$}}


\put(-40,0){\vector(0,1){60}}

\multiput(-39,30)(3,0){5}{\line(1,0){1}}

\put(-15,30){\oval(20,40)}

\put(-25,20){\vector(0,-1){0}}

\put(15,30){\oval(20,40)}

\put(25,40){\vector(0,1){0}}

\multiput(26,30)(3,0){5}{\line(1,0){1}}

\put(40,60){\vector(0,-1){60}}

\put(-45,28){\tiny{$n$}}

\put(-10,28){\tiny{$n$}}

\put(42,28){\tiny{$m$}}

\put(7,28){\tiny{$m$}}

\put(-35,32){\tiny{$\Box$}}

\put(30,32){\tiny{$\Diamond$}}

\put(-30,-15){$\Gamma_0\sqcup \bigcirc_m \sqcup \bigcirc_n$}

\put(60,30){\vector(1,0){30}}

\put(60,34){\small{$\eta_{\Box} \otimes \eta_{\Diamond}$}}


\put(110,0){\vector(0,1){60}}

\put(150,60){\vector(0,-1){60}}

\put(105,28){\tiny{$n$}}

\put(152,28){\tiny{$m$}}

\put(128,-15){$\Gamma_0$}

\end{picture}

\caption{}\label{Gamma-Gamma-0-figure-3}

\end{figure}

\subsection{Relating $\Gamma$ and $\Gamma_1$} Let $\Gamma$ and $\Gamma_1$ be the MOY graphs in Figure \ref{decomposition-III-figure}. In this subsection, we generalize the method in \cite[Section 6]{KR1} to construct morphisms between $C(\Gamma)$ and $C(\Gamma_1)$. To do this, we need the following special case of Proposition \ref{general-general-chi-maps}.

\begin{figure}[ht]

\setlength{\unitlength}{1pt}

\begin{picture}(360,75)(-180,-15)


\put(-120,0){\vector(1,1){20}}

\put(-100,20){\vector(1,-1){20}}

\put(-100,40){\vector(0,-1){20}}

\put(-100,40){\vector(1,1){20}}

\put(-120,60){\vector(1,-1){20}}

\put(-101,30){\line(1,0){2}}

\put(-115,45){\tiny{$m$}}

\put(-115,15){\tiny{$1$}}

\put(-90,45){\tiny{$1$}}

\put(-90,15){\tiny{$m$}}

\put(-95,28){\tiny{$m-1$}}

\put(-130,55){\small{$\mathbb{Y}$}}

\put(-75,0){\small{$\mathbb{X}$}}

\put(-113,27){\small{$\mathbb{W}$}}

\put(-135,0){\small{$\{t\}$}}

\put(-75,55){\small{$\{s\}$}}

\put(-102,-15){$\Gamma_4'$}


\put(-30,35){\vector(1,0){60}}

\put(30,25){\vector(-1,0){60}}

\put(-3,40){\small{$\chi^0$}}

\put(-3,15){\small{$\chi^1$}}


\put(60,10){\vector(1,1){20}}

\put(60,50){\vector(1,-1){20}}

\put(80,30){\vector(1,0){20}}

\put(100,30){\vector(1,1){20}}

\put(100,30){\vector(1,-1){20}}

\put(68,45){\tiny{$m$}}

\put(70,15){\tiny{$1$}}

\put(108,45){\tiny{$1$}}

\put(108,15){\tiny{$m$}}

\put(81,32){\tiny{$m+1$}}

\put(50,45){\small{$\mathbb{Y}$}}

\put(125,10){\small{$\mathbb{X}$}}

\put(45,10){\small{$\{t\}$}}

\put(122,45){\small{$\{s\}$}}

\put(88,-15){$\Gamma_5'$}

\end{picture}

\caption{}\label{chi-maps-figure}

\end{figure}

\begin{corollary}\label{chi-maps-def}
Let $\Gamma_4'$ and $\Gamma_5'$ be the MOY graphs in Figure \ref{chi-maps-figure}. Then there exist homogeneous morphisms 
\begin{eqnarray*}
\chi^0: C(\Gamma_4') \rightarrow C(\Gamma_5'), && \\
\chi^1: C(\Gamma_5') \rightarrow C(\Gamma_4'), &&
\end{eqnarray*}
such that
\begin{itemize}
	\item both $\chi^0$ and $\chi^1$ have quantum degree $1$ and $\zed_2$-degree $0$,
	\item $\chi^1 \circ \chi^0 \simeq (s-t) \cdot \id_{C(\Gamma_4')}$ and $\chi^0 \circ \chi^1 \simeq (s-t) \cdot \id_{C(\Gamma_5')}$.
\end{itemize}
\end{corollary}

\begin{figure}[ht]

\setlength{\unitlength}{1pt}

\begin{picture}(360,75)(-180,-15)

\put(-170,0){\vector(1,1){20}}

\put(-150,20){\vector(1,0){20}}

\put(-130,20){\vector(0,1){20}}

\put(-130,20){\vector(1,-1){20}}

\put(-130,40){\vector(-1,0){20}}

\put(-150,40){\vector(0,-1){20}}

\put(-150,40){\vector(-1,1){20}}

\put(-110,60){\vector(-1,-1){20}}

\put(-131,30){\line(1,0){2}}

\put(-165,0){\tiny{$1$}}

\put(-165,55){\tiny{$1$}}

\put(-127,33){\tiny{$1$}}

\put(-108,0){\tiny{$m$}}

\put(-108,55){\tiny{$m$}}

\put(-157,33){\tiny{$m$}}

\put(-148.5,43){\tiny{$m+1$}}

\put(-148.5,13){\tiny{$m+1$}}

\put(-127,25){\small{$\{s\}$}}

\put(-180,50){\small{$\{r\}$}}

\put(-180,5){\small{$\{t\}$}}

\put(-142,-15){$\Gamma$}

\put(-85,35){\vector(1,0){40}}

\put(-45,25){\vector(-1,0){40}}

\put(-80,40){\small{$\chi^1\otimes\chi^1$}}

\put(-80,15){\small{$\chi^0\otimes\chi^0$}}


\put(-20,0){\vector(2,1){20}}

\put(0,10){\vector(2,-1){20}}

\put(0,20){\vector(0,-1){10}}

\put(0,50){\vector(0,-1){10}}

\put(0,50){\vector(-2,1){20}}

\put(20,60){\vector(-2,-1){20}}

\qbezier(0,20)(-10,20)(-10,25)

\qbezier(0,20)(10,20)(10,25)

\qbezier(0,40)(-10,40)(-10,35)

\qbezier(0,40)(10,40)(10,35)

\put(-10,35){\vector(0,-1){10}}

\put(10,25){\vector(0,1){10}}

\put(9,28){\line(1,0){2}}

\put(-10,0){\tiny{$1$}}

\put(-10,56){\tiny{$1$}}

\put(12,33){\tiny{$1$}}

\put(-17,33){\tiny{$m$}}

\put(22,0){\tiny{$m$}}

\put(22,55){\tiny{$m$}}

\put(2,44){\tiny{$m-1$}}

\put(2,14){\tiny{$m-1$}}

\put(12,25){\small{$\{s\}$}}

\put(-30,50){\small{$\{r\}$}}

\put(-30,5){\small{$\{t\}$}}

\put(-2,-15){$\Gamma_7$}

\put(45,35){\vector(1,0){40}}

\put(85,25){\vector(-1,0){40}}

\put(63,40){\small{$\overline{\psi}$}}

\put(63,15){\small{$\psi$}}


\put(110,0){\vector(1,1){20}}

\put(130,20){\vector(1,-1){20}}

\put(130,40){\vector(0,-1){20}}

\put(130,40){\vector(-1,1){20}}

\put(150,60){\vector(-1,-1){20}}

\put(115,0){\tiny{$1$}}

\put(115,56){\tiny{$1$}}

\put(152,0){\tiny{$m$}}

\put(152,55){\tiny{$m$}}

\put(110,28){\tiny{$m-1$}}

\put(100,50){\small{$\{r\}$}}

\put(100,5){\small{$\{t\}$}}

\put(128,-15){$\Gamma_1$}

\end{picture}

\caption{}\label{gamma-gamma-1-figure}

\end{figure}

If we cut $\Gamma$ horizontally in half, then we get two copies of $\Gamma_5'$ in Figure \ref{chi-maps-figure}. These correspond to two copies of $\Gamma_4'$ in Figure \ref{chi-maps-figure}. Now we glue these two copies of $\Gamma_4'$ together along the original cutting points. This gives us $\Gamma_7$ in Figure \ref{gamma-gamma-1-figure}. There are two $\chi^0$ morphisms and two $\chi^1$ morphisms corresponding to the two pairs of $\Gamma_4'$ and $\Gamma_5'$. The morphism $\chi^0 \otimes \chi^0$ (resp. $\chi^1 \otimes \chi^1$) is the tensor product of these two $\chi^0$ morphisms (resp. $\chi^1$ morphisms.) Denote by $\psi:C(\Gamma_1)\rightarrow C(\Gamma_7)$ (resp. $\overline{\psi}:C(\Gamma_7)\rightarrow C(\Gamma_1)$) the morphism associated to the apparent loop addition (resp. removal) as defined in Subsection \ref{subsec-loop-morph}.

\begin{definition}\label{alpha-beta-def}
Define morphisms
\begin{eqnarray*}
\alpha & : & C(\Gamma_1)\left\langle 1\right\rangle \rightarrow C(\Gamma),  \\
\beta & : & C(\Gamma)\rightarrow C(\Gamma_1) \left\langle 1\right\rangle
\end{eqnarray*} 
by $\alpha=(\chi^0 \otimes \chi^0)\circ \psi$ and $\beta = \overline{\psi} \circ(\chi^1 \otimes \chi^1)$.

Moreover, for $j=0,1,\dots,N-m-2$, define morphisms
\begin{eqnarray*}
\alpha_j & : & C(\Gamma_1) \{q^{N-m-2-2j}\} \left\langle 1\right\rangle \rightarrow C(\Gamma), \\
\beta_j & : & C(\Gamma)\rightarrow C(\Gamma_1) \{q^{N-m-2-2j}\} \left\langle 1\right\rangle
\end{eqnarray*} 
by $\alpha_j = \mathfrak{m}(s^{N-m-2-j}) \circ \alpha$ and $\beta_j = \beta \circ \mathfrak{m}(h_j)$, where $\mathfrak{m}(\bullet)$ is the morphism induced by the multiplication by $\bullet$, and $h_j=h_j(\{r,s,t\})$ is the $j$-th complete symmetric polynomial in $\{r,s,t\}$.
\end{definition}

\begin{lemma}\label{alpha-beta-compose}
$\alpha_j$ and $\beta_j$ are homogeneous morphisms that preserve the $\zed_2\oplus\zed$-grading. Moreover,
\[
\beta_j\circ\alpha_i \approx \left\{%
\begin{array}{ll}
    \id & \text{if } ~i=j, \\ 
    0 & \text{if } ~i>j.
\end{array}%
\right.
\]
\end{lemma}
\begin{proof}
It is easy to verify the homogeneity and gradings of $\alpha_j$ and $\beta_j$. We leave it to the reader. Note that $\chi^0 \otimes \chi^0$ and $\chi^1 \otimes  \chi^1$ are both $\C[r,s,t]$-linear. So, by Corollary \ref{chi-maps-def}, 
\begin{eqnarray*}
\beta_j\circ\alpha_i & = & \overline{\psi} \circ (\chi^1 \otimes \chi^1) \circ \mathfrak{m}(h_j) \circ \mathfrak{m}(s^{N-m-2-i}) \circ (\chi^0 \otimes \chi^0) \circ \psi \\
& = & \overline{\psi} \circ \mathfrak{m}(h_j) \circ (\chi^1 \otimes \chi^1) \circ (\chi^0 \otimes \chi^0) \circ \mathfrak{m}(s^{N-m-2-i}) \circ \psi \\
& \simeq & \overline{\psi} \circ \mathfrak{m}(h_j) \circ \mathfrak{m}((r-s)(s-t)) \circ \mathfrak{m}(s^{N-m-2-i}) \circ \psi.
\end{eqnarray*}
Denote by $\hat{h}_j$ the $j$-th complete symmetric polynomial in $\{r,t\}$. Then, for $j \geq 0$, $h_j = \sum_{l=0}^j s^l \hat{h}_{j-l}$ and $\hat{h}_{j+1}= (r+t)\hat{h}_{j}-rt\hat{h}_{j-1}$. So 
\begin{eqnarray*}
&& s^{N-m-2-i}(r-s)(s-t)h_j \\
& = & \sum_{l=0}^j s^{N-m-2-i+l} (-s^2+(r+t)s-rt) \hat{h}_{j-l} \\
& = & -\sum_{l=0}^j s^{N-m-i+l} \hat{h}_{j-l} + \sum_{l=-1}^{j-1} s^{N-m-i+l}(r+t) \hat{h}_{j-l-1} - \sum_{l=-2}^{j-2} s^{N-m-i+l} rt \hat{h}_{j-l-2} \\
& = & -s^{N-m-i+j} + s^{N-m-i-1}\hat{h}_{j+1} - s^{N-m-i-2}rt\hat{h}_j \\
&& + \sum_{l=0}^{j-2} s^{N-m-i+l}(-\hat{h}_{j-l}+(r+t) \hat{h}_{j-l-1}-rt \hat{h}_{j-l-2}) \\
& = & -s^{N-m-i+j} + s^{N-m-i-1}\hat{h}_{j+1} - s^{N-m-i-2}rt\hat{h}_j.
\end{eqnarray*}
Note that $\overline{\psi}$ is $\C[r,t]$-linear. Thus, by Proposition \ref{decomposing-psi-bar},
\begin{eqnarray*}
\beta_j\circ\alpha_i & \simeq & \overline{\psi} \circ \mathfrak{m}(h_j) \circ \mathfrak{m}((r-s)(s-t)) \circ \mathfrak{m}(s^{N-m-2-i}) \circ \psi \\
& = & -\overline{\psi} \circ \mathfrak{m}(s^{N-m-i+j})\circ \psi + \mathfrak{m}(\hat{h}_{j+1}) \circ \overline{\psi} \circ \mathfrak{m}(s^{N-m-i-1})\circ \psi \\
&& - \mathfrak{m}(rt\hat{h}_j) \circ \overline{\psi} \circ \mathfrak{m}(s^{N-m-i-2})\circ \psi \\
& \approx & \left\{%
\begin{array}{ll}
    \id & \text{if } ~i=j, \\ 
    0 & \text{if } ~i>j.
\end{array}%
\right.
\end{eqnarray*}

\end{proof}

\begin{proposition}\label{alpha-beta-vec-def}
Let $\Gamma$ and $\Gamma_1$ be as in Theorem \ref{decomp-III}. Then there exist homogeneous morphisms 
\begin{eqnarray*}
\vec{\alpha} & : & C(\Gamma_1)\{[N-m-1]\} \left\langle 1 \right\rangle \rightarrow C(\Gamma), \\
\vec{\beta} & : & C(\Gamma) \rightarrow C(\Gamma_1)\{[N-m-1]\} \left\langle 1 \right\rangle,
\end{eqnarray*}
that preserve the $\zed_2\oplus\zed$-grading and satisfy $\vec{\beta} \circ \vec{\alpha} \simeq \id.$
\end{proposition}

\begin{proof}
The $\beta_j$ in Definition \ref{alpha-beta-def} is defined up to homotopy and scaling. From Lemma \ref{alpha-beta-compose}, we know 
\[
\beta_j\circ\alpha_i \approx \left\{%
\begin{array}{ll}
    \id & \text{if } ~i=j, \\ 
    0 & \text{if } ~i>j.
\end{array}%
\right.
\]
So, by choosing an appropriate scalar for each $\beta_j$, we can make
\begin{equation}\label{beta-scalar-choice}
\beta_j\circ\alpha_i \simeq \left\{%
\begin{array}{ll}
    \id & \text{if } ~i=j, \\ 
    0 & \text{if } ~i>j.
\end{array}%
\right.
\end{equation}
We assume \eqref{beta-scalar-choice} is true in the rest of this proof.

Define $\tau_{j,i} : C(\Gamma_1) \{q^{N-m-2-2i}\} \left\langle 1\right\rangle \rightarrow C(\Gamma_1) \{q^{N-m-2-2j}\} \left\langle 1\right\rangle$ by
\[
\tau_{j,i} = \left\{%
\begin{array}{ll}
   \sum_{l \geq 1} \sum_{i<k_1<\cdots<k_{l-1}<j} (-1)^l (\beta_j\circ\alpha_{k_{l-1}}) \circ (\beta_{k_{l-1}}\circ\alpha_{k_{l-2}})\circ\cdots\circ (\beta_{k_1}\circ\alpha_i) & \text{if } ~i<j \\
    \id & \text{if } ~i=j, \\ 
    0 & \text{if } ~i>j.
\end{array}%
\right.
\]
Then define $\hat{\beta}_j: C(\Gamma) \rightarrow C(\Gamma_1) \{q^{N-m-2-2j}\} \left\langle 1\right\rangle$ by
\[
\hat{\beta}_j = \sum_{k=0}^{N-m-2} \tau_{j,k} \circ \beta_k.
\]
Note that 
\[
C(\Gamma_1)\{[N-m-1]\} \left\langle 1 \right\rangle \cong \bigoplus_{j=0}^{N-m-2} C(\Gamma_1) \{q^{N-m-2-2j}\} \left\langle 1\right\rangle.
\]
We define $\vec{\alpha} : C(\Gamma_1)\{[N-m-1]\} \left\langle 1 \right\rangle \rightarrow C(\Gamma)$ by
\[
\vec{\alpha} = (\alpha_0,\dots,\alpha_{N-m-2}),
\]
and define $\vec{\beta} : C(\Gamma) \rightarrow C(\Gamma_1)\{[N-m-1]\} \left\langle 1 \right\rangle$ by
\[
\vec{\beta} = \left(%
\begin{array}{c}
    \hat{\beta}_0 \\
    \dots \\ 
    \hat{\beta}_{N-m-2}
\end{array}%
\right)
\]
It is easy to check that $\alpha_i$ and $\hat{\beta}_j$ are homogeneous morphisms preserving the $\zed_2\oplus\zed$-grading. So are $\vec{\alpha}$ and $\vec{\beta}$. 

Next we prove that $\vec{\beta} \circ \vec{\alpha} \simeq \id$. Consider 
\[
\hat{\beta}_{j}\circ\alpha_i=\sum_{k=0}^{N-m-2} \tau_{j,k} \circ (\beta_k\circ\alpha_i).
\] 
By \eqref{beta-scalar-choice} and the definition of $\tau_{j,k}$, it is easy to see that
\[
\hat{\beta}_j\circ\alpha_i \simeq \left\{%
\begin{array}{ll}
    \id & \text{if } ~i=j, \\ 
    0 & \text{if } ~i>j.
\end{array}%
\right.
\]
Now assume $i<j$. Again, by \eqref{beta-scalar-choice} and the definition of $\tau_{j,k}$, we have
\begin{eqnarray*}
&& \hat{\beta}_j\circ\alpha_i \\
& = & \sum_{k=0}^{N-m-2} \tau_{j,k} \circ (\beta_k\circ\alpha_i) \\ 
& \simeq & \sum_{k=i}^{j} \tau_{j,k} \circ (\beta_k\circ\alpha_i) \\
& = & \tau_{j,i} \circ (\beta_i\circ\alpha_i) + \tau_{j,j} \circ (\beta_j\circ\alpha_i) + \sum_{i<k<j} \tau_{j,k} \circ (\beta_k\circ\alpha_i) \\
&\simeq & \tau_{j,i} + \beta_j\circ\alpha_i \\
&& + \sum_{l \geq 1}\sum_{i<k<k_1<\cdots<k_{l-1}<j} (-1)^l (\beta_j\circ\alpha_{k_{l-1}}) \circ (\beta_{k_{l-1}}\circ\alpha_{k_{l-2}})\circ\cdots\circ (\beta_{k_1}\circ\alpha_k) \circ (\beta_k\circ\alpha_i) \\
& = & \tau_{j,i} - \tau_{j,i} =0
\end{eqnarray*}
Altogether, we have $\vec{\beta} \circ \vec{\alpha} \simeq \id$.
\end{proof}

\subsection{Proof of Theorem \ref{decomp-III}} With the morphisms constructed in the previous two subsections, we are now ready to prove Theorem \ref{decomp-III}. Our method is a generalization of that in \cite{KR1} and \cite{Wu7}.

\begin{lemma}\label{beta-F-G-alpha=0}
Let $\Gamma$, $\Gamma_0$ and $\Gamma_1$ be the MOY graphs in Figure \ref{decomposition-III-figure}. Suppose that $F$ and $G$ are the morphisms defined in Definition \ref{maps-Gamma-Gamma-0} (for $n=1$), and $\vec{\alpha}$ and $\vec{\beta}$ are the morphisms given in Proposition \ref{alpha-beta-vec-def}. Then $\vec{\beta} \circ F \simeq 0$ and $G \circ \vec{\alpha} \simeq 0$.
\end{lemma}

\begin{figure}[ht]

\setlength{\unitlength}{1pt}

\begin{picture}(360,75)(-180,-15)

\put(0,45){\vector(0,-1){30}}

\qbezier(0,45)(50,60)(50,45)

\qbezier(0,45)(-50,60)(-50,45)

\qbezier(0,15)(50,0)(50,15)

\qbezier(0,15)(-50,0)(-50,15)

\put(50,15){\vector(0,1){30}}

\put(-50,45){\vector(0,-1){30}}

\put(-60,47){\tiny{$1$}}

\put(55,47){\tiny{$m$}}

\put(5,30){\tiny{$m-1$}}

\put(-2,-15){$\Gamma_8$}

\end{picture}

\caption{}\label{special-butterfly}

\end{figure}

\begin{proof}
Let $\Gamma_8$ be the MOY graph in Figure \ref{special-butterfly}. Denote by $\overline{\Gamma}_0$ (resp. $\overline{\Gamma}_1$, $\overline{\Gamma}_8$) the MOY graph obtained by reversing the orientation of all edges of $\Gamma_0$ (resp. $\Gamma_1$, $\Gamma_8$.) Let $\bigcirc_m$ be a circle colored by $m$. Then 
\begin{eqnarray*}
\Hom_{HMF}(C(\Gamma_0),C(\Gamma_1)) & \cong & H(C(\Gamma_1)\otimes C(\overline{\Gamma}_0)) \{q^{m(N-m)+N-1}\} \left\langle m+1 \right\rangle  \\
& \cong & H(\Gamma_8) \{q^{m(N-m)+N-1}\} \left\langle m+1 \right\rangle \\
& \cong & H(\bigcirc_m) \{[m]\cdot q^{m(N-m)+N-1}\} \left\langle m+1 \right\rangle \\
& \cong & C(\emptyset) \{\qb{N}{m}\cdot[m]\cdot q^{m(N-m)+N-1}\} \left\langle 1 \right\rangle.
\end{eqnarray*}
In particular, the lowest non-vanishing quantum grading of $\Hom_{HMF}(C(\Gamma_0),C(\Gamma_1))$ is $N-m$.  But when viewed as a morphism $C(\Gamma_0) \rightarrow C(\Gamma_1)$, the quantum degree of $\hat{\beta}_j \circ F$ is $-N+m+2+2j$, which is less than $N-m$ for $j=0,1,\dots, N-m-2$. So $\hat{\beta}_j \circ F \simeq 0$ for $j=0,1,\dots, N-m-2$. That is, $\vec{\beta} \circ F \simeq 0$.

Clearly, the matrix factorization of $\Gamma_1$ is the same as that of $\Gamma_4'$ in Figure \ref{chi-maps-figure}. Similar to Lemma \ref{gen-chi-Gamma-0-lemma}, one can check that 
\[
C(\Gamma_1) \simeq M' := \left(%
\begin{array}{ll}
  \ast & (X_1-Y_1)+(s-t) \\
  \dots & \dots \\
  \ast & (X_k -Y_k) +(s-t) \sum_{l=0}^{k-1} (-t)^{k-1-l}X_l\\
  \dots & \dots \\
  \ast & (X_m -Y_m) +(s-t) \sum_{l=0}^{m-1} (-t)^{m-1-l}X_l\\
  \ast & \sum_{l=0}^{m} (-t)^{m-l}X_l
\end{array}%
\right)_{\Sym (\mathbb{X}|\mathbb{Y}|\{s\}|\{t\})}\{q^{-m+1}\},
\]
where $\mathbb{X},~\mathbb{Y},~\{s\},~\{t\}$ are markings of $\Gamma_4'$ in Figure \ref{chi-maps-figure}. Mark the corresponding end points of $\Gamma_0$ by the same alphabets. Then
\begin{eqnarray*}
\Hom_{HMF}(C(\Gamma_1),C(\Gamma_0)) & \cong & \Hom_{HMF}(M',C(\Gamma_0)) \\
& \cong & H(C(\Gamma_0)\otimes M'_\bullet) \\
& \cong & H(C(\Gamma_0)\otimes C(\overline{\Gamma}_1)) \{q^{m(N-m)+N-1}\}\left\langle m+1 \right\rangle \\
& \cong & H(\overline{\Gamma}_8) \{q^{(m(N-m)+N-1}\}\left\langle m+1 \right\rangle \\
& \cong & H(\Gamma_8) \{q^{(m(N-m)+N-1}\}\left\langle m+1 \right\rangle \\
& \cong & C(\emptyset) \{\qb{N}{m}\cdot[m]\cdot q^{m(N-m)+N-1}\} \left\langle 1 \right\rangle.
\end{eqnarray*}
In particular, the lowest non-vanishing quantum grading of $\Hom_{HMF}(C(\Gamma_1),C(\Gamma_0))$ is $N-m$.  But when viewed as a morphism $C(\Gamma_1) \rightarrow C(\Gamma_0)$, the quantum degree of $G \circ \alpha_j$ is $N-m-2-2j$, which is less than $N-m$ for $j=0,1,\dots, N-m-2$. So $G \circ \alpha_j \simeq 0$ for $j=0,1,\dots, N-m-2$. That is, $G \circ \vec{\alpha} \simeq 0$.
\end{proof}

Recall that the morphisms $F$ and $G$ are defined only up to scaling and homotopy, and, by Proposition \ref{Gamma-Gamma-0-proposition}, we have $G \circ F \approx \id_{C(\Gamma_0)}$. So, by choosing appropriate scalars, we can make
\begin{equation}\label{F-G-scaling}
G \circ F \simeq \id_{C(\Gamma_0)}.
\end{equation}
For minor technical convenience, we assume that \eqref{F-G-scaling} is true for the rest of this section.

\begin{lemma}\label{partial-decomp-III}
Let $\Gamma$, $\Gamma_0$ and $\Gamma_1$ be the MOY graphs in Figure \ref{decomposition-III-figure}. Then there exists a graded matrix factorization $M$, such that 
\[
C(\Gamma) \simeq C(\Gamma_0) \oplus C(\Gamma_1)\{[N-m-1]\} \left\langle 1 \right\rangle \oplus M.
\]
\end{lemma}

\begin{proof}
Define morphisms
\begin{eqnarray*}
\widetilde{F} & : & C(\Gamma_0) \oplus C(\Gamma_1)\{[N-m-1]\} \left\langle 1 \right\rangle \rightarrow C(\Gamma), \\
\widetilde{G} & : & C(\Gamma) \rightarrow C(\Gamma_0) \oplus C(\Gamma_1)\{[N-m-1]\} \left\langle 1 \right\rangle
\end{eqnarray*}
by
\[
\widetilde{F} = (F,\vec{\alpha}) \hspace{1cm} \text{and} \hspace{1cm} \widetilde{G} = \left(%
\begin{array}{c}
    G \\
    \vec{\beta}
\end{array}%
\right).
\]
Then, by Proposition \ref{Gamma-Gamma-0-proposition} (especially \eqref{F-G-scaling} above), Proposition \ref{alpha-beta-vec-def} and Lemma \ref{beta-F-G-alpha=0}, $\widetilde{F}$ and $\widetilde{G}$ are homogeneous morphisms preserving the $\zed_2\oplus\zed$-grading and satisfy 
\[
\widetilde{G} \circ \widetilde{F} \simeq \id_{C(\Gamma_0) \oplus C(\Gamma_1)\{[N-m-1]\} \left\langle 1 \right\rangle}. 
\]
Therefore, $\widetilde{F} \circ \widetilde{G} : C(\Gamma) \rightarrow C(\Gamma)$ preserves the $\zed_2\oplus\zed$-grading and satisfies
\[
(\widetilde{F} \circ \widetilde{G}) \circ (\widetilde{F} \circ \widetilde{G}) \simeq \widetilde{F} \circ \widetilde{G}.
\]
By Lemma \ref{fully-additive-hmf}, there exists a graded matrix factorization $M$ such that
\[
C(\Gamma) \simeq C(\Gamma_0) \oplus C(\Gamma_1)\{[N-m-1]\} \left\langle 1 \right\rangle \oplus M.
\]
\end{proof}

\begin{lemma}\label{M=0-decomp-III}
Let $M$ be as in Lemma \ref{partial-decomp-III}. Then $M \simeq 0$.
\end{lemma}

\begin{figure}[ht]

\setlength{\unitlength}{1pt}

\begin{picture}(360,75)(-185,-15)

\put(-170,0){\vector(1,1){20}}

\put(-150,20){\vector(1,0){20}}

\put(-130,20){\vector(0,1){20}}

\put(-130,20){\vector(1,-1){20}}

\put(-130,40){\vector(-1,0){20}}

\put(-150,40){\vector(0,-1){20}}

\put(-150,40){\vector(-1,1){20}}

\put(-110,60){\vector(-1,-1){20}}

\put(-164,0){\tiny{$1$}}

\put(-164,55){\tiny{$1$}}

\put(-134,28){\tiny{$1$}}

\put(-121,0){\tiny{$m$}}

\put(-121,55){\tiny{$m$}}

\put(-148,28){\tiny{$m$}}

\put(-148.5,43){\tiny{$m+1$}}

\put(-148.5,13){\tiny{$m+1$}}

\put(-185,0){\small{$\{r\}$}}

\put(-185,55){\small{$\{t\}$}}

\put(-108,0){\small{$\mathbb{X}$}}

\put(-108,55){\small{$\mathbb{Z}$}}

\put(-151,30){\line(1,0){2}}

\put(-160,28){\small{$\mathbb{Y}$}}

\put(-131,30){\line(1,0){2}}

\put(-128,28){\small{$\{s\}$}}

\put(-142,-15){$\Gamma$}


\put(-20,0){\vector(0,1){60}}

\put(20,60){\vector(0,-1){60}}

\put(-25,28){\tiny{$1$}}

\put(22,28){\tiny{$m$}}

\put(-35,0){\small{$\{r\}$}}

\put(-35,55){\small{$\{t\}$}}

\put(24,0){\small{$\mathbb{X}$}}

\put(24,55){\small{$\mathbb{Z}$}}

\put(-2,-15){$\Gamma_0$}


\put(110,0){\vector(1,1){20}}

\put(130,20){\vector(1,-1){20}}

\put(130,40){\vector(0,-1){20}}

\put(130,40){\vector(-1,1){20}}

\put(150,60){\vector(-1,-1){20}}

\put(118,0){\tiny{$1$}}

\put(118,55){\tiny{$1$}}

\put(138,0){\tiny{$m$}}

\put(138,55){\tiny{$m$}}

\put(134,28){\tiny{$m-1$}}

\put(95,0){\small{$\{r\}$}}

\put(95,55){\small{$\{t\}$}}

\put(152,0){\small{$\mathbb{X}$}}

\put(152,55){\small{$\mathbb{Z}$}}

\put(129,30){\line(1,0){2}}

\put(118,28){\small{$\mathbb{W}$}}

\put(128,-15){$\Gamma_1$}

\end{picture}

\caption{}\label{decomposition-III-figure-marked}

\end{figure}

\begin{proof}
Mark $\Gamma$, $\Gamma_0$ and $\Gamma_1$ as in Figure \ref{decomposition-III-figure-marked}. 

Consider homology of matrix factorizations with non-vanishing potentials defined in Definition \ref{homology-matrix-factorization-def}. By Corollary \ref{homology-detects-homotopy}, to show $M\simeq 0$, we only need to show that $H(M)=0$, or, equivalently, $\gdim (M)=0$. But, by Lemma \ref{partial-decomp-III}, we have 
\[
H(\Gamma) \cong H(\Gamma_0) \oplus H(\Gamma_1)\{[N-m-1]\} \left\langle 1 \right\rangle \oplus H(M).
\]
So,
\[
\gdim (C(\Gamma)) = \gdim(C(\Gamma_0)) + \tau\cdot[N-m-1]\cdot\gdim(C(\Gamma_1)) + \gdim(M).
\]
Therefore, to prove the lemma, we only need to show that
\begin{equation}\label{gdim-decomp-III}
\gdim (C(\Gamma)) = \gdim(C(\Gamma_0)) + \tau\cdot[N-m-1]\cdot\gdim(C(\Gamma_1))
\end{equation}
In the rest of this proof, we prove \eqref{gdim-decomp-III} by directly computing $\gdim (C(\Gamma))$, $\gdim(C(\Gamma_0))$ and $\gdim(C(\Gamma_1))$.

We compute $\gdim (C(\Gamma))$ first. Let $\mathbb{A}=\mathbb{X}\cup \{s\}$, $\mathbb{B}=\mathbb{Y}\cup \{r\}$, $\mathbb{D}=\mathbb{Y}\cup \{t\}$, $\mathbb{E}=\mathbb{Z}\cup \{s\}$. By Lemma \ref{edge-contraction}, we contract the two edges in $\Gamma$ of color $m+1$ and get
\[
C(\Gamma) \simeq \left(%
\begin{array}{cc}
  U_1 & A_1-B_1 \\
  \dots & \dots \\
  U_{m+1} & A_{m+1}-B_{m+1} \\
  V_1 & D_1-E_1 \\
  \dots & \dots \\
  V_{m+1} & D_{m+1}-E_{m+1} \\
\end{array}%
\right)_{\Sym(\mathbb{X}|\mathbb{Y}|\mathbb{Z}|\{r\}|\{s\}|\{t\})} \{q^{-2m}\},
\]
where $A_j$ is the $j$-th elementary symmetric function in $\mathbb{A}$ and so on, and
\begin{eqnarray*}
U_j & = & \frac{p_{m+1,N+1}(B_1,\dots,B_{j-1},A_j,\dots,A_{m+1}) - p_{m+1,N+1}(B_1,\dots,B_{j},A_{j+1},\dots,A_{m+1})}{A_j-B_j}, \\
V_j & = & \frac{p_{m+1,N+1}(E_1,\dots,E_{j-1},D_j,\dots,D_{m+1}) - p_{m+1,N+1}(E_1,\dots,E_{j},D_{j+1},\dots,D_{m+1})}{D_j-E_j}.
\end{eqnarray*}
Recall that $C(\Gamma)$ is viewed as a matrix factorization over $\Sym(\mathbb{X}|\mathbb{Z}|\{r\}|\{t\})$. So the corresponding maximal ideal for $C(\Gamma)$ is the ideal $\mathfrak{I}=(X_1,\dots,X_m,Z_1,\dots,Z_m,r,t)$ of $\Sym(\mathbb{X}|\mathbb{Z}|\{r\}|\{t\})$. Identify 
\[
\Sym(\mathbb{X}|\mathbb{Y}|\mathbb{Z}|\{r\}|\{s\}|\{t\})/\mathfrak{I}\cdot \Sym(\mathbb{X}|\mathbb{Y}|\mathbb{Z}|\{r\}|\{s\}|\{t\}) = \Sym (\mathbb{Y}|\{s\})
\]
by the relations
\begin{equation}\label{max-ideal-rel-gamma}
X_1=\cdots=X_m=Z_1=\cdots=Z_m=r=t=0.
\end{equation}
Then 
\begin{eqnarray*}
C(\Gamma)/\mathfrak{I}\cdot C(\Gamma) & \simeq & \left(%
\begin{array}{cc}
  U_1 & s-Y_1 \\
  U_2 & -Y_2 \\
  \dots & \dots \\
  U_m & -Y_m \\
  U_{m+1} & 0 \\
  V_1 & Y_1 - s \\
  V_2 & Y_2 \\
  \dots & \dots \\
  V_m & Y_m \\
  V_{m+1} & 0 \\\end{array}%
\right)_{\Sym(\mathbb{Y}|\{s\})} \{q^{-2m}\} \\
& \simeq & \left(%
\begin{array}{cc}
  U_{m+1} & 0 \\
  V_1 & 0 \\
  V_2 & 0 \\
  \dots & \dots \\
  V_m & 0 \\
  V_{m+1} & 0 \\\end{array}%
\right)_{\C[s]} \{q^{-2m}\},
\end{eqnarray*}
where we applied Proposition \ref{b-contraction} successively to the first $m$ rows. This gives the relations
\begin{equation}\label{max-ideal-rel-2-gamma}
Y_1-s=Y_2=\cdots=Y_m=0.
\end{equation}
Under relations \eqref{max-ideal-rel-gamma} and \eqref{max-ideal-rel-2-gamma}, we have 
\[
A_j=B_j=D_j=E_j= \left\{%
\begin{array}{ll}
    s & \text{if } ~j=1, \\ 
    0 & \text{if } ~i=2,\dots,m+1.
\end{array}%
\right.
\]
So, by Lemma \ref{power-derive}, we have
\begin{eqnarray*}
U_j=V_j & = & \frac{\partial p_{m+1,N+1}(A_1,\dots,A_{m+1})}{\partial A_j} |_{A_1=s,A_2=\cdots+A_{m+1}=0} \\ 
& = & (-1)^j(N+1) h_{m+1,N+1-j}(s,0,\dots,0) \\
& = & (-1)^j(N+1) s^{N+1-j}.
\end{eqnarray*}
Using Lemma \ref{jumping-factor} and Corollary \ref{a-contraction-weak}, we then have
\begin{eqnarray*}
H(\Gamma) & \cong & H( \left(%
\begin{array}{cc}
  s^{N-m} & 0 \\
  s^N & 0 \\
  s^{N-1} & 0 \\
  \dots & \dots \\
  s^{N-m} & 0 \\\end{array}%
\right)_{\C[s]}) \{q^{-2m}\} \\
& \cong & H( \left(%
\begin{array}{cc}
  0_N & 0 \\
  0_{N-1} & 0 \\
  \dots & \dots \\
  0_{N-m} & 0 \\\end{array}%
\right)_{\C[s]/(s^{N-m})}) \{q^{1-N}\} \left\langle 1 \right\rangle \\
& \cong & \left(%
\begin{array}{cc}
  0_N & 0 \\
  0_{N-1} & 0 \\
  \dots & \dots \\
  0_{N-m} & 0 \\\end{array}%
\right)_{\C[s]/(s^{N-m})} \{q^{1-N}\} \left\langle 1 \right\rangle,
\end{eqnarray*}
where $0_j$ is ``a $0$ that has degree $2j$". 
So
\begin{eqnarray*}
\gdim(C(\Gamma)) & = & \tau \cdot q^{1-N} \cdot (\sum_{k=0}^{N-m-1} q^{2k}) \cdot \prod_{j=1}^{m+1} (1+\tau q^{2j-N-1}) \\
& = & \tau \cdot q^{-m} \cdot [N-m] \cdot \prod_{j=1}^{m+1} (1+\tau q^{2j-N-1}).
\end{eqnarray*}

Next, we compute $\gdim(C(\Gamma_0))$. Let 
\begin{eqnarray*}
\hat{U} & = & \frac{t^{N+1} - r^{N+1}}{t-r} \\
\hat{U}_j & = & \frac{p_{m,N+1}(Z_1,\dots,Z_{j-1},X_j,\dots,X_{m}) - p_{m,N+1}(Z_1,\dots,Z_{j},X_{j+1},\dots,X_{m})}{X_j-Z_j}
\end{eqnarray*}
Then 
\[
C(\Gamma_0) = \left(%
\begin{array}{cc}
  \hat{U} & t-r \\
  \hat{U}_1 & X_1-Z_1 \\
  \dots & \dots \\
  \hat{U}_m & X_m-Z_m \\
\end{array}%
\right)_{\Sym(\mathbb{X}|\mathbb{Z}|\{r\}|\{t\})}
\]
So
\[
C(\Gamma_0)/\mathfrak{I} \cdot C(\Gamma_0) \cong \left(%
\begin{array}{cc}
  0_N & 0 \\
  0_N & 0 \\
  0_{N-1} & 0 \\
  \dots & \dots \\
  0_{N-m+1} & 0 \\
\end{array}%
\right)_{\C}
\]
and 
\[
\gdim(C(\Gamma_0)) = (1+\tau q^{1-N}) \cdot \prod_{j=1}^{m} (1+\tau q^{2j-N-1}).
\]

Now we compute $\gdim(C(\Gamma_1))$. Let $\mathbb{F}=\mathbb{W}\cup \{t\}$ and $\mathbb{G}=\mathbb{W} \cup \{r\}$. Define 
\begin{eqnarray*}
\bar{U}_j & = & \frac{p_{m,N+1}(Z_1,\dots,Z_{j-1},F_j,\dots,F_{m}) - p_{m,N+1}(Z_1,\dots,Z_{j},F_{j+1},\dots,F_{m})}{F_j-Z_j}, \\
\bar{V}_j & = & \frac{p_{m,N+1}(G_1,\dots,G_{j-1},X_j,\dots,X_{m}) - p_{m,N+1}(G_1,\dots,G_{j},X_{j+1},\dots,X_{m})}{X_j-G_j}.
\end{eqnarray*}
Then 
\[
C(\Gamma_1) = \left(%
\begin{array}{cc}
  \bar{U}_1 & F_1-Z_1 \\
  \dots & \dots \\
  \bar{U}_m & F_m-Z_m \\
  \bar{V}_1 & X_1-G_1 \\
  \dots & \dots \\
  \bar{V}_m & X_m-G_m \\
\end{array}%
\right)_{\Sym(\mathbb{X}|\mathbb{Z}|\mathbb{W}|\{r\}|\{t\})} \{q^{1-m}\}.
\]
Identify 
\[
\Sym(\mathbb{X}|\mathbb{Z}|\mathbb{W}|\{r\}|\{t\})/ \mathfrak{I}\cdot \Sym(\mathbb{X}|\mathbb{Z}|\mathbb{W}|\{r\}|\{t\}) = \Sym(\mathbb{W})
\]
by relations \eqref{max-ideal-rel-gamma}. Then, by Proposition \ref{b-contraction},
\begin{eqnarray*}
C(\Gamma_1)/\mathfrak{I}\cdot C(\Gamma_1) & \cong & \left(%
\begin{array}{cc}
  \bar{U}_1 & W_1 \\
  \dots & \dots \\
  \bar{U}_{m-1} & W_{m-1} \\
  \bar{U}_m & 0 \\
  \bar{V}_1 & -W_1 \\
  \dots & \dots \\
  \bar{V}_{m-1} & -W_{m-1} \\
  \bar{V}_m & 0 \\
\end{array}%
\right)_{\Sym(\mathbb{W})} \{q^{1-m}\} \\
& \simeq & \left(%
\begin{array}{cc}
  0_{N+1-m} & 0 \\
  0_N & 0 \\
  0_{N-1} & 0 \\
  \dots & \dots \\
  0_{N-m+1} & 0 \\
\end{array}%
\right)_{\C} \{q^{1-m}\}.
\end{eqnarray*}
So 
\[
\gdim(C(\Gamma_1)) = q^{1-m}\cdot(1+\tau q^{2m-N-1}) \cdot \prod_{j=1}^{m} (1+\tau q^{2j-N-1}).
\]

Write 
\[
P = \prod_{j=1}^{m} (1+\tau q^{2j-N-1}).
\]
Then 
\begin{eqnarray*}
\gdim(C(\Gamma)) & = & \tau \cdot q^{-m} \cdot [N-m] \cdot (1+\tau q^{2m-N+1}) \cdot P, \\
\gdim(C(\Gamma_0)) & = & (1+\tau q^{1-N}) \cdot P, \\
\gdim(C(\Gamma_1)) & = & q^{1-m}\cdot(1+\tau q^{2m-N-1}) \cdot P.
\end{eqnarray*}
Note that
\[
[N-m] = [N-m-1]\cdot q +q^{-(N-m-1)}.
\]
So,
\begin{eqnarray*}
&& \gdim (C(\Gamma)) - \gdim(C(\Gamma_0)) - \tau\cdot[N-m-1]\cdot\gdim(C(\Gamma_1)) \\
& = & ((\tau \cdot q^{-m} + q^{m-N+1})\cdot [N-m] -1-\tau \cdot q^{1-N} -(q^{1-m} +\tau \cdot q^{m-N})\cdot[N-m-1])\cdot P \\
& = & ((\tau \cdot q^{-m} + q^{m-N+1})\cdot ([N-m-1]\cdot q +q^{-(N-m-1)}) -1-\tau \cdot q^{1-N} \\
&& -(q^{1-m} +\tau \cdot q^{m-N})\cdot[N-m-1])\cdot P \\
& = & ([N-m-1]\cdot(q-q^{-1})\cdot q^{m-N+1} +q^{2(m-N+1)}-1)\cdot P \\
& = & 0.
\end{eqnarray*}
This shows that \eqref{gdim-decomp-III} is true.
\end{proof}

\begin{proof}[Proof of Theorem \ref{decomp-III}]
Lemmas \ref{partial-decomp-III} and \ref{M=0-decomp-III} imply Theorem \ref{decomp-III}.
\end{proof}

\section{Direct Sum Decomposition (IV)}\label{sec-MOY-IV}

The objective of this section is to prove Theorem \ref{decomp-IV}, which categorifies \cite[Lemma A.7]{MOY} and generalizes direct sum decomposition (IV) in \cite{KR1}.

\begin{figure}[ht]

\setlength{\unitlength}{1pt}

\begin{picture}(360,105)(-180,-15)


\put(-160,0){\vector(0,1){25}}

\put(-160,65){\vector(0,1){25}}

\put(-160,25){\vector(0,1){40}}

\put(-160,65){\vector(1,0){40}}

\put(-120,25){\vector(0,1){40}}

\put(-120,25){\vector(-1,0){40}}

\put(-120,0){\vector(0,1){25}}

\put(-120,65){\vector(0,1){25}}

\put(-140,64){\line(0,1){2}}

\put(-140,24){\line(0,1){2}}

\put(-161,40){\line(1,0){2}}

\put(-121,40){\line(1,0){2}}

\put(-167,13){\tiny{$1$}}

\put(-167,72){\tiny{$l$}}

\put(-177,48){\tiny{$l+n$}}

\put(-117,13){\tiny{$m+l-1$}}

\put(-117,72){\tiny{$m$}}

\put(-117,48){\tiny{$m-n$}}

\put(-153,28){\tiny{$l+n-1$}}

\put(-142,58){\tiny{$n$}}

\put(-117,82){\small{$\mathbb{X}$}}

\put(-117,37){\small{$\mathbb{Y}$}}

\put(-117,0){\small{$\mathbb{W}$}}

\put(-170,82){\small{$\mathbb{T}$}}

\put(-170,37){\small{$\mathbb{S}$}}

\put(-175,0){\small{$\{r\}$}}

\put(-142,-15){$\Gamma$}

\put(-143,67){\small{$\mathbb{A}$}}

\put(-143,15){\small{$\mathbb{B}$}}


\put(-20,0){\vector(0,1){45}}

\put(-20,45){\vector(0,1){45}}

\put(20,0){\vector(0,1){45}}

\put(20,45){\vector(0,1){45}}

\put(20,45){\vector(-1,0){40}}

\put(0,44){\line(0,1){2}}

\put(-27,20){\tiny{$1$}}

\put(23,20){\tiny{$m+l-1$}}

\put(-27,65){\tiny{$l$}}

\put(23,65){\tiny{$m$}}

\put(-5,38){\tiny{$l-1$}}

\put(23,82){\small{$\mathbb{X}$}}

\put(23,0){\small{$\mathbb{W}$}}

\put(-30,82){\small{$\mathbb{T}$}}

\put(-35,0){\small{$\{r\}$}}

\put(-3,48){\small{$\mathbb{D}$}}

\put(-2,-15){$\Gamma_0$}


\put(110,0){\vector(2,3){20}}

\put(150,0){\vector(-2,3){20}}

\put(130,30){\vector(0,1){30}}

\put(130,60){\vector(-2,3){20}}

\put(130,60){\vector(2,3){20}}

\put(117,20){\tiny{$1$}}

\put(140,20){\tiny{$m+l-1$}}

\put(117,65){\tiny{$l$}}

\put(140,65){\tiny{$m$}}

\put(133,42){\tiny{$m+l$}}

\put(153,82){\small{$\mathbb{X}$}}

\put(153,0){\small{$\mathbb{W}$}}

\put(100,82){\small{$\mathbb{T}$}}

\put(95,0){\small{$\{r\}$}}

\put(128,-15){$\Gamma_1$}

\end{picture}

\caption{}\label{decomposition-IV-figure}

\end{figure}

\begin{theorem}\label{decomp-IV}
Let $\Gamma$, $\Gamma_0$ and $\Gamma_1$ be the MOY graphs in Figure \ref{decomposition-IV-figure}, where $l,m,n$ are integers satisfying $0\leq n \leq m \leq N$ and $0\leq l, m+l-1 \leq N$. Then
\begin{equation}\label{decomp-IV-main}
C(\Gamma) \simeq C(\Gamma_0)\{\qb{m-1}{n}\} \oplus C(\Gamma_1)\{\qb{m-1}{n-1}\}.
\end{equation}

Similarly, if $\overline{\Gamma}$, $\overline{\Gamma}_0$ and $\overline{\Gamma}_1$ are $\Gamma$, $\Gamma_0$, $\Gamma_1$ with the orientation of every edge reversed, then 
\begin{equation}\label{decomp-IV-opposite}
C(\overline{\Gamma}) \simeq C(\overline{\Gamma}_0)\{\qb{m-1}{n}\} \oplus C(\overline{\Gamma}_1)\{\qb{m-1}{n-1}\}.
\end{equation}
\end{theorem}

The proofs of decompositions \eqref{decomp-IV-main} and \eqref{decomp-IV-opposite} are almost identical. So we only prove \eqref{decomp-IV-main} in this paper and leave \eqref{decomp-IV-opposite} to the reader. 
\begin{figure}[ht]

\setlength{\unitlength}{1pt}

\begin{picture}(360,105)(-180,-15)


\put(-160,0){\vector(0,1){25}}

\put(-160,65){\vector(0,1){25}}

\put(-160,25){\vector(0,1){40}}

\put(-160,65){\vector(1,0){40}}

\put(-120,25){\vector(0,1){40}}

\put(-120,25){\vector(-1,0){40}}

\put(-120,0){\vector(0,1){25}}

\put(-120,65){\vector(0,1){25}}

\put(-167,13){\tiny{$1$}}

\put(-167,72){\tiny{$1$}}

\put(-167,48){\tiny{$2$}}

\put(-117,13){\tiny{$2$}}

\put(-117,72){\tiny{$2$}}

\put(-117,48){\tiny{$1$}}

\put(-142,28){\tiny{$1$}}

\put(-142,58){\tiny{$1$}}

\put(-142,-15){$\Gamma'$}


\put(-20,0){\vector(0,1){90}}

\put(20,0){\vector(0,1){90}}

\put(-27,45){\tiny{$1$}}

\put(23,45){\tiny{$2$}}

\put(-2,-15){$\Gamma_0'$}


\put(110,0){\vector(2,3){20}}

\put(150,0){\vector(-2,3){20}}

\put(130,30){\vector(0,1){30}}

\put(130,60){\vector(-2,3){20}}

\put(130,60){\vector(2,3){20}}

\put(117,20){\tiny{$1$}}

\put(140,20){\tiny{$2$}}

\put(117,65){\tiny{$1$}}

\put(140,65){\tiny{$2$}}

\put(133,42){\tiny{$3$}}

\put(128,-15){$\Gamma_1'$}

\end{picture}

\caption{}\label{decomposition-IV-special-figure}

\end{figure}

\begin{remark}
Although direct sum decomposition (IV) is formulated in a different form in \cite{KR1}, its proof there comes down to establishing the decomposition
\begin{equation}\label{decomp-IV-special}
C(\Gamma') \simeq C(\Gamma_0') \oplus C(\Gamma_1'),
\end{equation}
where $\Gamma'$, $\Gamma_0'$ and $\Gamma_1'$ are given in Figure \ref{decomposition-IV-special-figure}. This is also what is actually used in the proof of the invariance of the $\mathfrak{sl}(N)$ Khovanov-Rozansky homology under Reidemeister move III. Clearly, if we specify that $l=n=1, m=2$ in Theorem \ref{decomp-IV}, then we get decomposition \eqref{decomp-IV-special}.
\end{remark}

To prove Theorem \ref{decomp-IV}, we need the following special case of Proposition \ref{general-general-chi-maps}.

\begin{figure}[ht]

\setlength{\unitlength}{1pt}

\begin{picture}(360,75)(-180,-15)


\put(-120,0){\vector(1,1){20}}

\put(-100,20){\vector(1,-1){20}}

\put(-100,40){\vector(0,-1){20}}

\put(-100,40){\vector(1,1){20}}

\put(-120,60){\vector(1,-1){20}}

\put(-101,30){\line(1,0){2}}

\put(-127,45){\tiny{$n+1$}}

\put(-127,15){\tiny{$m-n$}}

\put(-90,45){\tiny{$1$}}

\put(-90,15){\tiny{$m$}}

\put(-95,28){\tiny{$n$}}

\put(-130,55){\small{$\mathbb{S}$}}

\put(-75,0){\small{$\mathbb{X}$}}

\put(-113,27){\small{$\mathbb{A}$}}

\put(-130,0){\small{$\mathbb{Y}$}}

\put(-75,55){\small{$\{r\}$}}

\put(-102,-15){$\Gamma_4$}


\put(-30,35){\vector(1,0){60}}

\put(30,25){\vector(-1,0){60}}

\put(-3,40){\small{$\chi^0$}}

\put(-3,15){\small{$\chi^1$}}


\put(60,10){\vector(1,1){20}}

\put(60,50){\vector(1,-1){20}}

\put(80,30){\vector(1,0){20}}

\put(100,30){\vector(1,1){20}}

\put(100,30){\vector(1,-1){20}}

\put(68,45){\tiny{$n+1$}}

\put(70,15){\tiny{$m-n$}}

\put(108,45){\tiny{$1$}}

\put(106,15){\tiny{$m$}}

\put(81,32){\tiny{$m+1$}}

\put(50,45){\small{$\mathbb{S}$}}

\put(125,10){\small{$\mathbb{X}$}}

\put(50,10){\small{$\mathbb{Y}$}}

\put(122,45){\small{$\{r\}$}}

\put(88,-15){$\Gamma_5$}

\end{picture}

\caption{}\label{general-chi-maps-figure}

\end{figure}

\begin{corollary}\label{general-chi-maps-def}
Let $\Gamma_4$ and $\Gamma_5$ be the MOY graphs in Figure \ref{general-chi-maps-figure}, where $m,n$ are integers such that $0\leq n \leq m \leq N$. Then there exist homogeneous morphisms 
\begin{eqnarray*}
\chi^0: C(\Gamma_4) \rightarrow C(\Gamma_5), && \\
\chi^1: C(\Gamma_5) \rightarrow C(\Gamma_4), &&
\end{eqnarray*}
both of quantum degree $m-n$ and $\zed_2$-degree $0$ such that
\begin{eqnarray*}
\chi^1 \circ \chi^0 & \simeq & (\sum_{k=0}^{m-n} (-r)^{m-n-k} Y_k) \cdot \id_{C(\Gamma_4)}, \\
\chi^0 \circ \chi^1 & \simeq & (\sum_{k=0}^{m-n} (-r)^{m-n-k} Y_k) \cdot \id_{C(\Gamma_5)},
\end{eqnarray*}
where $Y_k$ is the $k$-th elementary symmetric polynomial in $\mathbb{Y}$.
\end{corollary}

\subsection{Relating $\Gamma$ and $\Gamma_0$}\label{decomp-IV-subsection-Gamma-Gamma0}

\begin{figure}[ht]

\setlength{\unitlength}{1pt}

\begin{picture}(360,255)(-180,-15)


\put(-110,150){\vector(0,1){25}}

\put(-110,215){\vector(0,1){25}}

\put(-110,175){\vector(0,1){40}}

\put(-110,215){\vector(1,0){40}}

\put(-70,175){\vector(0,1){40}}

\put(-70,175){\vector(-1,0){40}}

\put(-70,157){\vector(0,1){25}}

\put(-70,215){\vector(0,1){25}}

\put(-90,214){\line(0,1){2}}

\put(-90,174){\line(0,1){2}}

\put(-111,190){\line(1,0){2}}

\put(-71,190){\line(1,0){2}}

\put(-117,163){\tiny{$1$}}

\put(-117,222){\tiny{$l$}}

\put(-127,198){\tiny{$l+n$}}

\put(-67,163){\tiny{$m+l-1$}}

\put(-67,222){\tiny{$m$}}

\put(-67,198){\tiny{$m-n$}}

\put(-103,178){\tiny{$l+n-1$}}

\put(-92,208){\tiny{$n$}}

\put(-67,232){\small{$\mathbb{X}$}}

\put(-67,187){\small{$\mathbb{Y}$}}

\put(-67,150){\small{$\mathbb{W}$}}

\put(-120,232){\small{$\mathbb{T}$}}

\put(-120,187){\small{$\mathbb{S}$}}

\put(-125,150){\small{$\{r\}$}}

\put(-93,217){\small{$\mathbb{A}$}}

\put(-93,165){\small{$\mathbb{B}$}}

\put(-150,190){$\Gamma$}

\put(-100,140){\vector(-1,-1){30}}

\put(-125,125){\small{$\chi^1$}}

\put(-120,110){\vector(1,1){30}}

\put(-100,125){\small{$\chi^0$}}

\put(-30,190){\vector(1,0){60}}

\put(30,200){\vector(-1,0){60}}

\put(-3,204){\small{$f$}}

\put(-3,180){\small{$g$}}


\put(70,150){\vector(0,1){45}}

\put(70,195){\vector(0,1){45}}

\put(110,150){\vector(0,1){45}}

\put(110,195){\vector(0,1){45}}

\put(110,195){\vector(-1,0){40}}

\put(63,170){\tiny{$1$}}

\put(113,170){\tiny{$m+l-1$}}

\put(63,215){\tiny{$l$}}

\put(113,215){\tiny{$m$}}

\put(85,188){\tiny{$l-1$}}

\put(113,232){\small{$\mathbb{X}$}}

\put(113,150){\small{$\mathbb{W}$}}

\put(60,232){\small{$\mathbb{T}$}}

\put(55,150){\small{$\{r\}$}}

\put(150,190){$\Gamma_0$}

\put(90,140){\vector(1,-1){30}}

\put(125,125){\small{$\overline{\phi}$}}

\put(130,110){\vector(-1,1){30}}

\put(90,125){\small{$\phi$}}


\put(-160,0){\vector(0,1){45}}

\put(-160,45){\vector(0,1){45}}

\put(-130,45){\vector(-1,0){30}}

\put(-130,45){\vector(2,1){30}}

\put(-100,30){\vector(-2,1){30}}

\put(-100,0){\vector(0,1){30}}

\put(-100,30){\vector(0,1){30}}

\put(-100,60){\vector(0,1){30}}

\put(-101,45){\line(1,0){2}}

\put(-115,51.5){\line(0,1){2}}

\put(-115,36.5){\line(0,1){2}}

\put(-167,13){\tiny{$1$}}

\put(-167,72){\tiny{$l$}}

\put(-97,13){\tiny{$m+l-1$}}

\put(-97,72){\tiny{$m$}}

\put(-97,48){\tiny{$m-n$}}

\put(-137,30){\tiny{$l+n-1$}}

\put(-127,51){\tiny{$n$}}

\put(-150,38){\tiny{$l-1$}}

\put(-142,-15){$\Gamma_{10}$}

\put(-97,82){\small{$\mathbb{X}$}}

\put(-97,37){\small{$\mathbb{Y}$}}

\put(-97,0){\small{$\mathbb{W}$}}

\put(-170,82){\small{$\mathbb{T}$}}

\put(-175,0){\small{$\{r\}$}}

\put(-115,57){\small{$\mathbb{A}$}}

\put(-115,40){\small{$\mathbb{B}$}}

\put(-40,40){\vector(1,0){80}}

\put(40,50){\vector(-1,0){80}}

\put(-3,54){\small{$h_0$}}

\put(-3,30){\small{$h_1$}}


\put(80,0){\vector(0,1){30}}

\put(80,30){\vector(0,1){60}}

\put(140,30){\vector(-1,0){60}}

\put(140,75){\vector(0,1){15}}

\put(140,30){\vector(0,1){15}}

\put(130,50){\vector(0,1){20}}

\put(150,50){\vector(0,1){20}}

\put(140,0){\vector(0,1){30}}

\qbezier(140,75)(130,75)(130,70)

\qbezier(140,75)(150,75)(150,70)

\qbezier(140,45)(130,45)(130,50)

\qbezier(140,45)(150,45)(150,50)

\put(129,58){\line(1,0){2}}

\put(149,58){\line(1,0){2}}

\put(73,13){\tiny{$1$}}

\put(73,72){\tiny{$l$}}

\put(153,62){\tiny{$m-n$}}

\put(120,62){\tiny{$n$}}

\put(143,13){\tiny{$m+l-1$}}

\put(130,80){\tiny{$m$}}

\put(143,35){\tiny{$m$}}

\put(105,23){\tiny{$l-1$}}

\put(108,-15){$\Gamma_{11}$}

\put(143,82){\small{$\mathbb{X}$}}

\put(143,0){\small{$\mathbb{W}$}}

\put(153,52){\small{$\mathbb{Y}$}}

\put(70,82){\small{$\mathbb{T}$}}

\put(65,0){\small{$\{r\}$}}

\put(120,52){\small{$\mathbb{A}$}}

\end{picture}

\caption{}\label{decomposition-IV-f-g-def}

\end{figure}

Consider the diagram in Figure \ref{decomposition-IV-f-g-def}, in which
\begin{itemize}
	\item $\phi$ and $\overline{\phi}$ are the morphisms associated to the apparent edge splitting and merging,
	\item $h_0$ and $h_1$ are the homotopy equivalences induced by the apparent bouquet moves and are inverses of each other,
	\item $\chi^0$ and $\chi^1$ are the morphisms from applying Corollary \ref{general-chi-maps-def} to the left half of $\Gamma$.
\end{itemize}
All these morphisms are $\Sym(\mathbb{X}|\mathbb{W}|\mathbb{T}|\{r\})$-linear. Moreover, $h_0$, $h_1$, $\chi^0$ and $\chi^1$ are also $\Sym(\mathbb{A}|\mathbb{Y})$-linear. By Corollary \ref{general-chi-maps-def}, we know that
\begin{equation}\label{Gamma-0-eq1}
\chi^1 \circ \chi^0 = (\sum_{k=0}^n (-r)^k A_{n-k}) \cdot \id_{C(\Gamma_{10})}.
\end{equation}

\begin{definition}\label{decomp-IV-f-g-def}
Define $f: C(\Gamma_0) \rightarrow C(\Gamma)$ by $f=\chi^0 \circ h_0 \circ \phi$ and $g: C(\Gamma) \rightarrow C(\Gamma_0)$ by $g=\overline{\phi} \circ h_1 \circ \chi^1$.

Note that $f$ and $g$ are both homogeneous morphisms of quantum degree $-n(m-n-1)$ and $\zed_2$-degree $0$.
\end{definition}

\begin{definition}\label{decomp-IV-f-g-lambda-def}
Let $\Lambda=\Lambda_{n,m-n-1} =\{\lambda~|~l(\lambda)\leq n,~\lambda_1\leq m-n-1\}$. For $\lambda=(\lambda_1\geq\cdots\geq\lambda_n) \in \Lambda$, define $\lambda^c=(\lambda^c_1\geq\cdots\geq\lambda^c_n) \in \Lambda$ by $\lambda^c_j =m-n-1-\lambda_{n+1-j}$ for $j=1,\dots,n$.

For $\lambda \in \Lambda$, define $f_{\lambda}: C(\Gamma_0) \rightarrow C(\Gamma)$ by $f_{\lambda} = \mathfrak{m}(S_{\lambda}(\mathbb{A})) \circ f$, where $S_{\lambda}(\mathbb{A})$ is the Schur polynomial in $\mathbb{A}$ associated to $\lambda$, and $\mathfrak{m}(S_{\lambda}(\mathbb{A}))$ is the morphism given by the multiplication by $S_{\lambda}(\mathbb{A})$. Then $f_{\lambda}$ is a homogeneous morphism of quantum degree $2|\lambda|-n(m-n-1)$ and $\zed_2$-degree $0$. 

Also, define $g_{\lambda}: C(\Gamma) \rightarrow C(\Gamma_0)$ by $g_{\lambda} = g \circ \mathfrak{m}(S_{\lambda^c}(-\mathbb{Y}))$, where $S_{\lambda^c}(-\mathbb{Y})$ is the Schur polynomial in $-\mathbb{Y}$ associated to $\lambda^c$. Then $g_{\lambda}$ is a homogeneous morphism of quantum degree $n(m-n-1)-2|\lambda|$ and $\zed_2$-degree $0$. 
\end{definition}

\begin{lemma}\label{multiply-schur-elementary}
Let $\mathbb{A}$ be an alphabet with $n$ indeterminates. Denote by $A_k$ the $k$-th elementary symmetric polynomial in $\mathbb{A}$. For any $k=1,\dots,n$ and any partition $\lambda=(\lambda_1\geq\dots\geq\lambda_n)$, there is an expansion
\[
A_k \cdot S_{\lambda}(\mathbb{A}) = \sum_{l(\mu)\leq n} c_{\mu} \cdot S_{\mu}(\mathbb{A}),
\]
where $c_\mu \in \zed_{\geq0}$. If $c_\mu\neq0$, then $|\mu|-|\lambda|=k$ and $\lambda_j \leq \mu_j \leq \lambda_j +1$ $\forall j=1,\dots,n$. In particular,
\[
A_n \cdot S_{\lambda}(\mathbb{A}) = S_{(\lambda_1+1\geq\lambda_2+1\geq\cdots\geq\lambda_n+1)}(\mathbb{A}).
\]
\end{lemma}

\begin{proof}
Note that $A_k=S_{\lambda_{k,1}}(\mathbb{A})=S_{(\underbrace{1 \geq \cdots \geq 1}_{k \text{ parts}})}(\mathbb{A})$. This lemma is a special case of the Littlewood-Richardson rule (see for example \cite[Appendix A]{Fulton-Harris}.)
\end{proof}

\begin{lemma}\label{decomp-IV-f-g-lambda-compose}
For $\lambda,\mu\in\Lambda$, 
\[
g_{\mu} \circ f_\lambda \approx \left\{%
\begin{array}{ll}
    \id_{C(\Gamma_0)} & \text{if } \lambda=\mu, \\
    0 & \text{if } \lambda<\mu. \\
\end{array}%
\right.
\]
\end{lemma}

\begin{proof}
For $\lambda,\mu\in\Lambda$, by \eqref{Gamma-0-eq1}, we have
\begin{eqnarray*}
g_{\mu} \circ f_\lambda & = & g \circ \mathfrak{m}(S_{\mu^c}(-\mathbb{Y})) \circ \mathfrak{m}(S_{\lambda}(\mathbb{A})) \circ f \\
& = & \overline{\phi} \circ h_1 \circ \chi^1 \circ \mathfrak{m}(S_{\mu^c}(-\mathbb{Y}) \cdot S_{\lambda}(\mathbb{A})) \circ \chi^0 \circ h_0 \circ \phi \\
& = & \overline{\phi} \circ h_1 \circ \chi^1 \circ \chi^0 \circ h_0 \circ \mathfrak{m}(S_{\mu^c}(-\mathbb{Y})\cdot S_{\lambda}(\mathbb{A})) \circ \phi \\
& \simeq & \overline{\phi} \circ \mathfrak{m}((\sum_{k=0}^n (-r)^k A_{n-k}) \cdot S_{\lambda}(\mathbb{A})\cdot S_{\mu^c}(-\mathbb{Y})) \circ \phi.
\end{eqnarray*}
Write $\lambda= (\lambda_1\geq\cdots\geq\lambda_n)$ and $\widetilde{\lambda}=(\lambda_1+1\geq\cdots\geq\lambda_n+1)$. By Lemma \ref{multiply-schur-elementary}, we know that
\[
(\sum_{k=0}^n (-r)^k A_{n-k}) \cdot S_{\lambda}(\mathbb{A}) = S_{\widetilde{\lambda}}(\mathbb{A}) + \sum_{\lambda \leq \nu < \widetilde{\lambda}} c_{\nu}(r) \cdot S_{\nu}(\mathbb{A}),
\]
where $c_{\nu}(r) \in \zed[r]$. So 
\[
g_{\mu} \circ f_\lambda \simeq \overline{\phi} \circ \mathfrak{m}(S_{\widetilde{\lambda}}(\mathbb{A}) \cdot S_{\mu^c}(-\mathbb{Y})) \circ \phi + \sum_{\lambda \leq \nu < \widetilde{\lambda}} c_{\nu}(r) \cdot \overline{\phi} \circ \mathfrak{m}(S_{\nu}(\mathbb{A})\cdot S_{\mu^c}(-\mathbb{Y})) \circ \phi.
\]
Now the lemma follows from Lemma \ref{phibar-compose-phi}.
\end{proof}

\begin{lemma}\label{decomp-IV-F-G}
There exist homogeneous morphisms $F:C(\Gamma_0)\{\qb{m-1}{n}\}\rightarrow C(\Gamma)$ and $G:C(\Gamma)\rightarrow C(\Gamma_0)\{\qb{m-1}{n}\}$ preserving the $\zed_2\oplus\zed$-grading such that $G \circ F \simeq \id_{C(\Gamma_0)\{\qb{m-1}{n}\}}$.
\end{lemma}

\begin{proof}
Note that
\[
C(\Gamma_0)\{\qb{m-1}{n}\} = \bigoplus_{\lambda \in \Lambda} C(\Gamma_0)\{q^{2|\lambda|-n(m-n-1)}\}.
\]
We view $f_\lambda$ as a homogeneous morphism $f_\lambda:C(\Gamma_0)\{q^{2|\lambda|-n(m-n-1)}\} \rightarrow C(\Gamma)$ preserving the $\zed_2\oplus\zed$-grading and $g_{\lambda}$ as a homogeneous morphism $g_\lambda:C(\Gamma) \rightarrow C(\Gamma_0)\{q^{2|\lambda|-n(m-n-1)}\}$ preserving the $\zed_2\oplus\zed$-grading. Also, by choosing appropriate constants, we make
\begin{equation}\label{decomp-IV-f-g-lambda-compose-scaled}
g_{\mu} \circ f_\lambda \simeq \left\{%
\begin{array}{ll}
    \id_{C(\Gamma_0)} & \text{if } \lambda=\mu, \\
    0 & \text{if } \lambda<\mu. \\
\end{array}%
\right.
\end{equation}
Define $H_{\mu\lambda}: C(\Gamma_0)\{q^{2|\lambda|-n(m-n-1)}\} \rightarrow C(\Gamma_0)\{q^{2|\mu|-n(m-n-1)}\}$ by 
\[
H_{\mu\lambda} = \begin{cases}
    \id_{C(\Gamma_0)} & \text{if } \lambda=\mu, \\
    0 & \text{if } \lambda<\mu, \\
    \sum_{k\geq1} \sum_{\mu<\nu_1<\cdots<\nu_{k-1}<\lambda} (-1)^k (g_\mu \circ f_{\nu_1})\circ (g_{\nu_1} \circ f_{\nu_2}) \circ\cdots\circ (g_{\nu_{k-2}} \circ f_{\nu_{k-1}}) \circ (g_{\nu_{k-1}} \circ f_{\lambda}) & \text{if } \lambda>\mu.
\end{cases}
\]
Then define $\widetilde{g}_\mu: C(\Gamma) \rightarrow C(\Gamma_0)\{q^{2|\mu|-n(m-n-1)}\}$ by 
\[
\widetilde{g}_\mu = \sum_{\nu\geq\mu} H_{\mu\nu}\circ g_\nu.
\]
Note that $\widetilde{g}_\mu$ is a homogeneous morphism preserving the $\zed_2\oplus\zed$-grading. 

Next consider $\widetilde{g}_{\mu} \circ f_\lambda$.

(i) Suppose $\lambda < \mu$. Then, by \eqref{decomp-IV-f-g-lambda-compose-scaled},
\[
\widetilde{g}_{\mu} \circ f_\lambda = \sum_{\nu\geq\mu} H_{\mu\nu}\circ g_\nu \circ f_\lambda \simeq 0.
\]

(ii) Suppose $\lambda = \mu$. Then, by \eqref{decomp-IV-f-g-lambda-compose-scaled},
\[
\widetilde{g}_{\mu} \circ f_\lambda = \sum_{\nu\geq\mu} H_{\mu\nu}\circ g_\nu \circ f_\mu\simeq H_{\mu\mu} \circ g_\mu \circ f_\mu \simeq \id_{C(\Gamma_0)}.
\]

(iii) Suppose $\lambda > \mu$. Then
\begin{eqnarray*}
\widetilde{g}_{\mu} \circ f_\lambda & = & \sum_{\nu\geq\mu} H_{\mu\nu}\circ g_\nu \circ f_\lambda \\
& \simeq & H_{\mu\lambda}\circ g_\lambda \circ f_\lambda + H_{\mu\mu}\circ g_\mu \circ f_\lambda + \sum_{\mu<\nu<\lambda} H_{\mu\nu}\circ g_\nu \circ f_\lambda \\
& \simeq & H_{\mu\lambda} + g_\mu \circ f_\lambda \\
& & + \sum_{k\geq1} \sum_{\mu<\nu_1<\cdots<\nu_{k-1}<\nu<\lambda} (-1)^k (g_\mu \circ f_{\nu_1})\circ (g_{\nu_1} \circ f_{\nu_2}) \circ\cdots\circ (g_{\nu_{k-1}} \circ f_{\nu}) \circ (g_{\nu} \circ f_{\lambda}) \\
& = & H_{\mu\lambda} - H_{\mu\lambda} =0
\end{eqnarray*}

Now define 
\[
F:C(\Gamma_0)\{\qb{m-1}{n}\} (=\bigoplus_{\lambda \in \Lambda} C(\Gamma_0)\{q^{2|\lambda|-n(m-n-1)}\})\rightarrow C(\Gamma)
\]
by 
\[
F = \sum_{\lambda\in\Lambda} f_{\lambda},
\]
and
\[
G:C(\Gamma)\rightarrow C(\Gamma_0)\{\qb{m-1}{n}\} (=\bigoplus_{\lambda \in \Lambda} C(\Gamma_0)\{q^{2|\lambda|-n(m-n-1)}\})
\]
by
\[
G = \sum_{\lambda\in\Lambda} \widetilde{g}_{\lambda}.
\]
Then $F$ and $G$ are homogeneous morphisms preserving the $\zed_2\oplus\zed$-grading, and 
\[
G \circ F \simeq \id_{C(\Gamma_0)\{\qb{m-1}{n}\}}.
\]
\end{proof}

\subsection{Relating $\Gamma$ and $\Gamma_1$}\label{decomp-IV-subsection-Gamma-Gamma1}

\begin{figure}[ht]

\setlength{\unitlength}{1pt}

\begin{picture}(360,255)(-180,-15)


\put(-110,150){\vector(0,1){25}}

\put(-110,215){\vector(0,1){25}}

\put(-110,175){\vector(0,1){40}}

\put(-110,215){\vector(1,0){40}}

\put(-70,175){\vector(0,1){40}}

\put(-70,175){\vector(-1,0){40}}

\put(-70,157){\vector(0,1){25}}

\put(-70,215){\vector(0,1){25}}

\put(-90,214){\line(0,1){2}}

\put(-90,174){\line(0,1){2}}

\put(-111,190){\line(1,0){2}}

\put(-71,190){\line(1,0){2}}

\put(-117,163){\tiny{$1$}}

\put(-117,222){\tiny{$l$}}

\put(-127,198){\tiny{$l+n$}}

\put(-67,163){\tiny{$m+l-1$}}

\put(-67,222){\tiny{$m$}}

\put(-67,198){\tiny{$m-n$}}

\put(-103,178){\tiny{$l+n-1$}}

\put(-92,208){\tiny{$n$}}

\put(-67,232){\small{$\mathbb{X}$}}

\put(-67,187){\small{$\mathbb{Y}$}}

\put(-67,150){\small{$\mathbb{W}$}}

\put(-120,232){\small{$\mathbb{T}$}}

\put(-120,187){\small{$\mathbb{S}$}}

\put(-125,150){\small{$\{r\}$}}

\put(-93,217){\small{$\mathbb{A}$}}

\put(-93,165){\small{$\mathbb{B}$}}

\put(-150,190){$\Gamma$}

\put(-100,140){\vector(-1,-1){30}}

\put(-125,125){\small{$\chi^0$}}

\put(-120,110){\vector(1,1){30}}

\put(-100,125){\small{$\chi^1$}}

\put(-30,190){\vector(1,0){60}}

\put(30,200){\vector(-1,0){60}}

\put(-3,204){\small{$\alpha$}}

\put(-3,180){\small{$\beta$}}


\put(70,150){\vector(2,3){20}}

\put(110,150){\vector(-2,3){20}}

\put(90,210){\vector(-2,3){20}}

\put(90,210){\vector(2,3){20}}

\put(90,180){\vector(0,1){30}}

\put(73,170){\tiny{$1$}}

\put(103,170){\tiny{$m+l-1$}}

\put(73,215){\tiny{$l$}}

\put(103,215){\tiny{$m$}}

\put(95,188){\tiny{$m+l$}}

\put(113,232){\small{$\mathbb{X}$}}

\put(113,150){\small{$\mathbb{W}$}}

\put(60,232){\small{$\mathbb{T}$}}

\put(55,150){\small{$\{r\}$}}

\put(150,190){$\Gamma_1$}

\put(90,140){\vector(1,-1){30}}

\put(125,125){\small{$\overline{\phi}$}}

\put(130,110){\vector(-1,1){30}}

\put(90,125){\small{$\phi$}}


\put(-160,0){\vector(3,2){30}}

\put(-160,60){\vector(0,1){30}}

\put(-160,60){\vector(1,0){50}}

\put(-110,60){\line(1,0){10}}

\put(-130,20){\vector(0,1){20}}

\put(-130,40){\vector(3,2){30}}

\put(-130,40){\vector(-3,2){30}}

\put(-100,0){\vector(-3,2){30}}

\put(-100,60){\vector(0,1){30}}

\put(-130,59){\line(0,1){2}}

\put(-115,49.5){\line(0,1){2}}

\put(-147,13){\tiny{$1$}}

\put(-167,72){\tiny{$l$}}

\put(-107,13){\tiny{$m+l-1$}}

\put(-97,72){\tiny{$m$}}

\put(-107,48){\tiny{$m-n$}}

\put(-127,30){\tiny{$m+l$}}

\put(-131,53){\tiny{$n$}}

\put(-165,48){\tiny{$n+l$}}

\put(-142,-15){$\Gamma_{12}$}

\put(-97,82){\small{$\mathbb{X}$}}

\put(-115,40){\small{$\mathbb{Y}$}}

\put(-97,0){\small{$\mathbb{W}$}}

\put(-170,82){\small{$\mathbb{T}$}}

\put(-175,0){\small{$\{r\}$}}

\put(-132,63){\small{$\mathbb{A}$}}

\put(-40,40){\vector(1,0){80}}

\put(40,50){\vector(-1,0){80}}

\put(-3,54){\small{$h_1$}}

\put(-3,30){\small{$h_0$}}


\put(80,0){\vector(3,2){30}}

\put(80,55){\vector(0,1){35}}

\put(110,20){\vector(0,1){15}}

\put(110,35){\line(3,1){30}}

\put(140,45){\vector(0,1){10}}

\put(110,35){\line(-3,2){30}}

\put(140,75){\vector(0,1){15}}

\put(130,60){\vector(0,1){10}}

\put(150,60){\vector(0,1){10}}

\put(140,0){\vector(-3,2){30}}

\qbezier(140,75)(130,75)(130,70)

\qbezier(140,75)(150,75)(150,70)

\qbezier(140,55)(130,55)(130,60)

\qbezier(140,55)(150,55)(150,60)

\put(129,63){\line(1,0){2}}

\put(149,63){\line(1,0){2}}

\put(93,12){\tiny{$1$}}

\put(73,72){\tiny{$l$}}

\put(153,67){\tiny{$m-n$}}

\put(120,67){\tiny{$n$}}

\put(124,12){\tiny{$m+l-1$}}

\put(130,80){\tiny{$m$}}

\put(140,40){\tiny{$m$}}

\put(113,23){\tiny{$m+l$}}

\put(108,-15){$\Gamma_{13}$}

\put(143,82){\small{$\mathbb{X}$}}

\put(143,0){\small{$\mathbb{W}$}}

\put(153,57){\small{$\mathbb{Y}$}}

\put(70,82){\small{$\mathbb{T}$}}

\put(65,0){\small{$\{r\}$}}

\put(120,57){\small{$\mathbb{A}$}}

\end{picture}

\caption{}\label{decomposition-IV-alpha-beta-def}

\end{figure}

Consider the diagram in Figure \ref{decomposition-IV-alpha-beta-def}, in which
\begin{itemize}
	\item $\phi$ and $\overline{\phi}$ are the morphisms associated to the apparent edge splitting and merging,
	\item $h_0$ and $h_1$ are the homotopy equivalences induced by the apparent bouquet moves and are inverses of each other,
	\item $\chi^0$ and $\chi^1$ are the morphisms from applying Corollary \ref{general-chi-maps-def} to the lower half of $\Gamma$.
\end{itemize}
All these morphisms are $\Sym(\mathbb{X}|\mathbb{W}|\mathbb{T}|\{r\})$-linear. Moreover, $h_0$, $h_1$, $\chi^0$ and $\chi^1$ are also $\Sym(\mathbb{A}|\mathbb{Y})$-linear. By Corollary \ref{general-chi-maps-def}, we know
\begin{equation}\label{Gamma-0-eq2}
\chi^0 \circ \chi^1 = (\sum_{k=0}^{m-n} (-r)^k Y_{m-n-k}) \cdot \id_{C(\Gamma_{12})}.
\end{equation}

\begin{definition}\label{decomp-IV-alpha-beta-def}
Define $\alpha: C(\Gamma_1) \rightarrow C(\Gamma)$ by $\alpha=\chi^1 \circ h_1 \circ \phi$ and $\beta: C(\Gamma) \rightarrow C(\Gamma_1)$ by $\beta=\overline{\phi} \circ h_0 \circ \chi^0$.

Note that $\alpha$ and $\beta$ are both homogeneous morphisms with quantum degree  \linebreak $-(n-1)(m-n)$ and $\zed_2$-degree $0$.
\end{definition}

\begin{definition}\label{decomp-IV-alpha-beta-lambda-def}
Let $\Lambda' = \Lambda_{m-n,n-1} =\{\lambda~|~l(\lambda)\leq m-n,~\lambda_1\leq n-1\}$. For $\lambda=(\lambda_1\geq\cdots\geq\lambda_{m-n}) \in \Lambda'$, define $\lambda^\ast=(\lambda^\ast_1\geq\cdots\geq\lambda^\ast_{m-n}) \in \Lambda'$ by $\lambda^\ast_j =n-1-\lambda_{m-n+1-j}$ for $j=1,\dots,m-n$.

For $\lambda \in \Lambda'$, define $\alpha_{\lambda}: C(\Gamma_1) \rightarrow C(\Gamma)$ by $\alpha_{\lambda} = \mathfrak{m}(S_{\lambda}(\mathbb{Y})) \circ \alpha$, where $S_{\lambda}(\mathbb{Y})$ is the Schur polynomial in $\mathbb{Y}$ associated to $\lambda$. $\alpha_{\lambda}$ is a homogeneous morphism with quantum degree $2|\lambda|-(n-1)(m-n)$ and $\zed_2$-degree $0$. 

Also, define $\beta_{\lambda}: C(\Gamma) \rightarrow C(\Gamma_1)$ by $\beta_{\lambda} = \beta \circ \mathfrak{m}(S_{\lambda^\ast}(-\mathbb{A}))$, where $S_{\lambda^\ast}(-\mathbb{A})$ is the Schur polynomial in $-\mathbb{A}$ associated to $\lambda^\ast$. $\beta_{\lambda}$ is a homogeneous morphism with quantum degree $(n-1)(m-n)-2|\lambda|$ and $\zed_2$-degree $0$. 
\end{definition}

\begin{lemma}\label{decomp-IV-alpha-beta-lambda-compose}
For $\lambda,\mu\in\Lambda'$, 
\[
\beta_{\mu} \circ \alpha_\lambda \approx \left\{%
\begin{array}{ll}
    \id_{C(\Gamma_1)} & \text{if } \lambda=\mu, \\
    0 & \text{if } \lambda<\mu. \\
\end{array}%
\right.
\]
\end{lemma}

\begin{proof}
For $\lambda,\mu\in\Lambda'$, by \eqref{Gamma-0-eq2}, we have
\begin{eqnarray*}
\beta_{\mu} \circ \alpha_\lambda & = & \beta \circ \mathfrak{m}(S_{\mu^\ast}(-\mathbb{A})) \circ \mathfrak{m}(S_{\lambda}(\mathbb{Y})) \circ \alpha \\
& = & \overline{\phi} \circ h_0 \circ \chi^0 \circ \mathfrak{m}(S_{\mu^\ast}(-\mathbb{A}) \cdot S_{\lambda}(\mathbb{Y})) \circ \chi^1 \circ h_1 \circ \phi \\
& = & \overline{\phi} \circ h_0 \circ \chi^0 \circ \chi^1 \circ h_1 \circ \mathfrak{m}(S_{\mu^\ast}(-\mathbb{A}) \cdot S_{\lambda}(\mathbb{Y})) \circ \phi \\
& \simeq & \overline{\phi} \circ \mathfrak{m}((\sum_{k=0}^{m-n} (-r)^k Y_{m-n-k}) \cdot S_{\lambda}(\mathbb{Y}) \cdot S_{\mu^\ast}(-\mathbb{A})) \circ \phi.
\end{eqnarray*}
Write $\lambda= (\lambda_1\geq\cdots\geq\lambda_{m-n})$ and $\widetilde{\lambda}=(\lambda_1+1\geq\cdots\geq\lambda_{m-n}+1)$. By Lemma \ref{multiply-schur-elementary}, we know that
\[
(\sum_{k=0}^{m-n} (-r)^k Y_{m-n-k}) \cdot S_{\lambda}(\mathbb{Y}) = S_{\widetilde{\lambda}}(\mathbb{Y}) + \sum_{\lambda \leq \nu < \widetilde{\lambda}} c_{\nu}(r) \cdot S_{\nu}(\mathbb{Y}),
\]
where $c_{\nu}(r) \in \zed[r]$. So 
\[
\beta_{\mu} \circ \alpha_\lambda \simeq \overline{\phi} \circ \mathfrak{m}(S_{\widetilde{\lambda}}(\mathbb{Y}) \cdot S_{\mu^\ast}(-\mathbb{A})) \circ \phi + \sum_{\lambda \leq \nu < \widetilde{\lambda}} c_{\nu}(r) \cdot \overline{\phi} \circ \mathfrak{m}(S_{\nu}(\mathbb{Y})\cdot S_{\mu^\ast}(-\mathbb{A})) \circ \phi.
\]
Now the lemma follows from Lemma \ref{phibar-compose-phi}.
\end{proof}

\begin{lemma}\label{decomp-IV-Alpha-Beta}
There exist homogeneous morphisms $\vec{\alpha}:C(\Gamma_1)\{\qb{m-1}{n-1}\}\rightarrow C(\Gamma)$ and $\vec{\beta}:C(\Gamma)\rightarrow C(\Gamma_1)\{\qb{m-1}{n-1}\}$ preserving the $\zed_2\oplus\zed$-grading such that $\vec{\beta} \circ \vec{\alpha} \simeq \id_{C(\Gamma_1)\{\qb{m-1}{n-1}\}}$.
\end{lemma}

\begin{proof}
Note that
\[
C(\Gamma_1)\{\qb{m-1}{n-1}\} = \bigoplus_{\lambda \in \Lambda'} C(\Gamma_1)\{q^{2|\lambda|-(n-1)(m-n)}\}.
\]
We view $\alpha_\lambda$ as a homogeneous morphism $\alpha_\lambda:C(\Gamma_1)\{q^{2|\lambda|-(n-1)(m-n)}\} \rightarrow C(\Gamma)$ preserving the $\zed_2\oplus\zed$-grading and $\beta_{\lambda}$ as a homogeneous morphism $\beta_\lambda:C(\Gamma) \rightarrow C(\Gamma_1)\{q^{2|\lambda|-(n-1)(m-n)}\}$ preserving the $\zed_2\oplus\zed$-grading. Also, by choosing appropriate constants, we make
\begin{equation}\label{decomp-IV-alpha-beta-lambda-compose-scaled}
\beta_{\mu} \circ \alpha_\lambda \simeq \left\{%
\begin{array}{ll}
    \id_{C(\Gamma_1)} & \text{if } \lambda=\mu, \\
    0 & \text{if } \lambda<\mu. \\
\end{array}%
\right.
\end{equation}
Define $\tau_{\mu\lambda}: C(\Gamma_1)\{q^{2|\lambda|-(n-1)(m-n)}\} \rightarrow C(\Gamma_1)\{q^{2|\mu|-(n-1)(m-n)}\}$ by 
\[
\tau_{\mu\lambda} = \begin{cases}
    \id_{C(\Gamma_1)} & \text{if } \lambda=\mu, \\
    0 & \text{if } \lambda<\mu, \\
    \sum_{k\geq1} \sum_{\mu<\nu_1<\cdots<\nu_{k-1}<\lambda} (-1)^k (\beta_\mu \circ \alpha_{\nu_1})\circ (\beta_{\nu_1} \circ \alpha_{\nu_2}) \circ\cdots\circ (\beta_{\nu_{k-2}} \circ \alpha_{\nu_{k-1}}) \circ (\beta_{\nu_{k-1}} \circ \alpha_{\lambda}) & \text{if } \lambda>\mu.
\end{cases}
\]
Then define $\widetilde{\beta}_\mu: C(\Gamma) \rightarrow C(\Gamma_0)\{q^{2|\mu|-(n-1)(m-n)}\}$ by 
\[
\widetilde{\beta}_\mu = \sum_{\nu\geq\mu} \tau_{\mu\nu}\circ \beta_\nu.
\]
Note that $\widetilde{\beta}_\mu$ is a homogeneous morphism preserving the $\zed_2\oplus\zed$-grading. 

Next consider $\widetilde{\beta}_{\mu} \circ \alpha_\lambda$.

(i) Suppose $\lambda < \mu$. Then, by \eqref{decomp-IV-alpha-beta-lambda-compose-scaled},
\[
\widetilde{\beta}_{\mu} \circ \alpha_\lambda = \sum_{\nu\geq\mu} \tau_{\mu\nu}\circ \beta_\nu \circ \alpha_\lambda \simeq 0.
\]

(ii) Suppose $\lambda = \mu$. Then, by \eqref{decomp-IV-alpha-beta-lambda-compose-scaled},
\[
\widetilde{\beta}_{\mu} \circ \alpha_\lambda = \sum_{\nu\geq\mu} \tau_{\mu\nu}\circ \beta_\nu \circ \alpha_\mu \simeq \tau_{\mu\mu} \circ \beta_\mu \circ \alpha_\mu \simeq \id_{C(\Gamma_1)}.
\]

(iii) Suppose $\lambda > \mu$. Then
\begin{eqnarray*}
\widetilde{\beta}_{\mu} \circ \alpha_\lambda & = & \sum_{\nu\geq\mu} \tau_{\mu\nu}\circ \beta_\nu \circ \alpha_\lambda \\
& \simeq & \tau_{\mu\lambda}\circ \beta_\lambda \circ \alpha_\lambda + \tau_{\mu\mu}\circ \beta_\mu \circ \alpha_\lambda + \sum_{\mu<\nu<\lambda} \tau_{\mu\nu}\circ \beta_\nu \circ \alpha_\lambda \\
& \simeq & \tau_{\mu\lambda} + \beta_\mu \circ \alpha_\lambda \\
& & + \sum_{k\geq1} \sum_{\mu<\nu_1<\cdots<\nu_{k-1}<\nu<\lambda} (-1)^k (\beta_\mu \circ \alpha_{\nu_1})\circ (\beta_{\nu_1} \circ \alpha_{\nu_2}) \circ\cdots\circ (\beta_{\nu_{k-1}} \circ \alpha_{\nu}) \circ (\beta_{\nu} \circ \alpha_{\lambda}) \\
& = & \tau_{\mu\lambda} - \tau_{\mu\lambda} =0
\end{eqnarray*}

Now define 
\[
\vec{\alpha}:C(\Gamma_1)\{\qb{m-1}{n-1}\} (=\bigoplus_{\lambda \in \Lambda'} C(\Gamma_1)\{q^{2|\lambda|-(n-1)(m-n)}\})\rightarrow C(\Gamma)
\]
by 
\[
\vec{\alpha} = \sum_{\lambda\in\Lambda'} \alpha_{\lambda},
\]
and
\[
\vec{\beta}:C(\Gamma)\rightarrow C(\Gamma_1)\{\qb{m-1}{n-1}\} (=\bigoplus_{\lambda \in \Lambda'} C(\Gamma_1)\{q^{2|\lambda|-(n-1)(m-n)}\})
\]
by
\[
\vec{\beta} = \sum_{\lambda\in\Lambda'} \widetilde{\beta}_{\lambda}.
\]
Then $\vec{\alpha}$ and $\vec{\beta}$ are homogeneous morphisms preserving the $\zed_2\oplus\zed$-grading, and 
\[
\vec{\beta} \circ \vec{\alpha} \simeq \id_{C(\Gamma_1)\{\qb{m-1}{n-1}\}}.
\]
\end{proof}

\subsection{Homotopic nilpotency of $\vec{\beta} \circ F \circ G \circ \vec{\alpha}$ and $G \circ \vec{\alpha} \circ \vec{\beta} \circ F$}

\begin{lemma}\label{decomp-IV-nilopotency-hmf}
Let $\Gamma_0$ and $\Gamma_1$ be as in Figure \ref{decomposition-IV-figure}. Then
\begin{eqnarray*}
& & \Hom_{HMF}(C(\Gamma_0),C(\Gamma_1)) \cong \Hom_{HMF}(C(\Gamma_1),C(\Gamma_0)) \\
& \cong & C(\emptyset) \{ \qb{l+m-1}{m} \cdot \qb{l+m}{1} \cdot \qb{N}{l+m} \cdot q^{(l+m)(N+1-l-m)+ml-1}\},
\end{eqnarray*}
where $C(\emptyset)$ is $\C\rightarrow0\rightarrow\C$. In particular, the lowest non-vanishing quantum grading of these spaces is $m$.
\end{lemma}

\begin{figure}[ht]

\setlength{\unitlength}{1pt}

\begin{picture}(360,75)(-180,-15)


\put(-180,60){\vector(0,-1){40}}

\qbezier(-180,20)(-180,0)(-160,0)

\put(-160,0){\line(1,0){10}}

\put(-150,0){\vector(0,1){20}}

\put(-150,20){\vector(-3,4){30}}

\put(-150,20){\vector(3,4){30}}

\put(-180,60){\vector(1,0){60}}

\put(-120,60){\vector(0,-1){40}}

\qbezier(-120,20)(-120,0)(-140,0)

\put(-140,0){\line(-1,0){10}}

\put(-117,28){\tiny{$m+l-1$}}

\put(-177,28){\tiny{$1$}}

\put(-163,40){\tiny{$l$}}

\put(-143,40){\tiny{$m$}}

\put(-148,8){\tiny{$m+l$}}

\put(-155,53){\tiny{$l-1$}}

\put(-152,-15){$\Gamma_{14}$}


\qbezier(-30,60)(-60,60)(-60,40)

\put(-60,40){\vector(0,-1){20}}

\qbezier(-60,20)(-60,0)(-30,0)

\put(-30,0){\vector(0,1){60}}

\put(-30,60){\vector(1,0){20}}

\qbezier(-10,60)(0,60)(0,40)

\qbezier(0,40)(-10,40)(-10,35)

\qbezier(0,40)(10,40)(10,35)

\qbezier(0,20)(-10,20)(-10,25)

\qbezier(0,20)(10,20)(10,25)

\put(-10,35){\vector(0,-1){10}}

\put(10,35){\vector(0,-1){10}}

\put(0,20){\vector(0,-1){10}}

\qbezier(0,10)(0,0)(-30,0)

\put(0,0){\tiny{$m+l-1$}}

\put(0,55){\tiny{$m+l-1$}}

\put(-57,28){\tiny{$1$}}

\put(-19,28){\tiny{$m$}}

\put(-28,8){\tiny{$m+l$}}

\put(13,28){\tiny{$l-1$}}

\put(-32,-15){$\Gamma_{15}$}


\put(80,30){\oval(50,60)}

\put(80,0){\vector(0,1){60}}

\put(55,25){\vector(0,-1){0}}

\put(105,25){\vector(0,-1){0}}

\put(100,2){\tiny{$m+l-1$}}

\put(57,30){\tiny{$1$}}

\put(82,30){\tiny{$m+l$}}

\put(78,-15){$\Gamma_{16}$}


\put(160,30){\oval(40,60)}

\put(180,35){\vector(0,1){0}}

\put(142,30){\tiny{$m+l$}}

\put(158,-15){$\Gamma_{17}$}

\end{picture}

\caption{}\label{decomposition-IV-hmf-gamma-0-1-figure}

\end{figure}

\begin{proof}
Mark $\Gamma_0$ and $\Gamma_1$ as in Figure \ref{decomposition-IV-figure}. Then
\[
C(\Gamma_0) = \left(%
\begin{array}{ll}
  \ast & T_1-r-D_1 \\
  \dots & \dots \\
  \ast & T_k-rD_{k-1}-D_k\\
  \dots & \dots \\
  \ast & T_l-rD_{l-1}\\
  \ast & D_1+X_1-W_1 \\
  \dots & \dots \\
  \ast & \sum_{j=0}^kD_jX_{k-j}-W_k \\
  \dots & \dots \\
  \ast & D_{l-1}X_m-W_{m+l-1}
\end{array}%
\right)_{\Sym (\mathbb{X}|\mathbb{W}|\mathbb{D}|\mathbb{T}|\{r\})}\{q^{-(l-1)m}\},
\]
where $X_k$ is the $k$-th elementary symmetric polynomial in $\mathbb{X}$ and so on. By Proposition \ref{b-contraction}, we exclude $D_1,\dots,D_{l-1}$ from this matrix factorization using the right entries of the first $l-1$ rows. We get the relations 
\[
D_k = \begin{cases}
\sum_{j=0}^k (-r)^j T_{k-j} & \text{if } 0\leq k \leq l-1, \\
0 & \text{if } k<0 \text{ or } k>l-1,
\end{cases}
\]
and
\[
C(\Gamma_0) \simeq \left(%
\begin{array}{ll}
  \ast & \sum_{j=0}^l (-r)^j T_{l-j}\\
  \ast & T_1-r+X_1-W_1 \\
  \dots & \dots \\
  \ast & \sum_{j=0}^{l-1} \sum_{i=0}^j (-r)^i T_{j-i}X_{k-j}-W_k \\
  \dots & \dots \\
  \ast & \sum_{i=0}^{l-1}(-r)^i T_{l-1-i} X_m-W_{m+l-1}
\end{array}%
\right)_{\Sym (\mathbb{X}|\mathbb{W}|\mathbb{T}|\{r\})}\{q^{-(l-1)m}\}.
\]
So
\[
C(\Gamma_0)_\bullet  \simeq  \left(%
\begin{array}{ll}
  \ast & -\sum_{j=0}^l (-r)^j T_{l-j}\\
  \ast & -(T_1-r+X_1-W_1) \\
  \dots & \dots \\
  \ast & -(\sum_{j=0}^{l-1} \sum_{i=0}^j (-r)^i T_{j-i}X_{k-j}-W_k) \\
  \dots & \dots \\
  \ast & -(\sum_{i=0}^{l-1}(-r)^i T_{l-1-i} X_m-W_{m+l-1})
\end{array}%
\right)_{\Sym (\mathbb{X}|\mathbb{W}|\mathbb{T}|\{r\})}\{q^{(l+m)(N+1-l-m)+(l-1)m}\} \left\langle l+m \right\rangle.
\]
Let $\overline{\Gamma}_0$ be $\Gamma_0$ with the orientation reversed. Similar to above, we have
\[
C(\overline{\Gamma}_0) \simeq \left(%
\begin{array}{ll}
  \ast & -\sum_{j=0}^l (-r)^j T_{l-j}\\
  \ast & -(T_1-r+X_1-W_1) \\
  \dots & \dots \\
  \ast & -(\sum_{j=0}^{l-1} \sum_{i=0}^j (-r)^i T_{j-i}X_{k-j}-W_k) \\
  \dots & \dots \\
  \ast & -(\sum_{i=0}^{l-1}(-r)^i T_{l-1-i} X_m-W_{m+l-1})
\end{array}%
\right)_{\Sym (\mathbb{X}|\mathbb{W}|\mathbb{T}|\{r\})}\{q^{-l+1}\}.
\]
Thus, $C(\Gamma_0)_\bullet \simeq C(\overline{\Gamma}_0)\{q^{(l+m)(N+1-l-m)+lm-1}\} \left\langle l+m \right\rangle$ and, therefore,
\[
\Hom(C(\Gamma_0),C(\Gamma_1)) \cong C(\Gamma_1) \otimes C(\Gamma_0)_\bullet \simeq C(\Gamma_1) \otimes C(\overline{\Gamma}_0) \{q^{(l+m)(N+1-l-m)+lm-1}\} \left\langle l+m \right\rangle.
\]
Let $\Gamma_{14},\dots,\Gamma_{17}$ be the MOY graphs in Figure \ref{decomposition-IV-hmf-gamma-0-1-figure}. Then
\begin{eqnarray*}
C(\overline{\Gamma}_0) \otimes C(\Gamma_1) & \simeq & C(\Gamma_{14}) \simeq C(\Gamma_{15}) \simeq C(\Gamma_{16})\{\qb{m+l-1}{m}\} \\
& \simeq & C(\Gamma_{17})\{\qb{m+l}{1}\cdot\qb{m+l-1}{m}\} \\
& \simeq & C(\emptyset)\{\qb{N}{m+l}\cdot\qb{m+l}{1}\cdot\qb{m+l-1}{m}\}.
\end{eqnarray*}
So
\[
\Hom_{HMF}(C(\Gamma_0),C(\Gamma_1)) \cong C(\emptyset) \{ \qb{l+m-1}{m} \cdot \qb{l+m}{1} \cdot \qb{N}{l+m} \cdot q^{(l+m)(N+1-l-m)+ml-1}\}.
\]

The computation of $\Hom_{HMF}(C(\Gamma_1),C(\Gamma_0))$ is very similar. Using the fact that
\[
C(\Gamma_1) \simeq \left(%
\begin{array}{ll}
  \ast & T_1+X_1-r-W_1 \\
  \dots & \dots \\
  \ast & \sum_{j=0}^k T_{j}X_{k-j}- rW_{k-1}-W_k \\
  \dots & \dots \\
  \ast & T_lX_m-rW_{m+l-1}
\end{array}%
\right)_{\Sym (\mathbb{X}|\mathbb{W}|\mathbb{T}|\{r\})}\{q^{-lm}\},
\]
one gets $C(\Gamma_1)_\bullet \simeq C(\overline{\Gamma}_1)\{q^{(l+m)(N+1-l-m)+lm-1}\} \left\langle l+m \right\rangle$, where $\overline{\Gamma}_1$ is $\Gamma_1$ with the orientation reversed. So 
\begin{eqnarray*}
\Hom(C(\Gamma_1),C(\Gamma_0)) & \cong & C(\Gamma_0) \otimes C(\Gamma_1)_\bullet \\
& \simeq & C(\Gamma_0) \otimes C(\overline{\Gamma}_1) \{q^{(l+m)(N+1-l-m)+lm-1}\} \left\langle l+m \right\rangle \\
& \simeq & C(\overline{\Gamma}_{14})\{q^{(l+m)(N+1-l-m)+lm-1}\} \left\langle l+m \right\rangle \\
& \simeq & \cdots \simeq \\
& \simeq & C(\emptyset) \{ \qb{l+m-1}{m} \cdot \qb{l+m}{1} \cdot \qb{N}{l+m} \cdot q^{(l+m)(N+1-l-m)+ml-1}\},
\end{eqnarray*}
where $\overline{\Gamma}_{14}$ is $\Gamma_{14}$ with the orientation reversed.
\end{proof}

\begin{lemma}\label{decomp-IV-nilopotency-sum-difference}
For $\mu\in\Lambda$ and $\lambda\in\Lambda'$, let $\alpha_{\lambda}$, $\widetilde{\beta}_\lambda$, $f_\mu$ and $\widetilde{g}_\mu$ be the morphisms defined in the two preceding subsections. We have
\begin{itemize}
	\item If $|\lambda|-|\mu|<n$, then $\widetilde{g}_\mu \circ \alpha_{\lambda} \simeq 0$.
	\item If $|\mu|-|\lambda|<m-n$, then $\widetilde{\beta}_\lambda \circ f_\mu \simeq 0$.
\end{itemize}
\end{lemma}

\begin{proof}
Note that $\widetilde{g}_\mu \circ \alpha_{\lambda}: C(\Gamma_1) \rightarrow C(\Gamma_0)$ is a homogeneous morphism of quantum degree 
\[
2|\lambda| - (n-1)(m-n) - 2 |\mu| + n(m-n-1) = 2(|\lambda|-|\mu|-n)+m,
\]
and $\widetilde{\beta}_\lambda \circ f_\mu : C(\Gamma_0) \rightarrow C(\Gamma_1)$ is a homogeneous morphism of quantum degree 
\[
-2|\lambda| + (n-1)(m-n) + 2 |\mu| - n(m-n-1) = 2(|\mu|-|\lambda|-(m-n))+m.
\]
Then the lemma follows from Lemma \ref{decomp-IV-nilopotency-hmf}.
\end{proof}

\begin{lemma}\label{decomp-IV-nilopotency}
Let $\vec{\alpha}$, $\vec{\beta}$, $F$ and $G$ be the morphisms defined in the two preceding subsections. Then $\vec{\beta} \circ F \circ G \circ \vec{\alpha}$ and $G \circ \vec{\alpha} \circ \vec{\beta} \circ F$ are both homotopically nilpotent.
\end{lemma}

\begin{proof}
For $\lambda,\mu \in \Lambda'$, the $(\mu,\lambda)$-component of $(\vec{\beta} \circ F \circ G \circ \vec{\alpha})^k$ is
\[
\sum_{\lambda_1,\dots,\lambda_{k-1} \in \Lambda', \nu_1,\dots,\nu_k\in\Lambda} (\widetilde{\beta}_\mu \circ f_{\nu_1} \circ \widetilde{g}_{\nu_1} \circ \alpha_{\lambda_1}) \circ (\widetilde{\beta}_{\lambda_1} \circ f_{\nu_2} \circ \widetilde{g}_{\nu_2} \circ \alpha_{\lambda_2}) \circ\cdots\circ (\widetilde{\beta}_{\lambda_{k-1}} \circ f_{\nu_k} \circ \widetilde{g}_{\nu_k} \circ \alpha_{\lambda}).
\]
By Lemma \ref{decomp-IV-nilopotency-sum-difference}, for the term corresponding to $\lambda_1,\dots,\lambda_{k-1} \in \Lambda', \nu_1,\dots,\nu_k\in\Lambda$ to be homotopically non-trivial, we must have
\begin{eqnarray*}
|\lambda|-|\nu_k| & \geq & n, \\
|\nu_1|-|\mu| & \geq & m-n, \\
|\lambda_j| - |\nu_j| & \geq & n, \text{ for } j=1,\dots,k-1,\\
|\nu_{j+1}| - |\lambda_j| & \geq & m-n, \text{ for } j=1,\dots,k-1.
\end{eqnarray*}
Adding all these inequalities together, we get $|\lambda|-|\mu| \geq km$. Note that $|\lambda|-|\mu| \leq (n-1)(m-n)$. This implies that $(\vec{\beta} \circ F \circ G \circ \vec{\alpha})^k \simeq 0$ if $km>(n-1)(m-n)$. Thus, $\vec{\beta} \circ F \circ G \circ \vec{\alpha}$ is homotopically nilpotent. Since 
\[
(G \circ \vec{\alpha} \circ \vec{\beta} \circ F)^{k+1} = G \circ \vec{\alpha} \circ (\vec{\beta} \circ F \circ G \circ \vec{\alpha})^k \circ \vec{\beta} \circ F,
\]
$G \circ \vec{\alpha} \circ \vec{\beta} \circ F$ is also homotopically nilpotent.
\end{proof}

\subsection{Graded dimensions of $C(\Gamma)$, $C(\Gamma_0)$ and $C(\Gamma_1)$}

\begin{lemma}\label{decomp-IV-dimensions}
Let $\Gamma$, $\Gamma_0$ and $\Gamma_1$ be the MOY graphs in Figure \ref{decomposition-IV-figure}, where $l,m,n$ are integers satisfying $0\leq n \leq m \leq N$ and $0\leq l, m+l-1 \leq N$. Then
\begin{eqnarray*}
\gdim C(\Gamma_0) & = & q^{-lm+m}(1+\tau q^{2l-N-1}) \prod_{j=1}^{m+l-1}(1+\tau q^{2j-N-1}), \\
\gdim C(\Gamma_1) & = & \begin{cases}
    q^{-lm}\prod_{j=1}^{m+l}(1+\tau q^{2j-N-1}) & \text{if } ~l+m\leq N, \\
    0 & \text{if } ~l+m = N+1, \\
\end{cases} \\
\gdim C(\Gamma) & = & \begin{cases}
    q^{-lm+m-n}\qb{m}{n}(1+\tau q^{2n+2l-N-1})\prod_{j=1}^{m+l-1}(1+\tau q^{2j-N-1}) & \text{if } ~l+m\leq N, \\
    q^{-lm+m}\qb{m-1}{n}(1+\tau q^{N+1-2m})\prod_{j=1}^{m+l-1}(1+\tau q^{2j-N-1}) & \text{if } ~l+m = N+1. \\
\end{cases}
\end{eqnarray*}
In particular, 
\begin{equation}\label{decomp-IV-dimensions-sum}
\gdim C(\Gamma) = \qb{m-1}{n} \cdot \gdim C(\Gamma_0) + \qb{m-1}{n-1} \cdot \gdim C(\Gamma_1).
\end{equation}
\end{lemma}

\begin{proof}
We mark $\Gamma$, $\Gamma_0$ and $\Gamma_1$ as in Figure \ref{decomposition-IV-figure}. Then $C(\Gamma)$, $C(\Gamma_0)$ and $C(\Gamma_1)$ are matrix factorizations over $\Sym(\mathbb{X}|\mathbb{W}|\mathbb{T}|\{r\})$. The corresponding maximal ideal is
\[
\mathfrak{I}= (X_1,\dots,X_m,W_1,\dots,W_{l+m-1},T_1,\dots,T_l,r),
\]
where $X_j$ is the $j$-th elementary symmetric polynomial in $\mathbb{X}$ and so on.

We compute $\gdim C(\Gamma_0)$ first. 

From the proof of Lemma \ref{decomp-IV-nilopotency-hmf}, we know that
\[
C(\Gamma_0) \simeq \left(%
\begin{array}{ll}
  \ast & \sum_{j=0}^l (-r)^j T_{l-j}\\
  \ast & T_1-r+X_1-W_1 \\
  \dots & \dots \\
  \ast & \sum_{j=0}^k \sum_{i=0}^j (-r)^i T_{j-i}X_{k-j}-W_k \\
  \dots & \dots \\
  \ast & \sum_{i=0}^{l-1}(-r)^i T_{l-1-i} X_m-W_{m+l-1}
\end{array}%
\right)_{\Sym (\mathbb{X}|\mathbb{W}|\mathbb{T}|\{r\})}\{q^{-(l-1)m}\}.
\]
So
\[
C(\Gamma_0) / \mathfrak{I} \cdot C(\Gamma_0) \simeq \left(%
\begin{array}{ll}
  0 & 0_l \\
  0 & 0_1 \\
  \dots & \dots \\
  0 & 0_{m+l-1}
\end{array}%
\right)_{\C}\{q^{-(l-1)m}\},
\]
where $0_j$ means ``a $0$ of degree $2j$". then it follows easily that 
\[
\gdim C(\Gamma_0) = q^{-lm+m}(1+\tau q^{2l-N-1}) \prod_{j=1}^{m+l-1}(1+\tau q^{2j-N-1}).
\]

Next we compute $\gdim C(\Gamma_1)$. 

If $l+m = N+1$, then $C(\Gamma_1)\simeq 0$. So $\gdim C(\Gamma_1)=0$.

If $l+m \leq N$, then 
\[
C(\Gamma_1) \simeq \left(%
\begin{array}{ll}
  \ast & T_1+X_1-r-W_1 \\
  \dots & \dots \\
  \ast & \sum_{j=0}^k T_{j}X_{k-j}- rW_{k-1}-W_k \\
  \dots & \dots \\
  \ast & T_lX_m-rW_{m+l-1}
\end{array}%
\right)_{\Sym (\mathbb{X}|\mathbb{W}|\mathbb{T}|\{r\})}\{q^{-lm}\}
\]
and, therefore,
\[
C(\Gamma_1) / \mathfrak{I} \cdot C(\Gamma_1) \simeq \left(%
\begin{array}{ll}
  0 & 0_1 \\
  \dots & \dots \\
  0 & 0_{m+l}
\end{array}%
\right)_{\C}\{q^{-lm}\}.
\]
So
\[
\gdim C(\Gamma_1) = q^{-lm}\prod_{j=1}^{m+l}(1+\tau q^{2j-N-1}).
\]

Now we compute $C(\Gamma)$.

Let $\mathbb{D} = \mathbb{A} \cup \mathbb{T}$ and $\mathbb{E} = \{r\} \cup \mathbb{B}$. Denote by $D_j$ and $E_j$ the $j$-th elementary symmetric polynomials in $\mathbb{D}$ and $\mathbb{E}$. Define
\begin{eqnarray*}
U_j & = & \frac{p_{l+n,N+1}(E_1,\dots,E_{j-1},D_j,\dots,D_{l+n}) - p_{l+n,N+1}(E_1,\dots,E_{j},D_{j+1},\dots,D_{l+n})}{D_j-E_j}, \\
V_j & = & (-1)^{j-1} p_{N+1-j}(\mathbb{A}\cup\mathbb{Y}) + \sum_{k=1}^m (-1)^{k+j} j X_k h_{N+1-j-k}(\mathbb{A}\cup\mathbb{Y}) \\
    &   & + \sum_{k=1}^m \sum_{i=1}^m (-1)^{k+i} i X_k X_i \xi_{N+1-k-i,j}(\mathbb{X},\mathbb{A}\cup\mathbb{Y}), \\
\hat{V}_j & = & (-1)^{j-1} p_{N+1-j}(\mathbb{B}\cup\mathbb{Y}) + \sum_{k=1}^{m+l-1} (-1)^{k+j} j W_k h_{N+1-j-k}(\mathbb{B}\cup\mathbb{Y}) \\
    &   & + \sum_{k=1}^{m+l-1} \sum_{i=1}^{m+l-1} (-1)^{k+i} i W_k W_i \xi_{N+1-k-i,j}(\mathbb{W},\mathbb{B}\cup\mathbb{Y}),
\end{eqnarray*}
where $\xi_{k,j}$ is defined as in Lemma \ref{fancy-U}. Then, by Lemma \ref{fancy-U}, we have
\[
C(\Gamma) \cong \left(%
\begin{array}{ll}
  U_1 & D_1 -E_1 \\
  \dots & \dots \\
  U_{n+l} & D_{n+l} - E_{n+l} \\
  V_1 & X_1 -A_1 -Y_1 \\
  \dots & \dots \\
  V_{m} & X_{m} - A_{n}Y_{m-n} \\
  \hat{V}_1 & B_1 +Y_1-W_1 \\
  \dots & \dots \\
  \hat{V}_{m+l-1} & B_{n+l-1}Y_{m-n} - W_{m+l-1} 
\end{array}%
\right)_{\Sym(\mathbb{X}|\mathbb{Y}|\mathbb{W}|\mathbb{A}|\mathbb{B}|\mathbb{T}|\{r\})} \{q^{-ln-(l+n-1)(m-n)}\}.
\]
Note that
\begin{eqnarray*}
V_j|_{X_1=\cdots=X_m=0} & = & (-1)^{j-1} p_{N+1-j}(\mathbb{A}\cup\mathbb{Y}), \\
\hat{V}_j|_{W_1=\cdots=W_{m+l-1}=0} & = & (-1)^{j-1} p_{N+1-j}(\mathbb{B}\cup\mathbb{Y}), \\
D_j|_{T_1=\cdots=T_l=0} & = & A_j, \\
E_j|_{r=0} & = & B_j.
\end{eqnarray*}
So
\[
C(\Gamma) / \mathfrak{I} \cdot C(\Gamma) \cong \left(%
\begin{array}{ll}
  \tilde{U}_1 & A_1 -B_1 \\
  \dots & \dots \\
  \tilde{U}_n & A_n -B_n \\
  \tilde{U}_{n+1} & -B_{n+1} \\
  \dots & \dots \\
  \tilde{U}_{n+l-1} & - B_{n+l-1} \\
  \tilde{U}_{n+l} & 0 \\
  p_{N}(\mathbb{A}\cup\mathbb{Y}) & -A_1 -Y_1 \\
  \dots & \dots \\
  (-1)^{m-1} p_{N+1-m}(\mathbb{A}\cup\mathbb{Y}) & - A_{n}Y_{m-n} \\
  p_{N}(\mathbb{B}\cup\mathbb{Y}) & B_1 +Y_1 \\
  \dots & \dots \\
  (-1)^{m+l-2} p_{N+1-(m+l-1)}(\mathbb{B}\cup\mathbb{Y}) & B_{n+l-1}Y_{m-n} 
\end{array}%
\right)_{\Sym(\mathbb{Y}|\mathbb{A}|\mathbb{B})} \{q^{-ln-(l+n-1)(m-n)}\},
\]
where
\[
\tilde{U}_j = U_j|_{T_1=\cdots=T_l=r=0}.
\]
Next, we exclude $B_1,\dots,B_{n+l-1}$ by applying Proposition \ref{b-contraction} to the first $n+l-1$ rows of this matrix factorization. This gives the relations
\[
B_j = A_j=\begin{cases}
A_j & \text{if } ~1\leq j \leq n, \\
0 & \text{if } ~n+1\leq j \leq n+l-1, \\
\end{cases}
\]
and
\[
C(\Gamma) / \mathfrak{I} \cdot C(\Gamma) \simeq \left(%
\begin{array}{ll}
  \tilde{U}_{n+l}|_{B_j=A_j} & 0 \\
  p_{N}(\mathbb{A}\cup\mathbb{Y}) & -A_1 -Y_1 \\
  \dots & \dots \\
  (-1)^{m-1} p_{N+1-m}(\mathbb{A}\cup\mathbb{Y}) & - A_{n}Y_{m-n} \\
  p_{N}(\mathbb{A}\cup\mathbb{Y}) & A_1 +Y_1 \\
  \dots & \dots \\
  (-1)^{m-1} p_{N+1-m}(\mathbb{A}\cup\mathbb{Y}) & A_m Y_{m-n} \\
  (-1)^{m} p_{N+1-(m+1)}(\mathbb{A}\cup\mathbb{Y}) & 0_{m+1} \\
  \dots & \dots \\
  (-1)^{m+l-2} p_{N+1-(m+l-1)}(\mathbb{A}\cup\mathbb{Y}) & 0_{m+l-1}
\end{array}%
\right)_{\Sym(\mathbb{Y}|\mathbb{A})} \{q^{-ln-(l+n-1)(m-n)}\}.
\]
By Corollary \ref{row-op}, we have
\[
C(\Gamma) / \mathfrak{I} \cdot C(\Gamma) \simeq \left(%
\begin{array}{ll}
  \tilde{U}_{n+l}|_{B_j=A_j} & 0 \\
  0 & -A_1 -Y_1 \\
  \dots & \dots \\
  0 & - A_{n}Y_{m-n} \\
  p_{N}(\mathbb{A}\cup\mathbb{Y}) & 0_1 \\
  \dots & \dots \\
  (-1)^{m+l-2} p_{N+1-(m+l-1)}(\mathbb{A}\cup\mathbb{Y}) & 0_{m+l-1}
\end{array}%
\right)_{\Sym(\mathbb{Y}|\mathbb{A})} \{q^{-ln-(l+n-1)(m-n)}\}.
\]
Since $m+l-1 \leq N$, $p_{N+1-(m+l-1)}(\mathbb{A}\cup\mathbb{Y}),\dots,p_{N}(\mathbb{A}\cup\mathbb{Y})$ belong to the ideal generated by $A_1 +Y_1,\dots,\sum_{j=0}^k A_jY_{k-j},\dots,A_nY_{m-n}$. So, by Corollary \ref{twist}, we have
\[
C(\Gamma) / \mathfrak{I} \cdot C(\Gamma) \simeq \left(%
\begin{array}{ll}
  \tilde{U}_{n+l}|_{B_j=A_j} & 0 \\
  0 & -A_1 -Y_1 \\
  \dots & \dots \\
  0 & - A_{n}Y_{m-n} \\
  0 & 0_1 \\
  \dots & \dots \\
  0 & 0_{m+l-1}
\end{array}%
\right)_{\Sym(\mathbb{Y}|\mathbb{A})} \{q^{-ln-(l+n-1)(m-n)}\}.
\]
Note that, by Lemma \ref{power-derive},
\begin{eqnarray*}
\tilde{U}_{n+l}|_{B_j=A_j} & = & U_{n+l}|_{T_1=\cdots=T_l=r=0, B_j=A_j} \\
& = & \frac{\partial}{\partial D_{l+n}} p_{l+n,N+1}(D_1,\dots,D_{l+n})|_{D_j=A_j} \\
& = & (-1)^{l+n+1} (N+1) h_{l+n,N+1-l-n}(A_1,\dots,A_n,0\dots,0) \\
& = & (-1)^{l+n+1} (N+1) h_{N+1-l-n}(\mathbb{A}).
\end{eqnarray*}
Using Corollary \ref{jumping-factor}, it is easy to see that
\[
C(\Gamma) / \mathfrak{I} \cdot C(\Gamma) \simeq \left(%
\begin{array}{ll}
  h_{N+1-l-n}(\mathbb{A}) & 0 \\
  0 & -A_1 -Y_1 \\
  \dots & \dots \\
  0 & - A_{n}Y_{m-n} \\
  0 & 0_1 \\
  \dots & \dots \\
  0 & 0_{m+l-1}
\end{array}%
\right)_{\Sym(\mathbb{Y}|\mathbb{A})} \{q^{-ln-(l+n-1)(m-n)}\}.
\]
Next we exclude $Y_1,\dots,Y_{m-n}$ by applying Proposition \ref{b-contraction} to the second row through the $(m-n+1)$-th row. This gives the relations
\[
Y_j = (-1)^j h_j (\mathbb{A}) \text{ for } j=0,1,\dots,m-n,
\] 
and 
\[
C(\Gamma) / \mathfrak{I} \cdot C(\Gamma) \simeq \left(%
\begin{array}{ll}
  h_{N+1-l-n}(\mathbb{A}) & 0 \\
  0 & -\sum_{j=0}^{m-n} (-1)^j h_j (\mathbb{A}) A_{m-n+1-j} \\
  \dots & \dots \\
  0 & -\sum_{j=0}^{m-n} (-1)^j h_j (\mathbb{A}) A_{k-j} \\
  \dots & \dots \\
  0 & - (-1)^{m-n} h_{m-n} (\mathbb{A})A_{n} \\
  0 & 0_1 \\
  \dots & \dots \\
  0 & 0_{m+l-1}
\end{array}%
\right)_{\Sym(\mathbb{A})} \{q^{-ln-(l+n-1)(m-n)}\}.
\]
By equation \eqref{complete-recursion}, we have that, for $k=m-n+1,\dots m$
\[
\sum_{j=0}^{m-n} (-1)^j h_j (\mathbb{A}) A_{k-j} = - \sum_{j=m-n+1}^{k} (-1)^j h_j (\mathbb{A}) A_{k-j}.
\]
So, using Corollaries \ref{row-op} and \ref{jumping-factor}, we get
\[
C(\Gamma) / \mathfrak{I} \cdot C(\Gamma) \simeq \left(%
\begin{array}{ll}
  h_{N+1-l-n}(\mathbb{A}) & 0 \\
  0 & h_{m-n+1}(\mathbb{A}) \\
  \dots & \dots \\
  0 & h_{m}(\mathbb{A}) \\
  0 & 0_1 \\
  \dots & \dots \\
  0 & 0_{m+l-1}
\end{array}%
\right)_{\Sym(\mathbb{A})} \{q^{-ln-(l+n-1)(m-n)}\}.
\]

If $m+l\leq N$, then $N+1-l-n \geq m-n+1$ and, therefore, $h_{N+1-l-n}(\mathbb{A})$ is in the ideal $(h_{m-n+1}(\mathbb{A}),\dots,h_{m}(\mathbb{A}))$. So, by Corollary \ref{twist},
\[
C(\Gamma) / \mathfrak{I} \cdot C(\Gamma) \simeq \left(%
\begin{array}{ll}
  0 & 0_{n+l} \\
  0 & h_{m-n+1}(\mathbb{A}) \\
  \dots & \dots \\
  0 & h_{m}(\mathbb{A}) \\
  0 & 0_1 \\
  \dots & \dots \\
  0 & 0_{m+l-1}
\end{array}%
\right)_{\Sym(\mathbb{A})} \{q^{-ln-(l+n-1)(m-n)}\}.
\]
Thus, by Proposition \ref{b-contraction-weak},
\[
H(C(\Gamma) / \mathfrak{I} \cdot C(\Gamma)) \cong \left(%
\begin{array}{ll}
  0 & 0_{n+l} \\
  0 & 0_1 \\
  \dots & \dots \\
  0 & 0_{m+l-1}
\end{array}%
\right)_{\Sym(\mathbb{A})/(h_{m-n+1}(\mathbb{A}),\dots,h_{m}(\mathbb{A}))} \{q^{-ln-(l+n-1)(m-n)}\}.
\]
Since the graded dimension of $\Sym(\mathbb{A})/(h_{m-n+1}(\mathbb{A}),\dots,h_{m}(\mathbb{A}))$ is $\qb{m}{n}\cdot q^{n(m-n)}$, it follows that
\[
\gdim C(\Gamma) = q^{-lm+m-n}\qb{m}{n}(1+\tau q^{2n+2l-N-1})\prod_{j=1}^{m+l-1}(1+\tau q^{2j-N-1}).
\]

If $m+l=N+1$, then $N+1-l-n =m-n$ and $h_{m}(\mathbb{A})$ is in the ideal $(h_{m-n}(\mathbb{A}),\dots,h_{m-1}(\mathbb{A}))$. By Lemma \ref{column-reverse-signs} and Corollary \ref{row-op}, we have
\begin{eqnarray*}
C(\Gamma) / \mathfrak{I} \cdot C(\Gamma) & \simeq &  \left(%
\begin{array}{ll}
  0 & h_{m-n}(\mathbb{A}) \\
  0 & h_{m-n+1}(\mathbb{A}) \\
  \dots & \dots \\
  0 & h_{m}(\mathbb{A}) \\
  0 & 0_1 \\
  \dots & \dots \\
  0 & 0_{m+l-1}
\end{array}%
\right)_{\Sym(\mathbb{A})} \{q^{-ln-(l+n-1)(m-n)+N+1-2(m-n)}\} \left\langle 1 \right\rangle \\
& \simeq &  \left(%
\begin{array}{ll}
  0 & h_{m-n}(\mathbb{A}) \\
  \dots & \dots \\
  0 & h_{m-1}(\mathbb{A}) \\
  0 & 0_m \\
  0 & 0_1 \\
  \dots & \dots \\
  0 & 0_{m+l-1}
\end{array}%
\right)_{\Sym(\mathbb{A})} \{q^{-ln-(l+n+1)(m-n)+N+1}\}\left\langle 1 \right\rangle.
\end{eqnarray*}
Thus, by Proposition \ref{b-contraction-weak},
\[
H(C(\Gamma) / \mathfrak{I} \cdot C(\Gamma)) \simeq  \left(%
\begin{array}{ll}
  0 & 0_m \\
  0 & 0_1 \\
  \dots & \dots \\
  0 & 0_{m+l-1}
\end{array}%
\right)_{\Sym(\mathbb{A})/(h_{m-n}(\mathbb{A}),\dots,h_{m-1}(\mathbb{A}))} \{q^{-ln-(l+n+1)(m-n)+N+1}\}\left\langle 1 \right\rangle.
\]
Since the graded dimension of $\Sym(\mathbb{A})/(h_{m-n}(\mathbb{A}),\dots,h_{m-1}(\mathbb{A}))$ is $\qb{m-1}{n} \cdot q^{n(m-n-1)}$, it follows that 
\begin{eqnarray*}
&& \gdim C(\Gamma) \\
& = & \tau q^{-ln-(l+n+1)(m-n)+N+1+n(m-n-1)}\qb{m-1}{n} (1+\tau q^{2m-N-1}) \prod_{j=1}^{m+l-1}(1+\tau q^{2j-N-1}) \\
& = & q^{-lm+m}\qb{m-1}{n}(1+\tau q^{N+1-2m})\prod_{j=1}^{m+l-1}(1+\tau q^{2j-N-1}).
\end{eqnarray*}

Finally, let us consider equation \eqref{decomp-IV-dimensions-sum}.

Assume $m+l=N+1$. Then $\gdim C(\Gamma_1)=0$ and it is straightforward to see that $\gdim C(\Gamma) = \qb{m-1}{n} \gdim C(\Gamma_0)$. So \eqref{decomp-IV-dimensions-sum} is true.

Assume $m+l\leq N$. Note that
\[
\qb{m}{n} = q^{-n} \qb{m-1}{n} + q^{m-n} \qb{m-1}{n-1} = q^n \qb{m-1}{n} + q^{-m+n} \qb{m-1}{n-1}.
\]
So 
\begin{eqnarray*}
& & \qb{m}{n} (1+\tau q^{2n+2l-N-1}) \\
& = & (q^n \qb{m-1}{n} + q^{-m+n} \qb{m-1}{n-1}) + \tau q^{2n+2l-N-1} (q^{-n} \qb{m-1}{n} + q^{m-n} \qb{m-1}{n-1}) \\
& = & q^n \qb{m-1}{n}(1+\tau q^{2l-N-1}) + q^{-m+n} \qb{m-1}{n-1} (1+ \tau q^{2m+2l-N-1}).
\end{eqnarray*}
Multiplying $q^{-lm+m-n} \prod_{j=1}^{m+l-1}(1+\tau q^{2j-N-1})$ to this equation, we get \eqref{decomp-IV-dimensions-sum}.
\end{proof}

\subsection{Proof of Theorem \ref{decomp-IV}} With all the above preparations, we are now ready to prove Theorem \ref{decomp-IV}.

\begin{lemma}\label{decomp-IV-Psi-Phi}
Let $\Gamma$, $\Gamma_0$ and $\Gamma_1$ be the MOY graphs in Figure \ref{decomposition-IV-figure}, where $l,m,n$ are integers satisfying $0\leq n \leq m \leq N$ and $0\leq l, m+l-1 \leq N$. Then there exist homogeneous morphisms 
\begin{eqnarray*}
&& \Phi: C(\Gamma_0)\{\qb{m-1}{n}\} \oplus C(\Gamma_1)\{\qb{m-1}{n-1} \rightarrow C(\Gamma), \\ 
&& \Psi: C(\Gamma) \rightarrow C(\Gamma_0)\{\qb{m-1}{n}\} \oplus C(\Gamma_1)\{\qb{m-1}{n-1}\}
\end{eqnarray*} 
preserving the $\zed_2\oplus\zed$-grading such that
\[
\Psi \circ \Phi \simeq \id_{C(\Gamma_0)\{\qb{m-1}{n}\} \oplus C(\Gamma_1)\{\qb{m-1}{n-1}\}}.
\]
\end{lemma}

\begin{proof}
Let $F,G,\vec{\alpha},\vec{\beta}$ be the morphisms defined in Subsections \ref{decomp-IV-subsection-Gamma-Gamma0} and \ref{decomp-IV-subsection-Gamma-Gamma1}. Define 
\begin{eqnarray*}
&& \Phi_0: C(\Gamma_0)\{\qb{m-1}{n}\} \oplus C(\Gamma_1)\{\qb{m-1}{n-1} \rightarrow C(\Gamma), \\ 
&& \Psi_0: C(\Gamma) \rightarrow C(\Gamma_0)\{\qb{m-1}{n}\} \oplus C(\Gamma_1)\{\qb{m-1}{n-1}\}
\end{eqnarray*} 
by $\Phi_0 = (F,\vec{\alpha})$ and $\Psi_0 = (G,\vec{\beta})^T$. Then
\[
\Psi_0 \circ \Phi_0 \simeq \left(%
\begin{array}{ll}
  \id & G \circ \vec{\alpha} \\
  \vec{\beta} \circ F & \id 
\end{array}%
\right).
\]
By Lemma \ref{decomp-IV-nilopotency}, $\vec{\beta} \circ F \circ G \circ \vec{\alpha}$ and $G \circ \vec{\alpha} \circ \vec{\beta} \circ F$ are homotopically nilpotent. So $\id - \vec{\beta} \circ F \circ G \circ \vec{\alpha}$ and $\id - G \circ \vec{\alpha} \circ \vec{\beta} \circ F$ are homotopically invertible. In fact, their homotopic inverses are
\begin{eqnarray*}
(\id - \vec{\beta} \circ F \circ G \circ \vec{\alpha})^{-1} & \simeq &  \sum_{k=0}^{\infty} (\vec{\beta} \circ F \circ G \circ \vec{\alpha})^k, \\
(\id - G \circ \vec{\alpha} \circ \vec{\beta} \circ F)^{-1} & \simeq &  \sum_{k=0}^{\infty} (G \circ \vec{\alpha} \circ \vec{\beta} \circ F)^k.
\end{eqnarray*}
Note that the sums on the right hand side are finite sums in the $\Hom_\HMF$. Now define 
\begin{eqnarray*}
\Phi & = & \Phi_0, \\
\Psi & = & \left(%
\begin{array}{ll}
  (\id - G \circ \vec{\alpha} \circ \vec{\beta} \circ F)^{-1}& 0 \\
  0 & (\id - \vec{\beta} \circ F \circ G \circ \vec{\alpha})^{-1}
\end{array}%
\right) 
\circ
\left(%
\begin{array}{ll}
  \id & -G \circ \vec{\alpha} \\
  -\vec{\beta} \circ F & \id 
\end{array}%
\right)
\circ
\Psi_0.
\end{eqnarray*}
It is straightforward to check that $\Phi$ and $\Psi$ satisfy all the requirements in the lemma.
\end{proof}

\begin{proof}[Proof of Theorem \ref{decomp-IV}]
By Lemmas \ref{decomp-IV-Psi-Phi} and \ref{fully-additive-hmf}, we know that there exists a graded matrix factorization $M$ such that
\[
C(\Gamma) \simeq C(\Gamma_0)\{\qb{m-1}{n}\} \oplus C(\Gamma_1)\{\qb{m-1}{n-1}\} \oplus M.
\]
But, by Lemma \ref{decomp-IV-dimensions}, 
\[
\gdim M = \gdim C(\Gamma) - \qb{m-1}{n} \cdot \gdim C(\Gamma_0) - \qb{m-1}{n-1} \cdot \gdim C(\Gamma_1) =0.
\]
Thus, by Corollary \ref{homology-detects-homotopy}, $M\simeq 0$. So 
\[
C(\Gamma) \simeq C(\Gamma_0)\{\qb{m-1}{n}\} \oplus C(\Gamma_1)\{\qb{m-1}{n-1}\}.
\]
\end{proof}

\section{Direct Sum Decomposition (V)}\label{sec-MOY-V}

The objective of this section is to prove Theorem \ref{decomp-V}, which categorifies \cite[Proposition A.10]{MOY} and further generalizes direct sum decomposition (IV) (Theorem \ref{decomp-IV}.) The proof of Decomposition (V) is different from that of Decompositions (I-IV) in the sense that we do not explicitly construct the homotopy equivalences in Decomposition (V). Instead, we use the Krull-Schmidt property of the category $\hmf$ to prove this decomposition.

\begin{figure}[ht]
$
\xymatrix{
\input{square-m-n-l-right-k-low} && \input{square-m-n-l-right-j-high} \\
\input{square-m-n-l-left-k-high} && \input{square-m-n-l-left-j-low} \\
}
$
\caption{}\label{decomp-V-figure}

\end{figure}

\begin{theorem}\label{decomp-V}
Let $m,n,l$ be non-negative integers satisfying $n+l,m+l\leq N$. For $\max\{m-n,0\}\leq k \leq m+l$ and $\max\{m-n,0\}\leq j \leq m$, define $\Gamma_k^1$, $\Gamma_k^3$, $\Gamma_j^2$ and $\Gamma_j^4$ to be the MOY graphs in Figure \ref{decomp-V-figure}. Then, for $\max\{m-n,0\}\leq k \leq m+l$,
\begin{eqnarray}
\label{decomp-V-1} C(\Gamma_k^1) & \simeq & \bigoplus_{j=\max\{m-n,0\}}^m C(\Gamma_j^2) \{\qb{l}{k-j}\}, \\
\label{decomp-V-2} C(\Gamma_k^3) & \simeq & \bigoplus_{j=\max\{m-n,0\}}^m C(\Gamma_j^4) \{\qb{l}{k-j}\},
\end{eqnarray}
where we use the convention $\qb{a}{b}=0$ if $b<0$ or $b>a$.
\end{theorem}

\subsection{The proof} The $n\geq m$ case and the $n<m$ case of Theorem \ref{decomp-V} may seem different. But, by flipping $\Gamma_k^1$, $\Gamma_k^3$, $\Gamma_j^2$ and $\Gamma_j^4$ horizontally and shifting the indicies $k,j$, one can easily check that the $n\geq m$ (resp. $m\geq n$) case of equation \eqref{decomp-V-1} is equivalent to the $m\geq n$ (resp. $n\geq m$) case of equation \eqref{decomp-V-2}. So, without loss of generality, we prove Theorem \ref{decomp-V} under the assumption $n\geq m$.

We prove Theorem \ref{decomp-V} by inducting on $k$. If $k=0$, then decompositions \eqref{decomp-V-1} and \eqref{decomp-V-2} are trivially true. We prove the $k=1$ case in the following lemma.

\begin{figure}[ht]
$
\xymatrix{
\input{square-m-n-l-right-1-low} && \input{square-m-n-l-right-1-high} & \input{square-m-n-l-right-0-high} \\
}
$
\caption{}\label{decomp-V-induction-1-figure}
\end{figure}

\begin{lemma}\label{decomp-V-induction-1}
Let $\Gamma_k^1$, $\Gamma_k^3$, $\Gamma_j^2$ and $\Gamma_j^4$ to be as in Theorem \ref{decomp-V}. Assume that $n\geq m$. Then
\begin{eqnarray}
\label{decomp-V-induction-1-1} C(\Gamma_1^1) & \simeq & C(\Gamma_1^2) \oplus C(\Gamma_0^2)\{[l]\}, \\
\label{decomp-V-induction-1-2} C(\Gamma_1^3) & \simeq & C(\Gamma_1^4) \oplus C(\Gamma_0^4)\{[l]\}.
\end{eqnarray}
\end{lemma}

\begin{proof}
The proofs of \eqref{decomp-V-induction-1-1} and \eqref{decomp-V-induction-1-2} are very similar. So we only prove \eqref{decomp-V-induction-1-1} here and leave \eqref{decomp-V-induction-1-2} to the reader.

\begin{figure}[ht]

\setlength{\unitlength}{1pt}

\begin{picture}(360,105)(-180,-15)


\put(-40,0){\vector(0,1){25}}

\put(-40,65){\vector(0,1){25}}

\put(-40,25){\vector(0,1){40}}

\put(0,25){\vector(0,1){40}}

\put(0,65){\vector(-1,0){40}}

\put(0,65){\vector(1,0){40}}

\put(40,25){\vector(-1,0){40}}

\put(40,25){\vector(0,1){40}}

\put(-40,25){\vector(1,0){40}}

\put(40,0){\vector(0,1){25}}

\put(40,65){\vector(0,1){25}}

\put(-47,13){\tiny{$n$}}

\put(-47,72){\tiny{$m$}}

\put(-63,42){\tiny{$m-1$}}

\put(43,13){\tiny{$m+l$}}

\put(43,72){\tiny{$n+l$}}

\put(43,42){\tiny{$m+l-1$}}

\put(-35,28){\tiny{$n-m+1$}}

\put(-15,42){\tiny{$n-m+2$}}

\put(-22,58){\tiny{$1$}}

\put(18,28){\tiny{$1$}}

\put(5,58){\tiny{$n-m+1$}}

\put(-2,-15){$\Gamma$}

\end{picture}

\caption{}\label{decomposition-IV-cor1-proof-figure1}

\end{figure}

Consider the MOY graph $\Gamma$ in Figure \ref{decomposition-IV-cor1-proof-figure1}. Applying Decomposition (IV) (Theorem \ref{decomp-IV}) to the left square in $\Gamma$, we get $C(\Gamma) \simeq C(\Gamma_1^1) \oplus C(\Gamma')\{[m-1]\}$, where $\Gamma'$ is given in Figure \ref{decomposition-IV-cor1-proof-figure2}. By Corollary \ref{contract-expand} and Decomposition (II) (Theorem \ref{decomp-II}), we have $C(\Gamma') \simeq C(\Gamma'')\simeq C(\Gamma_0^2)\{[m+l]\}$. Thus,
\begin{equation}\label{decomp-V-induction-1-1-left}
C(\Gamma) \simeq C(\Gamma_1^1) \oplus C(\Gamma_0^2)\{[m-1][m+l]\}.
\end{equation}

\begin{figure}[ht]

\setlength{\unitlength}{1pt}

\begin{picture}(360,105)(-180,-15)


\put(-160,0){\vector(0,1){45}}

\put(-160,45){\vector(0,1){45}}

\put(-160,45){\vector(1,0){30}}

\put(-130,45){\vector(2,1){30}}

\put(-100,30){\vector(-2,1){30}}

\put(-100,0){\vector(0,1){30}}

\put(-100,30){\vector(0,1){30}}

\put(-100,60){\vector(0,1){30}}

\put(-167,13){\tiny{$n$}}

\put(-167,72){\tiny{$m$}}

\put(-97,13){\tiny{$m+l$}}

\put(-97,72){\tiny{$n+l$}}

\put(-97,42){\tiny{$m+l-1$}}

\put(-115,30){\tiny{$1$}}

\put(-155,48){\tiny{$_{n-m}$}}

\put(-140,55){\tiny{$_{n-m+1}$}}

\put(-142,-15){$\Gamma'$}


\put(80,0){\vector(0,1){60}}

\put(80,60){\vector(0,1){30}}

\put(80,60){\vector(1,0){60}}

\put(140,60){\vector(0,1){30}}

\put(140,45){\vector(0,1){15}}

\put(130,20){\vector(0,1){20}}

\put(150,20){\vector(0,1){20}}

\put(140,0){\vector(0,1){15}}

\qbezier(140,45)(130,45)(130,40)

\qbezier(140,45)(150,45)(150,40)

\qbezier(140,15)(130,15)(130,20)

\qbezier(140,15)(150,15)(150,20)

\put(73,13){\tiny{$n$}}

\put(73,72){\tiny{$m$}}

\put(100,62){\tiny{$n-m$}}

\put(143,5){\tiny{$m+l$}}

\put(143,50){\tiny{$m+l$}}

\put(143,72){\tiny{$n+l$}}

\put(153,28){\tiny{$m+l-1$}}

\put(123,28){\tiny{$1$}}

\put(108,-15){$\Gamma''$}

\end{picture}

\caption{}\label{decomposition-IV-cor1-proof-figure2}

\end{figure}

Now apply Decomposition (IV) (Theorem \ref{decomp-IV}) to the right square in $\Gamma$. This gives $C(\Gamma) \simeq C(\Gamma_1^2) \oplus C(\Gamma''')\{[m+l-1]\}$, where $\Gamma'''$ is given in Figure \ref{decomposition-IV-cor1-proof-figure3}. By Corollary \ref{contract-expand} and Decomposition (II) (Theorem \ref{decomp-II}), we have $C(\Gamma''') \simeq C(\Gamma'''')\simeq C(\Gamma_0^2)\{[m]\}$. Thus,
\begin{equation}\label{decomp-V-induction-1-1-right}
C(\Gamma) \simeq C(\Gamma_1^2) \oplus C(\Gamma_0^2)\{[m][m+l-1]\}.
\end{equation}

\begin{figure}[ht]

\setlength{\unitlength}{1pt}

\begin{picture}(360,105)(-180,-15)


\put(-100,0){\vector(0,1){45}}

\put(-100,45){\vector(0,1){45}}

\put(-130,45){\vector(1,0){30}}

\put(-160,30){\vector(2,1){30}}

\put(-130,45){\vector(-2,1){30}}

\put(-160,0){\vector(0,1){30}}

\put(-160,30){\vector(0,1){30}}

\put(-160,60){\vector(0,1){30}}

\put(-167,13){\tiny{$n$}}

\put(-167,72){\tiny{$m$}}

\put(-180,43){\tiny{$m-1$}}

\put(-97,13){\tiny{$m+l$}}

\put(-97,72){\tiny{$m+l$}}

\put(-145,53){\tiny{$1$}}

\put(-125,48){\tiny{$n-m$}}

\put(-146,32){\tiny{$n-m+1$}}

\put(-142,-15){$\Gamma'''$}


\put(80,0){\vector(0,1){30}}

\put(80,30){\vector(0,1){15}}

\put(80,75){\vector(0,1){15}}

\put(80,30){\vector(1,0){60}}

\put(140,30){\vector(0,1){60}}

\put(70,50){\vector(0,1){20}}

\put(90,50){\vector(0,1){20}}

\put(140,0){\vector(0,1){30}}

\qbezier(80,45)(70,45)(70,50)

\qbezier(80,45)(90,45)(90,50)

\qbezier(80,75)(70,75)(70,70)

\qbezier(80,75)(90,75)(90,70)

\put(73,5){\tiny{$n$}}

\put(73,80){\tiny{$m$}}

\put(73,35){\tiny{$m$}}

\put(93,58){\tiny{$1$}}

\put(50,58){\tiny{$m-1$}}

\put(100,32){\tiny{$n-m$}}

\put(143,5){\tiny{$m+l$}}

\put(143,72){\tiny{$n+l$}}

\put(108,-15){$\Gamma''''$}

\end{picture}

\caption{}\label{decomposition-IV-cor1-proof-figure3}

\end{figure}

Note that $[m]\cdot [m+l-1] - [m-1]\cdot [m+l] =[l]$. So, by the Krull-Schmidt property of the category $\hmf$ (Proposition \ref{hmf-is-KS} and Lemma \ref{KS-oplus-cancel}), we know that \eqref{decomp-V-induction-1-1-left} and \eqref{decomp-V-induction-1-1-right} imply \eqref{decomp-V-induction-1-1}.
\end{proof}

With the above initial case in hand, we are ready to prove Theorem \ref{decomp-V} in general.

\begin{proof}[Proof of Theorem \ref{decomp-V}]
From above, we know that \eqref{decomp-V-1} and \eqref{decomp-V-2} are true for $k=0,1$. Now assume \eqref{decomp-V-1} and \eqref{decomp-V-2} are true for a given $k\geq 1$ and all $m,n,l$ satisfying the conditions in Theorem \ref{decomp-V}. We claim that \eqref{decomp-V-1} and \eqref{decomp-V-2} are also true for $k+1$. The proofs for the $k+1$ cases of \eqref{decomp-V-1} and \eqref{decomp-V-2} are very similar. We only proof \eqref{decomp-V-1} for $k+1$ here and leave \eqref{decomp-V-2} to the reader.

\begin{figure}[ht]
$
\xymatrix{
\input{square-m-n-l-right-k+1-low} && \input{square-m-n-l-right-j+1-high} \\
\input{square-m-n-l-right-k+1-low-tilde} && \input{square-m-n-l-right-j+1-high-tilde} \\
}
$
\caption{}\label{decomp-V-proof-figure1}

\end{figure}

Recall that $\Gamma_{k+1}^1$ and $\Gamma_{j+1}^2$ are the MOY graphs in the first row of Figure \ref{decomp-V-proof-figure1}. We define $\widetilde{\Gamma}_{k+1}^1$ and $\widetilde{\Gamma}_{j+1}^2$ to be the MOY graphs in the second row in Figure \ref{decomp-V-proof-figure1}. By Corollary \ref{contract-expand} and Decomposition (II) (Theorem \ref{decomp-II}), we have
\begin{eqnarray*}
C(\widetilde{\Gamma}_{k+1}^1) & \simeq & C(\Gamma_{k+1}^1)\{[k+1]\}, \\
C(\widetilde{\Gamma}_{j+1}^2) & \simeq & C(\Gamma_{j+1}^2)\{[j+1]\}.
\end{eqnarray*}

\begin{figure}[ht]
$
\xymatrix{
\input{square-m-n-l-right-k-low-hat}
}
$
\caption{}\label{decomp-V-proof-figure2}

\end{figure}

\textit{Case 1.} $k \leq l$. Apply \eqref{decomp-V-induction-1-1} to the upper rectangle in $\widetilde{\Gamma}_{k+1}^1$. This gives
\[
C(\widetilde{\Gamma}_{k+1}^1) \simeq C(\widehat{\Gamma}_{k}^1) \oplus C(\Gamma_k^1)\{[l-k]\},
\]
where $\widehat{\Gamma}_{k}^1$ is the MOY graph in Figure \ref{decomp-V-proof-figure2} and $\Gamma_k^1$ is given in Figure \ref{decomp-V-figure}. Recall that we assume \eqref{decomp-V-1} is true for the given $k$ and \textbf{all} $m,n,l$ satisfying the conditions in Theorem \ref{decomp-V}. Thus, we can apply \eqref{decomp-V-1} to the lower rectangle in $\widehat{\Gamma}_{k}^1$ and get
\begin{eqnarray*}
C(\widehat{\Gamma}_{k}^1) & \simeq & \bigoplus_{j=0}^{m-1} C(\widetilde{\Gamma}_{j+1}^2)\{\qb{l+1}{k-j}\} \\
& \simeq & \bigoplus_{j=0}^{m-1} C(\Gamma_{j+1}^2)\{[j+1]\qb{l+1}{k-j}\} = \bigoplus_{j=0}^{m} C(\Gamma_{j}^2)\{[j]\qb{l+1}{k-j+1}\}.
\end{eqnarray*}
Again, recall that we assume \eqref{decomp-V-1} is true for $\Gamma_k^1$. That is,
\[
C(\Gamma_k^1) \simeq \bigoplus_{j=0}^m C(\Gamma_j^2) \{\qb{l}{k-j}\}.
\]
Note that $[j]\qb{l+1}{k-j+1} + [l-k]\qb{l}{k-j} = \qb{l}{k+1-j}[k+1]$. So, combining the above, we get 
\[
C(\Gamma_{k+1}^1)\{[k+1]\} \simeq C(\widetilde{\Gamma}_{k+1}^1) \simeq \bigoplus_{j=0}^{m} C(\Gamma_{j}^2) \{\qb{l}{k+1-j}[k+1]\}.
\]
By Proposition \ref{yonezawa-lemma-hmf}, this implies 
\[
C(\Gamma_{k+1}^1) \simeq \bigoplus_{j=0}^{m} C(\Gamma_{j}^2) \{\qb{l}{k+1-j}\}.
\]
So \eqref{decomp-V-1} is true for $k+1$ if $k\leq l$.

\textit{Case 2.} $k>l$. In this case, we apply \eqref{decomp-V-induction-1-2} to the upper rectangle of $\widehat{\Gamma}_{k}^1$. This gives
\[
C(\widehat{\Gamma}_{k}^1) \simeq C(\widetilde{\Gamma}_{k+1}^1) \oplus C(\Gamma_k^1)\{[k-l]\}.
\]
Note that, in this case, we also have 
\[
C(\widehat{\Gamma}_{k}^1) \simeq \bigoplus_{j=0}^{m} C(\Gamma_{j}^2)\{[j]\qb{l+1}{k-j+1}\}
\]
and 
\[
C(\Gamma_k^1) \simeq \bigoplus_{j=0}^m C(\Gamma_j^2) \{\qb{l}{k-j}\}.
\]
Note that $[j]\qb{l+1}{k-j+1} - [k-l]\qb{l}{k-j} = \qb{l}{k+1-j}[k+1]$. So, by Lemma \ref{KS-oplus-cancel}, we have 
\[
C(\Gamma_{k+1}^1)\{[k+1]\} \simeq C(\widetilde{\Gamma}_{k+1}^1) \simeq \bigoplus_{j=0}^{m} C(\Gamma_{j}^2) \{\qb{l}{k+1-j}[k+1]\}.
\]
By Proposition \ref{yonezawa-lemma-hmf}, this implies 
\[
C(\Gamma_{k+1}^1) \simeq \bigoplus_{j=0}^{m} C(\Gamma_{j}^2) \{\qb{l}{k+1-j}\}.
\]
So \eqref{decomp-V-1} is true for $k+1$ if $k>l$.
\end{proof}

\section{Chain Complexes Associated to Knotted MOY Graphs}\label{sec-chain-complex-def}

\begin{figure}[h]

\setlength{\unitlength}{1pt}

\begin{picture}(360,50)(-180,-30)

\linethickness{.5pt}


\put(-100,-20){\vector(1,1){40}}

\put(-60,-20){\line(-1,1){15}}

\put(-85,5){\vector(-1,1){15}}

\put(-84,-30){$+$}


\put(100,-20){\vector(-1,1){40}}

\put(60,-20){\line(1,1){15}}

\put(85,5){\vector(1,1){15}}

\put(76,-30){$-$}

\end{picture}

\caption{}\label{crossing-sign-figure}

\end{figure}

\begin{definition}\label{knotted-MOY-def}
A knotted MOY graph is an immersion of an abstract MOY graph into $\mathbb{R}^2$ such that
\begin{itemize}
	\item the only singularities are finitely many transversal double points in the interior of edges (that is, away from the vertices),
	\item we specify the upper edge and the lower edge at each of these transversal double points.
\end{itemize}
Each transversal double point in a knotted MOY graph is called a crossing. We follow the usual sign convention for crossings given in Figure \ref{crossing-sign-figure}.

If there are crossings in an edge, these crossing divide this edge into several parts. We call each part a segment of the edge.
\end{definition}

Note that colored oriented link/tangle diagrams and (embedded) MOY graphs are special cases of knotted MOY graphs.

\begin{definition}\label{knotted-MOY-marking-def}
A marking of a knotted MOY graph $D$ consists of the following:
\begin{enumerate}
	\item A finite collection of marked points on $D$ such that
	\begin{itemize}
	\item every segment of every edge of $D$ has at least one marked point;
	\item all the end points (vertices of valence $1$) are marked;
	\item none of the crossings and internal vertices (vertices of valence at least $2$) are marked.
  \end{itemize}
  \item An assignment of pairwise disjoint alphabets to the marked points such that the alphabet associated to a marked point on an edge of color $m$ has $m$ independent indeterminates. (Recall that an alphabet is a finite collection of homogeneous indeterminates of degree $2$.)
\end{enumerate}

Given a knotted MOY graph $D$ with a marking, we cut $D$ at the marked points. This produces a collection $\{D_1,\dots,D_m\}$ of simple knotted MOY graphs marked only at their end points. We call each $D_i$ a piece of $D$. It is easy to see that each $D_i$ is one of the following:
\begin{enumerate}[(i)]
	\item an oriented arc from one marked point to another,
	\item a star-shaped neighborhood of a vertex in an (embedded) MOY graph,
	\item a crossing with colored branches.
\end{enumerate}
\end{definition}

For a given $D_i$, let $\mathbb{X}_1,\dots,\mathbb{X}_{n_i}$ be the alphabets assigned to the end points of $D_i$, among which $\mathbb{X}_1,\dots,\mathbb{X}_{k_i}$ are assigned to exits and $\mathbb{X}_{k_i+1},\dots,\mathbb{X}_{n_i}$ are assigned to entrances. Let $R_i=\Sym(\mathbb{X}_1|\cdots|\mathbb{X}_{n_i})$ and $w_i= \sum_{j=1}^{k_i} p_{N+1}(\mathbb{X}_j) - \sum_{j=k_i+1}^{n_i} p_{N+1}(\mathbb{X}_j)$. Then the chain complex $C(D_i)$ associated to $D_i$ is an object of $\hch(\hmf_{R_i,w_i})$.

If $D_i$ is of type (i) or (ii), then it is an (embedded) MOY graph, and its matrix factorization $C(D_i)$ is an object of $\hmf_{R_i,w_i}$. We define both the unnormalized chain complex $\hat{C}(D_i)$ and the normalized chain complex $C(D_i)$ to be 
\begin{equation}\label{eq-def-complex-embedded}
\hat{C}(D_i) = C(D_i)= 0 \rightarrow C(D_i) \rightarrow 0,
\end{equation}
where $C(D_i)$  has homological grading $0$. (The abuse of notations here should not be confusing.)

If $D_i$ is of type (iii), that is, a colored crossing, then the definitions of $\hat{C}(D_i)$ and $C(D_i)$ are much more complex. The chain complexes associated to colored crossings will be defined in Definition \ref{complex-colored-crossing-def} below.

\begin{remark}
In the present paper, $\hat{C}(\ast)$ stands for the unnormalized chain complex of $\ast$ and $C(\ast)$ stands for the normalized chain complex of $\ast$. For pieces of types (i) and (ii), there is no difference between their normalized and unnormalized chain complexes. For a piece of type (iii), that is, a colored crossing, these two complexes differ by a shift in the $\zed_2\oplus \zed^{\oplus 2}$ grading. See Definition \ref{complex-colored-crossing-def} below for details.
\end{remark}

\begin{definition}\label{complex-knotted-MOY-def}
\begin{eqnarray*}
\hat{C}(D) & := & \bigotimes_{i=1}^m \hat{C}(D_i), \\
C(D) & := & \bigotimes_{i=1}^m C(D_i),
\end{eqnarray*}
where the tensor product is done over the common end points. For example, for two pieces $D_{i_1}$ and $D_{i_2}$ of $D$, let $\mathbb{W}_1,\dots,\mathbb{W}_l$ be the alphabets associated to their common end points. Then, in the above tensor product, 
\[
C(D_{i_1}) \otimes C(D_{i_2}) = C(D_{i_1}) \otimes_{\Sym(\mathbb{W}_1|\cdots|\mathbb{W}_l)} C(D_{i_2}).
\]

If $D$ is closed, that is, has no endpoints, then $\hat{C}(D)$ and $C(D)$ are objects of $\hch(\hmf_{\C,0})$. 

If $D$ has endpoints, denote by $\mathbb{E}_1,\dots,\mathbb{E}_n$ the alphabets assigned to all end points of $D$. Assume that $\mathbb{E}_1,\dots,\mathbb{E}_k$ are assigned to exits and $\mathbb{E}_{k+1},\dots,\mathbb{E}_n$ are assigned to entrances. Let $R=\Sym(\mathbb{E}_1|\cdots|\mathbb{E}_n)$ and $w= \sum_{i=1}^k p_{N+1}(\mathbb{E}_i) - \sum_{j=k+1}^n p_{N+1}(\mathbb{E}_j)$. Then $\hat{C}(D)$ and $C(D)$ are objects of $\hch(\hmf_{R,w})$.

As objects of $\hch(\hmf_{R,w})$, $\hat{C}(D)$ and $C(D)$ have a $\zed_2$-grading, a quantum grading and a homological grading.
\end{definition}

In the rest of this section, we define and study the chain complexes associated to colored crossings. For this purpose, we need to understand morphisms between matrix factorizations associated to MOY graphs of the type shown in Figure \ref{l-right-k-high}.

\begin{figure}[ht]

$
\xymatrix{
\input{square-m-n-l-right-k-high}
}
$

\caption{}\label{l-right-k-high}

\end{figure}

\subsection{Change of base ring} There is a change of base ring involved in the computation of $\Hom_\HMF (C(\Gamma_k^2),\ast)$, which is the subject of this subsection.

Let $\mathbb{A}=\{a_1,\dots,a_m\}$, $\mathbb{B}=\{b_1,\dots,b_n\}$ and $\mathbb{X}=\{x_1,\dots,x_{m+n}\}$ be alphabets. Denote by $A_k$, $B_k$ and $X_k$ the $k$-th elementary symmetric polynomials in $\mathbb{A}$, $\mathbb{B}$ and $\mathbb{X}$. Define
\begin{eqnarray}
\label{eq-def-E-k-base-ring} E_k & = & X_k - \sum_{j=0}^k A_j B_{k-j}, \\
\label{eq-def-H-k-base-ring} H_k & = & \sum_{j=0}^k (-1)^j h_j(\mathbb{A}) X_{k-j} -B_k \\
\nonumber & = & \begin{cases}
\sum_{j=0}^k (-1)^j h_j(\mathbb{A}) X_{k-j} -B_k & \text{if } k=0,1,\dots,n, \\
\sum_{j=0}^k (-1)^j h_j(\mathbb{A}) X_{k-j} &  \text{if } k=n+1,\dots,n+m.
\end{cases}
\end{eqnarray}
Define $I_1$ and $I_2$ to be the homogeneous ideals of $\Sym(\mathbb{A}|\mathbb{B}|\mathbb{X})$ given by 
\begin{eqnarray*}
I_1 & = & (E_1,\dots,E_{m+n}), \\
I_2 & = & (H_1,\dots,H_{m+n}).
\end{eqnarray*}

\begin{lemma}\label{complex-def-two-ideals}
$I_1=I_2$.
\end{lemma}

\begin{proof}
First, note that
\begin{eqnarray*}
\sum_{i=0}^k (-1)^i h_i(\mathbb{A}) E_{k-i}
& = & \sum_{i=0}^k (-1)^i h_i(\mathbb{A}) X_{k-i} - \sum_{i=0}^k \sum_{j=0}^{k-i} (-1)^i h_i(\mathbb{A}) A_{k-i-j} B_j \\
& = & \sum_{i=0}^k (-1)^i h_i(\mathbb{A}) X_{k-i} - \sum_{j=0}^k B_j \sum_{i=0}^{k-j} (-1)^i h_i(\mathbb{A}) A_{k-i-j} \\
_{\text{(by equation \eqref{complete-recursion})}} & = & \sum_{i=0}^k (-1)^i h_i(\mathbb{A}) X_{k-i} - B_k = H_k.
\end{eqnarray*}
This shows that $I_2 \subset I_1$.

Next, we have
\begin{eqnarray*}
\sum_{i=0}^k A_i H_{k-i} & = & \sum_{i=0}^k A_i\sum_{j=0}^{k-i} (-1)^{k-i-j} h_{k-i-j}(\mathbb{A}) X_j - \sum_{i=0}^k A_i B_{k-i} \\
& = & \sum_{j=0}^k  X_j \sum_{i=0}^{k-j} (-1)^{k-i-j} h_{k-i-j}(\mathbb{A})A_i - \sum_{i=0}^k A_i B_{k-i} \\
_{\text{(by equation \eqref{complete-recursion})}} & = & X_k - \sum_{i=0}^k A_i B_{k-i} = E_k.
\end{eqnarray*}
So $I_1 \subset I_2$. Altogether, we have $I_1 =I_2$.
\end{proof}

Note that, for $k=n+1,\dots,n+m$, $H_k \in \Sym(\mathbb{A}|\mathbb{X})$. Define $I_3$ to be the homogeneous ideal of $\Sym(\mathbb{A}|\mathbb{X})$ given by $I_3 = (H_{n+1},\dots,H_{n+m})$.

\begin{lemma}\label{complex-def-computing-quotient-ring}
The quotient ring $\Sym(\mathbb{A}|\mathbb{X})/I_3$ is a finitely generated graded-free $\Sym(\mathbb{X})$-module of graded rank $\qb{m+n}{n}$. 

As graded $\Sym(\mathbb{A}|\mathbb{X})/I_3$-modules, we have 
\begin{equation}\label{complex-def-computing-quotient-ring-1}
\Hom_{\Sym(\mathbb{X})} (\Sym(\mathbb{A}|\mathbb{X})/I_3,\Sym(\mathbb{X})) \cong \Sym(\mathbb{A}|\mathbb{X})/I_3 ~\{q^{-2mn}\}.
\end{equation}
\end{lemma}

\begin{proof}
Note that
\[
\Sym(\mathbb{A}|\mathbb{X})/I_3 \cong \Sym(\mathbb{A}|\mathbb{B}|\mathbb{X})/I_2 \cong \Sym(\mathbb{A}|\mathbb{B}|\mathbb{X})/I_1,
\]
where the isomorphisms preserve both the graded ring structure and the graded $\Sym(\mathbb{X})$-module structure.

By Theorem \ref{part-symm-str}, $\Sym(\mathbb{A}|\mathbb{B}|\mathbb{X})/I_1$ is a finitely generated graded-free $\Sym(\mathbb{X})$-module of graded rank $\qb{m+n}{n}$. From the above isomorphism, so is $\Sym(\mathbb{A}|\mathbb{X})/I_3$. 

Note that $\Sym(\mathbb{A}|\mathbb{X})/I_3 \cong \Sym(\mathbb{A}|\mathbb{B}|\mathbb{X})/I_1 \cong \Sym(\mathbb{A}|\mathbb{B})$. By Theorem \ref{part-symm-str}, there are a Sylvester operator on $\Sym(\mathbb{A}|\mathbb{B})$ and a pair of homogeneous $\Sym(\mathbb{A}\cup\mathbb{B})$-bases for $\Sym(\mathbb{A}|\mathbb{B})$ that are duals of each other under the Sylvester operator. These induce a pair of homogeneous $\Sym(\mathbb{X})$-bases $\{S_\lambda|\lambda\in\Lambda_{m,n}\}$ and $\{S'_\lambda|\lambda\in\Lambda_{m,n}\}$ for $\Sym(\mathbb{A}|\mathbb{X})/I_3$ and a Sylvester operator 
\[
\zeta: \Sym(\mathbb{A}|\mathbb{X})/I_3 \rightarrow \Sym(\mathbb{X})
\] 
such that, for $\lambda,\mu\in\Lambda_{m,n}$, 
\[
\zeta(S_\lambda \cdot S'_\mu) = \left\{%
\begin{array}{ll}
    1 & \text{if } \mu=\lambda^c, \\
    0 & \text{if } \mu\neq\lambda^c. \\
\end{array}%
\right.
\]
(Recall that $\Lambda_{m,n} = \{\lambda=(\lambda_1\geq\cdots\geq\lambda_m)~|~\lambda_1\leq n\}$, and $\lambda^c = (n-\lambda_m\geq\cdots\geq n-\lambda_1)$.)

One can see from the above that $\{\zeta(S_\lambda \cdot \ast)|\lambda\in\Lambda_{m,n}\}$ is the $\Sym(\mathbb{X})$-basis of $\Hom_{\Sym(\mathbb{X})} (\Sym(\mathbb{A}|\mathbb{X})/I_3,\Sym(\mathbb{X}))$ dual to $\{S'_\lambda|\lambda\in\Lambda_{m,n}\}$. So the $\Sym(\mathbb{X})$-module map 
\[
\Sym(\mathbb{A}|\mathbb{X})/I_3 \rightarrow \Hom_{\Sym(\mathbb{X})} (\Sym(\mathbb{A}|\mathbb{X})/I_3,\Sym(\mathbb{X}))
\] 
given by $u \mapsto \zeta(u \cdot \ast)$ is a homogeneous isomorphism of $\Sym(\mathbb{X})$-modules of degree $-2mn$. It is easy to see that this map is also $\Sym(\mathbb{A}|\mathbb{X})/I_3$-linear. This proves \eqref{complex-def-computing-quotient-ring-1}.
\end{proof}

\begin{lemma}\label{complex-def-general-ring-change}
Let $\mathbb{A}=\{a_1,\dots,a_m\}$, $\mathbb{X}=\{x_1,\dots,x_{m+n}\}$, $\mathbb{Y}_1, \dots,\mathbb{Y}_k$ be alphabets. Define
\begin{eqnarray*}
R & = & \Sym(\mathbb{A}|\mathbb{X}|\mathbb{Y}_1|\cdots|\mathbb{Y}_k)/(H_{n+1},\dots,H_{n+m}), \\
\hat{R} & = & \Sym(\mathbb{X}|\mathbb{Y}_1|\cdots|\mathbb{Y}_k),
\end{eqnarray*}
where $H_{n+1},\dots,H_{n+m}$ are the polynomials given in \eqref{eq-def-H-k-base-ring}. Then $\hat{R}$ is a subring of $R$ through the composition of the standard inclusion and quotient maps $\hat{R} \hookrightarrow \Sym(\mathbb{A}|\mathbb{X}|\mathbb{Y}_1|\cdots|\mathbb{Y}_k) \rightarrow R$.

Suppose that $w$ is a homogeneous element of $\hat{R}$ of degree $2(N+1)$ and $M$ is a finitely generated graded matrix factorization over $R$ with potential $w$. Then, $\Hom_{\hat{R}} (M,\hat{R})$ and $\Hom_R (M, \Hom_{\hat{R}} (R,\hat{R}))$ are both graded matrix factorizations over $R$ of potential $-w$. Moreover, as graded matrix factorizations over $R$,
\begin{equation}\label{complex-def-general-ring-change-1}
\Hom_{\hat{R}} (M,\hat{R}) \cong \Hom_R (M, \Hom_{\hat{R}} (R,\hat{R})) \cong \Hom_R (M, R)\{q^{-2mn}\}.
\end{equation}
\end{lemma}

\begin{proof}
Recall that the $R$-module structures on $\Hom_{\hat{R}} (M,\hat{R})$ and $\Hom_{\hat{R}} (R,\hat{R})$ are given by ``multiplication on the inside". From Lemma \ref{complex-def-computing-quotient-ring}, we know that, as graded $\hat{R}$-modules, $R\cong\hat{R}\{\qb{m+n}{n}\}$ and, as graded $R$-modules, $\Hom_{\hat{R}} (R,\hat{R}) \cong R\{q^{-2mn}\}$. So $\Hom_R (M, \Hom_{\hat{R}} (R,\hat{R})) \cong \Hom_R (M, R)\{q^{-2mn}\}$ is a graded matrix factorization over $R$ of potential $-w$.

Define $\alpha: \Hom_{\hat{R}} (M,\hat{R}) \rightarrow \Hom_R (M, \Hom_{\hat{R}} (R,\hat{R}))$ by $\alpha(f)(m)(r) = f(r\cdot m)$ $\forall ~f\in \Hom_{\hat{R}} (M,\hat{R}), ~m\in M, ~r\in R$. Define $\beta:\Hom_R (M, \Hom_{\hat{R}} (R,\hat{R})) \rightarrow \Hom_{\hat{R}} (M,\hat{R})$ by $\beta(g)(m)=g(m)(1)$ $\forall ~g\in \Hom_R (M, \Hom_{\hat{R}} (R,\hat{R})), ~m\in M$. It is straightforward to check that 
\begin{itemize}
	\item $\alpha$ and $\beta$ are $R$-module isomorphisms and are inverses of each other.
	\item $\alpha$ and $\beta$ preserve both the $\zed_2$-grading and the quantum grading.
\end{itemize}
This implies that $\Hom_{\hat{R}} (M,\hat{R})$ is a $\zed_2\oplus \zed$-graded-free $R$-module isomorphic to $\Hom_R (M, \Hom_{\hat{R}} (R,\hat{R})) \cong \Hom_R (M, R)\{q^{-2mn}\}$. The differential of $M$ induces on $\Hom_{\hat{R}} (M,\hat{R})$ an $R$-linear differential making it a graded matrix factorization over $R$ of potential $-w$.

To prove the lemma, it remains to check that $\alpha$ and $\beta$ commute with the differentials of $\Hom_{\hat{R}} (M,\hat{R})$ and $\Hom_R (M, \Hom_{\hat{R}} (R,\hat{R}))$. Since $\alpha$ and $\beta$ are inverses of each other, we only need to show that $\alpha$ commutes with the differentials. Recall that, if $f\in \Hom_{\hat{R}} (M,\hat{R})$ and $g\in \Hom_R (M, \Hom_{\hat{R}} (R,\hat{R}))$ have $\zed_2$-degree $\ve$, then $df = (-1)^{\ve+1} f \circ d_M$ and $dg = (-1)^{\ve+1} g \circ d_M$. So, for any $f\in \Hom_{\hat{R}} (M,\hat{R})$ with $\zed_2$-degree $\ve$ and $m \in M,~ r\in R$, we have 
\begin{eqnarray*}
\alpha(df)(m)(r) & = & (df)(r\cdot m) = (-1)^{\ve+1} f(d_M (r\cdot m)) = (-1)^{\ve+1} f(r\cdot d_M (m)) \\
& = & (-1)^{\ve+1} \alpha(f)(d_M (m)) (r) = d(\alpha(f)) (m) (r).
\end{eqnarray*}
This shows that $\alpha \circ d = d \circ \alpha$.
\end{proof}

\subsection{Computing $\Hom_\HMF (C(\Gamma_k^2),\ast)$} Let $\Gamma_k^2$ be the MOY graph in Figure \ref{l-right-k-high}. We mark it as in Figure \ref{l-right-k-high-marked}, where we omit the markings on the two horizontal edges since these are not explicitly used.

\begin{figure}[ht]

$
\xymatrix{
\input{square-m-n-l-right-k-high-marked} && \input{upper-gamma-2} && \input{lower-gamma-2}
}
$

\caption{}\label{l-right-k-high-marked}

\end{figure}

\begin{lemma}\label{complex-computing-gamma-lemma-1}
\[
C(\Gamma_k^2) \simeq \left(%
\begin{array}{ll}
 \ast & S_1 +Y_1-T_1-B_1 \\
  \dots & \dots \\
  \ast & \sum_{i=0}^j (S_iY_{j-i} - T_iB_{j-i}) \\
  \dots & \dots \\
  \ast & \sum_{i=0}^{n+l+k} (S_iY_{n+l+k-i} - T_iB_{n+l+k-i}) \\
  \ast & S_m \\
  \dots & \dots \\
  \ast & S_{k+1} \\
  \ast & T_n \\
  \dots & \dots \\
  \ast & T_{n-k+m+1}
\end{array}%
\right)_{\Sym(\mathbb{X}|\mathbb{Y}|\mathbb{A}|\mathbb{B}|\mathbb{D})} \{q^{-k(n+l)-(m-k)(n+k-m)}\},
\]
where 
\begin{eqnarray*}
S_j & = & \sum_{i=0}^j (-1)^i h_i(\mathbb{D}) X_{j-i}, \\
T_j & = & \sum_{i=0}^j (-1)^i h_i(\mathbb{D}) A_{j-i},
\end{eqnarray*}
and $X_j,~Y_j,~A_j,~B_j,~D_j,~E_j$ are the $j$-th elementary symmetric polynomials in the corresponding alphabets.
\end{lemma}

\begin{proof}
Cutting $\Gamma_k^2$ horizontally in the middle, we get the MOY graphs $\Gamma_{upper}$ and $\Gamma_{lower}$ in Figure \ref{l-right-k-high-marked}. Applying Lemma \ref{gen-chi-Gamma-0-lemma} to $\Gamma_{upper}$, we get 
\[
C(\Gamma_{upper}) \simeq \left(%
\begin{array}{ll}
 \ast & S_1 +Y_1-E_1 \\
  \dots & \dots \\
  \ast & (\sum_{i=0}^j S_iY_{j-i}) - E_j\\
  \dots & \dots \\
  \ast & (\sum_{i=0}^{n+l+k} S_iY_{n+l+k-i}) - E_{n+l+k} \\
  \ast & S_m \\
  \dots & \dots \\
  \ast & S_{k+1} \\
\end{array}%
\right)_{\Sym(\mathbb{X}|\mathbb{Y}|\mathbb{D}|\mathbb{E})} \{q^{-k(n+l)}\}.
\]
Applying Lemma \ref{gen-chi-Gamma-0-lemma} to $\Gamma_{lower}$, we get 
\[
C(\Gamma_{lower}) \simeq \left(%
\begin{array}{ll}
 \ast & E_1-T_1-B_1 \\
  \dots & \dots \\
  \ast & E_j - \sum_{i=0}^j  T_iB_{j-i} \\
  \dots & \dots \\
  \ast & E_{n+l+k} - \sum_{i=0}^{n+l+k} T_iB_{n+l+k-i} \\
  \ast & T_n \\
  \dots & \dots \\
  \ast & T_{n-k+m+1}
\end{array}%
\right)_{\Sym(\mathbb{A}|\mathbb{B}|\mathbb{D}|\mathbb{E})} \{q^{-(m-k)(n+k-m)}\},
\]
Thus,
\begin{eqnarray*}
C(\Gamma_k^2) & \simeq & C(\Gamma_{upper}) \otimes_{\Sym(\mathbb{D}|\mathbb{E})} C(\Gamma_{lower}) \\
& \simeq & \left(%
\begin{array}{ll}
 \ast & S_1 +Y_1-E_1 \\
  \dots & \dots \\
  \ast & (\sum_{i=0}^j S_iY_{j-i}) - E_j\\
  \dots & \dots \\
  \ast & (\sum_{i=0}^{n+l+k} S_iY_{n+l+k-i}) - E_{n+l+k} \\
  \ast & S_m \\
  \dots & \dots \\
  \ast & S_{k+1} \\
 \ast & E_1-T_1-B_1 \\
  \dots & \dots \\
  \ast & E_j - \sum_{i=0}^j  T_iB_{j-i} \\
  \dots & \dots \\
  \ast & E_{n+l+k} - \sum_{i=0}^{n+l+k} T_iB_{n+l+k-i} \\
  \ast & T_n \\
  \dots & \dots \\
  \ast & T_{n-k+m+1}
\end{array}%
\right)_{\Sym(\mathbb{X}|\mathbb{Y}|\mathbb{A}|\mathbb{B}|\mathbb{D}|\mathbb{E})} \{q^{-k(n+l)-(m-k)(n+k-m)}\},
\end{eqnarray*}
From here on, the lemma is obtained by excluding $E_1,\dots E_{n+l+k}$ from the base ring by applying Proposition \ref{b-contraction} to the rows 
\[
\left(%
\begin{array}{ll}
 \ast & E_1-T_1-B_1 \\
  \dots & \dots \\
  \ast & E_j - \sum_{i=0}^j  T_iB_{j-i} \\
  \dots & \dots \\
  \ast & E_{n+l+k} - \sum_{i=0}^{n+l+k} T_iB_{n+l+k-i} \\
\end{array}%
\right)
\] 
in the above Koszul matrix factorization.
\end{proof}

\begin{lemma}\label{complex-computing-gamma-HMF-lemma}
Let $\Gamma_k^2$ be the MOY graph in Figure \ref{l-right-k-high-marked}, and $\overline{\Gamma_k^2}$ the MOY graph obtained by reversing the orientation of $\Gamma_k^2$. Suppose that $M$ is a matrix factorization over $\hat{R}:=\Sym(\mathbb{X}|\mathbb{Y}|\mathbb{A}|\mathbb{B})$ with potential 
\[
w= p_{N+1}(\mathbb{X}) +p_{N+1}(\mathbb{Y}) -p_{N+1}(\mathbb{A}) -p_{N+1}(\mathbb{B}).
\]
Then, 
\[
\Hom_{\HMF,\hat{R}} (C(\Gamma_k^2), M) \cong H(M \otimes_{\hat{R}} C(\overline{\Gamma_k^2})) \left\langle m+n+l \right\rangle \{q^{(l+m+n)(N-l)-m^2-n^2}\},
\]
where $H(M \otimes_{\hat{R}} C(\overline{\Gamma_k^2}))$ is the usual homology of the chain complex $M \otimes_{\hat{R}} C(\overline{\Gamma_k^2})$.
\end{lemma}

\begin{proof}
Let $S_j$ and $T_j$ be as in Lemma \ref{complex-computing-gamma-lemma-1}. Define 
\[
R=\Sym(\mathbb{X}|\mathbb{Y}|\mathbb{A}|\mathbb{B}|\mathbb{D})/(S_{k+1},\dots,S_m).
\]
Let 
\begin{eqnarray*}
\mathcal{M} & = & \left(%
\begin{array}{ll}
 \ast & S_1 +Y_1-T_1-B_1 \\
  \dots & \dots \\
  \ast & \sum_{i=0}^j (S_iY_{j-i} - T_iB_{j-i}) \\
  \dots & \dots \\
  \ast & \sum_{i=0}^{n+l+k} (S_iY_{n+l+k-i} - T_iB_{n+l+k-i}) \\
  \ast & T_n \\
  \dots & \dots \\
  \ast & T_{n-k+m+1}
\end{array}%
\right)_R, \\
\overline{\mathcal{M}} & = & \left(%
\begin{array}{ll}
 \ast & -(S_1 +Y_1-T_1-B_1) \\
  \dots & \dots \\
  \ast & -\sum_{i=0}^j (S_iY_{j-i} - T_iB_{j-i}) \\
  \dots & \dots \\
  \ast & -\sum_{i=0}^{n+l+k} (S_iY_{n+l+k-i} - T_iB_{n+l+k-i}) \\
  \ast & -T_n \\
  \dots & \dots \\
  \ast & -T_{n-k+m+1}
\end{array}%
\right)_R.
\end{eqnarray*}
Then 
\[
\Hom_R (\mathcal{M},R) \cong \overline{\mathcal{M}} \left\langle m+n+l\right\rangle \{q^{(m+n+l)(N+1) - \sum_{i=1}^{n+l+k} 2i - \sum_{j=n-k+m+1}^n 2j} \}.
\]
By Lemma \ref{complex-computing-gamma-lemma-1} and Proposition \ref{b-contraction-dual}, we have
\begin{eqnarray*}
\Hom_{\HMF,\hat{R}}(C(\Gamma_k^2), M) & := & H (\Hom_{\hat{R}}(C(\Gamma_k^2), M)) \\
& = & H( \Hom_{\hat{R}} (\mathcal{M}\{q^{-k(n+l)-(m-k)(n+k-m)}\},M)) \\
& = & H( \Hom_{\hat{R}} (\mathcal{M},M)) \{q^{k(n+l)+(m-k)(n+k-m)}\}.
\end{eqnarray*}
Note that $\mathcal{M}$ is finitely generated over $R$ and over $\hat{R}$. By Lemma \ref{complex-def-general-ring-change}, we have
\begin{eqnarray*}
\Hom_{\hat{R}} (\mathcal{M},M) & \cong & M \otimes_{\hat{R}} \Hom_{\hat{R}} (\mathcal{M},\hat{R}) \\
& \cong & M \otimes_{\hat{R}} \Hom_{R} (\mathcal{M},R) \{q^{-2k(m-k)}\}.
\end{eqnarray*}
Altogether, we have
\begin{eqnarray*}
\Hom_{\HMF,\hat{R}}(C(\Gamma_k^2), M) & \cong & H( \Hom_{\hat{R}} (\mathcal{M},M)) \{q^{k(n+l)+(m-k)(n+k-m)}\} \\
& \cong & H( M \otimes_{\hat{R}} \Hom_{R} (\mathcal{M},R))  \{q^{k(n+l)+(m-k)(n+k-m)-2k(m-k)}\} \\
& \cong & H( M \otimes_{\hat{R}} \overline{\mathcal{M}}) \left\langle m+n+l\right\rangle \{ q^{\varsigma}\},
\end{eqnarray*}
where
\[
\varsigma =   k(n+l)+(m-k)(n+k-m)-2k(m-k) + (m+n+l)(N+1) - \sum_{i=1}^{n+l+k} 2i - \sum_{j=n-k+m+1}^n 2j.
\]
On the other hand, 
\[
H(M \otimes_{\hat{R}} C(\overline{\Gamma_k^2})) \cong H(M \otimes_{\hat{R}} \overline{\mathcal{M}}) \{q^{-k(m-k)-(m+l)(n+k-m)} \}.
\]
So
\[
\Hom_{\HMF,\hat{R}}(C(\Gamma_k^2), M) \cong H(M \otimes_{\hat{R}} C(\overline{\Gamma_k^2})) \left\langle m+n+l\right\rangle \{ q^{\varsigma+k(m-k)+(m+l)(n+k+m)}\}.
\]
One can check that 
\[
\varsigma+k(m-k)+(m+l)(n+k+m) = (l+m+n)(N-l)-m^2-n^2.
\]
This proves the lemma.
\end{proof}

\subsection{The chain complex associated to a colored crossing} Let $c_{m,n}^+$ and $c_{m,n}^-$ be the colored crossings with marked end points in Figure \ref{colored-crossing-sign-figure}. In this subsection, we define the chain complexes associated to them, which completes the definition of chain complexes associated to knotted MOY graphs.

\begin{figure}[h]

\setlength{\unitlength}{1pt}

\begin{picture}(360,50)(-180,-30)

\linethickness{.5pt}


\put(-100,-20){\vector(1,1){40}}

\put(-60,-20){\line(-1,1){15}}

\put(-85,5){\vector(-1,1){15}}

\put(-84,-30){$c_{m,n}^+$}

\put(-92,16){\tiny{$_m$}}

\put(-70,16){\tiny{$_n$}}

\put(-105,-20){\tiny{$\mathbb{A}$}}
\put(-105,15){\tiny{$\mathbb{X}$}}

\put(-55,-20){\tiny{$\mathbb{B}$}}
\put(-55,15){\tiny{$\mathbb{Y}$}}


\put(100,-20){\vector(-1,1){40}}

\put(60,-20){\line(1,1){15}}

\put(85,5){\vector(1,1){15}}

\put(76,-30){$c_{m,n}^-$}

\put(68,16){\tiny{$_m$}}

\put(90,16){\tiny{$_n$}}

\put(55,-20){\tiny{$\mathbb{A}$}}
\put(55,15){\tiny{$\mathbb{X}$}}

\put(105,-20){\tiny{$\mathbb{B}$}}
\put(105,15){\tiny{$\mathbb{Y}$}}

\end{picture}

\caption{}\label{colored-crossing-sign-figure}

\end{figure}

For $\max\{m-n,0\} \leq k \leq m$, we call $\Gamma_k^L$ and $\Gamma_k^R$ in Figure \ref{decomp-V-special-1-figure} the $k$-th left and right resolutions of $c_{m,n}^\pm$. The following lemma is a special case of Decomposition (V) (Theorem \ref{decomp-V}.)

\begin{figure}[ht]
$
\xymatrix{
\input{square-m-n-k-left} && \input{square-m-n-k-right}
}
$
\caption{}\label{decomp-V-special-1-figure}

\end{figure}

\begin{lemma}\label{decomp-V-special-1} 
Let $m,n$ be integers such that $0\leq m,n \leq N$. For $\max\{m-n,0\} \leq k \leq m$, define $\Gamma_k^L$ and $\Gamma_k^R$ to be the MOY graphs in Figure \ref{decomp-V-special-1-figure}. Then $C(\Gamma_k^L) \simeq C(\Gamma_k^R)$.
\end{lemma}

\begin{lemma}\label{colored-crossing-res-HMF}
Let $m,n$ be integers such that $0\leq m,n \leq N$. For $\max\{m-n,0\} \leq j,k \leq m$, 
\begin{eqnarray*}
& & \Hom_\HMF (C(\Gamma_j^L), C(\Gamma_k^L)) \cong \Hom_\HMF (C(\Gamma_j^R), C(\Gamma_k^L)) \\
& \cong & \Hom_\HMF (C(\Gamma_j^L), C(\Gamma_k^R))  \cong \Hom_\HMF (C(\Gamma_j^R), C(\Gamma_k^R)) \\
& \cong & C(\emptyset) \{\qb{_{n+j+k-m}}{_k}\qb{_{n+j+k-m}}{_j} \qb{_{N+m-n-j-k}}{_{m-k}} \qb{_{N+m-n-j-k}}{_{m-j}} \qb{_N}{_{n+j+k-m}} q^{(m+n)N -n^2-m^2}\}.
\end{eqnarray*}

In particular, 
\begin{itemize}
	\item the lowest non-vanishing quantum grading of these spaces are all $(k-j)^2$,
	\item the subspaces of homogeneous elements of quantum degree $(k-j)^2$ of these spaces are $1$-dimensional and have $\zed_2$ grading $0$.
\end{itemize}
\end{lemma}

\begin{proof}
By Lemma \ref{decomp-V-special-1}, the above four $\Hom_\HMF$ spaces are isomorphic. So, to prove the lemma, we only need to compute one of these, say $\Hom_\HMF (C(\Gamma_j^R), C(\Gamma_k^L))$.

Let $\hat{R}=\Sym(\mathbb{X}|\mathbb{Y}|\mathbb{A}|\mathbb{B})$. By Lemma \ref{complex-computing-gamma-HMF-lemma},
\[
\Hom_\HMF (C(\Gamma_j^R), C(\Gamma_k^L)) \cong H(C(\Gamma_k^L) \otimes_{\hat{R}} C(\overline{\Gamma_j^R})) \left\langle m+n \right\rangle \{q^{(m+n)N-n^2-m^2} \},
\]
where $\overline{\Gamma_j^R}$ is $\Gamma_j^R$ with orientation reversed.

\begin{figure}[ht]
$
\xymatrix{
\input{colored-crossing-res-HMF-figure-1} & \input{colored-crossing-res-HMF-figure-2} & \input{colored-crossing-res-HMF-figure-3}
}
$
\caption{}\label{colored-crossing-res-HMF-figure}

\end{figure}

Let $\Gamma$, $\Gamma'$ and $\Gamma''$ be the MOY graphs in Figure \ref{colored-crossing-res-HMF-figure}. Then, by Corollary \ref{contract-expand} and Decompositions (I-II) (Theorems \ref{decomp-I} and \ref{decomp-II}), we have
\begin{eqnarray*}
& & C(\Gamma_k^L) \otimes_{\hat{R}} C(\overline{\Gamma_j^R}) = C(\Gamma) \\
& \simeq & C(\Gamma') \{\qb{_{n+j+k-m}}{_k}\qb{_{n+j+k-m}}{_j}\} \\
& \simeq & C(\Gamma'') \left\langle j+k \right\rangle \{\qb{_{n+j+k-m}}{_k}\qb{_{n+j+k-m}}{_j} \qb{_{N+m-n-j-k}}{_{m-k}} \qb{_{N+m-n-j-k}}{_{m-j}}\} \\
& \simeq & C(\emptyset)\left\langle m+n \right\rangle \{\qb{_{n+j+k-m}}{_k}\qb{_{n+j+k-m}}{_j} \qb{_{N+m-n-j-k}}{_{m-k}} \qb{_{N+m-n-j-k}}{_{m-j}} \qb{_N}{_{n+j+k-m}}\}.
\end{eqnarray*}
This shows that
\begin{eqnarray*}
& & \Hom_\HMF (C(\Gamma_j^R), C(\Gamma_k^L)) \\
& \cong & C(\emptyset) \{\qb{_{n+j+k-m}}{_k}\qb{_{n+j+k-m}}{_j} \qb{_{N+m-n-j-k}}{_{m-k}} \qb{_{N+m-n-j-k}}{_{m-j}} \qb{_N}{_{n+j+k-m}} q^{(m+n)N -n^2-m^2}\}.
\end{eqnarray*}
The rest of the lemma follows from the above isomorphism.
\end{proof}

\begin{corollary}\label{left-right-naturally-homotopic}
Let $m,n$ be integers such that $0\leq m,n \leq N$. For $\max\{m-n,0\} \leq k \leq m$, the matrix factorizations $C(\Gamma_k^L)$ and $C(\Gamma_k^R)$ are naturally homotopic in the sense that the homotopy equivalences $C(\Gamma_k^L) \xrightarrow{\simeq} C(\Gamma_k^R)$ and $C(\Gamma_k^R) \xrightarrow{\simeq} C(\Gamma_k^L)$ are unique up to homotopy and scaling.
\end{corollary}
\begin{proof}
The existence of the homotopy equivalences follows from Lemma \ref{decomp-V-special-1}. The uniqueness follows from the $j=k$ case of Lemma \ref{colored-crossing-res-HMF}.
\end{proof}

\begin{corollary}\label{left-right-naturally-homotopic-2}
Let $m,n$ be integers such that $0\leq m,n \leq N$. For $\max\{m-n,0\} \leq j,k \leq m$, up to homotopy and scaling, there exist unique homogeneous morphisms
\begin{eqnarray*}
d_{j,k}^{LL} & : & C(\Gamma_j^L) \rightarrow C(\Gamma_k^L), \\
d_{j,k}^{RL} & : & C(\Gamma_j^R) \rightarrow C(\Gamma_k^L), \\
d_{j,k}^{LR} & : & C(\Gamma_j^L) \rightarrow C(\Gamma_k^R), \\
d_{j,k}^{RR} & : & C(\Gamma_j^R) \rightarrow C(\Gamma_k^R) \\
\end{eqnarray*}
satisfying
\begin{itemize}
	\item $d_{j,k}^{LL}$, $d_{j,k}^{RL}$, $d_{j,k}^{LR}$ and $d_{j,k}^{RR}$ have quantum degree $(j-k)^2$ and $\zed_2$-degree $0$,
	\item $d_{j,k}^{LL}$, $d_{j,k}^{RL}$, $d_{j,k}^{LR}$ and $d_{j,k}^{RR}$ are homotopically non-trivial.
\end{itemize}
Moreover, up to homotopy and scaling, every square in the following diagram commutes, where the vertical morphisms are either identity or the natural homotopy equivalences from Corollary \ref{left-right-naturally-homotopic}.
\[
\xymatrix{
C(\Gamma_j^L) \ar[rr]^{d_{j,k}^{LL}}  \ar[d]^{\simeq} && C(\Gamma_k^L) \ar[d]^{=} \\
C(\Gamma_j^R) \ar[rr]^{d_{j,k}^{RL}}  \ar[d]^{\simeq} && C(\Gamma_k^L) \ar[d]^{\simeq} \\
C(\Gamma_j^L) \ar[rr]^{d_{j,k}^{LR}}  \ar[d]^{\simeq} && C(\Gamma_k^R) \ar[d]^{=} \\
C(\Gamma_j^R) \ar[rr]^{d_{j,k}^{RR}}   && C(\Gamma_k^R)
}
\]
\end{corollary}
\begin{proof}
This corollary follows easily from Lemma \ref{colored-crossing-res-HMF}.
\end{proof}

From Corollary \ref{left-right-naturally-homotopic-2}, we know that, up to homotopy and scaling, the morphisms $d_{j,k}^{LL}$, $d_{j,k}^{RL}$, $d_{j,k}^{LR}$ and $d_{j,k}^{RR}$ are identified with each other under the natural homotopy equivalences $C(\Gamma_j^L) \simeq C(\Gamma_j^R)$ and $C(\Gamma_k^L) \simeq C(\Gamma_k^R)$. So, without creating any confusion, we drop the superscripts in the notations and simple denote these morphisms by $d_{j,k}$.

\begin{definition}\label{complex-colored-crossing-chain-maps-def}
Let $m,n$ be integers such that $0\leq m,n \leq N$. For $\max\{m-n,0\}+1 \leq k \leq m$, define $d_k^+ = d_{k,k-1}$. For $\max\{m-n,0\} \leq k \leq m-1$, define $d_k^- = d_{k,k+1}$. Note that these are homogeneous morphisms of quantum degree $1$ and $\zed_2$-degree $0$.
\end{definition}

\begin{theorem}\label{complex-colored-crossing-well-defined}
Let $m,n$ be integers such that $0\leq m,n \leq N$. 
\begin{itemize}
	\item For $\max\{m-n,0\}+2 \leq k \leq m$, $d_{k-1}^+ \circ d_k^+ \simeq 0$.
	\item For $\max\{m-n,0\} \leq k \leq m-2$, $d_{k+1}^- \circ d_k^- \simeq 0$.
\end{itemize}
\end{theorem}
\begin{proof}
For $\max\{m-n,0\}+2 \leq k \leq m$, $d_{k-1}^+ \circ d_k^+:C(\Gamma_k^L) \rightarrow C(\Gamma_{k-2}^L)$ is a homogeneous morphism of quantum degree $2$. But, by Lemma \ref{colored-crossing-res-HMF}, the lowest non-vanishing quantum grading of $\Hom_\HMF(C(\Gamma_k^L), C(\Gamma_{k-2}^L))$ is $2^2=4$. This implies that $d_{k-1}^+ \circ d_k^+ \simeq 0$. The proof of $d_{k+1}^- \circ d_k^- \simeq 0$ is very similar and left to the reader.
\end{proof}

\begin{definition}\label{complex-colored-crossing-def}
Let $c^\pm_{m,n}$ be the colored crossings in Figure \ref{colored-crossing-sign-figure}, $\hat{R}=\Sym(\mathbb{X}|\mathbb{Y}|\mathbb{A}|\mathbb{B})$ and 
\[
w= p_{N+1}(\mathbb{X}) +p_{N+1}(\mathbb{Y}) -p_{N+1}(\mathbb{A}) - p_{N+1}(\mathbb{B}).
\] We first define the unnormalized chain complexes $\hat{C}(c^\pm_{m,n})$. 

If $m \leq n$, then $\hat{C}(c^+_{m,n})$ is defined to be the object
\[
0\rightarrow C(\Gamma^L_m) \xrightarrow{d^{+}_m} C(\Gamma^L_{m-1})\{q^{-1}\} \xrightarrow{d^{+}_{m-1}} \cdots \xrightarrow{d^{+}_1} C(\Gamma^L_0)\{q^{-m}\} \rightarrow 0
\]
of $\hch(\hmf_{\hat{R},w})$, where the homological grading on $\hat{C}(c^+_{m,n})$ is defined so that the term $C(\Gamma^L_{k})\{q^{-(m-k)}\}$ have homological grading $m-k$.

If $m>n$ , then $\hat{C}(c^+_{m,n})$ is defined to be the object
\[
0\rightarrow C(\Gamma^L_m) \xrightarrow{d^{+}_m} C(\Gamma^L_{m-1})\{q^{-1}\} \xrightarrow{d^{+}_{m-1}} \cdots \xrightarrow{d^{+}_{m-n+1}} C(\Gamma^L_{m-n})\{q^{-n}\} \rightarrow 0
\]
of $\hch(\hmf_{\hat{R},w})$, where the homological grading on $\hat{C}(c^+_{m,n})$ is defined so that the term $C(\Gamma^L_{k})\{q^{-(m-k)}\}$ has homological grading $m-k$.

If $m \leq n$, then $\hat{C}(c^-_{m,n})$ is defined to be the object
\[
0\rightarrow C(\Gamma^L_0)\{q^{m}\} \xrightarrow{d^{-}_0} \cdots \xrightarrow{d^{-}_{m-2}} C(\Gamma^L_{m-1}) \{ q \} \xrightarrow{d^{-}_{m-1}} C(\Gamma^L_m) \rightarrow 0
\]
of $\hch(\hmf_{\hat{R},w})$, where the homological grading on $\hat{C}(c^-_{m,n})$ is defined so that the term $C(\Gamma^L_{k})\{q^{m-k}\}$ has homological grading $k-m$.

If $m>n$ , then $\hat{C}(c^-_{m,n})$ is defined to be the object
\[
0\rightarrow C(\Gamma_{m-n})\{q^{n}\} \xrightarrow{d^{-}_{m-n}} \cdots \xrightarrow{d^{-}_{m-2}} C(\Gamma^L_{m-1})\{q\} \xrightarrow{d^{-}_{m-1}} C(\Gamma^L_{m}) \rightarrow 0
\]
of $\hch(\hmf_{\hat{R},w})$, where the homological grading on $\hat{C}(c^-_{m,n})$ is defined so that the term $C(\Gamma^L_{k})\{q^{m-k}\}$ has homological grading $k-m$.

The normalized chain complex $C(c^\pm_{m,n})$ is defined to be
\begin{eqnarray*}
C(c^+_{m,n}) & = & \begin{cases}
\hat{C}(c^+_{m,m})\left\langle m \right\rangle\|-m\| \{q^{m(N+1-m)}\} & \text{if } m=n, \\
\hat{C}(c^+_{m,n}) & \text{if } m\neq n,
\end{cases} \\
C(c^-_{m,n}) & = & \begin{cases}
\hat{C}(c^-_{m,m})\left\langle m \right\rangle\|m\| \{q^{-m(N+1-m)}\} & \text{if } m=n, \\
\hat{C}(c^-_{m,n}) & \text{if } m\neq n.
\end{cases}
\end{eqnarray*}
(Recall that $\|m\|$ means shifting the homological grading by $m$. See Definition \ref{categories-of-complexes}.)
\end{definition}

\begin{corollary}\label{complex-colored-crossing-def-l-r-res-changeable}
Replacing the left resolutions $\Gamma_k^L$ in Definition \ref{complex-colored-crossing-def} by the right resolutions $\Gamma_k^R$ does not change the isomorphism types of $\hat{C}(c^\pm_{m,n})$ and $C(c^\pm_{m,n})$ as objects of $\ch(\hmf_{\hat{R},w})$.
\end{corollary}
\begin{proof}
This is an easy consequence of Lemma \ref{decomp-V-special-1} and Corollaries \ref{left-right-naturally-homotopic} and \ref{left-right-naturally-homotopic-2}.
\end{proof}

\begin{corollary}\label{complex-knotted-MOY-marking-independence}
The isomorphism type of the chain complexes $\hat{C}(D)$ and $C(D)$ associated to a knotted MOY graph $D$ (see Definition \ref{complex-knotted-MOY-def}) is independent of the choice of the marking of $D$.
\end{corollary}
\begin{proof}
We only need to show that adding or removing an extra marked point on a segment of $D$ does not change the isomorphism type. Note that adding or removing such an extra marked point is equivalent to adding or removing an internal marked point in a piece $D_j$ of $D$. (See Definition \ref{knotted-MOY-marking-def}.)

If $D_j$ is of type (i) or (ii), that is, an (embedded) MOY graph, then, by Lemma \ref{marking-independence}, adding or removing an internal marked point does not change the homotopy type of the matrix factorization of this piece. Moreover, it is easy to see that the differential map of this piece is $0$ with or without the extra internal marked point. So, in this case, the addition or removal of the extra marked point does not change the isomorphism type of $C(D)$.

If $D_j$ is of type (iii), that is, a colored crossing, then, by Lemma \ref{marking-independence}, adding or removing an internal marked point does not change the homotopy types of the matrix factorizations associated to the resolutions of this colored crossing. Moreover, by the uniqueness part of Corollary \ref{left-right-naturally-homotopic-2}, up to homotopy and scaling, the differential map is the same with or without the extra internal marked point. So, again, the addition or removal of the extra marked point does not change the isomorphism type of $C(D)$.
\end{proof}

\subsection{A null-homotopic chain complex} In this subsection, we construct a null-homotopic chain complex that will be useful in Section \ref{sec-inv-fork} below. The construction of this chain complex is similar to the chain complex of a colored crossing. 

The following lemma is a special case of Decomposition (V) (Theorem \ref{decomp-V}.)

\begin{figure}[ht]
$
\xymatrix{
\input{square-m-n-1-right-k-low} && \input{square-m-n-1-right-j-high}
}
$
\caption{}\label{decomp-V-special-2-figure}

\end{figure}

\begin{lemma}\label{decomp-V-special-2} 
Let $m,n$ be integers such that $0\leq m,n \leq N-1$. For $\max\{m-n,0\} \leq k \leq m+1$ and $\max\{m-n,0\} \leq j \leq m$, define $\Gamma_k$ and $\Gamma_j'$ to be the MOY graphs in Figure \ref{decomp-V-special-2-figure}. Then, for $\max\{m-n,0\} \leq k \leq m+1$, 
\[
C(\Gamma_k) \simeq \begin{cases}
C(\Gamma_m') & \text{if } k=m+1, \\
C(\Gamma_k') \oplus C(\Gamma_{k-1}') & \text{if } \max\{m-n,0\}+1 \leq k \leq m, \\
C(\Gamma_{\max\{m-n,0\}}') & \text{if } k = \max\{m-n,0\}.
\end{cases}
\]
\end{lemma}

\begin{lemma}\label{trivial-complex-lemma-1}
Let $\Gamma_k$ and $\Gamma_j'$ be as in Lemma \ref{decomp-V-special-2}. Then
\begin{eqnarray*}
&& \Hom_\HMF (C(\Gamma_j'), C(\Gamma_k)) \\
& \cong & C(\emptyset) \{\qb{_{n+k+j-m}}{_k} \qb{_{n+k+j-m}}{_j} \qb{_{N+m-n-k-j}}{_{m-j}} \qb{_{N+m-n-k-j}}{_{m+1-k}} \qb{_N}{_{n+k+j-m}} q^{(m+n+1)(N-1)-m^2-n^2} \}.
\end{eqnarray*}
In particular, 
\begin{itemize}
	\item $\Hom_\HMF (C(\Gamma_j'), C(\Gamma_k))$ is supported on $\zed_2$-degree $0$,
	\item the lowest non-vanishing quantum grading of $\Hom_\HMF (C(\Gamma_j'), C(\Gamma_k))$ is $(j-k)(j-k+1)$,
	\item the subspace of homogeneous elements of $\Hom_\HMF (C(\Gamma_j'), C(\Gamma_k))$ of quantum degree $(j-k)(j-k+1)$ is $1$-dimensional.
\end{itemize}
\end{lemma}

\begin{proof}
By Lemma \ref{complex-computing-gamma-HMF-lemma}, we have
\[
\Hom_\HMF (C(\Gamma_j'), C(\Gamma_k)) \cong H(C(\Gamma_k) \otimes_{\hat{R}} C(\overline{\Gamma_j'}))\left\langle m+n+1 \right\rangle \{q^{(m+n+1)(N-1)-m^2-n^2} \},
\]
where $\hat{R}=\Sym(\mathbb{X}|\mathbb{Y}|\mathbb{A}|\mathbb{B})$ and $\overline{\Gamma_j'}$ is $\Gamma_j'$ with its orientation reversed.

\begin{figure}[ht]
$
\xymatrix{
\input{trivial-complex-lemma-1-figure1} & \input{trivial-complex-lemma-1-figure2} & \input{colored-crossing-res-HMF-figure-3}
}
$
\caption{}\label{trivial-complex-lemma-1-figure}

\end{figure}

Let $\Gamma$, $\Gamma'$ and $\Gamma''$ be the MOY graphs in Figure \ref{trivial-complex-lemma-1-figure}. Then, by Corollary \ref{contract-expand} and Decompositions (I-II) (Theorems \ref{decomp-I} and \ref{decomp-II}), we have 
\begin{eqnarray*}
&& C(\Gamma_k) \otimes_{\hat{R}} C(\overline{\Gamma_j'}) = C(\Gamma) \\
& \simeq & C(\Gamma') \{\qb{_{n+k+j-m}}{_k} \qb{_{n+k+j-m}}{_j} \} \\
& \simeq & C(\Gamma'') \left\langle j+k+1 \right\rangle\{\qb{_{n+k+j-m}}{_k} \qb{_{n+k+j-m}}{_j} \qb{_{N+m-n-k-j}}{_{m-j}} \qb{_{N+m-n-k-j}}{_{m+1-k}}\} \\
& \simeq & C(\emptyset) \left\langle m+n+1 \right\rangle \{\qb{_{n+k+j-m}}{_k} \qb{_{n+k+j-m}}{_j} \qb{_{N+m-n-k-j}}{_{m-j}} \qb{_{N+m-n-k-j}}{_{m+1-k}} \qb{_N}{_{n+k+j-m}}\}.
\end{eqnarray*}
Thus,
\begin{eqnarray*}
&& \Hom_\HMF (C(\Gamma_j'), C(\Gamma_k)) \\
& \cong & C(\emptyset) \{\qb{_{n+k+j-m}}{_k} \qb{_{n+k+j-m}}{_j} \qb{_{N+m-n-k-j}}{_{m-j}} \qb{_{N+m-n-k-j}}{_{m+1-k}} \qb{_N}{_{n+k+j-m}} q^{(m+n+1)(N-1)-m^2-n^2} \}.
\end{eqnarray*}
The rest of the lemma follows from this isomorphism.
\end{proof}

\begin{lemma}\label{trivial-complex-lemma-2}
For $\max\{m-n,0\} \leq i,j \leq m$,
\[
\Hom_\hmf( C(\Gamma_i'),C(\Gamma_j')) \cong \begin{cases}
\C & \text{if } i=j, \\
0 & \text{if } i \neq j.
\end{cases}
\]
In the case $i=j$, $\Hom_\hmf( C(\Gamma_i'),C(\Gamma_i'))$ is spanned by $\id_{C(\Gamma_i')}$.
\end{lemma}

\begin{proof}
If $i>j$, then by Lemma \ref{trivial-complex-lemma-1}, $\Hom_\hmf( C(\Gamma_i'),C(\Gamma_j)) =0$. But, by Lemma \ref{decomp-V-special-2}, $C(\Gamma_j)=C(\Gamma_j')\oplus C(\Gamma_{j-1}')$. This implies that $\Hom_\hmf( C(\Gamma_i'),C(\Gamma_j')) \cong 0$. 

If $i<j$, then by Lemma \ref{trivial-complex-lemma-1}, $\Hom_\hmf( C(\Gamma_i'),C(\Gamma_{j+1})) =0$. But, by Lemma \ref{decomp-V-special-2}, $C(\Gamma_{j+1})=C(\Gamma_{j+1}')\oplus C(\Gamma_{j}')$. This implies that $\Hom_\hmf( C(\Gamma_i'),C(\Gamma_j')) \cong 0$.

If $i=j$, then by Lemma \ref{trivial-complex-lemma-1}, $\Hom_\hmf( C(\Gamma_i'),C(\Gamma_{i+1})) \cong \C$. But, by Lemma \ref{decomp-V-special-2}, $C(\Gamma_{i+1})=C(\Gamma_{i+1}')\oplus C(\Gamma_{i}')$ and, from above, $\Hom_\hmf( C(\Gamma_i'),C(\Gamma_{i+1}')) =0$. This implies that $\Hom_\hmf( C(\Gamma_i'),C(\Gamma_{i}')) \cong \C$. It follows that $C(\Gamma_{i}')$ is not null-homotopic and, therefore, $\id_{C(\Gamma_i')}$ is not null-homotopic. So $\id_{C(\Gamma_i')}$ spans the $1$-dimensional space $\Hom_\hmf( C(\Gamma_i'),C(\Gamma_{i}'))$.

\end{proof}

\begin{lemma}\label{trivial-complex-lemma-3}
For $\max\{m-n,0\} \leq j,k \leq m+1$,
\[
\Hom_\hmf( C(\Gamma_j),C(\Gamma_k)) \cong \begin{cases}
\C \oplus \C & \text{if } \max\{m-n,0\}+1 \leq j=k \leq m, \\
\C & \text{if } j=k=\max\{m-n,0\} \text{ or } m+1, \\
\C & \text{if } |j-k|=1,\\
0 & \text{if } |j-k|>1.
\end{cases}
\]
\end{lemma}
\begin{proof}
This follows easily from Lemmas \ref{decomp-V-special-2} and \ref{trivial-complex-lemma-2}.
\end{proof}

\begin{definition}\label{trivial-complex-differential-def}
Denote by 
\begin{eqnarray*}
J_{k,k} & : & C(\Gamma_k') \rightarrow C(\Gamma_k) \\
J_{k,k-1} & : & C(\Gamma_{k-1}') \rightarrow C(\Gamma_k) \\
P_{k,k} & : & C(\Gamma_k) \rightarrow C(\Gamma_k') \\
P_{k,k-1} & : & C(\Gamma_k) \rightarrow C(\Gamma_{k-1}') 
\end{eqnarray*}
the inclusion and projection morphisms in the decomposition 
\[
C(\Gamma_k) \simeq C(\Gamma_k') \oplus C(\Gamma_{k-1}').
\]
Define 
\begin{eqnarray*}
\delta_k^+ & = & J_{k-1,k-1}\circ P_{k,k-1} : C(\Gamma_k) \rightarrow C(\Gamma_{k-1}), \\
\delta_k^- & = & J_{k+1,k} \circ P_{k,k} : C(\Gamma_k) \rightarrow C(\Gamma_{k+1}).
\end{eqnarray*}
Then $\delta_k^+$ and $\delta_k^-$ are both homotopically non-trivial homogeneous morphisms preserving both the $\zed_2$-grading and the quantum grading. 
By Lemma \ref{trivial-complex-lemma-3}, up to homotopy and scaling, $\delta_k^+$ and $\delta_k^-$ are the unique morphisms with such properties.
\end{definition}

\begin{lemma}\label{trivial-complex-lemma-4}
$\delta_{k-1}^+\circ\delta_k^+ \simeq 0$, $\delta_{k+1}^-\circ\delta_k^- \simeq 0$.
\end{lemma}

\begin{proof}
From Lemma \ref{trivial-complex-lemma-3}, we have that 
\[
\Hom_\hmf( C(\Gamma_k),C(\Gamma_{k-2})) \cong \Hom_\hmf( C(\Gamma_k),C(\Gamma_{k+2})) \cong0.
\] 
The lemma follows from this.
\end{proof}

Let $\hat{R}=\Sym(\mathbb{X}|\mathbb{Y}|\mathbb{A}|\mathbb{B})$ and $w= p_{N+1}(\mathbb{X}) +p_{N+1}(\mathbb{Y}) -p_{N+1}(\mathbb{A}) - p_{N+1}(\mathbb{B})$. The above discussion implies the following.

\begin{proposition}\label{trivial-complex-prop}
Let $k_1$ and $k_2$ be integers such that $\max\{m-n,0\}+1 \leq k_1\leq k_2 \leq m$. Then 
\[
\xymatrix{
0 \ar[r] & C(\Gamma_{k_2}')\ar[r]^{J_{k_2,k_2}} & C(\Gamma_{k_2}) \ar[r]^{\delta_{k_2}^+} & \cdots \ar[r]^{\delta_{k_1+1}^+} & C(\Gamma_{k_1}) \ar[r]^{P_{k_1,k_1-1}} & C(\Gamma_{k_1-1}') \ar[r] & 0, \\
0 \ar[r] & C(\Gamma_{k_1-1}')\ar[r]^{J_{k_1,k_1-1}} & C(\Gamma_{k_1}) \ar[r]^{\delta_{k_1}^-} & \cdots \ar[r]^{\delta_{k_2-1}^-} & C(\Gamma_{k_2}) \ar[r]^{P_{k_2,k_2}} & C(\Gamma_{k_2}') \ar[r] & 0
}
\]
are both chain complexes over $\hmf_{\hat{R},w}$ and are isomorphic in $\ch(\hmf_{\hat{R},w})$ to
\[
\bigoplus_{j=k_1-1}^{k_2} (\xymatrix{
0 \ar[r] & C(\Gamma_j') \ar[r]^{\simeq} & C(\Gamma_j') \ar[r] & 0,
})
\]
which is homotopic to $0$. (That is, isomorphic in $\hch(\hmf_{\hat{R},w})$ to $0$.)
\end{proposition}

\subsection{Explicit forms of the differential maps} In the proof of the invariance of the $\mathfrak{sl}(N)$ homology, we need to use explicit forms of the differential maps in the chain complexes defined in the previous two subsections. In this subsection, we give one construction of such explicit forms. (There are more than one explicit constructions of the same differential maps. See for example \cite[Figure 17]{Mackaay-Stosic-Vaz2}.)

\begin{figure}[ht]
$
\xymatrix{
\input{square-m-n-l-right-k-low} \ar@<11ex>[rr]^{d_k^+} \ar@<1ex>[d]^<<<<<{\phi_{k,1}} & & \input{square-m-n-l-right-k-1-low} \ar@<-9ex>[ll]^{d_{k-1}^-} \ar@<1ex>[d]^<<<<<{\phi_{k,2}} \\
\input{square-m-n-l-right-k-low-bubble} \ar@<8ex>[r]^{\chi^1\otimes\chi^1} \ar@<1ex>[u]^>>>>>{\overline{\phi}_{k,1}} & \input{double-square-m-n-l-right-k-low} \ar@<8ex>[r]^{h_k} \ar@<-6ex>[l]^{\chi^0\otimes\chi^0} & \input{square-m-n-l-right-k-1-low-bubble} \ar@<-6ex>[l]^{\overline{h}_k} \ar@<1ex>[u]^>>>>>{\overline{\phi}_{k,2}}
}
$
\caption{}\label{explicit-differential-general-figure1}

\end{figure}

Consider the MOY graphs and morphisms in Figure \ref{explicit-differential-general-figure1}. 
\begin{itemize}
	\item $\phi_{k,1}, ~ \overline{\phi}_{k,1}, ~\phi_{k,2}, ~ \overline{\phi}_{k,2}$ are the morphisms associated to the apparent edge splittings and mergings. (See Definition \ref{morphism-edge-splitting-merging-def}.)
	\item $\chi^0$ and $\chi^1$ are the morphisms from Proposition \ref{general-general-chi-maps} (more precisely, Corollary \ref{general-chi-maps-def}.)
	\item $h_k, ~\overline{h}_k$ are the morphisms induced by the bouquet moves. (See Corollary \ref{contract-expand}, Lemma \ref{bouquet-move-lemma} and Remark \ref{bouquet-move-remark}.)
\end{itemize}
We define $d_k^+$ and $d_{k-1}^-$ to be
\begin{eqnarray*}
d_k^+ & = & \overline{\phi}_{k,2} \circ h_k \circ (\chi^1\otimes\chi^1) \circ \phi_{k,1}. \\
d_{k-1}^- & = & \overline{\phi}_{k,1} \circ (\chi^0\otimes\chi^0) \circ \overline{h}_k \circ \phi_{k,2}.
\end{eqnarray*}

\begin{theorem}\label{explicit-differential-general}
$d_k^+$ and $d_{k-1}^-$ are homotopically non-trivial homogeneous morphisms of $\zed_2$-degree $0$ and quantum degree $1-l$. 

When $l=0$, $d_k^+$ and $d_{k-1}^-$ are explicit forms of the differential maps of the chain complexes associated to colored crossings defined in Definition \ref{complex-colored-crossing-def}. 

When $l=1$, $d_k^+$ and $d_{k-1}^-$ are explicit forms of the differential maps $\delta_k^+$ and $\delta_{k-1}^-$  of the null-homotopic chain complexes in Proposition \ref{trivial-complex-prop}.
\end{theorem}

\begin{figure}[ht]
$
\xymatrix{
\input{v-vector-n+k} \ar@<12ex>[r]^{\phi_3'} \ar@<1ex>[d]^<<<<<{\phi_1 \otimes \phi_2} & \input{v-vector-n+k-bubble} \ar@<12ex>[r]^{\phi_1' \otimes \phi_2'} \ar@<-10ex>[l]^{\overline{\phi_3'}} & \input{v-vector-n+k-bubbles-in-bubble} \ar@<-10ex>[l]^{\overline{\phi_1'} \otimes \overline{\phi_2'} } \ar@<1ex>[d]^<<<<<{\overline{h}}\\
\input{v-vector-n+k-2-bubbles} \ar@<1ex>[u]^>>>>>{\overline{\phi}_1 \otimes \overline{\phi}_2} \ar@<10ex>[r]^{\phi_3} & \input{v-vector-n+k-3-bubbles} \ar@<-8ex>[l]^{\overline{\phi}_3} \ar@<10ex>[r]^{\chi^1 \otimes \chi^1} & \input{v-vector-n+k-divided-bubble} \ar@<1ex>[u]^>>>>>{h} \ar@<-8ex>[l]^{\chi^0 \otimes \chi^0}
}
$
\caption{}\label{explicit-differential-lemma-figure1}

\end{figure}

Consider the diagram in Figure \ref{explicit-differential-lemma-figure1}, where the morphisms are induced by the apparent local changes of MOY graphs. To prove Theorem \ref{explicit-differential-general}, we need the following lemma.

\begin{lemma}\label{explicit-differential-lemma}
\begin{eqnarray*}
(\phi_1' \otimes \phi_2') \circ \phi_3' & \approx & h \circ (\chi^1 \otimes \chi^1) \circ \phi_3 \circ (\phi_1 \otimes \phi_2), \\
\overline{\phi_3'} \circ(\overline{\phi_1'} \otimes \overline{\phi_2'}) & \approx & (\overline{\phi}_1 \otimes \overline{\phi}_2) \circ \overline{\phi}_3 \circ (\chi^0 \otimes \chi^0) \circ \overline{h}.
\end{eqnarray*}
That is, the diagram in Figure \ref{explicit-differential-lemma-figure1} commutes up to homotopy and scaling in both directions.
\end{lemma}

\begin{proof}
Let
\[
\left.%
\begin{array}{ll}
f=(\phi_1' \otimes \phi_2') \circ \phi_3', & \overline{f} = \overline{\phi_3'} \circ(\overline{\phi_1'} \otimes \overline{\phi_2'}), \\
g= h \circ (\chi^1 \otimes \chi^1) \circ \phi_3 \circ (\phi_1 \otimes \phi_2), & \overline{g} = (\overline{\phi}_1 \otimes \overline{\phi}_2) \circ \overline{\phi}_3 \circ (\chi^0 \otimes \chi^0) \circ \overline{h}.
\end{array}%
\right.
\]
Then $f,~\overline{f}, ~g, ~\overline{g}$ are homogeneous morphisms of $\zed_2$-degree $0$ and quantum degree $\tau:= m-k+1-m(n+k-m)-nk$. 

Using Decomposition {II} (Theorem \ref{decomp-II}), we have 
\[
C(\Gamma') \simeq C(\Gamma)\{[n+k]\qb{n+k-1}{m}\qb{n+k-1}{n}\}.
\]
Denote by $\bigcirc_{n+k}$ an oriented circle colored by $n+k$. It is easy to check that
\begin{eqnarray*}
&& \Hom_\HMF (C(\Gamma),C(\Gamma')) \cong \Hom_\HMF (C(\Gamma'),C(\Gamma)) \\
& \cong & H(\bigcirc_{n+k}) \left\langle n+k\right\rangle \{[n+k]\qb{n+k-1}{m}\qb{n+k-1}{n}q^{(n+k)(N-n-k)}\} \\
& \cong & C(\emptyset) \{\qb{N}{n+k}[n+k]\qb{n+k-1}{m}\qb{n+k-1}{n}q^{(n+k)(N-n-k)}\}.
\end{eqnarray*}
Observe that:
\begin{itemize}
	\item These spaces are supported on $\zed_2$-degree $0$.
	\item The lowest non-vanishing quantum grading of these spaces is $\tau$.
	\item The subspaces of these spaces of homogeneous elements of quantum grading $\tau$ are $1$-dimensional.
\end{itemize}
Thus, to prove that $f\approx g$ and $ \overline{f} \approx \overline{g}$, we only need to show that $f,~\overline{f}, ~g, ~\overline{g}$ are all homotopically non-trivial.

By Lemma \ref{phibar-compose-phi}, we have
\[
\overline{f} \circ \mathfrak{m} (S_{\lambda_{m,n+k-1-m}}(\mathbb{A}) \cdot S_{\lambda_{n,k-1}}(\mathbb{B}) \cdot (-r)^{n+k-1}) \circ f \approx \id_{C(\Gamma)}.
\]
This implies that $f,~\overline{f}$ are not homotopic to $0$. 

By Corollary \ref{general-chi-maps-def}, we have

{\tiny
\begin{eqnarray*}
&& \overline{g} \circ \mathfrak{m} (S_{\lambda_{m,n+k-1-m}}(-\mathbb{X}) \cdot S_{\lambda_{n,k-1}}(-\mathbb{Y}) \cdot (-r)^{n+k-1}) \circ g \\
& \approx & (\overline{\phi}_1 \otimes \overline{\phi}_2) \circ \overline{\phi}_3 \circ \mathfrak{m} (S_{\lambda_{m,n+k-1-m}}(-\mathbb{X}) \cdot S_{\lambda_{n,k-1}}(-\mathbb{Y}) \cdot (-r)^{n+k-1}) \circ (\chi^0 \otimes \chi^0) \circ (\chi^1 \otimes \chi^1) \circ \phi_3 \circ (\phi_1 \otimes \phi_2) \\
& \approx & (\overline{\phi}_1 \otimes \overline{\phi}_2) \circ \overline{\phi}_3 \circ \mathfrak{m} (S_{\lambda_{m,n+k-1-m}}(-\mathbb{X}) \cdot S_{\lambda_{n,k-1}}(-\mathbb{Y}) \cdot (-r)^{n+k-1}\cdot(\sum_{j=0}^m (-r)^{m-j}A_j) \cdot(\sum_{i=0}^n (-r)^{n-i}B_i) )  \circ \phi_3 \circ (\phi_1 \otimes \phi_2) \\
& = & \sum_{j=0}^m \sum_{i=0}^n(\overline{\phi}_1 \otimes \overline{\phi}_2) \circ \overline{\phi}_3 \circ \mathfrak{m} (S_{\lambda_{m,n+k-1-m}}(-\mathbb{X})\cdot A_j  \cdot S_{\lambda_{n,k-1}}(-\mathbb{Y})\cdot B_i \cdot (-r)^{2n+m+k-1-i-j} )  \circ \phi_3 \circ (\phi_1 \otimes \phi_2),
\end{eqnarray*}
}

\noindent where $A_j$, $B_j$ are the $j$-th elementary symmetric polynomials in $\mathbb{A}$ and $\mathbb{B}$. But, by Lemma \ref{phibar-compose-phi}, the only homotopically non-trivial term on the right hand side is the one with $j=m, ~i=n$. So 
\[
\overline{g} \circ \mathfrak{m} (S_{\lambda_{m,n+k-1-m}}(-\mathbb{X}) \cdot S_{\lambda_{n,k-1}}(-\mathbb{Y}) \cdot (-r)^{n+k-1}) \circ g \approx \overline{\phi}_3 \circ \mathfrak{m} ( (-r)^{n+k-1} )  \circ \phi_3 \approx \id_{C(\Gamma)}.
\]
Thus, $g,~\overline{g}$ are not homotopic to $0$. 
\end{proof}

\begin{figure}[ht]
$
\xymatrix@R=4pc{
\input{v-vector-n+m+l} \ar@<12ex>[r]^{\phi_0} & \input{v-vector-n+m+l-bubble}  \ar@<-10ex>[l]^{\overline{\phi}_0}  \ar@<12ex>[r]^{\phi_1 \otimes \phi_2} & \input{v-vector-n+m+l-bubbles-in-bubble} \ar@<-10ex>[l]^{\overline{\phi}_1 \otimes \overline{\phi}_2} \ar@<-10ex>[lld]_{h_1} \\ 
\input{v-vector-n+m+l-divided-bubble-k}  \ar@<9ex>[rru]_{\overline{h}_1} \ar@<12ex>[r]^{\phi_3} & \input{v-vector-n+m+l-bubble-in-divided-bubble-k} \ar@<-10ex>[l]^{\overline{\phi}_3} \ar@<12ex>[r]^{\chi^1 \otimes \chi^1} & \input{v-vector-n+m+l-double-bubble} \ar@<-10ex>[l]^{\chi^0 \otimes \chi^0} \ar@<-10ex>[lld]_{h_2}  \\ 
\input{v-vector-n+m+l-bubble-in-divided-bubble-k-1} \ar@<9ex>[rru]_{\overline{h}_2} \ar@<12ex>[r]^{\overline{\phi}_4} & \input{v-vector-n+m+l-divided-bubble-k-1} \ar@<-10ex>[l]^{\phi_4} \ar@<12ex>[r]^{h_3} & \input{v-vector-n+m+l-bubbles-in-bubble-1} \ar@<-10ex>[l]^{\overline{h}_3} \ar@<-10ex>[lld]_{\overline{\phi}_5 \otimes \overline{\phi}_6} \\
\input{v-vector-n+m+l-bubble-1} \ar@<9ex>[rru]_{\phi_5 \otimes \phi_6} \ar@<12ex>[rr]^{\overline{\phi}_7} && \input{v-vector-n+m+l} \ar@<-10ex>[ll]^{\phi_7}
}
$
\caption{}\label{explicit-differential-proof-figure1}

\end{figure}

\begin{proof}[Proof of Theorem \ref{explicit-differential-general}]
It is easy to check that $d_k^+$ and $d_{k-1}^-$ are homogeneous morphisms of $\zed_2$-degree $0$ and quantum degree $1-l$. Recall that the differential maps of the complexes in Definition \ref{complex-colored-crossing-def} and Proposition \ref{trivial-complex-prop} are homotopically non-trivial homogeneous morphisms uniquely determined up to homotopy and scaling by their quantum degrees. So, to prove Theorem \ref{explicit-differential-general}, we only need to show that, as morphisms of matrix factorizations, $d_k^+$ and $d_{k-1}^-$ are not null-homotopic.

Consider the MOY graphs in Figure \ref{explicit-differential-proof-figure1}, where the morphisms are induced by the apparent local changes of the MOY graphs. Note that, as morphisms between $C(\hat{\Gamma}_k)$ and $C(\hat{\Gamma}_{k-1})$,
\begin{eqnarray*}
d_k^+ & = & \overline{\phi}_4 \circ h_2 \circ (\chi^1\otimes\chi^1) \circ \phi_3, \\
d_{k-1}^- & = & \overline{\phi}_3 \circ (\chi^0 \otimes \chi^0) \circ \overline{h}_2 \circ \phi_4.
\end{eqnarray*}
So, by Lemma \ref{explicit-differential-lemma}, we have 
\begin{eqnarray*}
h_3 \circ d_k^+ \circ h_1 \circ (\phi_1 \otimes \phi_2) & \approx & \overline{\phi}_4 \circ (\phi_5 \otimes \phi_6) \circ h_4 \circ \phi_3, \\
(\overline{\phi}_1 \otimes \overline{\phi}_2)\circ \overline{h} _1 \circ d_{k-1}^- \overline{h}_3 & \approx & \overline{\phi}_3 \circ \overline{h}_4 \circ (\overline{\phi}_5 \otimes \overline{\phi}_6) \circ \phi_4 \approx \overline{\phi}_3 \circ \overline{h}_4 \circ \phi_4 \circ (\overline{\phi}_5 \otimes \overline{\phi}_6),
\end{eqnarray*}
where the morphisms on the right hand side are depicted in Figure \ref{explicit-differential-proof-figure2}. Note that some morphisms in Figures \ref{explicit-differential-proof-figure1} and \ref{explicit-differential-proof-figure2} are given the same notations. This is because they are induced by the same local changes of MOY graphs.

\begin{figure}[ht]
$
\xymatrix{
\input{v-vector-n+m+l-bubble} \ar@<12ex>[r]^{\phi_3} & \input{v-vector-n+m+l-bubble-in-bubble-left} \ar@<-10ex>[l]^{\overline{\phi}_3} \ar@<12ex>[r]^{h_4} & \input{v-vector-n+m+l-bubble-in-bubble-right} \ar@<-10ex>[l]^{\overline{h}_4} \ar@<-10ex>[lld]_{\phi_5 \otimes \phi_6} \\
\input{v-vector-n+m+l-bubbles-in-bubble-2} \ar@<9ex>[rru]_{\overline{\phi}_5 \otimes \overline{\phi}_6} \ar@<12ex>[rr]^{\overline{\phi}_4} & & \input{v-vector-n+m+l-bubbles-in-bubble-1} \ar@<-10ex>[ll]^{\phi_4} 
}
$
\caption{}\label{explicit-differential-proof-figure2}

\end{figure}

Note that $\overline{\phi}_7 \circ (\overline{\phi}_5 \otimes \overline{\phi}_6) \circ \overline{\phi}_4 \approx \overline{\phi}_0 \circ \overline{\phi}_3 \circ \overline{h}_4 \circ (\overline{\phi}_5 \otimes \overline{\phi}_6)$. So, by Lemma \ref{phibar-compose-phi}, we have

{\tiny
\begin{eqnarray*}
&& \overline{\phi}_7 \circ (\overline{\phi}_5 \otimes \overline{\phi}_6) \circ \mathfrak{m}(S_{\lambda_{m,n+k-1-m}}(\mathbb{D}) \cdot  S_{\lambda_{n,k-1}}(\mathbb{E}) \cdot  B_{n+k-1}) \circ  h_3 \circ d_k^+ \circ h_1 \circ (\phi_1 \otimes \phi_2) \circ \mathfrak{m}(S_{\lambda_{n+k,m+l-k}}(-\mathbb{X})) \circ \phi_0 \\
& \approx & \overline{\phi}_7 \circ (\overline{\phi}_5 \otimes \overline{\phi}_6) \circ \mathfrak{m}(S_{\lambda_{m,n+k-1-m}}(\mathbb{D}) \cdot  S_{\lambda_{n,k-1}}(\mathbb{E}) \cdot B_{n+k-1}) \circ \overline{\phi}_4 \circ  (\phi_5 \otimes \phi_6) \circ h_4 \circ \phi_3 \circ \mathfrak{m}(S_{\lambda_{n+k,m+l-k}}(-\mathbb{X})) \circ \phi_0 \\
& \approx & \overline{\phi}_7 \circ (\overline{\phi}_5 \otimes \overline{\phi}_6) \circ \overline{\phi}_4 \circ \mathfrak{m}(S_{\lambda_{m,n+k-1-m}}(\mathbb{D}) \cdot  S_{\lambda_{n,k-1}}(\mathbb{E}) \cdot B_{n+k-1}) \circ  (\phi_5 \otimes \phi_6) \circ h_4 \circ \phi_3 \circ \mathfrak{m}(S_{\lambda_{n+k,m+l-k}}(-\mathbb{X})) \circ \phi_0 \\
& \approx & \overline{\phi}_0 \circ \overline{\phi}_3 \circ \overline{h}_4 \circ (\overline{\phi}_5 \otimes \overline{\phi}_6) \circ \mathfrak{m}(S_{\lambda_{m,n+k-1-m}}(\mathbb{D}) \cdot  S_{\lambda_{n,k-1}}(\mathbb{E}) \cdot B_{n+k-1}) \circ  (\phi_5 \otimes \phi_6) \circ h_4 \circ \phi_3 \circ \mathfrak{m}(S_{\lambda_{n+k,m+l-k}}(-\mathbb{X})) \circ \phi_0 \\
& \approx & \overline{\phi}_0 \circ \mathfrak{m}(S_{\lambda_{n+k,m+l-k}}(-\mathbb{X})) \circ \phi_0 \approx \id_{C(\Gamma)},
\end{eqnarray*}
}

\noindent where $B_{n+k-1}$ is the $(n+k-1)$-th elementary symmetric polynomial in $\mathbb{B}$. This implies that $d_k^+$ is not null-homotopic.

Similarly, note that $\overline{\phi}_0 \circ \overline{\phi}_3 \approx \overline{\phi}_7 \circ \overline{\phi}_4 \circ h_4$. So, by Lemma \ref{phibar-compose-phi}, we have

{\tiny
\begin{eqnarray*}
&& \overline{\phi}_0 \circ \mathfrak{m}(X_{m+l-k}) \circ (\overline{\phi}_1 \otimes \overline{\phi}_2)\circ \overline{h} _1 \circ d_{k-1}^- \overline{h}_3 \circ \mathfrak{m}( S_{\lambda_{m,n+k-1-m}}(\mathbb{D}) \cdot  S_{\lambda_{n,k-1}}(\mathbb{E})) \circ (\phi_5 \otimes \phi_6) \circ \mathfrak{m}(S_{\lambda_{n+k-1,m+l+1-k}}(\mathbb{B})) \circ \phi_7 \\
& \approx & \overline{\phi}_0 \circ \mathfrak{m}(X_{m+l-k}) \circ \overline{\phi}_3 \circ \overline{h}_4 \circ \phi_4 \circ (\overline{\phi}_5 \otimes \overline{\phi}_6) \circ \mathfrak{m}( S_{\lambda_{m,n+k-1-m}}(\mathbb{D}) \cdot  S_{\lambda_{n,k-1}}(\mathbb{E})) \circ (\phi_5 \otimes \phi_6) \circ \mathfrak{m}(S_{\lambda_{n+k-1,m+l+1-k}}(\mathbb{B})) \circ \phi_7 \\
& \approx & \overline{\phi}_0 \circ \overline{\phi}_3 \circ \overline{h}_4 \circ \mathfrak{m}(X_{m+l-k}) \circ \phi_4 \circ (\overline{\phi}_5 \otimes \overline{\phi}_6) \circ \mathfrak{m}( S_{\lambda_{m,n+k-1-m}}(\mathbb{D}) \cdot  S_{\lambda_{n,k-1}}(\mathbb{E})) \circ (\phi_5 \otimes \phi_6) \circ \mathfrak{m}(S_{\lambda_{n+k-1,m+l+1-k}}(\mathbb{B})) \circ \phi_7 \\
& \approx & \overline{\phi}_7 \circ \overline{\phi}_4 \circ \mathfrak{m}(X_{m+l-k}) \circ \phi_4 \circ (\overline{\phi}_5 \otimes \overline{\phi}_6) \circ \mathfrak{m}( S_{\lambda_{m,n+k-1-m}}(\mathbb{D}) \cdot  S_{\lambda_{n,k-1}}(\mathbb{E})) \circ (\phi_5 \otimes \phi_6) \circ \mathfrak{m}(S_{\lambda_{n+k-1,m+l+1-k}}(\mathbb{B})) \circ \phi_7 \\
& \approx & \overline{\phi}_7 \circ \mathfrak{m}(S_{\lambda_{n+k-1,m+l+1-k}}(\mathbb{B})) \circ \phi_7 \approx \id_{C(\Gamma)},
\end{eqnarray*}
}

\noindent where $X_j$ is the $j$-th elementary symmetric polynomial in $\mathbb{X}$. This implies that $d_{k-1}^-$ is not null-homotopic.
\end{proof}

If, in a colored crossing, one of the two branches is colored by $1$, then we have a simpler explicit description of the chain complex associated to this crossing. 

\begin{figure}[ht]
$
\xymatrix{
\input{crossing-1-n-+} && \input{crossing-1-n--}
}
$
\caption{}\label{explicit-differential-1-n-crossings-fig}

\end{figure}

Consider the colored crossings $c_{1,n}^+$ and $c_{1,n}^-$ in Figure \ref{explicit-differential-1-n-crossings-fig}. Their MOY resolutions are given in Figure \ref{explicit-differential-1-n-crossings--res-fig}.

\begin{figure}[ht]
$
\xymatrix{
\input{crossing-1-n-res-0} && \input{crossing-1-n-res-1}
}
$
\caption{}\label{explicit-differential-1-n-crossings--res-fig}

\end{figure}

Recall that Proposition \ref{general-general-chi-maps} (or, more precisely, Corollary \ref{chi-maps-def}) gives homogeneous morphisms $\chi^0:C(\Gamma_0) \rightarrow C(\Gamma_1)$ and $\chi^1:C(\Gamma_1) \rightarrow C(\Gamma_0)$ that have $\zed_2$-degree $0$ and quantum degree $1$. By Proposition \ref{general-general-chi-maps-HMF}, up to homotopy and scaling, $\chi^0$ and $\chi^1$ are the unique homotopically non-trivial homogeneous morphisms with such degrees. Thus, we have the following corollary.

\begin{corollary}\label{explicit-differential-1-n-crossings--res}
The unnormalized chain complexes of $c_{1,n}^+$ and $c_{1,n}^-$ are
\begin{eqnarray*}
\hat{C}(c_{1,n}^+) & = & ``0 \rightarrow \underbrace{C(\Gamma_1)}_{0} \xrightarrow{\chi^1} \underbrace{C(\Gamma_0)\{q^{-1}\}}_{1} \rightarrow 0", \\
\hat{C}(c_{1,n}^-) & = & ``0 \rightarrow \underbrace{C(\Gamma_0)\{q\}}_{-1} \xrightarrow{\chi^0} \underbrace{C(\Gamma_1)}_{0} \rightarrow 0",
\end{eqnarray*}
where the numbers in the under-braces are the homological gradings.

The differential maps in the chain complexes of $c_{m,1}^{\pm}$ can also be similarly expressed as the corresponding $\chi^0$ and $\chi^1$. The details are left to the reader. 
\end{corollary}

\begin{remark}\label{generalizing-KR-homology}
Corollary \ref{explicit-differential-1-n-crossings--res} shows that, for $c_{1,n}^{\pm}$ and $c_{m,1}^{\pm}$, the chain complexes defined in Definition \ref{complex-colored-crossing-def} specialize to the corresponding chain complexes defined in \cite{Yonezawa3}. In particular, for $c_{1,1}^\pm$, the chain complexes defined in Definition \ref{complex-colored-crossing-def} specialize to the corresponding complexes in \cite{KR1}. So our construction is a generalization of the $\mathfrak{sl}(N)$ Khovanov-Rozansky homology.
\end{remark}

\subsection{The graded Euler characteristic and the $\zed_2$-grading} We now prove Theorem \ref{euler-char-main}. First, we introduce the colored rotation number of a closed trivalent MOY graph.

\begin{figure}[ht]
$
\xymatrix{
\input{tri-vertex-s-split} & \text{or} & \input{tri-vertex-m-split}
} 
$
\caption{}\label{tri-vertex-split} 

\end{figure}

Let $\Gamma$ be a closed trivalent MOY graph. Replace each edge of $\Gamma$ of color $m$ by $m$ parallel edges colored by $1$ and replace each vertex of $\Gamma$, as depicted in Figure \ref{tri-vertex}, to the corresponding configuration in Figure \ref{tri-vertex-split}, in which each strand is an edge colored by $1$. This changes $\Gamma$ into a collection of disjoint embedded circles in the plane. 

\begin{definition}
The colored rotation number $\mathrm{cr}(\Gamma)$ of $\Gamma$ is defined to be the sum of the usual rotation numbers of these circles. (See equation \eqref{eq-def-rot-usual}.)
\end{definition}

Recall that the homology $H(\Gamma)$ of a MOY graph $\Gamma$ is defined in Definition \ref{homology-MOY-def}, and the graded dimension $\gdim(C(\Gamma))$ is defined to be
\[
\gdim(C(\Gamma)) = \sum_{\ve, i} \tau^\ve q^i H^{\ve,i}(\Gamma) \in \C[\tau,q]/(\tau^2),
\]
where $H^{\ve,i}(\Gamma)$ is the subspace of $H(\Gamma)$ of homogeneous elements of $\zed_2$-degree $\ve$ and quantum degree $i$. 

Theorem \ref{euler-char-main} follows from the next lemma.

\begin{lemma}\label{MOY-gdim-rt}
Let $\Gamma$ be a closed trivalent MOY graph. Then
\begin{enumerate}
	\item $\gdim(C(\Gamma))|_{\tau=1} = \left\langle \Gamma \right\rangle_N$,
	\item $H^{\ve,i}(\Gamma)=0$ if $\ve-\mathrm{cr}(\Gamma)=1$.
\end{enumerate}
\end{lemma}

\begin{proof}
Recall that by Theorem \ref{MOY-poly-skein-unique}, the $\mathfrak{sl}(N)$ MOY polynomial $\left\langle \Gamma \right\rangle_N$ is uniquely determined by the equations in Theorem \ref{MOY-poly-skein}. But these equations have been categorified in Corollaries \ref{contract-expand}, \ref{circle-dimension} and Theorems \ref{decomp-II}, \ref{decomp-I}, \ref{decomp-III}, \ref{decomp-IV}, \ref{decomp-V}. Thus, $\gdim(C(\Gamma))|_{\tau=1}$ satisfies all the equations in Theorem \ref{MOY-poly-skein}. So $\gdim(C(\Gamma))|_{\tau=1} = \left\langle \Gamma \right\rangle_N$ by Theorem \ref{MOY-poly-skein-unique}. 

Part (2) of the lemma can be proved by a double induction on the highest color of edges of $\Gamma$ and on the number of edges of $\Gamma$ with the highest color. The argument is extremely similar to that in the proof of Theorem \ref{MOY-poly-skein-unique}. We leave the details to the reader.
\end{proof}

\begin{proof}[Proof of Theorem \ref{euler-char-main}]
First, by comparing Definitions \ref{MOY-poly-def} and \ref{complex-colored-crossing-def}, we can see that the equation $\mathrm{P}_L (1, q, -1) = \mathrm{RT}_L(q)$ follows easily from part (1) of Lemma \ref{MOY-gdim-rt}. 

Next, we consider the $\zed_2$-grading. Let $D$ be a diagram of $L$ and $\Gamma$ be any complete resolution of $D$. Note that the number $\mathrm{cr}(\Gamma)$ does not depend on the choice of $\Gamma$. We define $\mathrm{cr}(D)=\mathrm{cr}(\Gamma)$. At each crossing $c$ of $D$, define an adjustment term $\mathsf{a}(c)$ by 
\[
\mathsf{a}\left(\setlength{\unitlength}{1pt}
\begin{picture}(40,40)(-20,0)

\put(-20,-20){\vector(1,1){40}}

\put(20,-20){\line(-1,1){15}}

\put(-5,5){\vector(-1,1){15}}

\put(-11,15){\tiny{$_m$}}

\put(9,15){\tiny{$_n$}}

\end{picture}\right) = \mathsf{a}\left(\setlength{\unitlength}{1pt}
\begin{picture}(40,40)(-20,0)

\put(20,-20){\vector(-1,1){40}}

\put(-20,-20){\line(1,1){15}}

\put(5,5){\vector(1,1){15}}

\put(-11,15){\tiny{$_m$}}

\put(9,15){\tiny{$_n$}}

\end{picture}\right) =
\begin{cases}
m & \text{if } m=n,\\
0 & \text{if } m \neq n,
\end{cases}
\]
Define $\hat{\tc}(D) = \mathrm{cr}(D) + \sum_c \mathsf{a}(c)$, where $c$ runs through all crossings of $D$. Then, by Definition \ref{complex-colored-crossing-def} and part (2) of Lemma \ref{MOY-gdim-rt}, 
\[
H^{\ve,i,j}(L) =0 \text{ if } \ve-\hat{\tc}(D) = 1 \in \zed_2,
\] 
Note that the parity of $\hat{\tc}(D)$ is invariant under Reidemeister moves and unknotting\footnote{``Unknotting" means switching the top- and bottom- strands at a crossing.}. Using these moves, we can change $D$ into a link diagram $U$ without crossings. That is, a collection of disjoint colored circles. It is clear that $\hat{\tc}(U) = \tc(L)$. So, as elements of $\zed_2$, $\hat{\tc}(D) = \hat{\tc}(U) = \tc(L)$. This completes the proof.
\end{proof}

\section{Invariance under Fork Sliding}\label{sec-inv-fork}

In this section, we prove the invariance of the homotopy type of the unnormalized chain complex associated to a knotted MOY graph under fork sliding. This is the most complex part of the proof of the invariance of the colored $\mathfrak{sl}(N)$ link homology. Once we have the invariance under fork sliding, the invariance of the colored $\mathfrak{sl}(N)$ link homology reduces to an easy induction based on the highest color of the link. Theorem \ref{fork-sliding-invariance-general} below is the main result of this section.

\begin{figure}[ht]
$
\xymatrix{
\input{fork-sliding-general-10} & \input{fork-sliding-general-11} && \input{fork-sliding-general-12} & \input{fork-sliding-general-13} \\
\input{fork-sliding-general-20} & \input{fork-sliding-general-21} && \input{fork-sliding-general-22} & \input{fork-sliding-general-23} \\
\input{fork-sliding-general-30} & \input{fork-sliding-general-31} && \input{fork-sliding-general-32} & \input{fork-sliding-general-33} \\
\input{fork-sliding-general-40} & \input{fork-sliding-general-41} && \input{fork-sliding-general-42} & \input{fork-sliding-general-43}
}
$
\caption{}\label{fork-sliding-invariance-general-fig}

\end{figure}

\begin{theorem}\label{fork-sliding-invariance-general}
Let $D_{i,j}^\pm$ be the knotted MOY graphs in Figure \ref{fork-sliding-invariance-general-fig}. Then $\hat{C}(D_{i,0}^+) \simeq \hat{C}(D_{i,1}^+)$ and $\hat{C}(D_{i,0}^-) \simeq \hat{C}(D_{i,1}^-)$. That is, $\hat{C}(D_{i,0}^+)$ (resp. $\hat{C}(D_{i,0}^-)$) is isomorphic in $\hch(\hmf)$ to $\hat{C}(D_{i,1}^+)$ (resp. $\hat{C}(D_{i,1}^-)$).
\end{theorem}

We prove Theorem \ref{fork-sliding-invariance-general} by induction. The hardest part of the proof is to show that Theorem \ref{fork-sliding-invariance-general} is true for certain special cases in which either $m=1$ or $l=1$. Once these special cases are proved, the rest of the induction is quite easy. Next, we state these special cases of Theorem \ref{fork-sliding-invariance-general} separately as Proposition \ref{fork-sliding-invariance-special} and then use this proposition to prove Theorem \ref{fork-sliding-invariance-general}. After that, we devote the rest of this section to proving Proposition \ref{fork-sliding-invariance-special}.

\begin{proposition}\label{fork-sliding-invariance-special}
Let $D_{i,j}^\pm$ be the knotted MOY graphs in Figure \ref{fork-sliding-invariance-general-fig}. 
\begin{enumerate}[(i)]
	\item If $l=1$, then $\hat{C}(D_{i,0}^+) \simeq \hat{C}(D_{i,1}^+)$ and $\hat{C}(D_{i,0}^-) \simeq \hat{C}(D_{i,1}^-)$ for $i=1,4$.
	\item If $m=1$, then $\hat{C}(D_{i,0}^+) \simeq \hat{C}(D_{i,1}^+)$ and $\hat{C}(D_{i,0}^-) \simeq \hat{C}(D_{i,1}^-)$ for $i=2,3$.
\end{enumerate}
\end{proposition}

\begin{proof}[Proof of Theorem \ref{fork-sliding-invariance-general} (assuming Proposition \ref{fork-sliding-invariance-special} is true)]
Each homotopy equivalence in Theorem \ref{fork-sliding-invariance-general} can be proved by an induction on $m$ or $l$. We only give details for the proof of 
\begin{equation}\label{fork-sliding-invariance-general-induction-1-+}
\hat{C}(D_{1,0}^+) \simeq \hat{C}(D_{1,1}^+).
\end{equation} 
The proof of the rest of Theorem \ref{fork-sliding-invariance-general} is very similar and left to the reader.

We prove \eqref{fork-sliding-invariance-general-induction-1-+} by an induction on $l$. The $l=1$ case is covered by Part (i) of Proposition \ref{fork-sliding-invariance-special}. Assume that \eqref{fork-sliding-invariance-general-induction-1-+} is true for some $l=k\geq 1$. Consider $l=k+1$. 

\begin{figure}[ht]
$
\xymatrix{
\input{fork-sliding-induction-10} \ar@<10ex>[rr]^{h}_{\cong} && \input{fork-sliding-induction-morph1} \ar@<10ex>[rr]^{\alpha}_{\simeq} && \input{fork-sliding-induction-morph2} \ar@<-10ex>[lllld]_{\beta}^{\simeq} \\
\input{fork-sliding-induction-morph3} \ar@<10ex>[rr]^{\xi}_{\simeq} && \input{fork-sliding-induction-morph4} \ar@<10ex>[rr]^{\overline{h}}_{\cong} && \input{fork-sliding-induction-11}
}
$
\caption{}\label{fork-sliding-invariance-induction-fig2}

\end{figure}

Let $\widetilde{D}_{10}^+$ and $\widetilde{D}_{11}^+$ be the first and last knotted MOY graphs in Figure \ref{fork-sliding-invariance-induction-fig2}. By Decomposition (II) (Theorem \ref{decomp-II}), we have  $\hat{C}(\widetilde{D}_{10}^+) \cong \hat{C}(D_{10}^+)\{[k+1]\}$ and $\hat{C}(\widetilde{D}_{11}^+) \cong \hat{C}(D_{11}^+)\{[k+1]\}$ in $\ch(\hmf)$. Consider the diagram in Figure \ref{fork-sliding-invariance-induction-fig2}, where
\begin{itemize}
	\item $h$ and $\overline{h}$ are the isomorphisms in $\ch(\hmf)$ induced by the apparent bouquet moves,
	\item $\alpha$ is the isomorphism in $\hch(\hmf)$ given by induction hypothesis,
	\item $\beta$ is the isomorphism in $\hch(\hmf)$ given by Part (i) of Proposition \ref{fork-sliding-invariance-special},
	\item $\xi$ is again the isomorphism in $\hch(\hmf)$ given by Part (i) of Proposition \ref{fork-sliding-invariance-special}.
\end{itemize}
Altogether, we have
\[
\hat{C}(D_{10}^+)\{[k+1]\} \cong \hat{C}(\widetilde{D}_{10}^+) \simeq \hat{C}(\widetilde{D}_{11}^+) \cong \hat{C}(D_{11}^+)\{[k+1]\}.
\]
So, by Proposition \ref{yonezawa-lemma-hmf}, $\hat{C}(D_{10}^+) \simeq \hat{C}(D_{11}^+)$ when $l=k+1$.
\end{proof}

In the remainder of this section, we concentrate on proving Proposition \ref{fork-sliding-invariance-special}. We only give the detailed proofs of $\hat{C}(D_{1,0}^\pm) \simeq \hat{C}(D_{1,1}^\pm)$ when $l=1$. The proof of the rest of Proposition \ref{fork-sliding-invariance-special} is very similar and left to the reader.

\subsection{Notations used in the proof}\label{fork-sliding-complexes-involved} In the rest of this section, we fix $l=1$. Then $D_{10}^\pm$ and $D_{11}^\pm$ are the knotted MOY graphs in Figure \ref{fork-sliding-invariance-special-fig}. Keep in mind that we are trying to prove 
\begin{equation}\label{fork-sliding-invariance-special-eq}
\hat{C}(D_{10}^\pm) \simeq \hat{C}(D_{11}^\pm) \text{ if } l=1.
\end{equation} 
Several chain complexes appear in the proof of \eqref{fork-sliding-invariance-special-eq}. We list these chain complexes in this subsection. In particular, we give names to the MOY graphs and morphisms of matrix factorizations appearing in these chain complexes. These names will be used throughout the rest of this section.

\begin{figure}[ht]
$
\xymatrix{
\input{fork-sliding-special-10} & \input{fork-sliding-special-11} && \input{fork-sliding-special-12} & \input{fork-sliding-special-13} 
}
$
\caption{}\label{fork-sliding-invariance-special-fig}

\end{figure}

Note that there is only one crossing in $D_{10}^\pm$, which is of the type $c_{m+1,n}^\pm$. We denote by $\tilde{d}_k^\pm$ the differential map of $\hat{C}(c_{m+1,n}^\pm)$. 

\begin{figure}[ht]
$
\xymatrix{
\input{tilde-gamma-k} 
}
$
\caption{}\label{fork-sliding-special-complex1-fig}

\end{figure}

Denote by $\widetilde{\Gamma}_k$ the MOY graph in Figure \ref{fork-sliding-special-complex1-fig}. Then $\hat{C}(D_{10}^+)$ is 
\begin{equation}\label{complex-D-10-+}
0 \rightarrow C(\widetilde{\Gamma}_{m+1}) \xrightarrow{\tilde{d}_{m+1}^+} C(\widetilde{\Gamma}_{m})\{q^{-1}\} \xrightarrow{\tilde{d}_{m}^+} \cdots \xrightarrow{\tilde{d}_{\tilde{k}_0+1}^+} C(\widetilde{\Gamma}_{\tilde{k}_0})\{q^{\tilde{k}_0-m-1}\} \rightarrow 0,
\end{equation}
where $\tilde{k}_0 := \max\{0,m+1-n\}$. Similarly, $\hat{C}(D_{10}^-)$ is 
\begin{equation}\label{complex-D-10--}
0 \rightarrow C(\widetilde{\Gamma}_{\tilde{k}_0})\{q^{m+1-\tilde{k}_0}\} \xrightarrow{\tilde{d}_{\tilde{k}_0}^-} \cdots \xrightarrow{\tilde{d}_{m-1}^-} C(\widetilde{\Gamma}_{m})\{q\} \xrightarrow{\tilde{d}_{m}^-} C(\widetilde{\Gamma}_{m+1}) \rightarrow 0.
\end{equation}

\begin{figure}[ht]
$
\xymatrix{
\input{gamma-k-prime} &&& \input{gamma-k-double-prime} 
}
$
\caption{}\label{fork-sliding-special-complex2-fig}

\end{figure}

Let $\Gamma_k'$ and $\Gamma_k''$ be the MOY graphs in Figure \ref{fork-sliding-special-complex2-fig}. Let $\delta_k^\pm: C(\Gamma_k') \rightarrow C(\Gamma_{k\mp 1}')$ be the morphisms defined in Definition \ref{trivial-complex-differential-def} with explicit form given in Theorem \ref{explicit-differential-general}. Let $C^+$ be the chain complex
\begin{equation}\label{complex-D-contractible-+}
0 \rightarrow C(\Gamma_{m-1}'') \xrightarrow{J_{m-1,m-1}} C(\Gamma_{m-1}') \xrightarrow{\delta_{m-1}^+} \cdots \xrightarrow{\delta_{k_0+1}^+} C(\Gamma_{k_0}') \rightarrow 0,
\end{equation}
and $C^-$ the chain complex 
\begin{equation}\label{complex-D-contractible--}
0 \rightarrow C(\Gamma_{k_0}') \xrightarrow{\delta_{k_0}^-} \cdots \xrightarrow{\delta_{m-2}^-} C(\Gamma_{m-1}') \xrightarrow{P_{m-1,m-1}}  C(\Gamma_{m-1}'') \rightarrow 0,
\end{equation}
where $k_0 = \max \{m-n,0\}$ and $J_{m-1,m-1}$, $P_{m-1,m-1}$ are defined in Definition \ref{trivial-complex-differential-def}. Then, by Lemma \ref{decomp-V-special-2} and Proposition \ref{trivial-complex-prop}, both $C^+$ and $C^-$ are isomorphic in $\ch(\hmf)$ to
\[
\bigoplus_{j=k_0}^{m-1} (\xymatrix{
0 \ar[r] & C(\Gamma_j'') \ar[r]^{\simeq} & C(\Gamma_j'') \ar[r] & 0
}),
\]
which means they are homotopic to $0$.

\begin{figure}[ht]
$
\xymatrix{
\input{gamma-k0} && \input{gamma-k1} 
}
$
\caption{}\label{fork-sliding-special-complex3-fig}

\end{figure}

Now consider $\hat{C}(D_{1,1}^\pm)$. Note that $D_{1,1}^\pm$ has two crossings -- one $c_{m,n}^\pm$ and one $c_{1,n}^\pm$. Denote by $d_k^\pm$ the differential map of the $c_{m,n}^\pm$ crossing. From Corollary \ref{explicit-differential-1-n-crossings--res}, the differential map $c_{1,n}^+$ (resp. $c_{1,n}^-$) is $\chi^1$ (resp. $\chi^0$.) Let $\Gamma_{k,0}$ and $\Gamma_{k,1}$ be the MOY graphs in Figure \ref{fork-sliding-special-complex3-fig}. Then $d_k^\pm$ acts on the left square in $\Gamma_{k,0}$ and $\Gamma_{k,1}$, and $\chi^0$, $\chi^1$ act on the upper right corners of $\Gamma_{k,0}$ and $\Gamma_{k,1}$. The chain complex $\hat{C}(D_{1,1}^+)$ is 
{\tiny
\begin{equation}\label{complex-D-11-+}
0 \rightarrow C(\Gamma_{m,1}) \xrightarrow{\mathfrak{d}_m^+} \left.%
\begin{array}{c}
C(\Gamma_{m,0}) \{q^{-1}\}\\
\oplus \\
C(\Gamma_{m-1,1})\{q^{-1}\}
\end{array}%
\right. 
\xrightarrow{\mathfrak{d}_{m-1}^+} \cdots \xrightarrow{\mathfrak{d}_{k+1}^+} \left.%
\begin{array}{c}
C(\Gamma_{k+1,0}) \{q^{k-m}\}\\
\oplus \\
C(\Gamma_{k,1})\{q^{k-m}\}
\end{array}%
\right. 
\xrightarrow{\mathfrak{d}_{k}^+} \cdots\xrightarrow{\mathfrak{d}_{k_0}^+} C(\Gamma_{k_0,0}) \{q^{k_0-1-m}\} \rightarrow 0,
\end{equation}
}
where $k_0 = \max \{m-n,0\}$ as above and
\begin{eqnarray*}
\mathfrak{d}_m^+ & = & \left(%
\begin{array}{c}
\chi^1\\
-d_m^+
\end{array}%
\right),\\
\mathfrak{d}_k^+ & = & \left(%
\begin{array}{cc}
d_{k+1}^+ & \chi^1\\
0 & -d_k^+
\end{array}%
\right) ~\text{ for } k_0<k<m, \\
\mathfrak{d}_{k_0}^+ & = & \left(%
\begin{array}{cc}
d_{k_0 +1}^+, & \chi^1\\
\end{array}%
\right).
\end{eqnarray*}
Similarly, The chain complex $\hat{C}(D_{1,1}^-)$ is 
{\tiny
\begin{equation}\label{complex-D-11--}
0 \rightarrow C(\Gamma_{k_0,0}) \{q^{m+1-k_0}\} \xrightarrow{\mathfrak{d}_{k_0}^-} \cdots \xrightarrow{\mathfrak{d}_{k-1}^-} \left.%
\begin{array}{c}
C(\Gamma_{k,0}) \{q^{m+1-k}\}\\
\oplus \\
C(\Gamma_{k-1,1})\{q^{m+1-k}\}
\end{array}%
\right. 
\xrightarrow{\mathfrak{d}_{k}^-} \cdots  \xrightarrow{\mathfrak{d}_{m-1}^-} \left.%
\begin{array}{c}
C(\Gamma_{m,0}) \{q\}\\
\oplus \\
C(\Gamma_{m-1,1})\{q\}
\end{array}%
\right. 
\xrightarrow{\mathfrak{d}_m^-} C(\Gamma_{m,1}) \rightarrow 0,
\end{equation}
}
where $k_0 = \max \{m-n,0\}$ as above and
\begin{eqnarray*}
\mathfrak{d}_{k_0}^- & = & \left(%
\begin{array}{c}
d_{k_0}^-\\ 
\chi^0
\end{array}%
\right),\\
\mathfrak{d}_k^- & = & \left(%
\begin{array}{cc}
d_{k}^- & 0\\
\chi^0 & -d_{k-1}^-
\end{array}%
\right) ~\text{ for } k_0<k<m, \\
\mathfrak{d}_{m}^- & = & \left(%
\begin{array}{cc}
\chi^0, & -d_{m-1}^-\\
\end{array}%
\right).
\end{eqnarray*}

Next, we study relations between the chain complexes $\hat{C}(D_{1,1}^\pm)$, $\hat{C}(D_{1,0}^\pm)$ and $C^\pm$.

\subsection{Commutativity lemmas} To prove \eqref{fork-sliding-invariance-special-eq}, we will frequently use the fact that certain morphisms of matrix factorizations of MOY graphs commute with each other. We establish two basic commutativity lemmas in this subsection.

\begin{figure}[ht]
$
\xymatrix{
\input{chi-commute-chi-chi-fig2} \ar@<12ex>[rrrr]^{\chi^1_\triangle} \ar@<1ex>[d]^<<<<<<{h_1} &&&& \input{chi-commute-chi-chi-fig0} \ar@<-10ex>[llll]^{\chi^0_\triangle} \ar@<1ex>[d]^<<<<<<{\overline{h}_2} \\
\input{chi-commute-chi-chi-fig1} \ar@<1ex>[u]^>>>>>>{\overline{h}_1} \ar@<10ex>[rr]^{\chi^1_\Box} && \input{chi-commute-chi-chi-fig3} \ar@<-8ex>[ll]^{\chi^0_\Box} \ar@<10ex>[rr]^{\chi^1_\dag} &&\input{chi-commute-chi-chi-fig4} \ar@<-8ex>[ll]^{\chi^0_\dag} \ar@<1ex>[u]^>>>>>>{h_2}
}
$
\caption{}\label{chi-commute-chi-chi-figure}

\end{figure}

\begin{lemma}\label{chi-commute-chi-chi}
Consider the diagram in Figure \ref{chi-commute-chi-chi-figure}, where the morphisms are induced by the apparent local changes of MOY graphs. Then $\chi^1_\triangle \approx h_2 \circ \chi^1_\dag \circ \chi^1_\Box \circ h_1$ and $\chi^0_\triangle \approx  \overline{h}_1 \circ \chi^0_\Box \circ \chi^0_\dag \circ \overline{h}_2$. That is, up to homotopy and scaling, the diagram in Figure \ref{chi-commute-chi-chi-figure} commutes in both directions.
\end{lemma}

\begin{figure}[ht]
$
\xymatrix{
\input{chi-commute-chi-chi-fig5} & \input{chi-commute-chi-chi-fig6}
}
$
\caption{}\label{chi-commute-chi-chi-figure1}

\end{figure}

\begin{proof}
Denote by $\bigcirc_{m+n+1}$ an oriented circle colored by $m+n+1$, and by $\Gamma$, $\Gamma'$ and MOY graphs in Figure \ref{chi-commute-chi-chi-figure1}. Let $\overline{\Gamma}$ be $\Gamma$ with its orientation reversed. Then, by Corollary \ref{contract-expand}, Theorem \ref{decomp-II} and Corollary \ref{circle-dimension}, 
\begin{eqnarray*}
&& \Hom_\HMF (C(\Gamma_1), C(\Gamma_0)) \cong H(\Gamma) \left\langle m+n+1 \right\rangle\{q^{(m+n+1)(N-m-n-1)+2m+2n+mn}\} \\
& \cong & H(\Gamma') \left\langle m+n+1 \right\rangle\{q^{(m+n+1)(N-m-n-1)+2m+2n+mn}\} \\
& \cong & H(\bigcirc_{m+n+1}) \left\langle m+n+1 \right\rangle\{[m+1]\qb{m+n}{m+1} [m+n+1]  q^{(m+n+1)(N-m-n-1)+2m+2n+mn}\} \\
& \cong & C(\emptyset) \{[m+1]\qb{m+n}{m+1} [m+n+1] \qb{N}{m+n+1} q^{(m+n+1)(N-m-n-1)+2m+2n+mn}\}.
\end{eqnarray*}
Similarly,
\begin{eqnarray*}
&& \Hom_\HMF (C(\Gamma_0), C(\Gamma_1)) \cong H(\overline{\Gamma}) \left\langle m+n+1 \right\rangle\{q^{(m+n+1)(N-m-n-1)+2m+2n+mn}\} \\
& \cong & C(\emptyset) \{[m+1]\qb{m+n}{m+1} [m+n+1] \qb{N}{m+n+1} q^{(m+n+1)(N-m-n-1)+2m+2n+mn}\}.
\end{eqnarray*}
So $\Hom_\HMF (C(\Gamma_1), C(\Gamma_0))$ and $\Hom_\HMF (C(\Gamma_0), C(\Gamma_1))$ are supported on $\zed_2$-degree $0$, have lowest non-vanishing quantum grading $m+1$. And the subspaces of homogeneous elements of quantum degree $m+1$ of these spaces are $1$-dimensional.

\begin{figure}[ht]
$
\xymatrix@R=4pc{
\input{v-vector-n+m+1} \ar@<10ex>[r]^{\phi_1} & \input{v-vector-n+m+1-bubble} \ar@<-8ex>[l]^{\overline{\phi}_1} \ar@<10ex>[r]^{\phi_2}& \input{v-vector-n+m+1-bubble-in-bubble-left}\ar@<-8ex>[l]^{\overline{\phi}_2} \ar@<-10ex>[lld]_{h_3}\\
\input{v-vector-n+m+1-theta-shape} \ar@<9ex>[rru]_{\overline{h}_3} \ar@<10ex>[r]^{\phi_3} & \input{v-vector-n+m+1-bubble-in-theta} \ar@<-8ex>[l]^{\overline{\phi}_3} \ar@<10ex>[r]^{\chi^0_\triangle,~\overline{g}} & \input{v-vector-n+m+1-bubble-in-bubbles} \ar@<-8ex>[l]^{\chi^1_\triangle,~g}
}
$
\caption{}\label{chi-commute-chi-chi-figure2}

\end{figure}

Let $g=h_2 \circ \chi^1_\dag \circ \chi^1_\Box \circ h_1$ and $\overline{g} = \overline{h}_1 \circ \chi^0_\Box \circ \chi^0_\dag \circ \overline{h}_2$. Note that $\chi^1_\triangle, ~\chi^0_\triangle, ~g$ and $\overline{g}$ are all homogeneous of quantum degree $m+1$. To show that $\chi^1_\triangle \approx g$ and $\chi^0_\triangle \approx \overline{g}$, we only need to show that none of these morphisms are null-homotopic. For this purpose, consider the diagram in Figure \ref{chi-commute-chi-chi-figure2}, where $\phi_i$, $\overline{\phi}_i$, $h_3$ and $\overline{h}_3$ are induced by the apparent local changes of MOY graphs. 

Let $u=(-r)^{m+n}$, $v=S_{\lambda_{m+1,n-2}}(-\mathbb{Y})$ and $w=X_m$. Here $X_j$ is the $j$-th elementary symmetric polynomial in $\mathbb{X}$. Then, by Corollary \ref{general-chi-maps-def} and Lemma \ref{phibar-compose-phi},
\begin{eqnarray*}
&& \overline{\phi}_1 \circ \mathfrak{m}(u) \circ \overline{\phi}_2 \circ \mathfrak{m}(v) \circ \overline{h}_3 \circ \overline{\phi}_3 \circ \mathfrak{m}(w) \circ \chi^1_\triangle \circ \chi^0_\triangle \circ \phi_3 \circ h_3 \circ \phi_2 \circ \phi_1 \\
& \approx & \overline{\phi}_1 \circ \mathfrak{m}(u) \circ \overline{\phi}_2 \circ \mathfrak{m}(v) \circ \overline{h}_3 \circ \overline{\phi}_3 \circ \mathfrak{m}(w\sum_{k=0}^{m+1}(-r)^k A_{m+1-k}) \circ \phi_3 \circ h_3 \circ \phi_2 \circ \phi_1 \\
& \approx & \overline{\phi}_1 \circ \mathfrak{m}(u) \circ \overline{\phi}_2 \circ \mathfrak{m}(v) \circ \mathfrak{m}(\sum_{k=0}^{m+1}(-r)^k A_{m+1-k})  \circ \overline{h}_3 \circ \overline{\phi}_3 \circ \mathfrak{m}(w) \circ \phi_3 \circ h_3 \circ \phi_2 \circ \phi_1 \\
& \approx & \overline{\phi}_1 \circ \mathfrak{m}(u) \circ \overline{\phi}_2 \circ \mathfrak{m}(v\sum_{k=0}^{m+1}(-r)^k A_{m+1-k})  \circ \phi_2 \circ \phi_1 \\
& \approx & \overline{\phi}_1 \circ \mathfrak{m}(u) \circ (\mathfrak{m}(\sum_{k=0}^{m+1}(-r)^k) \circ
\overline{\phi}_2 \circ \mathfrak{m}(v  A_{m+1-k})  \circ \phi_2) \circ \phi_1 \\
& \approx & \overline{\phi}_1 \circ \mathfrak{m}(u) \circ \phi_1 \approx \id_{C(\uparrow_{m+n+1})},
\end{eqnarray*}
where $A_j$ is the $j$-th elementary symmetric polynomial in $\mathbb{A}$. This shows that $\chi^1_\triangle$ and $\chi^0_\triangle$ are both homotopically non-trivial.

Note that, by Corollary \ref{general-chi-maps-def},
\begin{eqnarray*}
g \circ \overline{g} & = & h_2 \circ \chi^1_\dag \circ \chi^1_\Box \circ h_1 \circ \overline{h}_1 \circ \chi^0_\Box \circ \chi^0_\dag \circ \overline{h}_2 \\
& \approx & h_2 \circ \chi^1_\dag \circ \chi^1_\Box \circ \chi^0_\Box \circ \chi^0_\dag \circ \overline{h}_2 \\
& \approx & h_2 \circ \mathfrak{m}((s-r)\sum_{k=0}^m (-r)^k X_{m-k}) \circ \overline{h}_2 \\
& \approx & h_2 \circ \mathfrak{m}(\sum_{k=0}^{m+1} (-r)^k (X_{m+1-k}+ sX_{m-k})) \circ \overline{h}_2 \\
& \approx & \mathfrak{m}(\sum_{k=0}^{m+1} (-r)^k A_{m+1-k}) \approx \chi^1_\triangle \circ \chi^0_\triangle.
\end{eqnarray*}
So, the above argument also implies that
\[
\overline{\phi}_1 \circ \mathfrak{m}(u) \circ \overline{\phi}_2 \circ \mathfrak{m}(v) \circ \overline{h}_3 \circ \overline{\phi}_3 \circ \mathfrak{m}(w) \circ g \circ \overline{g} \circ \phi_3 \circ h_3 \circ \phi_2 \circ \phi_1 \approx \id_{C(\uparrow_{m+n+1})}.
\]
This shows that $g$ and $\overline{g}$ are both homotopically non-trivial.
\end{proof}

Before stating the second commutativity lemma, we introduce a shorthand notation, which will be used throughout the rest of this section.

\begin{figure}[ht]
$
\xymatrix{
\input{varphi-def-figure1} \ar@<10ex>[rr]^{\varphi} \ar@<1ex>[rd]^<<<<<<<<<<{\phi}&& \input{varphi-def-figure2} \ar@<-8ex>[ll]^{\overline{\varphi}} \ar@<1ex>[ld]^<<<<<<<<<<{\overline{h}} \\
& \input{varphi-def-figure3} \ar@<1ex>[lu]^>>>>>>>>>>{\overline{\phi}} \ar@<1ex>[ru]^>>>>>>>>>>{h} &
}
$
\caption{}\label{varphi-def-figure}

\end{figure}

\begin{definition}\label{varphi-def}
Consider the morphisms in Figure \ref{varphi-def-figure}, where $\phi$ and $\overline{\phi}$ are the morphisms induced by the apparent edge splitting and merging, $h$ and $\overline{h}$ are induced by the apparent bouquet moves. Define $\varphi:= h \circ \phi$ and $\overline{\varphi}:= \overline{\phi} \circ \overline{h}$.
\end{definition}

By Corollary \ref{contract-expand}, Lemmas \ref{edge-splitting-lemma} and \ref{phibar-compose-phi}, it is easy to check that, up to homotopy and scaling, $\varphi$ and $\overline{\varphi}$ are the unique homotopically non-trivial homogeneous morphisms between $C(\Gamma)$ and $C(\widetilde{\Gamma})$ of $\zed_2$-degree $0$ and quantum degree $-mn$. Moreover, they satisfy that, for $\lambda,\mu \in \Lambda_{m,n}=\{(\lambda_1\geq\cdots\geq\lambda_m)~|~ \lambda_1\leq n\}$,
\begin{equation}\label{varphi-compose}
\overline{\varphi} \circ \mathfrak{m}(S_{\lambda}(\mathbb{X})\cdot S_{\mu}(-\mathbb{Y})) \circ \varphi \approx \left\{%
\begin{array}{ll}
    \id_{C(\Gamma)} & \text{if } \lambda_i + \mu_{m+1-i} = n ~\forall i=1,\dots,m, \\ 
    0 & \text{otherwise.}
\end{array}%
\right. 
\end{equation}

\begin{lemma}\label{varphis-commute}
Consider the diagram in Figure \ref{varphis-commute-figure}, where $\varphi_i$ and $\overline{\varphi}_i$ are the morphisms defined in Definition \ref{varphi-def} associated to the apparent local changes of the MOY graphs. Then $\varphi_2\circ\varphi_1 \approx \varphi_4\circ\varphi_3$ and $\overline{\varphi}_1 \circ \overline{\varphi}_2 \approx \overline{\varphi}_3 \circ \overline{\varphi}_4$. That is, the diagram in Figure \ref{varphis-commute-figure} commutes up to homotopy and scaling in both directions.
\end{lemma}

\begin{figure}[ht]
$
\xymatrix{
\input{varphis-commute-figure1} \ar@<10ex>[rr]^{\varphi_1} \ar@<1ex>[d]^<<<<<{\varphi_3} && \input{varphis-commute-figure3} \ar@<-8ex>[ll]^{\overline{\varphi}_1} \ar@<1ex>[d]^<<<<<{\varphi_2} \\
\input{varphis-commute-figure2} \ar@<1ex>[u]^>>>>>{\overline{\varphi}_3} \ar@<10ex>[rr]^{\varphi_4} && \input{varphis-commute-figure4} \ar@<1ex>[u]^>>>>>{\overline{\varphi}_2}  \ar@<-8ex>[ll]^{\overline{\varphi}_4}
}
$
\caption{}\label{varphis-commute-figure}

\end{figure}

\begin{proof}
By Corollary \ref{contract-expand} and Decomposition (II) (Theorem \ref{decomp-II}), we have 
\[
C(\Gamma_1) \simeq C(\Gamma_0) \{\qb{m+n+l}{l}\qb{m+n}{n}\}.
\]
So
\begin{eqnarray*}
&& \Hom_\HMF (C(\Gamma_1), C(\Gamma_0)) \cong \Hom_\HMF (C(\Gamma_0), C(\Gamma_1)) \\ 
& \cong & \Hom_\HMF (C(\Gamma_0), C(\Gamma_0))\{\qb{m+n+l}{l}\qb{m+n}{n}\}.
\end{eqnarray*}

\begin{figure}[ht]
$
\xymatrix{
\input{varphis-commute-proof-figure1} 
}
$
\caption{}\label{varphis-commute-figure2}

\end{figure}

Let $\Gamma$ be the MOY graphs in Figure \ref{varphis-commute-figure2}. Then
\begin{eqnarray*}
&& \Hom_\HMF (C(\Gamma_0), C(\Gamma_0)) \\
& \cong & H(\Gamma) \left\langle m+n+l+j+k \right\rangle \{q^{(m+n+l+j+k)(N-m-n-l)-j^2-k^2}\} \\
& \cong & C(\emptyset) \{\qb{N-m-n-l}{j}\qb{N-m-n-l}{k}\qb{N}{m+n+l}q^{(m+n+l+j+k)(N-m-n-l)-j^2-k^2}\}.
\end{eqnarray*}
So, 

{\tiny
\begin{eqnarray*}
&& \Hom_\HMF (C(\Gamma_1), C(\Gamma_0)) \cong \Hom_\HMF (C(\Gamma_0), C(\Gamma_1)) \\ 
& \cong & C(\emptyset) \{\qb{N-m-n-l}{j}\qb{N-m-n-l}{k}\qb{N}{m+n+l} \qb{m+n+l}{l}\qb{m+n}{n}q^{(m+n+l+j+k)(N-m-n-l)-j^2-k^2}\}.
\end{eqnarray*}
}

\noindent Thus, $\Hom_\HMF (C(\Gamma_1), C(\Gamma_0))$ and $\Hom_\HMF (C(\Gamma_0), C(\Gamma_1))$ are supported on $\zed_2$-degree $0$ and have lowest non-vanishing quantum grading $-mn-ml-nl$. And the subspaces of $\Hom_\HMF (C(\Gamma_1), C(\Gamma_0))$ and $\Hom_\HMF (C(\Gamma_0), C(\Gamma_1))$ of homogeneous elements of quantum grading $-mn-ml-nl$ are $1$-dimensional.

Note that $\varphi_2\circ\varphi_1$, $\varphi_4\circ\varphi_3$, $\overline{\varphi}_1 \circ \overline{\varphi}_2$ and $\overline{\varphi}_3 \circ \overline{\varphi}_4$ are all homogeneous of quantum degree $-mn-ml-nl$. So, to prove that $\varphi_2\circ\varphi_1 \approx \varphi_4\circ\varphi_3$ and $\overline{\varphi}_1 \circ \overline{\varphi}_2 \approx \overline{\varphi}_3 \circ \overline{\varphi}_4$, we only need to show that $\varphi_2\circ\varphi_1$, $\varphi_4\circ\varphi_3$, $\overline{\varphi}_1 \circ \overline{\varphi}_2$ and $\overline{\varphi}_3 \circ \overline{\varphi}_4$ are homotopically non-trivial. For this purpose, consider equation \eqref{varphi-compose} above. We get 
\begin{eqnarray*}
& & \overline{\varphi}_1 \circ \overline{\varphi}_2 \circ \mathfrak{m}(S_{\lambda_{n,m}}(\mathbb{X}) \cdot S_{\lambda_{l,m+n}}(\mathbb{Y})) \circ \varphi_2\circ\varphi_1 \\
& \simeq & \overline{\varphi}_1 \circ \overline{\varphi}_2 \circ \mathfrak{m}(S_{\lambda_{n,m}}(\mathbb{X})) \circ \varphi_2 \circ \mathfrak{m}(S_{\lambda_{l,m+n}}(\mathbb{Y})) \circ \varphi_1 \\
& \approx & \id_{C(\Gamma_0)}.
\end{eqnarray*}
This shows that $\varphi_2\circ\varphi_1$ and $\overline{\varphi}_1 \circ \overline{\varphi}_2$ are homotopically non-trivial. Similarly,
\begin{eqnarray*}
& & \overline{\varphi}_3 \circ \overline{\varphi}_4 \circ \mathfrak{m}(S_{\lambda_{n,l}}(\mathbb{X}) \cdot S_{\lambda_{m,n+l}}(\mathbb{W})) \circ \varphi_4\circ\varphi_3 \\
& \simeq & \overline{\varphi}_3 \circ \overline{\varphi}_4 \circ \mathfrak{m}(S_{\lambda_{n,l}}(\mathbb{X})) \circ \varphi_4 \circ \mathfrak{m}(S_{\lambda_{m,n+l}}(\mathbb{W})) \circ \varphi_3 \\
& \approx & \id_{C(\Gamma_0)}.
\end{eqnarray*}
This shows that $\varphi_4\circ\varphi_3$ and $\overline{\varphi}_3 \circ \overline{\varphi}_4$ are homotopically non-trivial.
\end{proof}

\subsection{Another look at Decomposition (IV)} Decomposition (IV) (Theorem \ref{decomp-IV}) plays an important role in relating $\hat{C}(D_{1,1}^\pm)$ to $\hat{C}(D_{1,0}^\pm)$ and $C^\pm$. In this subsection, we review a special case of Decomposition (IV), including the construction of all the morphisms involved.

\begin{figure}[ht]
$
\xymatrix{
\input{fork-special-decomp-iv-1}  & \input{fork-special-decomp-iv-2} & \input{fork-special-decomp-iv-3}
}
$
\caption{}\label{fork-special-decomp-iv-fig}

\end{figure}

Consider the MOY graphs in Figure \ref{fork-special-decomp-iv-fig}. By Decomposition (IV) (Theorem \ref{decomp-IV}), we have
\begin{equation}\label{fork-special-decomp-iv}
C(\Gamma) \simeq C(\Gamma') \oplus C(\Gamma'')\{[m-k]\}.
\end{equation}

\begin{figure}[ht]
$
\xymatrix{
\input{fork-special-decomp-iv-2} \ar@<1ex>[d]^<<<<<<{\phi_1} \ar@<10ex>[rr]^{f} && \input{fork-special-decomp-iv-1} \ar@<-8ex>[ll]^{g} \ar@<1ex>[d]^<<<<<<{\chi^1} \\
\input{fork-special-decomp-iv-4} \ar@<1ex>[u]^>>>>>>{\overline{\phi}_1} \ar@<10ex>[rr]^{h_1} && \input{fork-special-decomp-iv-5} \ar@<-8ex>[ll]^{\overline{h}_1} \ar@<1ex>[u]^>>>>>>{\chi^0}
}
$
\caption{}\label{fork-special-decomp-iv-fig2}

\end{figure}

By the construction in Subsection \ref{decomp-IV-subsection-Gamma-Gamma0}, especially Lemma \ref{decomp-IV-f-g-lambda-compose}, we know that the inclusion and projection morphisms of the component $C(\Gamma')$ in decomposition \eqref{fork-special-decomp-iv} are given by the compositions in Figure \ref{fork-special-decomp-iv-fig2}. That is, if
\begin{eqnarray*}
f & = & \chi^0 \circ h_1 \circ \phi_1, \\
g & = & \overline{\phi}_1 \circ \overline{h}_1 \circ \chi^1,
\end{eqnarray*}
where the morphisms on the right hand side are induced by the apparent local changes of MOY graphs, then $f$ and $g$ are homogeneous morphisms preserving the $\zed_2\oplus\zed$-grading and, after possibly a scaling, $g\circ f \simeq \id_{C(\Gamma')}$.

\begin{figure}[ht]
$
\xymatrix{
\input{fork-special-decomp-iv-3}  \ar@<1ex>[d]^<<<<<<{\phi_2} \ar@<10ex>[rr]^{\alpha} && \input{fork-special-decomp-iv-1} \ar@<-8ex>[ll]^{\beta} \ar@<1ex>[d]^<<<<<<{\chi^0} \\
\input{fork-special-decomp-iv-6} \ar@<1ex>[u]^>>>>>>{\overline{\phi}_2} \ar@<10ex>[rr]^{h_2}&& \input{fork-special-decomp-iv-7} \ar@<-8ex>[ll]^{\overline{h}_2} \ar@<1ex>[u]^>>>>>>{\chi^1}
}
$
\caption{}\label{fork-special-decomp-iv-fig3}

\end{figure}

Similarly, consider the diagram in Figure \ref{fork-special-decomp-iv-fig3}, where 
\begin{eqnarray*}
\alpha & = & \chi^1 \circ h_2 \circ \phi_2, \\
\beta & = & \overline{\phi}_2 \circ \overline{h}_2 \circ \chi^0,
\end{eqnarray*}
and the morphisms on the right hand side are induced by the apparent local changes of MOY graphs. Recall that, from Subsection \ref{decomp-IV-subsection-Gamma-Gamma1}, especially the proof of Lemma \ref{decomp-IV-Alpha-Beta}, if we define
\begin{eqnarray*}
\vec{\alpha} & = & \sum_{j=0}^{m-k-1} \mathfrak{m}(r^j) \circ\alpha = (\alpha, ~\mathfrak{m}(r) \circ\alpha,~\dots, ~\mathfrak{m}(r^{m-k-1}) \circ\alpha),  \\
\vec{\beta} & = & \bigoplus_{j=0}^{m-k-1} \beta\circ\mathfrak{m}((-1)^{m-k-1-j}A_{m-k-1-j}) = \left(%
\begin{array}{c}
\beta\circ\mathfrak{m}((-1)^{m-k-1}A_{m-k-1})\\
\cdots \\
\beta\circ\mathfrak{m}(-A_1)\\
\beta
\end{array}%
\right),
\end{eqnarray*}
where $A_j$ is the $j$-th elementary symmetric polynomial in $\mathbb{A}$, then there is a homogeneous morphism $\tau:C(\Gamma'')\{[m-k]\} \rightarrow C(\Gamma'')\{[m-k]\}$ preserving the $\zed_2\oplus\zed$-grading such that $\tau \circ \vec{\beta}\circ \vec{\alpha} \simeq \vec{\beta}\circ \vec{\alpha} \circ \tau \simeq \id_{C(\Gamma'')\{[m-k]\}}$.

Now consider the morphisms
\begin{equation}\label{fork-special-decomp-iv-lemma-1}
\xymatrix{
C(\Gamma) \ar@<1ex>[rrr]^{\left(%
\begin{array}{c}
g\\
\tau\circ\vec{\beta}
\end{array}%
\right)} &&& {\left.%
\begin{array}{c}
C(\Gamma') \\
\oplus \\
C(\Gamma'')\{[m-k]\}\end{array}%
\right.}
\ar@<1ex>[lll]^{\left(%
\begin{array}{cc}
f, & \vec{\alpha}
\end{array}%
\right)}
}
\end{equation}
\begin{equation}\label{fork-special-decomp-iv-lemma-2}
\xymatrix{
C(\Gamma) \ar@<1ex>[rrr]^{\left(%
\begin{array}{c}
g\\
\vec{\beta}
\end{array}%
\right)} &&& {\left.%
\begin{array}{c}
C(\Gamma') \\
\oplus \\
C(\Gamma'')\{[m-k]\}\end{array}%
\right.}
\ar@<1ex>[lll]^{\left(%
\begin{array}{cc}
f, & \vec{\alpha}\circ \tau
\end{array}%
\right)}
}
\end{equation}

\begin{lemma}\label{fork-special-decomp-iv-lemma}
Each of diagrams \eqref{fork-special-decomp-iv-lemma-1} and \eqref{fork-special-decomp-iv-lemma-2} gives a pair of homogeneous homotopy equivalences preserving the $\zed_2\oplus\zed$-grading that are inverses of each other.
\end{lemma}

\begin{proof}
We know that $C(\Gamma) \simeq C(\Gamma') \oplus C(\Gamma'')\{[m-k]\}$. So, to prove the lemma, we only need to show that
\[
\left(%
\begin{array}{c}
g\\
\tau\circ\vec{\beta}
\end{array}%
\right)
\left(%
\begin{array}{cc}
f, & \vec{\alpha}
\end{array}%
\right)
\simeq 
\left(%
\begin{array}{c}
g\\
\vec{\beta}
\end{array}%
\right)
\left(%
\begin{array}{cc}
f, & \vec{\alpha}\circ \tau
\end{array}%
\right)
\simeq \id_{C(\Gamma') \oplus C(\Gamma'')\{[m-k]\}}.
\]

Consider $g\circ \vec{\alpha}$ and $\vec{\beta}\circ f$. By Lemma \ref{decomp-IV-nilopotency-sum-difference}, we know that
\[
\begin{cases}
g\circ\mathfrak{m}(r^j) \circ\alpha \simeq 0 & \text{if } j\leq m-k-1, \\
\beta\circ\mathfrak{m}((-1)^{m-k-1-j}A_{m-k-1-j}) \circ f \simeq 0 & \text{if } j\geq 0.
\end{cases}
\]
This shows that 
\begin{equation}\label{fork-special-decomp-iv-alpha-f}
g\circ \vec{\alpha}\simeq 0 \text{ and } \vec{\beta}\circ f \simeq 0.
\end{equation}
So, 
\begin{eqnarray*}
\left(%
\begin{array}{c}
g\\
\tau\circ\vec{\beta}
\end{array}%
\right)
\left(%
\begin{array}{cc}
f, & \vec{\alpha}
\end{array}%
\right)
& \simeq &
\left(%
\begin{array}{cc}
\id_{C(\Gamma')} & g\circ \vec{\alpha}\\
\tau\circ\vec{\beta}\circ f & \id_{C(\Gamma'')\{[m-k]\}}
\end{array}%
\right) \\
& \simeq &
\left(%
\begin{array}{cc}
\id_{C(\Gamma')} & 0\\
0 & \id_{C(\Gamma'')\{[m-k]\}}
\end{array}%
\right) \\
& = &
\id_{C(\Gamma') \oplus C(\Gamma'')\{[m-k]\}}.
\end{eqnarray*}
Similarly,
\[
\left(%
\begin{array}{c}
g\\
\vec{\beta}
\end{array}%
\right)
\left(%
\begin{array}{cc}
f, & \vec{\alpha}\circ \tau
\end{array}%
\right)
\simeq \id_{C(\Gamma') \oplus C(\Gamma'')\{[m-k]\}}.
\]
\end{proof}

Next, we apply the above discussion to the MOY graphs that appear in the chain complexes in Subsection \ref{fork-sliding-complexes-involved}.

\begin{figure}[ht]
$
\xymatrix{
\input{gamma-k0} & \input{tilde-gamma-k} & \input{gamma-k-prime} \\
\input{gamma-k0-bouquet} & \input{gamma-k3} 
}
$
\caption{}\label{fork-special-decomp-iv-fig4}

\end{figure}

Consider the MOY graphs in Figure \ref{fork-special-decomp-iv-fig4}. By Corollary \ref{contract-expand}, we have $C(\Gamma_{k,0}) \simeq C(\Gamma_{k,2})$ and $C(\widetilde{\Gamma}_k) \simeq C(\Gamma_{k,3})$. By \eqref{fork-special-decomp-iv}, $C(\Gamma_{k,2}) \simeq C(\Gamma_{k,3}) \oplus C(\Gamma_k')\{[m-k]\}$. Altogether, we have
\begin{equation}\label{fork-special-decomp-iv-eq}
C(\Gamma_{k,0}) \simeq C(\widetilde{\Gamma}_k) \oplus C(\Gamma_k')\{[m-k]\}.
\end{equation}

\begin{figure}[ht]
$
\xymatrix@R=4pc{
\input{tilde-gamma-k} \ar@<10ex>[rr]^{f_k} \ar@<1ex>[d]^<<<<<<<<{h} && \input{gamma-k0} \ar@<-8ex>[ll]^{g_k} \ar@<1ex>[d]^<<<<<<<<{\chi^1} \ar@<-5ex>[dll]_{\hat{g}_k} \\
\input{gamma-k3} \ar@<1ex>[u]^>>>>>>>>{\overline{h}} \ar@<10ex>[rr]^{\varphi_1} \ar@<4ex>[urr]_{\hat{f}_k} && \input{gamma-k3-3-squares} \ar@<-8ex>[ll]^{\overline{\varphi}_1} \ar@<1ex>[u]^>>>>>>>>{\chi^0}
}
$
\caption{}\label{fork-special-decomp-iv-fig5}

\end{figure}

In Figure \ref{fork-special-decomp-iv-fig5}, the morphism $f_k$, $g_k$, $\hat{f}_k$ and $\hat{g}_k$ are defined by
\begin{eqnarray*}
f_k & = & \chi^0 \circ \varphi_1 \circ h, \\
g_k & = & \overline{h} \circ \overline{\varphi}_1 \circ \chi^1, \\
\hat{f}_k & = & \chi^0 \circ \varphi_1, \\
\hat{g}_k & = & \overline{\varphi}_1 \circ \chi^1,
\end{eqnarray*}
where the morphisms on the right hand side are induced by the apparent local changes of MOY graphs. Then, after possibly a scaling,
\begin{equation}\label{fork-special-decomp-iv-f-g-k}
g_k \circ f_k \simeq \id_{C(\widetilde{\Gamma}_k)}.
\end{equation}

\begin{figure}[ht]
$
\xymatrix{
\input{gamma-k-prime} \ar@<11ex>[rr]^{\alpha_k} \ar@<1ex>[dr]^>>>>>>>>>>>>>>>{\varphi_2} && \input{gamma-k0-marked} \ar@<-9ex>[ll]^{\beta_k} \ar@<1ex>[dl]^>>>>>>>>>>>>>>>{\chi^0} \\
& \input{gamma-k-prime-2-squares} \ar@<1ex>[ul]^<<<<<<<<<<<<<<<{\overline{\varphi}_2} \ar@<1ex>[ur]^<<<<<<<<<<<<<<<{\chi^1} &
}
$
\caption{}\label{fork-special-decomp-iv-fig6}

\end{figure}

In Figure \ref{fork-special-decomp-iv-fig6}, the morphisms $\alpha_k$ and $\beta_k$ are defined by 
\begin{eqnarray*}
\alpha_k & = & \chi^1 \circ \varphi_2, \\
\beta_k & = & \overline{\varphi}_2 \circ \chi^0,
\end{eqnarray*}
where the morphisms on the right hand side are induced by the apparent local changes of MOY graphs. Define
\begin{eqnarray*}
\vec{\alpha}_k & = & \sum_{j=0}^{m-k-1} \mathfrak{m}(r^j) \circ\alpha_k = (\alpha_k, ~\mathfrak{m}(r) \circ\alpha_k,~\dots, ~\mathfrak{m}(r^{m-k-1}) \circ\alpha_k),  \\
\vec{\beta}_k & = & \bigoplus_{j=0}^{m-k-1} \beta_k \circ\mathfrak{m}((-1)^{m-k-1-j}A_{m-k-1-j}) = \left(%
\begin{array}{c}
\beta_k \circ\mathfrak{m}((-1)^{m-k-1}A_{m-k-1})\\
\cdots \\
\beta_k \circ\mathfrak{m}(-A_1)\\
\beta_k
\end{array}%
\right).
\end{eqnarray*}
Then there is a homogeneous morphism  $\tau_k:C(\Gamma_k')\{[m-k]\} \rightarrow C(\Gamma_k')\{[m-k]\}$ preserving the $\zed_2\oplus\zed$-grading such that
\begin{equation}\label{fork-special-decomp-iv-alpha-beta-k}
\tau_k \circ \vec{\beta}_k \circ \vec{\alpha}_k \simeq \vec{\beta}_k \circ \vec{\alpha}_k \circ \tau_k \simeq \id_{C(\Gamma_k')\{[m-k]\}}.
\end{equation}
We also have
\begin{equation}\label{fork-special-decomp-iv-alpha-f-k}
g_k\circ \vec{\alpha}_k\simeq 0 \text{ and } \vec{\beta}_k\circ f_k \simeq 0.
\end{equation}

From Lemma \ref{fork-special-decomp-iv-lemma}, we get the following corollary.

\begin{corollary}\label{fork-special-decomp-iv-corollary}
\begin{equation}\label{fork-special-decomp-iv-cor-1}
\xymatrix{
C(\Gamma_{k,0}) \ar@<1ex>[rrr]^{\left(%
\begin{array}{c}
g_k\\
\tau_k\circ\vec{\beta}_k
\end{array}%
\right)} &&& {\left.%
\begin{array}{c}
C(\widetilde{\Gamma}_k) \\
\oplus \\
C(\Gamma_k')\{[m-k]\}\end{array}%
\right.}
\ar@<1ex>[lll]^{\left(%
\begin{array}{cc}
f_k, & \vec{\alpha}_k
\end{array}%
\right)}
}
\end{equation}
\begin{equation}\label{fork-special-decomp-iv-cor-2}
\xymatrix{
C(\Gamma_k) \ar@<1ex>[rrr]^{\left(%
\begin{array}{c}
g_k\\
\vec{\beta}_k
\end{array}%
\right)} &&& {\left.%
\begin{array}{c}
C(\widetilde{\Gamma}_k) \\
\oplus \\
C(\Gamma_k')\{[m-k]\}\end{array}%
\right.}
\ar@<1ex>[lll]^{\left(%
\begin{array}{cc}
f_k, & \vec{\alpha}_k\circ \tau_k
\end{array}%
\right)}
}
\end{equation}
are two ways to explicitly write down the inclusion and projection morphisms in decomposition \eqref{fork-special-decomp-iv-eq}.
\end{corollary}

\subsection{Relating the differential maps of $C^\pm$ and $\hat{C}(D_{1,1}^\pm)$} In this subsection, we proved the following lemma, which relates the differential map of $C^\pm$ to that of $\hat{C}(D_{1,1}^\pm)$.

\begin{figure}[ht]
$
\xymatrix{
\input{square-m-n-1-right-k-low-new-mark} \ar@<9ex>[rr]^{\delta_k^+} \ar@<1ex>[d]^<<<<<{\mathfrak{m}(r^{m-k})\circ \varphi_1} && \input{square-m-n-1-right-k-1-low-new-mark} \ar@<-7ex>[ll]^{\delta_{k-1}^-} \ar@<1ex>[d]^<<<<<{\varphi_2} \\
\input{double-square-m-n-1-right-k-m-low} \ar@<1ex>[u]^>>>>>{\overline{\varphi}_1 \circ \mathfrak{m}(r^{m-k})} \ar@<9ex>[rr]^{d_k^+} && \input{double-square-m-n-1-right-k-1-m-low} \ar@<1ex>[u]^>>>>>{\overline{\varphi}_2} \ar@<-7ex>[ll]^{d_{k-1}^-}
}
$
\caption{}\label{relating-C-D-11-fig}

\end{figure}

\begin{lemma}\label{relating-C-D-11-lemma}
Consider the diagram in Figure \ref{relating-C-D-11-fig}, where $\delta_k^+$, $\delta_{k-1}^-$ are defined in Definition \ref{trivial-complex-differential-def}, $d_k^+$, $d_{k-1}^-$ act on the left square, and $\varphi_i$, $\overline{\varphi}_i$ are induced by the apparent local changes of MOY graphs.

Then $\delta_k^+ \approx \overline{\varphi}_2 \circ d_k^+ \circ \mathfrak{m}(r^{m-k})\circ \varphi_1$ and $\delta_{k-1}^- \approx \overline{\varphi}_1 \circ \mathfrak{m}(r^{m-k}) \circ d_{k-1}^- \circ \varphi_2$. That is, the diagram in Figure \ref{relating-C-D-11-fig} commutes up to homotopy and scaling in both directions.
\end{lemma}

\begin{proof}
\begin{figure}[ht]
$
\xymatrix{
\input{square-m-n-1-right-k-low-new-mark} \ar@<7ex>[rr]^{\mathfrak{m}(r^{m-k})\circ \varphi_1} \ar[d]^<<<<<{\phi_3} && \input{double-square-m-n-1-right-k-m-low} \ar[d]^<<<<<{\phi_3}\\
\input{square-m-n-1-right-k-low-bubble} \ar[d]^<<<<<{\chi^1\otimes\chi^1} && \input{double-square-m-n-1-right-k-bubble} \ar[d]^<<<<<{\chi^1\otimes\chi^1} \\
\input{double-square-m-n-1-right-k-1-k-low} \ar[d]^<<<<<{\overline{\varphi}_4} \ar@<8ex>[rr]^{\mathfrak{m}(r^{m-k})\circ \varphi_1} &&  \input{triple-square-m-n-1-right-k-m-low} \ar[d]^<<<<<{\overline{\varphi}_5} \ar@<-6ex>[ll]^{\overline{\varphi}_1}\\
\input{square-m-n-1-right-k-1-low-new-mark}  &&   \input{double-square-m-n-1-right-k-1-m-low} \ar@<-7ex>[ll]_{\overline{\varphi}_2}
}
$
\caption{}\label{relating-C-D-11-fig2}

\end{figure}

Consider the diagram in Figure \ref{relating-C-D-11-fig2}, where the morphisms are induces by the apparent local changes of MOY graphs. By Theorem \ref{explicit-differential-general}, the composition of the morphisms in the left column is $\delta_k^+$ and the composition of the morphisms in the right column is $d_k^+$. That is, 
\begin{eqnarray*}
\delta_k^+ & \approx & \overline{\varphi}_4 \circ (\chi^1\otimes\chi^1)\circ\phi_3, \\
d_k^+ & \approx & \overline{\varphi}_5 \circ (\chi^1\otimes\chi^1)\circ\phi_3.
\end{eqnarray*}
Since $(\chi^1\otimes\chi^1)\circ\phi_3$ and $\mathfrak{m}(r^{m-k})\circ \varphi_1$ act on different parts of the MOY graphs, they commute with each other. So 
\[
\mathfrak{m}(r^{m-k})\circ \varphi_1 \circ (\chi^1\otimes\chi^1)\circ\phi_3 \approx (\chi^1\otimes\chi^1)\circ\phi_3 \circ \mathfrak{m}(r^{m-k})\circ \varphi_1.
\]
That is, the upper rectangle in Figure \ref{relating-C-D-11-fig2} commutes up to homotopy and scaling. By Lemma \ref{varphis-commute}, the lower square in Figure \ref{relating-C-D-11-fig2} commutes up to homotopy and scaling. That is, $\overline{\varphi}_4 \circ \overline{\varphi}_1 \approx \overline{\varphi}_2 \circ \overline{\varphi}_5$. Recall that, by Lemma \ref{phibar-compose-phi}, we have $\overline{\varphi}_1 \circ \mathfrak{m}(r^{m-k})\circ \varphi_1 \approx \id$. Altogether, we have
\begin{eqnarray*}
\delta_k^+ & \approx & \overline{\varphi}_4 \circ (\chi^1\otimes\chi^1)\circ\phi_3 \\
& \approx & \overline{\varphi}_4 \circ \overline{\varphi}_1 \circ \mathfrak{m}(r^{m-k})\circ \varphi_1 \circ (\chi^1\otimes\chi^1)\circ\phi_3 \\
& \approx & \overline{\varphi}_2 \circ \overline{\varphi}_5 \circ (\chi^1\otimes\chi^1)\circ\phi_3 \circ \mathfrak{m}(r^{m-k})\circ \varphi_1 \\
& \approx & \overline{\varphi}_2 \circ d_k^+ \circ \mathfrak{m}(r^{m-k})\circ \varphi_1.
\end{eqnarray*}

\begin{figure}[ht]
$
\xymatrix{
\input{square-m-n-1-right-k-1-low-new-mark}  \ar@<7ex>[rr]^{{\varphi}_2} \ar[d]^<<<<<{{\varphi}_4} &&   \input{double-square-m-n-1-right-k-1-m-low} \ar[d]^<<<<<{{\varphi}_5} \\
\input{double-square-m-n-1-right-k-1-k-low} \ar[d]^<<<<<{\chi^0\otimes\chi^0} \ar@<8ex>[rr]^{\varphi_1} &&  \input{triple-square-m-n-1-right-k-m-low}  \ar@<-6ex>[ll]^{\overline{\varphi}_1 \circ \mathfrak{m}(r^{m-k})} \ar[d]^<<<<<{\chi^0\otimes\chi^0}\\
\input{square-m-n-1-right-k-low-bubble} \ar[d]^<<<<<{\overline{\phi}_3} && \input{double-square-m-n-1-right-k-bubble} \ar[d]^<<<<<{\overline{\phi}_3}\\
\input{square-m-n-1-right-k-low-new-mark}   && \input{double-square-m-n-1-right-k-m-low} \ar@<-7ex>[ll]_{\overline{\varphi}_1 \circ\mathfrak{m}(r^{m-k})}
}
$
\caption{}\label{relating-C-D-11-fig3}

\end{figure}

Similarly, consider the diagram in Figure \ref{relating-C-D-11-fig3}, where the morphisms are induces by the apparent local changes of MOY graphs. By Theorem \ref{explicit-differential-general}, the composition of the morphisms in the left column is $\delta_{k-1}^-$ and the composition of the morphisms in the right column is $d_{k-1}^-$. That is,
\begin{eqnarray*}
\delta_{k-1}^- & \approx &  \overline{\phi}_3\circ(\chi^0\otimes\chi^0) \circ \varphi_4,\\
d_{k-1}^- & \approx &  \overline{\phi}_3\circ(\chi^0\otimes\chi^0) \circ \varphi_5.
\end{eqnarray*}
Since $\overline{\phi}_3\circ(\chi^0\otimes\chi^0)$ and $\overline{\varphi}_1 \circ\mathfrak{m}(r^{m-k})$ act on different parts of the MOY graphs, they commute with each other. So 
\[
\overline{\varphi}_1 \circ\mathfrak{m}(r^{m-k}) \circ  \overline{\phi}_3\circ(\chi^0\otimes\chi^0) \approx \overline{\phi}_3\circ(\chi^0\otimes\chi^0) \circ \overline{\varphi}_1 \circ\mathfrak{m}(r^{m-k}).
\]
That is, the lower square in Figure \ref{relating-C-D-11-fig3} commutes up to homotopy and scaling. By Lemma \ref{varphis-commute}, the upper square in Figure \ref{relating-C-D-11-fig3} commutes up to homotopy and scaling. That is, $\varphi_1 \circ \varphi_4 \approx \varphi_5 \circ \varphi_2$. Again, we have $\overline{\varphi}_1 \circ \mathfrak{m}(r^{m-k})\circ \varphi_1 \approx \id$. Altogether, we have
\begin{eqnarray*}
\delta_{k-1}^- & \approx &  \overline{\phi}_3\circ(\chi^0\otimes\chi^0) \circ \varphi_4\\
& \approx &  \overline{\phi}_3\circ(\chi^0\otimes\chi^0)\circ \overline{\varphi}_1 \circ \mathfrak{m}(r^{m-k})\circ \varphi_1 \circ \varphi_4\\
& \approx & \overline{\varphi}_1 \circ\mathfrak{m}(r^{m-k}) \circ  \overline{\phi}_3\circ(\chi^0\otimes\chi^0) \circ \varphi_5 \circ \varphi_2 \\
& \approx & \overline{\varphi}_1 \circ\mathfrak{m}(r^{m-k}) \circ d_{k-1}^- \circ \varphi_2.
\end{eqnarray*}
\end{proof}

\subsection{Relating the differential maps of $\hat{C}(D_{1,0}^\pm)$ and $\hat{C}(D_{1,1}^\pm)$}

\begin{figure}[ht]
$
\xymatrix{
\input{gamma-k3} \ar@<10ex>[rr]^{\hat{f}_k} \ar@<1ex>[d]^<<<<<{\phi_1} && \input{gamma-k0} \ar@<-8ex>[ll]^{\hat{g}_k} \ar@<1ex>[d]^<<<<<{\phi_1} \\
\input{gamma-k3-bubble} \ar@<10ex>[rr]^{\hat{f}_k} \ar@<1ex>[u]^>>>>>{\overline{\phi}_1} \ar@<1ex>[d]^<<<<<{\chi^1\otimes \chi^1}&& \input{gamma-k0-bubble} \ar@<-8ex>[ll]^{\hat{g}_k} \ar@<1ex>[u]^>>>>>{\overline{\phi}_1} \ar@<1ex>[d]^<<<<<{\chi^1\otimes \chi^1} \\
\input{gamma-k3-left-square} \ar@<10ex>[rr]^{\hat{f}_k} \ar@<1ex>[u]^>>>>>{\chi^0\otimes \chi^0}&& \input{gamma-k0-left-square} \ar@<-8ex>[ll]^{\hat{g}_k} \ar@<1ex>[u]^>>>>>{\chi^0\otimes \chi^0}
}
$
\caption{}\label{relating-D-0-1-fig1}

\end{figure}

First, consider the diagram in Figure \ref{relating-D-0-1-fig1}, where $\hat{f}_k$ and $\hat{g}_k$ are the diagonal morphisms in Figure \ref{fork-special-decomp-iv-fig5} and the vertical morphisms are induced by the apparent local changes of MOY graphs. Note that $\hat{f}_k$ and $\hat{g}_k$ act on the right side of the MOY graphs only, and the vertical morphisms act on the left side only. So each square in the Figure \ref{relating-D-0-1-fig1} commutes in both directions up to homotopy and scaling. Thus, we have the following lemma.

\begin{lemma}\label{relating-D-0-1-claim1}
In Figure \ref{relating-D-0-1-fig1}, we have $\hat{f}_k \circ (\chi^1\otimes \chi^1) \circ \phi_1 \approx (\chi^1\otimes \chi^1) \circ \phi_1 \circ \hat{f}_k$ and $\hat{g}_k \circ \overline{\phi}_1 \circ (\chi^0\otimes \chi^0) \approx \overline{\phi}_1 \circ (\chi^0\otimes \chi^0) \circ \hat{g}_k$.
\end{lemma}

\begin{figure}[ht]
$
\xymatrix@R=4pc{
\input{gamma-k0-left-square-box} \ar@<8ex>[r]^{\chi^1_\Box} \ar@<1ex>[dd]^{\overline{h}_3}& \ar@<-6ex>[l]^{\chi^0_\Box} \input{triple-squares-dag} \ar@<1ex>[d]^<<<<<<<<{\chi^1_\dag} \ar@<8ex>[r]^{\overline{\varphi}_5} & \input{triple-squares-dag-merged} \ar@<1ex>[d]^<<<<<<<<{\chi^1_\dag} \ar@<-6ex>[l]^{{\varphi}_5} \\
\ar@{}[r]|<<<<<<<<<<<<<<<{\mathbf{(A)}} &  \input{triple-squares-after-dag} \ar@{}[ur]|>>>>>>>>>>>>>>>>>>>>{\mathbf{(C)}}  \ar@<1ex>[u]^>>>>>>>>{\chi^0_\dag} \ar@<1ex>[d]^<<<<<<<<{\overline{h}_4} \ar@<8ex>[r]^{\overline{\varphi}_5} &  \input{triple-squares-after-dag-merged} \ar@<1ex>[u]^>>>>>>>>{\chi^0_\dag} \ar@<-6ex>[l]^{{\varphi}_5} \ar@<1ex>[dd]^{\overline{\varphi}_7} \\
\input{gamma-k-1-0-bubble}  \ar@{}[dr]|>>>>>>>>>>>>>>>{\mathbf{(B)}} \ar@<8ex>[r]^{\chi^1_\triangle} \ar@<1ex>[uu]^{{h}_3} \ar@<1ex>[d]^<<<<<<<<{\overline{\phi}_3} & \input{gamma-k-1-0-bubble-chi} \ar@<1ex>[d]^<<<<<<<<{\overline{\phi}_3} \ar@<-6ex>[l]^{\chi^0_\triangle} \ar@<1ex>[u]^>>>>>>>>{{h}_4} \ar@{}[r]|>>>>>>>>>>>>>>>>>{\mathbf{(D)}} & \\
\input{gamma-k-1-0} \ar@<8ex>[r]^{\chi^1_\triangle} \ar@<1ex>[u]^>>>>>>>>{{\phi}_3} & \input{gamma-k-1-0-chi} \ar@<-6ex>[l]^{\chi^0_\triangle} \ar@<1ex>[u]^>>>>>>>>{{\phi}_3} \ar@<8ex>[r]^{\overline{\varphi}_6} & \input{gamma-k-1-3} \ar@<1ex>[uu]^{{\varphi}_7} \ar@<-6ex>[l]^{{\varphi}_6}
} 
$
\caption{}\label{relating-D-0-1-fig2}

\end{figure}

Next, consider the diagram in Figure \ref{relating-D-0-1-fig2}, where all morphisms are induced by the apparent local changes of the MOY graphs. We have the following lemma.

\begin{lemma}\label{relating-D-0-1-claim2}
The four squares $\mathbf{(A)}$, $\mathbf{(B)}$, $\mathbf{(C)}$ and $\mathbf{(D)}$ in Figure \ref{relating-D-0-1-fig2} all commute up to homotopy and scaling in both directions. More precisely, we have
\begin{enumerate}[(A)]
	\item $\chi^1_\triangle \circ \overline{h}_3 \approx \overline{h}_4 \circ \chi^1_\dag \circ \chi^1_\Box$, $h_3 \circ \chi^0_\triangle \approx \chi^0_\Box \circ \chi^0_\dag \circ h_4$,
	\item $\chi^1_\triangle \circ \overline{\phi}_3 \approx \overline{\phi}_3 \circ \chi^1_\triangle$, $\phi_3 \circ \chi^0_\triangle \approx \chi^0_\triangle \circ \phi_3$,
	\item $\overline{\varphi}_5 \circ \chi^1_\dag \approx \overline{\varphi}_5 \circ \chi^1_\dag$, $ \chi^0_\dag \circ \varphi_5 \approx \varphi_5 \circ \chi^0_\dag$,
	\item $\overline{\varphi}_7 \circ \overline{\varphi}_5 \approx \overline{\varphi}_6 \circ \overline{\phi}_3 \circ \overline{h}_4$, $\varphi_5 \circ \varphi_7 \approx h_4 \circ \phi_3 \circ \varphi_6$.
\end{enumerate}

Altogether, we have 
\begin{eqnarray*}
\overline{\varphi}_6 \circ \chi^1_\triangle \circ \overline{\phi}_3 \circ \overline{h}_3 & \approx & \overline{\varphi}_7 \circ \chi^1_\dag \circ \overline{\varphi}_5 \circ \chi^1_\Box, \\
h_3 \circ \phi_3\circ \chi^0_\triangle \circ \varphi_6 & \approx & \chi^0_\Box \circ \varphi_5 \circ \chi^0_\dag \circ \varphi_7.
\end{eqnarray*}
\end{lemma}

\begin{proof}
Part (A) follows from Lemma \ref{chi-commute-chi-chi}. (B) and (C) are true because the horizontal and vertical morphisms act on different parts of the MOY graphs. Part (D) follows from Lemma \ref{varphis-commute}.
\end{proof}

\begin{figure}[ht]
$
\xymatrix{
\input{tilde-gamma-k} \ar@<11ex>[rr]^{\tilde{d}_k^+} \ar@<1ex>[d]^<<<<<{f_k} && \input{tilde-gamma-k-1} \ar@<-9ex>[ll]^{\tilde{d}_{k-1}^-} \ar@<1ex>[d]^<<<<<{f_{k-1}} \\
\input{gamma-k0} \ar@<11ex>[rr]^{d_k^+} \ar@<1ex>[u]^>>>>>{g_k} && \input{gamma-k-1-0-no-tri} \ar@<-9ex>[ll]^{d_{k-1}^-} \ar@<1ex>[u]^>>>>>{g_{k-1}}
} 
$
\caption{}\label{relating-D-0-1-fig3}

\end{figure}

We are now ready to relate the differential map of $\hat{C}(D_{1,0}^\pm)$ to that of $\hat{C}(D_{1,1}^\pm)$. Consider the diagram in Figure \ref{relating-D-0-1-fig3}, where $d_\ast^\pm$ and $\tilde{d}_\ast^\pm$ are defined in Subsection \ref{fork-sliding-complexes-involved}, and $f_k$, $g_k$ are defined in Figure \ref{fork-special-decomp-iv-fig5}. We have the following lemma.

\begin{lemma}\label{relating-D-0-1-lemma}
In Figure \ref{relating-D-0-1-fig3}, $\tilde{d}_k^+ \approx g_{k-1} \circ d_k^+\circ f_k$ and $\tilde{d}_{k-1}^- \approx g_k \circ d_{k-1}^- \circ f_{k-1}$. That is, the diagram in Figure \ref{relating-D-0-1-fig3} commutes in both directions up to homotopy and scaling.
\end{lemma}

\begin{proof}
Denote by $h^{(k)}$, $\overline{h}^{(k)}$, $(\chi^1\otimes\chi^1)^{(k)}$ and $(\chi^0\otimes\chi^0)^{(k)}$ the morphisms induced by the local changes  of MOY graphs in Figures \ref{relating-D-0-1-fig4} and \ref{relating-D-0-1-fig5}. By the definitions of $f_k$, $g_k$, $\hat{f}_k$ and $\hat{g}_k$ in Figure \ref{fork-special-decomp-iv-fig5}, we know that $f_k \approx \hat{f}_k \circ h^{(k)}$ and $g_k \approx \overline{h}^{(k)} \circ \hat{g}_k$.

\begin{figure}[ht]
$
\xymatrix{
\input{tilde-gamma-k} \ar@<11ex>[rr]^{h^{(k)}} && \input{gamma-k3} \ar@<-9ex>[ll]^{\overline{h}^{(k)}}
} 
$
\caption{}\label{relating-D-0-1-fig4} 

\end{figure}

\begin{figure}[ht]
$
\xymatrix{
\input{tilde-gamma-k-bubble} \ar@<8ex>[rr]^{(\chi^1\otimes\chi^1)^{(k)}} && \input{tilde-gamma-k-divided} \ar@<-6ex>[ll]^{(\chi^0\otimes\chi^0)^{(k)}}
} 
$
\caption{}\label{relating-D-0-1-fig5} 

\end{figure}

By the definitions of $\hat{f}_k$, $\hat{g}_k$ and $d_k^+$, using the morphisms in Figures \ref{relating-D-0-1-fig1} and \ref{relating-D-0-1-fig2}, we have 
\begin{eqnarray*}
g_{k-1} \circ d_k^+\circ f_k & \approx & \overline{h}^{(k-1)} \circ \hat{g}_{k-1} \circ d_k^+ \circ \hat{f}_k \circ h^{(k)} \\
& \approx & \overline{h}^{(k-1)} \circ (\overline{\varphi}_6 \circ \chi^1_\triangle) \circ (\overline{\phi}_3 \circ \overline{h}_3 \circ (\chi^1\otimes \chi^1) \circ \phi_1) \circ \hat{f}_k \circ h^{(k)} \\
& \approx & \overline{h}^{(k-1)} \circ (\overline{\varphi}_6 \circ \chi^1_\triangle \circ \overline{\phi}_3 \circ \overline{h}_3) \circ ((\chi^1\otimes \chi^1) \circ \phi_1) \circ \hat{f}_k \circ h^{(k)} \\
_{(\text{by Lemma \ref{relating-D-0-1-claim2}})}& \approx & \overline{h}^{(k-1)} \circ (\overline{\varphi}_7 \circ \chi^1_\dag \circ \overline{\varphi}_5 \circ \chi^1_\Box) \circ ((\chi^1\otimes \chi^1) \circ \phi_1) \circ \hat{f}_k \circ h^{(k)} \\
_{(\text{by Lemma \ref{relating-D-0-1-claim1}})}& \approx & \overline{h}^{(k-1)} \circ (\overline{\varphi}_7 \circ \chi^1_\dag \circ \overline{\varphi}_5 \circ \chi^1_\Box) \circ \hat{f}_k \circ ((\chi^1\otimes \chi^1) \circ \phi_1) \circ h^{(k)}. \\
\end{eqnarray*}
Note that $\overline{\varphi}_5 \circ \chi^1_\Box \approx \hat{g}_k$, $\hat{g}_k \circ \hat{f}_k \approx \id$ and $\phi_1 \circ h^{(k)} \approx h^{(k)} \circ \phi_1$. Putting these together, we get 
\begin{eqnarray*}
g_{k-1} \circ d_k^+\circ f_k & \approx & \overline{h}^{(k-1)} \circ (\overline{\varphi}_7 \circ \chi^1_\dag) \circ (\overline{\varphi}_5 \circ \chi^1_\Box) \circ \hat{f}_k \circ (\chi^1\otimes \chi^1) \circ \phi_1 \circ h^{(k)} \\
& \approx & \overline{h}^{(k-1)} \circ \overline{\varphi}_7 \circ \chi^1_\dag \circ (\chi^1\otimes \chi^1) \circ h^{(k)} \circ \phi_1
\end{eqnarray*}
By Lemma \ref{chi-commute-chi-chi}, we know that
\[
\chi^1_\dag \circ (\chi^1\otimes \chi^1) \circ h^{(k)} \approx h^{(k-1)} \circ (\chi^1\otimes\chi^1)^{(k)}.
\]
Also, it is easy to see that $\overline{h}^{(k-1)} \circ \overline{\varphi}_7 \approx \overline{\varphi}_7 \circ\overline{h}^{(k-1)}$. So 
\begin{eqnarray*}
g_{k-1} \circ d_k^+\circ f_k & \approx & \overline{h}^{(k-1)} \circ \overline{\varphi}_7 \circ \chi^1_\dag \circ (\chi^1\otimes \chi^1) \circ h^{(k)} \circ \phi_1 \\
& \approx & \overline{\varphi}_7 \circ \overline{h}^{(k-1)} \circ h^{(k-1)} \circ (\chi^1\otimes\chi^1)^{(k)} \circ \phi_1 \\
& \approx & \overline{\varphi}_7 \circ (\chi^1\otimes\chi^1)^{(k)} \circ \phi_1  \\
& \approx & \tilde{d}_k^+.
\end{eqnarray*}

Similarly, using $\overline{\phi}_1 \circ \overline{h}^{(k)} \approx \overline{h}^{(k)} \circ \overline{\phi}_1$ and $\varphi_7 \circ h^{(k-1)} \approx h^{(k-1)} \circ \varphi_7$, we get
\begin{eqnarray*}
g_k \circ d_{k-1}^- \circ f_{k-1} & \approx & \overline{h}^{(k)} \circ \hat{g}_k \circ d_{k-1}^- \circ \hat{f}_{k-1} \circ h^{(k-1)} \\
& \approx & \overline{h}^{(k)} \circ \hat{g}_k \circ (\overline{\phi}_1 \circ (\chi^0\otimes\chi^0) \circ h_3 \circ \phi_3)\circ (\chi^0_\triangle \circ \varphi_6) \circ h^{(k-1)} \\
_{(\text{by Lemma \ref{relating-D-0-1-claim1}})} & \approx & \overline{h}^{(k)} \circ (\overline{\phi}_1 \circ (\chi^0\otimes\chi^0)) \circ \hat{g}_k \circ (h_3 \circ \phi_3\circ \chi^0_\triangle \circ \varphi_6) \circ h^{(k-1)} \\
_{(\text{by Lemma \ref{relating-D-0-1-claim2}})} & \approx & \overline{h}^{(k)} \circ (\overline{\phi}_1 \circ (\chi^0\otimes\chi^0)) \circ \hat{g}_k \circ (\chi^0_\Box \circ \varphi_5 \circ \chi^0_\dag \circ \varphi_7) \circ h^{(k-1)} \\
_{(\text{since } \hat{g}_k \circ (\chi^0_\Box \circ \varphi_5) \approx \hat{g}_k \circ \hat{f}_k \approx \id )} & \approx & \overline{\phi}_1 \circ \overline{h}^{(k)} \circ (\chi^0\otimes\chi^0) \circ \chi^0_\dag \circ \varphi_7 \circ h^{(k-1)} \\
& \approx & \overline{\phi}_1 \circ \overline{h}^{(k)} \circ (\chi^0\otimes\chi^0) \circ \chi^0_\dag \circ h^{(k-1)} \circ \varphi_7 \\
_{(\text{by Lemma \ref{chi-commute-chi-chi}})} & \approx & \overline{\phi}_1 \circ \overline{h}^{(k)} \circ h^{(k)} \circ (\chi^0\otimes\chi^0)^{(k)} \circ \varphi_7 \\
& \approx & \overline{\phi}_1 \circ (\chi^0\otimes\chi^0)^{(k)} \circ \varphi_7 \\
& \approx & \tilde{d}_{k-1}^-.
\end{eqnarray*}
\end{proof}

\subsection{Decomposing $C(\Gamma_{m,1})=C(\Gamma_m')$}

\begin{figure}[ht]
$
\xymatrix{
\input{gamma-m-prime} & \input{gamma-m-1-double-prime} & \input{tilde-gamma-m+1-prime} & \input{tilde-gamma-m+1}  
} 
$
\caption{}\label{gamma-m-1-fig} 

\end{figure}

Note that the MOY graphs $\Gamma_{m,1}$ and $\Gamma_m'$ are identical. Consider the MOY graphs in Figure \ref{gamma-m-1-fig}. By Corollary \ref{contract-expand}, $C(\Gamma_{m}'') \simeq C(\widetilde{\Gamma}_{m+1})$. By Decomposition (V) (Theorem \ref{decomp-V}), $C(\Gamma_m') \simeq C(\Gamma_{m-1}'') \oplus C(\Gamma_{m}'')$. So
\begin{equation}\label{gamma-m-1-decomp-V}
C(\Gamma_{m,1}) \simeq C(\Gamma_{m-1}'') \oplus C(\widetilde{\Gamma}_{m+1}).
\end{equation}

\begin{figure}[ht]
$
\xymatrix{
\input{gamma-m-1-prime}
} 
$
\caption{}\label{gamma-m-1-prime-fig} 

\end{figure}

Recall that $\Gamma_{m-1}'$ is the MOY graph in Figure \ref{gamma-m-1-prime-fig}. We have the following lemma.

\begin{lemma}\label{hmf-tilde-gamma-prime-m}
\[
\Hom_\hmf(C(\widetilde{\Gamma}_{m+1}),C(\Gamma_{m-1}')) \cong \Hom_\hmf(C(\Gamma_{m-1}'), C(\widetilde{\Gamma}_{m+1})) \cong 0.
\]
\end{lemma}

\begin{figure}[ht]
$
\xymatrix{
\input{tilde-gamma-prime-m}
} 
$
\caption{}\label{hmf-tilde-gamma-prime-m-fig}

\end{figure}

\begin{proof}
Let $\Gamma$ be the MOY graph in Figure \ref{hmf-tilde-gamma-prime-m-fig}. Recall that $C(\widetilde{\Gamma}_{m+1}) \simeq C(\Gamma_{m}'')$. So
\begin{eqnarray*}
&& \Hom_\HMF(C(\widetilde{\Gamma}_{m+1}),C(\Gamma_{m-1}')) \cong  \Hom_\HMF(C(\Gamma_{m}''),C(\Gamma_{m-1}')), \\
& \cong & H(\Gamma)\left\langle m+n+1\right\rangle \{ q^{(m+n+1)(N-1)-m^2-n^2+n} \} \\
& \cong & C(\emptyset) \{ [n+1] \qb{n+m-1}{m} \qb{m+1}{2} \qb{m+n+1}{n} \qb{N}{m+n+1}  q^{(m+n+1)(N-1)-m^2-n^2+n} \}.
\end{eqnarray*}
One can check that the lowest non-vanishing quantum grading of the above space is $2$. So $\Hom_\hmf(C(\widetilde{\Gamma}_{m+1}),C(\Gamma_{m-1}'')) \cong 0$.

Denote by $\overline{\Gamma}$ the MOY graph obtained by reversing the orientation of $\Gamma$. By Decomposition (V) (Theorem \ref{decomp-V}), we have $C(\Gamma_{m-1}') \simeq C(\Gamma_{m-1}'') \oplus C(\Gamma_{m-2}'')$. By Lemma \ref{complex-computing-gamma-HMF-lemma}, we have that 
\[
\Hom_\HMF(C(\Gamma_{k}''), C(\Gamma_{m}'')) \cong H(C(\Gamma_{m}'') \otimes C(\overline{\Gamma}_{k}'')) \left\langle m+n+1\right\rangle \{ q^{(m+n+1)(N-1)-m^2-n^2+n} \},
\]
where $\overline{\Gamma}_{k}''$ is $\Gamma_{k}''$ with reverse orientation, and the tensor is over the ring of partial symmetric polynomials in the alphabets marking the end points. Therefore,
\begin{eqnarray*}
&& \Hom_\HMF(C(\Gamma_{m-1}'), C(\widetilde{\Gamma}_{m+1})) \cong \Hom_\HMF(C(\Gamma_{m-1}'), C(\Gamma_{m}''))\\
&\cong & \Hom_\HMF(C(\Gamma_{m-1}''), C(\Gamma_{m}'')) \oplus \Hom_\HMF(C(\Gamma_{m-2}''), C(\Gamma_{m}'')) \\
& \cong & (H(C(\Gamma_{m}'') \otimes C(\overline{\Gamma}_{m-1}'')) \oplus H(C(\Gamma_{m}'') \otimes C(\overline{\Gamma}_{m-2}''))) \left\langle m+n+1\right\rangle \{ q^{(m+n+1)(N-1)-m^2-n^2+n} \} \\
& \cong & H(C(\Gamma_{m}'') \otimes C(\overline{\Gamma}_{m-1}')) \left\langle m+n+1\right\rangle \{ q^{(m+n+1)(N-1)-m^2-n^2+n} \} \\
& \cong & H(\overline{\Gamma}) \left\langle m+n+1\right\rangle \{ q^{(m+n+1)(N-1)-m^2-n^2+n} \} \\
& \cong & C(\emptyset) \{ [n+1] \qb{n+m-1}{m} \qb{m+1}{2} \qb{m+n+1}{n} \qb{N}{m+n+1}  q^{(m+n+1)(N-1)-m^2-n^2+n} \}.
\end{eqnarray*}
where $\overline{\Gamma}_{m-1}'$ is $\Gamma_{m-1}'$ with orientation reversed, and the tensor is over the ring of partial symmetric polynomials in the alphabets marking the end points. So the lowest non-vanishing quantum grading of $\Hom_\HMF(C(\Gamma_{m-1}'), C(\widetilde{\Gamma}_{m+1}))$ is also $2$. Thus, $\Hom_\hmf(C(\Gamma_{m-1}'), C(\widetilde{\Gamma}_{m+1})) \cong 0$.
\end{proof}

\begin{corollary}\label{hmf-tilde-gamma-double-prime-m}
\[
\Hom_\hmf(C(\widetilde{\Gamma}_{m+1}),C(\Gamma_{m-1}'')) \cong \Hom_\hmf(C(\Gamma_{m-1}''), C(\widetilde{\Gamma}_{m+1})) \cong 0.
\]
\end{corollary}

\begin{proof}
By Decomposition (V) (Theorem \ref{decomp-V}), $C(\Gamma_{m-1}') \simeq C(\Gamma_{m-1}'') \oplus C(\Gamma_{m-2}'')$. So $\Hom_\hmf(C(\widetilde{\Gamma}_{m+1}),C(\Gamma_{m-1}''))$ (resp. $\Hom_\hmf(C(\Gamma_{m-1}''), C(\widetilde{\Gamma}_{m+1}))$) is a subspace of $\Hom_\hmf(C(\widetilde{\Gamma}_{m+1}),C(\Gamma_{m-1}'))$ (resp. $\Hom_\hmf(C(\Gamma_{m-1}'), C(\widetilde{\Gamma}_{m+1}))$.) Then the corollary follows from Lemma \ref{hmf-tilde-gamma-prime-m}.
\end{proof}

\begin{lemma}\label{hmf-tilde-gamma-m} 
$\Hom_\hmf(C(\widetilde{\Gamma}_{m+1}),C(\widetilde{\Gamma}_{m+1})) \cong \C$.
\end{lemma}

\begin{figure}[ht]
$
\xymatrix{
\input{tilde-gamma-m+1-self-close}  
} 
$
\caption{}\label{hmf-tilde-gamma-m-fig} 

\end{figure}

\begin{proof}
Let $\Gamma$ be the MOY graph in Figure \ref{hmf-tilde-gamma-m-fig}. Then
\begin{eqnarray*}
&& \Hom_\HMF(C(\widetilde{\Gamma}_{m+1}),C(\widetilde{\Gamma}_{m+1})) \\
& \cong & H(\Gamma)\left\langle m+n+1 \right\rangle \{ q^{(m+n+1)(N-1)-m^2-n^2+n} \} \\
& \cong & C(\emptyset) \{ [m+1] \qb{m+n+1}{n} \qb{m+n+1}{n} \qb{N}{m+n+1} q^{(m+n+1)(N-1)-m^2-n^2+n} \}.
\end{eqnarray*}
It is easy to check that the above space is supported on $\zed_2$-degree $0$. Its lowest non-vanishing quantum grading is $0$. And its subspace of homogeneous elements of quantum degree $0$ is $1$-dimensional. Thus, $\Hom_\hmf(C(\widetilde{\Gamma}_{m+1}),C(\widetilde{\Gamma}_{m+1})) \cong \C$.
\end{proof}

\begin{corollary}\label{hmf-tilde-gamma-gamma-m1} 
\[
\Hom_\hmf(C(\widetilde{\Gamma}_{m+1}),C(\Gamma_{m,1})) \cong \Hom_\hmf(C(\Gamma_{m,1}),C(\widetilde{\Gamma}_{m+1})) \cong \C.
\]
\end{corollary}

\begin{proof}
This follows easily from decomposition \eqref{gamma-m-1-decomp-V}, Corollary \ref{hmf-tilde-gamma-double-prime-m} and Lemma \ref{hmf-tilde-gamma-m}.
\end{proof}

\begin{figure}[ht]
$
\xymatrix{
\input{gamma-m-prime-marked} \ar@<10ex>[rr]^{\tilde{p}} \ar@<1ex>[d]^<<<<<{\chi^0 \otimes \chi^0} & & \input{tilde-gamma-m+1} \ar@<-8ex>[ll]^{\tilde{\jmath}} \ar@<1ex>[d]^<<<<<{h} \\
\input{tilde-gamma-m+1-bubble} \ar@<1ex>[u]^>>>>>{\chi^1 \otimes \chi^1} \ar@<10ex>[rr]^{\overline{\phi}} && \input{tilde-gamma-m+1-prime} \ar@<-8ex>[ll]^{\phi} \ar@<1ex>[u]^>>>>>{\overline{h}}
} 
$
\caption{}\label{gamma-m-1-p-j-def} 

\end{figure}

Consider the diagram in Figure \ref{gamma-m-1-p-j-def}, where 
\begin{eqnarray*}
\tilde{p} & := & \overline{h} \circ \overline{\phi} \circ (\chi^0 \otimes \chi^0), \\
\tilde{\jmath} & := & (\chi^1 \otimes \chi^1) \circ \phi \circ h,
\end{eqnarray*}
and morphisms on the right hand side are induced by the apparent local changes of the MOY graphs. 

\begin{lemma}\label{gamma-m-1-p-j-tilde}
Up to homotopy and scaling, $\tilde{\jmath}$ is the inclusion of $C(\widetilde{\Gamma}_{m+1})$ into $C(\Gamma_{m,1})$ in decomposition \eqref{gamma-m-1-decomp-V}, and $\tilde{p}$ is the projection of $C(\Gamma_{m,1})$ onto $C(\widetilde{\Gamma}_{m+1})$ in decomposition \eqref{gamma-m-1-decomp-V}.
\end{lemma}

\begin{proof}
From Corollary \ref{hmf-tilde-gamma-gamma-m1}, one can see that $\Hom_\hmf(C(\widetilde{\Gamma}_{m+1}),C(\Gamma_{m,1}))$ (resp. $\Hom_\hmf(C(\Gamma_{m,1}),C(\widetilde{\Gamma}_{m+1}))$) is $1$-dimensional and spanned by the inclusion $C(\widetilde{\Gamma}_{m+1})\rightarrow C(\Gamma_{m,1})$ (resp. the projection $C(\Gamma_{m,1})\rightarrow C(\widetilde{\Gamma}_{m+1})$) in decomposition \eqref{gamma-m-1-decomp-V}. Note that $\tilde{\jmath}$ and $\tilde{p}$ are both homogeneous morphisms of $\zed_2$-degree $0$ and quantum degree $0$. To prove the lemma, we only need to show that $\tilde{\jmath}$ and $\tilde{p}$ are not homotopic to $0$. But, by Corollary \ref{general-chi-maps-def} and Lemma \ref{phibar-compose-phi}, 
\begin{eqnarray*}
\tilde{p} \circ \tilde{\jmath} & \approx & \overline{h} \circ \overline{\phi} \circ (\chi^0 \otimes \chi^0) \circ (\chi^1 \otimes \chi^1) \circ \phi \circ h \\
& \approx & \overline{h} \circ \overline{\phi} \circ \mathfrak{m}((\sum_{i=0}^n (-r)^iY_{n-i}) \cdot (\sum_{i=0}^m (-r)^iX_{m-i})) \circ \phi \circ h \\
& \approx & \overline{h} \circ \overline{\phi} \circ \mathfrak{m}((-r)^{m+n}) \circ \phi \circ h \approx \id_{C(\widetilde{\Gamma}_{m+1})}.
\end{eqnarray*}
This shows that $\tilde{\jmath}$ and $\tilde{p}$ are not homotopic to $0$ and completes the proof.
\end{proof}

\begin{figure}[ht]
$
\xymatrix{
\input{tilde-gamma-m+1} \ar@<1ex>[d]^<<<<<{\tilde{\jmath}} \ar@<10ex>[rr]^{\tilde{d}_{m+1}^+} && \input{tilde-gamma-m} \ar@<-8ex>[ll]^{\tilde{d}_{m}^-} \ar@<1ex>[d]^<<<<<{h^{(m)}} \\
\input{gamma-m-prime} \ar@<1ex>[u]^>>>>>{\tilde{p}} \ar@<10ex>[rr]^{\chi^1} && \input{gamma-m3} \ar@<-8ex>[ll]^{\chi^0} \ar@<1ex>[u]^>>>>>{\overline{h}^{(m)}}
} 
$
\caption{}\label{gamma-m-1-p-j-differential-fig} 

\end{figure}

Consider the diagram in Figure \ref{gamma-m-1-p-j-differential-fig}, where $\tilde{d}_{m+1}^+$ (resp. $\tilde{d}_{m}^-$) is the differential map of the chain complex $\hat{C}(D_{10}^+)$ (resp. $\hat{C}(D_{10}^-)$) at homological degree $0$ (resp. $-1$)\footnote{See Subsection \ref{fork-sliding-complexes-involved}, especially the chain complexes $\hat{C}(D_{10}^\pm)$ in \eqref{complex-D-10-+} and \eqref{complex-D-10-+}.}, and $\chi^0$, $\chi^1$, $h^{(m)}$, $\overline{h}^{(m)}$ are induced by the apparent local changes of MOY graphs. We have the following lemma.

\begin{lemma}\label{gamma-m-1-p-j-differential} 
$\tilde{d}_{m+1}^+ \approx \overline{h}^{(m)} \circ \chi^1 \circ \tilde{\jmath}$ and $\tilde{d}_{m}^- \approx \tilde{p} \circ \chi^0 \circ h^{(m)}$. That is, the diagram in Figure \ref{gamma-m-1-p-j-differential-fig} commutes in both directions up to homotopy and scaling.
\end{lemma}

\begin{proof}
This follows easily from the definitions of $\tilde{d}_{m+1}^+$, $\tilde{d}_{m}^-$, $\tilde{\jmath}$, $\tilde{p}$ and Lemma \ref{chi-commute-chi-chi}.
\end{proof}

\begin{figure}[ht]
$
\xymatrix{
\input{gamma-m-1-double-prime} \ar@<1ex>[rd]^>>>>>>>>>>>>>>>{\jmath''} \ar@<10ex>[rr]^{J_{m-1,m-1}} &  & \input{gamma-m-1-prime} \ar@<1ex>[ld]^>>>>>>>>>>>>>>>{\delta_{m-1}^-} \ar@<-8ex>[ll]^{P_{m-1,m-1}} \\
& \input{gamma-m-prime} \ar@<1ex>[lu]^<<<<<<<<<<<<<<<{p''} \ar@<1ex>[ru]^<<<<<<<<<<<<<<<{\delta_m^+} &
} 
$
\caption{}\label{gamma-m-1-p-j-prime-dif-fig} 

\end{figure}

Denote by $\jmath'':C(\Gamma_{m-1}'') \rightarrow C(\Gamma_{m,1})$ and $p'':C(\Gamma_{m,1}) \rightarrow C(\Gamma_{m-1}'')$ the inclusion and projection morphisms of the component $C(\Gamma_{m-1}'')$ in decomposition \eqref{gamma-m-1-decomp-V}. Consider the diagram in Figure \ref{gamma-m-1-p-j-prime-dif-fig}, where $\delta_m^+$, $\delta_{m-1}^-$, $J_{m-1,m-1}$ and $P_{m-1,m-1}$ are defined in Definition \ref{trivial-complex-differential-def}. We have the following lemma.

\begin{lemma}\label{gamma-m-1-p-j-prime-dif}
$\delta_m^+ \circ \jmath'' \approx J_{m-1,m-1}$ and $p'' \circ \delta_{m-1}^- \approx P_{m-1,m-1}$. That is, the diagram in Figure \ref{gamma-m-1-p-j-prime-dif-fig} commutes in both directions up to homotopy and scaling.
\end{lemma}

\begin{proof}
Using Lemmas \ref{complex-computing-gamma-HMF-lemma} and \ref{decomp-V-special-2}, one can check that 
\[
\Hom_\hmf(C(\Gamma_{m-1}''),C(\Gamma_{m-1}')) \cong \Hom_\hmf(C(\Gamma_{m-1}'),C(\Gamma_{m-1}'')) \cong \C.
\]
Recall that $J_{m-1,m-1}$ and  $P_{m-1,m-1}$ are both homogeneous morphisms of $\zed_2$-degree $0$ and quantum degree $0$, and $P_{m-1,m-1} \circ J_{m-1,m-1} \approx \id_{C(\Gamma_{m-1}'')}$. So $J_{m-1,m-1}$ and $P_{m-1,m-1}$ span these $1$-dimensional spaces. Note that $\delta_m^+ \circ \jmath''$ and  $p'' \circ \delta_{m-1}^-$ are also homogeneous morphisms of $\zed_2$-degree $0$ and quantum degree $0$. To prove the lemma, we only need to show that $\delta_m^+ \circ \jmath''$ and  $p'' \circ \delta_{m-1}^-$ are not homotopic to $0$. But, by their definitions, we know that $p''\circ \delta_{m-1}^-\circ\delta_m^+ \circ \jmath'' \approx \id_{C(\Gamma_{m-1}'')}$. So $\delta_m^+ \circ \jmath''$ and  $p'' \circ \delta_{m-1}^-$ are homotopically non-trivial.
\end{proof}

\subsection{Proof of Proposition \ref{fork-sliding-invariance-special}} In this subsection, we prove \eqref{fork-sliding-invariance-special-eq}, that is, 
\[
\hat{C}(D_{10}^\pm) \simeq \hat{C}(D_{11}^\pm) \text{ if } l=1.
\] 
The proof of the rest of Proposition \ref{fork-sliding-invariance-special} is very similar and left to the reader. We prove \eqref{fork-sliding-invariance-special-eq} by simplifying $\hat{C}(D_{11}^\pm)$ and reducing it to $\hat{C}(D_{10}^\pm)$. To do this, we need to use the Gaussian Elimination Lemma \cite[Lemma 4.2]{Bar-fast}.

\begin{lemma}\cite[Lemma 4.2]{Bar-fast}\label{gaussian-elimination}
Let $\mathcal{C}$ be an additive category, and
\[
\mathtt{I}=``\cdots\rightarrow C\xrightarrow{\left(%
\begin{array}{c}
  \alpha\\
  \beta \\
\end{array}%
\right)}
\left.%
\begin{array}{c}
  A\\
  \oplus \\
  D
\end{array}%
\right.
\xrightarrow{
\left(%
\begin{array}{cc}
  \phi & \delta\\
  \gamma & \varepsilon \\
\end{array}%
\right)}
\left.%
\begin{array}{c}
  B\\
  \oplus \\
  E
\end{array}%
\right.
\xrightarrow{
\left(%
\begin{array}{cc}
  \mu & \nu\\
\end{array}%
\right)} F \rightarrow \cdots"
\]
an object of $\ch(\mathcal{C})$, that is, a bounded chain complex over $\mathcal{C}$. Assume that $A\xrightarrow{\phi} B$ is an isomorphism in $\mathcal{C}$ with inverse $\phi^{-1}$. Then $\mathtt{I}$ is homotopic to (that is, isomorphic in $\hch(\mathcal{C})$ to)
\[
\mathtt{II}=
``\cdots\rightarrow C \xrightarrow{\beta} D
\xrightarrow{\varepsilon-\gamma\phi^{-1}\delta} E\xrightarrow{\nu} F \rightarrow \cdots".
\]
In particular, if $\delta$ or $\gamma$ is $0$, then $\mathtt{I}$ is homotopic to 
\[
\mathtt{II}=
``\cdots\rightarrow C \xrightarrow{\beta} D
\xrightarrow{\varepsilon} E\xrightarrow{\nu} F \rightarrow \cdots".
\]
\end{lemma}

\begin{proof}
Consider the chain complex
\[
\mathtt{I}'=
``\cdots\rightarrow C\xrightarrow{\left(%
\begin{array}{c}
  0\\
  \beta \\
\end{array}%
\right)}
\left.%
\begin{array}{c}
  A\\
  \oplus \\
  D
\end{array}%
\right.
\xrightarrow{
\left(%
\begin{array}{cc}
  \phi & 0\\
  0 & \varepsilon-\gamma\phi^{-1}\delta \\
\end{array}%
\right)}
\left.%
\begin{array}{c}
  B\\
  \oplus \\
  E
\end{array}%
\right.
\xrightarrow{
\left(%
\begin{array}{cc}
  0 & \nu\\
\end{array}%
\right)} F \rightarrow \cdots".
\]
Define $f:\mathtt{I}\rightarrow\mathtt{I}'$ and $g:\mathtt{I}'\rightarrow\mathtt{I}$ by
\[
\begin{CD}
\cdots @>>> C @>{\left(%
\begin{array}{c}
  \alpha\\
  \beta \\
\end{array}%
\right)}>> \left.%
\begin{array}{c}
  A\\
  \oplus \\
  D
\end{array}%
\right.  
@>{\left(%
\begin{array}{cc}
  \phi & \delta\\
  \gamma & \varepsilon \\
\end{array}%
\right)}>> 
\left.%
\begin{array}{c}
  B\\
  \oplus \\
  E
\end{array}%
\right.
@>{
\left(%
\begin{array}{cc}
  \mu & \nu\\
\end{array}%
\right)}>> F @>>> \cdots \\
@VV{\id}V @VV{\id}V @VV{\left(%
\begin{array}{cc}
  \id & \phi^{-1}\delta\\
  0 & \id \\
\end{array}%
\right)}V @VV{\left(%
\begin{array}{cc}
  \id & 0\\
  -\gamma\phi^{-1} & \id \\
\end{array}%
\right)}V @VV{\id}V @VV{\id}V\\
\cdots @>>> C @>{\left(%
\begin{array}{c}
  0\\
  \beta \\
\end{array}%
\right)}>> \left.%
\begin{array}{c}
  A\\
  \oplus \\
  D
\end{array}%
\right. @>{\left(%
\begin{array}{cc}
  \phi & 0\\
  0 & \varepsilon-\gamma\phi^{-1}\delta \\
\end{array}%
\right)}>> \left.%
\begin{array}{c}
  B\\
  \oplus \\
  E
\end{array}%
\right.
@>{
\left(%
\begin{array}{cc}
  0 & \nu\\
\end{array}%
\right)}>> F @>>> \cdots \\
@VV{\id}V @VV{\id}V @VV{\left(%
\begin{array}{cc}
  \id & -\phi^{-1}\delta\\
  0 & \id \\
\end{array}%
\right)}V @VV{\left(%
\begin{array}{cc}
  \id & 0\\
  \gamma\phi^{-1} & \id \\
\end{array}%
\right)}V @VV{\id}V @VV{\id}V \\
\cdots @>>> C @>{\left(%
\begin{array}{c}
  \alpha\\
  \beta \\
\end{array}%
\right)}>> \left.%
\begin{array}{c}
  A\\
  \oplus \\
  D
\end{array}%
\right.  
@>{\left(%
\begin{array}{cc}
  \phi & \delta\\
  \gamma & \varepsilon \\
\end{array}%
\right)}>> 
\left.%
\begin{array}{c}
  B\\
  \oplus \\
  E
\end{array}%
\right.
@>{
\left(%
\begin{array}{cc}
  \mu & \nu\\
\end{array}%
\right)}>> F @>>> \cdots \\
\end{CD}
\]
It is easy to check that $f$ and $g$ are isomorphisms in $\ch(\mathcal{C})$. Thus,
\[
\mathtt{I} \cong \mathtt{I}' \cong \mathtt{II} \oplus ``0\rightarrow A \xrightarrow{\phi} B \rightarrow 0". 
\]
But $0\rightarrow A \xrightarrow{\phi} B \rightarrow 0$ is homotopic to $0$ since $\phi$ is an isomorphism in $\mathcal{C}$. So $\mathtt{I} \simeq \mathtt{II}$.
\end{proof}

\begin{figure}[ht]
$
\xymatrix{
\input{gamma-k-prime} && \input{gamma-k-1-prime}
} 
$
\caption{}\label{hmf-gamma-prime-k-k-1-fig} 

\end{figure}

\begin{lemma}\label{hmf-gamma-prime-k-k-1}
\begin{eqnarray*}
\Hom_\hmf (C(\Gamma_k')\{[m-k]q^{k-1-m}\}, C(\Gamma_{k-1}')) & \cong & 0, \\
\Hom_\hmf (C(\Gamma_{k-1}'), C(\Gamma_k')\{[m-k]q^{m+1-k}\}) & \cong & 0.
\end{eqnarray*}
\end{lemma}

\begin{proof}
By Decomposition (V) (more precisely, Lemma \ref{decomp-V-special-2}), we have that 
\[
C(\Gamma_k') \simeq C(\Gamma_k'') \oplus C(\Gamma_{k-1}'').
\]
Similar to Lemma \ref{trivial-complex-lemma-1}, one can check that the lowest non-vanishing quantum grading of $\Hom_\HMF (C(\Gamma_j''), C(\Gamma_k'))$ is $(j-k)(j-k+1)$. So the lowest non-vanishing quantum grading of $\Hom_\HMF (C(\Gamma_k'), C(\Gamma_{k-1}'))$ and $\Hom_\HMF(C(\Gamma_{k-1}'), C(\Gamma_k'))$ is $0$. Note that 
\begin{eqnarray*}
&& \Hom_\HMF (C(\Gamma_k')\{[m-k]q^{k-1-m}\}, C(\Gamma_{k-1}')) \\
& \cong & \Hom_\HMF (C(\Gamma_k'), C(\Gamma_{k-1}'))\{[m-k]q^{m+1-k}\} \\
& \cong & \bigoplus_{j=0}^{m-1-k} \Hom_\HMF (C(\Gamma_k'), C(\Gamma_{k-1}')), \{q^{2+2j}\}
\end{eqnarray*}
\begin{eqnarray*}
&& \Hom_\HMF (C(\Gamma_{k-1}'), C(\Gamma_k')\{[m-k]q^{m+1-k}\}) \\
& \cong & \Hom_\HMF (C(\Gamma_{k-1}'), C(\Gamma_k'))\{[m-k]q^{m+1-k}\} \\
& \cong & \bigoplus_{j=0}^{m-1-k} \Hom_\HMF (C(\Gamma_{k-1}'), C(\Gamma_k')) \{q^{2+2j}\},
\end{eqnarray*}
and the lowest non-vanishing quantum grading of the right hand side is $2$ in both cases. So
\begin{eqnarray*}
\Hom_\hmf (C(\Gamma_k')\{[m-k]q^{k-1-m}\}, C(\Gamma_{k-1}')) & \cong & 0, \\
\Hom_\hmf (C(\Gamma_{k-1}'), C(\Gamma_k')\{[m-k]q^{m+1-k}\}) & \cong & 0.
\end{eqnarray*}
\end{proof}

We are now ready to prove \eqref{fork-sliding-invariance-special-eq}. We prove $\hat{C}(D_{10}^+) \simeq \hat{C}(D_{11}^+)$ first and then $\hat{C}(D_{10}^-) \simeq \hat{C}(D_{11}^-)$. 

\begin{proof}[Proof of $\hat{C}(D_{10}^+) \simeq \hat{C}(D_{11}^+)$ when $l=1$]
Recall that the chain complex $\hat{C}(D_{1,1}^+)$ is 

{\tiny
\[
0 \rightarrow C(\Gamma_{m,1}) \xrightarrow{\mathfrak{d}_m^+} \left.%
\begin{array}{c}
C(\Gamma_{m,0}) \{q^{-1}\}\\
\oplus \\
C(\Gamma_{m-1,1})\{q^{-1}\}
\end{array}%
\right. 
\xrightarrow{\mathfrak{d}_{m-1}^+} \cdots \xrightarrow{\mathfrak{d}_{k+1}^+} \left.%
\begin{array}{c}
C(\Gamma_{k+1,0}) \{q^{k-m}\}\\
\oplus \\
C(\Gamma_{k,1})\{q^{k-m}\}
\end{array}%
\right. 
\xrightarrow{\mathfrak{d}_{k}^+} \cdots\xrightarrow{\mathfrak{d}_{k_0}^+} C(\Gamma_{k_0,0}) \{q^{k_0-1-m}\} \rightarrow 0,
\]
}

\noindent where $k_0 = \max \{m-n,0\}$ as above and
\begin{eqnarray*}
\mathfrak{d}_m^+ & = & \left(%
\begin{array}{c}
\chi^1\\
-d_m^+
\end{array}%
\right),\\
\mathfrak{d}_k^+ & = & \left(%
\begin{array}{cc}
d_{k+1}^+ & \chi^1\\
0 & -d_k^+
\end{array}%
\right) ~\text{ for } k_0<k<m, \\
\mathfrak{d}_{k_0}^+ & = & \left(%
\begin{array}{cc}
d_{k_0 +1}^+, & \chi^1\\
\end{array}%
\right).
\end{eqnarray*}

From Decomposition (IV) (more precisely, \eqref{fork-special-decomp-iv-eq}), we have 
\[
C(\Gamma_{k,0}) \simeq C(\widetilde{\Gamma}_k) \oplus C(\Gamma_k')\{[m-k]\}.
\] 
By Corollary \ref{contract-expand} and Decomposition (II) (Theorem \ref{decomp-II}), we have 
\[
C(\Gamma_{k,1}) \simeq C(\Gamma_k')\{[m+1-k]\} \cong C(\Gamma_k')\{q^{m-k}\} \oplus C(\Gamma_k')\{[m-k]q^{-1}\}.
\]
Therefore,
\[
\left.%
\begin{array}{c}
C(\Gamma_{k+1,0}) \{q^{k-m}\}\\
\oplus \\
C(\Gamma_{k,1})\{q^{k-m}\}
\end{array}%
\right. \simeq 
\left.%
\begin{array}{c}
C(\widetilde{\Gamma}_{k+1}) \{q^{k-m}\}\\
\oplus \\
C(\Gamma_{k+1}')\{[m-k-1]q^{k-m}\} \\
\oplus \\
C(\Gamma_k') \\
\oplus \\
C(\Gamma_k')\{[m-k]q^{k-m-1}\}
\end{array}%
\right. \text{ for } k_0<k<m,
\]
and
\[
C(\Gamma_{k_0,0}) \{q^{k_0-1-m}\} \simeq 
\left.%
\begin{array}{c}
C(\widetilde{\Gamma}_{k_0}) \{q^{k_0-1-m}\}\\
\oplus \\
C(\Gamma_{k_0}')\{[m-k_0]q^{k_0-1-m}\} \\
\end{array}%
\right..
\]
So, $\hat{C}(D_{1,1}^+)$ is isomorphic to 

{\tiny
\[
_{0 \rightarrow C(\Gamma_{m,1}) \xrightarrow{\mathfrak{d}_m^+} \left.%
\begin{array}{c}
C(\widetilde{\Gamma}_{m}) \{q^{-1}\}\\
\oplus \\
C(\Gamma_{m-1}') \\
\oplus \\
C(\Gamma_{m-1}')\{q^{-2}\}
\end{array}%
\right.
\xrightarrow{\mathfrak{d}_{m-1}^+} \cdots \xrightarrow{\mathfrak{d}_{k+1}^+} \left.%
\begin{array}{c}
C(\widetilde{\Gamma}_{k+1}) \{q^{k-m}\}\\
\oplus \\
C(\Gamma_{k+1}')\{[m-k-1]q^{k-m}\} \\
\oplus \\
C(\Gamma_k') \\
\oplus \\
C(\Gamma_k')\{[m-k]q^{k-m-1}\}
\end{array}%
\right.
\xrightarrow{\mathfrak{d}_{k}^+} \cdots\xrightarrow{\mathfrak{d}_{k_0}^+} \left.%
\begin{array}{c}
C(\widetilde{\Gamma}_{k_0}) \{q^{k_0-1-m}\}\\
\oplus \\
C(\Gamma_{k_0}')\{[m-k_0]q^{k_0-1-m}\} \\
\end{array}%
\right. \rightarrow 0.}
\]
}

\noindent In this form, $\mathfrak{d}_k^+$ is given by a $4\times 4$ matrix $(\mathfrak{d}_{k;i,j}^+)_{4\times4}$ for $k_0<k<m-1$. Clearly, 
\[
\mathfrak{d}_{k;i,j}^+=0 \text{ for } (i,j)=(3,1),~(3,2),~(4,1),~(4,2). 
\]
By Lemma \ref{relating-D-0-1-lemma}, 
\[
\mathfrak{d}_{k;1,1}^+ \approx \tilde{d}_{k+1}^+.
\]
By Lemma \ref{relating-C-D-11-lemma},
\[
\mathfrak{d}_{k;3,3}^+ \approx \delta_{k}^+.
\]
By \eqref{fork-special-decomp-iv-cor-1} in Corollary \ref{fork-special-decomp-iv-corollary}, we know that
\begin{eqnarray*}
\mathfrak{d}_{k;1,4}^+ & \simeq & 0,\\
\mathfrak{d}_{k;2,4}^+ & \approx & \id_{C(\Gamma_k')\{[m-k]q^{k-m-1}\}}.
\end{eqnarray*}
By Lemma \ref{hmf-gamma-prime-k-k-1}, we have
\[
\mathfrak{d}_{k;3,4}^+ \simeq 0.
\]
Altogether, we have that, for $k_0<k<m-1$,
\[
\mathfrak{d}_k^+ \simeq \left(%
\begin{array}{cccc}
c_k \tilde{d}_{k+1}^+ & \ast & \ast & 0 \\
\ast & \ast & \ast & c_k''\id_{C(\Gamma_k')\{[m-k]q^{k-m-1}\}}\\
0 & 0 & c_k' \delta_{k}^+ & 0 \\
0 &  0 & \ast & \ast
\end{array}%
\right),
\]
where $c_k$, $c_k'$ and $c_k''$ are non-zero scalars and $\ast$'s stand for morphisms we have not identified. Similarly,
\begin{eqnarray*}
\mathfrak{d}_{k_0}^+ & \simeq & \left(%
\begin{array}{cccc}
c_{k_0} \tilde{d}_{k_0+1}^+ & \ast & \ast & 0 \\
\ast & \ast & \ast & c_{k_0}''\id_{C(\Gamma_{k_0}')\{[m-k+0]q^{k_0-1-m}\}}
\end{array}%
\right), \\
\mathfrak{d}_{m-1}^+ & \simeq & \left(%
\begin{array}{ccc}
c_{m-1} \tilde{d}_{m}^+ & \ast & 0 \\
\ast  & \ast & c_{m-1}''\id_{C(\Gamma_{m-1}')\{q^{-2}\}}\\
0  & c_{m-1}' \delta_{m-1}^+ & 0 \\
0  & \ast & \ast
\end{array}%
\right),
\end{eqnarray*}
where $c_{k_0}$, $c_{k_0}''$, $c_{m-1}$, $c_{m-1}'$ and $c_{m-1}''$ are non-zero scalars.

Now apply Gaussian Elimination (Lemma \ref{gaussian-elimination}) to $c_k''\id_{C(\Gamma_k')\{[m-k]q^{k-m-1}\}}$ in $\mathfrak{d}_{k}^+$ for $k=k_0,k_0+1,\dots,m-1$ in that order. We get that $\hat{C}(D_{11}^+)$ is homotopic to 

{\tiny
\[
0 \rightarrow C(\Gamma_{m,1}) \xrightarrow{\hat{\mathfrak{d}}_m^+} \left.%
\begin{array}{c}
C(\widetilde{\Gamma}_{m}) \{q^{-1}\}\\
\oplus \\
C(\Gamma_{m-1}') \\
\end{array}%
\right. 
\xrightarrow{\hat{\mathfrak{d}}_{m-1}^+} \cdots \xrightarrow{\hat{\mathfrak{d}}_{k+1}^+} \left.%
\begin{array}{c}
C(\widetilde{\Gamma}_{k+1}) \{q^{k-m}\}\\
\oplus \\
C(\Gamma_k') \\
\end{array}%
\right. 
\xrightarrow{\hat{\mathfrak{d}}_{k}^+} \cdots\xrightarrow{\hat{\mathfrak{d}}_{k_0}^+} C(\widetilde{\Gamma}_{k_0}) \{q^{k_0-1-m}\} \rightarrow 0,
\]
}

\noindent where
\begin{eqnarray}
\label{hat-d-k}\hat{\mathfrak{d}}_k^+ & \simeq & \left(%
\begin{array}{cc}
c_k \tilde{d}_{k+1}^+ & \ast\\
0 & c_k' \delta_{k}^+
\end{array}%
\right) ~\text{ for } k_0<k<m, \\
\label{hat-d-k0}\hat{\mathfrak{d}}_{k_0}^+ & \simeq & \left(%
\begin{array}{cc}
c_{k_0} \tilde{d}_{k_0+1}^+, & \ast
\end{array}%
\right).
\end{eqnarray}
Next we determine $\hat{\mathfrak{d}}_m^+$. By Decomposition (V) (more precisely, \eqref{gamma-m-1-decomp-V}), we have
\[
C(\Gamma_{m,1}) \simeq \left.%
\begin{array}{c}
C(\widetilde{\Gamma}_{m+1})\\
\oplus \\
C(\Gamma_{m-1}'')\\
\end{array}%
\right. .
\]
Under this decomposition, $\hat{\mathfrak{d}}_m^+$ is represented by a $2\times2$ matrix. By Lemmas \ref{hmf-tilde-gamma-prime-m}, \ref{gamma-m-1-p-j-differential} and \ref{gamma-m-1-p-j-prime-dif}, we know that
\begin{equation}\label{hat-d-m}
\hat{\mathfrak{d}}_m^+ \simeq \left(%
\begin{array}{cc}
c_{m} \tilde{d}_{m+1}^+ & \ast \\
0 & c_{m}' J_{m-1,m-1}
\end{array}%
\right),
\end{equation}
where $c_m$ and $c_m'$ are non-zero scalars. So $\hat{C}(D_{11}^+)$ is homotopic to 

{\tiny
\[
0 \rightarrow \left.%
\begin{array}{c}
C(\widetilde{\Gamma}_{m+1})\\
\oplus \\
C(\Gamma_{m-1}'')\\
\end{array}%
\right. \xrightarrow{\hat{\mathfrak{d}}_m^+} \left.%
\begin{array}{c}
C(\widetilde{\Gamma}_{m}) \{q^{-1}\}\\
\oplus \\
C(\Gamma_{m-1}') \\
\end{array}%
\right. 
\xrightarrow{\hat{\mathfrak{d}}_{m-1}^+} \cdots \xrightarrow{\hat{\mathfrak{d}}_{k+1}^+} \left.%
\begin{array}{c}
C(\widetilde{\Gamma}_{k+1}) \{q^{k-m}\}\\
\oplus \\
C(\Gamma_k') \\
\end{array}%
\right. 
\xrightarrow{\hat{\mathfrak{d}}_{k}^+} \cdots\xrightarrow{\hat{\mathfrak{d}}_{k_0}^+} C(\widetilde{\Gamma}_{k_0}) \{q^{k_0-1-m}\} \rightarrow 0,
\]
}

\noindent where $\hat{\mathfrak{d}}_m^+,\dots,\hat{\mathfrak{d}}_{k_0}^+$ are given in \eqref{hat-d-k}, \eqref{hat-d-k0} and \eqref{hat-d-m}.

Recall that, by Decomposition (V) (more precisely, Lemma \ref{decomp-V-special-2}), 
\[
C(\Gamma_k') \simeq \begin{cases}
C(\Gamma_k'') \oplus C(\Gamma_{k-1}'') & \text{if } k_0+1 \leq l \leq m-1 ,\\
C(\Gamma_k'') & \text{if } k=k_0.
\end{cases}
\]
By Proposition \ref{trivial-complex-prop}, under the decomposition
\[
\left.%
\begin{array}{c}
C(\widetilde{\Gamma}_{k+1}) \{q^{k-m}\}\\
\oplus \\
C(\Gamma_k') \\
\end{array}%
\right. \simeq 
\left.%
\begin{array}{c}
C(\widetilde{\Gamma}_{k+1}) \{q^{k-m}\}\\
\oplus \\
C(\Gamma_k'') \\
\oplus \\
C(\Gamma_{k-1}'')
\end{array}%
\right. ,
\]
we have
\begin{eqnarray}
\label{hat-d-m-1}\hat{\mathfrak{d}}_m^+ & \simeq &
\left(%
\begin{array}{cc}
c_{m} \tilde{d}_{m+1}^+ & \ast \\
0 & c_{m}''' \id_{C(\Gamma_{m-1}'')} \\
0 & 0
\end{array}%
\right), \\
\label{hat-d-k-1}\hat{\mathfrak{d}}_k^+ & \simeq & \left(%
\begin{array}{ccc}
c_k \tilde{d}_{k+1}^+ & \ast& \ast \\
0 & 0& c_k''' \id_{C(\Gamma_{k-1}'')} \\
0 & 0& 0
\end{array}%
\right) ~\text{ for } k_0+1<k<m,
\end{eqnarray}
where $c_k'''$ is a non-zero scalar for $k_0+1<k\leq m$. Since $C(\Gamma_{k_0}') \simeq C(\Gamma_{k_0}'')$, we have 
\[
\left.%
\begin{array}{c}
C(\widetilde{\Gamma}_{k_0+1}) \{q^{k_0-m}\}\\
\oplus \\
C(\Gamma_{k_0}') \\
\end{array}%
\right. \simeq 
\left.%
\begin{array}{c}
C(\widetilde{\Gamma}_{k_0+1}) \{q^{k_0-m}\}\\
\oplus \\
C(\Gamma_{k_0}'') \\
\end{array}%
\right. 
\]
and 
\begin{eqnarray}
\label{hat-d-k0+1-1} \hat{\mathfrak{d}}_{k_0+1}^+ & \simeq & \left(%
\begin{array}{ccc}
c_{k_0+1} \tilde{d}_{k_0+2}^+ & \ast& \ast \\
0 & 0 & c_{k_0+1}''' \id_{C(\Gamma_{k_0}'')}
\end{array}%
\right), \\
\label{hat-d-k0-1} \hat{\mathfrak{d}}_{k_0}^+ & \simeq & \left(%
\begin{array}{cc}
c_{k_0} \tilde{d}_{k_0+1}^+, & \ast
\end{array}%
\right),
\end{eqnarray}
where $c_{k_0+1}'''$ is a non-zero scalar. Putting these together, we know that $\hat{C}(D_{11}^+)$ is homotopic to 

{\tiny
\[_{
0 \rightarrow \left.%
\begin{array}{c}
C(\widetilde{\Gamma}_{m+1})\\
\oplus \\
C(\Gamma_{m-1}'')\\
\end{array}%
\right. \xrightarrow{\hat{\mathfrak{d}}_m^+} \left.%
\begin{array}{c}
C(\widetilde{\Gamma}_{m}) \{q^{-1}\}\\
\oplus \\
C(\Gamma_{m-1}'') \\
\oplus \\
C(\Gamma_{m-2}'')
\end{array}%
\right. 
\xrightarrow{\hat{\mathfrak{d}}_{m-1}^+} \cdots \xrightarrow{\hat{\mathfrak{d}}_{k+1}^+} \left.%
\begin{array}{c}
C(\widetilde{\Gamma}_{k+1}) \{q^{k-m}\}\\
\oplus \\
C(\Gamma_k'') \\
\oplus \\
C(\Gamma_{k-1}'')
\end{array}%
\right.
\xrightarrow{\hat{\mathfrak{d}}_{k}^+} \cdots \xrightarrow{\hat{\mathfrak{d}}_{k_0+1}^+} \left.%
\begin{array}{c}
C(\widetilde{\Gamma}_{k_0+1}) \{q^{k_0-m}\}\\
\oplus \\
C(\Gamma_{k_0}'') \\
\end{array}%
\right. \xrightarrow{\hat{\mathfrak{d}}_{k_0}^+} C(\widetilde{\Gamma}_{k_0}) \{q^{k_0-1-m}\} \rightarrow 0,}
\]
}

\noindent where $\hat{\mathfrak{d}}_m^+,\dots,\hat{\mathfrak{d}}_{k_0}^+$ are given in \eqref{hat-d-m-1},\eqref{hat-d-k-1}, \eqref{hat-d-k0+1-1} and \eqref{hat-d-k0-1}.

Applying Gaussian Elimination (Lemma \ref{gaussian-elimination}) to $c_k''' \id_{C(\Gamma_{k-1}'')}$ in $\hat{\mathfrak{d}}_k^+$ for $k=m,m-1,\dots,k_0+1$, we get that $\hat{C}(D_{11}^+)$ is homotopic to

{\tiny
\[
0 \rightarrow 
C(\widetilde{\Gamma}_{m+1}) \xrightarrow{\check{\mathfrak{d}}_m^+} 
C(\widetilde{\Gamma}_{m}) \{q^{-1}\}
\xrightarrow{\check{\mathfrak{d}}_{m-1}^+} \cdots \xrightarrow{\check{\mathfrak{d}}_{k+1}^+} 
C(\widetilde{\Gamma}_{k+1}) \{q^{k-m}\}
\xrightarrow{\check{\mathfrak{d}}_{k}^+} \cdots\xrightarrow{\check{\mathfrak{d}}_{k_0}^+} C(\widetilde{\Gamma}_{k_0}) \{q^{k_0-1-m}\} \rightarrow 0,
\]
}

\noindent where $\check{\mathfrak{d}}_{k}^+ \simeq c_k \tilde{d}_{k+1}^+$ for $k=m,m-1,\dots,k_0$. Recall that $c_k\neq 0$ for $k=m,\dots,k_0$. So this last chain complex is isomorphic to $\hat{C}(D_{10}^+)$ in $\ch(\hmf)$. Therefore, $\hat{C}(D_{11}^+) \simeq \hat{C}(D_{10}^+)$.
\end{proof}

\begin{proof}[Proof of $\hat{C}(D_{10}^-) \simeq \hat{C}(D_{11}^-)$ when $l=1$]
Recall that the chain complex $\hat{C}(D_{1,1}^-)$ is

{\tiny
\[
0 \rightarrow C(\Gamma_{k_0,0}) \{q^{m+1-k_0}\} \xrightarrow{\mathfrak{d}_{k_0}^-} \cdots \xrightarrow{\mathfrak{d}_{k-1}^-} \left.%
\begin{array}{c}
C(\Gamma_{k,0}) \{q^{m+1-k}\}\\
\oplus \\
C(\Gamma_{k-1,1})\{q^{m+1-k}\}
\end{array}%
\right. 
\xrightarrow{\mathfrak{d}_{k}^-} \cdots  \xrightarrow{\mathfrak{d}_{m-1}^-} \left.%
\begin{array}{c}
C(\Gamma_{m,0}) \{q\}\\
\oplus \\
C(\Gamma_{m-1,1})\{q\}
\end{array}%
\right. 
\xrightarrow{\mathfrak{d}_m^-} C(\Gamma_{m,1}) \rightarrow 0,
\]
}

\noindent where $k_0 = \max \{m-n,0\}$ as above and
\begin{eqnarray*}
\mathfrak{d}_{k_0}^- & = & \left(%
\begin{array}{c}
d_{k_0}^-\\ 
\chi^0
\end{array}%
\right),\\
\mathfrak{d}_k^- & = & \left(%
\begin{array}{cc}
d_{k}^- & 0\\
\chi^0 & -d_{k-1}^-
\end{array}%
\right) ~\text{ for } k_0<k<m, \\
\mathfrak{d}_{m}^- & = & \left(%
\begin{array}{cc}
\chi^0 & -d_{m-1}^-\\
\end{array}%
\right).
\end{eqnarray*}

From Decomposition (IV) (more precisely, \eqref{fork-special-decomp-iv-eq}), we have 
\[
C(\Gamma_{k,0}) \simeq C(\widetilde{\Gamma}_k) \oplus C(\Gamma_k')\{[m-k]\}.
\] 
By Corollary \ref{contract-expand} and Decomposition (II) (Theorem \ref{decomp-II}), we have 
\[
C(\Gamma_{k,1}) \simeq C(\Gamma_k')\{[m+1-k]\} \cong C(\Gamma_k')\{q^{k-m}\} \oplus C(\Gamma_k')\{[m-k]\cdot q\}.
\]
Therefore,
\[
\left.%
\begin{array}{c}
C(\Gamma_{k,0}) \{q^{m+1-k}\}\\
\oplus \\
C(\Gamma_{k-1,1})\{q^{m+1-k}\}
\end{array}%
\right. \simeq 
\left.%
\begin{array}{c}
C(\widetilde{\Gamma}_{k}) \{q^{m+1-k}\}\\
\oplus \\
C(\Gamma_{k}')\{[m-k]q^{m+1-k}\} \\
\oplus \\
C(\Gamma_{k-1}') \\
\oplus \\
C(\Gamma_{k-1}')\{[m+1-k]q^{m+2-k}\}
\end{array}%
\right. \text{ for } k_0<k<m,
\]
and 
\[
C(\Gamma_{k_0,0}) \{q^{m+1-k_0}\} \simeq 
\left.%
\begin{array}{c}
C(\widetilde{\Gamma}_{k_0}) \{q^{m+1-k_0}\}\\
\oplus \\
C(\Gamma_{k_0}')\{[m-k_0]q^{m+1-k_0}\} \\
\end{array}%
\right..
\]
So, $\hat{C}(D_{1,1}^-)$ is isomorphic to 

{\tiny
\[
_{0 \rightarrow \left.%
\begin{array}{c}
C(\widetilde{\Gamma}_{k_0}) \{q^{m+1-k_0}\}\\
\oplus \\
C(\Gamma_{k_0}')\{[m-k_0]q^{m+1-k_0}\} \\
\end{array}%
\right. \xrightarrow{\mathfrak{d}_{k_0}^-} \cdots \xrightarrow{\mathfrak{d}_{k-1}^-} \left.%
\begin{array}{c}
C(\widetilde{\Gamma}_{k}) \{q^{m+1-k}\}\\
\oplus \\
C(\Gamma_{k}')\{[m-k]q^{m+1-k}\} \\
\oplus \\
C(\Gamma_{k-1}') \\
\oplus \\
C(\Gamma_{k-1}')\{[m+1-k]q^{m+2-k}\}
\end{array}%
\right.
\xrightarrow{\mathfrak{d}_{k}^-} \cdots  \xrightarrow{\mathfrak{d}_{m-1}^-} \left.%
\begin{array}{c}
C(\widetilde{\Gamma}_{m}) \{q\}\\
\oplus \\
C(\Gamma_{m-1}') \\
\oplus \\
C(\Gamma_{m-1}')\{q^{2}\}
\end{array}%
\right.
\xrightarrow{\mathfrak{d}_m^-} C(\Gamma_{m,1}) \rightarrow 0.}
\]
}

\noindent In this form, $\mathfrak{d}_k^-$ is given by a $4\times 4$ matrix $(\mathfrak{d}_{k;i,j}^-)_{4\times4}$ for $k_0<k<m-1$. Clearly, 
\[
\mathfrak{d}_{k;i,j}^-=0 \text{ for } (i,j)=(1,3),~(1,4),~(2,3),~(2,4). 
\]
By Lemma \ref{relating-D-0-1-lemma}, 
\[
\mathfrak{d}_{k;1,1}^- \approx \tilde{d}_{k}^-.
\]
By Lemma \ref{relating-C-D-11-lemma},
\[
\mathfrak{d}_{k;3,3}^- \approx \delta_{k-1}^-.
\]
By \eqref{fork-special-decomp-iv-cor-2} in Corollary \ref{fork-special-decomp-iv-corollary}, we know that
\begin{eqnarray*}
\mathfrak{d}_{k;4,1}^- & \simeq & 0,\\
\mathfrak{d}_{k;4,2}^- & \approx & \id_{C(\Gamma_{k}')\{[m-k]q^{m+1-k}\}}.
\end{eqnarray*}
By Lemma \ref{hmf-gamma-prime-k-k-1}, we have
\[
\mathfrak{d}_{k;4,3}^- \simeq 0.
\]
Altogether, we have that, for $k_0<k<m-1$,
\[
\mathfrak{d}_k^- \simeq \left(%
\begin{array}{cccc}
c_k \tilde{d}_{k}^- & \ast & 0 & 0 \\
\ast & \ast & 0 & 0\\
\ast & \ast & c_k'\delta_{k-1}^- & \ast \\
0 &  c_k''\id_{C(\Gamma_{k}')\{[m-k]q^{m+1-k}\}} & 0 & \ast
\end{array}%
\right),
\]
where $c_k$, $c_k'$ and $c_k''$ are non-zero scalars and $\ast$'s stand for morphisms we have not identified. Similarly,
\begin{eqnarray*}
\mathfrak{d}_{k_0}^- & \simeq & \left(%
\begin{array}{cc}
c_{k_0} \tilde{d}_{k_0}^- & \ast  \\
\ast & \ast \\
\ast & \ast  \\
0 &  c_{k_0}''\id_{C(\Gamma_{k_0}')\{[m-k_0]q^{m+1-k_0}\}} 
\end{array}%
\right), \\
\mathfrak{d}_{m-1}^+ & \simeq & \left(%
\begin{array}{cccc}
c_{m-1} \tilde{d}_{m-1}^- & \ast & 0 & 0 \\
\ast & \ast & c_{m-1}'\delta_{m-2}^- & \ast \\
0 &  c_{m-1}''\id_{C(\Gamma_{m-1}')\{q^{2}\}} & 0 & \ast
\end{array}%
\right),
\end{eqnarray*}
where $c_{k_0}$, $c_{k_0}''$, $c_{m-1}$, $c_{m-1}'$ and $c_{m-1}''$ are non-zero scalars.

Now apply Gaussian Elimination (Lemma \ref{gaussian-elimination}) to $c_{k}''\id_{C(\Gamma_{k}')\{[m-k]q^{m+1-k}\}}$ in $\mathfrak{d}_{k}^-$ for $k=k_0,k_0+1,\dots,m-1$ in that order. We get that $\hat{C}(D_{11}^-)$ is homotopic to

{\tiny
\[
0 \rightarrow 
C(\widetilde{\Gamma}_{k_0}) \{q^{m+1-k_0}\} \xrightarrow{\hat{\mathfrak{d}}_{k_0}^-} \cdots \xrightarrow{\hat{\mathfrak{d}}_{k-1}^-} \left.%
\begin{array}{c}
C(\widetilde{\Gamma}_{k}) \{q^{m+1-k}\}\\
\oplus \\
C(\Gamma_{k-1}') \\
\end{array}%
\right.
\xrightarrow{\hat{\mathfrak{d}}_{k}^-} \cdots  \xrightarrow{\hat{\mathfrak{d}}_{m-1}^-} \left.%
\begin{array}{c}
C(\widetilde{\Gamma}_{m}) \{q\}\\
\oplus \\
C(\Gamma_{m-1}') \\
\end{array}%
\right.
\xrightarrow{\hat{\mathfrak{d}}_m^-} C(\Gamma_{m,1}) \rightarrow 0,
\]
}

\noindent where
\begin{eqnarray}
\label{hat-d-k-}\hat{\mathfrak{d}}_k^- & \simeq & \left(%
\begin{array}{cc}
c_k \tilde{d}_{k}^-  & 0 \\
\ast & c_k'\delta_{k-1}^- 
\end{array}%
\right) ~\text{ for } k_0<k<m, \\
\label{hat-d-k0-}\hat{\mathfrak{d}}_{k_0}^- & \simeq & \left(%
\begin{array}{c}
c_{k_0} \tilde{d}_{k_0}^-  \\
\ast  \\
\ast  
\end{array}%
\right).
\end{eqnarray}
Next we determine $\hat{\mathfrak{d}}_m^-$. By Decomposition (V) (more precisely, \eqref{gamma-m-1-decomp-V}), we have
\[
C(\Gamma_{m,1}) \simeq \left.%
\begin{array}{c}
C(\widetilde{\Gamma}_{m+1})\\
\oplus \\
C(\Gamma_{m-1}'')
\end{array}%
\right. .
\]
Under this decomposition, $\hat{\mathfrak{d}}_m^-$ is represented by a $2\times2$ matrix. By Lemmas \ref{hmf-tilde-gamma-prime-m}, \ref{gamma-m-1-p-j-differential} and \ref{gamma-m-1-p-j-prime-dif}, we know that
\begin{equation}\label{hat-d-m-}
\hat{\mathfrak{d}}_m^- \simeq \left(%
\begin{array}{cc}
c_{m} \tilde{d}_{m}^- & 0 \\
\ast & c_{m}' P_{m-1,m-1}
\end{array}%
\right),
\end{equation}
where $c_m$ and $c_m'$ are non-zero scalars. So $\hat{C}(D_{11}^-)$ is homotopic to 

{\tiny
\[
0 \rightarrow 
C(\widetilde{\Gamma}_{k_0}) \{q^{m+1-k_0}\} \xrightarrow{\hat{\mathfrak{d}}_{k_0}^-} \cdots \xrightarrow{\hat{\mathfrak{d}}_{k-1}^-} \left.%
\begin{array}{c}
C(\widetilde{\Gamma}_{k}) \{q^{m+1-k}\}\\
\oplus \\
C(\Gamma_{k-1}') \\
\end{array}%
\right.
\xrightarrow{\hat{\mathfrak{d}}_{k}^-} \cdots  \xrightarrow{\hat{\mathfrak{d}}_{m-1}^-} \left.%
\begin{array}{c}
C(\widetilde{\Gamma}_{m}) \{q\}\\
\oplus \\
C(\Gamma_{m-1}') \\
\end{array}%
\right.
\xrightarrow{\hat{\mathfrak{d}}_m^-} \left.%
\begin{array}{c}
C(\widetilde{\Gamma}_{m+1})\\
\oplus \\
C(\Gamma_{m-1}'')
\end{array}%
\right. \rightarrow 0,
\]
}

\noindent where $\hat{\mathfrak{d}}_m^-,\dots,\hat{\mathfrak{d}}_{k_0}^-$ are given in \eqref{hat-d-k-}, \eqref{hat-d-k0-} and \eqref{hat-d-m-}.

Recall that, by Decomposition (V) (more precisely, Lemma \ref{decomp-V-special-2}), 
\[
C(\Gamma_k') \simeq \begin{cases}
C(\Gamma_k'') \oplus C(\Gamma_{k-1}'') & \text{if } k_0+1 \leq l \leq m-1 ,\\
C(\Gamma_k'') & \text{if } k=k_0.
\end{cases}
\]
By Proposition \ref{trivial-complex-prop}, under the decomposition
\[
\left.%
\begin{array}{c}
C(\widetilde{\Gamma}_{k}) \{q^{m+1-k}\}\\
\oplus \\
C(\Gamma_{k-1}') \\
\end{array}%
\right. \simeq 
\left.%
\begin{array}{c}
C(\widetilde{\Gamma}_{k}) \{q^{m+1-k}\}\\
\oplus \\
C(\Gamma_{k-1}'') \\
\oplus \\
C(\Gamma_{k-2}'')
\end{array}%
\right. ,
\]
we have
\begin{eqnarray}
\label{hat-d-m-1-}\hat{\mathfrak{d}}_m^- & \simeq & \left(%
\begin{array}{ccc}
c_{m} \tilde{d}_{m}^- & 0 & 0\\
\ast & c_{m}''' \id_{C(\Gamma_{m-1}'')} & 0
\end{array}%
\right), \\
\label{hat-d-k-1-}\hat{\mathfrak{d}}_k^- & \simeq & \left(%
\begin{array}{ccc}
c_k \tilde{d}_{k}^-  & 0 & 0\\
\ast & 0 & 0 \\
\ast & c_k'''\id_{C(\Gamma_{k-1}'')} & 0
\end{array}%
\right) ~\text{ for } k_0+1<k<m,
\end{eqnarray}
where $c_k'''$ is a non-zero scalar for $k_0+1<k\leq m$.
Since $C(\Gamma_{k_0}') \simeq C(\Gamma_{k_0}'')$, we have 
\[
\left.%
\begin{array}{c}
C(\widetilde{\Gamma}_{k_0+1}) \{q^{m-k_0}\}\\
\oplus \\
C(\Gamma_{k_0}') \\
\end{array}%
\right. \simeq 
\left.%
\begin{array}{c}
C(\widetilde{\Gamma}_{k_0+1}) \{q^{m-k_0}\}\\
\oplus \\
C(\Gamma_{k_0}'') \\
\end{array}%
\right. 
\]
and 
\begin{eqnarray}
\label{hat-d-k0+1-1-} \hat{\mathfrak{d}}_{k_0+1}^- & \simeq & \left(%
\begin{array}{cc}
c_{k_0+1} \tilde{d}_{k_0+1}^-  & 0 \\
\ast & 0 \\
\ast & c_{k_0+1}'''\id_{C(\Gamma_{k_0}'')}
\end{array}%
\right), \\
\label{hat-d-k0-1-} \hat{\mathfrak{d}}_{k_0}^- & \simeq & \left(%
\begin{array}{c}
c_{k_0} \tilde{d}_{k_0+1}^+ \\ 
\ast
\end{array}%
\right),
\end{eqnarray}
where $c_{k_0+1}'''$ is a non-zero scalar. Putting these together, we know that $\hat{C}(D_{11}^-)$ is homotopic to 

{\tiny
\[
_{0 \rightarrow 
C(\widetilde{\Gamma}_{k_0}) \{q^{m+1-k_0}\} \xrightarrow{\hat{\mathfrak{d}}_{k_0}^-} \left.%
\begin{array}{c}
C(\widetilde{\Gamma}_{k_0+1}) \{q^{m-k_0}\}\\
\oplus \\
C(\Gamma_{k_0}'') \\
\end{array}%
\right. \xrightarrow{\hat{\mathfrak{d}}_{k_0+1}^-} \cdots \xrightarrow{\hat{\mathfrak{d}}_{k-1}^-} \left.%
\begin{array}{c}
C(\widetilde{\Gamma}_{k}) \{q^{m+1-k}\}\\
\oplus \\
C(\Gamma_{k-1}'') \\
\oplus \\
C(\Gamma_{k-2}'')
\end{array}%
\right.
\xrightarrow{\hat{\mathfrak{d}}_{k}^-} \cdots  \xrightarrow{\hat{\mathfrak{d}}_{m-1}^-} \left.%
\begin{array}{c}
C(\widetilde{\Gamma}_{m}) \{q\}\\
\oplus \\
C(\Gamma_{m-1}'') \\
\oplus \\
C(\Gamma_{m-2}'')
\end{array}%
\right.
\xrightarrow{\hat{\mathfrak{d}}_m^-} \left.%
\begin{array}{c}
C(\widetilde{\Gamma}_{m+1})\\
\oplus \\
C(\Gamma_{m-1}'')
\end{array}%
\right. \rightarrow 0,}
\]
}

\noindent where $\hat{\mathfrak{d}}_m^-,\dots,\hat{\mathfrak{d}}_{k_0}^-$ are given in \eqref{hat-d-m-1-},\eqref{hat-d-k-1-}, \eqref{hat-d-k0+1-1-} and \eqref{hat-d-k0-1-}.

Applying Gaussian Elimination (Lemma \ref{gaussian-elimination}) to $c_k''' \id_{C(\Gamma_{k-1}'')}$ in $\hat{\mathfrak{d}}_k^-$ for $k=m,m-1,\dots,k_0+1$, we get that $\hat{C}(D_{11}^-)$ is homotopic to
\[
0 \rightarrow 
C(\widetilde{\Gamma}_{k_0}) \{q^{m+1-k_0}\} \xrightarrow{\check{\mathfrak{d}}_{k_0}^-}  \cdots \xrightarrow{\check{\mathfrak{d}}_{k-1}^-} 
C(\widetilde{\Gamma}_{k}) \{q^{m+1-k}\}
\xrightarrow{\check{\mathfrak{d}}_{k}^-} \cdots  
\xrightarrow{\check{\mathfrak{d}}_m^-} 
C(\widetilde{\Gamma}_{m+1})\rightarrow 0,
\]
where $\check{\mathfrak{d}}_{k}^- \simeq c_k \tilde{d}_{k}^-$ for $k=m,\dots,k_0$. Recall that $c_k\neq 0$ for $k=m,\dots,k_0$. So this last chain complex is isomorphic to $\hat{C}(D_{10}^-)$ in $\ch(\hmf)$. Therefore, $\hat{C}(D_{11}^-) \simeq \hat{C}(D_{10}^-)$.
\end{proof}

So we have completed the proof of \eqref{fork-sliding-invariance-special-eq}, that is, $\hat{C}(D_{10}^\pm) \simeq \hat{C}(D_{11}^\pm)$ if $l=1$. The proof of the rest of Proposition \ref{fork-sliding-invariance-special} is very similar and left to the reader. This completes the proof of Theorem \ref{fork-sliding-invariance-general}.

\section{Invariance under Reidemeister Moves}\label{sec-inv-reidemeister}

In this section, we prove that the homotopy type of the normalized chain complex associated to a knotted MOY graph is invariant under Reidemeister moves. The main result of this section is Theorem \ref{invariance-reidemeister-all} below. Note that Theorem \ref{main} is a spacial case of Theorem \ref{invariance-reidemeister-all}.

\begin{theorem}\label{invariance-reidemeister-all}
Let $D_0$ and $D_1$ be two knotted MOY graphs. Assume that there is a finite sequence of Reidemeister moves that changes $D_0$ into $D_1$. Then $C(D_0) \simeq C(D_1)$, that is, they are isomorphic as objects of $\hch(\hmf)$.
\end{theorem}

Theorem \ref{invariance-reidemeister-all} follows from Lemmas \ref{invariance-reidemeister-II-III} and \ref{invariance-reidemeister-I} below, in which we establish the invariance of the homotopy type under Reidemeister moves I, II$_a$, II$_b$ and III given in Figures \ref{reidemeisterI}, \ref{reidemeisterII-a}, \ref{reidemeisterII-b} and \ref{reidemeisterIII}. The proofs of these lemmas are based on an induction on the highest color of the edges involved in the Reidemeister move. The starting point of our induction is the following theorem by Khovanov and Rozansky \cite{KR1}.

\begin{theorem}\cite[Theorem 2]{KR1}\label{invariance-reidemeister-all-color=1}
Let $D_0$ and $D_1$ be two knotted MOY graphs. Assume that there is a Reidemeister move changing $D_0$ into $D_1$ that involves only edges colored by $1$. Then $C(D_0) \simeq C(D_1)$, that is, they are isomorphic as objects of $\hch(\hmf)$. 
\end{theorem}

\begin{remark}
The original statement of \cite[Theorem 2]{KR1} covers only link diagrams colored entirely by $1$. But its proof in \cite[Section 8]{KR1} is local in the sense that it is based on homotopy equivalences of the chain complex associated to the part of the link diagram involved in the Reidemeister move. So the slightly more general statement of Theorem \ref{invariance-reidemeister-all-color=1} also follows from the proof in \cite{KR1}.
\end{remark}

\begin{figure}[ht]
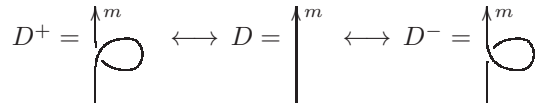

$
D^+= \xygraph{
!{0;/r1pc/:}
[u]
!{\xcapv@(0)=>>{m}}
!{\hover}
!{\hcap}
[ld]!{\xcapv@(0)}
}\, \longleftrightarrow \,
D=\xygraph{
!{0;/r1pc/:}
[u]
!{\xcapv@(0)=>>{m}}
!{\xcapv@(0)}
!{\xcapv@(0)}
}\, \longleftrightarrow \,
D^-=\xygraph{
!{0;/r1pc/:}
[u]
!{\xcapv@(0)=>>{m}}
!{\hunder}
!{\hcap}
[ld]!{\xcapv@(0)}
}
$
\caption{Reidemeister Move I}\label{reidemeisterI}

\end{figure}

\begin{figure}[ht]
$
\xygraph{
!{0;/r2pc/:}
[u]
!{\vtwist=<>{m}|{n}}
!{\vtwistneg}
} \, \longleftrightarrow \,
\xygraph{
!{0;/r2pc/:}
[u]
!{\xcapv[-1]@(0)=<>{m}}
!{\xcapv@(0)}
[ruu]!{\xcapv@(0)=><{n}}
!{\xcapv@(0)}
} \, \longleftrightarrow \,
\xygraph{
!{0;/r2pc/:}
[u]
!{\vtwistneg=<>{n}|{m}}
!{\vtwist}
}
$
\caption{Reidemeister Move II$_a$}\label{reidemeisterII-a}

\end{figure}

\begin{figure}[ht]
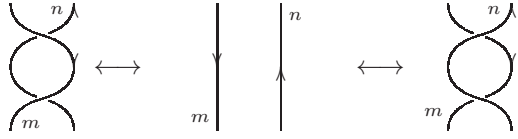

$
\xygraph{
!{0;/r2pc/:}
[u]
!{\vcross=>>{n}}
!{\vcrossneg[-1]|{m}}
} \, \longleftrightarrow \,
\xygraph{
!{0;/r2pc/:}
[u]
!{\xcapv[-1]@(0)=>}
!{\xcapv[-1]@(0)>{m}}
[ruu]!{\xcapv@(0)>{n}}
!{\xcapv@(0)=>}
} \, \longleftrightarrow \,
\xygraph{
!{0;/r2pc/:}
[u]
!{\vcrossneg=>|{n}}
!{\vcross<{m}}}
$
\caption{Reidemeister Move II$_b$}\label{reidemeisterII-b}

\end{figure}

\begin{figure}[ht]
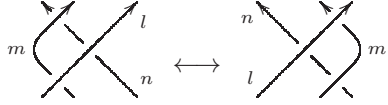

$
\xygraph{
!{0;/r1pc/:}
[uu]
!{\xoverv=<}
!{\xcapv[-1]|{m}}
!{\xoverv}
[uur]!{\xoverv}
[r]!{\zbendh@(0)>{n}}
[uul]!{\sbendh@(0)=>>{l}}
}\, \longleftrightarrow \,
\xygraph{
!{0;/r-1pc/:}
[u]
!{\xoverv}
!{\xcapv[-1]|{m}}
!{\xoverv=>}
[uur]!{\xoverv}
[r]!{\zbendh@(0)=>>{n}}
[uul]!{\sbendh@(0)>{l}}
}
$
\caption{Reidemeister Move III}\label{reidemeisterIII}

\end{figure}

\subsection{Invariance under Reidemeister moves II$_a$, II$_b$ and III} With the invariance under fork sliding (Theorem \ref{fork-sliding-invariance-general}) in hand, we can easily prove the invariance of the homotopy type under Reidemeister moves II$_a$, II$_b$, III by an induction using the ``sliding bi-gon" method introduced in \cite{MOY} (and used in \cite{Mackaay-Stosic-Vaz2}.) The proof of the invariance under Reidemeister move I is somewhat different and is postponed to the next subsection.

\begin{lemma}\label{invariance-reidemeister-II-III}
Let $D_0$ and $D_1$ be two knotted MOY graphs. Assume that there is a Reidemeister move of type II$_a$, II$_b$ or III that changes $D_0$ into $D_1$. Then $C(D_0) \simeq C(D_1)$, that is, they are isomorphic as objects of $\hch(\hmf)$.
\end{lemma}

\begin{proof}
The proofs for Reidemeister moves II$_a$, II$_b$ and III are quite similar. We only give details for Reidemeister move II$_a$ here and leave the other two moves to the reader.

\begin{figure}[ht]
$
D_0=\xygraph{
!{0;/r2pc/:}
[u]
!{\xcapv[-1]@(0)=<>{m}}
!{\xcapv@(0)}
[ruu]!{\xcapv@(0)=><{n}}
!{\xcapv@(0)}
}  \hspace{2cm}
D_1=\xygraph{
!{0;/r2pc/:}
[u]
!{\vtwist=<>{m}|{n}}
!{\vtwistneg}
}
$
\caption{}\label{reidemeisterII-a-proof}

\end{figure}

Let $D_0$ and $D_1$ be the knotted MOY graphs in Figure \ref{reidemeisterII-a-proof}. We prove by an induction on $k$ that $C(D_0) \simeq C(D_1)$ if $1 \leq m,n \leq k$. When $k=1$, this statement is a special case of Theorem \ref{invariance-reidemeister-all-color=1}. Assume that this statement is true for some $k\geq 1$.

\begin{figure}[ht]
$\Gamma_0 =\xygraph{
!{0;/r3pc/:}
[u(1.25)]
!{\xcapv[-0.75]@(0)=<<{m}}
[u(0.25)]!{\xcapv=>|{1}}
[u]!{\xcapv[-1]=<|{m-1}}
!{\xcapv[-0.75]@(0)=<>{m}}
[r][u(2.75)]!{\xcapv[-0.75]@(0)=<<{n}}
[u(0.25)]!{\xcapv=>|{n-1}}
[u]!{\xcapv[-1]=<|{1}}
!{\xcapv[-0.75]@(0)=<>{n}}
}
\Gamma_2 = \xygraph{
!{0;/r1pc/:}
[uuuu]
!{\xcapv[-1]@(0)=<<{m}}
[ur]!{\xcapv@(0)=>>{n}}
[dll]!{\zbendv<{m-1}}
!{\vtwist}
[ru]!{\sbendv[-1]<{n-1}}
[lll]!{\vtwist}
[urr]!{\vtwist}
[l]!{\vtwist=<}
[lu]!{\xcapv@(0)=>}
!{\xcapv[-1]@(0)<{1}}
[rrruu]!{\xcapv@(0)=>}
!{\xcapv@(0)>{1}}
[llu]!{\vtwistneg}
[l]!{\vtwistneg}
[urr]!{\vtwistneg}
[l]!{\vtwistneg}
[lu]!{\sbendv}
[r]!{\zbendv}
[dl]!{\xcapv@(0)=><{n}}
[lu]!{\xcapv[-1]@(0)=<>{m}}
}
\Gamma_1=
\xygraph{
!{0;/r2pc/:}
[uu]
!{\xcapv[-0.5]@(0)=<<{m}}
[u(0.5)]!{\xcapv=>>{1}}
[u]!{\xcapv[-1]=<|{m-1}}
[r][u(1.5)]!{\xcapv[-0.5]@(0)=<<{n}}
[u(0.5)]!{\xcapv=>|{n-1}}
[u]!{\xcapv[-1]=<>{1}}
[l]!{\vtwist}
!{\vtwistneg}
!{\xcapv[-0.5]@(0)=<>{m}}
[ru]!{\xcapv[0.5]@(0)=><{n}}
}
$
\caption{}\label{reidemeisterII-a-bubbles}

\end{figure}

Now consider $k+1$. Assume that $1\leq m,n \leq k+1$ in $D_0$ and $D_1$. Let $\Gamma_0$, $\Gamma_1$ and $\Gamma_2$ be in the knotted MOY graphs in Figure \ref{reidemeisterII-a-bubbles}. Here, in case $m$ or $n=1$, we use the convention that an edge colored by $0$ is an edge that does not exist. By Decomposition (II) (Theorem \ref{decomp-II}), we know that $\hat{C}(\Gamma_0) \simeq \hat{C}(D_0)\{[m][n]\}$ and $\hat{C}(\Gamma_1) \simeq \hat{C}(D_1)\{[m][n]\}$. Note that $m-1,n-1\leq k$. By induction hypothesis and the normalization in Definition \ref{complex-colored-crossing-def}, we know that $\hat{C}(\Gamma_0) \simeq \hat{C}(\Gamma_2)$. By the invariance under fork sliding (Theorem \ref{fork-sliding-invariance-general}), we know that $\hat{C}(\Gamma_1) \simeq \hat{C}(\Gamma_2)$. Thus, $\hat{C}(\Gamma_0) \simeq \hat{C}(\Gamma_1)$. By Proposition \ref{yonezawa-lemma-hmf}, it follows that $\hat{C}(D_0) \simeq \hat{C}(D_1)$, which, by the normalization in Definition \ref{complex-colored-crossing-def}, is equivalent to $C(D_0) \simeq C(D_1)$. This completes the induction.
\end{proof}

\subsection{Invariance under Reidemeister move I} The proof of invariance under Reidemeister move I is somewhat different from that under Reidemeister moves II and III. The basic idea is still the ``sliding bi-gon". But we also need to do some ``untwisting" to get the invariance.

\begin{figure}[ht]
$
\xymatrix{
\input{fork-m-n} 
} 
$
\caption{}\label{fork-hmf-fig} 

\end{figure}

\begin{lemma}\label{fork-hmf}
Let $\Gamma_{m,n}$ be the MOY graph in Figure \ref{fork-hmf-fig}. Then
\begin{eqnarray*}
&& \Hom_\HMF(C(\Gamma_{m,n}),C(\Gamma_{m,n})) \cong \Hom_\HMF(C(\overline{\Gamma}_{m,n}),C(\overline{\Gamma}_{m,n})) \\
& \cong & C(\emptyset)\{\qb{N}{m+n} \qb{m+n}{n}q^{(m+n)(N-m-n)+mn}\},
\end{eqnarray*}
where $\overline{\Gamma}_{m,n}$ is $\Gamma_{m,n}$ with orientation reversed. In particular, the lowest non-vanishing quantum degree of these spaces is $0$. Therefore, for $k<l$, 
\[
\Hom_\hmf(C(\Gamma_{m,n})\{q^k\},C(\Gamma_{m,n})\{q^l\}) \cong \Hom_\hmf(C(\overline{\Gamma}_{m,n})\{q^k\},C(\overline{\Gamma}_{m,n})\{q^l\}) \cong 0.
\]
\end{lemma}

\begin{proof}
Consider the MOY graph $\Gamma$ in Figure \ref{butterfly}. It is easy to check that 
\begin{eqnarray*}
&& \Hom_\HMF(C(\Gamma_{m,n}),C(\Gamma_{m,n})) \cong \Hom_\HMF(C(\overline{\Gamma}_{m,n}),C(\overline{\Gamma}_{m,n})) \\
& \cong & H(\Gamma)\left\langle m+n \right\rangle \{q^{(m+n)(N-m-n)+mn}\}.
\end{eqnarray*}
By Lemma \ref{butterfly-structure}, 
\[
H(\Gamma) \cong C(\emptyset)  \left\langle m+n \right\rangle \{\qb{N}{m+n} \qb{m+n}{n}\},
\]
whose lowest non-vanishing quantum grading is $-(m+n)(N-m-n)-mn$. This implies the lemma.
\end{proof}

The next lemma is \cite[Proposition 6.1]{Yonezawa3}. The special case when $m=n=1$ of this lemma has appeared in \cite{Ras-two-bridge}. For the convenience of the reader, we prove this lemma here instead of just citing \cite{Yonezawa3}.

\begin{figure}[ht]
$
\xymatrix{
\input{twisted-fork-1-n} & \input{twisted-fork-1-n-neg} & \input{fork-1-n} \\
\input{twisted-fork-m-1} & \input{twisted-fork-m-1-neg} & \input{fork-m-1}
} 
$
\caption{}\label{twisted-forks-fig} 

\end{figure}

\begin{lemma}\cite[Proposition 6.1]{Yonezawa3}\label{twisted-forks}
Let $\Gamma_{1,n}^\pm$ and $\Gamma_{m,1}^\pm$ be the knotted MOY graphs in Figure \ref{twisted-forks-fig}. Then
\begin{eqnarray*}
\hat{C}(\Gamma_{1,n}^+) \simeq \hat{C}(\Gamma_{1,n})\{q^{n}\}, && \hat{C}(\Gamma_{1,n}^-) \simeq \hat{C}(\Gamma_{1,n})\{q^{-n}\}, \\
\hat{C}(\Gamma_{m,1}^+) \simeq \hat{C}(\Gamma_{m,1})\{q^{m}\}, && \hat{C}(\Gamma_{m,1}^-) \simeq \hat{C}(\Gamma_{m,1})\{q^{-m}\},
\end{eqnarray*}
where ``$\simeq$" is the isomorphism in $\hch(\hmf)$.
\end{lemma}

\begin{figure}[ht]
$
\xymatrix{
\input{twisted-fork-1-n-res1} & \input{twisted-fork-1-n-res0}
} 
$
\caption{}\label{twisted-forks-res-fig} 

\end{figure}

\begin{proof}
We only give details for $\hat{C}(\Gamma_{1,n}^\pm) \simeq \hat{C}(\Gamma_{1,n})\{q^{\pm n}\}$ here. The proof for $\hat{C}(\Gamma_{m,1}^\pm) \simeq \hat{C}(\Gamma_{1,m})\{q^{\pm m}\}$ is very similar and left to the reader. Recall that
\[
\hat{C}(\Gamma_{1,n}) = `` 0 \rightarrow C(\Gamma_{1,n}) \rightarrow 0.
\]
Let $\Gamma_{1,n}'$ and $\Gamma_{1,n}''$ be the MOY graphs in Figure \ref{twisted-forks-res-fig}. Then, by Corollary \ref{explicit-differential-1-n-crossings--res},
\begin{eqnarray*}
\hat{C}(\Gamma_{1,n}^+) & = & `` 0 \rightarrow C(\Gamma_{1,n}') \xrightarrow{\chi^1} C(\Gamma_{1,n}'')\{q^{-1}\} \rightarrow 0", \\
\hat{C}(\Gamma_{1,n}^-) & = & `` 0 \rightarrow C(\Gamma_{1,n}'')\{q\} \xrightarrow{\chi^0} C(\Gamma_{1,n}') \rightarrow 0",
\end{eqnarray*}
where $\chi^0$ and $\chi^1$ are induced by the apparent local changes in MOY graphs. 

Note that $\Gamma_{1,n}'$ is obtained from $\Gamma_{1,n}$ by an edge splitting. Denote by $C(\Gamma_{1,n})\xrightarrow{\phi}C(\Gamma_{1,n}')$ and $C(\Gamma_{1,n}')\xrightarrow{\overline{\phi}}C(\Gamma_{1,n})$ the morphisms induced by this edge splitting and its reverse edge merging. By Decomposition (II) (Theorem \ref{decomp-II}), we know that 
\begin{equation}\label{twisted-forks-res-decomp1}
C(\Gamma_{1,n}')\simeq C(\Gamma_{1,n})\{[n+1]\} = \bigoplus_{j=0}^n C(\Gamma_{1,n})\{q^{-n+2j}\}.
\end{equation}
It is not hard to explicitly write down the inclusion and projection morphisms in this decomposition. For $j=0,\dots,n$, define $\alpha_j = \mathfrak{m}(r^j)\circ \phi$ and $\beta_j = \overline{\phi}\circ \mathfrak{m}(X_{n-j})$, where $X_{k}$ is the $k$-th elementary symmetric polynomial in $\mathbb{X}$. Then $C(\Gamma_{1,n})\{q^{-n+2j}\} \xrightarrow{\alpha_j} C(\Gamma_{1,n}')$ and $C(\Gamma_{1,n}') \xrightarrow{\beta_j} C(\Gamma_{1,n})\{q^{-n+2j}\}$ are homogeneous morphisms preserving the $\zed_2\oplus\zed$-grading. And, by Lemma \ref{phibar-compose-phi},
\[
\beta_j \circ \alpha_i \approx \begin{cases}
\id_{C(\Gamma_{1,n})\{q^{-n+2j}\}} & \text{if } i=j, \\
0 & \text{otherwise}.
\end{cases}
\]
Clearly, $\alpha_i$ and $\beta_j$ are the inclusion and projection morphisms in decomposition \eqref{twisted-forks-res-decomp1}.

By Corollary \ref{contract-expand} and Decomposition (II) (Theorem \ref{decomp-II}), we have 
\begin{equation}\label{twisted-forks-res-decomp2}
C(\Gamma_{1,n}'')\simeq C(\Gamma_{1,n})\{[n]\} = \bigoplus_{j=0}^{n-1} C(\Gamma_{1,n})\{q^{-n+1+2j}\}.
\end{equation}
By decompositions \eqref{twisted-forks-res-decomp1} and \eqref{twisted-forks-res-decomp2}, $\hat{C}(\Gamma_{1,n}^+)$ is isomorphic to 
\[
0\rightarrow \left.%
\begin{array}{c}
C(\Gamma_{1,n})\{q^{-n}\}\\
\oplus \\
C(\Gamma_{1,n})\{q^{-n+2}\} \\
\oplus \\
\vdots \\
\oplus \\
C(\Gamma_{1,n})\{q^{n}\}
\end{array}%
\right. \xrightarrow{\chi^1} \left.%
\begin{array}{c}
C(\Gamma_{1,n})\{q^{-n}\}\\
\oplus \\
C(\Gamma_{1,n})\{q^{-n+2}\} \\
\oplus \\
\vdots \\
\oplus \\
C(\Gamma_{1,n})\{q^{n-2}\}
\end{array}%
\right. \rightarrow 0,
\]
where $\chi^1$ is represented by a $n\times (n+1)$ matrix $(\chi^1_{i,j})_{n\times (n+1)}$. By Lemma \ref{fork-hmf}, we have that 
\begin{equation}\label{twisted-fork-entry-vanish1}
\chi^1_{i,j} \simeq 0 \text{ if } i>j. 
\end{equation}
Similarly, $\hat{C}(\Gamma_{1,n}^-)$ is isomorphic to 
\[
0\rightarrow \left.%
\begin{array}{c}
C(\Gamma_{1,n})\{q^{-n+2}\}\\
\oplus \\
C(\Gamma_{1,n})\{q^{-n+4}\} \\
\oplus \\
\vdots \\
\oplus \\
C(\Gamma_{1,n})\{q^{n}\}
\end{array}%
\right. \xrightarrow{\chi^0} \left.%
\begin{array}{c}
C(\Gamma_{1,n})\{q^{-n}\}\\
\oplus \\
C(\Gamma_{1,n})\{q^{-n+2}\} \\
\oplus \\
\vdots \\
\oplus \\
C(\Gamma_{1,n})\{q^{n}\}
\end{array}%
\right.  \rightarrow 0,
\]
where $\chi^0$ is represented by a $(n+1) \times n$ matrix $(\chi^0_{i,j})_{(n+1)\times n}$. By Lemma \ref{fork-hmf}, we have that 
\begin{equation}\label{twisted-fork-entry-vanish2}
\chi^0_{i,j} \simeq 0 \text{ if }i>j+1.
\end{equation}

Consider the composition $\beta_{j+1} \circ \chi^0 \circ \chi^1 \circ \alpha_j$. On the one hand, by Lemma \ref{phibar-compose-phi} and Corollary \ref{chi-maps-def}, we have
\[
\beta_{j+1} \circ \chi^0 \circ \chi^1 \circ \alpha_j \approx \beta_{j+1} \circ \mathfrak{m}(r-s) \circ \alpha_j \approx \overline{\phi} \circ \mathfrak{m}(X_{n-j-1}r^j(r-s)) \circ \phi \approx \id_{C(\Gamma_{1,n})}.
\]
On the other hand, by \eqref{twisted-fork-entry-vanish1} and \eqref{twisted-fork-entry-vanish2}, we have
\[
\beta_{j+1} \circ \chi^0 \circ \chi^1 \circ \alpha_j \approx \sum_{k=1}^{n-1} \chi^0_{j+1,k} \circ \chi^1_{k,j} \simeq \chi^0_{j+1,j} \circ \chi^1_{j,j}.
\]
So, $\chi^0_{j+1,j} \circ \chi^1_{j,j} \approx \id_{C(\Gamma_{1,n})}$. This shows that $\chi^0_{j+1,j}$ and $\chi^1_{j,j}$ are both isomorphisms in $\hmf$. 

Using \eqref{twisted-fork-entry-vanish1}, we apply Gaussian Elimination (Lemma \ref{gaussian-elimination}) to $\chi^1_{j,j}$ in $\hat{C}(\Gamma_{1,n}^+)$ for $j=1,2,\dots,n$ in that order. This reduces $\hat{C}(\Gamma_{1,n}^+)$ to
\[
0 \rightarrow C(\Gamma_{1,n})\{q^{n}\} \rightarrow 0.
\]
So $\hat{C}(\Gamma_{1,n}^+) \simeq \hat{C}(\Gamma_{1,n})\{q^{n}\}$. Similarly, using \eqref{twisted-fork-entry-vanish2}, we apply Gaussian Elimination (Lemma \ref{gaussian-elimination}) to $\chi^0_{j+1,j}$ $\hat{C}(\Gamma_{1,n}^-)$ for $j=n,n-1,\dots,1$ in that order. This reduces $\hat{C}(\Gamma_{1,n}^-)$ to
\[
0 \rightarrow C(\Gamma_{1,n})\{q^{-n}\} \rightarrow 0.
\]
So $\hat{C}(\Gamma_{1,n}^-) \simeq \hat{C}(\Gamma_{1,n})\{q^{-n}\}$.
\end{proof}

\begin{lemma}\label{invariance-reidemeister-I-unnormal}
Let $D^+$, $D^-$ and $D$ be the knotted MOY graphs in Figure \ref{reidemeisterI}. Then
\begin{eqnarray}
\label{invariance-reidemeister-I-unnormal+}\hat{C}(D^+) & \simeq & \hat{C}(D)\left\langle m \right\rangle\|m\| \{q^{-m(N+1-m)}\}, \\
\label{invariance-reidemeister-I-unnormal-}\hat{C}(D^-) & \simeq & \hat{C}(D)\left\langle m \right\rangle\|-m\| \{q^{m(N+1-m)}\},
\end{eqnarray}
where $\|\ast\|$ means shifting the homological grading by $\ast$. (See Definition \ref{categories-of-complexes}.)
\end{lemma}

\begin{figure}[ht]
\[
\Gamma_1=\xygraph{
!{0;/r2pc/:}
[uu]
!{\xcapv@(0)=>>{m+1}}
!{\hover}
!{\hcap}
[ld]!{\xcapv=><{1}}
[u]!{\xcapv[-1]=<>{m}}
!{\xcapv@(0)=><{m+1}}
}\hspace{.5cm}
\Gamma_2=\xygraph{
!{0;/r1pc/:}
[uuuu]
!{\xcapv@(0)=>>{m+1}}
!{\zbendh}
!{\hcross}
[d]!{\hcross}
!{\hcap=<}
[lld]!{\hcross}
[llu]!{\hover}
[uul]!{\xcapv[2]@(0)}
[dd]!{\xcapv[-2]@(0)|{m}}
[d]!{\sbendh|{1}}
[dl]!{\xcapv@(0)=><{m+1}}
[uuuuurr]!{\hloop[3]=<}
}\hspace{.5cm}
\Gamma_3=\xygraph{
!{0;/r1pc/:}
[uuuu]
!{\xcapv@(0)=>>{m+1}}
!{\zbendh}
!{\hcross}
[d]!{\hcap=>}
[ld]!{\hcross}
[llu]!{\hover}
[uul]!{\xcapv[2]@(0)}
[dd]!{\xcapv[-2]@(0)<{m}}
[d]!{\sbendh|{1}}
[dl]!{\xcapv@(0)=><{m+1}}
[uuuuurr]!{\hcap[3]=<}
}
\]
\[
\Gamma_4=\xygraph{
!{0;/r1.5pc/:}
[uu]
!{\xcapv[0.5]@(0)=>>{m+1}}
[u(0.5)]!{\xcapv@(0)}
[u]!{\sbendv}
[l]!{\vtwist}
!{\xcapv[-1]@(0)=<<{1}}
[ur]!{\hover}
!{\hcap}
[lld]!{\vtwist}
!{\xcapv[-1]@(0)=<<{m}}
!{\zbendv}
[dl]!{\xcapv[0.5]@(0)=>>{m+1}}
}\hspace{.5cm}
\Gamma_5=\xygraph{
!{0;/r2pc/:}
[uu]
!{\xcapv[0.5]@(0)=>>{m+1}}
[u(0.5)]!{\xcapv@(0)}
[u]!{\sbendv}
[l]!{\vtwist}
!{\vtwist}
!{\xcapv@(0)=>>{m}}
!{\zbendv[-1]=<<{1}}
[dl]!{\xcapv[0.5]@(0)=>>{m+1}}
}\hspace{.5cm}
\Gamma_6=\xygraph{
!{0;/r2pc/:}
[uu]
!{\xcapv[0.5]@(0)=>>{m+1}}
[u(0.5)]!{\xcapv@(0)}
[u]!{\sbendv}
[l]!{\vtwist}
!{\xcapv@(0)=>>{m}}
!{\zbendv[-1]=<<{1}}
[dl]!{\xcapv[0.5]@(0)=>>{m+1}}
}\hspace{.5cm}
\Gamma_7=\xygraph{
!{0;/r2pc/:}
[uu]
!{\xcapv[0.5]@(0)=>>{m+1}}
[u(0.5)]!{\xcapv@(0)}
[u]!{\sbendv}
[l]!{\xcapv@(0)=>>{m}}
!{\zbendv[-1]=<<{1}}
[dl]!{\xcapv[0.5]@(0)=>>{m+1}}
}~
\]
\caption{}\label{reidemeisterI-bubble}

\end{figure}

\begin{proof}
We prove \eqref{invariance-reidemeister-I-unnormal+} by an induction on $m$. The proof of \eqref{invariance-reidemeister-I-unnormal-} is similar and left to the reader. 

If $m=1$, then \eqref{invariance-reidemeister-I-unnormal+} follows from \cite[Theorem 2]{KR1}. (See Theorem \ref{invariance-reidemeister-all-color=1} above.) Assume that \eqref{invariance-reidemeister-I-unnormal+} is true for some $m\geq 1$. Let us prove \eqref{invariance-reidemeister-I-unnormal+} for $m+1$. 

Consider the knotted MOY graphs $\Gamma_1,\dots,\Gamma_7$ in Figure \ref{reidemeisterI-bubble}. By Decomposition (II) (Theorem \ref{decomp-II}), we have 
\begin{eqnarray*}
\hat{C}(\Gamma_1) \simeq \hat{C}(D^+)\{[m+1]\} & \text{ and } & \hat{C}(\Gamma_7) \simeq \hat{C}(D)\{[m+1]\}.
\end{eqnarray*}
By Theorem \ref{fork-sliding-invariance-general}, we have $\hat{C}(\Gamma_1) \simeq \hat{C}(\Gamma_2)$. Since \eqref{invariance-reidemeister-I-unnormal+} is true for $1$, we know that $\hat{C}(\Gamma_2) \simeq \hat{C}(\Gamma_3)\left\langle 1 \right\rangle\|1\| \{q^{-N}\}$. From Lemma \ref{invariance-reidemeister-II-III}, one can see that $\hat{C}(\Gamma_3) \simeq \hat{C}(\Gamma_4)$. Since \eqref{invariance-reidemeister-I-unnormal+} is true for $m$, we know that $\hat{C}(\Gamma_4) \simeq \hat{C}(\Gamma_5)\left\langle m \right\rangle\|m\| \{q^{-m(N+1-m)}\}$. By Lemma \ref{twisted-forks}, we know that $\hat{C}(\Gamma_5) \simeq \hat{C}(\Gamma_6)\{q^{m}\}$ and $\hat{C}(\Gamma_6) \simeq \hat{C}(\Gamma_7)\{q^{m}\}$. Putting these together, we get that
\[
\hat{C}(\Gamma_1) \simeq \hat{C}(\Gamma_7)\left\langle m+1 \right\rangle\|m+1\| \{q^{-(m+1)(N-m)}\}.
\]
From Proposition \ref{yonezawa-lemma-hmf}, it follows that \eqref{invariance-reidemeister-I-unnormal+} is true for $m+1$. This completes the induction.
\end{proof}

\begin{lemma}\label{invariance-reidemeister-I}
Let $D_0$ and $D_1$ be two knotted MOY graphs. Assume that there is a Reidemeister move of type I that changes $D_0$ into $D_1$. Then $C(D_0) \simeq C(D_1)$, that is, they are isomorphic as objects of $\hch(\hmf)$.
\end{lemma}

\begin{proof}
The lemma follows easily from Lemma \ref{invariance-reidemeister-I-unnormal} and the normalization in Definition \ref{complex-colored-crossing-def}.
\end{proof}

\end{document}